%% file: open-kpz-invariant.tex
\documentclass[11pt,reqno,american]{extarticle}
\usepackage{amsmath}
\usepackage{amsthm}
\usepackage{libertineRoman}

\usepackage{libertineMono}
\usepackage[libertine]{newtxmath}
\usepackage[T1]{fontenc}
\usepackage{textcomp}
\usepackage[utf8]{inputenc}
\usepackage{xcolor}
\usepackage{colortbl}
\usepackage{babel}
\usepackage{zref-clever}
\usepackage{booktabs}
\usepackage{units}
\usepackage{mathrsfs}
\usepackage{mathtools}
\usepackage{url}
\usepackage{enumitem}
\usepackage{geometry}
\usepackage{xargs}[2008/03/08]
\usepackage[pdfusetitle,
 bookmarks=true,bookmarksnumbered=false,bookmarksopen=false,
 breaklinks=true,pdfborder={0 0 0},pdfborderstyle={},backref=false,colorlinks=true]
 {hyperref}
\hypersetup{
 urlcolor=blue,linkcolor=blue,citecolor=blue}

\makeatletter

\providecommand{\tabularnewline}{\\}

\numberwithin{equation}{section}
\numberwithin{figure}{section}
\numberwithin{table}{section}
\newlist{thmstepnv}{enumerate}{4}
\setlist[thmstepnv]{leftmargin=*,align=left,wide,labelwidth=0pt,labelindent=0pt}
\setlist[thmstepnv,1]{label={\itshape {\thmstepname} \arabic*.},ref=\arabic*}
\setlist[thmstepnv,2]{label={\itshape {\thmstepname} {\thethmstepnvi\alph*}.},ref=\thethmstepnvi\alph*}
\setlist[thmstepnv,3]{label={\itshape {\thmstepname\ \alph*.}},ref=\alph*}
\setlist[thmstepnv,4]{label={\itshape {\thmstepname} \arabic*.},ref=\arabic*}

\DeclareFontFamily{U}{matha}{\hyphenchar\font45}
\DeclareFontShape{U}{matha}{m}{n}{
      <5> <6> <7> <8> <9> <10> gen * matha
      <10.95> matha10 <12> <14.4> <17.28> <20.74> <24.88> matha12
      }{}
\DeclareSymbolFont{matha}{U}{matha}{m}{n}
\DeclareFontSubstitution{U}{matha}{m}{n}

\DeclareFontFamily{U}{mathx}{\hyphenchar\font45}
\DeclareFontShape{U}{mathx}{m}{n}{
      <5> <6> <7> <8> <9> <10>
      <10.95> <12> <14.4> <17.28> <20.74> <24.88>
      mathx10
      }{}
\DeclareSymbolFont{mathx}{U}{mathx}{m}{n}
\DeclareFontSubstitution{U}{mathx}{m}{n}

\DeclareMathDelimiter{\vvvert}{0}{matha}{"7E}{mathx}{"17}

\hypersetup{final}

\usepackage{mleftright}
\mleftright

\usepackage{aligned-overset}
\let\oversetx\overset
\usepackage{placeins}

\usepackage{csquotes}
\AtBeginDocument{\AtEveryBibitem{\clearfield{eprintclass}\clearfield{primaryclass}}}

    \DeclareFontFamily{U}{wncy}{}
    \DeclareFontShape{U}{wncy}{m}{n}{<->wncyr10}{}
    \DeclareSymbolFont{mcy}{U}{wncy}{m}{n}
    \DeclareMathSymbol{\Sha}{\mathord}{mcy}{"58} 
    \DeclareMathSymbol{\NNNN}{\mathord}{mcy}{'111}

\makeatletter
\DeclareFontFamily{OMX}{MnSymbolE}{}
\DeclareSymbolFont{MnLargeSymbols}{OMX}{MnSymbolE}{m}{n}
\SetSymbolFont{MnLargeSymbols}{bold}{OMX}{MnSymbolE}{b}{n}
\DeclareFontShape{OMX}{MnSymbolE}{m}{n}{
    <-6>  MnSymbolE5
   <6-7>  MnSymbolE6
   <7-8>  MnSymbolE7
   <8-9>  MnSymbolE8
   <9-10> MnSymbolE9
  <10-12> MnSymbolE10
  <12->   MnSymbolE12
}{}
\DeclareFontShape{OMX}{MnSymbolE}{b}{n}{
    <-6>  MnSymbolE-Bold5
   <6-7>  MnSymbolE-Bold6
   <7-8>  MnSymbolE-Bold7
   <8-9>  MnSymbolE-Bold8
   <9-10> MnSymbolE-Bold9
  <10-12> MnSymbolE-Bold10
  <12->   MnSymbolE-Bold12
}{}

\let\llangle\@undefined
\let\rrangle\@undefined
\DeclareMathDelimiter{\llangle}{\mathopen}%
                     {MnLargeSymbols}{'164}{MnLargeSymbols}{'164}
\DeclareMathDelimiter{\rrangle}{\mathclose}%
                     {MnLargeSymbols}{'171}{MnLargeSymbols}{'171}
\makeatother

\zcsetup{
    abbrev=true, %
    cap=true,
    nameinlink=false,
	rangetopair=false,
	endrange=stripprefix,
	sort=false,
	reffont=\upshape,
	lastsep={, and },
	comp=false
    }

\zcRefTypeSetup{equation}{
Name-sg = ,
name-sg = ,
Name-pl = ,
name-pl = ,
refbounds = {\textup{(},,,\textup{)}},
+refbounds-rb = {\textup{(},,,},
+refbounds-re = {,,,\textup{)}},
rangesep = {--}
}
\zcRefTypeSetup{item}{
Name-sg = ,
name-sg = ,
Name-pl = ,
name-pl = ,
refbounds = {\textup{(},,,\textup{)}},
+refbounds-rb = {\textup{(},,,},
+refbounds-re = {,,,\textup{)}},
rangesep = {--}
}
\zcRefTypeSetup{thmstepnvi}{
Name-sg = Step,
name-sg = step,
Name-pl = Steps,
name-pl = steps
}
\usepackage{xfrac}
\AtBeginEnvironment{prop}{\setlist{before=\leavevmode, nosep}}

\usepackage{ajdtrees2}
\usepackage{bigdelim}

\usepackage{letltxmacro}
\LetLtxMacro{\oldzcref}{\zcref}
\renewcommand{\zcref}[2][]%
             {\ifmmode\text{\oldzcref[{#1}]{#2}}\else\oldzcref[{#1}]{#2}\fi}

\usepackage{forest}

\newcommand{\myheppsector}[1]{%
\mathord{\vcenter{\hbox{\begin{forest}
      for tree={
        draw = none,
        align=center,
        minimum size=0pt,   %
        inner sep=0pt,      %
        s sep=4mm,          %
        l sep=1mm,          %
        l = 1mm,
              edge path={
                        \noexpand\path [\forestoption{edge}]
          (!u.parent anchor) -| (.child anchor)\forestoption{edge label};
      }
      }
      [{ }, draw=none, no edge [#1]]
      \end{forest}}}}
}

\usepackage{style_commands}

\usepackage[normalem]{ulem}

\AtBeginDocument{
\DeclareFieldFormat{doi}{%
  \ifhyperref
    {\href{https://doi.org/#1}{DOI}}
    {\textsc{DOI}}}
}

\renewcommand{\eqref}{\zcref}

\ifdefined\showcaptionsetup
 \PassOptionsToPackage{caption=false}{subfig}
\fi
\usepackage{subfig}

\usepackage[nocfg,refpage,intoc]{nomencl}

\makenomenclature

\usepackage{multicol}
\renewcommand\nompreamble{\begin{multicols}{2}\small\raggedright}
\renewcommand\nompostamble{\end{multicols}}

\makeatother

\theoremstyle{plain}
\newtheorem{thm}{\protect\theoremname}[section]
\theoremstyle{remark}
\newtheorem{rem}[thm]{\protect\remarkname}
\theoremstyle{plain}
\newtheorem{prop}[thm]{\protect\propositionname}
\newtheorem{lem}[thm]{\protect\lemmaname}
\theoremstyle{definition}
\newtheorem{defn}[thm]{\protect\definitionname}
\theoremstyle{plain}
\newtheorem{cor}[thm]{\protect\corollaryname}
\usepackage[style=trad-plain,url=false,isbn=false]{biblatex}
\zcsetup{countertype={thm=theorem}}
\AddToHook{env/rem/begin}{\zcsetup{countertype={thm=remark}}}
\AddToHook{env/prop/begin}{\zcsetup{countertype={thm=proposition}}}
\AddToHook{env/lem/begin}{\zcsetup{countertype={thm=lemma}}}
\AddToHook{env/defn/begin}{\zcsetup{countertype={thm=definition}}}
\AddToHook{env/cor/begin}{\zcsetup{countertype={thm=corollary}}}
\providecommand{\corollaryname}{Corollary}
\providecommand{\definitionname}{Definition}
\providecommand{\lemmaname}{Lemma}
\providecommand{\propositionname}{Proposition}
\providecommand{\remarkname}{Remark}
\providecommand{\theoremname}{Theorem}
\providecommand{\thmstepname}{Step}

\usepackage{amsbsy}
\usepackage{amstext}
\begin{document}
\input{makeboxes2}

\global\long\def\supp{\operatorname{supp}}%

\global\long\def\Sym{\operatorname{Sym}}%

\global\long\def\Anti{\operatorname{Anti}}%

\global\long\def\diam{\operatorname{diam}}%

\global\long\def\id{\operatorname{id}}%

\global\long\def\Uniform{\operatorname{Uniform}}%

\global\long\def\dif{\mathrm{d}}%

\global\long\def\e{\mathrm{e}}%

\global\long\def\ii{\mathrm{i}}%

\global\long\def\Cov{\operatorname{Cov}}%

\global\long\def\dist{\operatorname{dist}}%

\global\long\def\Var{\operatorname{Var}}%

\global\long\def\pv{\operatorname{p.v.}}%

\global\long\def\spn{\operatorname{span}}%

\global\long\def\e{\mathrm{e}}%

\global\long\def\p{\mathrm{p}}%

\global\long\def\Law{\operatorname{Law}}%

\global\long\def\supp{\operatorname{supp}}%

\global\long\def\image{\operatorname{image}}%

\global\long\def\dif{\mathrm{d}}%

\global\long\def\eps{\varepsilon}%

\global\long\def\sgn{\operatorname{sgn}}%

\global\long\def\tr{\operatorname{tr}}%

\global\long\def\Hess{\operatorname{Hess}}%

\global\long\def\Re{\operatorname{Re}}%

\global\long\def\Im{\operatorname{Im}}%

\global\long\def\hash{\mathbin{\#}}%

\global\long\def\Dif{\mathrm{D}}%

\global\long\def\divg{\operatorname{div}}%

\global\long\def\arsinh{\operatorname{arsinh}}%

\global\long\def\sech{\operatorname{sech}}%

\global\long\def\erf{\operatorname{erf}}%

\global\long\def\Cauchy{\operatorname{Cauchy}}%

\global\long\def\artanh{\operatorname{artanh}}%

\global\long\def\diag{\operatorname{diag}}%

\global\long\def\pr{\operatorname{pr}}%

\global\long\def\Left{\mathrm{L}}%

\global\long\def\Right{\mathrm{R}}%

\global\long\def\LR{\mathrm{LR}}%

\global\long\def\ssS{\mathsf{S}}%

\global\long\def\GS{\mathsf{GS}}%

\global\long\def\clA{\mathcal{A}}%

\global\long\def\ssN{\mathsf{N}}%

\global\long\def\ssM{\mathsf{M}}%

\global\long\def\ssH{\mathsf{H}}%

\global\long\def\ssQ{\mathsf{Q}}%

\global\long\def\sfs{\mathsf{s}}%

\global\long\def\sfr{\mathsf{r}}%

\global\long\def\sff{\mathsf{f}}%

\global\long\def\sfY{\mathfrak{Y}}%

\global\long\def\TT{\mathbb{T}}%

\global\long\def\RR{\mathbb{R}}%

\global\long\def\ZZ{\mathbb{Z}}%

\global\long\def\PP{\mathbb{P}}%

\global\long\def\NN{\mathbb{N}}%

\global\long\def\scM{\mathscr{M}}%

\global\long\def\st{\mathrel{:}}%

\global\long\def\uu{\mathsf{u}}%

\global\long\def\vv{\mathsf{v}}%

\global\long\def\Sh{\Sha}%

\global\long\def\NNN{\NNNN}%

\global\long\def\oh{\sfrac{1}{2}}%

\global\long\def\thrh{\sfrac{3}{2}}%

\include{rs-commands-lyx-2}

\global\long\def\boxop#1{\rsmath[rsK]{#1}}%

\global\long\def\bboxop#1{\rsmath[rsP]{#1}}%

\global\long\def\ovset#1{\oversetx{#1}}%

\global\long\def\heppsector#1{\myheppsector{#1}}%

\title{Invariant measures for the open KPZ equation:\\
an analytic perspective}
\author{Alexander Dunlap\thanks{Department of Mathematics, Duke University, Durham, NC 27708, USA.
Email: \protect\url{alexander.dunlap@duke.edu}}\and Yu Gu\thanks{Department of Mathematics, University of Maryland, College Park, MD
20742, USA. Email: \protect\url{ygu7@umd.edu}}\and Tommaso Rosati\thanks{Department of Statistics, University of Warwick, Coventry CV4 7AL,
UK. Email: \protect\url{t.rosati@warwick.ac.uk}}}
\maketitle
\begin{abstract}
The ergodic theory of the open KPZ equation has seen significant progress in recent years, with explicit invariant measures described in a series of works by Corwin--Knizel, Barraquand--Le Doussal, and Bryc--Kuznetsov--Wang--Weso\l owski. %
In this paper, we provide  a stochastic analytic proof of the formula for the invariant measures. %
Our approach starts from the Gaussian invariant
measure for the case of homogeneous boundary conditions. We approximate
the inhomogeneous problem by a homogeneous one with a singular boundary
potential. Using tools including change of measure, time reversal
for Markov processes, and Itô's formula, we then reduce the problem
to analyzing the KPZ nonlinearity in a thin boundary layer. Finally,
using the theory of regularity structures, we establish a central
limit theorem for the time-integrated nonlinearity near the boundary,
which completes the proof of the invariance. Although it is known
that different boundary parameters give rise to distinct physical
regimes for the invariant measures, our method is robust and does
not rely on any particular choice of boundary parameters. 
\end{abstract}
\tableofcontents{}

\listoftables

\section{Introduction}

Fix $L\in(0,\infty)$ and $\uu,\vv\in\mathbb{R}$. We consider the
\emph{open KPZ equation}, which is the KPZ equation with inhomogeneous
Neumann boundary conditions on the strip $\mathbb{R}\times[0,L]$,
given formally by
\begin{subequations}
\label{eq:huv}
\begin{align}
\dif h_{\uu,\vv;t}(x) & =\frac{1}{2}\left(\Delta h_{\uu,\vv;t}(x)+(\partial_{x}h_{\uu,\vv;t}(x))^{2}\right)\dif t+\dif W_{t}(x), &  & t\in\mathbb{R},x\in(0,L);\label{eq:huv-eqn}\\
\partial_{x}h_{\uu,\vv;t}(0) & =\uu\qquad\text{and}\qquad\partial_{x}h_{\uu,\vv;t}(L)=-\vv, &  & t\in\mathbb{R};\label{eq:huv-Neumann}
\end{align}
\end{subequations}
where $(\dif W_{t})$ is a space-time white noise. The problem \zcref{eq:huv}
cannot be interpreted as a stochastic PDE as it is written, since
solutions are expected to have the spatial regularity of Brownian
motion, and thus neither the nonlinearity in \zcref{eq:huv-eqn} nor
the boundary conditions \zcref{eq:huv-Neumann} can be understood
in the classical sense. However, the ``correct'' or ``physical''
meaning of the problem is now well-understood \cite{corwin:shen:2018:weakly,parekh:2019:KPZ,gerencer:hairer:2019:singular}:
it can be interpreted through the Cole--Hopf transform and the mild
formulation of a stochastic heat equation with Robin boundary conditions,
as we detail in \zcref{sec:solntheory} below.

The goal of the present work is to investigate the ergodic behavior
of \zcref{eq:huv}, in particular the invariant measure. It is known
that, for a given choice of $\uu$ and $\vv$, there is a unique invariant
measure for $h_{\uu,\vv}$ \emph{up to a spatially constant height
shift}, or equivalently a unique invariant measures for the derivative
$u_{\uu,\vv;t}\coloneqq\partial_{x}h_{\uu,\vv;t}$, which formally
satisfies the \emph{open stochastic Burgers equation}
\begin{subequations}
\label{eq:uuv}
\begin{align}
\dif u_{\uu,\vv;t}(x) & =\frac{1}{2}\left(\Delta u_{\uu,\vv;t}(x)+\partial_{x}((u_{\uu,\vv;t})^{2})(x)\right)\dif t+\partial_{x}\dif W_{t}(x), &  & t\in\mathbb{R},x\in(0,L);\label{eq:uuv-eqn}\\
u_{\uu,\vv;t}(0) & =\uu\qquad\text{and}\qquad u_{\uu,\vv;t}(L)=-\vv, &  & t\in\mathbb{R};\label{eq:uuv-Dirichlet}
\end{align}
\end{subequations}
The existence of invariant measures for \zcref{eq:uuv} (along with much more, as will be discussed below) was shown in \cite{corwin:knizel:2024:stationary}. Uniqueness was proved
in \cite{knizel:matetski:arXiv2022:strong,parekh:2022:ergodicity},
which relied on the compactness of the domain $[0,L]$. This
is in some sense a generalization of the classical work of Sinai in
the periodic setting \cite{sinai:1991:two}.

Describing the invariant measure for \zcref{eq:uuv} has been the
subject of significant work. All existing work takes the approach
of first identifying an appropriate discrete model which on one hand
admits an explicit description of the invariant measure and on the
other hand approximates the open KPZ equation in a certain asymptotic
regime, then passing to the limit of the invariant measure on the
discrete level to obtain explicit descriptions of the invariant measure
on the continuous level. The same spirit applies to the KPZ equation
on the whole line, half line or in the periodic setting. Compared
to the whole line or the periodic setting, the boundary condition
\zcref{eq:huv-Neumann} makes the problem much less tractable. As
a matter of fact, the difficulty of the problem changes substantially
with different choices of the boundary parameters $\uu$ and $\vv$.
The reason is that, for certain chosen discrete models, the explicit
description of the invariant measure is only available for $\uu,\vv$
(and even $L$) in some specific regimes, and performing an analytic
continuation to other values of $\uu$ and $\vv$ is by no means straightforward.

There are relatively easier cases, in particular when $\uu+\vv=0$
so that the slopes of the height function $h_{\uu,\vv}$ at the two
boundaries are the same. In this case, one may guess that a Brownian
motion with drift $\uu=-\vv$ is invariant under the evolution of
\zcref{eq:huv}, in light of the known fact the two-sided Brownian
motion with drift is invariant under the KPZ evolution on the whole
line. Indeed, in this case, it is not hard to find discrete models
which have product invariant measures and approximate \zcref{eq:huv}.
In this case, as one passes to the limit as a drifted random walk
approximates a drifted Brownian motion which is then invariant under
\zcref{eq:huv}. For example, in the special case $\uu=\vv=0$, it
was shown in \cites{corwin:shen:2018:weakly}{goncalves:perkowski:simon:2020:derivation}
that the invariant measure for \zcref{eq:uuv} is spatial white noise.

As soon as $\uu+\vv\ne0$, the problem becomes much more difficult
and one generally does not expect Gaussian invariant measures. The
first breakthrough in this direction was obtained in \cite{corwin:knizel:2024:stationary},
where an explicit characterization of the invariant measure was obtained in the regime $\uu+\vv\ge0$,
in terms of the multipoint Laplace transform. Shortly afterwards, the
Laplace transform was inverted in \cite{bryc:kuznetsov:wang:wesolowski:2023:markov,barraquand:le_doussal:2022:steady},
in the mathematics and physics literature respectively. The resulting probabilistic description can in fact be extended to all $\uu,\vv\in\mathbb{R}$, and hence it was conjectured
that it indeed describes the invariant measure for all $\uu,\vv$.
This description can be written as follows. Let $\mu_{\uu,\vv}$ be the invariant measure
for \zcref{eq:uuv}. Then $\mu_{\uu,\vv}$ is absolutely continuous
with respect to the law $\mu_{0,0}$ of spatial white noise on $[0,L]$,
and the Radon--Nikodym derivative is given by
\begin{equation}
\frac{\dif\mu_{\uu,\vv}}{\dif\mu_{0,0}}(u)=\mathcal{Y}_{\uu,\vv}(u)\coloneqq\mathfrak{Z}^{-1}_{\uu,\vv}\mathrm{E}_{B}\left[\e^{-\mathsf{u}(h(0)-B(0))-\mathsf{v}(h(L)-B(L))}\left(\int^{L}_{0}\e^{-(h(x)-B(x))}\,\dif x\right)^{-\mathsf{u}-\mathsf{v}}\right].\label{eq:Duv-def}
\end{equation}
Here $h$ is any antiderivative of $u$, $\mathrm{E}_{B}$ is expectation
with respect to an auxiliary standard Brownian motion $B$ on $[0,L]$
with $B(0)=0$, and $\mathfrak{Z}_{\uu,\vv}$ is the deterministic
constant chosen such that the Radon--Nikodym derivative has expectation
$1$. For a more detailed discussion on the relevant literature, we
refer to the review \cite{corwin:2022:some} and the references cited
therein. Subsequent developments on this topic have included \cite{barraquand:le_doussal:2023:stationary,himwich:arXiv2024:stationary}.
In particular, \cite{himwich:arXiv2024:stationary} proved the explicit
description of the invariant measures for parameters satisfying $\uu+\vv>0$. 

Most recently, \cite{barraquand:corwin:yang:2024:stationary} studied
models of integrable polymers on the strip, including geometric last
passage percolation and the log-gamma polymer. By unraveling a two-layer
Gibbs measure structure and performing an analytic continuation, they
were able to describe explicit invariant measures for all $\uu,\vv\in\mathbb{R}$
that are discrete analogues of \zcref{eq:Duv-def}. It is natural
to conjecture that the open KPZ equation arises as the limit of the
log-gamma polymer in the intermediate disorder regime, and modulo
this convergence, they were able to justify \zcref{eq:Duv-def} for
the open KPZ equation for all $\uu,\vv\in\mathbb{R}$. 

As mentioned already, all of the previous works relied on the analysis
of integrable discrete models and taking limits to pass to the stochastic
PDE \zcref{eq:huv}. While the methods developed have uncovered rich
integrable structures, it is a natural and compelling question whether
one can develop a more direct approach based on the equation itself.
This problem is surprisingly difficult even for the equation without
boundary conditions. Ultimately, the difficulty lies in the singular
nature of the equation. This makes it quite challenging to proceed
as if one is dealing with a typical stochastic differential equation,
i.e.~by constructing the generator of the corresponding Markov process
and then checking directly through the generator that a certain measure
is invariant under the evolution. (See \cite{gubinelli:perkowski:2020:generator}
for some of the difficulties involved in working with the generator.)
Some progress on proving Gaussian invariant measures for KPZ on the
whole line or in the periodic setting, without relying on a discrete
integrable approximation, can be found in \cite{gu:quastel:arXiv2024:integration}.
This method can also be used to handle the open boundary condition
$\uu+\vv=0$, but seems to break down in the case when $\uu+\vv\ne0$.
In particular, we note that the approach taken in the present work
is completely orthogonal to the methods of \cite{gu:quastel:arXiv2024:integration}.

The purpose of the present work is to provide a stochastic analytic
proof of the characterization \zcref{eq:Duv-def} of the invariant
measure, which indeed is entirely different from all of the aforementioned
previous works concerning $(\uu,\vv)\ne(0,0)$. Through the analytic
lens, one may gain new understanding and insight into the structure
of the invariant measures, as well as the properties of the solutions
to \zcref{eq:huv}. In particular, one can better understand the height
growth near the boundaries, which is where all of the interesting
physics takes place that leads to the generally non-Gaussian measure
described in \zcref{eq:Duv-def}.

Our starting point is \cite{goncalves:perkowski:simon:2020:derivation},
which considers the case $\uu=\vv=0$, and hence concerns a Gaussian
invariant measure. Our strategy is based on the following four ingredients:
\begin{enumerate}
\item the time reversal property of the stationary Markov process when $\uu=\vv=0$,
studied in \cite{goncalves:perkowski:simon:2020:derivation};
\item using the Cameron--Martin theorem to treat the actual boundary condition
with $\uu,\vv\in\mathbb{R}$ as a (singular) perturbation of the noise;
\item applying Itô's formula to a certain functional of the solution to
the stochastic heat equation, which unravels a crucial martingale
structure associated with the Radon--Nikodym derivative in \zcref{eq:Duv-def},
modulo the understanding of the \emph{formal} nonlinear term $(\partial_{x}h_{\uu,\vv;t}(x))^{2}$;
and
\item using the theory of regularity structures \cite{hairer:2014:theory}
to analyze the behavior of the nonlinearity at the boundaries via
a local expansion.
\end{enumerate}
Our proof proceeds in a uniform manner for all $(\uu,\vv)\ne(0,0)$,
without using analytic continuation. We also do not use the two-layer
structure of the invariant measures, although we believe it would
be interesting to extend our method to consider the two-layer problem.

As expected, here is the main result:
\begin{thm}
\label{thm:mainthm}Let $\uu,\vv\in\mathbb{R}$. The invariant measure
$\mu_{\uu,\vv}$ for \zcref{eq:uuv} is absolutely continuous with
respect to the law $\mu_{0,0}$ of spatial white noise on $[0,L]$
with Radon--Nikodym derivative given in \zcref{eq:Duv-def}.
\end{thm}

\subsection{\label{subsec:Our-method}Our method}

In this section, we explain on a heuristic level the main ideas in
the proof. We omit the subscript $\uu,\vv$ here to simplify the notation,
writing $h=h_{\uu,\vv}$.

The starting point of our approach is to view \zcref{eq:huv} as a
perturbation of the same equation with $\uu=\vv=0$. To see why this
is possible, we note that, for the standard heat equation $\partial_{t}f=\frac{1}{2}\Delta f$
on $\mathbb{R}\times[0,L]$ with Neumann boundary conditions $\partial_{x}f(0)=\uu$,
$\partial_{x}f(L)=-\vv$, one can check through an integration by
parts that the even extension of $f$ solves the equation $\partial_{t}f=\frac{1}{2}\Delta f-\uu\delta_{0}-\vv\delta_{L}$
on $\mathbb{R}\times[-L,L]$ with periodic boundary conditions. In
other words, the inhomogeneous boundary condition in \zcref{eq:huv-Neumann}
can be interpreted as a Dirac forcing on the boundary. Thus, on a
formal level, one can rewrite \zcref{eq:huv} as
\begin{equation}
\dif h_{t}(x)=\frac{1}{2}\left(\Delta h_{t}(x)+(\partial_{x}h_{t}(x))^{2}\right)\dif t+\dif W_{t}(x)-\uu\delta_{0}-\vv\delta_{L},\qquad t\in\mathbb{R},x\in\mathbb{R}/(2L\mathbb{Z}),\label{eq:hdirac}
\end{equation}
where $\dif W_{t}(x)$ is extended evenly from $[0,L]$ to $\mathbb{R}/(2L\mathbb{Z})$.
With the above equation, the idea is to change the underlying probability
measure so that, under the new measure, the white noise has the law
of $\dif W_{t}(x)-\uu\delta_{0}-\vv\delta_{L}$. Of course, the additional
forcing term $-\uu\delta_{0}-\vv\delta_{L}$ does not live in the
Cameron--Martin space associated with the white noise, so one cannot
really view the effect of the boundary conditions as a change of measure.
Nevertheless, we regard it as a ``singular'' change of measure,
in a sense that we now make precise.

As usual, we proceed through an approximation. We let $\varphi^{\eps}_{\uu,\vv}$
be an $\eps$-approximation of the singular forcing $-\uu\delta_{0}-\vv\delta_{L}$
and then consider the equation \zcref{eq:hdirac} with $-\uu\delta_{0}-\vv\delta_{L}$
replaced by $\varphi^{\eps}_{\uu,\vv}$. Let $\mathcal{Q}^{\eps}_{\uu,\vv;0,T}$
be the Radon--Nikodym derivative associated with the change of measure
$\dif W_{t}(x)\mapsto\dif W_{t}(x)+\varphi^{\eps}_{\uu,\vv}$, restricted
to a fixed time interval $[0,T]$:
\begin{equation}
\mathcal{Q}^{\eps}_{\uu,\vv;0,T}\coloneqq\exp\left\{ \int^{T}_{0}\int^{L}_{0}\varphi^{\eps}_{\uu,\vv}(y)\,\dif W_{t}(y)-\frac{1}{2}T\|\varphi^{\eps}_{\uu,\vv}\|^{2}_{L^{2}([0,L])}\right\} .\label{eq:defRN}
\end{equation}
In this way, we simply consider the solution to \zcref{eq:huv} with
$\uu=\vv=0$ except that when computing statistical quantities for
general $\uu,\vv\in\mathbb{R}$, the law of the white noise forcing
needs to be tilted by $\mathcal{Q}^{\eps}_{\uu,\vv;0,T}$. In other
words, we stick to the equation with homogeneous boundary conditions
and incorporate all of the statistical information induced by the
inhomogeneous boundary conditions into the Radon--Nikodym derivative
\zcref{eq:defRN}.

For any $\uu,\vv\in\mathbb{R}$, to prove that the measure $\mu_{\uu,\vv}$
given by \zcref{eq:Duv-def} is invariant under the stochastic Burgers
equation, it is enough to show that, for any $T>0$ and a large class
of test functions $F$, we have
\begin{equation}
\mathbb{E}[\mathcal{Y}_{\uu,\vv}(u_{0})F(u_{0})]=\lim_{\eps\to0}\mathbb{E}[\mathcal{Y}_{\uu,\vv}(u_{0})F(u_{T})\mathcal{Q}^{\eps}_{\uu,\vv;0,T}],\label{eq:9111}
\end{equation}
where, under $\mathbb{E}$, $u_{0}$ is a spatial white noise and
$u_{t}=\partial_{x}h_{t}$ solves \zcref{eq:uuv} with $\uu=\vv=0$.
Since the spatial white noise is invariant in the case $\uu=\vv=0$
as proved in \cite{goncalves:perkowski:simon:2020:derivation}, the
above equation can be rewritten as
\[
\mathbb{E}[\mathcal{Y}_{\uu,\vv}(u_{T})F(u_{T})]=\lim_{\eps\to0}\mathbb{E}[\mathcal{Y}_{\uu,\vv}(u_{0})F(u_{T})\mathcal{Q}^{\eps}_{\uu,\vv;0,T}]=\lim_{\eps\to0}\mathbb{E}\left[\mathbb{E}[\mathcal{Y}_{\uu,\vv}(u_{0})\mathcal{Q}^{\eps}_{\uu,\vv;0,T}\mid u_{T}]F(u_{T})\right],
\]
where, unusually, the inner conditional expectation on the right side
is taken with respect to the ``future information'' $u_{T}$. Since
$F$ is an arbitrary test function, this is equivalent to proving
the following relation:
\begin{equation}
\mathcal{Y}_{\uu,\vv}(u_{T})=\lim_{\eps\to0}\mathbb{E}[\mathcal{Y}_{\uu,\vv}(u_{0})\mathcal{Q}^{\eps}_{\uu,\vv;0,T}\mid u_{T}].\label{eq:key911}
\end{equation}

At this stage, we use another important feature of the case $\uu=\vv=0$:
the solution to \zcref{eq:uuv} at stationarity satisfies a type of
time-reversal skew symmetry. Precisely, for fixed $T>0$, the time-reversal
$\hat{u}_{\hat{t}}\coloneqq u_{T-\hat{t}}$ solves the same equation
with the opposite sign in front of the nonlinear term and a different
space-time white noise:
\begin{equation}
\dif\hat{u}_{\hat{t}}(x)=\frac{1}{2}\left(\Delta\hat{u}_{\hat{t}}(x)-\partial_{x}(\hat{u}^{2}_{\hat{t}})(x)\right)\dif t+\partial_{x}\dif\hat{W}_{\hat{t}}(x),\label{eq:backu}
\end{equation}
where $\dif\hat{W}_{\hat{t}}(x)$ is another space-time white noise
that is correlated with $\dif W_{t}(x)$ in a rather complicated way.
Since the conditional expectation in \zcref{eq:key911} is taken with
respect to the future, it is natural to rewrite \zcref{eq:key911}
in terms of the backward solution $(\hat{u}_{\hat{t}})$:
\begin{equation}
\mathcal{Y}_{\uu,\vv}(\hat{u}_{0})=\lim_{\eps\to0}\mathbb{E}[\mathcal{Y}_{\uu,\vv}(\hat{u}_{T})\mathcal{Q}^{\eps}_{\uu,\vv;0,T}\mid\hat{u}_{0}].\label{eq:key1911}
\end{equation}

The immediate difficulty arising from the above expression is that
the Radon--Nikodym derivative $\mathcal{Q}_{\uu,\vv;0,T}$ is expressed
in terms of the forward noise $\dif W_{t}(x)$, with which it is quite
challenging to compute the conditional expectation given the future,
since $(\dif W_{t}(x))_{t}$ is not adapted to the backward filtration.
However, we will be able to rewrite $\mathcal{Q}^{\eps}_{\uu,\vv;0,T}$
in terms of the backward solution $(\hat{u}_{\hat{t}})_{\hat{t}}$
by comparing the two equations \zcref{eq:uuv-eqn,eq:backu}. Thus,
it is not hard to imagine that to prove \zcref{eq:key1911}, it will
suffice to show that a certain functional of the backward noise and
the solution is a martingale with respect to the backward filtration.
To illustrate the main ideas and discuss the difficulties, in the
next section we consider a toy example in which a similar strategy
can be implemented.

\subsubsection{A toy example}

We consider the one-dimensional SDE 
\[
\dif X_{t}=-V'(X_{t})\dif t+\dif B_{t},
\]
where $(B_{t})$ is a standard Brownian motion and $V$ is a smooth
potential that grows rapidly at infinity. It is well-known that the
Markov process $(X_{t})_{t}$ has a unique invariant measure with
density 
\[
p_{0}(x)=\mathfrak{Z}^{-1}_{0}\e^{-2V(x)},
\]
where $\mathfrak{Z}$ is the normalization constant. Now we suppose
that we perturb the dynamics by adding an additional drift $\theta\in\mathbb{R}$
to the Brownian motion: $B_{t}\mapsto B_{t}+\theta t$. The goal is
to understand the invariant measure for these modified dynamics.

The usual approach is to absorb the drift into the potential, writing
the new dynamics as 
\[
\dif X_{t}=-(V'(X_{t})-\theta)\dif t+\dif B_{t}.
\]
This new dynamics has a unique invariant measure with new density
\[
p_{\theta}(x)\coloneqq\mathfrak{Z}^{-1}_{\theta}\e^{-2(V(x)-\theta x)}.
\]
The new invariant measure is absolutely continuous with respect to
the original one, and the Radon--Nikodym derivative is given by
\begin{equation}
\mathcal{Y}_{\theta}(x)\coloneqq\frac{\mathfrak{Z}_{0}}{\mathfrak{Z}_{\theta}}\e^{2\theta x}.\label{eq:Ytheta}
\end{equation}
All of this is classical and well-known, but let us try to take a
more complicated approach which will illustrate the strategy in the
proof of \zcref{thm:mainthm}.

To show that the Radon--Nikodym derivative given by $\mathcal{Y}_{\theta}$
leads to the invariant measure of the perturbed dynamics, it is in
fact equivalent to show that, for the unperturbed dynamics $(X_{t})_{t}$,
any $T>0$, and any bounded function $F\colon\mathbb{R}\to\mathbb{R}$,
it holds that
\begin{equation}
\mathbb{E}[\mathcal{Y}_{\theta}(X_{0})F(X_{0})]=\mathbb{E}[\mathcal{Y}_{\theta}(X_{0})\mathcal{Q}_{\theta;0,T}F(X_{T})],\label{eq:9112}
\end{equation}
where $X_{0}$ is sampled from the density $p_{0}$ and the Radon--Nikodym
derivative induced by the change $B_{t}\mapsto B_{t}+\theta t$ is
given by 
\[
\mathcal{Q}_{\theta;0,T}\coloneqq\exp\left\{ \theta B_{T}-\frac{1}{2}\theta^{2}T\right\} .
\]
Equation \zcref{eq:9112} should be compared with \zcref{eq:9111},
which is somewhat more complicated in that a limiting procedure must
be used to deal with the ``singularity'' of the change of measure
in that case. Again proceeding similarly to the above, we see that
\[
\mathbb{E}[\mathcal{Y}_{\theta}(X_{0})F(X_{0})]=\mathbb{E}[\mathcal{Y}_{\theta}(X_{T})F(X_{T})],
\]
and by taking conditional expectation with respect to the future,
\zcref{eq:9112} reduces to 
\begin{equation}
\mathcal{Y}_{\theta}(X_{T})=\mathbb{E}[\mathcal{Y}_{\theta}(X_{0})\mathcal{Q}_{\theta;0,T}\mid X_{T}],\label{eq:keytoy}
\end{equation}
which corresponds to \zcref{eq:key911}. Now, in this highly simplified
case, $(X_{t})_{t}$ is reversible, and indeed the time-reversed process
$(\hat{X}_{\hat{t}})_{\hat{t}}\coloneqq(X_{T-\hat{t}})_{\hat{t}}$
satisfies 
\[
\dif\hat{X}_{\hat{t}}=-V'(\hat{X}_{\hat{t}})\dif\hat{t}+\dif\hat{B}_{t}
\]
for another standard Brownian motion $(\hat{B}_{\hat{t}})_{\hat{t}}$.
By combining the two equations 
\begin{equation}
X_{T}=X_{0}-\int^{T}_{0}V'(X_{s})\,\dif s+B_{T}\qquad\text{and}\qquad\hat{X}_{T}=\hat{X}_{0}-\int^{T}_{0}V'(\hat{X}_{\hat{s}})\,\dif\hat{s}+\hat{B}_{T},\label{e.toyequation}
\end{equation}
we express $B_{T}$, which appears in the expression of $\mathcal{Q}_{\theta;0,T}$,
in terms of $\hat{B}$ and $\hat{X}$: 
\begin{equation}
B_{T}=\hat{B}_{T}-2(\hat{X}_{T}-\hat{X}_{0}).\label{e.BhatB}
\end{equation}
Therefore, \zcref{eq:keytoy} can be further rewritten as 
\begin{equation}
\mathcal{Y}_{\theta}(\hat{X}_{0})=\mathbb{E}[\mathcal{Y}_{\theta}(\hat{X}_{T})\e^{\theta\hat{B}_{T}-2\theta(\hat{X}_{T}-\hat{X}_{0})-\tfrac{1}{2}\theta^{2}T}\mid\hat{X}_{0}].\label{eq:oldkey}
\end{equation}
That is, to check that the $\mathcal{Y}_{\theta}$ is the desired
change of measure, we need to show that the identity \zcref{eq:oldkey}
holds. But if we plug in the definition \zcref{eq:Ytheta} of $\mathcal{Y}_{\theta}$
into \zcref{eq:oldkey}, the desired identity \zcref{eq:oldkey} simply
reduces to the elementary fact that 
\[
1=\mathbb{E}[\e^{\theta\hat{B}_{T}-\tfrac{1}{2}\theta^{2}T}\mid\hat{X}_{0}].
\]

To summarize, by taking an apparently more complicated approach, we
have reached the same conclusion that the invariant measure for the
perturbed dynamics is absolutely continuous with respect to the old
one, with the Radon-Nikodym derivative given by \eqref{eq:Ytheta}.
We highlight two aspects of the preceding argument, as they will later
constitute the principal difficulties in the proof of \zcref{thm:mainthm}:
\begin{enumerate}
\item As one may have noticed, in this toy example, a key step is to express
the change of measure factor $\mathcal{Q}_{\theta;0,T}$ appearing
in \zcref{eq:keytoy} in terms of the backward noise and solution,
so that one can take the conditional expectation given the future.
This was done through combining the two equations in \zcref{e.toyequation}.
In this example, the unperturbed dynamics is reversible, so when one
subtracts one equation from the other to obtain \zcref{e.BhatB},
the drift does not appear. This is not the case of the open KPZ equation:
for the unperturbed dynamics with $\uu=\vv=0$, we only have the time-reversal
\emph{skew} symmetry, so if we combine the two equations for $u$
and $\hat{u}$ to express $(\partial_{x}\dif W_{t}(x))$ in terms
of $(\partial_{x}\dif\hat{W}_{\hat{t}}(x))$ and the backward solution,
part of the drift inevitably appears. Indeed, this creates the main
technical difficulty that we need to overcome in the paper.
\item As one may not have noticed, what matters in the proof of \zcref{eq:oldkey}
is actually not to compute explicitly the conditional expectation
but only to uncover a martingale structure. That is, we need to show
that after rewriting $\mathcal{Q}_{\theta;0,T}$ in terms of $(\hat{B}_{\hat{t}})$
and $(\hat{X}_{\hat{t}})$, the process
\[
\mathcal{Y}_{\theta}(X_{0})\mathcal{Q}_{\theta;0,T}=\mathcal{Y}_{\theta}(\hat{X}_{T})\e^{\theta\hat{B}_{T}-2\theta(\hat{X}_{T}-\hat{X}_{0})-\tfrac{1}{2}\theta^{2}T}
\]
is a martingale in $T$ in the backward filtration, provided that
$\mathcal{Y}_{\theta}$ is chosen as in \zcref{eq:Ytheta}. In this
case, the martingale is simple enough to be a geometric Brownian motion.
The case of the open KPZ equation is substantially more involved.
In \zcref{eq:key911}, after rewriting $\mathcal{Q}^{\eps}_{\uu,\vv;0,T}$
in terms of the backward noise and solutions, we would obtain a complicated
expression involving $\mathcal{Y}_{\uu,\vv}(\hat{u}_{T})$, the backward
noise $\dif\hat{W}_{t}(x)$, and the backward solution $\hat{u}_{t}(x)$.
The fact that there exists an underlying martingale so that \zcref{eq:key1911}
holds relies, on the one hand, on the form of $\mathcal{Y}_{\uu,\vv}$
in \zcref{eq:Duv-def}, and on the other hand, the ``singular''
behavior of the nonlinear term at the boundary.
\end{enumerate}
Having outlined the main difficulties in implementing the above argument
to prove \zcref{eq:key1911}, we provide further details in the next
section.

\subsubsection{Key steps in the proof}

\label{s.keysteps}In this section, we sketch the main steps in the
proof of \zcref{eq:key1911}. First, to rewrite the change of measure
$\mathcal{Q}^{\eps}_{\uu,\vv;0,T}$ in terms of the backward noise
$\dif\hat{W}_{t}(x)$, it is easier to compare the forward and backward
KPZ equations (rather than the Burgers equation). If we fix $T>0$
and define $\hat{h}_{\hat{t}}=h_{T-\hat{t}}$, then we obtain
\begin{subequations}
\label{e.hhath}
\begin{align*}
\dif h_{t}(x) & =\frac{1}{2}\big(\Delta h_{t}(x)+(\partial_{x}h_{t}(x))^{2}\big)\dif t+\dif W_{t}(x),\\
\dif\hat{h}_{\hat{t}}(x) & =\frac{1}{2}\big(\Delta\hat{h}_{\hat{t}}(x)-(\partial_{x}\hat{h}_{\hat{t}}(x))^{2}\big)\dif\hat{t}+\dif\hat{W}_{\hat{t}}(x).
\end{align*}
\end{subequations}
The above are only formal expressions. Let us, however, ignore this
technical issue at this stage, and proceed by writing the two equations
in their integral forms and subtracting one from the other, as in
\zcref{e.toyequation}. With $\langle\cdot,\cdot\rangle$ denoting
the $L^{2}([0,L])$ inner product, we obtain 
\begin{equation}
\int^{T}_{0}\langle\varphi^{\eps}_{\uu,\vv},\dif W_{t}\rangle=-2\langle\hat{h}_{T},\varphi^{\eps}_{\uu,\vv}\rangle+2\langle\hat{h}_{0},\varphi^{\eps}_{\uu,\vv}\rangle-\int^{T}_{0}\langle(\partial_{x}\hat{h}_{\hat{t}})^{2},\varphi^{\eps}_{\uu,\vv}\rangle\dif\hat{t}+\int^{T}_{0}\langle\varphi^{\eps}_{\uu,\vv},\dif\hat{W}_{t}\rangle,\label{e.917integral}
\end{equation}
which is a much more complicated version of \zcref{e.BhatB} in the
toy example.

Since $\mathcal{Q}^{\eps}_{\uu,\vv;0,T}$ is given by \zcref{eq:defRN},
proving \zcref{thm:mainthm} reduces to showing that 
\begin{equation}
\begin{aligned}\mathcal{Y}_{\uu,\vv}(\hat{u}_{0}) & =\lim_{\eps\to0}\mathbb{E}[\mathcal{Y}_{\uu,\vv}(\hat{u}_{T})\mathcal{Q}^{\eps}_{\uu,\vv;0,T}\mid\hat{u}_{0}]\\
 & =\lim_{\eps\to0}\mathbb{E}\left[\mathcal{Y}_{\uu,\vv}(\hat{u}_{T})\e^{-2\langle\hat{h}_{T},\varphi^{\eps}_{\uu,\vv}\rangle}\e^{2\langle\hat{h}_{0},\varphi^{\eps}_{\uu,\vv}\rangle}\e^{-\int^{T}_{0}\langle(\partial_{x}\hat{h}_{\hat{t}})^{2},\varphi^{\eps}_{\uu,\vv}\rangle\,\dif\hat{t}}\e^{\int^{T}_{0}\langle\varphi^{\eps}_{\uu,\vv},\dif\hat{W}_{\hat{t}}\rangle-\tfrac{1}{2}T\|\varphi^{\eps}_{\uu,\vv}\|^{2}_{L^{2}([0,L])}}\,\bigg|\,\hat{u}_{0}\right].
\end{aligned}
\label{e.9121}
\end{equation}
There are multiple exponential factors inside the expectation on the
right side. The one containing the nonlinear term is the most difficult
to analyze. Recall that $\varphi^{\eps}_{\uu,\vv}$ approximates $-\uu\delta_{0}-\vv\delta_{L}$
in the limit when these objects are viewed as functions/distributions
on $\mathbb{R}/(2L\mathbb{Z})$, so we have
\[
\langle\hat{h}_{\hat{t}},\varphi^{\eps}_{\uu,\vv}\rangle\xrightarrow[\eps\to0]{}-\tfrac{1}{2}\uu\hat{h}_{\hat{t}}(0)-\tfrac{1}{2}\vv\hat{h}_{\hat{t}}(0)\quad\quad\mbox{ for }t=0,T.
\]
The factors of $1/2$ arise because $\langle\cdot,\cdot\rangle$ is
the inner product on $L^{2}([0,L])$, i.e. half of $\mathbb{R}/(2L\mathbb{Z})$.
Taking the limit in the first two exponential factors on the right
side of \zcref{e.9121}, we see that \zcref{e.9121} is equivalent
to 
\begin{equation}
\mathcal{Y}_{\uu,\vv}(\hat{u}_{0})\e^{\uu\hat{h}_{0}(0)+\vv\hat{h}_{0}(L)}=\lim_{\eps\to0}\mathbb{E}\left[\mathcal{Y}_{\uu,\vv}(\hat{u}_{T})\e^{\uu\hat{h}_{T}(0)+\vv\hat{h}_{T}(L)}\e^{-\int^{T}_{0}\langle(\partial_{x}\hat{h}_{\hat{t}})^{2},\varphi^{\eps}_{\uu,\vv}\rangle\,\dif\hat{t}}\hat{\mathcal{Q}}^{\eps}_{\uu,\vv;0,T}\,\bigg|\,\hat{u}_{0}\right],\label{e.9152}
\end{equation}
where we defined the new change of measure $\hat{\mathcal{Q}}^{\eps}_{\uu,\vv;0,T}\coloneqq\e^{\int^{T}_{0}\langle\varphi^{\eps}_{\uu,\vv},\dif\hat{W}_{\hat{t}}\rangle-\tfrac{1}{2}T\|\varphi^{\eps}_{\uu,\vv}\|^{2}_{L^{2}([0,L])}}$.
In the above expression, the conditional expectation is straightforward
to evaluate: given the initial data $\hat{u}_{0}$ for the backward
Burgers equation, we construct $\hat{h}_{0}$ as the initial data
for the backward KPZ equation and solve 
\begin{equation}
\dif\hat{h}_{t}(x)=\frac{1}{2}\big(\Delta\hat{h}_{t}(x)-(\partial_{x}\hat{h}_{t}(x))^{2}\big)\dif t+\dif\hat{W}_{t}(x)+\varphi^{\eps}_{\uu,\vv},\label{e.neweqhath}
\end{equation}
and the goal is to show that
\begin{equation}
\mathcal{Y}_{\uu,\vv}(\hat{u}_{0})\e^{\uu\hat{h}_{0}(0)+\vv\hat{h}_{0}(L)}=\lim_{\eps\to0}\mathbb{E}\left[\mathcal{Y}_{\uu,\vv}(\hat{u}_{T})\e^{\uu\hat{h}_{T}(0)+\vv\hat{h}_{T}(L)}\e^{-\int^{T}_{0}\langle(\partial_{x}\hat{h}_{t})^{2},\varphi^{\eps}_{\uu,\vv}\rangle\dif t}\,\bigg|\,\hat{u}_{0}\right].\label{e.9151}
\end{equation}
In other words, we have replaced the Radon--Nikodym derivative $\hat{\mathcal{Q}}^{\eps}_{\uu,\vv;0,T}$
in \zcref{e.9152} with the supposition that $\hat{h}$ solves the
backward KPZ equation with a perturbed noise: $\dif\hat{W}_{\hat{t}}(x)\mapsto\dif\hat{W}_{\hat{t}}(x)+\varphi^{\eps}_{\uu,\vv}$.
In this way, \zcref{e.9152} reduces to \zcref{e.9151}. The above
equation should be compared to \zcref{eq:oldkey}. The extra term
$\e^{-\int^{T}_{0}\langle(\partial_{x}\hat{h}_{\hat{t}})^{2},\varphi^{\eps}_{\uu,\vv}\rangle\,\dif\hat{t}}$
in \zcref{e.9151} arises from the non-reversibility of the dynamics.

Two results then combine to complete the proof of \zcref{e.9151}.
First, with $\hat{h}$ solving \zcref{e.neweqhath}, we have 
\begin{equation}
\int^{T}_{0}\langle(\partial_{x}\hat{h}_{t})^{2},\varphi^{\eps}_{\uu,\vv}\rangle\,\dif t\xrightarrow[\eps\downarrow0]{\mathrm{law}}\mathcal{N}(\mu T,\sigma^{2}T),\label{e.paininthe}
\end{equation}
with the mean $\mu$ and the variance $\sigma^{2}$ depending explicitly
on $\uu,\vv$. Furthermore, the limiting Gaussian random variable
is independent of everything else, including the backward noise $\dif\hat{W}_{t}(x)$
and the initial data $\hat{u}_{0}$. (The last statement should be
interpreted in terms of a joint convergence in law of the left side
of \zcref{e.paininthe} and the noise.) Second,
\begin{equation}
\text{the process }\left(\mathcal{Y}_{\uu,\vv}(\hat{u}_{T})\e^{\uu\hat{h}_{T}(0)+\vv\hat{h}_{T}(L)}\e^{-\mu T+\tfrac{1}{2}\sigma^{2}T}\right)_{T\geq0}\mbox{ is a martingale in the backward filtration. }\label{e.915ma}
\end{equation}

The factor $\e^{-\mu T+\tfrac{1}{2}\sigma^{2}T}$ in \zcref{e.915ma}
simply comes from the convergence in \zcref{e.paininthe}, together
with taking the expectation of $\e^{-\mathcal{N}(\mu T,\sigma^{2}T)}$.
We view the term $\int^{T}_{0}\langle(\partial_{x}\hat{h}_{t})^{2},\varphi^{\eps}_{\uu,\vv}\rangle\,\dif t$
as the integrated boundary flux, drawing an analogy with the model
of ASEP. Technically, the reason it shows up in our analysis is because
we have treated the boundary condition as a singular boundary potential
and incorporated its effects into the noise through the Cameron--Martin
theorem. From a physical perspective, it is also an extremely interesting
quantity because, if we make the analogue between $\hat{u}=\partial_{x}\hat{h}$
and the particle density in ASEP, the term $\langle(\partial_{x}\hat{h}_{t})^{2},\varphi^{\eps}_{\uu,\vv}\rangle$
is a linear combination of the particle fluxes at the two boundaries.
The above result provides a very detailed description of the particle
behaviors near the boundaries, showing that the time-integrated flux
in a thin boundary layer is approximately a Gaussian random variable,
independent from the ``bulk,'' and the mean and the variance depend
explicitly on the boundary parameters $\uu,\vv$. As a matter of fact,
with some extra effort, one may attempt to show a process-level convergence
to a drifted Brownian motion in \zcref{e.paininthe}.

The proof of the martingale property in \zcref{e.915ma} is a relatively
straightforward application of Itô's formula, provided that we nail
down the precise value of $-\mu+\tfrac{1}{2}\sigma^{2}$. The main
technical difficulty of this whole paper lies in proving \zcref{e.paininthe}.
It is for this latter purpose that we use the theory of regularity
structures.

There are different ways to understand the Gaussianity coming out
of the boundary flux in \zcref{e.paininthe}. On a heuristic level,
one may say that $\hat{h}$ is supposed to satisfy the boundary conditions
$\partial_{x}\hat{h}\approx\uu,-\vv$ near $x\approx0,L$, which implies
that in some sense the low frequency modes of $\hat{h}$ is small
near the boundaries. On the other hand, for the high frequency modes,
the heat semigroup imposes a fast mixing in time which leads to the
Gaussian behavior in \zcref{e.paininthe}. The way we convinced ourselves
in the first place that there could be nontrivial Gaussian fluctuations
arising from the boundary flux was through a formal expansion. For
$\hat{h}$ solving \zcref{e.neweqhath} which is evenly extended to
$[-L,L]$, we study the first few terms in the formal expansion. With
$\mathcal{G}$ representing the solution operator for the standard
heat equation with a periodic boundary condition, we treat the nonlinear
term $-(\partial_{x}\hat{h}_{t}(x))^{2}$ as a perturbation and iterate
the mild formulation of \zcref{e.neweqhath} to obtain
\begin{subequations}
\label{e.formalex}
\begin{equation}
\hat{h}=(\mathcal{G}\dif\hat{W}+\mathcal{G}\varphi^{\eps}_{\uu,\vv})+\tfrac{1}{2}\mathcal{G}[\nabla(\mathcal{G}\dif\hat{W}+\mathcal{G}\varphi^{\eps}_{\uu,\vv})]^{2}+\ldots\label{eq:formalex-hhat}
\end{equation}
and
\begin{equation}
\begin{aligned}|\partial_{x}\hat{h}|^{2} & =|\nabla\mathcal{G}\dif\hat{W}+\nabla\mathcal{G}\varphi^{\eps}_{\uu,\vv}|^{2}+\tfrac{1}{4}|\nabla\mathcal{G}[\nabla(\mathcal{G}\dif\hat{W}+\mathcal{G}\varphi^{\eps}_{\uu,\vv})]^{2}|^{2}\\
 & \qquad+(\nabla\mathcal{G}\dif\hat{W}+\nabla\mathcal{G}\varphi^{\eps}_{\uu,\vv})\cdot\nabla\mathcal{G}[\nabla(\mathcal{G}\dif\hat{W}+\mathcal{G}\varphi^{\eps}_{\uu,\vv})]^{2}+\cdots.
\end{aligned}
\label{eq:formalex-nonlinearity}
\end{equation}
\end{subequations}

Note that in the above expansion of $|\partial_{x}\hat{h}|^{2}$,
we simply treat $\nabla\mathcal{G}\dif\hat{W}+\nabla\mathcal{G}\varphi^{\eps}_{\uu,\vv}$
and $\tfrac{1}{2}\nabla\mathcal{G}[\nabla(\mathcal{G}\dif\hat{W}+\mathcal{G}\varphi^{\eps}_{\uu,\vv})]^{2}$
as the first and the second term in the expansion of $\nabla\hat{h}$,
then we expand the square. One might guess that the term $|\nabla\mathcal{G}\varphi^{\eps}_{\uu,\vv}|^{2}$
is the major deterministic contribution and $|\nabla\mathcal{G}\dif\hat{W}|^{2}$,
interpreted in the Wick sense, is the major random contribution. Indeed,
a preliminary calculation shows that 
\begin{equation}
\int^{T}_{0}\langle|\nabla\mathcal{G}\dif\hat{W}|^{2}(t,\cdot),\varphi^{\eps}_{\uu,\vv}\rangle\dif t\xrightarrow[\eps\downarrow0]{\mathrm{law}}\mathcal{N}(0,\sigma^{2}T),\label{e.cltleading}
\end{equation}
which suggests that \zcref{e.paininthe} might hold. However, a closer
look at the convergence in \zcref{e.cltleading} reveals that the
limiting variance $\sigma^{2}$ depending on the choice of the mollifier
$\varphi^{\eps}_{\uu,\vv}$. Since the process 
\[
(\mathcal{Y}_{\uu,\vv}(\hat{u}_{T})\e^{\uu\hat{h}_{T}(0)+\vv\hat{h}_{T}(L)}\e^{-\mu T+\tfrac{1}{2}\sigma^{2}T})_{T\geq0}
\]
is expected to be a martingale, and the backward solutions $\hat{u},\hat{h}$
do not depend on the mollifier, the dependence of $\sigma^{2}$ on
the details of the mollifier suggests that the convergence in \zcref{e.paininthe}
is likely more subtle than one might have anticipated. Indeed, both
the mean $\mu$ and the variance $\sigma^{2}$ of the limiting Gaussian
depend on the choice of the mollifier $\varphi^{\eps}_{\uu,\vv}$,
but the sum $-\mu+\tfrac{1}{2}\sigma^{2}$ does not. This turns out
to be one of the main puzzles we need to figure out in this paper,
namely, in the expansion \zcref{e.formalex}, which terms contribute
to the mean $\mu$, which terms contribute to the variance $\sigma^{2}$,
and how they combine together so that $\e^{-\mu T+\tfrac{1}{2}\sigma^{2}T}$
is what we need to compensate $\mathcal{Y}_{\uu,\vv}(\hat{u}_{T})\e^{\uu\hat{h}_{T}(0)+\vv\hat{h}_{T}(L)}$
to make a martingale. At the end, we were able to show that $|\nabla\mathcal{G}\dif\hat{W}|^{2}$
is the only contributor to the limiting variance. On the other hand,
$|\nabla\mathcal{G}\varphi^{\eps}_{\uu,\vv}|^{2}$, together with
\emph{three} higher order terms, gives the desired mean; see \zcref{tab:the-terms-1}
below.

Now the problem reduces to justifying the expansion in \zcref{e.formalex}
and in particular to showing that the rest of the infinitely many
terms do not contribute in the convergence of \zcref{e.paininthe}.
Before even trying, one should first realize that \zcref{e.paininthe}
is only a formal expression as it was written and $|\partial_{x}\hat{h}|^{2}$
is merely a symbol appearing in the equation. The fact that only finitely
many terms in the expansion contribute to the integrated boundary
flux is by no means trivial, and, as a matter of fact, it relies crucially
on the symmetry embedded into the dynamics. Recall that after absorbing
the boundary condition into the singular boundary potential, we extended
the solution evenly and periodically so that $\hat{h}$ is even around
both $0$ and $L$. As a result, $\partial_{x}\hat{h}$ is odd around
$0$ and $L$, so if it were an actual continuous function, both $\partial_{x}\hat{h}$
and $|\partial_{x}\hat{h}|^{2}$ would be zero at the boundaries which
makes the integral $\int^{T}_{0}\langle(\partial_{x}\hat{h}_{\hat{t}})^{2},\varphi^{\eps}_{\uu,\vv}\rangle\,\dif\hat{t}$
vanish as $\eps\to0$. Since it is precisely the singularity of $\partial_{x}\hat{h}$
and $|\partial_{x}\hat{h}|^{2}$ that contributes to the integrated
boundary flux, one could imagine that, if the remainder in the formal
expansion \zcref{e.formalex} is a continuous function, the symmetry
may help in the plain way as ``the square of a continuous odd function
is small near the origin,'' using which one may show that the remainder
does not contribute in \zcref{e.paininthe}.

Rigorously justifying this type of expansion is, from the classical
perspective, highly nontrivial, since at some point adding more terms
in the expansion simply stops improving the regularity of the remainder.
This was precisely the obstacle confronted in the development of the
theory of singular SPDE \cite{hairer:2014:theory,gubinelli:imkeller:perkowski:2014:paracontrolled}.
Therefore, in this work, we use the theory of regularity structures
\cite{hairer:2014:theory}, which provides a very precise description
of the local behavior of solutions to singular SPDEs. This is well-suited
to our problem since we are indeed interested in the local behavior
of the solution near the boundary. In particular, it is the local
expansion of $(\partial_{x}\hat{h})^{2}$ that drives the convergence
in \zcref{e.paininthe}.\footnote{This quantity turns out to be far more challenging to study than $\partial_{x}\hat{h}$,
even in the $\uu=\vv=0$ case, for which the time-integral of the
latter object was shown to converge to a constant in \cites[Prop.~3.13]{goncalves:perkowski:simon:2020:derivation}.} The linchpin of the whole theory of regularity structures is the
reconstruction theorem, which on one hand is used to stitch together
many local expansions to form a Schwartz distribution and on the other
hand provides a very precise local expansion of the Schwartz distribution
at hand. It is the second aspect that plays a crucial role in our
analysis. Namely, with a well-developed solution theory for $\hat{h}$
in the framework of regularity structures, we have a local expansion
of $\partial_{x}\hat{h}$ in a space of modeled distributions, and
this leads to a local expansion of $(\partial_{x}\hat{h})^{2}$ in
a (different) space of modeled distributions. Applying the reconstruction
theorem to the modeled distribution corresponding to $(\partial_{x}\hat{h})^{2}$,
gives an estimate on the remainder, which we can show to be small,
provided that the basepoint in the local expansion is chosen to be
at the boundary. In this case, certain Gubinelli derivatives are zero
due to the aforementioned symmetry. It turns out that this means that
only a finite number of explicitly computable terms remain in the
expansion.
\begin{rem}
\label{rem:why-not-add}Given these challenges in interpreting the
nonlinearity $(\partial_{x}\hat{h})^{2}$, the reader may reasonably
ask why, in \zcref{e.917integral}, we chose to subtract rather than
add the forward and backward equations. Indeed, adding the equations
would cancel the challenging nonlinear term, and leave just a similar
time integral of $\Delta\hat{h}$. This strategy has been exploited
extensively (via the so-called \emph{Itô trick}) in the energy solutions
literature; see e.g. \cite{gubinelli:perkowski:2018:energy,goncalves:perkowski:simon:2020:derivation}.
For our purposes, however, it seems better to keep the nonlinearity
term and cancel the Laplacian term. This is because the boundary conditions
do not really tell us anything about $\Delta\hat{h}$, and we do not
believe that this term would have a universal behavior at the boundary
in the same way that $(\partial_{x}\hat{h})^{2}$ does. It seems that,
even though $(\partial_{x}\hat{h})^{2}$ is more challenging, it is
really the term we want to study. See \zcref{rem:whyweneedh} below
for another reflection of this issue.
\end{rem}

We have now outlined all the main ideas of the paper. We will provide
more detailed explanations as we work through the proof.

\subsection{Related work}

In this section we mention a few other related works on the study
of invariant measures for the KPZ-type equations. Broadly speaking,
the questions one may ask fall into two categories: (i) proving that
a certain explicit measure is invariant; and (ii) proving the existence
and uniqueness of the invariant measure, and further studying the
synchronization/one force one solution principle for the associated
random dynamical system. The above two questions are generally separate
from each other since they address different issues and require completely
different techniques. For example, the result presented in this paper
falls into category (i) and is quite different from the existence
and uniqueness results established in \cite{knizel:matetski:arXiv2022:strong,parekh:2022:ergodicity}.

In the case without boundaries and with spacetime white noise, the
Brownian invariance was first established in the seminal work \cite{MR1462228},
using the invariance of i.i.d.\ Bernoulli for ASEP. Several alternative
proofs have since been developed using different discrete approximations
\cites{sasamoto2009superdiffusivity}{funaki2015kpz}{gubinelli:perkowski:2017:kpz}{cannizzaro2018space},
and the main difficulty there lies in justifying the approximation
of the infinite dimensional SPDE by the corresponding finite dimensional
dynamics. Regarding the uniqueness of the invariant measure, the periodic
setting is much easier than the whole space setting, see \cite{hairermattingly,tommaso,gu:quastel:arXiv2024:integration}.
For the whole space case, the recent works \cite{janjigian2022ergodicity,dunlap2024viscous}
provide a complete characterization of invariant measures.

The aforementioned works concern solvable models, in the sense that
these models possess explicit invariant measures. For more general
models---for instance, equations driven by noise that is white in
time but colored in space with an arbitrary covariance function---one
generally does not expect the existence of any explicit invariant
measure. Nevertheless, existence and uniqueness can still be established
in certain cases. In the periodic setting, this essentially follows
from the classical work of Sinai \cite{sinai:1991:two}, while in
the non-compact setting the problem is significantly more difficult,
and we are aware of two works in this direction \cite{bakhtin2019thermodynamic,dunlap2021stationary}.
This line of research is closely related to the study of Busemann
functions and the coalescence of geodesics in the context of last-
and first-passage percolation \cite{janjigian:rassoul-agha:seppalainen:2023:geometry}.

As mentioned earlier, introducing boundary effects greatly complicates
the problem, in the sense that the questions in category (i) becomes
much more difficult. For the half-space KPZ, we refer to \cite{barraquand2023stationary}
for a more or less complete description of all possible invariant
measures (see Conjecture~1.5 there for the set of all extremal stationary
measures). Related works on the study of the invariant measure of
the open KPZ equation on a bounded interval have already been discussed
in the introduction, before the statement of the main result. In
the recent work \cite{contreras_hip:das:zitridis:2025:fluctuation},
the form of the invariant measure was used to study the rate of growth
of the height fluctuations of the open KPZ equation.

\subsection{Outline of the paper}

The conceptual strategy of the paper outlined in \zcref{s.keysteps}
is laid out rigorously in the first three sections of the paper. In
particular, we recall the solution theory of the open KPZ equation,
including the crucial time-reversal property, and introduce the boundary
potentials in \zcref{sec:solntheory}. (Some standard technical pieces
are relegated to \zcref{sec:Basic-properties-of}.) In \zcref{sec:Proof-of-the},
we give the proof of \zcref{thm:mainthm} (essentially a rigorous
version of the strategy already outlined), conditional on two key
statements: \zcref{prop:DtildeMG,prop:boundaryterm} on the behavior
of the semimartingale ``bulk'' term and the singular ``boundary''
term, respectively. The analysis of the semimartingale term is handled
via the Itô formula in \zcref{sec:It=0000F4-formula-and}.

The analysis of the singular boundary term is rather technical and
occupies the remainder of the paper, Sections~\ref{sec:Hairer-theory}
through~\ref{sec:BPHZ}. These sections are devoted to the proof
of \zcref{prop:boundaryterm}, which is indeed the only result of
these sections that is used in the proof of \zcref{thm:mainthm}.
In \zcref{sec:Hairer-theory} we recall the construction of the KPZ
solution using regularity structures from \cite{hairer:2013:solving,hairer:2014:theory,hairer:quastel:2018:class,friz:hairer:2020:course},
with adaptations to our setting of Neumann boundary conditions. In
\zcref{sec:Analysis-of-the}, we explain how to use the theory of
regularity structures to prove \zcref{prop:boundaryterm}, modulo
the finite number of stochastic calculations and estimates that are
necessary both to bound the regularity structure model and to compute
the contributions of these terms on the boundary. Sections~\ref{sec:Explicit-calculations}
and~\ref{sec:BPHZ} then contain the requisite analysis of the finite
number of relevant stochastic terms. In particular, in \zcref{sec:Explicit-calculations},
we compute the nonzero contributions of these terms to the KPZ nonlinearity,
and in \zcref{sec:BPHZ}, we perform the stochastic estimates necessary
to bound the model as well as to show that the remaining contributions
are small. The approach to bounding the model is rather more involved
than the approach used for the open KPZ equation via regularity structures
in \cite{gerencer:hairer:2019:singular} due to the fact that we need
to use the regularity expansion close the boundary; see the discussion
at the beginning of \zcref{sec:BPHZ} for details.

\subsection{Notation}

We let $\mathfrak{s}\coloneqq(2,1)$ denote the parabolic scaling
on $\mathbb{R}^{2}$. For $\alpha\in(0,1)$ and $U\subseteq\mathbb{R}^{2}$,
we define the parabolic Hölder norm
\begin{equation}
\|f\|_{\mathcal{C}^{\alpha}_{\mathfrak{s}}(U)}\coloneqq\sup_{(t,x)\in U}|f_{t}(x)|+\sup_{(t,x),(s,y)\in U}\frac{|f_{t}(x)-f_{s}(y)|}{|t-s|^{\alpha/2}+|x-y|^{\alpha}},\label{eq:parabolic-holder-norm}
\end{equation}
and define $\mathcal{C}^{\alpha}_{\mathfrak{s}}(U)$ as the closure
of $\mathcal{C}^{\infty}(U)$ under this norm. (Note that this is
slightly different from the set of functions on $U$ such that $\|f\|_{\mathcal{C}^{\alpha}_{\mathfrak{s}}(U)}$
is finite, but has the advantage of being separable.) Let
\begin{equation}
\langle f,g\rangle\coloneqq\int^{L}_{0}f(x)g(x)\,\dif x,\label{eq:fgIP}
\end{equation}
a notation that we use liberally whenever $f$ and $g$ are functions/distribution
for which this pairing is well-defined. For functions $f$ and $g$,
we use $f*g$ to denote spatial convolution and $f\circledast g$
to denote space-time convolution.

\subsection{Acknowledgments}

The authors are grateful to Ivan Corwin, Martin Hairer, Jonathan Mattingly,
Nicolas Perkowski, and Guangqu Zheng for helpful discussions and pointers
to the literature at various stages of this project. A.D.\ was partially
supported by the National Science Foundation under grant no.~DMS-2346915.
Part of this work was completed while A.D.\ was in residence at the
Simons Laufer Mathematical Sciences Institute in Berkeley, California,
during the Fall 2025 semester, supported by the National Science Foundation
under grant no.\ DMS-2424139. Y.G. was partially supported by the
National Science Foundation under grant no.~DMS-2203014. T.R.\ was
partially supported by the Leverhulme Trust ECF~2024-543.

\section{Equations, time reversal, and approximations\label{sec:solntheory}}

In this section, we present some preliminary results. Before delving
into details, we provide motivation based on the proof sketch in \zcref{subsec:Our-method}.

Recall that our starting point is the result in \cite{goncalves:perkowski:simon:2020:derivation}
for the case $\uu=\vv=0$. The analysis there was based on the theory
of energy solutions \cite{gubinelli:perkowski:2018:energy} and the
discrete approximation of \zcref{eq:uuv} by WASEP. This is the only
``integrable'' input needed for our approach. More precisely, we
rely on the following two facts from \cite{goncalves:perkowski:simon:2020:derivation}
in the case $\uu=\vv=0$: (i) the white noise is invariant under the
dynamics; (ii) the stationary Markov process $(u_{0,0;t})_{t}$ satisfies
a time-reversal skew symmetry in the sense that for any $T>0$, $(\hat{u}_{0,0;\hat{t}})_{\hat{t}}:=(u_{0,0;T-t})_{t}$
solves the same equation, but with the opposite sign in front of the
nonlinear term. The consideration of time reversal plays a key role
in the development of the energy solution theory. Nevertheless, the
solution of \zcref{eq:huv} and \zcref{eq:uuv} we consider in this
paper is the Hopf-Cole solution, defined through the stochastic heat
equation imposed with a Robin boundary condition interpreted properly.
It is known that the two notions of solutions are not exactly the
same, so we need to keep track of the discrepancy, which is simply
a time-dependent spatial constant. In particular, to use the time
reversal and the skew symmetry, we need to study the forward solution/noise
and the backward solution/noise defined on the \emph{same} probability
space. The first part of this section is dedicated to these issues.

Building on the aforementioned integrable inputs, we aim at studying
the case with general boundary parameters $\uu,\vv\in\mathbb{R}$.
The idea is to consider the singularly perturbed noise $\dif W_{t}\mapsto\dif W_{t}-\uu\delta_{0}-\vv\delta_{L}$
and use the Cameron--Martin theorem to absorb the boundary effects
into the noise. Since the Dirac function is not square integrable,
we introduce an approximation of $-\uu\delta_{0}-\vv\delta_{L}$.
Throughout the paper, $\eps>0$ represents the scale on which the
boundary potential $\varphi^{\eps}_{\uu,\vv}$ approximates $-\uu\delta_{0}-\vv\delta_{L}$.
On the other hand, the noise $\dif W_{t}$ is white in space and time,
which makes the equations \zcref{eq:huv,eq:uuv} singular and thus
to some extent unusable. In order to use the equations in the usual
way so that one may be able to represent the forward noise in terms
of the backward noise and solution, we introduce another parameter
$\zeta>0$, representing the spatial scale on which we mollify the
noise $\dif W_{t}\mapsto\dif W^{\zeta}_{t}$. \zcref{subsec:approx-inhomog,sec:mollifying}
are devoted to justifying these approximations.

\subsection{Mild and energy solutions}

Let $(\dif W_{t})_{t}$ be a space-time white noise on a probability
space $(\Omega,\mathcal{F},\mathbb{P})$ generating a natural filtration
$\{\mathscr{F}_{t}\}_{t}$, and let $h_{0}$ be a standard Brownian
motion on $[0,L]$, with $h_{0}(0)=0$, independent of the noise $(\dif W_{t})_{t}$.
For boundary parameters $\uu,\vv\in\mathbb{R}$, we define $Z_{\uu,\vv;t}$
to be the mild solution to the stochastic heat equation
\begin{subequations}
\label{Z-Robin}
\begin{align}
\dif Z_{\uu,\vv;t}(x) & =\oh\Delta Z_{\uu,\vv;t}(x)\dif t+Z_{\uu,\vv;t}(x)\dif W_{t}(x), &  & t\in(0,T],x\in(0,L);\label{eq:ZPDE-Robin}\\
\partial_{x}Z_{\uu,\vv;t}(0) & =(\uu-\oh)Z_{\uu,\vv;t}(0)\quad\text{and}\quad\partial_{x}Z_{\uu,\vv;t}(L)=-\left(\vv-\oh\right)Z_{\uu,\vv;t}(L), &  & t\in(0,T];\label{eq:Zbc-Robin}\\
Z_{\uu,\vv;0}(x) & =\e^{h_{0}(x)}, &  & x\in(0,L),\label{eq:Zic-Robin}
\end{align}
\end{subequations}
where the Robin boundary conditions \zcref{eq:Zbc-Robin} are interpreted
through the Robin heat kernel in the mild formulation as in \cites[Defn.~4.1]{parekh:2019:KPZ}.
Define
\begin{equation}
h_{\uu,\vv;t}\coloneqq\log Z_{\uu,\vv;t}\label{eq:hfromZ}
\end{equation}
and
\begin{equation}
u_{\uu,\vv;t}\coloneqq\partial_{x}\log Z_{\uu,\vv;t}=\partial_{x}h_{\uu,\vv;t}.\label{eq:uthtderiv}
\end{equation}
We abbreviate 
\begin{equation}
Z_{t}\coloneqq Z_{0,0;t},\qquad h_{t}\coloneqq h_{0,0;t}=\log Z_{t},\qquad\text{and}\qquad u_{t}\coloneqq u_{0,0;t}=\partial_{x}h_{t}=\partial_{x}\log Z_{t},\label{eq:abbrv}
\end{equation}
and define
\begin{equation}
\tilde{h}_{t}\coloneqq h_{t}+\frac{t}{24}.\label{eq:httildedef}
\end{equation}
The definition of $h_{t}$ is consistent with the use of $h_{0}$
in \zcref{eq:Zic-Robin}, and moreover
\begin{equation}
u_{\uu,\vv;0}=\partial_{x}\log Z_{\uu,\vv;0}=u_{0}\qquad\text{for all }\uu,\vv\in\mathbb{R}.\label{eq:ic-u0}
\end{equation}
We note for later use that 
\begin{equation}
\{u_{t}\}_{t\in[0,T]}\text{ is independent of }\{\langle W_{t}-W_{0},1\rangle\}_{t\in[0,T]}.\label{eq:ut_independent_of_Wtflat}
\end{equation}

As mentioned already, our analysis will be based on time reversal
for the $\uu=\vv=0$ problem, which was studied in the context of
\emph{almost stationary energy solutions} to the open KPZ equation
with homogeneous boundary conditions in \cite{goncalves:perkowski:simon:2020:derivation}.
Here the term ``almost'' refers to the fact that, for each $t\ge0$,
$h_{t}$ is a standard Brownian motion only modulo a height shift,
i.e.\ only after $h_{t}(0)$ is subtracted. In order to use the energy
solution theory, we first must know that our solutions actually are
stationary energy solutions to the KPZ equation.  Following \cites[(2.4–5)]{goncalves:perkowski:simon:2020:derivation},
define
\begin{equation}
\mathcal{S}_{\mathrm{Dir}}\coloneqq\left\{ \varphi\in\mathcal{C}^{\infty}([0,L])\st\varphi^{(2k)}(0)=\varphi^{(2k)}(L)=0\text{ for all }k=0,1,2,\ldots\right\} \label{eq:SDir}
\end{equation}
and
\begin{equation}
\mathcal{S}_{\mathrm{Neu}}\coloneqq\left\{ \varphi\in\mathcal{C}^{\infty}([0,L])\st\varphi^{(2k+1)}(0)=\varphi^{(2k+1)}(L)=0\text{ for all }k=0,1,2,\ldots\right\} .\label{eq:SNeu}
\end{equation}
We also define a discretization of the gradient at scale $\kappa$,
as in \cites[(3.9)]{goncalves:perkowski:simon:2020:derivation},
\begin{equation}
\nabla_{\kappa}f(x)=\begin{cases}
\kappa^{-1}(f(x+\kappa)-f(x)), & x\in[0,L-2\kappa);\\
\kappa^{-1}(f(x)-f(x-\kappa)), & x\in[L-2\kappa,L].
\end{cases}\label{eq:deltakappadef}
\end{equation}
The following is the main result of this section. It states that the
Cole--Hopf solution to KPZ, shifted by $t/24$, has the law of the
energy solution. (This somewhat strange phrasing is because, strictly
speaking, the notion of energy solution as considered in \cites{goncalves:perkowski:simon:2020:derivation}
is defined only as the \emph{law} of a process.)
\begin{prop}
\label{prop:The-process-}The process $(\tilde{h}_{t})_{t\in[0,T]}$
is an almost stationary energy solution to the open KPZ equation
\[
\dif\tilde{h}_{t}=\frac{1}{2}\Delta\tilde{h}_{t}+\frac{1}{2}(\partial_{x}\tilde{h}_{t})^{2}+\dif W_{t}
\]
with homogeneous Neumann boundary conditions on $[0,L]$ and initial
data $h_{0}$, in the sense of \cites[Thm.~3.7]{goncalves:perkowski:simon:2020:derivation}.
In particular, for each $\varphi\in\mathcal{S}_{\mathrm{Neu}}$ and
each $s,t\in[0,T]$, the limit
\begin{equation}
\mathcal{B}_{s,t}(\varphi)\coloneqq\lim_{\kappa\to0}\mathcal{B}^{[\kappa]}_{s,t}(\varphi)\label{eq:Bstdef}
\end{equation}
exists in $L^{2}(\Omega)$, where
\begin{equation}
\mathcal{B}^{[\kappa]}_{s,t}(\varphi)\coloneqq\int^{t}_{s}\int^{L}_{0}\varphi(x)\left\{ (\nabla_{\kappa}\tilde{h}_{r}(x))^{2}-\frac{1}{\kappa}\right\} \,\dif x\,\dif r=\int^{t}_{s}\int^{L}_{0}\varphi(x)\left\{ (\nabla_{\kappa}h_{r}(x))^{2}-\frac{1}{\kappa}\right\} \,\dif x\,\dif r.\label{eq:Bkappastdef}
\end{equation}
Furthermore, for any $\varphi\in\mathcal{S}_{\mathrm{Neu}}$ and $0\le s<t\le T$,
we have the integral form
\begin{align}
\langle W_{t}-W_{s},\varphi\rangle & =\langle\tilde{h}_{t},\varphi\rangle-\langle\tilde{h}_{s},\varphi\rangle-\frac{1}{2}\int^{t}_{s}\langle\tilde{h}_{r},\Delta\varphi\rangle\,\dif s-\frac{1}{2}\mathcal{B}_{s,t}(\varphi)\label{eq:Wintermsoftildeh}\\
\ovset{\zcref{eq:httildedef}} & =\langle h_{t}-h_{s}+\nicefrac{1}{24}(t-s),\varphi\rangle-\frac{1}{2}\int^{t}_{s}\langle h_{r},\Delta\varphi\rangle\,\dif r-\frac{1}{2}\mathcal{B}_{s,t}(\varphi).\label{eq:noiserel}
\end{align}
Moreover, the process $(u_{t})_{t\in[0,T]}$ is a stationary energy
solution to the stochastic Burgers equation with homogeneous Dirichlet
boundary conditions on $[0,L]$ in the sense of \cites[Thm.~3.3]{goncalves:perkowski:simon:2020:derivation}. 
\end{prop}

\begin{proof}
Let $(k_{t})_{t\in[0,T]}$ be an almost stationary energy solution
to the KPZ equation with homogeneous Neumann boundary conditions on
$[0,L]$ with initial data $k_{0}(\cdot)$ having the law of a standard
Brownian motion with $k_{0}(0)=0$. The definition of almost stationary
energy solution only defines the law of the process $(k_{t})_{t\ge0}$,
so we are free to choose the coupling with the random variables already
defined. In particular, we can choose the coupling such that $k_{0}=h_{0}=\tilde{h}_{0}$
and the process $(\dif W_{t})_{t}$ in \zcref{Z-Robin} is the same
as the process $(\dif W_{t})_{t}$ constructed from $(k_{t})_{t}$
by \zcref{eq:Wtdef} below. By \zcref{prop:cole-hopf} below, we
see that $\e^{k_{t}-t/24}$ is a mild solution to the open stochastic
heat equation \zcref{Z-Robin}. By the strong uniqueness of solutions
to this equation, this means that we in fact must have $\e^{k_{t}-t/24}=Z_{t}$,
and hence $k_{t}=\log Z_{t}+t/24\overset{\zcref{eq:httildedef}}{=}\tilde{h}_{t}$.
Therefore, $\tilde{h}_{t}$ is (has the law of) an almost stationary
energy solution, which implies by definition that \zcref{eq:Bstdef}
holds, and \zcref{eq:Wintermsoftildeh} is a consequence of our coupling.
The last claim of the proposition follows immediately from \zcref{eq:uthtderiv}
and \cites[Prop.~3.7(2)]{goncalves:perkowski:simon:2020:derivation}.
\end{proof}

\begin{rem}
As one may have noticed, in the case $\uu=\vv=0$, the KPZ equation
\zcref{eq:huv} is subject to homogeneous boundary conditions, whereas
the corresponding stochastic heat equation satisfies the Robin boundary
condition \zcref{eq:Zbc-Robin}, with the “extra’’ coefficients $-\tfrac{1}{2}$
and $\tfrac{1}{2}$ at $x=0$ and $x=L$, respectively. This phenomenon
can be interpreted as a form of boundary renormalization; see, for
example, the discussions in \cite[Section 3.5]{goncalves:perkowski:simon:2020:derivation}
and \cite{gerencer:hairer:2019:singular}, as well as the computations
in \zcref{sec:mollifying} below.
\end{rem}

From \cites[Thm.~3.3(1)]{goncalves:perkowski:simon:2020:derivation},
we know that
\begin{equation}
\Law(u_{t})\text{ does not depend on }t,\label{eq:utstationary}
\end{equation}
and indeed, for each $t$, this law is that of a spatial white noise.
Moreover, according to \cites[Thm.~3.3(2–3)]{goncalves:perkowski:simon:2020:derivation},
for each $\varphi\in\mathcal{S}_{\mathrm{Dir}}$ and each $s,t\in[0,T]$,
there exists a limit in $L^{2}(\Omega)$
\begin{equation}
\mathcal{A}_{s,t}(\varphi)\coloneqq\lim_{\kappa\to0}\mathcal{A}^{\kappa}_{s,t}(\varphi),\qquad\text{where }\mathcal{A}^{\kappa}_{s,t}(\varphi)\coloneqq-\int^{t}_{s}\int^{L}_{0}\partial_{x}\varphi(x)\langle\iota^{\kappa}_{x},u_{r}\rangle^{2}\,\dif x\,\dif r,\label{eq:Astdef}
\end{equation}
where we have defined as in \cite[Defn.~3.2]{goncalves:perkowski:simon:2020:derivation}
\[
\iota^{\kappa}_{x}(y)\coloneqq\begin{cases}
\kappa^{-1}\mathbf{1}_{(x,x+\kappa]}(y), & x\in[0,L-2\kappa);\\
\kappa^{-1}\mathbf{1}_{[x-\kappa,x)}(y)), & x\in[L-2\kappa,L].
\end{cases}
\]
Comparing \cite[(3.6) and (3.10)]{goncalves:perkowski:simon:2020:derivation},
we see that, for $\varphi\in\mathcal{S}_{\mathrm{Dir}}$, we have
\begin{equation}
\mathcal{A}_{s,t}(\varphi)=-\mathcal{B}_{s,t}(\partial_{x}\varphi).\label{eq:ABreln}
\end{equation}

\subsection{Time reversal\label{subsec:Time-reversal}}

A crucial property of stationary energy solutions is that the time-reversed
solution to the homogeneous Dirichlet stochastic Burgers equation
has the same law as the forward solution, except with the opposite
sign of the nonlinearity. This is stated in \cites[Thm.~3.3(4)]{goncalves:perkowski:simon:2020:derivation}.
For our analysis, however, we need to consider the stochastic Burgers
solution and its time reversal simultaneously as solutions to stochastic
PDEs on the same probability space. This requires some understanding
the space-time white noise with respect to which the time reversal
satisfies an equation. Throughout the paper, we distinguish time-reversed
quantities by decorating them with the ``$\hat{\phantom{\bullet}}$''
symbol. In particular, for time-reversed quantities, we will use the
time variable $\hat{t}=T-t$. 

For $\hat{t}\in[0,T]$, define
\begin{equation}
\hat{u}_{\hat{t}}=u_{T-\hat{t}},\label{eq:uhatdef}
\end{equation}
and let $\{\hat{\mathscr{G}}_{\hat{t}}\}_{\hat{t}}$ to be the filtration
given by
\[
\hat{\mathscr{G}}_{\hat{t}}=\sigma\left(\langle\hat{u}_{\hat{s}},\varphi\rangle\st\varphi\in\mathcal{S}_{\mathrm{Dir}},\hat{s}\in[0,\hat{t}]\right).
\]
According to \cites[Thm.~3.3(4)]{goncalves:perkowski:simon:2020:derivation},
if we define
\begin{equation}
\hat{\mathcal{A}}_{\hat{s},\hat{t}}(\varphi)=\mathcal{A}_{T-\hat{t},T-\hat{s}}(\varphi)\qquad\text{for all }\varphi\in\mathcal{S}_{\mathrm{Dir}},\label{eq:Ahatdef}
\end{equation}
then for each $\varphi\in\mathcal{S}_{\mathrm{Dir}}$, the process
\begin{equation}
\hat{\mathcal{M}}_{\hat{t}}(\varphi)\coloneqq\langle\hat{u}_{\hat{t}},\varphi\rangle-\langle\hat{u}_{0},\varphi\rangle-\frac{1}{2}\int^{\hat{t}}_{0}\langle\hat{u}_{\hat{s}},\Delta\varphi\rangle\,\dif s+\frac{1}{2}\hat{\mathcal{A}}_{0,\hat{t}}(\varphi)\label{eq:Mhatdef}
\end{equation}
is a continuous $\{\hat{\mathscr{G}}_{\hat{t}}\}$-martingale with
quadratic variation
\[
[\hat{\mathcal{M}}(\varphi)]_{\hat{t}}=\hat{t}\|\partial_{x}\varphi\|^{2}_{L^{2}([0,L])}.
\]
This expression of the quadratic variation implies that, thought of
as a space-time distribution, the process $(\dif\hat{\mathcal{M}}_{\hat{t}})$
has the distribution of the spatial derivative of a space-time white
noise.

For our purposes, we need to write the process $(\dif\hat{\mathcal{M}}_{\hat{t}})$
as the spatial derivative of a \emph{particular} space-time white
noise, which essentially amounts to the choice of a zero-frequency
mode for the space-time white noise. While in principle this choice
could be arbitrary, it will be very useful for our applications to
choose this zero-frequency mode carefully. First, we first note that
for any $\varphi\in\mathcal{S}_{\mathrm{Neu}}$, there exist unique
elements $\tilde{\varphi}\in\mathcal{S}_{\mathrm{Dir}}$ and $\overline{\varphi}\in\mathbb{R}$
such that 
\begin{equation}
\varphi=-\partial_{x}\tilde{\varphi}+\overline{\varphi}\label{eq:phidecomp}
\end{equation}
and hence
\begin{equation}
\|\varphi\|^{2}_{L^{2}([0,L])}=\|\partial_{x}\tilde{\varphi}\|^{2}_{L^{2}([0,L])}+|\overline{\varphi}|^{2}L.\label{eq:normdecomp}
\end{equation}
Now we define 
\[
\hat{\mathscr{F}}_{\hat{t}}=\hat{\mathscr{G}}_{\hat{t}}\vee\sigma\left(\langle W_{T-\hat{s}},1\rangle-\langle W_{T},1\rangle\st\hat{s}\in[0,\hat{t}]\right)
\]
and
\begin{equation}
\hat{\mathcal{N}}_{\hat{t}}(\varphi)\coloneqq\hat{\mathcal{M}}_{\hat{t}}(\tilde{\varphi})+\langle W_{T-\hat{t}},\overline{\varphi}\rangle-\langle W_{T},\overline{\varphi}\rangle.\label{eq:Nhatdef}
\end{equation}
Using \zcref{eq:ut_independent_of_Wtflat}, we see that $(\hat{\mathcal{M}}_{\hat{t}}(\tilde{\varphi}))_{\hat{t}\in[0,T]}$
and $(\langle W_{T-\hat{t}},\overline{\varphi}\rangle-\langle W_{T},\overline{\varphi}\rangle)_{\hat{t}\in[0,T]}$
are independent $\{\hat{\mathscr{F}}_{\hat{t}}\}$-martingales, and
in particular they are both Brownian motions. Therefore, $(\hat{\mathcal{N}}_{\hat{t}}(\varphi))_{\hat{t}}$
is also an $\{\hat{\mathscr{F}}_{\hat{t}}\}$-martingale (and a Brownian
motion) with quadratic variation
\begin{equation}
[\hat{\mathcal{N}}(\varphi)]_{\hat{t}}=\hat{t}\left(\|\partial_{x}\tilde{\varphi}\|_{L^{2}([0,L])}+|\overline{\varphi}|^{2}L\right)^{\oh}\overset{\zcref{eq:normdecomp}}{=}\hat{t}\|\varphi\|^{2}_{L^{2}([0,L])}.\label{eq:NhatQV}
\end{equation}
This means that there is a space-time white noise $(\dif\hat{W}_{\hat{t}})_{\hat{t}}$,
adapted to the filtration $\{\hat{\mathscr{F}}_{\hat{t}}\}$, such
that
\begin{equation}
\langle\hat{W}_{\hat{t}}-\hat{W}_{\hat{s}},\varphi\rangle=\hat{\mathcal{N}}_{\hat{t}}(\varphi)-\hat{\mathcal{N}}_{\hat{s}}(\varphi)\qquad\text{for all }\hat{s},\hat{t}\in[0,T].\label{eq:Whatdef}
\end{equation}

Now we define the initial data for the backward KPZ equation
\begin{equation}
\hat{h}_{0}(x)=\int^{x}_{0}\hat{u}_{0}(y)\,\dif y,\qquad x\in[0,L].\label{eq:h0hatdef}
\end{equation}
Then, for $\hat{t}\in[0,T]$, we define the solution to the backward
KPZ equation
\begin{equation}
\hat{h}_{\hat{t}}(x)=h_{T-\hat{t}}(x)+\hat{h}_{0}(x)-h_{T}(x).\label{eq:hhatdef-new}
\end{equation}
We note that 
\[
\partial_{x}\hat{h}_{0}(x)\overset{\zcref{eq:h0hatdef}}{=}\hat{u}_{0}(x)\overset{\zcref{eq:uhatdef}}{=}u_{T}(x)\overset{\zcref{eq:abbrv}}{=}\partial_{x}h_{T}(x),
\]
so from \zcref{eq:hhatdef-new} we have for all $\hat{t}\in[0,T]$
that
\begin{equation}
\partial_{x}\hat{h}_{\hat{t}}(x)=\hat{u}_{\hat{t}}(x)=u_{T-\hat{t}}(x)=\partial_{x}h_{T-\hat{t}}(x).\label{eq:htderivtimerev}
\end{equation}
It is worth emphasizing that $\hat{u}$ is the exact time reversal
of $u$, as defined in \zcref{eq:uhatdef}, while $\hat{h}$ is not
the time reversal of $h$, due to the additional term $\hat{h}_{0}(x)-h_{T}(x)$
in \zcref{eq:h0hatdef}. By definition, this extra term is a (random)
constant, chosen so that the initial data $\hat{h}_{0}$ for $\hat{h}$
is a standard Brownian motion with $\hat{h}_{0}(0)=0$.

Continuing from \zcref{eq:htderivtimerev}, we see that if we define
\begin{align*}
\hat{\mathcal{B}}^{[\kappa]}_{\hat{s},\hat{t}}(\varphi) & \coloneqq\int^{\hat{t}}_{\hat{s}}\int^{L}_{0}\varphi(x)\left\{ (\nabla_{\kappa}\hat{h}_{\hat{r}}(x))^{2}-\frac{1}{\kappa}\right\} \,\dif x\,\dif\hat{r}\overset{\zcref{eq:htderivtimerev}}{=}\int^{\hat{t}}_{\hat{s}}\int^{L}_{0}\varphi(x)\left\{ (\nabla_{\kappa}h_{T-\hat{r}}(x))^{2}-\frac{1}{\kappa}\right\} \,\dif x\,\dif\hat{r}\\
 & =\int^{T-\hat{s}}_{T-\hat{t}}\int^{L}_{0}\varphi(x)\left\{ (\nabla_{\kappa}h_{r}(x))^{2}-\frac{1}{\kappa}\right\} \,\dif x\,\dif r\overset{\zcref{eq:Bkappastdef}}{=}\mathcal{B}^{[\kappa]}_{T-\hat{t},T-\hat{s}}(\varphi),
\end{align*}
then by \zcref{eq:Bstdef} we have the limit
\begin{equation}
\hat{\mathcal{B}}_{\hat{s},\hat{t}}(\varphi)\coloneqq\lim_{\kappa\to0}\hat{\mathcal{B}}^{[\kappa]}_{\hat{s},\hat{t}}(\varphi)\qquad\text{in the }L^{2}\text{ sense}.\label{eq:Bhatstdef}
\end{equation}
Using this in \zcref{eq:Bhatstdef} and comparing with \zcref{eq:Bstdef},
we see that in fact
\begin{equation}
\hat{\mathcal{B}}_{\hat{s},\hat{t}}(\varphi)=\mathcal{B}_{T-\hat{t},T-\hat{s}}(\varphi).\label{eq:Bssame}
\end{equation}
We note that, for $\varphi\in\mathcal{S}_{\mathrm{Dir}}$, we have
\begin{equation}
-\hat{\mathcal{B}}_{\hat{s},\hat{t}}(\partial_{x}\varphi)\overset{\zcref{eq:Bssame}}{=}\mathcal{B}_{T-\hat{t},T-\hat{s}}(-\partial_{x}\varphi)\overset{\zcref{eq:ABreln}}{=}\hat{\mathcal{A}}_{T-\hat{t},T-\hat{s}}(\varphi)\overset{\zcref{eq:Ahatdef}}{=}\mathcal{A}_{\hat{t},\hat{s}}(\varphi).\label{eq:ABhat}
\end{equation}

The following lemma shows that the process $\hat{h}$ defined in \zcref{eq:hhatdef-new},
after a proper shift, is indeed an almost energy solution to the backward
KPZ equation. Consequently, an integral formulation exists for the
backward equation. In \zcref{prop:noise-relation} below, we will
combine this integral formulation with \zcref{eq:noiserel} to express
the forward noise in terms of the backward noise and the corresponding
solutions. 
\begin{lem}
\label{lem:hhat-as}The process $(\hat{h}_{\hat{t}}-\hat{t}/24)_{\hat{t}\in[0,T]}$
is an almost stationary energy solution for the open KPZ equation
\[
\dif\tilde{h}_{t}=\frac{1}{2}\Delta\tilde{h}_{t}-\frac{1}{2}(\partial_{x}\tilde{h}_{t})^{2}+\dif W_{t}
\]
with homogeneous Neumann boundary conditions, in the sense of \cites[Thm.~3.7]{goncalves:perkowski:simon:2020:derivation}.
In particular, for any $\varphi\in\mathcal{S}_{\mathrm{Neu}}$, we
have
\begin{equation}
\hat{\mathcal{N}}_{\hat{t}}(\varphi)=\langle\varphi,\hat{h}_{\hat{t}}-\hat{t}/24\rangle-\langle\varphi,\hat{h}_{0}\rangle-\frac{1}{2}\int^{\hat{t}}_{0}\langle\Delta\varphi,\hat{h}_{\hat{s}}\rangle\,\dif\hat{s}+\frac{1}{2}\hat{\mathcal{B}}_{0,\hat{t}}(\varphi),\label{eq:Ntexpr}
\end{equation}
with $\hat{\mathcal{N}}_{\hat{t}}(\varphi)$ defined in \zcref{eq:Nhatdef}. 
\end{lem}

\begin{proof}
Since $(\hat{u}_{\hat{t}})$ is a stationary energy solution for homogeneous
open Burgers by definition, the only thing that needs to be checked
is that, for any $\varphi\in\mathcal{S}_{\mathrm{Neu}}$, the process
\[
\hat{t}\mapsto\langle\varphi,\hat{h}_{\hat{t}}-\hat{t}/24\rangle-\langle\varphi,\hat{h}_{0}\rangle-\frac{1}{2}\int^{\hat{t}}_{0}\langle\Delta\varphi,\hat{h}_{s}\rangle\,\dif s+\frac{1}{2}\hat{\mathcal{B}}_{0,\hat{t}}(\varphi)
\]
is a martingale with respect to the filtration generated by the process
$(\hat{h}_{\hat{t}})_{\hat{t}}$, and that the quadratic variation
process is given by $\hat{t}\mapsto\hat{t}\|\varphi\|^{2}_{L^{2}([0,L])}$.
The adaptedness is clear, so it remains to check the martingale property.
Decompose $\varphi=-\partial_{x}\tilde{\varphi}+\overline{\varphi}$
with $\tilde{\varphi}\in\mathcal{S}_{\mathrm{Dir}}$ and $\overline{\varphi}\in\mathbb{R}$
as in \zcref{eq:phidecomp}. Then we can write
\begin{align}
\langle\varphi, & \hat{h}_{\hat{t}}-\hat{t}/24\rangle-\langle\varphi,\hat{h}_{0}\rangle-\frac{1}{2}\int^{\hat{t}}_{0}\langle\Delta\varphi,\hat{h}_{s}\rangle\,\dif s+\frac{1}{2}\hat{\mathcal{B}}_{0,\hat{t}}(\varphi)\nonumber \\
 & =\langle-\partial_{x}\tilde{\varphi}+\overline{\varphi},\hat{h}_{\hat{t}}-\hat{t}/24\rangle-\langle-\partial_{x}\tilde{\varphi}+\overline{\varphi},\hat{h}_{0}\rangle-\frac{1}{2}\int^{\hat{t}}_{0}\langle\Delta(-\partial_{x}\tilde{\varphi}),\hat{h}_{s}\rangle\,\dif s-\frac{1}{2}\hat{\mathcal{B}}_{0,\hat{t}}(\partial_{x}\tilde{\varphi})+\frac{1}{2}\hat{\mathcal{B}}_{0,\hat{t}}(\overline{\varphi})\nonumber \\
\ovset{\zcref{eq:ABhat}} & =\langle\tilde{\varphi},\hat{u}_{\hat{t}}\rangle+\langle\overline{\varphi},\hat{h}_{\hat{t}}-\hat{t}/24\rangle-\langle\tilde{\varphi},\hat{u}_{0}\rangle-\langle\overline{\varphi},\hat{h}_{0}\rangle-\frac{1}{2}\int^{\hat{t}}_{0}\langle\Delta\tilde{\varphi},\hat{u}_{s}\rangle\,\dif s+\frac{1}{2}\hat{\mathcal{A}}_{0,\hat{t}}(\tilde{\varphi})+\frac{1}{2}\hat{\mathcal{B}}_{0,\hat{t}}(\overline{\varphi})\nonumber \\
\ovset{\zcref{eq:Mhatdef}} & =\hat{\mathcal{M}}_{\hat{t}}(\tilde{\varphi})+\langle\overline{\varphi},\hat{h}_{\hat{t}}-\hat{h}_{0}-\hat{t}/24\rangle+\frac{1}{2}\hat{\mathcal{B}}_{0,\hat{t}}(\overline{\varphi}).\label{eq:firstexpansion}
\end{align}
Using \zcref{eq:noiserel}, we derive
\begin{align*}
\langle W_{T},\overline{\varphi}\rangle-\langle W_{T-\hat{t}},\overline{\varphi}\rangle & =\langle h_{T}-h_{T-\hat{t}}+\hat{t}/24,\overline{\varphi}\rangle-\frac{1}{2}\mathcal{B}_{T-\hat{t},T}(\overline{\varphi})=\langle\hat{h}_{0}-\hat{h}_{\hat{t}}+\hat{t}/24,\overline{\varphi}\rangle-\frac{1}{2}\hat{\mathcal{B}}_{0,\hat{t}}(\overline{\varphi}),
\end{align*}
with the second identity by \zcref{eq:hhatdef-new,eq:Bssame}. Further
using this in \zcref{eq:firstexpansion}, we obtain
\[
\langle\varphi,\hat{h}_{\hat{t}}-\hat{t}/24\rangle-\langle\varphi,\hat{h}_{0}\rangle-\frac{1}{2}\int^{\hat{t}}_{0}\langle\Delta\varphi,\hat{h}_{\hat{s}}\rangle\,\dif\hat{s}+\frac{1}{2}\hat{\mathcal{B}}_{0,\hat{t}}(\varphi)=\hat{\mathcal{M}}_{\hat{t}}(\tilde{\varphi})+\langle W_{T-\hat{t}},\overline{\varphi}\rangle-\langle W_{T},\overline{\varphi}\rangle\overset{\zcref{eq:Nhatdef}}{=}\hat{\mathcal{N}}_{\hat{t}}(\varphi),
\]
which is \zcref{eq:Ntexpr}, and is indeed a martingale. The correct
form of the quadratic variation process is verified by \zcref{eq:NhatQV}.
\end{proof}

We note at this stage that it follows immediately from \zcref{eq:Ntexpr,eq:Whatdef}
that, for any $\varphi\in\mathcal{S}_{\mathrm{Neu}}$, we have
\begin{align}
\langle\hat{W}_{\hat{t}}-\hat{W}_{\hat{s}},\varphi\rangle & =\langle\hat{h}_{\hat{t}}-\hat{h}_{\hat{s}}-(\hat{t}-\hat{s})/24,\varphi\rangle-\frac{1}{2}\int^{\hat{t}}_{\hat{s}}\langle\hat{h}_{\hat{r}},\Delta\varphi\rangle\,\dif\hat{r}+\frac{1}{2}\hat{\mathcal{B}}_{\hat{s},\hat{t}}(\varphi).\label{eq:noiserel-rev}
\end{align}

Now that we have constructed the backward noise $(\dif\hat{W}_{\hat{t}})_{\hat{t}}$,
we use it to construct mild solutions to the stochastic heat equation.
For $\uu,\vv\in\mathbb{R}$, let the process $(\hat{Z}_{\uu,\vv;\hat{t}})_{\hat{t}\in[0,T]}$
be the mild solution to the stochastic heat equation
\begin{subequations}
\label{eq:Zproblem-Robin-1}
\begin{align}
\dif\hat{Z}_{\uu,\vv;\hat{t}}(x) & =\frac{1}{2}\Delta\hat{Z}_{\uu,\vv;\hat{t}}(x)\dif\hat{t}-\hat{Z}_{\uu,\vv;\hat{t}}(x)\dif\hat{W}_{\hat{t}}(x), &  & \hat{t}>0,x\in(0,L);\label{eq:ZPDE-Robin-1}\\
\partial_{x}\hat{Z}_{\uu,\vv;\hat{t}}(0) & =-(\uu+\oh)\hat{Z}_{\uu,\vv;\hat{t}}(0)\qquad\text{and}\qquad\partial_{x}\hat{Z}_{\uu,\vv;\hat{t}}(L)=(\vv+\oh)\hat{Z}_{\hat{t}}(L), &  & t>0;\label{eq:Zbc-left-Robin-1}\\
\hat{Z}_{\uu,\vv;0}(x) & =\e^{-\int^{x}_{0}\hat{u}_{0}(y)\,\dif y}, &  & x\in(0,L),\label{eq:Zic-Robin-1}
\end{align}
\end{subequations}
again in the sense of \cites[Defn.~4.1]{parekh:2019:KPZ}. This stochastic
heat equation differs from that appearing in \zcref{Z-Robin} in that
the forward noise $(\dif W_{t})$ is replaced by the backward noise
$(\dif\hat{W}_{\hat{t}})$ and that some signs are changed. Similar
to the above, we abbreviate
\[
\hat{Z}_{\hat{t}}=\hat{Z}_{0,0;\hat{t}}.
\]
By \zcref{prop:cole-hopf,lem:hhat-as}, $\e^{-(\hat{h}_{\hat{t}}(x)-\hat{t}/24)-\hat{t}/24}=\e^{-\hat{h}_{\hat{t}}(x)}$
is a mild solution to \zcref{eq:Zproblem-Robin-1}, and thus by the
uniqueness of mild solutions we have
\begin{equation}
\hat{h}_{\hat{t}}=-\log\hat{Z}_{\hat{t}},\qquad\hat{t}\in[0,T].\label{eq:uhatZhat}
\end{equation}

\begin{rem}
The reader may find it slightly strange that in \zcref{eq:Zic-Robin-1},
the initial condition for the time-reversed problem is constructed
using $\hat{u}_{0}\overset{\zcref{eq:uhatdef}}{=}u_{T}\overset{\zcref{eq:abbrv}}{=}u_{0,0;T}$
rather than $u_{\uu,\vv;T}$. This is indeed essential, since $\hat{u}_{0}$
is independent of $(\dif\hat{W}_{\hat{t}})_{\hat{t}\in[0,T]}$, and
so we can consider an adapted solution to \zcref{eq:Zproblem-Robin-1}.
On the other hand, we have no reason to expect that $u_{\uu,\vv;T}$
is independent of $(\dif\hat{W}_{\hat{t}})_{\hat{t}\in[0,T]}$.
\end{rem}

The following proposition provides a rigorous counterpart of \zcref{e.917integral},
which is crucial for us to take the conditional expectation given
the ``future'' in the proof of the main theorem. It comes as a consequence
of \zcref{eq:noiserel,eq:noiserel-rev} and establishes a formula
to relate the forward and backward noises, the solutions, and the
nonlinear terms appearing in the KPZ equation. 
\begin{prop}
\label{prop:noise-relation}We have, for all $\varphi\in\mathcal{S}_{\mathrm{Neu}}$,
that
\begin{align}
\langle\varphi,W_{T}-W_{0}-(\hat{W}_{T}-\hat{W}_{0})\rangle & =-2\langle\varphi,\hat{h}_{T}-\hat{h}_{0}-T/24\rangle-\hat{\mathcal{B}}_{0,T}(\varphi)\label{eq:noise-relation}\\
 & =2\langle\varphi,h_{T}-h_{0}+T/24\rangle-\mathcal{B}_{0,T}(\varphi).\label{eq:noise-relation-fwd}
\end{align}
\end{prop}

\begin{proof}
Taking $t=T$ and $s=0$ in \zcref{eq:noiserel}, we see that
\begin{align}
\langle\varphi,W_{T}-W_{0}\rangle & =\langle\varphi,h_{T}-h_{0}+T/24\rangle-\frac{1}{2}\int^{T}_{0}\langle\Delta\varphi,h_{s}\rangle\,\dif s-\frac{1}{2}\mathcal{B}_{0,T}(\varphi)\nonumber \\
\ovset{\zcref{eq:Bssame}} & =\langle\varphi,h_{T}-h_{0}+T/24\rangle-\frac{1}{2}\int^{T}_{0}\langle\Delta\varphi,h_{s}\rangle\,\dif s-\frac{1}{2}\hat{\mathcal{B}}_{0,T}(\varphi)\label{eq:noise-eqn-T}
\end{align}
Also, taking $\hat{t}=T$ in \zcref{eq:noiserel-rev} and using \zcref{eq:htderivtimerev},
we get 
\begin{align}
\langle\varphi,\hat{W}_{T}-\hat{W}_{0}\rangle & =\langle\varphi,\hat{h}_{T}-T/24\rangle-\langle\varphi,\hat{h}_{0}\rangle-\frac{1}{2}\int^{T}_{0}\langle\Delta\varphi,h_{s}\rangle\,\dif s+\frac{1}{2}\hat{\mathcal{B}}_{0,T}(\varphi).\label{eq:thatequalsT}
\end{align}
Now we subtract \zcref{eq:thatequalsT} from \zcref{eq:noise-eqn-T}
to obtain
\begin{align*}
\langle\varphi,W_{T}-W_{0}-(\hat{W}_{T}-\hat{W}_{0})\rangle & =\left\langle \varphi,h_{T}-h_{0}-(\hat{h}_{T}-\hat{h}_{0})+T/12\right\rangle -\hat{\mathcal{B}}_{0,T}(\varphi)\\
\ovset{\zcref{eq:hhatdef-new}} & =\left\langle \varphi,-2(\hat{h}_{T}-\hat{h}_{0})+T/12\right\rangle -\hat{\mathcal{B}}_{0,T}(\varphi),
\end{align*}
which is \zcref{eq:noise-relation}. The identity \zcref{eq:noise-relation-fwd}
in terms of the forward processes then follows from \zcref{eq:Bssame}
and \zcref{eq:hhatdef-new}.
\end{proof}

\subsection{The boundary potential approximation and change of measure\label{subsec:approx-inhomog}}

The previous two sections (aside from the definitions \zcref[comp, range]{Z-Robin,eq:uthtderiv})
concerned the case of homogeneous boundary condition $\uu=\vv=0$.
To deal with general $\uu,\vv\in\mathbb{R}$, first recall that on
a formal level, the inhomogeneous Neumann boundary condition is equivalent
with adding a singular boundary potential $-\uu\delta_{0}-\vv\delta_{L}$
to the equation and then studying the periodic case. With the change
of noise $\dif W_{t}\mapsto\dif W_{t}-\uu\delta_{0}-\vv\delta_{L}$,
it is natural to employ the Cameron--Martin theorem and consider
a change of measure. For this approach to work, we need the shift
to live in the Cameron--Martin space associated with the spacetime
white noise $\dif W$, which is the $L^{2}$ space. The goal of this
section is to approximate the singular boundary potential $-\uu\delta_{0}-\vv\delta_{L}$
in order to make such a change of measure possible.

\subsubsection{Equations with boundary potentials}

We start by defining a boundary potential. Some of the notation we
introduce here is not needed in the present section but is helpful
for consistency with the mollifiers that appear later in \zcref{sec:mollifying}.
Fix $\psi\in\mathcal{C}^{\infty}(\mathbb{R})$ such that 
\begin{equation}
\supp\psi\subset(-3/4,-1/4)\cup(1/4,3/4),\qquad\psi\ge0,\qquad\psi(\cdot)=\psi(-\cdot),\qquad\text{and}\qquad\int_{\mathbb{R}}\psi(x)\,\dif x=1.\label{eq:psiproperties}
\end{equation}
For $\eps>0$, define
\begin{equation}
\psi^{\eps}(x)=\eps^{-1}\psi(\eps^{-1}x).\label{eq:psiepsdef}
\end{equation}
The above assumptions imply that
\begin{equation}
\int^{\infty}_{0}\psi^{\eps}(x)\,\dif x=\int^{0}_{-\infty}\psi^{\eps}(x)\,\dif x=\frac{1}{2}.\label{eq:psiepshalflineint}
\end{equation}
For $M\in(0,\infty)$, we define
\begin{equation}
\Sh_{M}=\sum_{q\in M\mathbb{Z}}\delta_{q},\label{eq:Shdef}
\end{equation}
the sum of a Dirac delta distribution at each element of $M\mathbb{Z}$.
(The notation is chosen because the sequence of delta functions resembles
a repeating sequence of Cyrillic capital letter ``$\Sh$''s.) Now
we define, for $\eps\in(0,L^{-1})$,
\begin{equation}
\varphi^{\eps}_{\uu,\vv}(x)=-\uu(\Sh_{2L}*\psi^{\eps})(x)-\vv(\Sh_{2L}*\psi^{\eps})(x-L),\qquad x\in\mathbb{R}.\label{eq:varphiuvdef}
\end{equation}
The choice of the support for $\psi$ is to make sure that the boundary
potential $\varphi^{\eps}_{\uu,\vv}$ lives in each of the test function
spaces $\mathcal{S}_{\mathrm{Dir}}$ and $\mathcal{S}_{\mathrm{Neu}}$.
We also note in particular that
\begin{equation}
\begin{aligned}\langle\varphi^{\eps}_{\uu,\vv},1\rangle & =-\uu\int^{L}_{0}(\Sh_{2L}*\psi^{\eps})(x)\,\dif x-\vv\int^{0}_{-L}(\Sh_{2L}*\psi^{\eps})(x-L)\,\dif x\\
 & =-\uu\int^{\infty}_{0}\psi^{\eps}(x)\,\dif x-\vv\int^{0}_{-\infty}\psi^{\eps}(x)\,\dif x\overset{\zcref{eq:psiepshalflineint}}{=}-\frac{1}{2}(\uu+\vv),
\end{aligned}
\label{eq:phiint}
\end{equation}
where we used that $\eps<L^{-1}$ in the second identity, and also
for future use that (by \zcref{eq:psiepshalflineint,eq:varphiuvdef})
\begin{equation}
\lim_{\eps\downarrow0}\int^{x}_{0}\varphi^{\eps}_{\uu,\vv}(y)\,\dif x=-\frac{\uu}{2}\qquad\text{for any }x\in(0,L).\label{eq:convtodelta}
\end{equation}
It is also worth noting that, when restricted to $[0,L]$, the boundary
potential $\varphi^{\eps}_{\uu,\vv}$ approximates $-\tfrac{1}{2}\uu\delta_{0}-\tfrac{1}{2}\vv\delta_{L}$.
As we perform an even and periodic extension later, it approximates
$-\uu\delta_{0}-\vv\delta_{L}$ on the circle $\mathbb{R}/(2L\mathbb{Z})$.

Consider the mild solution to the stochastic heat equation
\begin{subequations}
\label{eq:Zeps}
\begin{align}
\dif Z^{\eps}_{\uu,\vv;t}(x) & =\oh\Delta Z^{\eps}_{\uu,\vv;t}(x)\,\dif t+\varphi^{\eps}_{\uu,\vv}(x)Z^{\eps}_{\uu,\vv;t}(x)\,\dif t+Z_{\uu,\vv;t}(x)\dif W_{t}(x), &  & t>0,x\in(0,L);\label{eq:Zepsuv-eqn}\\
\partial_{x}Z^{\eps}_{\uu,\vv;t}(0) & =-\oh Z^{\eps}_{\uu,\vv;t}(0)\qquad\text{and}\qquad\partial_{x}Z^{\eps}_{\uu,\vv;t}(L)=\oh Z^{\eps}_{\uu,\vv;t}(L), &  & t>0;\label{eq:Zepsuv-Robin}\\
Z^{\eps}_{\uu,\vv;0}(x) & =\e^{h_{0}(x)}, &  & x\in(0,L).\label{eq:Zepsuv-ic}
\end{align}
\end{subequations}
Compared to \zcref{Z-Robin}, the potential $\varphi^{\eps}_{\uu,\vv}$
in the above equation replaces the part of the boundary condition
coming from the $\uu,\vv$ parameters. Define
\begin{equation}
h^{\eps}_{\uu,\vv;t}\coloneqq\log Z^{\eps}_{\uu,\vv;t}\qquad\text{and}\qquad u^{\eps}_{\uu,\vv;t}\coloneqq\partial_{x}h^{\eps}_{\uu,\vv;t}.\label{eq:huepsuvt}
\end{equation}
Since $\varphi^{\eps}_{0,0}\equiv0$, the equations \zcref{Z-Robin,eq:Zeps}
match when $\uu=\vv=0$, and indeed we have
\begin{equation}
Z^{\eps}_{0,0;t}=Z_{0,0;t},\qquad h^{\eps}_{0,0;t}=h_{0,0;t},\qquad u^{\eps}_{0,0;t}=u_{0,0;t}\qquad\text{for any }\eps>0.\label{eq:zeroepsdoesntmatter}
\end{equation}

Similarly, we can consider the solution to the backward stochastic
heat equation
\begin{subequations}
\label{eq:Zhateps}
\begin{align}
\dif\hat{Z}^{\eps}_{\uu,\vv;\hat{t}}(x) & =\oh\Delta\hat{Z}^{\eps}_{\uu,\vv;\hat{t}}(x)\,\dif\hat{t}-\varphi^{\eps}_{\uu,\vv}(x)\hat{Z}^{\eps}_{\uu,\vv;\hat{t}}(x)\,\dif\hat{t}-Z_{\uu,\vv;\hat{t}}(x)\dif\hat{W}_{\hat{t}}(x), &  & \hat{t}>0,x\in(0,L);\label{eq:Zhatepsuv-eqn}\\
\partial_{x}\hat{Z}^{\eps}_{\uu,\vv;\hat{t}}(0) & =-\oh\hat{Z}^{\eps}_{\uu,\vv;\hat{t}}(0)\qquad\text{and}\qquad\partial_{x}\hat{Z}^{\eps}_{\uu,\vv;\hat{t}}(L)=\oh\hat{Z}^{\eps}_{\uu,\vv;\hat{t}}(L), &  & \hat{t}>0;\label{eq:Zhatepsuv-Robin}\\
\hat{Z}^{\eps}_{\uu,\vv;0}(x) & =\e^{-\int^{x}_{0}\hat{u}_{0}(y)\,\dif y}, &  & x\in(0,L),\label{eq:Zhatepsuv-ic}
\end{align}
\end{subequations}
and define
\begin{equation}
\hat{h}^{\eps}_{\uu,\vv;\hat{t}}\coloneqq-\log Z^{\eps}_{\uu,\vv;\hat{t}}\qquad\text{and}\qquad\hat{u}^{\eps}_{\uu,\vv;\hat{t}}\coloneqq\partial_{x}\hat{h}^{\eps}_{\uu,\vv;\hat{t}}.\label{eq:huhatepsuvt}
\end{equation}

The next proposition states that, as $\eps\downarrow0$, the boundary
potential $\varphi^{\eps}_{\uu,\vv}$ indeed plays the role of the
inhomogeneous Neumann boundary condition for the KPZ equation.
\begin{prop}
\label{prop:epstozeroconv}We have
\begin{equation}
\lim_{\eps\to0}\sup_{\substack{x\in[0,L]\\
t\in[0,T]
}
}|h^{\eps}_{\uu,\vv;t}-h_{\uu,\vv;t}|=0\qquad\text{and}\qquad\lim_{\eps\to0}\sup_{\substack{x\in[0,L]\\
\hat{t}\in[0,T]
}
}|\hat{h}^{\eps}_{\uu,\vv;\hat{t}}-\hat{h}_{\uu,\vv;\hat{t}}|=0\qquad\text{in probability.}\label{eq:bcconverges}
\end{equation}
In addition, for any $\chi\in(0,\oh)$ and $T>0$, the sequence $(\|h^{\eps}_{\uu,\vv}\|_{\mathcal{C}^{\chi}_{\mathfrak{s}}([0,T]\times[0,L])})_{\eps>0}$
is bounded in probability.
\end{prop}

The proof of the above result relies on a corresponding approximation
result for the stochastic heat equation, combined with uniform positive
and negative moment estimates for $Z^{\eps}_{\uu,\vv}$. Since this
is a fairly classical argument, we defer the proof to \zcref{sec:Basic-properties-of}.

\subsubsection{Cameron--Martin changes of measure}

With the smooth boundary potential $\varphi^{\eps}_{\uu,\vv}$, we
change the underlying probability measure so that the law of $\dif W_{t}$
becomes that of $\dif W_{t}+\varphi^{\eps}_{\uu,\vv}$, and thus it
is enough for us to consider the homogeneous problem. More precisely,
by the Cameron--Martin theorem, the laws of the pairs $\left(Z^{\eps}_{\uu,\vv;t},\dif W_{t}+\varphi^{\eps}_{\uu,\vv}\right)_{t\in[0,T]}$
and $\left(Z_{t},\dif W_{t}\right)_{t\in[0,T]}$ are absolutely continuous
with respect to one another, with Radon--Nikodym derivative taking
the form
\begin{equation}
\frac{\dif\Law\left(\left(Z^{\eps}_{\uu,\vv;t},\dif W_{t}+\varphi^{\eps}_{\uu,\vv}\right)_{t\in[0,T]}\right)}{\dif\Law\left(\left(Z_{t},\dif W_{t}\right)_{t\in[0,T]}\right)}=\mathcal{Q}^{\eps}_{\uu,\vv;0,T}.\label{eq:apply-CM}
\end{equation}
Here we have defined
\begin{equation}
\mathcal{Q}^{\eps}_{\uu,\vv;0,T}\coloneqq\exp\left\{ \langle\varphi^{\eps}_{\uu,\vv},W_{T}-W_{0}\rangle-\frac{1}{2}T\|\varphi^{\eps}_{\uu,\vv}\|^{2}_{L^{2}([0,L])}\right\} ,\label{eq:Qepsdef}
\end{equation}
recalling that $\langle\cdot,\cdot\rangle$ is the pairing extending
the inner product on $L^{2}([0,L])$. Similarly, the laws of the processes
$\left(\hat{Z}^{\eps}_{\uu,\vv;\hat{t}},\dif\hat{W}_{\hat{t}}+\varphi^{\eps}_{\uu,\vv}\right)_{\hat{t}\in[0,T]}$
and $\left(\hat{Z}_{\hat{t}},\dif\hat{W}_{\hat{t}}\right)_{\hat{t}\in[0,T]}$
are absolutely continuous with Radon--Nikodym derivative
\begin{equation}
\frac{\dif\Law\left(\left(\hat{Z}^{\eps}_{\uu,\vv;\hat{t}},\dif\hat{W}_{\hat{t}}+\varphi^{\eps}_{\uu,\vv}\right)_{\hat{t}\in[0,T]}\right)}{\dif\Law\left(\left(\hat{Z}_{\hat{t}},\dif\hat{W}_{\hat{t}}\right)_{\hat{t}\in[0,T]}\right)}=\hat{\mathcal{Q}}^{\eps}_{\uu,\vv;0,T},\label{eq:apply-CM-rev}
\end{equation}
with
\begin{equation}
\hat{\mathcal{Q}}^{\eps}_{\uu,\vv;0,T}\coloneqq\exp\left\{ \langle\varphi^{\eps}_{\uu,\vv},\hat{W}_{T}-\hat{W}_{0}\rangle-\frac{1}{2}T\|\varphi^{\eps}_{\uu,\vv}\|^{2}_{L^{2}}\right\} .\label{eq:Qepshatdef}
\end{equation}

In particular, we have
\begin{equation}
\mathbb{E}\left[\mathcal{Q}^{\eps}_{\uu,\vv;0,T}\right]=\mathbb{E}\left[\hat{\mathcal{Q}}^{\eps}_{\uu,\vv;0,T}\right]=1.\label{eq:Qshavemean1}
\end{equation}
In the time reversal argument sketched in \zcref{subsec:Our-method},
it is crucial to rewrite the forward Radon-Nikodym derivative $\mathcal{Q}^{\eps}_{\uu,\vv;0,T}$
in terms of the backward solution and noise. Here is a key relationship
between $\mathcal{Q}^{\eps}_{\uu,\vv;0,T}$ and $\hat{\mathcal{Q}}^{\eps}_{\uu,\vv;0,T}$:
using \zcref{eq:noise-relation-fwd} in \zcref{eq:Qepsdef}, we have
\begin{align}
\mathcal{Q}^{\eps}_{\uu,\vv;0,T} & =\exp\left\{ \langle\varphi^{\eps}_{\uu,\vv},\hat{W}_{T}-\hat{W}_{0}\rangle+2\langle\varphi^{\eps}_{\uu,\vv},h_{T}-h_{0}+T/24\rangle-\mathcal{B}_{0,T}(\varphi^{\eps}_{\uu,\vv})-\frac{1}{2}T\|\varphi^{\eps}_{\uu,\vv}\|^{2}_{L^{2}}\right\} \nonumber \\
\ovset{\zcref{eq:Qepshatdef}} & =\hat{\mathcal{Q}}^{\eps}_{\uu,\vv;0,T}\exp\left\{ 2\langle\varphi^{\eps}_{\uu,\vv},h_{T}-h_{0}+T/24\rangle-\mathcal{B}_{0,T}(\varphi^{\eps}_{\uu,\vv})\right\} \label{eq:Qepsrel-fwd}\\
 & =\hat{\mathcal{Q}}^{\eps}_{\uu,\vv;0,T}\exp\left\{ -2\langle\varphi^{\eps}_{\uu,\vv},\hat{h}_{T}-\hat{h}_{0}-T/24\rangle-\hat{\mathcal{B}}_{0,T}(\varphi^{\eps}_{\uu,\vv})\right\} ,\label{eq:Qepsrel-rev}
\end{align}
where in the last identity we used \zcref{eq:Bssame} and \zcref{eq:hhatdef-new}.
This relation will be used in the proof of the main theorem when we
take the conditional expectation given the future. One should try
to draw the connection between the above expression and what appeared
on the right side of \eqref{e.9152}, where the formal term $\int^{T}_{0}\langle(\partial_{x}\hat{h}_{\hat{t}})^{2},\varphi^{\eps}_{\uu,\vv}\rangle\,\dif\hat{t}$
is now replaced by $\hat{\mathcal{B}}_{0,T}(\varphi^{\eps}_{\uu,\vv})$
which was defined through the energy solution theory.

With the change of measure, whatever statistical quantities of $Z^{\eps}_{\uu,\vv},h^{\eps}_{\uu,\vv},u^{\eps}_{\uu,\vv}$
one may want to calculate, it reduces to the case $\uu=\vv=0$, as
long as we include the Radon-Nikodym derivative $\mathcal{Q}^{\eps}_{\uu,\vv;0,T}$.
The same applies to backward quantities. In particular, we have 
\begin{lem}
\label{lem:l.conZeps}For any $F\in\mathcal{C}_{\mathrm{b}}(\mathcal{C}([0,T]\times[0,L]))$
we have
\begin{equation}
\lim_{\eps\to0}\mathbb{E}\left[\mathcal{Q}^{\eps}_{\uu,\vv;0,T}F\left((Z_{t})_{t\in[0,T]}\right)\right]=\mathbb{E}\left[F\left((Z_{\uu,\vv;t})_{t\in[0,T]}\right)\right]\label{eq:bcconverges-CM}
\end{equation}
and
\begin{equation}
\lim_{\eps\to0}\mathbb{E}\left[\hat{\mathcal{Q}}^{\eps}_{\uu,\vv;0,T}F\left((\hat{Z}_{\hat{t}})_{\hat{t}\in[0,T]}\right)\right]=\mathbb{E}\left[F\left((\hat{Z}_{\uu,\vv;\hat{t}})_{\hat{t}\in[0,T]}\right)\right].\label{eq:bcconverges-CM-rev}
\end{equation}
\end{lem}

The proof of the above result relies on a standard approximation result
for the stochastic heat equation, namely $Z^{\eps}_{\uu,\vv;t}\to Z_{\uu,\vv;t}$
as $\eps\downarrow0$. We defer the proof to \zcref{subsec:Convergence-as-:}.

On a different note, the nonlinear terms appearing in \zcref{eq:Qepsrel-fwd}
and \zcref{eq:Qepsrel-rev} are defined for $h$ and $\hat{h}$. It
is convenient for us to consider the same quantity with $h,\hat{h}$
replaced by $h^{\eps},\hat{h}^{\eps}$ respectively. Similarly to
\zcref{eq:Bkappastdef} and \zcref{eq:Bhatstdef}, for any $\varphi\in\mathcal{S}_{\mathrm{Neu}}$,
define
\begin{equation}
\mathcal{B}^{\eps,[\kappa]}_{\uu,\vv;s,t}(\varphi)\coloneqq\int^{t}_{s}\int^{L}_{0}\left\{ (\nabla_{\kappa}h^{\eps}_{\uu,\vv;r}(x))^{2}-\frac{1}{\kappa}\right\} \varphi(x)\,\dif x\,\dif r\label{eq:Bepskappadef}
\end{equation}
 and
\begin{equation}
\hat{\mathcal{B}}^{\eps,[\kappa]}_{\uu,\vv;\hat{s},\hat{t}}(\varphi)\coloneqq\int^{\hat{t}}_{\hat{s}}\int^{L}_{0}\left\{ (\nabla_{\kappa}\hat{h}^{\eps}_{\uu,\vv;\hat{r}}(x))^{2}-\frac{1}{\kappa}\right\} \varphi(x)\,\dif x\,\dif\hat{r}.\label{eq:Bhatepskappadef}
\end{equation}
For any $p\in[1,2)$, we have
\begin{align*}
\mathbb{E}\left|\mathcal{B}^{\eps,[\kappa_{1}]}_{\uu,\vv;s,t}(\varphi)-\mathcal{B}^{\eps,[\kappa_{2}]}_{\uu,\vv;s,t}(\varphi)\right|^{p}\ovset{\zcref{eq:apply-CM}} & =\mathbb{E}\left[\mathcal{Q}^{\eps}_{\uu,\vv;0,T}\left|\mathcal{B}^{[\kappa_{1}]}_{s,t}(\varphi)-\mathcal{B}^{[\kappa_{2}]}_{s,t}(\varphi)\right|^{p}\right]\\
 & \le\left(\mathbb{E}[(\mathcal{Q}^{\eps}_{\uu,\vv;0,T})^{\frac{2}{2-p}}]\right)^{1-p/2}\left(\mathbb{E}\left|\mathcal{B}^{[\kappa_{1}]}_{s,t}(\varphi)-\mathcal{B}^{[\kappa_{2}]}_{s,t}(\varphi)\right|^{2}\right)^{p/2},
\end{align*}
and, for any fixed $\eps>0$, the right side goes to $0$ as $\kappa_{1},\kappa_{2}\downarrow0$
by \zcref{eq:Bstdef}. Therefore, the limit
\begin{equation}
\mathcal{B}^{\eps}_{\uu,\vv;s,t}(\varphi)\coloneqq\lim_{\kappa\downarrow0}\mathcal{B}^{\eps,[\kappa]}_{\uu,\vv;s,t}(\varphi)\qquad\text{exists in }L^{p}(\Omega)\text{ for }p\in[1,2).\label{eq:Bepslimit}
\end{equation}
Similarly, the limit
\begin{equation}
\hat{\mathcal{B}}^{\eps}_{\uu,\vv;\hat{s},\hat{t}}(\varphi)\coloneqq\lim_{\kappa\downarrow0}\hat{\mathcal{B}}^{\eps,[\kappa]}_{\uu,\vv;\hat{s},\hat{t}}(\varphi)\qquad\text{exists in }L^{p}(\Omega)\text{ for }p\in[1,2).\label{eq:Bepshatlimit}
\end{equation}

The following lemma is an ``inhomogeneous'' version of \eqref{eq:noiserel},
which we will use later in the proof.
\begin{lem}
We have, for any $f\in\mathcal{S}_{\mathrm{Neu}}$,
\begin{align}
\mathcal{B}^{\eps}_{\uu,\vv;0,T}(f) & =\langle2(h^{\eps}_{\uu,\vv;T}-h^{\eps}_{\uu,\vv;0})+T/12,f\rangle-\int^{T}_{0}\langle h^{\eps}_{\uu,\vv;r},\Delta f\rangle\,\dif r-2\langle W_{T}-W_{0}+T\varphi^{\eps}_{\uu,\vv},f\rangle\qquad\text{a.s.}\label{eq:Bepsnoiserel}
\end{align}
\end{lem}

\begin{proof}
We have by \zcref{eq:noiserel} that
\[
\mathbb{P}\left(\mathcal{B}_{0,T}(f)=\langle2(h_{T}-h_{0})+T/12,f\rangle-\int^{T}_{0}\langle h_{r},\Delta f\rangle\,\dif r-2\langle W_{T}-W_{0},f\rangle\right)=1,
\]
which means that
\begin{align*}
1 & =\mathbb{E}\left[\mathcal{Q}^{\eps}_{\uu,\vv;0,T};\mathcal{B}_{0,T}(f)=\langle2(h_{T}-h_{0})+T/12,f\rangle-\int^{T}_{0}\langle h_{r},\Delta f\rangle\,\dif r-2\langle W_{T}-W_{0},f\rangle\right]\\
\ovset{\zcref{eq:apply-CM}} & =\mathbb{P}\left[\mathcal{B}_{0,T}(\varphi)=\langle2(h^{\eps}_{\uu,\vv;T}-h^{\eps}_{\uu,\vv;0})+T/12,f\rangle-\int^{T}_{0}\langle h^{\eps}_{\uu,\vv;r},\Delta f\rangle\,\dif r-2\langle W_{T}-W_{0}+T\varphi^{\eps}_{\uu,\vv},f\rangle\right],
\end{align*}
i.e. \zcref{eq:Bepsnoiserel} holds.
\end{proof}

For future use, we record the following symmetry in law under time
reversal. The symmetry requires reversing the signs of the boundary
parameters.
\begin{prop}
\label{prop:symmetry-timerev}We have
\begin{equation}
(Z^{\eps}_{\uu,\vv;t},h^{\eps}_{\uu,\vv;t},u^{\eps}_{\uu,\vv;t},\mathcal{B}^{\eps}_{\uu,\vv;0,t})_{t\in[0,T]}\overset{\mathrm{law}}{=}(\hat{Z}^{\eps}_{-\uu,-\vv;\hat{t}},-\hat{h}^{\eps}_{-\uu,-\vv;\hat{t}},-\hat{u}^{\eps}_{-\uu,-\vv;\hat{t}},\hat{\mathcal{B}}^{\eps}_{-\uu,-\vv;0,\hat{t}})_{\hat{t}\in[0,T]}\label{eq:epsmatchinlaw}
\end{equation}
and
\begin{equation}
(Z_{\uu,\vv;t},h_{\uu,\vv;t},u_{\uu,\vv;t})_{t\in[0,T]}\overset{\mathrm{law}}{=}(\hat{Z}_{-\uu,-\vv;\hat{t}},-\hat{h}_{-\uu,-\vv;\hat{t}},-\hat{u}_{-\uu,-\vv;\hat{t}})_{\hat{t}\in[0,T]}.\label{eq:eps0matchinlaw}
\end{equation}
\end{prop}

\begin{proof}
To see that $(Z^{\eps}_{\uu,\vv;t})_{t\in[0,T]}\overset{\mathrm{law}}{=}(\hat{Z}^{\eps}_{-\uu,-\vv;\hat{t}})_{\hat{t}\in[0,T]}$,
we simply compare \zcref{eq:Zeps} and \zcref{eq:Zhateps}, using
the facts that $\varphi^{\eps}_{\uu,\vv}=-\varphi^{\eps}_{-\uu,-\vv}$,
that $(\dif W_{t})_{t\in[0,T]}\overset{\mathrm{law}}{=}(-\dif\hat{W}_{\hat{t}})_{\hat{t}\in[0,T]}$,
and that $(h_{0}(x))_{x\in[0,L]}$ has the same law as $(-\int^{x}_{0}\hat{u}_{0}(y)\,\dif y)_{x}$
since both processes are standard Brownian motions starting at $0$.
This yields \zcref{eq:epsmatchinlaw}. The equivalence in law stated
in \zcref{eq:eps0matchinlaw} then follows from the convergence statements
in \zcref{prop:epstozeroconv}.
\end{proof}

\subsection{Noise mollification\label{sec:mollifying}}

Throughout our analysis, it will often be convenient to work with
equations with mollified noise, since this will allow us to take the
derivative with respect to the spatial variable and obtain a continuous
function. In this section we will introduce the mollified-noise problems
and show that they approximate the white-noise problems as one removes
the mollification. As these results are fairly standard, we defer
the proofs to the appendix.

Since our equations are posed on a finite domain, the notion of mollification
is slightly more subtle than in cases without boundary conditions,
in particular in terms of how the noise is treated near the boundary.
To address this issue, before introducing the mollification, we first
extend the noises $(\dif W_{t})$ and $(\dif\hat{W}_{\hat{t}})$ to
all of $\mathbb{R}$, first by reflecting them about the origin to
$[-L,L]$, and then by taking the $2L$-periodic extension. This means
that for any $f\in L^{2}(\mathbb{R})$, we have
\begin{equation}
\int^{t}_{s}\int_{\mathbb{R}}f(x)\,\dif W_{r}(x)=\int^{t}_{s}\int^{L}_{0}\sum_{q\in2L\mathbb{Z}}\left(f(x+q)+f(-x+q)\right)\,\dif W_{r}(x).\label{eq:extendthenoise}
\end{equation}
We perform a similar extension for $(\dif\hat{W}_{\hat{t}})$. Thus,
formally, we have
\[
\mathbb{E}[\dif W_{t}(x)\dif W_{t'}(x')]=\delta(t-t')\sum_{q\in2L\mathbb{Z}}[\delta(x-x'+q)+\delta(x+x'+q)]\overset{\zcref{eq:Shdef}}{=}\delta(t-t')\left[\Sh_{2L}(x-x')+\Sh_{2L}(x+x')\right],
\]
and similarly for $(\dif\hat{W}_{\hat{t}})$.

Now we introduce our choice of mollifier. Fix $\rho\in\mathcal{C}^{\infty}(\mathbb{R})$
such that 
\begin{equation}
\supp\rho\subset(-\oh,\oh),\qquad\rho\ge0,\qquad\rho(\cdot)=\rho(-\cdot),\qquad\text{and}\qquad\int_{\mathbb{R}}\rho(x)\,\dif x=1.\label{eq:rhoproperties}
\end{equation}
For $\zeta>0$, define
\begin{equation}
\rho^{\zeta}(x)\coloneqq\zeta^{-1}\rho(\zeta^{-1}x).\label{eq:rhoepsdef}
\end{equation}
We also define
\begin{equation}
R(x)\coloneqq\rho^{*2}(x)\qquad\text{and}\qquad R^{\zeta}(x)\coloneqq\rho^{\zeta}*\rho^{\zeta}(x)\overset{\zcref{eq:rhoepsdef}}{=}\zeta^{-1}R(\zeta^{-1}x).\label{eq:RRepsdef}
\end{equation}
Recalling the definition \zcref{eq:Shdef}, for $M>0$ and $x\in\mathbb{R}$,
we also define
\begin{equation}
\Sh^{\zeta}_{M}(x)=\Sh_{M}*R^{\zeta}(x)=\sum_{q\in M\mathbb{Z}}R^{\zeta}(x+q).\label{eq:Shepsdef}
\end{equation}
We note that
\begin{equation}
\Sh^{\zeta}_{M}(2x)\overset{\zcref{eq:Shepsdef}}{=}\sum_{q\in M\mathbb{Z}}R^{\zeta}(2x+q)=\sum_{q\in\frac{M}{2}\mathbb{Z}}R^{\zeta}(2(x+q))\overset{\zcref{eq:RRepsdef}}{=}\frac{1}{2}\sum_{q\in\frac{M}{2}\mathbb{Z}}(2/\zeta)R((2/\zeta)(x+q))\overset{\zcref{eq:Shepsdef}}{=}\frac{1}{2}\Sh^{\zeta/2}_{M/2}(x).\label{eq:Shrescale}
\end{equation}

Using the mollifier $\rho$ and its rescaling defined in \zcref[range]{eq:rhoproperties,eq:rhoepsdef}--\zcref{eq:rhoepsdef},
we define the mollified noises as
\begin{equation}
\dif W^{\zeta}_{r}(x)=\rho^{\zeta}*\dif W_{r}(x)\qquad\text{and}\qquad\dif\hat{W}^{\zeta}_{r}(x)=\rho^{\zeta}*\dif\hat{W}_{r}(x),\qquad x\in\mathbb{R}.\label{eq:difWzeta}
\end{equation}
Recalling the definitions \zcref{eq:RRepsdef,eq:Shepsdef}, we have
\begin{equation}
\begin{aligned}\mathbb{E}[\dif W^{\zeta}_{t}(x)\dif W^{\zeta}_{t'}(x')] & =\delta(t-t')\sum_{q\in2L\mathbb{Z}}[R^{\zeta}(x-x'+q)+R^{\zeta}(x+x'+q)]\\
 & =\delta(t-t')\left[\Sh^{\zeta}_{2L}(x-x')+\Sh^{\zeta}_{2L}(x+x')\right],
\end{aligned}
\label{eq:dWzetacov}
\end{equation}
and similarly for $(\dif\hat{W}_{\hat{t}})$. In particular, we have
the differential quadratic variation
\begin{equation}
\dif[W^{\zeta}(x)]_{t}=\left(\Sh^{\zeta}_{2L}(0)+\Sh^{\zeta}_{2L}(2x)\right)\dif t\overset{\zcref{eq:Shrescale}}{=}\left(\Sh^{\zeta}_{2L}(0)+\frac{1}{2}\Sh^{\zeta/2}_{L}(x)\right)\dif t.\label{eq:WzetaQV}
\end{equation}

We also need to mollify the initial data. First, extend the distributions
$u_{0}=\partial_{x}h_{0}$ and $\hat{u}_{0}=\partial_{x}\hat{h}_{0}$
(which are spatial white noises) from $[0,L]$ to $\mathbb{R}$ by
first performing an \emph{odd }extension to $[-L,L]$ and then periodizing.
Then we define
\begin{equation}
\eta^{\zeta}=\rho^{\zeta}*u_{0}\qquad\text{and}\qquad\hat{\eta}^{\zeta}=\rho^{\zeta}*\hat{u}_{0},\label{eq:etazetadef}
\end{equation}
as well as
\begin{equation}
A^{\zeta}(x)=\int^{x}_{0}\eta^{\zeta}(y)\,\dif y\qquad\text{and}\qquad\hat{A}^{\zeta}(x)=\int^{x}_{0}\hat{\eta}^{\zeta}(y)\,\dif y.\label{eq:Azetadef}
\end{equation}
Similarly to \zcref{eq:dWzetacov}, we have
\begin{equation}
\mathbb{E}[\eta^{\zeta}(x)\eta^{\zeta}(x')]=\Sh^{\zeta}_{2L}(x-x')-\Sh^{\zeta}_{2L}(x+x').\label{eq:etazetacov}
\end{equation}
The $-$ sign (as compared to the $+$ sign in \zcref{eq:dWzetacov})
is because we performed an odd rather than even extension. We also
note that $A^{\zeta}$ and $\hat{A}^{\zeta}$ are each even, $2L$-periodic
functions.

With the mollified noise and initial data, we define $(Z^{\eps,\zeta}_{\uu,\vv;t})_{t}$
to be the mild solution to the stochastic heat equation
\begin{subequations}
\label{Zepszeta}
\begin{align}
\dif Z^{\eps,\zeta}_{\uu,\vv;t}(x) & =\oh\Delta Z^{\eps,\zeta}_{\uu,\vv;t}(x)\dif t+Z^{\eps,\zeta}_{\uu,\vv;t}(x)\left(\varphi^{\eps}_{\uu,\vv}+\oh\Sh^{\zeta/2}_{L}\right)(x)\dif t+Z^{\eps,\zeta}_{\uu,\vv;t}(x)\dif W^{\zeta}_{t}(x), &  & t>0,x\in\mathbb{R};\label{eq:Zepszeta-eqn}\\
Z^{\eps,\zeta}_{\uu,\vv;0}(x) & =\e^{A^{\zeta}(x)}, &  & x\in\mathbb{R},\label{eq:Zepszeta-ic}
\end{align}
\end{subequations}
and similarly
\begin{subequations}
\label{Zepszetahat}
\begin{align}
\dif\hat{Z}^{\eps,\zeta}_{\uu,\vv;t}(x) & =\oh\Delta\hat{Z}^{\eps,\zeta}_{\uu,\vv;t}(x)\dif t-\hat{Z}^{\eps,\zeta}_{\uu,\vv;t}(x)\left(\varphi^{\eps}_{\uu,\vv}-\oh\Sh^{\zeta/2}_{L}\right)(x)\dif t-\hat{Z}^{\eps,\zeta}_{\uu,\vv;t}(x)\dif\hat{W}^{\zeta}_{t}(x), &  & t>0,x\in\mathbb{R};\label{eq:Zhatepszeta-eqn}\\
\hat{Z}^{\eps,\zeta}_{\uu,\vv;0}(x) & =\e^{-\hat{A}^{\zeta}(x)}, &  & x\in\mathbb{R},\label{eq:Zhatepszeta-ic}
\end{align}
\end{subequations}
These processes are smooth approximations of the processes $Z^{\eps}_{\uu,\vv},\hat{Z}^{\eps}_{\uu,\vv}$
defined in \zcref{eq:Zeps} and \zcref{eq:Zhateps}. In addition to
mollifying the white noise in the spatial variable and mollifying
the initial data, we have also approximated the boundary conditions
by the corresponding boundary potential terms. Although the equations
are written as if posed on $\mathbb{R}$, they should in fact be interpreted
on $[-L,L]$ with periodic boundary conditions.
\begin{rem}
In the equations for $Z^{\eps,\zeta}_{\uu,\vv}$ and $\hat{Z}^{\eps,\zeta}_{\uu,\vv}$,
two distinct boundary potentials appear: $\varphi^{\eps}_{\uu,\vv}$
and $\tfrac{1}{2}\Sh^{\zeta/2}_{L}$. These should not be confused.
The term $\varphi^{\eps}_{\uu,\vv}$ should be regarded as the approximation
of the \emph{inhomogeneous} Neumann boundary condition for the KPZ
equation with parameters $\uu,\vv$, whereas $\tfrac{1}{2}\Sh^{\zeta/2}_{L}$
approximates the Robin boundary condition for the \emph{homogeneous}
stochastic heat equation. For our purposes, we always assume $\zeta\ll\eps$. 
\end{rem}

Define
\begin{equation}
h^{\eps,\zeta}_{\uu,\vv;t}=\log Z^{\eps,\zeta}_{\uu,\vv;t}\qquad\text{and}\qquad\hat{h}^{\eps,\zeta}_{\uu,\vv;\hat{t}}=-\log Z^{\eps,\zeta}_{\uu,\vv;\hat{t}}\label{eq:hepszetalog}
\end{equation}
and
\[
u^{\eps,\zeta}_{\uu,\vv;t}=\partial_{x}h^{\eps,\zeta}_{\uu,\vv;t}\qquad\text{and}\qquad\hat{u}^{\eps,\zeta}_{\uu,\vv;\hat{t}}=\partial_{x}\hat{h}^{\eps,\zeta}_{\uu,\vv;\hat{t}}.
\]
In \zcref{Zepszeta,Zepszetahat}, the notion of mild solution is interpreted
in the classical sense, using the usual heat kernel
\begin{equation}
p_{t}(x)=\frac{\mathbf{1}_{t\ge0}}{\sqrt{2\pi t}}\e^{-x^{2}/(2t)}\label{eq:ptdef}
\end{equation}
(with the usual convention that $p_{0}$ is a delta distribution at
the origin), and we have
\begin{equation}
\begin{aligned}Z^{\eps,\zeta}_{\uu,\vv;t}(x) & =\int^{\infty}_{-\infty}p_{t}(x-y)\e^{A^{\zeta}(y)}\,\dif y+\int^{t}_{0}\int^{\infty}_{-\infty}p_{t-s}(x-y)Z^{\eps,\zeta}_{\uu,\vv;s}(y)\left(\varphi^{\eps}_{\uu,\vv}+\oh\Sh^{\zeta/2}_{L}\right)(y)\,\dif y\,\dif s\\
 & \qquad+\int^{t}_{0}\int^{\infty}_{-\infty}p_{t-s}(x-y)Z^{\eps,\zeta}_{\uu,\vv;s}(y)\,\dif W^{\zeta}_{s}(y).
\end{aligned}
\label{eq:Zepszetamild}
\end{equation}

By \zcref{eq:hepszetalog}, Itô's formula, and \zcref{eq:WzetaQV},
we deduce that the process $(h^{\eps,\zeta}_{\uu,\vv;t})$ satisfies
the KPZ equation with mollified noise:
\begin{subequations}
\label{eq:hepszeta}
\begin{align}
\dif h^{\eps,\zeta}_{\uu,\vv;t}(x) & =\left(\frac{1}{2}\Delta h^{\eps,\zeta}_{\uu,\vv;t}(x)+\frac{1}{2}|\nabla h^{\eps,\zeta}_{\uu,\vv;t}(x)|^{2}+\varphi^{\eps}_{\uu,\vv}(x)+\frac{1}{4}\Sh^{\zeta/2}_{L}(x)-\frac{1}{2}\Sh^{\zeta}_{2L}(0)\right)\dif t+\dif W^{\zeta}_{t}(x),\ \ t>0,x\in\mathbb{R};\label{eq:hepszeta-eqn}\\
h^{\eps,\zeta}_{\uu,\vv;0}(x) & =A^{\zeta}(x),\quad x\in\mathbb{R}.\label{eq:hepszeta-ic}
\end{align}
\end{subequations}
Here the coefficient $1/4$ in \eqref{eq:hepszeta-eqn} is different
from the $1/2$ in \eqref{eq:Zepszeta-eqn}, due to the Itô correction
term appearing in the Cole--Hopf transform. Similarly, the process
$(u^{\eps,\zeta}_{\uu,\vv;t})$ satisfies the stochastic Burgers equation
with mollified noise:
\begin{subequations}
\label{eq:uepszeta}
\begin{align}
\dif u^{\eps,\zeta}_{\uu,\vv;t}(x) & =\left(\frac{1}{2}\Delta u^{\eps,\zeta}_{\uu,\vv;t}+\frac{1}{2}\partial_{x}((u^{\eps,\zeta}_{\uu,\vv;t})^{2})+\partial_{x}\varphi^{\eps}_{\uu,\vv}+\frac{1}{4}\partial_{x}\Sh^{\zeta/2}_{L}\right)(x)\dif t+\partial_{x}\dif W^{\zeta}_{t}(x), &  & t>0,x\in\mathbb{R};\label{eq:uepszeta-eqn}\\
u^{\eps,\zeta}_{\uu,\vv;0}(x) & =\partial_{x}A^{\zeta}(x), &  & x\in\mathbb{R}.\label{eq:uepszeta-ic}
\end{align}
\end{subequations}

Here is the main result of this section, which will be proved in \zcref{subsec:Convergence-as-:}.

\begin{lem}
\label{lem:convofzeta}For each fixed $\eps>0$, we have
\[
\lim_{\zeta\to0}\sup_{\substack{t\in[0,T]\\
x\in[0,L]
}
}|h^{\eps,\zeta}_{\uu,\vv;t}(x)-h^{\eps}_{\uu,\vv;t}(x)|=0\qquad\text{and}\qquad\lim_{\zeta\to0}\sup_{\substack{\hat{t}\in[0,T]\\
x\in[0,L]
}
}|\hat{h}^{\eps,\zeta}_{\uu,\vv;\hat{t}}(x)-\hat{h}^{\eps}_{\uu,\vv;\hat{t}}(x)|=0
\]
in probability.
\end{lem}

\section{\label{sec:Proof-of-the}Proof of the main theorem}

In this section, we state the two main technical results of the paper
and show how they combine to yield the proof of the main theorem.
As discussed in \zcref{subsec:Our-method}, what we require are suitable
versions of \zcref{e.paininthe} and \zcref{e.915ma}. \zcref{prop:DtildeMG,prop:boundaryterm}
below provide their rigorous counterparts.

We first introduce some notation. Recalling the definition of $\mathcal{Y}_{\uu,\vv}(u)$
given in \zcref{eq:Duv-def}, define
\begin{equation}
\tilde{\mathcal{Y}}^{\eps}_{\uu,\vv}(h)\coloneqq\exp\left\{ -2\left\langle \varphi^{\eps}_{\uu,\vv},h\right\rangle \right\} \mathcal{Y}_{\uu,\vv}(\partial_{x}h)\label{eq:Dtildeepsdef}
\end{equation}
and
\begin{equation}
\tilde{\mathcal{Y}}_{\uu,\vv}(h)\coloneqq\e^{\uu h(0)+\vv h(L)}\mathcal{Y}_{\uu,\vv}(\partial_{x}h)\overset{\zcref{eq:Duv-def}}{=}\mathfrak{Z}^{-1}_{\uu,\vv}\mathrm{E}_{B}\left[\e^{\uu B(0)+\vv B(L)}\left(\int^{L}_{0}\e^{-(h(x)-B(x))}\,\dif x\right)^{-\mathsf{u}-\mathsf{v}}\right].\label{eq:Dtildedef}
\end{equation}
The main ingredients in the proof of \zcref{thm:mainthm}, and indeed
the main technical results of our work, are the following two propositions.

The first proposition provides a rigorous counterpart to the heuristic
description in \zcref{e.915ma}. 
\begin{prop}
\label{prop:DtildeMG}The process $(\hat{M}_{\uu,\vv;\hat{t}})_{\hat{t}\in[0,T]}$
defined by
\begin{equation}
\hat{M}_{\uu,\vv;\hat{t}}\coloneqq\e^{\left((\uu^{3}+\vv^{3})/6-(\uu+\vv)/24\right)\hat{t}}\tilde{\mathcal{Y}}_{\uu,\vv}(\hat{h}_{\uu,\vv;\hat{t}})\label{eq:DtildeMG}
\end{equation}
is an $\{\hat{\mathscr{F}}_{\hat{t}}\}$-martingale. In particular,
\begin{equation}
\mathbb{E}\left[\e^{\left((\uu^{3}+\vv^{3})/6-(\uu+\vv)/24\right)T}\tilde{\mathcal{Y}}_{\uu,\vv}(\hat{h}_{\uu,\vv;T})\ \middle|\ \hat{\mathscr{F}}_{0}\right]=\tilde{\mathcal{Y}}_{\uu,\vv}(\hat{h}_{0}).\label{eq:MGapp}
\end{equation}
\end{prop}

\begin{rem}
\label{rem:whyweneedh}The prefactor $\e^{\uu h(0)+\vv h(L)}$ in
\zcref{eq:Dtildedef} is crucial to ensure that the process $(\tilde{\mathcal{Y}}_{\uu,\vv}(\hat{h}_{\uu,\vv;\hat{t}}))_{\hat{t}}$
appearing in \zcref{eq:DtildeMG} is a semimartingale. The additional
exponential factor $\e^{\left((\uu^{3}+\vv^{3})/6-(\uu+\vv)/24\right)\hat{t}}$
cancels the drift in $\tilde{\mathcal{Y}}_{\uu,\vv}(\hat{h}_{\uu,\vv;\hat{t}})$
and yields a martingale. This is another reflection of the issue discussed
in \zcref{rem:why-not-add}, that in \zcref{prop:noise-relation},
we prefer to consider $(W_{T}-W_{0})-(\hat{W}_{T}-\hat{W}_{0})$ rather
than $(W_{T}-W_{0})+(\hat{W}_{T}-\hat{W}_{0})$. The latter expression
would not include the term $h_{T}-h_{0}$ that is essential to obtain
a semimartingale at this stage.
\end{rem}

Recall that $\hat{\mathcal{B}}^{\eps}_{\uu,\vv;0,T}$ was defined
in \zcref{eq:Bepshatlimit}, as a suitable version of the nonlinear
term appearing in the backward KPZ equation satisfied by $\hat{h}^{\eps}_{\uu,\vv}$.
The following proposition is a rigorous counterpart to the heuristic
description in \zcref{e.paininthe}. 
\begin{prop}
\label{prop:boundaryterm}Let $U$ be any $\hat{\mathscr{F}}_{0}$-measurable
random variable. Then 
\begin{equation}
\left((\hat{W}_{\hat{t}})_{\hat{t}\in[0,T]},U,\hat{\mathcal{B}}^{\eps}_{\uu,\vv;0,T}(\varphi^{\eps}_{\uu,\vv})\right)\xrightarrow[\eps\downarrow0]{\mathrm{law}}\left((\hat{W}_{\hat{t}})_{\hat{t}\in[0,T]},U,\hat{\Upsilon}_{\uu,\vv;0,T}\right),\label{eq:boundaryterm}
\end{equation}
where $\hat{\Upsilon}_{\uu,\vv;0,T}\sim\mathcal{N}(\frac{T}{2}(\uu^{2}+\vv^{2})V_{\psi}-\frac{T}{6}(\uu^{3}+\vv^{3}),T(\uu^{2}+\vv^{2})V_{\psi})$
is independent of $\left((\hat{W}_{\hat{t}})_{\hat{t}\in[0,T]},U\right)$,
and $V_{\psi}>0$ is the constant defined in \zcref{eq:redcherry-var-limit}
below which depends on $\psi$. In particular, we have
\begin{equation}
\mathbb{E}\e^{-\hat{\Upsilon}_{\uu,\vv;0,T}}=\e^{(\uu^{3}+\vv^{3})T/6}.\label{eq:exponential-moment-Upsilonhat}
\end{equation}
\end{prop}

Let us show how these two propositions combine together to prove \zcref{thm:mainthm}.
The proof is essentially a repetition of the first half of \zcref{s.keysteps},
except for a technical step to justify a uniform integrability argument. 
\begin{proof}[Proof of \zcref{thm:mainthm}]
Let $F$ be an arbitrary bounded function defined on the state space
of $(u_{t})_{t}$, which can be chosen for example as $\mathcal{C}^{-\kappa}[0,L]$
for some $\kappa>\tfrac{1}{2}$. The goal is to show that
\begin{equation}
\mathbb{E}[\mathcal{Y}_{\uu,\vv}(u_{\uu,\vv;0})F(u_{\uu,\vv;0})]=\mathbb{E}[\mathcal{Y}_{\uu,\vv}(u_{\uu,\vv;0})F(u_{\uu,\vv;T})].\label{eq:goal}
\end{equation}
 We can rewrite the left side of \zcref{eq:goal} as
\begin{equation}
\mathbb{E}[\mathcal{Y}_{\uu,\vv}(u_{\uu,\vv;0})F(u_{\uu,\vv;0})]\overset{\zcref{eq:ic-u0}}{=}\mathbb{E}[\mathcal{Y}_{\uu,\vv}(u_{0})F(u_{0})]\overset{\zcref{eq:utstationary}}{=}\mathbb{E}[\mathcal{Y}_{\uu,\vv}(u_{T})F(u_{T})]\overset{\zcref{eq:uhatdef}}{=}\mathbb{E}[\mathcal{Y}_{\uu,\vv}(\hat{u}_{0})F(\hat{u}_{0})].\label{eq:develop-left}
\end{equation}
We can also rewrite the right side of \zcref{eq:goal} as
\begin{equation}
\mathbb{E}[\mathcal{Y}_{\uu,\vv}(u_{\uu,\vv;0})F(u_{\uu,\vv;T})]\ovset{\zcref{eq:bcconverges-CM}}=\lim_{\eps\downarrow0}\mathbb{E}[\mathcal{Q}^{\eps}_{\uu,\vv;0,T}\mathcal{Y}_{\uu,\vv}(u_{0})F(u_{T})]\overset{\zcref{eq:uhatdef}}{=}\lim_{\eps\downarrow0}\mathbb{E}[\mathbb{E}[\mathcal{Q}^{\eps}_{\uu,\vv;0,T}\mathcal{Y}_{\uu,\vv}(u_{0})\mid\hat{\mathscr{F}}_{0}]F(\hat{u}_{0})].\label{eq:develop-right}
\end{equation}
We will show below that for any bounded random variable $U$ that
is $\hat{\mathscr{F}}_{0}-$measurable, 
\begin{equation}
\lim_{\eps\downarrow0}\mathbb{E}\left[\mathbb{E}[\mathcal{Q}^{\eps}_{\uu,\vv;0,T}\mathcal{Y}_{\uu,\vv}(u_{0})\mid\hat{\mathscr{F}}_{0}]U\right]=\mathbb{E}\left[\mathcal{Y}_{\uu,\vv}(\hat{u}_{0})U\right],\label{eq:goal1}
\end{equation}
which implies \zcref{eq:goal}.

To take the conditional expectation with respect to the backward
filtration, we first express the forward white noise which appears
in the expression of $\mathcal{Q}^{\eps}_{\uu,\vv;0,T}$ in terms
of the backward noise and solution:
\begin{align}
\mathbb{E}[\mathcal{Q}^{\eps}_{\uu,\vv;0,T}\mathcal{Y}_{\uu,\vv}(u_{0})\mid\hat{\mathscr{F}}_{0}]\ovset{\substack{\zcref{eq:uhatdef},\\
\zcref{eq:Qepsrel-rev}
}
} & =\mathbb{E}\left[\hat{\mathcal{Q}}^{\eps}_{\uu,\vv;0,T}\exp\left\{ -2\langle\varphi^{\eps}_{\uu,\vv},\hat{h}_{T}-\hat{h}_{0}-T/24\rangle-\hat{\mathcal{B}}_{0,T}(\varphi^{\eps}_{\uu,\vv})\right\} \mathcal{Y}_{\uu,\vv}(\hat{u}_{T})\ \middle|\ \hat{\mathscr{F}}_{0}\right]\nonumber \\
\ovset{\substack{\zcref{eq:phiint},\\
\zcref{eq:Dtildeepsdef}
}
} & =\e^{-\frac{1}{24}(\uu+\vv)T+2\langle\varphi^{\eps}_{\uu,\vv},\hat{h}_{0}\rangle}\mathbb{E}\left[\hat{\mathcal{Q}}^{\eps}_{\uu,\vv;0,T}\exp\left\{ -\hat{\mathcal{B}}_{\uu,\vv;0,T}(\varphi^{\eps}_{\uu,\vv})\right\} \tilde{\mathcal{Y}}^{\eps}_{\uu,\vv}(\hat{h}_{T})\ \middle|\ \hat{\mathscr{F}}_{0}\right]\nonumber \\
\ovset{\zcref{eq:apply-CM}} & =\e^{-\frac{1}{24}(\uu+\vv)T+2\langle\varphi^{\eps}_{\uu,\vv},\hat{h}_{0}\rangle}\mathbb{E}\left[\exp\left\{ -\hat{\mathcal{B}}^{\eps}_{\uu,\vv;0,T}(\varphi^{\eps}_{\uu,\vv})\right\} \tilde{\mathcal{Y}}^{\eps}_{\uu,\vv}(\hat{h}^{\eps}_{\uu,\vv;T})\ \middle|\ \hat{\mathscr{F}}_{0}\right].\label{eq:rewritewithtimerev}
\end{align}

We now define
\begin{equation}
\mathcal{J}^{\eps}_{\uu,\vv;0,T}\coloneqq\exp\left\{ -\frac{1}{24}(\uu+\vv)T+2\langle\varphi^{\eps}_{\uu,\vv},\hat{h}_{0}\rangle\right\} \tilde{\mathcal{Y}}^{\eps}_{\uu,\vv}(\hat{h}^{\eps}_{\uu,\vv;T})\label{eq:Jepsdef}
\end{equation}
and
\begin{equation}
\mathcal{J}_{\uu,\vv;0,T}\coloneqq\exp\left\{ -\frac{1}{24}(\uu+\vv)T-\uu\hat{h}_{0}(0)-\vv\hat{h}_{0}(L)\right\} \tilde{\mathcal{Y}}_{\uu,\vv}(\hat{h}_{\uu,\vv;T}),\label{eq:Jlimitdef}
\end{equation}
 so in particular
\begin{equation}
\mathbb{E}[\mathcal{Q}^{\eps}_{\uu,\vv;0,T}\mathcal{Y}_{\uu,\vv}(u_{0})\mid\hat{\mathscr{F}}_{0}]\overset{\zcref{eq:rewritewithtimerev}}{=}\mathbb{E}\left[\mathcal{J}^{\eps}_{\uu,\vv;0,T}\exp\left\{ -\hat{\mathcal{B}}^{\eps}_{0,T}(\varphi^{\eps}_{\uu,\vv})\right\} \ \middle|\ \hat{\mathscr{F}}_{0}\right].\label{eq:ourthingintermsofJ}
\end{equation}
Now it follows from \zcref{prop:boundaryterm,prop:epstozeroconv}
that, for any $\hat{\mathscr{F}}_{0}$-measurable bounded random variable
$U$, we have
\[
\left(\mathcal{J}^{\eps}_{\uu,\vv;0,T},U,\hat{\mathcal{B}}_{\uu,\vv;0,T}(\varphi^{\eps}_{\uu,\vv})\right)\xrightarrow[\eps\downarrow0]{\mathrm{law}}\left(\mathcal{J}_{\uu,\vv;0,T},U,\hat{\Upsilon}_{\uu,\vv;0,T}\right),
\]
and that the Gaussian random variable $\hat{\Upsilon}_{\uu,\vv;0,T}$
is independent of $(\mathcal{J}_{\uu,\vv;0,T},U)$. By Skorokhod's
representation theorem, we can therefore find a family of random variables
$((\underline{\mathcal{J}}^{\eps}_{\uu,\vv;0,T},\underline{U}^{\eps},\underline{\hat{\Upsilon}}^{\eps}_{\uu,\vv;0,T}))_{\eps>0}$
such that
\begin{equation}
\left(\underline{\mathcal{J}}^{\eps}_{\uu,\vv;0,T},\underline{U}^{\eps},\underline{\hat{\Upsilon}}^{\eps}_{\uu,\vv;0,T}\right)\overset{\mathrm{law}}{=}\left(\mathcal{J}^{\eps}_{\uu,\vv;0,T},U,\hat{\mathcal{B}}_{\uu,\vv;0,T}(\varphi^{\eps}_{\uu,\vv})\right)\label{eq:applySkorokhod}
\end{equation}
and
\begin{equation}
\lim_{\eps\downarrow0}\underline{\mathcal{J}}^{\eps}_{\uu,\vv;0,T}\e^{-\underline{\hat{\Upsilon}}^{\eps}_{\uu,\vv;0,T}}\underline{U}^{\eps}=\mathcal{J}_{\uu,\vv;0,T}\e^{-\hat{\Upsilon}_{\uu,\vv;0,T}}U\qquad\text{a.s.}\label{eq:asconvergence}
\end{equation}

We would like to upgrade \zcref{eq:asconvergence} to convergence
in $L^{1}$ so that we can take the conditional expectations. For
each $\eps>0$, we have
\begin{equation}
\begin{aligned}\mathbb{E}\left[\underline{\mathcal{J}}^{\eps}_{\uu,\vv;0,T}\e^{-\underline{\Upsilon}^{\eps}_{\uu,\vv;0,T}}\right]\ovset{\zcref{eq:applySkorokhod}} & =\mathbb{E}\left[\mathcal{J}^{\eps}_{\uu,\vv;0,T}\e^{-\hat{\mathcal{B}}_{\uu,\vv;0,T}(\varphi^{\eps}_{\uu,\vv})}\right]\\
\ovset{\zcref{eq:Jepsdef}} & =\mathbb{E}\left[\exp\left\{ -\frac{1}{24}(\uu+\vv)T+2\langle\varphi^{\eps}_{\uu,\vv},\hat{h}_{0}\rangle-\hat{\mathcal{B}}_{\uu,\vv;0,T}(\varphi^{\eps}_{\uu,\vv})\right\} \tilde{\mathcal{Y}}^{\eps}_{\uu,\vv}(\hat{h}^{\eps}_{\uu,\vv;T})\right]\\
\ovset{\zcref{eq:rewritewithtimerev}} & =\mathbb{E}\left[\mathcal{Q}^{\eps}_{\uu,\vv;0,T}\mathcal{Y}_{\uu,\vv}(u_{0})\right]=1.
\end{aligned}
\label{eq:limofexps-1}
\end{equation}
On the other hand, we have, using the independence of $\e^{-\hat{\Upsilon}_{\uu,\vv;0,T}}$
from everything else and the fact \zcref{eq:exponential-moment-Upsilonhat}
that $\mathbb{E}\e^{-\hat{\Upsilon}_{\uu,\vv;0,T}}=\e^{(\uu^{3}+\vv^{3})/6}$,
that
\begin{align}
\mathbb{E}\left[\mathcal{J}_{\uu,\vv;0,T}\e^{-\hat{\Upsilon}_{\uu,\vv;0,T}}\ \middle|\ \hat{\mathscr{F}}_{0}\right] & =\mathbb{E}[\e^{-\hat{\Upsilon}_{\uu,\vv;0,T}}]\mathbb{E}\left[\mathcal{J}_{\uu,\vv;0,T}\ \middle|\ \hat{\mathscr{F}}_{0}\right]\nonumber \\
\ovset{\zcref{eq:Jlimitdef}} & =\exp\left\{ -\uu\hat{h}_{0}(0)-\vv\hat{h}_{0}(L)\right\} \mathbb{E}\left[\exp\left\{ \left(\frac{1}{6}(\uu^{3}+\vv^{3})-\frac{1}{24}(\uu+\vv)\right)T\right\} \tilde{\mathcal{Y}}_{\uu,\vv}(\hat{h}_{\uu,\vv;T})\ \middle|\ \hat{\mathscr{F}}_{0}\right]\nonumber \\
\ovset{\zcref{eq:MGapp}} & =\exp\left\{ -\uu\hat{h}_{0}(0)-\vv\hat{h}_{0}(L)\right\} \tilde{\mathcal{Y}}_{\uu,\vv}(\hat{h}_{0})\overset{\zcref{eq:Dtildedef}}{=}\mathcal{Y}_{\uu,\vv}(\hat{u}_{0}),\label{eq:condexpoflim}
\end{align}
and in particular
\begin{equation}
\mathbb{E}\left[\mathcal{J}_{\uu,\vv;0,T}\e^{-\Upsilon_{\uu,\vv;0,T}}\right]=\mathbb{E}[\mathcal{Y}_{\uu,\vv}(\hat{u}_{0})]=1.\label{eq:expoflim-1}
\end{equation}
Since the right sides of \zcref{eq:limofexps-1} and \zcref{eq:expoflim-1}
match, we can use \cite[Thm.~4.6.3]{durrett:2019:probability} (with
the hypothesis of convergence in probability satisfied by \zcref{eq:asconvergence}
with $U\equiv1$) to conclude that the family $\left(\underline{\mathcal{J}}^{\eps}_{\uu,\vv;0,T}\e^{-\underline{\hat{\Upsilon}}^{\eps}_{\uu,\vv;0,T}}\right)_{\eps>0}$
is uniformly integrable. Since we assumed that $U$ and hence $\underline{U}^{\eps}$
is bounded, this implies that the family $\left(\underline{\mathcal{J}}^{\eps}_{\uu,\vv;0,T}\e^{-\underline{\hat{\Upsilon}}^{\eps}_{\uu,\vv;0,T}}\underline{U}^{\eps}\right)_{\eps>0}$
is uniformly integrable as well, and so another application of \cite[Thm.~4.6.3]{durrett:2019:probability}
implies that \zcref{eq:asconvergence} can be upgraded to
\[
\lim_{\eps\downarrow0}\mathbb{E}\left|\underline{\mathcal{J}}^{\eps}_{\uu,\vv;0,T}\e^{-\underline{\hat{\Upsilon}}^{\eps}_{\uu,\vv;0,T}}\underline{U}^{\eps}-\mathcal{J}_{\uu,\vv;0,T}\e^{-\hat{\Upsilon}_{\uu,\vv;0,T}}U\right|=0.
\]
In particular, we have
\[
\mathbb{E}\left[\mathcal{J}_{\uu,\vv;0,T}\e^{-\hat{\Upsilon}_{\uu,\vv;0,T}}U\right]=\lim_{\eps\downarrow0}\mathbb{E}\left[\underline{\mathcal{J}}^{\eps}_{\uu,\vv;0,T}\e^{-\underline{\hat{\Upsilon}}^{\eps}_{\uu,\vv;0,T}}\underline{U}^{\eps}\right]\overset{\zcref{eq:applySkorokhod}}{=}\lim_{\eps\downarrow0}\mathbb{E}\left[\mathcal{J}^{\eps}_{\uu,\vv;0,T}\e^{-\hat{\mathcal{B}}_{\uu,\vv;0,T}(\varphi^{\eps}_{\uu,\vv})}U\right].
\]
The left side equals $\mathbb{E}[\mathcal{Y}_{\uu,\vv}(\hat{u}_{0})U]$
by \zcref{eq:condexpoflim}, while the right side equals $\lim_{\eps\downarrow0}\mathbb{E}\left[\mathcal{Q}^{\eps}_{\uu,\vv;0,T}\mathcal{Y}_{\uu,\vv}(u_{0})U\right]$
by \zcref{eq:ourthingintermsofJ}. Thus, we have derived \zcref{eq:goal1}
and completed the proof. 
\end{proof}

\section{\label{sec:It=0000F4-formula-and}Analysis in the bulk: Itô's formula
and integration by parts}

In this section we prove \zcref{prop:DtildeMG}, which is the only
place in the paper where the explicit expression of the invariant
measure, given by \zcref{eq:Duv-def}, is used. The proof is an application
of Itô's formula to the solution of the mollified stochastic heat
equation, together with repeated applications of Gaussian integration
by parts with respect to the auxiliary Brownian motion appearing in
\zcref{eq:Duv-def}.

We define, for $\vv\in\mathbb{R}$,
\begin{equation}
B_{\mathsf{v}}(x)=B(x)+\vv x.\label{eq:Bvdef}
\end{equation}
By the Cameron--Martin theorem applied to \zcref{eq:Duv-def}, we
have
\begin{align*}
\mathcal{Y}_{\uu,\vv}(u) & =\tilde{\mathfrak{Z}}^{-1}_{\uu,\vv}\mathrm{E}_{B}\left[\e^{-\uu h(0)-\vv h(L)}\left(\int^{L}_{0}\e^{-(h(x)-B_{\vv}(x))}\,\dif x\right)^{-\uu-\vv}\right],
\end{align*}
where $\tilde{\mathfrak{Z}}^{-1}_{\uu,\vv}$ is a new normalizing
constant. Hence we have (recalling the definition \zcref{eq:Dtildedef})
\[
\tilde{\mathcal{Y}}_{\uu,\vv}(h)=\tilde{\mathfrak{Z}}^{-1}_{\uu,\vv}\mathrm{E}_{B}\left[\left(\int^{L}_{0}\e^{-(h(x)-B_{\vv}(x))}\,\dif x\right)^{-\uu-\vv}\right].
\]

\begin{proof}[Proof of \zcref{prop:DtildeMG}.]
To simplify the notation, let
\begin{equation}
\alpha=\frac{1}{6}(\uu^{3}+\vv^{3})-\frac{1}{24}(\uu+\vv).\label{eq:alphachoice}
\end{equation}
Recall that the goal was to show that $\big(\e^{\alpha\hat{t}}\tilde{\mathcal{Y}}_{\uu,\vv}(\hat{h}_{\uu,\vv;\hat{t}})\big)_{\hat{t}}$
is an $\{\hat{\mathscr{F}}_{\hat{t}}\}$-martingale. The starting
point is an application of Itô's formula to the smoothed solutions
introduced in \zcref{sec:mollifying}. Define 
\begin{align}
\hat{Y}^{\eps,\zeta}_{\hat{t}} & \coloneqq\tilde{\mathfrak{Z}}_{\uu,\vv}\e^{\alpha\hat{t}}\tilde{\mathcal{Y}}_{\uu,\vv}(\hat{h}^{\eps,\zeta}_{\uu,\vv;\hat{t}})\overset{\zcref{eq:Dtildedef}}{=}\e^{\alpha\hat{t}}\mathrm{E}_{B}\left[\left(\int^{L}_{0}\hat{Z}^{\eps,\zeta}_{\uu,\vv;\hat{t}}(x)\e^{B_{\vv}(x)}\,\dif x\right)^{-\uu-\vv}\right]=\e^{\alpha\hat{t}}\mathrm{E}_{B}\left[\hat{I}^{\eps,\zeta,0}_{\uu,\vv;\hat{t}}(B_{\vv})^{-\uu-\vv}\right],\label{eq:Dhatdef-1}
\end{align}
where we have defined
\begin{equation}
\hat{I}^{\eps,\zeta}_{\uu,\vv;\hat{t}}(B_{\vv})\coloneqq\int^{L}_{0}\hat{Z}^{\eps,\zeta}_{\uu,\vv;\hat{t}}(x)\e^{B_{\vv}(x)}\,\dif x.\label{eq:Ihatdef-1}
\end{equation}
Since $\hat{Z}^{\eps,\zeta}_{\uu,\vv;\hat{t}}(x)$ is the strong solution
to the mollified SHE \zcref{eq:Zhatepszeta-eqn}, we obtain by applying
Itô's formula that
\[
\hat{Y}^{\eps,\zeta}_{\hat{t}}-\hat{Y}^{\eps,\zeta}_{0}=\alpha\int^{\hat{t}}_{0}\hat{Y}^{\eps,\zeta}_{\hat{s}}\,\dif\hat{s}+J^{\eps,\zeta}_{1;\hat{t}}+J^{\eps,\zeta}_{2;\hat{t}}+J^{\eps,\zeta}_{3;\hat{t}}+N^{\eps,\zeta}_{\hat{t}},
\]
where
\begin{align*}
J^{\eps,\zeta}_{1;\hat{t}} & \coloneqq-\frac{1}{2}(\uu+\vv)\int^{\hat{t}}_{0}\e^{\alpha\hat{s}}\mathrm{E}_{B}\left[\hat{I}^{\eps,\zeta}_{\uu,\vv;\hat{s}}(B_{\vv})^{-\uu-\vv-1}\int^{L}_{0}\Delta\hat{Z}^{\eps,\zeta}_{\uu,\vv;\hat{s}}(x)\e^{B_{\vv}(x)}\,\dif x\right]\,\dif\hat{s},\\
J^{\eps,\zeta}_{2;\hat{t}} & \coloneqq(\uu+\vv)\int^{\hat{t}}_{0}\e^{\alpha\hat{s}}\mathrm{E}_{B}\left[\hat{I}^{\eps,\zeta}_{\uu,\vv;\hat{s}}(B_{\vv})^{-\uu-\vv-1}\int^{L}_{0}\left(\varphi^{\eps}_{\uu,\vv}-\frac{1}{2}\Sh^{\zeta/2}_{L}\right)(x)\hat{Z}^{\eps,\zeta}_{\uu,\vv;\hat{s}}(x)\e^{B_{\vv}(x)}\,\dif x\right]\,\dif\hat{s},\\
J^{\eps,\zeta}_{3;\hat{t}} & \coloneqq\frac{1}{2}(\uu+\vv)(\uu+\vv+1)\int^{\hat{t}}_{0}\e^{\alpha\hat{s}}\mathrm{E}_{\mathrm{B}}\left[\hat{I}^{\eps,\zeta}_{\uu,\vv;\hat{s}}(B_{\vv})^{-\uu-\vv-2}\iint_{[0,L]^{2}}\overline{R}^{\zeta}(x,y)\prod_{z\in\{x,y\}}\left(\hat{Z}^{\eps,\zeta}_{\uu,\vv;\hat{s}}(z)\e^{B_{\vv}(z)}\right)\,\dif x\,\dif y\right]\,\dif\hat{s},
\end{align*}
and
\[
\hat{N}^{\eps,\zeta}_{\hat{t}}\coloneqq(\uu+\vv)\int^{\hat{t}}_{0}\e^{\alpha\hat{s}}\mathrm{E}_{B}\left[\hat{I}^{\eps,\zeta}_{\uu,\vv;\hat{s}}(B_{\vv})^{-\uu-\vv-1}\left\langle \hat{Z}^{\eps,\zeta}_{\uu,\vv;\hat{s}}\e^{B_{\vv}},\dif\hat{W}_{\hat{s}}\right\rangle \right].
\]
Here, in the definition of $J^{\eps,\zeta}_{3;\hat{t}}$, we used
the abbreviated notation $\overline{R}^{\zeta}(x,y)=\Sh^{\zeta}_{2L}(x-y)+\Sh^{\zeta}_{2L}(x+y)$.

It remains to take the limit $\eps,\zeta\to0$ in the above expressions.
The main difficulty arises from the term $J^{\eps,\zeta}_{1;\hat{t}}$
which involves $\Delta\hat{Z}^{\eps,\zeta}_{\uu,\vv;\hat{s}}$ and
does not converge as $\eps,\zeta\to0$ as it is written. The key is
to exploit spatial integration together with the averaging induced
by the auxiliary Brownian motion. This is carried out in \zcref{lem:B2-computation-1}
below, which shows that, each fixed $\eps,\zeta>0$, 
\begin{align}
J^{\eps,\zeta}_{1;\hat{t}} & =-\alpha\int^{\hat{t}}_{0}\hat{Y}^{\eps,\zeta}_{\hat{s}}\,\dif\hat{s}+\frac{1}{2}(\uu+\vv)\left(\vv+\oh\right)\int^{\hat{t}}_{0}\e^{\alpha\hat{s}}\hat{Z}^{\eps,\zeta}_{\uu,\vv;\hat{s}}(L)\mathrm{E}_{B}\left[\hat{I}^{\eps,\zeta}_{\uu,\vv;\hat{s}}(B_{\vv})^{-\uu-\vv-1}\e^{B_{\vv}(L)}\right]\,\dif\hat{s}\nonumber \\
 & \qquad+\frac{1}{2}(\uu+\vv)\left(\uu+\oh\right)\int^{\hat{t}}_{0}\e^{\alpha\hat{s}}\hat{Z}^{\eps,\zeta}_{\uu,\vv;\hat{s}}(0)\mathrm{E}_{B}\left[\hat{I}^{\eps,\zeta,0}_{\uu,\vv;\hat{s}}(B_{\vv})^{-\uu-\vv-1}\right]\,\dif\hat{s}\nonumber \\
 & \qquad-\frac{1}{2}(\uu+\vv)(\uu+\vv+1)\int^{\hat{t}}_{0}\e^{\alpha\hat{s}}\mathrm{E}_{B}\left[\hat{I}^{\eps,\zeta,0}_{\uu,\vv;\hat{s}}(B_{\vv})^{-\uu-\vv-2}\int^{L}_{0}\hat{Z}^{\eps,\zeta}_{\uu,\vv;\hat{s}}(y)^{2}\e^{2B_{\vv}(y)}\,\dif y\right]\,\dif\hat{s},\label{eq:J3epszeta}
\end{align}
The first term on the right side of \zcref{eq:J3epszeta} comes from
the third term of the right side of \zcref{eq:key-comp}, together
with the fact that 
\[
\frac{1}{2}(\uu+\vv)\left(\frac{1}{3}\uu^{2}+\frac{1}{3}\vv^{2}-\frac{1}{3}\uu\vv-\frac{1}{12}\right)=\frac{1}{6}(\uu^{3}+\vv^{3})-\frac{1}{24}(\uu+\vv)\overset{\zcref{eq:alphachoice}}{=}\alpha.
\]

Let
\begin{align}
\hat{Y}_{\hat{t}} & \coloneqq\tilde{\mathfrak{Z}}_{\uu,\vv}\e^{\alpha\hat{t}}\tilde{\mathcal{Y}}_{\uu,\vv}(\hat{h}_{\uu,\vv;\hat{t}})\label{eq:Dhatdef-1-1}
\end{align}
and
\[
\hat{I}_{\uu,\vv;\hat{t}}(B_{\vv})\coloneqq\int^{L}_{0}\hat{Z}_{\uu,\vv;\hat{t}}(x)\e^{B_{\vv}(x)}\,\dif x.
\]
Then we can take $\zeta\to0$ and then $\eps\to0$ and use \zcref{prop:epstozeroconv}
and \zcref{lem:convofzeta} to obtain that
\begin{equation}
\hat{Y}_{\hat{t}}-\hat{Y}_{0}=J_{1;\hat{t}}+J_{2;\hat{t}}+J_{3;\hat{t}}+\hat{N}_{\hat{t}},\label{eq:Dhatteqn}
\end{equation}
where
\begin{align*}
J_{1;\hat{t}} & \coloneqq\frac{1}{2}(\uu+\vv)\left(\vv+\oh\right)\int^{\hat{t}}_{0}\e^{\alpha\hat{s}}\hat{Z}_{\uu,\vv;\hat{s}}(L)\mathrm{E}_{B}\left[\hat{I}_{\uu,\vv;\hat{s}}(B_{\vv})^{-\uu-\vv-1}\e^{B_{\vv}(L)}\right]\,\dif\hat{s}\\
 & \qquad+\frac{1}{2}(\uu+\vv)\left(\uu+\oh\right)\int^{\hat{t}}_{0}\e^{\alpha\hat{s}}\hat{Z}_{\uu,\vv;\hat{s}}(0)\mathrm{E}_{B}\left[\hat{I}_{\uu,\vv;\hat{s}}(B_{\vv})^{-\uu-\vv-1}\right]\,\dif\hat{s}\\
 & \qquad-\frac{1}{2}(\uu+\vv)(\uu+\vv+1)\int^{\hat{t}}_{0}\e^{\alpha\hat{s}}\mathrm{E}_{B}\left[\hat{I}_{\uu,\vv;\hat{s}}(B_{\vv})^{-\uu-\vv-2}\int^{L}_{0}\hat{Z}_{\uu,\vv;\hat{s}}(y)^{2}\e^{2B_{\vv}(y)}\,\dif y\right]\,\dif\hat{s},\\
J_{2;\hat{t}} & \coloneqq-\frac{1}{2}(\uu+\vv)\int^{\hat{t}}_{0}\e^{\alpha\hat{s}}\mathrm{E}_{B}\left[\hat{I}_{\uu,\vv;\hat{s}}(B_{\vv})^{-\uu-\vv-1}\left[\left(\uu+\oh\right)\hat{Z}_{\uu,\vv;\hat{s}}(0)+\left(\vv+\oh\right)\e^{B_{\vv}(L)}\hat{Z}_{\uu,\vv;\hat{s}}(L)\right]\right]\,\dif\hat{s},\\
J_{3;\hat{t}} & \coloneqq\frac{1}{2}(\uu+\vv)(\uu+\vv+1)\int^{\hat{t}}_{0}\e^{\alpha\hat{s}}\mathrm{E}_{\mathrm{B}}\left[\hat{I}_{\uu,\vv;\hat{s}}(B_{\vv})^{-\uu-\vv-2}\int^{L}_{0}\hat{Z}_{\uu,\vv;\hat{s}}(x)^{2}\e^{2B_{\vv}(x)}\,\dif x\right]\,\dif\hat{s},
\end{align*}
and 
\[
\hat{N}_{\hat{t}}=(\uu+\vv)\int^{\hat{t}}_{0}\e^{\alpha\hat{s}}\mathrm{E}_{B}\left[\hat{I}_{\uu,\vv;\hat{s}}(B_{\vv})^{-\uu-\vv-1}\left\langle \hat{Z}_{\uu,\vv;\hat{s}}\e^{B_{\vv}},\dif\hat{W}_{\hat{s}}\right\rangle \right].
\]
Note that in the convergence of $J^{\eps,\zeta}_{2;\hat{t}}\to J_{2;\hat{t}}$,
we use the facts that, restricted to $[0,L]$, $\varphi^{\eps}_{\uu,\vv}\to-\tfrac{1}{2}\uu\delta_{0}-\tfrac{1}{2}\vv\delta_{L}$
and $\Sh^{\zeta/2}_{L}\to\tfrac{1}{2}\delta_{0}+\tfrac{1}{2}\delta_{L}$.
Similarly, in the convergence of $J^{\eps,\zeta}_{3;\hat{t}}\to J_{3;\hat{t}}$,
we use that, restricted to $(x,y)\in[0,L]^{2}$, $\overline{R}^{\zeta}(x,y)=\Sh^{\zeta}_{2L}(x-y)+\Sh^{\zeta}_{2L}(x+y)\to\delta_{0}(x-y)$.
From these expressions, we check directly that $J_{1;\hat{t}}+J_{2;\hat{t}}+J_{3;\hat{t}}=0$,
so in fact we have $\hat{Y}_{\hat{t}}-\hat{Y}_{0}=\hat{N}_{\hat{t}}$.
Moreover, we see from \zcref{prop:moment-bd} that the quadratic variation
of $(\hat{N}_{\hat{t}})_{\hat{t}}$ has finite second moment, so $(\hat{N}_{\hat{t}})_{\hat{t}}$
and thus $(\hat{Y}_{\hat{t}})$ is an $\{\hat{\mathscr{F}}_{\hat{t}}\}_{\hat{t}}$-martingale.
But this means that $(\e^{\alpha\hat{t}}\tilde{\mathcal{Y}}_{\uu,\vv}(\hat{h}_{\uu,\vv;\hat{t}}))_{\hat{t}}$
is also an $\{\hat{\mathscr{F}}_{\hat{t}}\}$-martingale, as claimed.
\end{proof}

We dealt with the Laplacian term in the previous proof through the
following lemma.
\begin{lem}
\label{lem:B2-computation-1}Fix a deterministic positive function
$Z\in\mathcal{C}^{2}([0,L])$ such that
\begin{equation}
Z'(0)=Z'(L)=0.\label{eq:ZNeumann}
\end{equation}
Let $B$ be a standard Brownian motion on $[0,L]$ with $B(0)=0$,
and define $B_{\mathsf{v}}(x)=B(x)+\mathsf{v}x$ as in \zcref{eq:Bvdef}.
For $y\in[0,L]$, define
\begin{equation}
I_{k}(y)\coloneqq\int^{y}_{0}Z^{(k)}(x)\e^{B_{\mathsf{v}}(x)}\,\dif x,\qquad k=0,1,2,\label{eq:Ikdef-1}
\end{equation}
where $Z^{(k)}$ denotes the $k$th derivative of $Z$. Then we have,
for any $\uu,\vv\in\mathbb{R}$, that
\begin{align}
\mathrm{E}_{B}\left[I_{0}(L)^{-\uu-\vv-1}I_{2}(L)\right] & =-\left(\vv+\oh\right)Z(L)\mathrm{E}_{B}\left[I_{0}(L)^{-\uu-\vv-1}\e^{B_{\vv}(L)}\right]-\left(\uu+\oh\right)Z(0)\mathrm{E}_{B}\left[I_{0}(L)^{-\uu-\vv-1}\right]\nonumber \\
 & \qquad+\left(\frac{1}{3}\uu^{2}+\frac{1}{3}\vv^{2}-\frac{1}{3}\uu\vv-\frac{1}{12}\right)\mathrm{E}_{B}\left[I_{0}(L)^{-\uu-\vv}\right]\nonumber \\
 & \qquad+(\uu+\vv+1)\mathrm{E}_{B}\left[I_{0}(L)^{-\uu-\vv-2}\int^{L}_{0}Z(y)^{2}\e^{2B_{\vv}(y)}\,\dif y\right].\label{eq:key-comp}
\end{align}
\end{lem}

\begin{proof}
The proof consists of several steps. Since $B$ is a standard Brownian
motion, in this proof we will treat $x$ as the time variable and
apply Itô's formula to rewrite the integral \zcref{eq:Ikdef-1}. Each
application of Itô's formula reduces the order of derivatives while
introducing an additional Itô integral term. We deal with the Itô
integral terms using a Gaussian integration by parts \cite[Section 1.3.3]{nualart:2006:malliavin}.
\begin{thmstepnv}
\item \emph{Itô formula.} For $k\in\{1,2\}$, we apply Itô's formula to
obtain that
\begin{align}
I_{k}(y)\ovset{\zcref{eq:Ikdef-1}} & =\int^{y}_{0}Z^{(k)}(x)\e^{B_{\mathsf{v}}(x)}\,\dif x\nonumber \\
 & =Z^{(k-1)}(y)\e^{B_{\mathsf{v}}(y)}-Z^{(k-1)}(0)-\int^{y}_{0}Z^{(k-1)}(x)\e^{B_{\mathsf{v}}(x)}\,\dif B_{\mathsf{v}}(x)-\frac{1}{2}\int^{y}_{0}Z^{(k-1)}(x)\e^{B_{\mathsf{v}}(x)}\,\dif x\nonumber \\
\ovset{\zcref{eq:Bvdef}} & =Z^{(k-1)}(y)\e^{B_{\mathsf{v}}(y)}-Z^{(k-1)}(0)-\int^{y}_{0}Z^{(k-1)}(x)\e^{B_{\mathsf{v}}(x)}\,\dif B(x)-\left(\mathsf{v}+\oh\right)I_{k-1}(y).\label{eq:IkIto-1}
\end{align}
\item \emph{Reducing }$I_{2}(L)$. Continuing from \zcref{eq:IkIto-1} with
$k=2$ and applying \zcref{eq:ZNeumann}, we obtain
\begin{align}
I_{2}(L) & =-\int^{L}_{0}Z'(x)\e^{B_{\mathsf{v}}(x)}\,\dif B(x)-\left(\mathsf{v}+\oh\right)I_{1}(L).\label{eq:I2Lreduce-1}
\end{align}
Using Gaussian integration by parts, for any $k\in\{0,1\}$ we have
\begin{align}
\mathrm{E}_{B} & \left[I_{0}(L)^{-\mathsf{u}-\mathsf{v}-1}\int^{L}_{0}Z^{(k)}(x)\e^{B_{\mathsf{v}}(x)}\,\dif B(x)\right]\nonumber \\
 & =-(\mathsf{u}+\mathsf{v}+1)I_{0}(L)^{-\mathsf{u}-\mathsf{v}-2}\mathrm{E}_{B}\left[\int^{L}_{0}\int^{L}_{0}Z^{(k)}(x)Z(y)\e^{B_{\mathsf{v}}(x)}\e^{B_{\mathsf{v}}(y)}\mathbf{1}\{x\le y\}\,\dif y\,\dif x\right]\nonumber \\
 & =-(\mathsf{u}+\mathsf{v}+1)I_{0}(L)^{-\mathsf{u}-\mathsf{v}-2}\mathrm{E}_{B}\left[\int^{L}_{0}Z(y)\e^{B_{\mathsf{v}}(y)}I_{k}(y)\,\dif y\right].\label{eq:IBP-1}
\end{align}
Using \zcref{eq:I2Lreduce-1} along with \zcref{eq:IBP-1} with $k=1$,
we obtain
\begin{align}
\mathrm{E}_{B}[I_{0}(L)^{-\mathsf{u}-\mathsf{v}-1}I_{2}(L)] & =-\left(\mathsf{v}+\oh\right)\mathrm{E}_{B}\left[I_{0}(L)^{-\mathsf{u}-\mathsf{v}-1}I_{1}(L)\right]\nonumber \\
 & \qquad+(\mathsf{u}+\mathsf{v}+1)\mathrm{E}_{B}\left[I_{0}(L)^{-\mathsf{u}-\mathsf{v}-2}\int^{L}_{0}Z(y)\e^{B_{\mathsf{v}}(y)}I_{1}(y)\,\dif y\right].\label{eq:EI2thing-1}
\end{align}
\item \emph{Reducing $I_{1}$.} By \zcref{eq:IkIto-1} with $k=1$, we obtain
\begin{equation}
I_{1}(y)=Z(y)\e^{B_{\mathsf{v}}(y)}-Z(0)-\int^{y}_{0}Z(x)\e^{B_{\mathsf{v}}(x)}\,\dif B(x)-\left(\mathsf{v}+\oh\right)I_{0}(y).\label{eq:I1expand-1}
\end{equation}
This means that
\begin{equation}
\mathrm{E}_{B}\left[I_{0}(L)^{-\mathsf{u}-\mathsf{v}-1}I_{1}(L)\right]=\mathrm{E}_{B}\left[I_{0}(L)^{-\mathsf{u}-\mathsf{v}-1}\left(Z(L)\e^{B_{\mathsf{v}}(L)}-Z(0)-\int^{L}_{0}Z(x)\e^{B_{\mathsf{v}}(x)}\,\dif B(x)-\left(\mathsf{v}+\oh\right)I_{0}(L)\right)\right].\label{eq:EI0Luv1I1L-1}
\end{equation}
For the Itô integral term, using \zcref{eq:IBP-1} with $k=0$, we
get
\begin{align*}
\mathrm{E}_{B}\left[I_{0}(L)^{-\mathsf{u}-\mathsf{v}-1}\int^{L}_{0}Z(x)\e^{B_{\mathsf{v}}(x)}\,\dif B(x)\right] & =-(\mathsf{u}+\mathsf{v}+1)\mathrm{E}_{B}\left[I_{0}(L)^{-\mathsf{u}-\mathsf{v}-2}\int^{L}_{0}Z(y)\e^{B_{\mathsf{v}}(y)}I_{0}(y)\,\dif y\right]\\
 & =-\frac{1}{2}(\mathsf{u}+\mathsf{v}+1)\mathrm{E}_{B}\left[I_{0}(L)^{-\mathsf{u}-\mathsf{v}}\right].
\end{align*}
Using this in \zcref{eq:EI0Luv1I1L-1}, we get
\begin{align}
\mathrm{E}_{B}\left[I_{0}(L)^{-\mathsf{u}-\mathsf{v}-1}I_{1}(L)\right] & =\mathrm{E}_{B}\left[I_{0}(L)^{-\mathsf{u}-\mathsf{v}-1}\left(Z(L)\e^{B_{\mathsf{v}}(L)}-Z(0)\right)\right]\nonumber \\
 & \qquad+\frac{1}{2}(\mathsf{u}+\mathsf{v}+1)\mathrm{E}_{B}\left[I_{0}(L)^{-\mathsf{u}-\mathsf{v}}\right]-\left(\mathsf{v}+\oh\right)\mathrm{E}_{B}\left[I_{0}(L)^{-\mathsf{u}-\mathsf{v}}\right]\nonumber \\
 & =\mathrm{E}_{B}\left[I_{0}(L)^{-\mathsf{u}-\mathsf{v}-1}\left(Z(L)\e^{B_{\mathsf{v}}(L)}-Z(0)\right)\right]+\frac{1}{2}(\mathsf{u}-\mathsf{v})\mathrm{E}_{B}\left[I_{0}(L)^{-\mathsf{u}-\mathsf{v}}\right].\label{eq:firstline-1}
\end{align}
\item \emph{The second line of \zcref{eq:EI2thing-1}.} Using \zcref{eq:I1expand-1},
we can write 
\begin{align}
\mathrm{E}_{B} & \left[I_{0}(L)^{-\mathsf{u}-\mathsf{v}-2}\int^{L}_{0}Z(y)\e^{B_{\mathsf{v}}(y)}I_{1}(y)\,\dif y\right]\nonumber \\
 & =\mathrm{E}_{B}\left[I_{0}(L)^{-\mathsf{u}-\mathsf{v}-2}\int^{L}_{0}Z(y)^{2}\e^{2B_{\mathsf{v}}(y)}\,\dif y\right]-Z(0)\mathrm{E}_{B}\left[I_{0}(L)^{-\mathsf{u}-\mathsf{v}-1}\right]-\frac{1}{2}\left(\mathsf{v}+\oh\right)\mathrm{E}_{B}\left[I_{0}(L)^{-\mathsf{u}-\mathsf{v}}\right]\nonumber \\
 & \qquad-\mathrm{E}_{B}\left[I_{0}(L)^{-\mathsf{u}-\mathsf{v}-2}\int^{L}_{0}Z(y)\e^{B_{\mathsf{v}}(y)}\left(\int^{y}_{0}Z(x)\e^{B_{\mathsf{v}}(x)}\,\dif B(x)\right)\,\dif y\right].\label{eq:secondline-EI2thing-2}
\end{align}
We can evaluate the last term, using Gaussian integration by parts
again, as
\begin{align*}
 & \int^{L}_{0}Z(y)\mathrm{E}_{B}\left[I_{0}(L)^{-\mathsf{u}-\mathsf{v}-2}\e^{B_{\mathsf{v}}(y)}\int^{y}_{0}Z(x)\e^{B_{\mathsf{v}}(x)}\,\dif B(x)\right]\,\dif y\\
 & \quad=(-\mathsf{u}-\mathsf{v}-2)\int^{L}_{0}Z(y)\mathrm{E}_{B}\left[I_{0}(L)^{-\mathsf{u}-\mathsf{v}-3}\e^{B_{\mathsf{v}}(y)}\int^{y}_{0}\left(\int^{L}_{x}Z(z)\e^{B_{\mathsf{v}}(z)}\,\dif z\right)Z(x)\e^{B_{\mathsf{v}}(x)}\,\dif x\right]\,\dif y\\
 & \qquad+\int^{L}_{0}Z(y)\mathrm{E}_{B}\left[I_{0}(L)^{-\mathsf{u}-\mathsf{v}-2}\e^{B_{\mathsf{v}}(y)}\int^{y}_{0}Z(x)\e^{B_{\mathsf{v}}(x)}\,\dif x\right]\,\dif y\\
 & \quad=-(\mathsf{u}+\mathsf{v}+2)\mathrm{E}_{B}\left[I_{0}(L)^{-\mathsf{u}-\mathsf{v}-3}\iiint_{0\le x_{1}\le x_{2}\wedge x_{3}\le L}\prod^{3}_{i=1}\left(Z(x_{i})\e^{B_{\vv}(x_{i})}\right)\,\dif x_{1}\,\dif x_{2}\,\dif x_{3}\right]+\frac{1}{2}\mathrm{E}_{B}\left[I_{0}(L)^{-\mathsf{u}-\mathsf{v}}\right]\\
 & \quad=\left(\frac{1}{2}-\frac{1}{3}(\mathsf{u}+\mathsf{v}+2)\right)\mathrm{E}_{B}\left[I_{0}(L)^{-\mathsf{u}-\mathsf{v}}\right].
\end{align*}
Using the above equation in \zcref{eq:secondline-EI2thing-2}, we
get
\begin{align}
\mathbb{E}\left[I_{0}(L)^{-\mathsf{u}-\mathsf{v}-2}\int^{L}_{0}Z(y)\e^{B_{\mathsf{v}}(y)}I_{1}(y)\,\dif y\right] & =\mathrm{E}_{B}\left[I_{0}(L)^{-\mathsf{u}-\mathsf{v}-2}\int^{L}_{0}Z(y)^{2}\e^{2B_{\mathsf{v}}(y)}\,\dif y\right]-Z(0)\mathrm{E}_{B}\left[I_{0}(L)^{-\mathsf{u}-\mathsf{v}-1}\right]\nonumber \\
 & \qquad+\left(\frac{1}{3}(\mathsf{u}+\mathsf{v}+2)-\frac{1}{2}-\frac{1}{2}\left(\mathsf{v}+\oh\right)\right)\mathbb{E}\left[I_{0}(L)^{-\mathsf{u}-\mathsf{v}}\right].\label{eq:secondline-EI2thing-1-1}
\end{align}
\item \emph{Putting things together.} Using \zcref{eq:firstline-1,eq:secondline-EI2thing-1-1}
in \zcref{eq:EI2thing-1}, we obtain
\begin{align*}
\mathrm{E}_{B}\left[I_{0}(L)^{-\uu-\vv-1}I_{2}(L)\right] & =-\left(\vv+\oh\right)\left(\mathrm{E}_{B}\left[I_{0}(L)^{-\uu-\vv-1}\left(Z(L)\e^{B_{\vv}(L)}-Z(0)\right)\right]+\frac{1}{2}(\uu-\vv)\mathrm{E}_{B}[I_{0}(L)^{-\uu-\vv}]\right)\\
 & \qquad+(\uu+\vv+1)\left(\mathrm{E}_{B}\left[I_{0}(L)^{-\uu-\vv-2}\int^{L}_{0}Z(y)^{2}\e^{2B_{\vv}(y)}\,\dif y\right]-Z(0)\mathrm{E}_{B}\left[I_{0}(L)^{-\uu-\vv-1}\right]\right)\\
 & \qquad+(\uu+\vv+1)\left(\frac{1}{3}(\uu+\vv+2)-\frac{1}{2}-\frac{1}{2}\left(\vv+\frac{1}{2}\right)\right)\mathrm{E}_{B}\left[I_{0}(L)^{-\uu-\vv}\right].
\end{align*}
Then \zcref{eq:key-comp} follows by algebra.\qedhere
\end{thmstepnv}
\end{proof}

\section{Regularity structure for the KPZ equation with boundary potentials\label{sec:Hairer-theory}}

To prove \zcref{prop:boundaryterm}, we will use the theory of regularity
structures developed in \cite{hairer:2014:theory}. The KPZ equation
with Neumann boundary conditions was previously studied using regularity
structures in \cite{gerencer:hairer:2019:singular}. However, that
work, while more general, does not approach the boundary condition
via boundary potentials as we do in the present work, and also does
not obtain control of the solutions close to the boundary sufficient
for our purposes. (See the beginning of \zcref{sec:BPHZ} for further
discussion of these issues.) Therefore, we will use a slightly different
regularity structure to perform our analysis, in particular one which
features an additional ``noise''-type term representing the boundary
potential, and the stochastic estimates on the model (performed in
\zcref{sec:BPHZ} must be performed more carefully). The primary reference
for the theory of regularity structures remains the original paper
\cite{hairer:2014:theory}; we also refer the reader to the surveys
\cite{friz:hairer:2020:course,bailleul:hoshino:2025:tourists} and
to \cite{hairer:quastel:2018:class} for an explicit description of
the construction of the regularity structure describing the periodic
KPZ equation.

\subsection{The regularity structure\label{subsec:regularity-structure}}

To define a regularity structure in the sense of \cites[Defn.~2.1]{hairer:2014:theory},
we must build an index set $A\subseteq\mathbb{R}$ of \emph{homogeneities}
(assumed the locally finite, bounded below, and containing $0$);
an $A$-graded vector space $\mathcal{T}=\bigoplus_{\alpha\in A}\mathcal{T}_{\alpha}$,
called the \emph{model space}, with $\mathcal{T}_{0}=\langle\rsone\rangle\cong\mathbb{R}$;
and a group $\mathcal{G}$ of linear operators on $\mathcal{T}$,
called the \emph{structure group}, such that $\Gamma\rsone=\rsone$
and $\Gamma\tau-\tau\in\bigoplus_{\beta<\alpha}\mathcal{T}_{\beta}$
whenever $\Gamma\in\mathcal{G}$, $\alpha\in A$, and $\tau\in\mathcal{T}_{\alpha}$.
Because the Neumann boundary conditions also play an important role
in our study, we also introduce immediately the group $\mathscr{S}$
of Euclidean isometries of $\mathbb{R}$ generated by 
\begin{equation}
\sigma_{\mathrm{refl}}(x)\coloneqq-x\qquad\text{and}\qquad\sigma_{\mathrm{trans}}(x)\coloneqq x+2L.\label{eq:sym_group_generators}
\end{equation}
We use the usual notation $\sigma^{*}f=f\circ\sigma^{-1}$ and note
that $\sigma^{*}f=f$ for all $\sigma\in\mathscr{S}$ if and only
if $f$ is even and $2L$-periodic.

\subsubsection{Model space}

The model space $\mathcal{T}$ is spanned by a countable basis $\mathsf{T}$
of symbols, most of which are represented by trees representing successive
terms of the Wild expansion of the KPZ equation. Our construction
will exactly mirror the usual construction of the regularity structure
for the KPZ equation, as described in \cites[§15.2]{friz:hairer:2020:course}[§3.1]{hairer:quastel:2018:class},
except that we will include an additional forcing term to describe
the addition of the boundary potential $\varphi^{\eps}_{\uu,\vv}$
on the right side of \zcref{eq:hepszeta-eqn}.

It will be useful to keep track of several different types of basis
elements, so we decompose
\[
\mathsf{T}=\mathsf{T}_{\mathrm{poly}}\sqcup\mathsf{T}_{\bullet}\sqcup\mathsf{T}_{\mathcal{I}}\sqcup\mathsf{T}_{\mathcal{I}'}\sqcup\mathsf{T}_{\star}
\]
and define 
\[
\mathcal{T}=\langle\mathsf{T}\rangle\qquad\text{and}\qquad\mathcal{T}_{\square}=\langle\mathsf{T}_{\square}\rangle\text{ for }\square\in\{\mathrm{poly},\bullet,\mathcal{I},\mathcal{I}',\star\},
\]
where, for a finite set $\mathsf{S}$, $\langle\mathsf{S}\rangle$
denotes the $\mathbb{R}$-vector space of formal $\mathbb{R}$-linear
combinations of elements of $\mathsf{S}$. Thus, our model space
$\mathcal{T}$ is decomposed into subspaces of the form
\[
\mathcal{T}=\mathcal{T}_{\mathrm{poly}}\oplus\mathcal{T}_{\bullet}\oplus\mathcal{T}_{\mathcal{I}}\oplus\mathcal{T}_{\mathcal{I}'}\oplus\mathcal{T}_{\star}.
\]
We also equip each basis $\tau\in\mathsf{T}$ with a homogeneity,
denoted $|\tau|$ (which despite the notation need not be positive),
and we will ultimately define
\begin{equation}
A\coloneqq\{|\tau|\st\tau\in\mathsf{T}\}.\label{eq:Adef}
\end{equation}
In defining the homogeneities, we will make use of a small parameter
\begin{equation}
\kappa\in(0,\nicefrac{1}{100})\setminus\mathbb{Q},\label{eq:kappaprop}
\end{equation}
which we hereby fix once and for all.

Now we enumerate the basis elements.

\paragraph*{Polynomial terms}

We first define $\mathsf{T}_{\mathrm{poly}}$ to be the usual polynomial
basis in two variables, denoted $\rst$ and $\rsx$, namely
\[
\mathsf{T}_{\mathrm{poly}}=\{\rstn{n_{1}}\rsxn{n_{2}}\st n_{1},n_{2}\in\mathbb{Z}_{\ge0}\}.
\]
In practice we will not use the time variable in the expansion;
see \zcref{eq:Tstar} below. We abbreviate $\rsone\coloneqq\rstn 0\rsxn 0$
and $\rsx\coloneqq\rstn 0\rsxn 1$. We define 
\[
|\rstn{n_{1}}\rsxn{n_{2}}|\coloneqq2n_{1}+n_{2},
\]
reflecting the parabolic scaling $\mathfrak{s}$ of the problem\@.
In particular we have 
\[
|\rsone|=0\qquad\text{and}\qquad|\rsx|=1.
\]

\paragraph*{Basic forcing terms}

We define \nomenclature{$\rsnoise,\rspotential$}{Noise, boundary potential}
\begin{equation}
\mathsf{T}_{\bullet}\coloneqq\{\rsnoise,\rspotential\},\label{eq:Tdotdef}
\end{equation}
where $\rsnoise$ represents the random noise $(\dif W^{\zeta}_{t}(x))$
and $\rspotential$ represents the boundary potential $\varphi^{\eps}_{\uu,\vv}$
on the right side of \zcref{eq:hepszeta-eqn}. We set 
\[
|\rsnoise|\coloneqq-\thrh-\kappa\qquad\text{and}\qquad|\rspotential|\coloneqq-1-\kappa.
\]
This reflects the fact that one-dimensional space-time white noise
lives in any negative Hölder space of regularity strictly less than
$-\thrh$, while a spatial Dirac function lives in the negative Hölder
space of regularity $-1$ and hence in any space of lower regularity
as well. We do not precisely set $|\rspotential|$ as $-1$ to avoid
complications stemming from terms of integer regularity.
\begin{rem}
The reader will notice that, in addition to the terms $\dif W^{\zeta}_{t}(x)$
and $\varphi^{\eps}_{\uu,\vv}$, the right side of \zcref{eq:hepszeta-eqn}
also features the forcing term $\frac{1}{4}\Sh^{\zeta/2}_{L}(x)-\frac{1}{2}\Sh^{\zeta}_{2L}(0)$.
This term is not represented via the regularity structure but will
instead arise in the course of the renormalization procedure below;
see \zcref{subsec:Renormalized-canonical-lifts,eq:Echerryrr}.
\end{rem}

\paragraph*{Inductive construction}

Now we turn to the construction of terms arising in the iterative
Wild expansion of the solution to \zcref{eq:hepszeta}. This is of
course done iteratively. We construct these terms by applying the
following two rules repeatedly:
\begin{enumerate}
[label=(R\arabic*),ref=R\arabic*]
\item \label{enu:integrate}If $\tau\in\mathsf{T}_{\bullet}\cup\mathsf{T}_{\mathcal{I}}\cup\mathsf{T}_{\mathcal{I}'}\cup\mathsf{T}_{\star}$,
then we add new symbols $\mathcal{I}\tau$ and $\mathcal{I}'\tau$
to $\mathsf{T}_{\mathcal{I}}$ and $\mathsf{T}_{\mathcal{I}'}$, respectively.
These are abstract representations of the heat kernel and its derivative,
respectively, convolved with $\tau$. We set 
\[
|\mathcal{I}\tau|\coloneqq|\tau|+2\qquad\text{and}\qquad|\mathcal{I}'\tau|\coloneqq|\tau|+1.
\]
\item \label{enu:multiply}If $\tau_{1}\in\mathsf{T}_{\mathcal{I}'}\cup\mathsf{T}_{\mathrm{poly}}\setminus\{\rsone\}$
and $\tau_{2}\in\mathsf{T}_{\mathcal{I}'}$, or vice versa, then we
add a new symbol $\tau_{1}\tau_{2}$ to $\mathsf{T}_{\star}$, which
is an abstract representation of the product of $\tau_{1}$ and $\tau_{2}$.
We impose the relation that $\tau_{1}\tau_{2}=\tau_{2}\tau_{1}$.
We define the homogeneity
\[
|\tau_{1}\tau_{2}|=|\tau_{1}|+|\tau_{2}|.
\]
\end{enumerate}
The reason for the slightly strange restrictions on $\tau_{1}$ and
$\tau_{2}$ in \zcref{enu:multiply} is that if $\tau_{1}$ and $\tau_{2}$
are both polynomials, then their product already exists in $\mathsf{T}_{\mathrm{poly}}$,
and if $\tau_{1}=\rsone$, then we do not want to add another symbol
$\rsone\tau_{2}$ since we should have $\rsone\tau_{2}=\tau_{2}$.
See \zcref{subsec:Operations-on-the} below.

It is straightforward to check that this procedure generates a countably
infinite set of symbols, with a locally finite set of homogeneities
that is diverging to $+\infty$, so the first condition in \cites[Defn.~2.1]{hairer:2014:theory}
is satisfied.

\paragraph*{Homogeneity subspaces}

We also define, for $\alpha\in\mathbb{R}$,
\[
\mathsf{T}_{\alpha}\coloneqq\{\tau\in\mathsf{T}\st|\tau|=\alpha\},\qquad\mathcal{T}_{\alpha}\coloneqq\langle\mathsf{T}_{\alpha}\rangle,\qquad\mathsf{T}_{<\alpha}\coloneqq\bigcup_{\beta<\alpha}\mathsf{T}_{\beta}\qquad\text{and}\qquad\mathcal{T}_{<\alpha}\coloneqq\langle\mathsf{T}_{<\alpha}\rangle,
\]
and similarly with ``$<$'' replaced by ``$\le$'', ``$>$'',
or ``$\ge$''. Similarly, for $\square\in\{\mathrm{poly},\bullet,\mathcal{I},\mathcal{I}',\star\}$
and $\alpha\in\mathbb{R}$, we define
\[
\mathsf{T}_{\square;\alpha}\coloneqq\{\tau\in\mathsf{T}_{\square}\st|\tau|=\alpha\},\qquad\mathcal{T}_{\square;\alpha}\coloneqq\langle\mathsf{T}_{\square;\alpha}\rangle,
\]
and similarly with ``$\alpha$'' replaced by ``$<\alpha$'', ``$\le\alpha$'',
etc. In particular, this gives $\mathcal{T}$ the structure of a graded
vector space. We also define
\[
A_{<\alpha}\coloneqq\{\ell\in A\st\ell<\alpha\}
\]
and similarly with ``$<\alpha$'' replaced by ``$\le\alpha$'',
etc. For $\tau\in\mathcal{T}$ and $\ell\in A$, we define $\|\tau\|_{\ell}$
to be the norm of the projection of $\tau$ onto $\mathcal{T}_{\ell}$,
the choice of norm being irrelevant since $\mathcal{T}_{\ell}$ is
finite-dimensional.

\subsubsection{Structure group}

To complete the definition of the regularity structure, we must define
the structure group $\mathcal{G}$. Again, the construction here is
the same as that of \cites[§15.3]{friz:hairer:2020:course} except
that we add the potential term $\rspotential$. For completeness we
recall it here briefly. We define $\mathcal{T}^{+}$ to be the free
commutative algebra generated by $\rst$, $\rsx$, and the set of
formal expressions 
\[
\{\mathcal{J}_{n_{1},n_{2}}(\tau)\st n_{1},n_{2}\in\mathbb{Z},\tau\in\mathsf{T}_{\bullet}\sqcup\mathsf{T}_{\mathcal{I}}\sqcup\mathsf{T}_{\mathcal{I}'}\sqcup\mathsf{T}_{\star},|\tau|+2>2n_{1}+n_{2}\}.
\]
We will abuse notation and also write $\rsone$ for the identity element
of $\mathcal{T}^{+}$. Then we define a linear map $\varDelta\colon\mathcal{T}\to\mathcal{T}\otimes\mathcal{T}^{+}$
by 
\[
\varDelta\rsone=\rsone\otimes\rsone,\quad\varDelta\rsnoise=\rsnoise\otimes\rsone,\quad\varDelta\rspotential=\rspotential\otimes\rsone,\quad\varDelta\rst=\rst\otimes\rsone+\rsone\otimes\rst,\quad\varDelta\rsx=\rsx\otimes\rsone+\rsone\otimes\rsx,
\]
and via the recursion
\begin{align}
\varDelta(\tau_{1}\tau_{2}) & \coloneqq(\varDelta\tau_{1})(\varDelta\tau_{2});\label{eq:Delta-prod-1}\\
\varDelta\mathcal{I}(\tau) & \coloneqq(\mathcal{I}\otimes\id)\varDelta\tau+\sum_{\substack{n_{1},n_{2}\in\mathbb{Z}_{\ge0}\\
2n_{1}+n_{2}<|\tau|+2
}
}\frac{\rstn{n_{1}}\rsxn{n_{2}}}{n_{1}!n_{2}!}\otimes\mathcal{J}_{n_{1},n_{2}}(\tau);\label{eq:Delta-I-1}\\
\Delta\mathcal{I}'(\tau) & \coloneqq(\mathcal{I}'\otimes\id)\varDelta\tau+\sum_{\substack{n_{1},n_{2}\in\mathbb{Z}_{\ge0}\\
2n_{1}+n_{2}<|\tau|+1
}
}\frac{\rstn{n_{1}}\rsxn{n_{2}}}{n_{1}!n_{2}!}\otimes\mathcal{J}_{n_{1},n_{2}+1}(\tau).\label{eq:Delta-Iprime-1}
\end{align}

We let 
\begin{equation}
\mathcal{G}^{+}\text{ denote the set of algebra homomorphisms }f\colon\mathcal{T}^{+}\to\mathbb{R}\label{eq:Gplusdef}
\end{equation}
and let 
\begin{equation}
\mathcal{G}=\{\Gamma_{f}\colon f\in\mathcal{G}_{+}\},\qquad\text{where}\qquad\Gamma_{f}\tau\coloneqq(\id\otimes f)\Delta\tau.\label{eq:GGammadef}
\end{equation}
It can be checked that $\mathcal{G}$ is a group under composition,
and satisfies the other required conditions, following the proof given
in \cites[§8.1]{hairer:2014:theory}. The only difference in our
setting is the addition of $\rspotential$, which can be considered
as a second noise (with a different homogeneity) and handled in the
same way as pointed out in \cite[Rmk.~8.9]{hairer:2014:theory}.

\subsubsection{\label{subsec:Operations-on-the}Operations on the regularity structure}

The regularity structure $(A,\mathcal{T},\mathcal{G})$ is equipped
with product, integration, and differentiation maps, defined on sectors.
We first note that the subspaces $\mathcal{T}_{\mathrm{poly}}\oplus\mathcal{T}_{\mathcal{I}'}$
and $\mathcal{T}_{\mathrm{poly}}\oplus\mathcal{T}_{\mathcal{I}}$
are each sectors of the regularity structure $\mathcal{T}$, in the
sense of \cites[Defn.~2.5]{hairer:2014:theory}. The only condition
that requires checking is that these subspaces are closed under the
action of $\mathcal{G}$, but this is clear from the definitions \zcref[range]{eq:Delta-I-1,eq:GGammadef}.

\paragraph*{Product}

We have a product map $\star\colon(\mathcal{T}_{\mathrm{poly}}\oplus\mathcal{T}_{\mathcal{I}'})^{2}\to\mathcal{T}$.
This map is defined on $\mathsf{T}^{2}_{\mathrm{poly}}$ by the usual
polynomial multiplication, on $(\{\rsone\}\times\mathsf{T}_{\mathcal{I}'})\cup(\mathsf{T}_{\mathcal{I}'}\times\{\rsone\})$
by defining $\rsone\star\tau=\tau\star\rsone=\tau$, and on $(\mathsf{T}_{\mathcal{I}'}\cup\mathcal{T}_{\mathrm{poly}}\setminus\{\rsone\})\times\{\mathsf{T}_{\mathcal{I}'}\}$
and its symmetric set by putting $\tau_{1}\star\tau_{2}=\tau_{1}\tau_{2}$.
Then it is extended to all of $(\mathcal{T}_{\mathrm{poly}}\oplus\mathcal{T}_{\mathcal{I}'})^{2}$
by linearity. It is immediately clear that $\star$ is a product on
the regularity structure in the sense of \cites[Defn.~4.1]{hairer:2014:theory}.

\paragraph*{Integration}

We define an integration map $\mathcal{I}\colon\mathcal{T}\to\mathcal{T}_{\mathcal{I}}\subset\mathcal{T}$
by defining $\mathcal{I}(\tau)=\mathcal{I}\tau$ for $\tau\in\mathsf{T}_{\bullet}\cup\mathsf{T}_{\mathcal{I}}\cup\mathsf{T}_{\mathcal{I}'}\cup\mathsf{T}_{\star}$,
$\mathcal{I}\tau=0$ for $\tau\in\mathsf{T}_{\mathrm{poly}}$, and
then extending $\mathcal{I}$ to all of $\mathcal{T}$ by linearity.

\paragraph*{Differentiation}

Finally, we define a spatial differentiation map $\partial\colon\mathcal{T}_{\mathrm{poly}}\oplus\mathcal{T}_{\mathcal{I}}\to\mathcal{T}_{\mathrm{poly}}\oplus\mathcal{T}_{\mathcal{I}'}$
by setting $\partial(\rstn{n_{1}}\rsxn{n_{2}})=n_{2}\rstn{n_{1}}\rsxn{n_{2}-1}$,
$\partial(\mathcal{I}\tau)=\mathcal{I}'\tau$ on basis elements and
extending by linearity. It is not difficult to check that the map
$\partial$ is an abstract gradient in the sense of \cite[Defn.~5.25]{hairer:2014:theory},
the key point being that $\Delta\partial\tau=(\partial\otimes\id)\Delta\tau.$

\subsubsection{Symmetries}

We now discuss an important symmetry of the regularity structure,
reflecting the even/odd extension of the noise and solution corresponding
to the Neumann boundary conditions. We use the formalism introduced
in \cites[§3.6]{hairer:2014:theory}. We associate a ``parity''
$\sgn(\tau)\in\{\pm1\}$ to each basis element $\tau\in\mathsf{T}$
by 
\begin{equation}
\sgn(\rsnoise)=\sgn(\rspotential)=\sgn(\rsone)=\sgn(\rst)=1,\qquad\sgn(\rsx)=-1,\label{eq:sgnbasecase}
\end{equation}
and then impose the recursion
\begin{equation}
\sgn(\mathcal{I}\tau)=\sgn(\tau),\qquad\sgn(\mathcal{I}'\tau)=-\sgn(\tau),\qquad\sgn(\tau_{1}\tau_{2})=\sgn(\tau_{1})\sgn(\tau_{2}).\label{eq:sgnrecursion}
\end{equation}
Strictly speaking, we are imposing this relation only on basis elements
defined according to the recursion \zcref[range]{enu:integrate,enu:multiply},
but it is easy to check that in fact according to this definition,
$\sgn(\mathcal{I}\tau)=\sgn(\tau)$ for all $\tau\in\mathsf{T}$ and
$\sgn(\tau_{1}\star\tau_{2})=\sgn(\tau_{1})\sgn(\tau_{2})$ for all
$\tau_{1},\tau_{2}\in\mathsf{T}_{\mathrm{poly}}\cup\mathsf{T}_{\mathcal{I}'}$. 

Naturally, $\sgn(\tau)$ represents whether the object represented
by $\tau$ is even or odd; see \zcref{lem:reconstruct-symmetry} below.
Thus we define a representation of the symmetry group $\mathscr{S}$
(from \zcref{eq:sym_group_generators}) on $\mathcal{T}$ by defining,
for $\tau\in\underline{\mathcal{T}}$, 
\begin{equation}
\sigma_{\mathrm{trans}}\cdot\tau=\tau\qquad\text{and}\qquad\sigma_{\mathrm{refl}}\cdot\tau=\sgn(\tau)\tau,\label{eq:sigma-action}
\end{equation}
and extending each of these maps to $\mathcal{T}$ by linearity. It
is easy to check that these maps extend to a group representation
of $\mathscr{S}$. One can also readily check that 
\begin{equation}
\sigma_{\mathrm{refl}}\cdot(\partial\tau)=-\partial(\sigma_{\mathrm{refl}}\cdot\tau)\qquad\text{and}\qquad\sigma_{\mathrm{trans}}\cdot(\partial\tau)=\partial(\sigma_{\mathrm{trans}}\cdot\tau)\qquad\text{for }\tau\in\mathcal{T}_{\mathrm{poly}}\oplus\mathcal{T}_{\mathcal{I}},\label{eq:sigma-deriv-commute}
\end{equation}
as well as
\begin{equation}
\sigma\cdot(\tau_{1}\star\tau_{2})=(\sigma\cdot\tau_{1})\star(\sigma\cdot\tau_{2})\qquad\text{for }\tau_{1},\tau_{2}\in\mathcal{T}_{\mathrm{poly}}\oplus\mathcal{T}_{\mathcal{I}'}\text{ and }\sigma\in\mathscr{S}.\label{eq:sigma-product-commute}
\end{equation}

\subsubsection{Truncating the regularity structure\label{sec:trunc-model-space}}

It will turn out that we can do all of our computation in a truncated
regularity structure that includes just finitely many symbols. Namely,
we define
\begin{equation}
\check{\mathcal{T}}\coloneqq\langle\check{\mathsf{T}}\rangle,\qquad\check{\mathsf{T}}\coloneqq\bigcup_{\square\in\{\mathrm{poly},\bullet,\mathcal{I},\mathcal{I}',\star\}}\check{\mathsf{T}}_{\square},\label{eq:Tstar}
\end{equation}
where
\[
\check{\mathsf{T}}_{\square}\coloneqq\mathsf{T}_{\square;\le\gamma_{\square}},\qquad\gamma_{\square}\coloneqq\begin{cases}
1, & \text{if }\square=\mathrm{poly};\\
\infty, & \text{if }\square=\bullet;\\
\thrh, & \text{if }\square=\mathcal{I};\\
\oh, & \text{if }\square=\mathcal{I}';\\
0, & \text{if }\square=\star.
\end{cases}
\]
In other words, in the truncated regularity structure $\check{\mathcal{T}}$,
we ignore all polynomials of homogeneity greater than $1$, all elements
of $\mathcal{T}_{\bullet}$ of homogeneity greater than $\thrh$,
etc. It is straightforward to show that $\check{\mathcal{T}}$ defines
a new regularity structure when equipped with appropriate truncations
$\check{A}$ of $A$ and $\check{\mathcal{G}}$ of $\mathcal{G}$.
We define the projection 
\[
\check{\varpi}\colon\mathcal{T}\to\check{\mathcal{T}}
\]
according to the inclusion of bases $\check{\mathsf{T}}\subset\mathsf{T}$.

\subsubsection{Tree diagrams\label{subsec:Tree-diagrams}}

\nomenclature{$\rsballoonn,\rslollipopn$}{Integration against basic kernels}We
would like to write down all of the symbols in $\check{\mathsf{T}}$
explicitly. To facilitate this, we define a more suggestive representation
of these symbols using tree diagrams. This notation will be further
developed later on, in particular in \zcref{subsec:feynman-diagrams}
below; see p.~\pageref{tree-notation-index} for an index of the
tree notation introduced in the course of the paper. We will represent
$\mathcal{I}$ by a squiggly line $\rsballoonn$ and $\partial\mathcal{I}$
by a straight line $\rslollipopn$, and we will represent multiplication
by joining two trees at their roots. Thus, for example, we have $\rsballoonb=\mathcal{I}\rspotential$,
$\rscherryrr=(\mathcal{I}'\rsnoise)^{2}$, and $\rselkrrb=(\mathcal{I}'\rspotential)(\mathcal{I}'((\mathcal{I}'\rsnoise)^{2}))$.
We do not use a tree notation for polynomial symbols, which we write
simply as $\rsone$ and $\rsx$. Nor do we have a tree notation for
symbols that involve a product of a polynomial with a tree-represented
object, but it turns out that no symbol of this form occurs in $\check{\mathsf{T}}$
anyway. With this notation, we record the elements of $\check{\mathsf{T}}$,
along with their homogeneities, parities, coproducts, integration
maps, and derivative maps, in \zcref{tab:RS-table}. Similarly, we
record the table of values for $\star$, truncated by $\check{\varpi}$,
in \zcref{tab:multiplication-table}.
\begin{table}
\begin{centering}
\newcommand{\tablelabel}[2]{\hspace{-3em}\ldelim\lvert{#1}{*}[$\check{\mathsf{T}}_{#2}$ ]}%
\begin{tabular}{rccclcc}
\cmidrule{2-7}
 & $\tau$ & $|\tau|$ & $\sgn(\tau)$ & $\varDelta\tau-\tau\otimes\rsone$ & $\check{\varpi}(\mathcal{I}\tau)$ & $\partial\tau$\tabularnewline
\cmidrule{2-7}
\tablelabel{2}{\bullet} & $\rsnoise$ & $-\thrh-\kappa$ & $\hphantom{-}1$ & $0$ & $\rsballoonr$ & \tabularnewline
 & $\rspotential$ & $-1-\kappa$ & $\hphantom{-}1$ & $0$ & $\rsballoonb$ & \tabularnewline
\cmidrule{2-7}
\tablelabel{2}{\mathrm{poly}} & $\rsone$ & $0-0\kappa$ & $\hphantom{-}1$ & $0$ & $0$ & $0$\tabularnewline
 & $\rsx$ & $1-0\kappa$ & $-1$ & $\rsone\otimes\rsx$ & $0$ & $\rsone$\tabularnewline
\cmidrule{2-7}
\tablelabel{6}{\mathcal{I}} & $\rsballoonr$ & $\oh-\kappa$ & $\hphantom{-}1$ & $\rsone\otimes\mathcal{J}_{0,0}(\rsnoise)$ & $0$ & $\rslollipopr$\tabularnewline
 & $\rsicherryrr$ & $1-2\kappa$ & $\hphantom{-}1$ & $\rsone\otimes\mathcal{J}_{0,0}(\rscherryrr)$ & $0$ & $\rsipcherryrr$\tabularnewline
 & $\rsballoonb$ & $1-\kappa$ & $\hphantom{-}1$ & $\rsone\otimes\mathcal{J}_{0,0}(\rspotential)$ & $0$ & $\rslollipopb$\tabularnewline
 & $\rsielkrrr$ & $\thrh-3\kappa$ & $\hphantom{-}1$ & $\rsone\otimes\mathcal{J}_{0,0}(\rselkrrr)+\rsx\otimes\mathcal{J}_{0,1}(\rselkrrr)$ & $0$ & $\rsipelkrrr$\tabularnewline
 & $\rsicherryrb$ & $\thrh-2\kappa$ & $\hphantom{-}1$ & $\rsone\otimes\mathcal{J}_{0,0}(\rscherryrb)+\rsx\otimes\mathcal{J}_{0,1}(\rscherryrb)$ & $0$ & $\rsipcherryrb$\tabularnewline
 & $\rsilollipopr$ & $\thrh-\kappa$ & $-1$ & $\rsone\otimes\mathcal{J}_{0,0}(\rslollipopr)+\rsx\otimes\mathcal{J}_{0,1}(\rslollipopr)$ & $0$ & $\rsiplollipopr$\tabularnewline
\cmidrule{2-7}
\tablelabel{6}{\mathcal{I}'} & $\rslollipopr$ & $-\oh-\kappa$ & $-1$ & $0$ & $\rsilollipopr$ & \tabularnewline
 & $\rsipcherryrr$ & $0-2\kappa$ & $-1$ & $0$ & $0$ & \tabularnewline
 & $\rslollipopb$ & $0-\kappa$ & $-1$ & $0$ & $0$ & \tabularnewline
 & $\rsipelkrrr$ & $\oh-3\kappa$ & $-1$ & $\rsone\otimes\mathcal{J}_{0,1}(\rselkrrr)$ & $0$ & \tabularnewline
 & $\rsipcherryrb$ & $\oh-2\kappa$ & $-1$ & $\rsone\otimes\mathcal{J}_{0,1}(\rscherryrb)$ & $0$ & \tabularnewline
 & $\rsiplollipopr$ & $\oh-\kappa$ & $\hphantom{-}1$ & $\rsone\otimes\mathcal{J}_{0,1}(\rslollipopr)$ & $0$ & \tabularnewline
\cmidrule{2-7}
\tablelabel{9}{\star} & $\rscherryrr$ & $-1-2\kappa$ & $\hphantom{-}1$ & $0$ & $\rsicherryrr$ & \tabularnewline
 & $\rscherryrb$ & $-\oh-2\kappa$ & $\hphantom{-}1$ & $0$ & $\rsicherryrb$ & \tabularnewline
 & $\rselkrrr$ & $-\oh-3\kappa$ & $\hphantom{-}1$ & $0$ & $\rsielkrrr$ & \tabularnewline
 & $\rscandelabrarrrr$ & $0-4\kappa$ & $\hphantom{-}1$ & $0$ & $0$ & \tabularnewline
 & $\rsmooserrrr$ & $0-4\kappa$ & $\hphantom{-}1$ & $\rslollipopr\otimes\mathcal{J}_{0,1}(\rselkrrr)$ & $0$ & \tabularnewline
 & $\rselkrbr$ & $0-3\kappa$ & $\hphantom{-}1$ & $\rslollipopr\otimes\mathcal{J}_{0,1}(\rscherryrb)$ & $0$ & \tabularnewline
 & $\rselkrrb$ & $0-3\kappa$ & $\hphantom{-}1$ & $0$ & $0$ & \tabularnewline
 & $\rscherrybb$ & $0-2\kappa$ & $\hphantom{-}1$ & $0$ & $0$ & \tabularnewline
 & $\rsclawrr$ & $0-2\kappa$ & $-1$ & $\rslollipopr\otimes\mathcal{J}_{0,1}(\rslollipopr)$ & $0$ & \tabularnewline
\cmidrule{2-7}
\end{tabular}
\par\end{centering}
\caption[Elements of $\check{\mathsf{T}}$, homogeneities, coproducts, integrations,
and derivatives]{\label{tab:RS-table}The elements of $\check{\mathsf{T}}$. Each
basis element $\tau$ is listed along with its homogeneity $|\tau|$,
coproduct $\varDelta\tau$ (with $\tau\otimes\protect\rsone$, which
appears in each one, subtracted), integration $\mathcal{I}\tau$ (projected
back onto $\check{\mathcal{T}}$ via $\check{\varpi}$), and, where
applicable, derivative $\partial\tau$.}
\end{table}

\begin{table}
\begin{centering}
\begin{tabular}{c|cccccccc}
$\check{\varpi}\circ\star$ & $\rslollipopr$ & $\rsipcherryrr$ & $\rslollipopb$ & $\rsone$ & $\rsipelkrrr$ & $\rsipcherryrb$ & $\rsiplollipopr$ & $\rsx$\tabularnewline
\hline 
$\rslollipopr$ & $\rscherryrr$ & $\rselkrrr$ & $\rscherryrb$ & $\rslollipopr$ & $\rsmooserrrr$ & $\rselkrbr$ & $\rsclawrr$ & $0$\tabularnewline
$\rsipcherryrr$ & $\rselkrrr$ & $\rscandelabrarrrr$ & $\rselkrrb$ & $\rsipcherryrr$ & $0$ & $0$ & $0$ & $0$\tabularnewline
$\rslollipopb$ & $\rscherryrb$ & $\rselkrrb$ & $\rscherrybb$ & $\rslollipopb$ & $0$ & $0$ & $0$ & $0$\tabularnewline
$\rsone$ & $\rslollipopr$ & $\rsipcherryrr$ & $\rslollipopb$ & $\rsone$ & $\rsipelkrrr$ & $\rsipcherryrb$ & $\rsiplollipopr$ & $\rsx$\tabularnewline
$\rsipelkrrr$ & $\rsmooserrrr$ & $0$ & $0$ & $\rsipelkrrr$ & $0$ & $0$ & $0$ & $0$\tabularnewline
$\rsipcherryrb$ & $\rselkrbr$ & $0$ & $0$ & $\rsipcherryrb$ & $0$ & $0$ & $0$ & $0$\tabularnewline
$\rsiplollipopr$ & $\rsclawrr$ & $0$ & $0$ & $\rsiplollipopr$ & $0$ & $0$ & $0$ & $0$\tabularnewline
$\rsx$ & $0$ & $0$ & $0$ & $\rsx$ & $0$ & $0$ & $0$ & $0$\tabularnewline
\end{tabular}
\par\end{centering}
\caption[Multiplication table]{\label{tab:multiplication-table}Table of values of $\check{\varpi}(\tau_{1}\star\tau_{2})$
for $\tau_{1},\tau_{2}\in\check{\mathsf{T}}_{\mathrm{poly}}\cup\check{\mathsf{T}}_{\mathcal{I}'}$.}
\end{table}

\subsection{\label{subsec:The-kernels}The kernels}

Before we can define the model, we need to discuss the kernels that
we use to represent the integration maps. For two space-time functions
$f,g\colon\mathbb{R}^{2}\to\mathbb{R}$, we define the space-time
convolution
\[
(f\circledast g)_{t}(x)=\iint_{\mathbb{R}^{2}}f_{s}(y)g_{t-s}(x-y)\,\dif y\,\dif s
\]
whenever the right side is well-defined. By \cites[Lem.~5.24]{hairer:2014:theory},
we can find kernels $K$ and $\tilde{K}$ such that the heat kernel
$p_{t}(x)$ (recalling \zcref{eq:ptdef}) can be decomposed as
\begin{equation}
p_{t}(x)=K_{t}(x)+\tilde{K}_{t}(x),\label{eq:pissumofKs}
\end{equation}
$(t,x)\mapsto\tilde{K}_{t}(x)$ is smooth, $K_{t}(x)=p_{t}(x)$ whenever
$|t|^{\oh}+|x|\le\oh$, 
\begin{equation}
K_{t}(x)=0\qquad\text{whenever }|t|^{\oh}+|x|\ge1,\label{eq:Kcompactsupport}
\end{equation}
and
\begin{equation}
\int_{\mathbb{R}}K_{t}(x)\,\dif x=\int_{\mathbb{R}}xK_{t}(x)\,\dif x=0\qquad\text{for all }t>0.\label{eq:Kpolyvanish}
\end{equation}
Moreover, we can ensure that
\begin{equation}
K_{t}(x)=K_{t}(-x)\qquad\text{and}\qquad\tilde{K}_{t}(x)=\tilde{K}_{t}(-x)\qquad\text{for all }x\in\mathbb{R}.\label{eq:KKhateven}
\end{equation}
For notational convenience, we will use the notation
\[
p'\coloneqq\partial_{x}p,\qquad K'=\partial_{x}K,\qquad\text{and}\qquad\tilde{K}'=\partial_{x}\tilde{K}
\]
throughout.

\subsection{Canonical lifts\label{subsec:Canonical-lift}}

Before defining the model on our regularity structure, we define some
``canonical lifts'' (roughly in the sense of \cites[(15.12–14)]{friz:hairer:2020:course}).
These represent ``naïve'' reconstruction of the formal expressions
in the regularity structure $\mathcal{T}$. This is not to be confused
with the ``true'' reconstruction of a modeled distribution, discussed
in \zcref{subsec:Lift-of-the} below, but it will be used in the construction
of the model. The point is that the canonical lift does not include
recentering (also known as positive renormalization).

\subsubsection{Main canonical lift}

We let $\mathcal{D}'(\mathbb{R}^{2})$ denote the classical space
of distributions on $\mathbb{R}^{2}$. For $\eps,\zeta>0$, we define
a maps $\mathbf{\boldsymbol{\varPi}}^{\eps,\zeta}\colon\mathcal{T}\to\mathcal{D}'(\mathbb{R}^{2})$
by an inductive procedure, paralleling the steps described in \zcref{subsec:regularity-structure}.
We first define, for $(t,x)\in\mathbb{R}^{2}$,
\begin{equation}
\begin{aligned}\mathbf{\boldsymbol{\varPi}}^{\eps,\zeta}(\rsnoise)_{t}(x) & \coloneqq\dif W^{\zeta}_{t}(x), & \mathbf{\boldsymbol{\varPi}}^{\eps,\zeta}(\rspotential)_{t}(x) & \coloneqq\varphi^{\eps}_{\uu,\vv}(x),\\
\mathbf{\boldsymbol{\varPi}}^{\eps,\zeta}(\rsone)_{t}(x) & =1,\qquad\text{and} & \mathbf{\boldsymbol{\varPi}}^{\eps,\zeta}(\rsx)_{t}(x) & =x,
\end{aligned}
\label{eq:basic-reconstruction}
\end{equation}
with $\dif W^{\zeta}_{t}(x)$ and $\varphi^{\eps}_{\uu,\vv}(x)$ defined
in \zcref{eq:difWzeta,eq:varphiuvdef}, respectively.\footnote{We caution the reader that the symbol $\boldsymbol{\varPi}$ is not
to be confused with the symbol $\Pi$ used for the model in \zcref{subsec:Model}
below.} We then inductively define
\begin{subequations}
\label{eq:Qinductive}
\begin{gather}
\mathbf{\boldsymbol{\varPi}}^{\eps,\zeta}(\mathcal{I}\tau)_{t}(x)\coloneqq(K\circledast\mathbf{\boldsymbol{\varPi}}^{\eps,\zeta}(\tau))_{t}(x),\qquad\qquad\mathbf{\boldsymbol{\varPi}}^{\eps,\zeta}(\mathcal{I}'\tau)_{t}(x)\coloneqq(K'\circledast\mathbf{\boldsymbol{\varPi}}^{\eps,\zeta}(\tau))_{t}(x),\label{eq:QI}\\
\mathbf{\boldsymbol{\varPi}}^{\eps,\zeta}(\tau\rho)_{t}(x)\coloneqq\mathbf{\boldsymbol{\varPi}}^{\eps,\zeta}(\tau)_{t}(x)\cdot\mathbf{\boldsymbol{\varPi}}^{\eps,\zeta}(\rho)_{t}(x)\label{eq:Qprodderivs}
\end{gather}
\end{subequations}
 We note that, since $K$ is compactly supported, there is no issue
in defining the space-time convolutions in \zcref{eq:QI}. For notational
convenience going forward, we will denote
\begin{equation}
\boxop{\tau}^{\eps,\zeta}_{t}(x)\coloneqq\mathbf{\boldsymbol{\varPi}}^{\eps,\zeta}(\tau)_{t}(x).\label{eq:tauboxdef}
\end{equation}
When\nomenclature{$\boxop{\phantom{\rsmooserrrr}}$}{Canonical lift using $K$ kernel}
$\tau$ does not contain any $\rsnoise$ symbols or factors of $\rst$,
then this quantity is independent of $\zeta$ and $t$, so we will
abbreviate it by $\boxop{\tau}^{\eps}(x)$. Similarly, if $\tau$
does not contain any $\rspotential$ symbols, then the right side
of \zcref{eq:tauboxdef} is independent of $\eps$, so we will abbreviate
it by $\boxop{\tau}^{\zeta}_{t}(x)$.

\subsubsection{Modified canonical lift}

The theory of regularity structures has been developed using the kernel
$K$, which satisfies the somewhat restrictive conditions \zcref[range]{eq:Kcompactsupport,eq:Kpolyvanish}.
On the other hand, in several of the explicit calculations we perform
in \zcref{sec:Explicit-calculations} below, it will be much more
convenient to work with the original kernel $p$, which satisfies
the heat equation. Thus we will construct canonical lifts corresponding
to both kernels. Our canonical lifts will represent each of the symbols
at stationarity. A minor complication is that some of the terms of
$\mathsf{T}$ (in particular, many of the elements of $\mathsf{T}_{\mathcal{I}}$)
cannot actually correspond to time-stationary distributions when the
kernel $p$ is used, since they feature a growing zero-frequency mode.
It turns out that we will only need to consider this ``modified canonical''
lift on a restricted set of terms anyway, which (among other properties)
do indeed have stationary versions when interpreted using the kernel
$p$. Thus we begin by defining this restricted subset of terms.

\global\long\def\ES{\operatorname{E/S}}%

\begin{defn}
\label{def:We-define-a}We define a subset $\mathsf{T}_{\ES}\subset\mathsf{T}$
by an inductive procedure. We start by declaring that $\rsnoise,\rspotential\in\mathsf{T}_{\ES}$
(i.e. that $\mathsf{T}_{\bullet}\subseteq\mathsf{T}_{\ES}$). Then
we declare that if $\tau\in\mathsf{T}_{\ES}\cap(\mathsf{T}_{\star}\cup\mathsf{T}_{\bullet})$,
then $\mathcal{I}'\tau\in\mathsf{T}_{\ES}$ as well, and if $\tau_{1},\tau_{2}\in\mathsf{T}_{\ES}\cap\mathsf{T}_{\mathcal{I}'}$,
then $\tau_{1}\tau_{2}\in\mathsf{T}_{\ES}$ as well. We let $\mathsf{T}_{\ES}$
be the set of symbols obtained by applying these rules \emph{ad infinitum}.
In particular we see that $\mathsf{T}_{\ES}\subset\mathsf{T}_{\bullet}\cup\mathsf{T}_{\mathcal{I}'}\cup\mathsf{T}_{\star}$.
We also define
\begin{equation}
\check{\mathsf{T}}_{\ES}\coloneqq\mathsf{T}_{\ES}\cap\check{\mathsf{T}}\label{eq:Tcheckstatdef}
\end{equation}
and define
\[
\mathcal{T}_{\ES}\coloneqq\langle\mathsf{T}_{\ES}\rangle\qquad\text{and}\qquad\check{\mathcal{T}}_{\ES}\coloneqq\langle\check{\mathsf{T}}_{\ES}\rangle.
\]
\end{defn}

The subscript ``$\ES$'' stands for ``explicit/stationary.'' The
terms in $\mathsf{T}_{\ES}$ will represent \emph{explicit} terms
in the expansion of the solution, since they do not contain any polynomial
symbols that will be multiplied by non-explicit parts of solution.
They also are guaranteed to have stationary versions since terms in
$\mathcal{T}_{\mathcal{I}}$ are excluded. We also note that it follows
from \zcref{eq:Tcheckstatdef} and an inspection of \zcref{tab:RS-table}
that 
\begin{equation}
\check{\mathsf{T}}_{\mathrm{\ES}}=(\check{\mathsf{T}}_{\bullet}\cup\check{\mathsf{T}}_{\mathcal{I}'}\cup\check{\mathsf{T}}_{\star})\setminus\{\rsiplollipopr,\rsclawrr\}.\label{eq:Texpcheckexplicit}
\end{equation}
A particular simplification arising from restricting to $\mathsf{T}_{\ES}$,
which we will use later, is that
\begin{equation}
\sgn(\tau)=-1\text{ for }\tau\in\mathsf{T}_{\ES}\cap\mathsf{T}_{\mathcal{I}'}\qquad\text{and}\qquad\sgn(\tau)=1\ \text{for }\tau\in\mathsf{T}_{\ES}\cap(\mathsf{T}_{\bullet}\cup\mathsf{T}_{\star}),\label{eq:sgnofTexp}
\end{equation}
as is easily verified by induction (or, for $\tau\in\check{\mathsf{T}}$,
by a quick inspection of \zcref{tab:RS-table}).

Now we define a map $\tilde{\mathbf{\boldsymbol{\varPi}}}^{\eps,\zeta}\colon\mathcal{T}_{\ES}\to\mathcal{D}'(\mathbb{R}^{2})$,
the space of distributions on $\mathbb{R}^{2}$, inductively by first
defining
\begin{subequations}
\label{eq:Qinductive-hat}
\begin{equation}
\tilde{\mathbf{\boldsymbol{\varPi}}}^{\eps,\zeta}(\rsnoise)\coloneqq\dif W^{\zeta}\qquad\text{and}\qquad\tilde{\mathbf{\boldsymbol{\varPi}}}^{\eps,\zeta}(\rspotential)_{t}(x)\coloneqq\rspotential[rsK]^{\eps}_{t}(x),\label{eq:basic-reconstruction-2-1}
\end{equation}
and then inductively defining
\begin{align}
\tilde{\mathbf{\boldsymbol{\varPi}}}^{\eps,\zeta}(\mathcal{I}'\tau)_{t}(x) & \coloneqq(p'\circledast\tilde{\mathbf{\boldsymbol{\varPi}}}^{\eps,\zeta}(\tau))_{t}(x) &  & \text{for }\tau\in\mathsf{T}_{\ES}\cap(\mathsf{T}_{\bullet}\cup\mathsf{T}_{\star});\label{eq:Qpartial-hat}\\
\tilde{\mathbf{\boldsymbol{\varPi}}}^{\eps,\zeta}(\tau_{1}\tau_{2})_{t}(x) & \coloneqq\tilde{\mathbf{\boldsymbol{\varPi}}}^{\eps,\zeta}(\tau_{1})_{t}(x)\cdot\tilde{\mathbf{\boldsymbol{\varPi}}}^{\eps,\zeta}(\tau_{2})_{t}(x) &  & \text{for }\tau_{1},\tau_{2}\in\mathsf{T}_{\ES}\cap\mathsf{T}_{\mathcal{I}'}.\label{eq:Qprodderivs-hat}
\end{align}
\end{subequations}
Since the heat kernel $p$ is not compactly supported, we should make
sure that the convolutions in \zcref[comp=true]{eq:Qpartial-hat}
actually make sense. For $\tau=\rsnoise$, this is straightforward
to check directly (e.g. using Gaussian estimates). For other $\tau\in\mathsf{T}_{\ES}\cap(\mathsf{T}_{\bullet}\cup\mathsf{T}_{\star})\setminus\{\rspotential\}$,
$\tilde{\mathbf{\boldsymbol{\varPi}}}^{\eps,\zeta}(\tau)$ is a $2L$-periodic
function that is stationary in time with all moments uniformly bounded.
Since $\sup_{x}\sum_{q\in2L\mathbb{Z}}p'_{t}(x)$ decays exponentially
as $t\to\infty$, this readily implies that that \zcref{eq:Qpartial-hat}
is well-defined. (The same is not true for $\sum_{q\in2L\mathbb{Z}}p{}_{t}(x)$,
which is why we do not define $\tilde{\mathbf{\boldsymbol{\varPi}}}^{\eps,\zeta}$
on $\mathcal{T}_{\mathcal{I}}$.)

For notational convenience, analogously to \zcref{eq:tauboxdef},
we define
\[
\bboxop{\tau}^{\eps,\zeta}_{t}(x)\coloneqq\tilde{\mathbf{\boldsymbol{\varPi}}}^{\eps,\zeta}(\tau)_{t}(x).
\]
We\nomenclature{$\bboxop{\phantom{\rsmooserrrr}}$}{Canonical lift using $p$ kernel}
will abbreviate this symbol by $\bboxop{\tau}^{\eps}$ or $\bboxop{\tau}^{\zeta}_{t}$
as appropriate in the same way as described after \zcref{eq:tauboxdef}.

\subsubsection{\label{subsec:Renormalized-canonical-lifts}Renormalized canonical
lifts}

To aid in the definition of the renormalized model later on, we actually
define $\boldsymbol{\varPi}^{\eps,\zeta}$ and $\tilde{\boldsymbol{\varPi}}^{\eps,\zeta}$
on a few additional trees, which will represent renormalizations of
the trees that have already been defined. We introduce the new sets
of symbols
\[
\check{\mathsf{T}}_{\mathrm{R}}\coloneqq\{\rsspecial{C^{(2)}},\rsrenorm,\rsirenorm,\rsiprenorm,\rscherryrenormr,\rscherryrenormb,\rscherryrenormrenorm,\rselkrenormrr\}\qquad\text{and}\qquad\check{\mathsf{T}}_{\mathrm{R},\ES}\coloneqq\check{\mathsf{T}}_{\mathrm{R}}\setminus\{\rsirenorm\},
\]
and then define
\[
\check{\mathcal{T}}_{\mathrm{R}}\coloneqq\langle\check{\mathsf{T}}_{\mathrm{R}}\rangle\qquad\text{and}\qquad\mathcal{T}_{\mathrm{R,\ES}}\coloneqq\langle\check{\mathsf{T}}_{\mathrm{R},\ES}\rangle.
\]
The symbol $\rsrenorm$ should be thought of as a renormalized version
of $\rscherryrr$, and the symbol $\rsspecial{C^{(2)}}$ represents
a spatially constant renormalization; see \zcref[range]{eq:C2renorm,eq:Qtilderenorm}
below. We define a linear map $\hat{M}\colon\check{\mathcal{T}}\to\check{\mathcal{T}}\oplus\check{\mathcal{T}}_{\mathrm{R}}$
as follows. First, we define it on elements of $\check{\mathsf{T}}$
as in the following table:

\begin{equation}\label{tab:Mhatdef}\begin{tabular}{r|cccccccccc}
\toprule 
$\tau$ & $\rscherryrr$ & $\rselkrrr$ & $\rscandelabrarrrr$ & $\rsmooserrrr$ & $\rselkrrb$ & $\rsipcherryrr$ & $\rsicherryrr$ & $\rsielkrrr$ & $\rsipelkrrr$ & other $\tau\in\check{\mathsf{T}}$\tabularnewline
\midrule
$\hat{M}\tau$ & $\rsrenorm$ & $\rscherryrenormr$ & $\rscherryrenormrenorm-\rsspecial{C^{(2)}}$ & $\rselkrenormrr+\frac{1}{4}\rsspecial{C^{(2)}}$ & $\rscherryrenormb$ & $\rsiprenorm$ & $\rsirenorm$ & $\rsicherryrenormr$ & $\rsipcherryrenormr$ & $\tau$\tabularnewline
\bottomrule
\end{tabular}.
\end{equation}Then, we extend $\hat{M}$ to all of $\check{\mathcal{T}}$ by linearity.\footnote{The notation $\hat{M}$ overlaps with the notation for the martingale
defined in \zcref{eq:DtildeMG}. Since this martingale will not be
used in subsequent sections, we hope that this will not cause confusion.} Informally, $\hat{M}\tau$ replaces each occurrence of $\rscherryrr$
in $\tau$ by the symbol $\rsrenorm$, and adds an additional necessary
renormalization term to each of $\rsmooserrrr$ and $\rscandelabrarrrr$.

We define the renormalization constant
\begin{align}
C^{(1)}_{\zeta}(x) & \coloneqq\mathbb{E}\left[\rscherryrr[rsP]^{\zeta}_{t}(x)\right].\label{eq:C1zetadef}
\end{align}
This quantity does not depend on $t$ since the law of $\rscherryrr[rsP]^{\eps,\zeta}_{t}(x)$
is time-independent. Then we define
\begin{equation}
\mathbf{\boldsymbol{\varPi}}^{\eps,\zeta}(\rsrenorm)_{t}(x)\coloneqq\mathbf{\boldsymbol{\varPi}}^{\eps,\zeta}(\rscherryrr)_{t}(x)-C^{(1)}_{\zeta}(x),\label{eq:Qrenorm}
\end{equation}
and
\begin{equation}
\tilde{\mathbf{\boldsymbol{\varPi}}}^{\eps,\zeta}(\rsrenorm)_{t}(x)\coloneqq\tilde{\mathbf{\boldsymbol{\varPi}}}^{\eps,\zeta}(\rscherryrr)_{t}(x)-C^{(1)}_{\zeta}(x).\label{eq:Qtilderenorm}
\end{equation}
Finally, we impose \zcref{eq:Qinductive,eq:Qinductive-hat} again
(with the obvious notation) to extend the definitions of $\mathbf{\boldsymbol{\varPi}}^{\eps,\zeta}$
and $\tilde{\mathbf{\boldsymbol{\boldsymbol{\varPi}}}}^{\eps,\zeta}$
to all of the elements of $\check{\mathsf{T}}_{\mathrm{R}}\setminus\{\rsspecial{C^{(2)}}\}$
and $\check{\mathsf{T}}_{\mathrm{R},\ES}\setminus\{\rsspecial{C^{(2)}}\}$,
respectively. Then, we define
\begin{equation}
C^{(2)}_{\zeta}\coloneqq\frac{1}{L}\int^{L}_{0}\mathbb{E}\left[\rscherryrenormrenorm[rsP]^{\zeta}_{t}(x)\right]\,\dif x,\label{eq:C2zetadef}
\end{equation}
which similarly does not depend on $t$, and put
\begin{equation}
\mathbf{\boldsymbol{\varPi}}^{\eps,\zeta}(\rsspecial{C^{(2)}})_{t}(x)\coloneqq\mathbf{\tilde{\boldsymbol{\varPi}}}^{\eps,\zeta}(\rsspecial{C^{(2)}})_{t}(x)\coloneqq C^{(2)}_{\zeta}.\label{eq:C2renorm}
\end{equation}
Then we can extend $\mathbf{\boldsymbol{\varPi}}^{\eps,\zeta}$ and
$\tilde{\mathbf{\boldsymbol{\boldsymbol{\varPi}}}}^{\eps,\zeta}$
to all of $\check{\mathcal{T}}\oplus\check{\mathcal{T}}_{\mathrm{R}}$
and $\check{\mathcal{T}}_{\ES}\oplus\check{\mathcal{T}}_{\mathrm{R},\ES}$,
respectively, by linearity. We note in particular the immediate consequence
of \zcref{eq:C1zetadef,eq:Qtilderenorm} that
\begin{equation}
\rsrenorm[rsP]^{\zeta}_{t}(x)=\rscherryrr[rsP]^{\zeta}_{t}(x)-\mathbb{E}\left[\rscherryrr[rsP]^{\zeta}_{t}(x)\right],\label{eq:subtract-exp}
\end{equation}
which\nomenclature{$\rsrenorm$}{Renormalized version of $\rscherryrr$}
is the projection of $\rscherryrr[rsP]^{\eps,\zeta}_{t}$ onto the
homogeneous second Wiener chaos. In particular,
\begin{equation}
\mathbb{E}\left[\rsrenorm[rsP]^{\zeta}_{t}(x)\right]=0.\label{eq:renormmeanzero}
\end{equation}
On the other hand, we have
\begin{equation}
\mathbb{E}\left[\rsrenorm[rsK]^{\zeta}_{t}(x)\right]=\mathbb{E}\left[\rscherryrr[rsK]^{\zeta}_{t}(x)\right]-\mathbb{E}\left[\rscherryrr[rsP]^{\zeta}_{t}(x)\right],\label{eq:exp-renorm}
\end{equation}
which is in general nonzero, but which we will see later in \zcref{eq:renorm-constants-same}
is bounded uniformly in $t,x,\zeta$. Finally, we define a new ``renormalized''
lift 
\begin{equation}
\hat{\boldsymbol{\varPi}}^{\eps,\zeta}(\tau)\coloneqq\boldsymbol{\varPi}^{\eps,\zeta}(\hat{M}\tau).\label{eq:Pihatdef}
\end{equation}

The map $\hat{\boldsymbol{\varPi}}^{\eps,\zeta}$ involves renormalization
and thus does \emph{not} satisfy a multiplicative property analogous
to \zcref{eq:Qprodderivs}/\zcref{eq:Qprodderivs-hat}. However,
it is straightforward to check that
\begin{equation}
\hat{\boldsymbol{\varPi}}^{\eps,\zeta}(\tau_{1}\tau_{2})_{t}(x)=\hat{\boldsymbol{\varPi}}^{\eps,\zeta}(\tau_{1})_{t}(x)\cdot\hat{\boldsymbol{\varPi}}^{\eps,\zeta}(\tau_{2})_{t}(x)-C_{\zeta}[\tau_{1},\tau_{2}](x)\qquad\text{for }\tau_{1},\tau_{2}\in\check{\mathsf{T}}_{\mathrm{poly}}\cup\check{\mathsf{T}}_{\mathcal{I}'}\label{eq:Pihatmultiply}
\end{equation}
and also that
\begin{equation}
\tilde{\boldsymbol{\varPi}}^{\eps,\zeta}(\hat{M}(\tau_{1}\tau_{2}))_{t}(x)=\tilde{\boldsymbol{\varPi}}^{\eps,\zeta}(\hat{M}\tau_{1})_{t}(x)\cdot\tilde{\boldsymbol{\varPi}}^{\eps,\zeta}(\hat{M}\tau_{2})_{t}(x)-C_{\zeta}[\tau_{1},\tau_{2}](x)\qquad\text{for }\tau_{1},\tau_{2}\in\check{\mathsf{T}}_{\mathcal{I}'}\cap\check{\mathsf{T}}_{\ES},\label{eq:PitildeMhatmultiply}
\end{equation}
where
\begin{equation}
C_{\zeta}[\tau_{1},\tau_{2}](x)\coloneqq\begin{cases}
C^{(1)}_{\zeta}(x), & \tau_{1}=\tau_{2}=\rslollipopr;\\
C^{(2)}_{\zeta}, & \tau_{1}=\tau_{2}=\rsipcherryrr;\\
-\frac{1}{4}C^{(2)}_{\zeta}, & \{\tau_{1},\tau_{2}\}=\{\rslollipopr,\rsipelkrrr\};\\
0, & \text{otherwise}.
\end{cases}\label{eq:Ctau1tau2def}
\end{equation}
Moreover, the renormalized canonical lift crucially \emph{does} satisfy
the analogue of \zcref{eq:QI}: we have
\begin{equation}
\hat{\boldsymbol{\varPi}}^{\eps,\zeta}(\mathcal{I}(\tau))=K\circledast\hat{\boldsymbol{\varPi}}^{\eps,\zeta}\tau\qquad\text{and}\qquad\hat{\boldsymbol{\varPi}}^{\eps,\zeta}(\partial\mathcal{I}(\tau))=K'\circledast\hat{\boldsymbol{\varPi}}^{\eps,\zeta}\tau\qquad\text{for any }\tau\in\check{\mathcal{T}}\label{eq:Pihatintegration}
\end{equation}
as well as
\begin{equation}
\tilde{\boldsymbol{\varPi}}^{\eps,\zeta}(\hat{M}\mathcal{I}'(\tau))=p'\circledast\tilde{\boldsymbol{\varPi}}^{\eps,\zeta}(\hat{M}\tau)\qquad\text{for any }\tau\in\check{\mathcal{T}}_{\ES}\cap\check{\mathcal{T}}_{\star}.\label{eq:Pihatintegration-1}
\end{equation}
These identities are straightforward to check case-by-case using the
definitions.
\begin{rem}
The reader may wonder why we defined $C^{(1)}_{\zeta}(x)$ and $C^{(2)}_{\zeta}$
as in \zcref[range]{eq:C1zetadef,eq:C2zetadef}, using $\tilde{\boldsymbol{\varPi}}$
(i.e.\ via $\rscherryrr[rsP]$ and $\rscandelabrarrrr[rsP]$) rather
than $\boldsymbol{\varPi}$ (i.e.\ via $\rscherryrr[rsK]$ and $\rscandelabrarrrr[rsK]$).
The reason is that we ultimately want to recover the equation \zcref{eq:hepszeta-eqn}
exactly; see \zcref{prop:RSrecovers,prop:C1zeta} below. If we wanted
to use $\rscherryrr[rsK]$ to define $C^{(1)}_{\zeta}(x)$, then we
could have replaced $\oh\Sh^{\zeta/2}_{L}$ in \zcref{eq:Zepszeta-eqn}
with a different boundary renormalization function, but this different
function would have to depend on the choice of the kernel $K$.
\end{rem}

\begin{rem}
The symbols in $\check{\mathsf{T}}_{\mathrm{R}}$ are not part of
the regularity structure; they are solely defined for use with the
canonical lifts. Presumably one could make them part of the regularity
structure, in particular define the structure group on them. However,
in our setting the benefits that this would yield are easy enough
to check by hand, so we forego this additional complexity.
\end{rem}

The following lemma on the symmetry properties of the canonical lift
in particular justifies our definition of the $\sgn$ function on
$\mathsf{T}$ in \zcref[range]{eq:sgnbasecase,eq:sgnrecursion} above.
We extend the $\sgn$ function to $\mathsf{T}\cup\check{\mathsf{T}}_{\mathrm{R}}$
by setting $\sgn(\rsrenorm)=\sgn(\rsspecial{C^{(2)}})=1$ and then
extending by the recursions \zcref{eq:sgnrecursion}.
\begin{lem}
\label{lem:reconstruct-symmetry}For all $t,x\in\mathbb{R}$ and all
$\tau\in\mathsf{T}\cup\check{\mathsf{T}}_{\mathrm{R}}$ (only $\mathsf{T}_{\ES}\cup\check{\mathsf{T}}_{\mathrm{R},\ES}$
in the second identity below), we have
\begin{equation}
\boxop{\tau}^{\eps,\zeta}_{t}(-x)=\sgn(\tau)\boxop{\tau}^{\eps,\zeta}_{t}(x),\quad\bboxop{\tau}^{\eps,\zeta}_{t}(-x)=\sgn(\tau)\bboxop{\tau}^{\eps,\zeta}_{t}(x),\quad\text{and}\quad\hat{\boldsymbol{\varPi}}^{\eps,\zeta}(\tau)_{t}(x)=\sgn(\tau)\hat{\boldsymbol{\varPi}}^{\eps,\zeta}(\tau)_{t}(-x).\label{eq:varPireflect}
\end{equation}
Also, for all $t,x\in\mathbb{R}$ and $\tau\in(\check{\mathsf{T}}\setminus\{\rsx\})\cup\check{\mathsf{T}}_{\mathrm{R}}$
(again with $\check{\mathsf{T}}$ replaced by $\check{\mathsf{T}}_{\ES}$
for the second identity), we have
\begin{equation}
\boxop{\tau}^{\eps,\zeta}_{t}(x+2L)=\boxop{\tau}^{\eps,\zeta}_{t}(x),\quad\bboxop{\tau}^{\eps,\zeta}_{t}(x+2L)=\bboxop{\tau}^{\eps,\zeta}_{t}(x),\quad\text{and}\quad\hat{\boldsymbol{\varPi}}^{\eps,\zeta}(\tau)_{t}(x+2L)=\hat{\boldsymbol{\varPi}}^{\eps,\zeta}(\tau)_{t}(x).\label{eq:varPitranslate}
\end{equation}
\end{lem}

\begin{proof}
The first two identities in \zcref{eq:varPireflect} follow simply
by induction, noting that the relations \zcref{eq:sgnrecursion} match
the effects of the relations \zcref{eq:Qinductive,eq:Qinductive-hat}
on functions by \zcref{eq:KKhateven} (and the fact that $p_{t}$
is even for each $t$). The last identity in \zcref{eq:varPireflect}
is straightforward to check case-by-case using \zcref{tab:Mhatdef}.
The identities in \zcref{eq:varPitranslate} are similarly straightforward
to check, using the fact that $(\dif W^{\zeta}_{t})$, $\varphi^{\eps}_{\uu,\vv}$,
and $C^{(1)}_{\zeta}$ are $2L$-periodic functions, and that no element
of $(\check{\mathsf{T}}\setminus\{\rsx\})\cup\check{\mathsf{T}}_{\mathrm{R}}$
involves multiplication by $\rsx$ (which is the only building block
of the regularity that does not represent a $2L$-periodic object).
\end{proof}

\subsection{\label{subsec:Model}Model}

We can now construct a family of models; see \cites[Defn.~2.17]{hairer:2014:theory}
for the definition. For simplicity, and because it is all we will
need, we define the model only on the truncated regularity structure
$(\check{A},\check{\mathcal{T}},\check{\mathcal{G}})$. We will in
fact ensure that our models $\left((\hat{\Pi}^{\eps,\zeta}_{t,x},\hat{\Gamma}^{\eps,\zeta}_{t,x})\right)_{t,x\in\mathbb{R}}$
are \emph{admissible} in the sense of \cite[§3.4]{hairer:quastel:2018:class},
and we construct them from the renormalized canonical lift $\hat{\boldsymbol{\varPi}}$
exactly following the construction there. Rather than repeating the
definition of the construction here, we simply summarize the resulting
definition of $(\hat{\Pi}^{\eps,\zeta},\hat{\Gamma}^{\eps,\zeta})$
on $\check{\mathsf{T}}$ in \zcref{tab:model-defns}.
\begin{table}
\hfill{}%
\begin{tabular}{lll}
\toprule 
$\tau$ & $\hat{\Pi}^{\eps,\zeta}_{(s,y)}(\tau)_{t}(x)$ & $\hat{\Gamma}^{\eps,\zeta}_{(s,y),(s',y')}(\tau)-\tau$\tabularnewline
\midrule
$\begin{aligned} & \rsone,\rsnoise,\rspotential,\rslollipopr,\rslollipopb,\\
 & \quad\rscherryrb,\rscherrybb
\end{aligned}
$ & $\boxop{\tau}^{\eps,\zeta}_{t}(x)$ & $0$\tabularnewline
$\rsx,\rsballoonr,\rsballoonb,\rsipcherryrb,\rsiplollipopr$ & $\boxop{\tau}^{\eps,\zeta}_{t}(x)-\boxop{\tau}^{\eps,\zeta}_{s}(y)$ & $\left(\boxop{\tau}^{\eps,\zeta}_{s}(y)-\boxop{\tau}^{\eps,\zeta}_{s'}(y')\right)\rsone$\tabularnewline
$\rsicherryrb,\rsilollipopr$ & $\boxop{\tau}^{\eps,\zeta}_{t}(x)-\boxop{\tau}^{\eps,\zeta}_{s}(y)-\partial_{x}\boxop{\tau}^{\eps,\zeta}_{s}(y)(x-y)$ & $\begin{aligned} & \left(\boxop{\tau}^{\eps,\zeta}_{s}(y)-\boxop{\tau}^{\eps,\zeta}_{s'}(y')-\partial_{x}\boxop{\tau}^{\eps,\zeta}_{s'}(y')(y-y')\right)\rsone\\
 & \qquad+\left(\partial_{x}\boxop{\tau}^{\eps,\zeta}_{s}(y)-\partial_{x}\boxop{\tau}^{\eps,\zeta}_{s'}(y')\right)\rsx
\end{aligned}
$\tabularnewline
$\rscherryrr$ & $\rsrenorm[rsK]^{\eps,\zeta}_{t}(x)$ & $0$\tabularnewline
$\rsicherryrr$ & $\rsirenorm[rsK]^{\eps,\zeta}_{t}(x)-\rsirenorm[rsK]^{\eps,\zeta}_{s}(y)$ & $\left(\rsirenorm[rsK]^{\eps,\zeta}_{s}(y)-\rsirenorm[rsK]^{\eps,\zeta}_{s'}(y')\right)\rsone$\tabularnewline
$\rsipcherryrr$ & $\rsiprenorm[rsK]^{\eps,\zeta}_{t}(x)$ & $0$\tabularnewline
$\rselkrbr$ & $\left(\rsipcherryrb[rsK]^{\eps,\zeta}_{t}(x)-\rsipcherryrb[rsK]^{\eps,\zeta}_{s}(y)\right)\rslollipopr[rsK]^{\eps,\zeta}_{t}(x)$ & $\left(\rsipcherryrb[rsK]^{\eps,\zeta}_{s}(y)-\rsipcherryrb[rsK]^{\eps,\zeta}_{s'}(y')\right)\rslollipopr$\tabularnewline
$\rsclawrr$ & $\left(\rsiplollipopr[rsK]^{\eps,\zeta}_{t}(x)-\rsiplollipopr[rsK]^{\eps,\zeta}_{s}(y)\right)\rslollipopr[rsK]^{\eps,\zeta}_{t}(x)$ & $\left(\rsiplollipopr[rsK]^{\eps,\zeta}_{s}(y)-\rsiplollipopr[rsK]^{\eps,\zeta}_{s'}(y')\right)\rslollipopr$\tabularnewline
$\rselkrrr$ & $\rscherryrenormr[rsK]^{\eps,\zeta}_{t}(x)$ & $0$\tabularnewline
$\rselkrrb$ & $\rscherryrenormb[rsK]^{\eps,\zeta}_{t}(x)$ & $0$\tabularnewline
$\rscandelabrarrrr$ & $\rscherryrenormrenorm[rsK]^{\eps,\zeta}_{t}(x)-C^{(2)}_{\zeta}$ & $0$\tabularnewline
$\rsielkrrr$ & $\rsicherryrenormr[rsK]^{\eps,\zeta}_{t}(x)-\rsicherryrenormr[rsK]^{\eps,\zeta}_{s}(y)$ & $\left(\rsicherryrenormr[rsK]^{\eps,\zeta}_{s}(y)-\rsicherryrenormr[rsK]^{\eps,\zeta}_{s'}(y')\right)\rsone$\tabularnewline
$\rsipelkrrr$ & $\rsipcherryrenormr[rsK]^{\eps,\zeta}_{t}(x)-\rsipcherryrenormr[rsK]^{\eps,\zeta}_{s}(y)$ & $\left(\rsipcherryrenormr[rsK]^{\eps,\zeta}_{s}(y)-\rsipcherryrenormr[rsK]^{\eps,\zeta}_{s'}(y')\right)\rsone$\tabularnewline
$\rsmooserrrr$ & $\left(\rsipcherryrenormr[rsK]^{\eps,\zeta}_{t}(x)-\rsipcherryrenormr[rsK]^{\eps,\zeta}_{s}(y)\right)\rslollipopr[rsK]^{\eps,\zeta}_{t}(x)+\frac{1}{4}C^{(2)}_{\zeta}$ & $\left(\rsipcherryrenormr[rsK]^{\eps,\zeta}_{s}(y)-\rsipcherryrenormr[rsK]^{\eps,\zeta}_{s'}(y')\right)\rslollipopr$\tabularnewline
\bottomrule
\end{tabular}\hfill{}

\caption[Definition of the model.]{\label{tab:model-defns}Definition of the model on each of the trees
$\tau$ in $\check{\mathsf{T}}$.}
\end{table}
It is straightforward to check that the model $(\hat{\Pi}^{\eps,\zeta},\hat{\Gamma}^{\eps,\zeta})$
is compatible with $\partial$ in the sense of \cites[Defn.~5.26]{hairer:2014:theory}
and adapted to the action \zcref{eq:sigma-action} of $\mathscr{S}$
in the sense of \cite[Defn.~3.33]{hairer:2014:theory}. Similarly
to \zcref[range]{eq:Pihatmultiply,eq:PitildeMhatmultiply}, it can
also be checked case-by-case referring to \zcref{tab:model-defns}
that, for any $\tau_{1},\tau_{2}\in\check{\mathsf{T}}_{\mathrm{poly}}\cup\check{\mathsf{T}}_{\mathcal{I}'}$
and $t,x\in\mathbb{R}$, we have
\begin{equation}
\hat{\Pi}^{\eps,\zeta}_{t,x}(\tau_{1}\tau_{2})_{t}(x)=\hat{\Pi}^{\eps,\zeta}_{t,x}(\tau_{1})_{t}(x)\cdot\hat{\Pi}^{\eps,\zeta}_{t,x}(\tau_{2})_{t}(x)-C_{\zeta}[\tau_{1},\tau_{2}](x),\label{lem:mult-model}
\end{equation}
where $C_{\zeta}[\tau_{1},\tau_{2}]$ is defined in \zcref{eq:Ctau1tau2def}.

We will denote by $\|(\hat{\Pi}^{\eps,\zeta},\hat{\Gamma}^{\eps,\zeta})\|_{T}$
the norm of the model restricted to $[-T,T]\times\mathbb{R}$, in
the sense of \cite[Defn.~2.17 and Rmk.~2.20]{hairer:2014:theory}.
The following theorem encapsulates the stochastic estimates on the
model that are essential for the application of the regularity structure
theory.
\begin{thm}
\label{thm:model-norm-bounded}We have, for all $p\in[1,\infty)$
and $\kappa\in(0,\nicefrac{1}{100})\setminus\mathbb{Q}$, that
\[
\adjustlimits\sup_{\eps\in(0,1]}\sup_{\zeta\in(0,\eps)}\mathbb{E}\left[\|(\hat{\Pi}^{\eps,\zeta},\hat{\Gamma}^{\eps,\zeta})\|^{p}_{T}\right]<\infty.
\]
\end{thm}

\begin{proof}
This result is proven in \zcref{sec:BPHZ}. In particular, it is a
consequence of \zcref{lem:red-to-stoch,prop:hypercontractivity,prop:stochastic-estimates}.
\end{proof}

\subsection{\label{subsec:Lift-of-the}Lift of the KPZ solution}

With the regularity structure and the (admissible) model constructed,
the solution theory of the KPZ equation via regularity structures
follows exactly that developed in \cite[§7]{hairer:2014:theory}.
In this section we explain the details of the parts that are relevant
for our application. We emphasize that our goal in using regularity
structures is not to \emph{construct} the solution of the KPZ equation,
which we have already done in \zcref{sec:solntheory} via the Cole--Hopf
transform. Rather, we seek to use the reconstruction theorem to expand
the nonlinearity around the boundary. In order to do this, it is important
that the approximate solution defined via the regularity structure
actually matches the solution constructed in \zcref{sec:mollifying}.
This has been ensured by our precise choices of the renormalization
constants in \zcref{subsec:Canonical-lift} and will be verified carefully
below.

\subsubsection{Modeled distributions}

In order to state the results that we will use, we recall the notion
of modeled distribution from \cite[§3]{hairer:2014:theory} and of
singular modeled distribution from \cites[§6]{hairer:2014:theory}.
In particular, we need to take into account a possible singularity
at $t=0$ coming from the non-smooth initial condition and the singularities
of the kernels $p$ and $K$ at $t=0$, as in \cite[§6]{hairer:2014:theory}.
\begin{defn}[{Special cases of \cite[Defn.~3.1 and 6.2]{hairer:2014:theory}}]
For $\gamma>0$, the space $\mathcal{D}^{\gamma;\eps,\chi}$ of \emph{modeled
distributions }comprises all functions $f\colon\mathbb{R}^{2}\to\mathcal{T}_{<\gamma}$
such that for any $T<\infty$, we have
\[
\vvvert f\vvvert_{\mathcal{D}^{\gamma;\eps,\chi};T}\coloneqq\sup_{\substack{x\in\mathbb{R},t\in[-T,T]\\
\ell\in A_{<\gamma}
}
}\|f_{t}(x)\|_{\ell}+\sup_{\substack{x,y\in\mathbb{R}\\
t,s\in[-T,T]\\
\ell\in A_{<\gamma}
}
}\frac{\left\Vert f_{t}(x)-\hat{\Gamma}^{\eps,\zeta}_{(t,x),(s,y)}f_{s}(y)\right\Vert _{\ell}}{\left(|t-s|^{1/2}+|x-y|\right)^{\gamma-\ell}}<\infty.
\]
Also, for $\gamma>0$ and $\chi\in\mathbb{R}$, the space $\mathcal{D}^{\gamma,\chi;\eps,\zeta}$
of \emph{singular modeled distributions} comprises all functions $f\colon\mathbb{R}_{>0}\times\mathbb{R}^{2}\to\mathcal{T}_{<\gamma}$
such that for any $T<\infty$, we have 
\begin{equation}
\vvvert f\vvvert_{\mathcal{D}^{\gamma,\chi;\eps,\zeta};T}\coloneqq\sup_{\substack{x\in\mathbb{R},t\in(0,T]\\
\ell\in A_{<\gamma}
}
}\frac{\|f_{t}(x)\|_{\ell}}{t^{(\chi-\ell)\wedge0}}+\sup_{\substack{x,y\in\mathbb{R}\\
t,s\in(0,T]\\
\ell\in A_{<\gamma}
}
}\frac{\left\Vert f_{t}(x)-\hat{\Gamma}^{\eps,\zeta}_{(t,x),(s,y)}f_{s}(y)\right\Vert _{\ell}}{\left(|t-s|^{1/2}+|x-y|\right)^{\gamma-\ell}(t\wedge s)^{\chi-\gamma}}<\infty.\label{eq:modeled-distn-norm}
\end{equation}
We will also use the notation, for $V$ a sector of $\mathcal{T}$,
\[
\mathcal{D}^{\gamma,\chi;\eps,\zeta}(V)=\{f\in\mathcal{D}^{\gamma,\chi;\eps,\zeta}\st f_{t}(x)\in V\text{ for all }t,x\}\ \ \text{and}\ \ \mathcal{D}^{\gamma;\eps,\zeta}(V)=\{f\in\mathcal{D}^{\gamma;\eps,\zeta}\st f_{t}(x)\in V\text{ for all }t,x\}.
\]
Recalling the notation
\[
\mathcal{D}^{\gamma,\chi;\eps,\zeta}_{\beta}\coloneqq\mathcal{D}^{\gamma,\chi;\eps,\zeta}(\mathcal{T}_{\ge\beta})\qquad\text{and}\qquad\mathcal{D}^{\gamma;\eps,\zeta}_{\beta}\coloneqq\mathcal{D}^{\gamma,\chi;\eps,\zeta}(\mathcal{T}_{\ge\beta}).
\]
\end{defn}

It will also be important in our analysis to be able to keep track
of the symmetry of the problem. Recall the group $\mathscr{S}$ from
\zcref{eq:sym_group_generators} and its action \zcref{eq:sigma-action}
on the regularity structure $\mathcal{T}$.
\begin{defn}[{\cite[Defn.~3.33]{hairer:2014:theory}}]
\label{def:symmetric}We say that a modeled distribution is \emph{symmetric}
if, for all $\sigma\in\mathscr{S}$ and $t,x\in\mathbb{R}$, we have
\begin{equation}
\sigma\cdot f_{t}(\sigma(x))=f_{t}(x).\label{eq:symmetry-eqn}
\end{equation}
\end{defn}

\subsubsection{Reconstruction operator}

For $\beta>-3$, we define the reconstruction operator $\hat{\mathcal{R}}^{\eps,\zeta}\colon\mathcal{D}^{\gamma,\chi;\eps,\zeta}_{\beta}\to\mathcal{D}'(\mathbb{R}^{2})$
for the models $(\hat{\Pi}^{\eps,\zeta},\hat{\Gamma}^{\eps,\zeta})$,
as in \cite[Prop.~6.9]{hairer:2014:theory}. (The condition $\beta>-3$
is imposed in the statement of \cite[Prop.~6.9]{hairer:2014:theory}
and comes from the effective dimension of $\mathbb{R}^{2}$ under
the parabolic scaling, i.e. $3=2+1$ since ``time counts double.'')
In fact, because we work with $\eps,\zeta>0$, we can mostly define
the reconstruction operator simply in terms of the model by the argument
of \cites[Rmk.~3.15]{hairer:2014:theory}. There is a minor difference
in our setting, which is that we only consider spatial rather than
space-time mollifications of the white noise, and so the assumption
of \cites[Rmk.~3.15]{hairer:2014:theory} is not quite satisfied,
because $\hat{\Pi}^{\eps,\zeta}_{t,x}\rsnoise$ is not continuous
in time. (Indeed, it is white in time.) However, it is straightforward
to check that if $\tau\in\check{\mathsf{T}}\setminus\{\rsnoise\}=\check{\mathsf{T}}_{\ge-\thrh}$,
then, for any $t,x\in\mathbb{R}$, the distribution $\hat{\Pi}^{\eps,\zeta}_{t,x}\tau$
is a Hölder-continuous space-time function. (This Hölder continuity
is not uniform in $\eps$ and $\zeta$ unless $|\tau|\ge0$.) Since
$\langle\check{\mathsf{T}}\setminus\{\rsnoise\}\rangle$ is a sector
of $\check{\mathsf{T}}$ in the sense of \cite[Defn.~2.5]{hairer:2014:theory},
we can conclude by \cites[Rmk.~3.15]{hairer:2014:theory} that
\begin{equation}
\hat{\mathcal{R}}^{\eps,\zeta}(f)_{t}(x)=\hat{\Pi}^{\eps,\zeta}_{t,x}(f_{t}(x))_{t}(x)\qquad\text{if }f\in\mathcal{D}^{\gamma,\chi;\eps,\zeta}_{\beta}\text{ for some }\gamma>0\text{ and }\beta\ge-\thrh.\label{eq:reconstruction-is-model-1}
\end{equation}
On the other hand, it is also clear from the definitions that 
\begin{equation}
\hat{\mathcal{R}}^{\eps,\zeta}(\rsnoise)_{t}(x)=\dif W^{\zeta}_{t}(x)\label{eq:reconstruction-of-noise}
\end{equation}
in the sense of distributions. We will not have occasion to consider
modeled distributions that involve $\rsnoise$ multiplied by a non-constant
function, so \zcref{eq:reconstruction-is-model-1,eq:reconstruction-of-noise}
will suffice to compute $\hat{\mathcal{R}}^{\eps,\zeta}$ applied
to all modeled distributions of interest. We also note that the
hypothesis of \cites[Prop.~5.28]{hairer:2014:theory} is satisfied
by the compatibility of the model with $\partial$ noted in \zcref{subsec:Model},
and so in particular we have the identity
\begin{equation}
\hat{\mathcal{R}}^{\eps,\zeta}(\partial f)=\partial_{x}\hat{\mathcal{R}}^{\eps,\zeta}f.\label{eq:Rhat-commutes-deriv}
\end{equation}

\subsubsection{Integration operators and solution theory}

At this point, we fix once and for all an arbitrary 
\begin{equation}
\chi\in(\oh-\kappa,\oh),\label{eq:chirange}
\end{equation}
representing the regularity of the initial condition. We also will
define separate sectors of the regularity structure for the solution
of the KPZ equation and the forcing. We will need to expand the solution
up to regularity just greater than $\thrh$ and the forcing up to
regularity just greater than $0$. Therefore, we define
\[
\check{\mathcal{T}}_{\mathrm{S};\mathrm{poly}}\coloneqq\langle\rsone,\rsx\rangle\qquad\text{and}\qquad\check{\mathcal{T}}_{\mathrm{F};\mathrm{poly}}\coloneqq\langle\rsone\rangle
\]
as well as the sectors
\[
\check{\mathcal{T}}_{\mathrm{S}}\coloneqq\check{\mathcal{T}}_{\mathrm{S};\mathrm{poly}}\oplus\check{\mathcal{T}}_{\mathcal{I}}\qquad\text{and}\qquad\check{\mathcal{T}}_{\mathrm{F}}\coloneqq\check{\mathcal{T}}_{\mathrm{F};\mathrm{poly}}\oplus\check{\mathcal{T}}_{\bullet}\oplus\check{\mathcal{T}}_{\mathcal{I}'}\oplus\mathcal{\check{T}}_{\star}
\]
and the projection maps
\begin{equation}
\check{\varpi}_{\mathrm{S}}\colon\mathcal{T}\to\check{\mathcal{T}}_{\mathrm{S}}\qquad\text{and}\qquad\check{\varpi}_{\mathrm{F}}\colon\mathcal{T}\to\check{\mathcal{T}}_{\mathrm{F}}.\label{eq:SFprojections}
\end{equation}
Again, it is straightforward to check that these subspaces of the
model space really are sectors. The regularity (i.e., the lowest-degree
homogeneity) of the sector $\check{\mathcal{T}_{\mathrm{S}}}$ is
$0$, while the regularity of the sector $\check{\mathcal{T}}_{\mathrm{F}}$
is $-\thrh-\kappa$. We will solve the KPZ equation in the space of
modeled distributions
\begin{equation}
\mathcal{D}^{\eps,\zeta}_{\mathrm{S}}\coloneqq\mathcal{D}^{\thrh+10\kappa,\chi;\eps,\zeta}(\check{\mathcal{T}}_{\mathrm{S}}),\label{eq:DKPZ}
\end{equation}
and resolve the forcing on the right side of the KPZ equation in the
space
\begin{equation}
\mathcal{D}^{\eps,\zeta}_{\mathrm{F}}\coloneqq\mathcal{D}^{9\kappa,-1-2\kappa;\eps,\zeta}(\check{\mathcal{T}}_{\mathrm{F}}).\label{eq:DF}
\end{equation}
We will also use the notations
\begin{equation}
\mathcal{D}_{\mathrm{S};\mathrm{poly}}\coloneqq\mathcal{D}^{3/2+10\kappa,\chi;\eps,\zeta}(\check{\mathcal{T}}_{\mathrm{S};\mathrm{poly}})\qquad\text{and}\qquad\mathcal{D}_{\mathrm{F};\mathrm{poly}}\coloneqq\mathcal{D}^{9\kappa,-1-2\kappa;\eps,\zeta}(\check{\mathcal{T}}_{\mathrm{F};\mathrm{poly}}).\label{eq:DKPZpoly}
\end{equation}
(These spaces do not depend on $\eps$ and $\zeta$ since the model
restricted to $\mathcal{T}_{\mathrm{poly}}$ does not, but this will
not really be important.) The choices of $\gamma$ and $\chi$ in
these spaces are justified by the computations in \zcref[comp=true]{eq:DGspace,eq:DG2space,eq:squaring-cts}
below. The important point in the sequel will be that in \zcref{eq:DF}
we have $\gamma=9\kappa>0$, so we can apply the reconstruction theorem
in \zcref{subsec:Reduction-to-the} (see \zcref{prop:apply-reconstruction-boundary}).

We define the map $\mathcal{P}_{0}\coloneqq\mathcal{C}^{\chi}(\mathbb{R}/(2L\mathbb{Z}))\to\mathcal{D}_{\mathrm{S};\mathrm{poly}}$
by
\begin{equation}
(\mathcal{P}_{0}h_{0})_{t}(x)=(p_{t}*h_{0})(x)\rsone+(p'_{t}*h_{0})(x)\rsx.\label{eq:P0def}
\end{equation}
That the right side of \zcref{eq:P0def} does live in $\mathcal{D}_{\mathrm{S};\mathrm{poly}}$
is a consequence of \cite[Lem.~7.5]{hairer:2014:theory}. It is an
immediate consequence of \zcref{eq:P0def,eq:reconstruction-is-model-1}
that
\begin{equation}
(\hat{\mathcal{R}}^{\eps,\zeta}\mathcal{P}_{0}h_{0})_{t}(x)=(p_{t}*h_{0})(x).\label{eq:reconstruct-ic}
\end{equation}

We also recall from \cite{hairer:2014:theory} the definitions of
operators on the space of modeled distributions that represent convolution
by $K$ and by $\tilde{K}$. For $\tilde{K}$, we simply define
\begin{equation}
(\tilde{\mathcal{K}}^{\eps,\zeta}f)(x)=(\tilde{K}\circledast\hat{\mathcal{R}}^{\eps,\zeta}f)_{t}(x)\rsone+(\tilde{K}'\circledast\hat{\mathcal{R}}^{\eps,\zeta}f)_{t}(x)\rsx\label{eq:scrKhatdef}
\end{equation}
for $f\in\mathcal{D}^{\eps,\zeta}_{\mathrm{F}}$ for which the two
convolutions are well-defined. From this it is easy to conclude that
(for such $f$)
\begin{equation}
\hat{\mathcal{R}}^{\eps,\zeta}\tilde{\mathcal{K}}^{\eps,\zeta}f=\tilde{K}\circledast\hat{\mathcal{R}}^{\eps,\zeta}f.\label{eq:RKhat}
\end{equation}
For $K$, we define a map $\mathcal{K}^{\eps,\zeta}=\mathcal{D}^{\eps,\zeta}_{\mathrm{F}}\to\mathcal{D}^{\eps,\zeta}_{\mathrm{S}}$
by the construction of \cites[(5.15)]{hairer:2014:theory}. We use
\cites[Prop.~6.16]{hairer:2014:theory}, the assumption \zcref{eq:chirange}
that $\chi>\oh-\kappa$, and also the definition of $\mathcal{K}^{\eps,\zeta}$
along with the definitions of the relevant sectors of $\mathcal{T}$
in \zcref{subsec:regularity-structure} to note that it really does
map between those spaces. The key features of $\mathcal{K}^{\eps,\zeta}$
are that
\begin{equation}
\hat{\mathcal{R}}^{\eps,\zeta}\mathcal{K}^{\eps,\zeta}f=K\circledast\hat{\mathcal{R}}^{\eps,\zeta}f\label{eq:RK}
\end{equation}
by \cites[(5.17)]{hairer:2014:theory} and that
\begin{equation}
(\mathcal{K}^{\eps,\zeta}f)_{t}(x)-\mathcal{I}(f_{t}(x))\in\mathcal{T}_{\mathrm{poly}}\qquad\text{for all }t,x\label{eq:KfisIpluspoly}
\end{equation}
by the definition \cites[(5.15)]{hairer:2014:theory}. Combining \zcref{eq:RKhat,eq:RK},
we see that
\begin{equation}
\hat{\mathcal{R}}^{\eps,\zeta}(\mathcal{K}^{\eps,\zeta}+\tilde{\mathcal{K}}^{\eps,\zeta})f=(K+\tilde{K})\circledast\hat{\mathcal{R}}^{\eps,\zeta}f=p\circledast\hat{\mathcal{R}}^{\eps,\zeta}f.\label{eq:reconstruct-heatkernel}
\end{equation}

The final object necessary to write the solution theory for the KPZ
equation is the object $\mathbf{1}_{>0}\colon\mathbb{R}^{2}\to\mathbb{R}$
given by
\[
\left(\mathbf{1}_{>0}\right)_{t}(x)=\mathbf{1}\{t>0\}.
\]
Then, similarly to \cites[(15.8)]{friz:hairer:2020:course} (but with
the extra potential term $\rspotential$ added), we can write the
abstract version of the KPZ equation with boundary potential as
\begin{equation}
H=\check{\varpi}_{\mathrm{S}}(\mathcal{K}^{\eps,\zeta}+\tilde{\mathcal{K}}^{\eps,\zeta})\left(\mathbf{1}_{>0}\left(\oh(\partial H)^{2}+\rsnoise+\rspotential\right)\right)+\mathcal{P}_{0}h_{0}\label{eq:fixedpointeqn}
\end{equation}
for any initial condition $h_{0}\in\mathcal{C}^{\chi}$. As written,
this is a fixed point problem for functions defined on $\mathbb{R}^{2}$,
but we will say that $H$ is a solution to \zcref{eq:fixedpointeqn}
on $[0,T]$ if the left and right sides agree when evaluated at $(t,x)$
for $t\in[0,T]$ and $x\in\mathbb{R}$. 

To see that the equation \zcref{eq:fixedpointeqn} at least makes
sense, we note that if $G\in\mathcal{D}^{\eps,\zeta}_{\mathrm{S}}$,
then 
\begin{equation}
\partial G\in\mathcal{D}^{\oh+10\kappa,\chi-1;\eps,\zeta}(\mathcal{T}_{\mathrm{poly}}\oplus\mathcal{T}_{\mathcal{I}'})\label{eq:DGspace}
\end{equation}
by \cite[Prop.~6.15]{hairer:2014:theory}. From this, we then see
by \cite[Prop.~6.12]{hairer:2014:theory} (using that the regularity
of the sector $\mathcal{T}_{\mathrm{poly}}\oplus\mathcal{T}_{\mathcal{I}'}$
is $-\oh-\kappa$ by \zcref{tab:RS-table} to check the hypothesis)
that 
\begin{equation}
(\partial G)^{2}\in\mathcal{D}^{9\kappa,(-1-2\kappa)\wedge(2\chi-2);\eps,\zeta}(\mathcal{T}_{\mathrm{poly}}\oplus\mathcal{T}_{\mathcal{I}'}\oplus\mathcal{T}_{\star})=\mathcal{D}^{9\kappa,-1-2\kappa;\eps,\zeta}(\mathcal{T}_{\mathrm{poly}}\oplus\mathcal{T}_{\mathcal{I}'}\oplus\mathcal{T}_{\star})\subseteq\mathcal{D}^{\eps,\zeta}_{\mathrm{F}}.\label{eq:DG2space}
\end{equation}
In the second identity in \zcref{eq:DG2space} we used that, by the
assumption \zcref{eq:chirange} that $\chi>\oh-\kappa$, we have $2\chi-2>-1-2\kappa$.
In fact, the referenced propositions also include continuity statements,
and in particular
\begin{equation}
\text{the map }\mathcal{D}^{\eps,\zeta}_{\mathrm{S}}\ni G\mapsto(\partial G)^{2}\in\mathcal{D}^{\eps,\zeta}_{\mathrm{F}}\text{ is continuous.}\label{eq:squaring-cts}
\end{equation}

With these definitions in place, we have the following theorem.
\begin{thm}
\label{thm:fixedpointproblem}There is a $T>0$ (possibly depending
on the model and the initial condition $h_{0}$) such that \zcref{eq:fixedpointeqn}
has a unique solution $H$ on $[0,T]$, and moreover that
\begin{equation}
\text{if }T<\infty\text{, then }\lim_{t\uparrow T}\|(\hat{\mathcal{R}}^{\eps,\zeta}H)_{t}\|_{\mathcal{C}^{\chi}}=\infty.\label{eq:blowuptime}
\end{equation}
Moreover, the map $(\hat{\Pi}^{\eps,\zeta},\hat{\Gamma}^{\eps,\zeta})\mapsto(T,H)$
is can be chosen in a continuous manner. Finally, the solution $H$
is symmetric in the sense of \zcref{def:symmetric}.
\end{thm}

The uniqueness statement in \zcref{thm:fixedpointproblem} can be
proved in an identical manner to \cites[Thm.~7.8]{hairer:2014:theory}
or \cites[Thm.~4.16]{hairer:quastel:2018:class}, and the fact that
the solution can be continued up until the blow-up time in $\mathcal{C}^{\chi}$
is proved as in \cite[Prop.~7.11]{hairer:2014:theory}. Thus we omit
the details of the proof.\textbf{} We will shortly see that we can
in fact take $T=\infty$; see \zcref{cor:Tisinfinity} below.

\subsection{Relationship between the lifted and original problems}

In this section we show that the solution given via regularity structures
coincides with the Cole--Hopf solution to \zcref{eq:hepszeta} defined
in \zcref{sec:mollifying}. We begin by establishing an expansion
of a solution $H$ to \zcref{eq:fixedpointeqn} in terms of elements
of the regularity structure, which will also be crucial for our calculations
of the boundary flux in \zcref{prop:expand-nonlinearity-boundary}
below. The following proposition should be compared to the computations
in the proof of \cite[Prop.~15.26]{friz:hairer:2020:course} or \cite[(4.4)]{hairer:quastel:2018:class}.
\begin{prop}
Fix $\eps,\zeta>0$ and suppose that $H\in\mathcal{D}^{\eps,\zeta}_{\mathrm{S}}$
is a solution to \zcref{eq:fixedpointeqn} on $[0,T]$. Then we have
continuous functions $h,h'\colon(0,T]\times\mathbb{R}\to\mathbb{R}$
such that
\begin{equation}
H=\rsballoonr+\rsballoonb+\frac{1}{2}\rsicherryrr+\rsicherryrb+h'\rsilollipopr+\frac{1}{2}\rsielkrrr+h\rsone+h'\rsx\qquad\text{on }[0,T]\times\mathbb{R}\label{eq:Hexpansion-0-1-1}
\end{equation}
and so
\begin{equation}
(\partial H)^{2}=\rscherryrr+2\rscherryrb+\rselkrrr+2\rselkrbr+2h'\rsclawrr+\rsmooserrrr+\rscherrybb+\rselkrrb+\frac{1}{4}\rscandelabrarrrr+h'\left(2\rslollipopr+2\rslollipopb+\rsipcherryrr\right)+(h')^{2}\rsone\qquad\text{on }[0,T]\times\mathbb{R}.\label{eq:nonlinearity-expansion}
\end{equation}
Moreover, for each $t\in(0,T]$, $h_{t}$ and $h_{t}'$ are both $2L$-periodic,
while $h_{t}$ is even and $h'_{t}$ is odd, and in particular we
have
\begin{equation}
h'_{t}(x_{0})=0\qquad\text{for all }t\in(0,T]\text{ and }x_{0}\in L\mathbb{Z}.\label{eq:h'0is0}
\end{equation}
\end{prop}

\begin{proof}
Since $\tilde{\mathcal{K}}^{\eps,\zeta}f,\mathcal{K}^{\eps,\zeta}f-\mathcal{I}f\in\mathcal{D}_{\mathrm{S};\mathrm{poly}}$
(recalling \zcref[range]{eq:KfisIpluspoly,eq:DKPZpoly}), we have
functions $h,h'$ that are continuous on $(0,T]\times\mathbb{R}$
such that, on $[0,T]\times\mathbb{R}$, we have
\begin{equation}
H=\rsballoonr+\rsballoonb+\frac{1}{2}\check{\varpi}_{\mathrm{S}}\mathcal{I}(\partial H)^{2}+h\rsone+h'\rsx.\label{eq:Hexpansion-0}
\end{equation}
We can differentiate this to obtain
\begin{equation}
\partial H=\rslollipopr+\rslollipopb+\frac{1}{2}\partial\check{\varpi}_{\mathrm{S}}\mathcal{I}(\partial H)^{2}+h'\rsone,\label{eq:Hexpansion-deriv}
\end{equation}
and then square and apply $\mathcal{I}$ to obtain
\begin{equation}
\mathcal{I}(\partial H)^{2}=\rsicherryrr+2\rsicherryrb+2h'\rsilollipopr+\mathcal{I}\left((\rslollipopr+\rslollipopb+h'\rsone)\partial\check{\varpi}_{\mathrm{S}}\mathcal{I}(\partial H)^{2}\right)+\frac{1}{4}\mathcal{I}\left(\left(\partial\check{\varpi}_{\mathrm{S}}\mathcal{I}(\partial H)^{2}\right)^{2}\right).\label{eq:IpartialH}
\end{equation}
From \zcref{eq:Hexpansion-deriv} we see that $\partial H-\rslollipopr$
has terms only of homogeneity at least $-\kappa$, and so $(\partial H)^{2}-\rscherryrr$
has terms only of homogeneity at least $-\oh-2\kappa$, so $\partial\check{\varpi}_{\mathrm{S}}\mathcal{I}(\partial H)^{2}-\rsipcherryrr$
has terms only of homogeneity at least $\oh-2\kappa$. This has two
consequences in particular:
\begin{enumerate}
\item The quantity $\partial\check{\varpi}_{\mathrm{S}}\mathcal{I}(\partial H)^{2}$
has terms only of homogeneity at least $|\rsipcherryrr|=-2\kappa$,
so $\left(\partial\mathcal{I}(\partial H)^{2}\right)^{2}$ has terms
only of homogeneity at least $-4\kappa$, and hence $\check{\varpi}_{\mathrm{S}}\mathcal{I}\left(\left(\partial\mathcal{I}(\partial H)^{2}\right)^{2}\right)=0$.
\item We can write
\[
\mathcal{I}\left((\rslollipopr+\rslollipopb+h_{0}'\rsone)\partial\mathcal{I}(\partial H)^{2}\right)=\rsielkrrr+\mathcal{I}\left((\rslollipopb+h'\rsone)\partial\mathcal{I}(\partial H)^{2}\right)+\mathcal{I}\left(\rslollipopr\left(\partial\mathcal{I}(\partial H)^{2}-\rsipcherryrr\right)\right).
\]
Both the second and third terms on the right have terms only of homogeneity
at least $2-3\kappa$, so
\[
\check{\varpi}_{\mathrm{S}}\mathcal{I}\left((\rslollipopr+\rslollipopb+h'\rsone)\partial\mathcal{I}(\partial H)^{2}\right)=\rsielkrrr.
\]
\end{enumerate}
Using these two observations in \zcref{eq:IpartialH}, we see that
\begin{equation}
\check{\varpi}_{\mathrm{S}}\mathcal{I}(\partial H)^{2}=\rsicherryrr+2\rsicherryrb+2h'\rsilollipopr+\rsielkrrr,\label{eq:IpartialH-1}
\end{equation}
and using this in \zcref{eq:Hexpansion-0}, we obtain \zcref{eq:Hexpansion-0-1-1}.
Then \zcref{eq:nonlinearity-expansion} follows by the definitions
of the operations.

Since $H$ is symmetric as noted in \zcref{thm:fixedpointproblem},
we can use \zcref{eq:Hexpansion-0-1-1} twice along with \zcref{eq:sym_group_generators,eq:sigma-action},
to write
\begin{equation}
\begin{aligned}\rsballoonr & +\rsballoonb+\frac{1}{2}\rsicherryrr+\rsicherryrb+h'_{t}(x+2L)\rsilollipopr+\frac{1}{2}\rsielkrrr+h_{t}(x+2L)\rsone+h'_{t}(x+2L)\rsx=\sigma_{\mathrm{trans}}\cdot H_{t}(\sigma_{\mathrm{trans}}(x))=H_{t}(x)\\
 & =\rsballoonr+\rsballoonb+\frac{1}{2}\rsicherryrr+\rsicherryrb+h'_{t}(x)\rsilollipopr+\frac{1}{2}\rsielkrrr+h_{t}(x)\rsone+h'_{t}(x)\rsx.
\end{aligned}
\label{eq:trans-symmetry-apply}
\end{equation}
Hence, $h_{t}$ and $h_{t}'$ are $2L$-periodic for each $t$. Similarly,
we write
\begin{equation}
\begin{aligned}\rsballoonr & +\rsballoonb+\frac{1}{2}\rsicherryrr+\rsicherryrb-h'_{t}(-x)\rsilollipopr+\frac{1}{2}\rsielkrrr+h_{t}(-x)\rsone-h'_{t}(-x)\rsx=\sigma_{\mathrm{refl}}\cdot H_{t}(\sigma_{\mathrm{refl}}(x))=H_{t}(x)\\
 & =\rsballoonr+\rsballoonb+\frac{1}{2}\rsicherryrr+\rsicherryrb+h'_{t}(x)\rsilollipopr+\frac{1}{2}\rsielkrrr+h_{t}(x)\rsone+h'_{t}(x)\rsx,
\end{aligned}
\label{eq:refl-symmetry-apply}
\end{equation}
and from this we see that $h_{t}$ is even and $h'_{t}$ is odd for
each $t$. Given these considerations, \zcref{eq:h'0is0} follows
from the continuity of $h'_{t}$. 
\end{proof}

The following proposition is crucial to the argument. It says that
the reconstruction of the modeled distribution $(\partial H)^{2}$
is precisely the renormalized nonlinearity of the Cole--Hopf solution
to the KPZ equation.
\begin{prop}
\label{prop:RSrecovers}Let $A^{\zeta}$ be as in \zcref{eq:Azetadef}
and $h^{\eps,\zeta}_{\uu,\vv;t}$ be as in \zcref{eq:hepszetalog},
so $(h^{\eps,\zeta}_{\uu,\vv;t})$ solves \zcref[range]{eq:hepszeta-eqn,eq:hepszeta-ic}.
If $H$ solves \zcref{eq:fixedpointeqn} with 
\begin{equation}
h_{0}=A^{\zeta},\label{eq:h0isAzeta}
\end{equation}
then 
\begin{equation}
h^{\eps,\zeta}_{\uu,\vv;t}(x)=(\hat{\mathcal{R}}^{\eps,\zeta}H)_{t}(x)\qquad\text{for all }t\ge0\text{ and all }x\in\mathbb{R},\label{eq:hisRH}
\end{equation}
and moreover
\begin{equation}
\left(\partial_{x}h^{\eps,\zeta}_{\uu,\vv;t}(x)\right)^{2}-C^{(1)}_{\zeta}(x)=(\hat{\mathcal{R}}^{\eps,\zeta}(\partial H)^{2})_{t}(x)\qquad\text{for all }t\ge0\text{ and all }x\in\mathbb{R}.\label{eq:nonlinearityisRnonlinearity}
\end{equation}
\end{prop}

Before we prove \zcref{prop:RSrecovers}, we note the following corollary.
\begin{cor}
\label{cor:Tisinfinity}With probability $1$, \zcref{thm:fixedpointproblem}
holds with $T=\infty$. Moreover, for any $T<\infty$, the map $(\hat{\Pi}^{\eps,\zeta},\hat{\Gamma}^{\eps,\zeta})\mapsto H^{\eps,\zeta}$
is continuous.
\end{cor}

\begin{proof}
We know that $\|h^{\eps,\zeta}_{\uu,\vv;t}\|_{\mathcal{C}^{\chi}}<\infty$
for all $t<\infty$ with probability $1$ by the Cole--Hopf transform
(since the solution to \zcref{Zepszeta} remains finite and positive
for all time with probability $1$ by \zcref{prop:moment-bd}), so
the only way that \zcref{eq:blowuptime} can hold is if $T=\infty$.
The continuity of the solution map follows from the continuity statement
in \zcref{thm:fixedpointproblem} along with the fact that the solution
can be restarted at any time at which it does not blow up, as shown
in \cite[Prop.~7.11]{hairer:2014:theory}.
\end{proof}

To prove \zcref{prop:RSrecovers}, we start with the following lemma.
Recall the definition \zcref{eq:C1zetadef} of $C^{(1)}_{\zeta}(x)$.
\begin{lem}
If $H$ solves \zcref{eq:fixedpointeqn}, then
\begin{equation}
\left(\partial_{x}(\hat{\mathcal{R}}^{\eps,\zeta}H)_{s}(x)\right)^{2}-\left(\hat{\mathcal{R}}^{\eps,\zeta}(\partial H)^{2}\right)_{s}(x)=C^{(1)}_{\zeta}(x).\label{eq:sqdiff}
\end{equation}
\end{lem}

\begin{proof}
Applying $\partial$ to both sides of \zcref{eq:Hexpansion-0-1-1},
we get
\begin{equation}
\partial H=\rslollipopr+\rslollipopb+\frac{1}{2}\rsipcherryrr+\rsipcherryrb+h'\rsiplollipopr+\frac{1}{2}\rsipelkrrr+h'\rsone.\label{eq:dH}
\end{equation}
Thus we can expand
\begin{align*}
\partial_{x}(\hat{\mathcal{R}}^{\eps,\zeta}H)_{s}(x)\ovset{\zcref{eq:Rhat-commutes-deriv}} & =(\hat{\mathcal{R}}^{\eps,\zeta}\partial H)_{s}(x)\\
\ovset{\zcref{eq:dH}} & =\left(\hat{\mathcal{R}}^{\eps,\zeta}\left(\rslollipopr+\rslollipopb+\frac{1}{2}\rsipcherryrr+\rsipcherryrb+h'\rsiplollipopr+\frac{1}{2}\rsipelkrrr+h'\rsone\right)\right)_{s}(x)\\
\ovset{\zcref{eq:reconstruction-is-model-1}} & =\hat{\Pi}^{\eps,\zeta}_{s,x}(\rslollipopr)_{s}(x)+\hat{\Pi}^{\eps,\zeta}_{s,x}(\rslollipopb)_{s}(x)+\frac{1}{2}\hat{\Pi}^{\eps,\zeta}_{s,x}(\rsipcherryrr)_{s}(x)+\hat{\Pi}^{\eps,\zeta}_{s,x}(\rsipcherryrb)_{s}(x)\\
 & \qquad+h'_{s}(x)\cdot\hat{\Pi}^{\eps,\zeta}_{s,x}(\rsiplollipopr)_{s}(x)+\frac{1}{2}\hat{\Pi}^{\eps,\zeta}_{s,x}(\rsipelkrrr)_{s}(x)+h'_{s}(x)\hat{\Pi}^{\eps,\zeta}_{s,x}(\rsipcherryrb)_{s}(\rsone)\\
 & =\rslollipopr[rsK]^{\eps,\zeta}_{s}(x)+\rslollipopb[rsK]^{\eps,\zeta}_{s}(x)+\frac{1}{2}\rsiprenorm[rsK]^{\eps,\zeta}_{s}(x)+h'_{s}(x),
\end{align*}
where in the last identity we used the table of values for $\hat{\Pi}^{\eps,\zeta}_{s,x}$
recorded in \zcref{tab:model-defns}. Squaring this (in the ordinary
sense of functions), we obtain
\begin{align}
\left(\partial_{x}(\hat{\mathcal{R}}^{\eps,\zeta}H)_{s}(x)\right)^{2} & =\left(\rslollipopr[rsK]^{\eps,\zeta}_{s}(x)+\rslollipopb[rsK]^{\eps,\zeta}_{s}(x)+\frac{1}{2}\rsiprenorm[rsK]^{\eps,\zeta}_{s}(x)+h'_{s}(x)\right)^{2}\nonumber \\
 & =\rscherryrr[rsK]^{\eps,\zeta}_{s}(x)+\rscherrybb[rsK]^{\eps,\zeta}_{s}(x)+\frac{1}{4}\rscherryrenormrenorm[rsK]^{\eps,\zeta}_{s}(x)+2\rscherryrb[rsK]^{\eps,\zeta}_{s}(x)+\rscherryrenormr[rsK]^{\eps,\zeta}_{s}(x)+2\rscherryrenormb[rsK]^{\eps,\zeta}_{s}(x)\nonumber \\
 & \qquad+h'_{s}(x)\left(2\rslollipopr[rsK]^{\eps,\zeta}_{s}(x)+2\rslollipopb[rsK]^{\eps,\zeta}_{s}(x)+\rsiprenorm[rsK]^{\eps,\zeta}_{s}(x)\right)+h'_{s}(x)^{2}.\label{eq:applysquare}
\end{align}
On the other hand, performing a similar expansion on \zcref{eq:nonlinearity-expansion},
we obtain
\begin{align}
\hat{\mathcal{R}}^{\eps,\zeta}((\partial H)^{2})_{s}(x)\ovset{\zcref{eq:reconstruction-is-model-1}} & =\hat{\Pi}^{\eps,\zeta}_{s,x}(\rscherryrr)_{s}(x)+2\hat{\Pi}^{\eps,\zeta}_{s,x}(\rscherryrb)_{s}(x)+\hat{\Pi}^{\eps,\zeta}_{s,x}(\rselkrrr)_{s}(x)+\hat{\Pi}^{\eps,\zeta}_{s,x}(\rsmooserrrr)_{s}(x)+\hat{\Pi}^{\eps,\zeta}_{s,x}(\rscherrybb)_{s}(x)\nonumber \\
 & \quad+\hat{\Pi}^{\eps,\zeta}_{s,x}(\rselkrrb)_{s}(x)+\frac{1}{4}\hat{\Pi}^{\eps,\zeta}_{s,x}(\rscandelabrarrrr)_{s}(x)\nonumber \\
 & \quad+h'_{s}(x)\left(2\hat{\Pi}^{\eps,\zeta}_{s,x}(\rslollipopr)_{s}(x)+2\hat{\Pi}^{\eps,\zeta}_{s,x}(\rslollipopb)_{s}(x)+\hat{\Pi}^{\eps,\zeta}_{s,x}(\rsipcherryrr)_{s}(x)\right)+h'_{s}(x)^{2}\nonumber \\
 & =\rsrenorm[rsK]^{\eps,\zeta}_{s}(x)+2\rscherryrb[rsK]^{\eps,\zeta}_{s}(x)+\rscherryrenormr[rsK]^{\eps,\zeta}_{s}(x)+\rscherrybb[rsK]^{\eps,\zeta}_{s}(x)+\rscherryrenormb[rsK]^{\eps,\zeta}_{s}(x)+\frac{1}{4}\rscherryrenormrenorm[rsK]^{\eps,\zeta}_{s}(x)\nonumber \\
 & \quad+h'_{s}(x)\left(2\rslollipopr[rsK]^{\eps,\zeta}_{s}(x)+2\rslollipopb[rsK]^{\eps,\zeta}_{s}(x)+\rsiprenorm[rsK]^{\eps,\zeta}_{s}(x)\right)+h'_{s}(x)^{2},\label{eq:Rofsquare}
\end{align}
where in the last identity we again used \zcref{tab:model-defns}.
Now we can subtract \zcref{eq:Rofsquare} from \zcref{eq:applysquare}
to get
\[
\left(\partial_{x}(\hat{\mathcal{R}}^{\eps,\zeta}H)_{s}(x)\right)^{2}-\left(\hat{\mathcal{R}}^{\eps,\zeta}[(\partial H)^{2}]\right)_{s}(x)=\rscherryrr[rsK]^{\eps,\zeta}_{s}(x)-\rsrenorm[rsK]^{\eps,\zeta}_{s}(x)\overset{\zcref{eq:Qrenorm}}{=}C^{(1)}_{\zeta}(x),
\]
which is \zcref{eq:sqdiff}.
\end{proof}

We will also need to an explicit formula for the renormalization constant
$C^{(1)}_{\zeta}$. The following proposition is an analogue in our
setting of \cite[Lem.~6.3]{hairer:2013:solving}. Note that the reflection
of the noise leads to the spatial inhomogeneity. In the limit $\zeta\to0$,
the spatial inhomogeneity converges to a delta function on each boundary.
\begin{prop}
\label{prop:C1zeta}We have 
\begin{equation}
C^{(1)}_{\zeta}(x)=\Sh^{\zeta}_{2L}(0)-\oh\Sh^{\zeta/2}_{L}(x).\label{eq:Echerryrr}
\end{equation}
\end{prop}

The proof of \zcref{prop:C1zeta} is a rather straightforward calculation
that we carry out in \zcref{subsec:First-renormalization-constant}
below. For now, we use it to complete the proof of \zcref{prop:RSrecovers}.
\begin{proof}[Proof of \zcref{prop:RSrecovers}]
Applying the reconstruction operator $\hat{\mathcal{R}}^{\eps,\zeta}$
to both sides of \zcref{eq:fixedpointeqn}, and then using \zcref{eq:reconstruct-heatkernel,eq:reconstruct-ic,eq:nonlinearity-expansion},
we get
\begin{align}
\hat{\mathcal{R}}^{\eps,\zeta}H & =\hat{\mathcal{R}}^{\eps,\zeta}(\mathcal{K}^{\eps,\zeta}+\tilde{\mathcal{K}}^{\eps,\zeta})\left[\mathbf{1}_{>0}\left(\oh(\partial H)^{2}+\rsnoise+\rspotential\right)\right]+\hat{\mathcal{R}}\mathcal{P}_{0}h_{0}\nonumber \\
 & =p\circledast\left[\hat{\mathcal{R}}^{\eps,\zeta}\left(\mathbf{1}_{>0}\left(\oh(\partial H)^{2}+\rsnoise+\rspotential\right)\right)\right]+p*h_{0}\nonumber \\
\ovset{\zcref{eq:sqdiff}} & =p\circledast\left[\mathbf{1}_{>0}\left(\oh\left(\partial_{x}(\hat{\mathcal{R}}^{\eps,\zeta}H)\right)^{2}-\oh C^{(1)}_{\zeta}+\rsnoise[rsK]^{\eps,\zeta}+\rspotential[rsK]^{\eps,\zeta}\right)\right]+p*h_{0}.\label{eq:reconstruct-mild}
\end{align}
Using the values of $\rsnoise[rsK]^{\eps,\zeta}$, $\rspotential[rsK]^{\eps,\zeta}$,
$C^{(1)}_{\zeta}$, and $h_{0}$ fixed in \zcref{eq:basic-reconstruction,eq:IpartialH-1,eq:Echerryrr},
we see that \zcref{eq:reconstruct-mild} is exactly the mild solution
formula for the problem \zcref{eq:hepszeta}, and hence we obtain
\zcref{eq:hisRH}. Then \zcref{eq:nonlinearityisRnonlinearity} follows
from \zcref{eq:hisRH} and another application of \zcref{eq:sqdiff}.
\end{proof}

\subsection{Canonical lifts as modeled distributions}

It will be useful later on to be able to understand the canonical
lifts $\mathbf{\hat{\boldsymbol{\boldsymbol{\varPi}}}}^{\eps,\zeta}\tau$
and $\mathbf{\tilde{\boldsymbol{\boldsymbol{\varPi}}}}^{\eps,\zeta}(\hat{M}\tau)$
as modeled distributions, so that we can use the tools of regularity
structures to study them. This is a sort of ``inverse'' to the procedure
carried out in \zcref{subsec:Model}: there, we defined the model
in terms of the canonical lifts, whereas now we seek to write the
canonical lifts in terms of the model. A similar procedure was outlined
in \cite[Rmk.~15.13]{friz:hairer:2020:course}, which we essentially
follow, but we need to include additional considerations because we
want our relationship to apply to the renormalized model and lifts.
Since we will only need this construction for elements of $\check{\mathcal{T}}_{\ES}$
(recalling \zcref{def:We-define-a}) and $\mathbf{\tilde{\boldsymbol{\boldsymbol{\varPi}}}}^{\eps,\zeta}(\hat{M}\tau)$
has only been defined for elements of $\check{\mathcal{T}}_{\ES}$
anyway, we will restrict our construction that space.

For $V$ a sector of $\mathcal{T}$, we define the modeled distribution
spaces
\[
\mathcal{D}^{\eps,\zeta}_{\mathcal{I}'}(V)\coloneqq\mathcal{D}^{\oh+10\kappa;\eps,\zeta}(V)\qquad\text{and}\qquad\mathcal{D}^{\eps,\zeta}_{\bullet}\coloneqq\mathcal{D}^{\eps,\zeta}_{\star}\coloneqq\mathcal{D}^{9\kappa;\eps,\zeta}(V)
\]
and also the projection maps $\check{\varpi}_{\mathcal{I}'}\colon\mathcal{T}\to\mathcal{T}_{\le\oh+10\kappa}=\mathcal{T}_{\le\oh-\kappa}$
and $\check{\varpi}_{\bullet}=\check{\varpi}_{\star}\colon\mathcal{T}\to\mathcal{T}_{\le9\kappa}=\mathcal{T}_{\le0}$.
Along the lines of \cite[Rmk.~15.13]{friz:hairer:2020:course}, we
define, for $\square\in\{\bullet,\mathcal{I}',\star\}$, maps 
\begin{equation}
\mathcal{H}^{\eps,\zeta}_{\square},\tilde{\mathcal{H}}^{\eps,\zeta}_{\square}\colon\check{\mathcal{T}}_{\ES}\cap\check{\mathcal{T}}_{\square}\to\mathcal{D}^{\eps,\zeta}_{\square}(\check{\mathcal{T}})\label{eq:Hspaces}
\end{equation}
 by the mutually recursive inductive relations
\begin{equation}
\mathcal{H}^{\eps,\zeta}_{\bullet}(\tau)\coloneqq\tilde{\mathcal{H}}^{\eps,\zeta}_{\bullet}(\tau)\coloneqq\tau\qquad\text{for }\tau\in\{\rsnoise,\rspotential\},\label{eq:HHtildebullet}
\end{equation}
\begin{equation}
\begin{aligned}\mathcal{H}^{\eps,\zeta}_{\mathcal{I}'}(\mathcal{I}'\tau)=\check{\varpi}_{\mathcal{I}'}\partial\mathcal{K}^{\eps,\zeta}\mathcal{H}^{\eps,\zeta}_{\star}(\tau)\qquad\text{and}\qquad\tilde{\mathcal{H}}^{\eps,\zeta}_{\mathcal{I}'}(\mathcal{I}'\tau)=\check{\varpi}_{\mathcal{I}'}\partial(\mathcal{K}^{\eps,\zeta}+\tilde{\mathcal{K}}^{\eps,\zeta})\tilde{\mathcal{H}}^{\eps,\zeta}_{\star}(\tau)\qquad\\
\text{for }\tau\in\check{\mathsf{T}}_{\ES}\cap(\check{\mathsf{T}}_{\bullet}\cup\check{\mathsf{T}}_{\star});
\end{aligned}
\label{eq:HHtildeIprime}
\end{equation}
and
\begin{equation}
\begin{aligned}\mathcal{H}^{\eps,\zeta}_{\star}(\tau_{1}\tau_{2})=\check{\varpi}_{\star}\left[\mathcal{H}^{\eps,\zeta}_{\mathcal{I}'}(\tau_{1})\star\mathcal{H}^{\eps,\zeta}_{\mathcal{I}'}(\tau_{2})\right]\qquad\text{and}\qquad\tilde{\mathcal{H}}^{\eps,\zeta}_{\star}(\tau_{1}\tau_{2})=\check{\varpi}_{\star}\left[\tilde{\mathcal{H}}^{\eps,\zeta}_{\mathcal{I}'}(\tau_{1})\star\tilde{\mathcal{H}}^{\eps,\zeta}_{\mathcal{I}'}(\tau_{2})\right]\qquad\\
\text{for }\tau_{1},\tau_{2}\in\check{\mathsf{T}}_{\ES}\cap\check{\mathsf{T}}_{\mathcal{I}'};
\end{aligned}
\label{eq:HHtildeproduct}
\end{equation}
and then extending these operators to each of their respective $\check{\mathcal{T}}_{\ES}\cap\check{\mathcal{T}}_{\square}$
spaces by linearity. We will check in \zcref{lem:Lranges} below
that the range of these operators really is as claimed in \zcref{eq:Hspaces}.
We finally define $\mathcal{H}^{\eps,\zeta},\tilde{\mathcal{H}}^{\eps,\zeta}\colon\check{\mathcal{T}}_{\ES}\to\mathcal{D}^{\eps,\zeta}_{\star}(\check{\mathcal{T}})$
by
\[
\mathcal{H}^{\eps,\zeta}\tau=\mathcal{H}^{\eps,\zeta}_{\bullet}\tau_{\bullet}+\mathcal{H}^{\eps,\zeta}_{\mathcal{I}'}\tau_{\mathcal{I}'}+\mathcal{H}^{\eps,\zeta}_{\star}\tau_{\star}\qquad\text{and}\qquad\tilde{\mathcal{H}}^{\eps,\zeta}\tau=\tilde{\mathcal{H}}^{\eps,\zeta}_{\bullet}\tau_{\bullet}+\tilde{\mathcal{H}}^{\eps,\zeta}_{\mathcal{I}'}\tau_{\mathcal{I}'}+\tilde{\mathcal{H}}^{\eps,\zeta}_{\star}\tau_{\star}
\]
for $\tau=\tau_{\bullet}+\tau_{\mathcal{I}'}+\tau_{\star}$ with $\tau\in\check{\mathcal{T}}_{\ES}$
and $\tau_{\square}\in\check{\mathcal{T}}_{\ES}\cap\check{\mathcal{T}}_{\square}$
for $\square\in\{\bullet,\mathcal{I}',\star\}$.

The second definition in \zcref{eq:HHtildeIprime} is a minor abuse
of notation: in general, $\tilde{\mathcal{K}}^{\eps,\zeta}\tilde{\mathcal{H}}^{\eps,\zeta}_{\star}(\tau)$
may not be defined, since the first space-time convolution on the
right side of \zcref{eq:scrKhatdef} can blow up for large negative
times. However, because of the presence of the gradient, this term
does not appear in the definition of $\tilde{\mathcal{H}}^{\eps,\zeta}(\mathcal{I}'\tau)$
anyway. So the second definition in \zcref{eq:HHtildeIprime} should
really be interpreted as
\[
\tilde{\mathcal{H}}^{\eps,\zeta}(\mathcal{I}'\tau)=\check{\omega}_{\mathrm{F}}\partial\mathcal{K}^{\eps,\zeta}\tilde{\mathcal{H}}^{\eps,\zeta}(\tau)+\tilde{K}'\circledast(\hat{\mathcal{R}}^{\eps,\zeta}\tilde{\mathcal{H}}^{\eps,\zeta}(\tau))\rsone.
\]
The last integral is indeed well-defined, for a reason analogous to
that described in the discussion following \zcref{eq:Qinductive-hat}
(in particular given \zcref{eq:RHatHtildeisPitilde} below).

We define
\[
\mathcal{L}^{\eps,\zeta}_{\square}\tau=\mathcal{H}^{\eps,\zeta}_{\square}\tau-\tau\qquad\text{and}\qquad\tilde{\mathcal{L}}^{\eps,\zeta}_{\square}\tau=\tilde{\mathcal{H}}^{\eps,\zeta}_{\square}\tau-\tau.
\]
We have not yet checked that the ranges of $\mathcal{H}^{\eps,\zeta}_{\square},\tilde{\mathcal{H}}^{\eps,\zeta}_{\square}$
are actually as described in \zcref{eq:Hspaces}; we will do that
now, and actually prove a somewhat stronger result that will be useful
later. Define, for $\square\in\{\bullet,\mathcal{I}',\star\}$,
\begin{equation}
\mathcal{T}_{\mathrm{rem},\square}=\langle\mathsf{T}_{\mathrm{rem},\square}\rangle,\qquad\mathsf{T}_{\mathrm{rem},\square}\coloneqq\begin{cases}
\varnothing, & \text{if }\square=\bullet;\\
\{\rsone,\rsiplollipopr\}, & \text{if }\square=\mathcal{I}';\\
\{\rsone,\rsclawrr,\rsiplollipopr\}\cup\check{\mathsf{T}}_{\mathcal{I}'}, & \text{if }\square=\star.
\end{cases}\label{eq:Tremainder-1-1}
\end{equation}

\begin{lem}
\label{lem:Lranges}We have
\begin{equation}
\mathcal{L}^{\eps,\zeta}_{\square}\tau,\tilde{\mathcal{L}}^{\eps,\zeta}_{\square}\tau\in\mathcal{T}_{\mathrm{rem,\square}}\qquad\text{for each }\square\in\{\bullet,\mathcal{I},\star\}\text{ and }\tau\in\check{\mathcal{T}}_{\ES}\cap\check{\mathcal{T}}_{\square}.\label{eq:Lranges}
\end{equation}
As a consequence, the ranges of $\mathcal{H}^{\eps,\zeta}_{\square},\tilde{\mathcal{H}}^{\eps,\zeta}_{\square}$
are as described in \zcref{eq:Hspaces}. In fact, for each $\tau\in\check{\mathsf{T}}_{\ES}\cap\check{\mathsf{T}}_{\square}$,
$\|\mathcal{H}^{\eps,\zeta}_{\square}\tau\|_{\mathcal{D}^{\eps,\zeta}_{\square}}$
and $\|\tilde{\mathcal{H}}^{\eps,\zeta}_{\square}\tau\|_{\mathcal{D}^{\eps,\zeta}_{\square}}$
are each bounded by a polynomial in $\|(\hat{\Pi}^{\eps,\zeta},\hat{\Gamma}^{\eps,\zeta})\|$
depending on $\tau$.
\end{lem}

\begin{proof}
The proof of \zcref{eq:Lranges} (of course) proceeds by induction.
The base case $\tau\in\check{\mathcal{T}}_{\bullet}$ is obvious from
the definitions. If $\tau\in\check{\mathcal{T}}_{\ES}\cap(\check{\mathcal{T}}_{\bullet}\oplus\check{\mathcal{T}}_{\star})$
and \zcref{eq:Lranges} holds for $\tau$, then from the definitions
of $\mathcal{K}^{\eps,\zeta}$ and $\tilde{\mathcal{K}}^{\eps,\zeta}$
and the inductive hypothesis we see that $\partial\mathcal{K}^{\eps,\zeta}\mathcal{H}^{\eps,\zeta}_{\star}\tau-\tau$
and $\partial(\mathcal{K}^{\eps,\zeta}+\tilde{\mathcal{K}}^{\eps,\zeta})\tilde{\mathcal{H}}^{\eps,\zeta}_{\star}\tau-\tau$
take values in the span of $\rsone$ and $\mathcal{I}'\check{\mathsf{T}}_{\mathcal{I}'}$.
But the only basis element of $\mathcal{I}'\check{\mathsf{T}}_{\mathcal{I}'}$
that is not annihilated by the projection $\check{\varpi}_{\mathcal{I}'}$
is $\rsiplollipopr$, so \zcref{eq:Lranges} holds for $\mathcal{I}'\tau$.
Finally, if $\tau_{1},\tau_{2}\in\check{\mathcal{T}}_{\ES}\cap\check{\mathcal{T}}_{\mathcal{I}'}$,
then
\[
\mathcal{L}^{\eps,\zeta}_{\star}(\tau_{1}\tau_{2})=\check{\varpi}_{\star}\left[(\mathcal{H}^{\eps,\zeta}_{\mathcal{I}'}\tau_{1})(\mathcal{H}^{\eps,\zeta}_{\mathcal{I}'}\tau_{2})-\tau_{1}\tau_{2}\right]=\check{\varpi}_{\star}\left[\tau_{1}\mathcal{L}^{\eps,\zeta}_{\mathcal{I}'}\tau_{2}+\tau_{2}\mathcal{L}^{\eps,\zeta}_{\mathcal{I}'}\tau_{1}+(\mathcal{L}^{\eps,\zeta}_{\mathcal{I}'}\tau_{1})(\mathcal{L}^{\eps,\zeta}_{\mathcal{I}'}\tau_{2})\right],
\]
which takes values in the span of $\rsone,\rsclawrr,\rsiplollipopr,\check{\mathsf{T}}_{\mathcal{I}'}$
by the inductive hypothesis and \zcref{tab:RS-table,tab:multiplication-table}.
The same holds with $\mathcal{L}^{\eps,\zeta}_{\star}$ replaced by
$\tilde{\mathcal{L}}^{\eps,\zeta}_{\star}$, so \zcref{eq:Lranges}
holds for $\tau_{1}\tau_{2}$ as well.

The last claim of the lemma is also proved by induction. It follows
from \zcref{eq:Lranges} that, for any $\tau_{1},\tau_{2}\in\check{\mathcal{T}}_{\ES}\cap\check{\mathcal{T}}_{\mathcal{I}'}$,
$\mathcal{H}^{\eps,\zeta}_{\star}\tau_{i}$ ($i=1,2$) takes values
in a sector of regularity $-\oh-\kappa$, so by the inductive hypothesis
and \cites[Thm.~4.7]{hairer:2014:theory},\footnote{The hypothesis that the sectors are $\gamma$-regular is easily checked
in the same manner as for the usual KPZ equation.} $(\mathcal{H}^{\eps,\zeta}_{\star}\tau_{1})(\mathcal{H}^{\eps,\zeta}_{\star}\tau_{2})\in\mathcal{D}^{-\oh-\kappa+\oh+10\kappa;\eps,\zeta}=\mathcal{D}^{9\kappa;\eps,\zeta}$,
and so the same is true for the projection in the definition of $\mathcal{H}^{\eps,\zeta}_{\star}$
in \zcref{eq:HHtildeproduct}. Moreover, for $\tau\in\check{\mathcal{T}}_{\ES}\cap(\check{\mathcal{T}}_{\bullet}\oplus\check{\mathcal{T}}_{\star})$,
if $\mathcal{H}^{\eps,\zeta}_{\star}(\tau)\in\mathcal{D}^{9\kappa;\eps,\zeta}$,
then from the definition \zcref{eq:HHtildeIprime} and \cites[Thm.~5.12 and Prop.~5.28]{hairer:2014:theory},
we have $\mathcal{H}^{\eps,\zeta}_{\star}(\tau)\in\mathcal{D}^{1+9\kappa;\eps,\zeta}\supseteq\mathcal{D}^{\oh+10\kappa;\eps,\zeta}.$
The same arguments work with $\mathcal{H}$ replaced by $\tilde{\mathcal{H}}$.
Finally, each of these statements from \cite{hairer:2014:theory}
comes with a polynomial bound, so the polynomial boundedness statement
follows easily by induction as well.
\end{proof}

The purpose of defining $\mathcal{H}^{\eps,\zeta}$ and $\tilde{\mathcal{H}}^{\eps,\zeta}$
in this way was to obtain the following.
\begin{prop}
\label{prop:RH}For all $\tau\in\check{\mathcal{T}}_{\ES}$ and all
$\eps,\zeta>0$, we have 
\begin{equation}
\hat{\mathcal{R}}^{\eps,\zeta}\mathcal{H}^{\eps,\zeta}\tau=\hat{\boldsymbol{\varPi}}^{\eps,\zeta}\tau\label{eq:RhatHisPi}
\end{equation}
and
\begin{equation}
\hat{\mathcal{R}}^{\eps,\zeta}\tilde{\mathcal{H}}^{\eps,\zeta}\tau=\tilde{\boldsymbol{\varPi}}^{\eps,\zeta}(\hat{M}\tau).\label{eq:RHatHtildeisPitilde}
\end{equation}
\end{prop}

\begin{proof}
By linearity it suffices to check these relations on the basis $\check{\mathsf{T}}_{\ES}$.
We work by induction, following the recursion in \zcref{def:We-define-a}.
For the base case, we note that \zcref[comp=true, range]{eq:RhatHisPi,eq:RHatHtildeisPitilde}
are clear for $\tau\in\{\rsnoise,\rspotential\}$ immediately from
the definitions in \zcref{eq:HHtildebullet}.

Now suppose that $\tau\in\check{\mathsf{T}}_{\ES}\cap(\check{\mathsf{T}}_{\bullet}\cup\check{\mathsf{T}}_{\star})$
and \zcref[comp=true, range]{eq:RhatHisPi,eq:RHatHtildeisPitilde}
hold for $\tau$. We claim that they also hold for $\mathcal{I}'\tau$.
Indeed, we can compute
\begin{align*}
\hat{\mathcal{R}}^{\eps,\zeta}\mathcal{H}^{\eps,\zeta}(\mathcal{I}'\tau)\ovset{\zcref{eq:HHtildeIprime}} & =\hat{\mathcal{R}}^{\eps,\zeta}\left[\check{\varpi}_{\mathrm{F}}\partial\mathcal{K}^{\eps,\zeta}\mathcal{H}^{\eps,\zeta}(\tau)\right]\ovset{\zcref{eq:Rhat-commutes-deriv}}=\partial_{x}\hat{\mathcal{R}}^{\eps,\zeta}\left[\mathcal{K}^{\eps,\zeta}\mathcal{H}^{\eps,\zeta}(\tau)\right]\ovset{\zcref{eq:RK}}=K'\circledast\hat{\mathcal{R}}^{\eps,\zeta}\mathcal{H}^{\eps,\zeta}(\tau)\\
\ovset{\zcref{eq:RhatHisPi}} & =K'\circledast\hat{\boldsymbol{\varPi}}^{\eps,\zeta}(\tau)\overset{\zcref{eq:Pihatintegration}}{=}\hat{\boldsymbol{\varPi}}^{\eps,\zeta}(\mathcal{I}'\tau),
\end{align*}
which is \zcref{eq:RhatHisPi} for $\mathcal{I}'\tau$. Similarly,
using the second rather than the first definition in \zcref{eq:HHtildeIprime},
\zcref{eq:RKhat} in addition to \zcref{eq:RK}, and \zcref{eq:RHatHtildeisPitilde}
in place of \zcref{eq:RhatHisPi}, we have
\[
\hat{\mathcal{R}}^{\eps,\zeta}\tilde{\mathcal{H}}^{\eps,\zeta}(\mathcal{I}'\tau)=(K'+\tilde{K}')\circledast\tilde{\boldsymbol{\varPi}}^{\eps,\zeta}(\hat{M}\tau)\overset{\zcref{eq:pissumofKs}}{=}p'\circledast\tilde{\boldsymbol{\varPi}}^{\eps,\zeta}(\hat{M}\tau)\overset{\zcref{eq:Pihatintegration-1}}{=}\tilde{\boldsymbol{\varPi}}^{\eps,\zeta}(\mathcal{I}'\tau),
\]
which is \zcref{eq:RHatHtildeisPitilde} for $\mathcal{I}'\tau$.

Now suppose that $\tau_{1},\tau_{2}\in\check{\mathcal{T}}_{\ES}\cap\mathcal{T}_{\mathcal{I}'}$
and \zcref[comp=true]{eq:RhatHisPi,eq:RHatHtildeisPitilde} hold for
$\tau_{1}$ and $\tau_{2}$. We must verify that \zcref{eq:RhatHisPi,eq:RHatHtildeisPitilde}
hold for $\tau_{1}\tau_{2}$.  To evaluate the left side of \zcref{eq:RhatHisPi},
we use \zcref{lem:mult-model,lem:Lranges}. We can write
\begin{align}
\hat{\mathcal{R}}^{\eps,\zeta}\left(\mathcal{H}^{\eps,\zeta}(\tau_{1}\tau_{2})\right)_{t}(x)\ovset{\zcref{eq:HHtildeproduct}} & =\hat{\mathcal{R}}^{\eps,\zeta}\left(\check{\varpi}_{\mathrm{F}}\left[\mathcal{H}^{\eps,\zeta}(\tau_{1})\mathcal{H}^{\eps,\zeta}(\tau_{2})\right]\right)_{t}(x)\nonumber \\
\ovset{\zcref{eq:Lranges}} & =\hat{\mathcal{R}}^{\eps,\zeta}\left(\prod^{2}_{i=1}\left(\tau_{i}+f_{i}\rsone+g_{i}\rsiplollipopr\right)\right)_{t}(x)\nonumber \\
\ovset{\zcref{eq:reconstruction-is-model-1}} & =\hat{\Pi}^{\eps,\zeta}_{t,x}\left(\prod^{2}_{i=1}\left(\tau_{i}+f_{i}\rsone+g_{i}\rsiplollipopr\right)\right)_{t}(x).\label{eq:pihatproduct}
\end{align}
for some continuous functions $f_{i}$ and $g_{i}$. Now when we
expand the product on the right side of \zcref{eq:pihatproduct},
there are six terms. For the first one, we can use \zcref{lem:mult-model}
to write
\[
\hat{\Pi}^{\eps,\zeta}_{t,x}(\tau_{1}\tau_{2})_{t}(x)=\hat{\Pi}^{\eps,\zeta}_{t,x}(\tau_{1})_{t}(x)\hat{\Pi}^{\eps,\zeta}_{t,x}(\tau_{2})_{t}(x)-C_{\zeta}[\tau_{1},\tau_{2}](x)\overset{\zcref{eq:reconstruction-is-model-1}}{=}\hat{\mathcal{R}}^{\eps,\zeta}(\tau_{1})_{t}(x)\hat{\mathcal{R}}^{\eps,\zeta}(\tau_{2})_{t}(x)-C^{\tau_{1},\tau_{2}}_{\zeta}(x).
\]
We can also use \zcref{lem:mult-model} on the remaining five terms
in the product on the right side of \zcref{eq:pihatproduct}, but
for those terms there is no renormalization since
\begin{equation}
C_{\zeta}[\tau_{i},\rsone](x)=C_{\zeta}[\tau_{i},\rsiplollipopr](x)=C_{\zeta}[\rsone,\rsiplollipopr](x)=C_{\zeta}[\rsone,\rsone](x)=C_{\zeta}[\rsiplollipopr,\rsiplollipopr](x)=0\label{eq:Csarezero}
\end{equation}
as can be seen from the definition \zcref{eq:Ctau1tau2def} of $C^{\tau_{1},\tau_{2}}_{\zeta}$.
Thus, we in fact have 
\begin{align*}
\hat{\mathcal{R}}^{\eps,\zeta}\left(\mathcal{H}^{\eps,\zeta}(\tau_{1}\tau_{2})\right)_{t}(x) & =\hat{\mathcal{R}}^{\eps,\zeta}\left(\mathcal{H}^{\eps,\zeta}(\tau_{1})\right)_{t}(x)\hat{\mathcal{R}}^{\eps,\zeta}\left(\mathcal{H}^{\eps,\zeta}(\tau_{1})\right)_{t}(x)-C^{\tau_{1},\tau_{2}}_{\zeta}(x)\\
\ovset{\zcref{eq:RhatHisPi}} & =\hat{\boldsymbol{\varPi}}^{\eps,\zeta}(\tau_{1})_{t}(x)\hat{\boldsymbol{\varPi}}^{\eps,\zeta}(\tau_{2})_{t}(x)-C^{\partial\tau_{1},\partial\tau_{2}}_{\zeta}(x)\ovset{\zcref{eq:Pihatmultiply}}=\hat{\boldsymbol{\varPi}}^{\eps,\zeta}(\tau_{1}\tau_{2})_{t}(x).
\end{align*}
This completes the proof of \zcref{eq:RhatHisPi} in this final case.
\end{proof}

We will also need the following symmetry property of $\mathcal{H}^{\eps,\zeta}\tau$
and $\tilde{\mathcal{H}}^{\eps,\zeta}\tau$, which should be thought
of as a lift of \zcref{eq:sgnofTexp}.
\begin{lem}
We have, for each $\tau\in\check{\mathsf{T}}_{\ES}$ and each $\sigma\in\mathscr{S}$,
that
\begin{equation}
\sigma_{\mathrm{trans}}\cdot(\mathcal{H}^{\eps,\zeta}\tau)_{t}(\sigma_{\mathrm{trans}}(x))=(\mathcal{H}^{\eps,\zeta}\tau)_{t}(x)\quad\text{and}\quad\sigma_{\mathrm{trans}}\cdot(\tilde{\mathcal{H}}^{\eps,\zeta}\tau)_{t}(\sigma_{\mathrm{trans}}(x))=(\tilde{\mathcal{H}}^{\eps,\zeta}\tau)_{t}(x)\label{eq:Hsymmetry-trans}
\end{equation}
as well as
\begin{equation}
\sigma_{\mathrm{refl}}\cdot(\mathcal{H}^{\eps,\zeta}\tau)_{t}(\sigma_{\mathrm{refl}}(x))=\sgn(\tau)(\mathcal{H}^{\eps,\zeta}\tau)_{t}(x)\quad\text{and}\quad\sigma_{\mathrm{refl}}\cdot(\tilde{\mathcal{H}}^{\eps,\zeta}\tau)_{t}(\sigma_{\mathrm{refl}}(x))=\sgn(\tau)(\tilde{\mathcal{H}}^{\eps,\zeta}\tau)_{t}(x).\label{eq:Hsymmetry-refl}
\end{equation}
(In the case $\tau=\rsnoise$, these identities are interpreted in
the obvious distributional sense.)
\end{lem}

\begin{proof}
We proceed by induction. By \zcref{eq:HHtildebullet}, the conclusions
\zcref[range]{eq:Hsymmetry-trans,eq:Hsymmetry-refl} clearly hold
for $\tau\in\{\rspotential,\rsnoise\}$.

Now we proceed with the inductive step. First suppose that $\tau=\mathcal{I}\rho$
for $\rho\in\check{\mathsf{T}}_{\ES}\cap\check{\mathsf{T}}_{\star}$
and that \zcref[range]{eq:Hsymmetry-trans,eq:Hsymmetry-refl} hold
for $\rho$. Then we have, for $\sigma\in\mathscr{S}$, that
\[
\sigma\cdot(\mathcal{H}^{\eps,\zeta}(\mathcal{I}\rho))_{t}(\sigma(x))\overset{\zcref{eq:HHtildeIprime}}{=}\sigma\cdot(\mathcal{K}^{\eps,\rho}(\mathcal{H}^{\eps,\zeta}\rho))_{t}(\sigma(x))=(\mathcal{H}^{\eps,\zeta}\rho)_{t}(x),
\]
with the last identity by the inductive hypothesis and \cites[Prop.~5.23]{hairer:2014:theory},
as well as
\[
\sigma\cdot(\tilde{\mathcal{H}}^{\eps,\zeta}(\mathcal{I}\rho))_{t}(\sigma(x))\overset{\zcref{eq:HHtildeIprime}}{=}\sigma\cdot((\mathcal{K}^{\eps,\rho}+\tilde{\mathcal{K}}^{\eps,\rho})(\tilde{\mathcal{H}}^{\eps,\zeta}\rho))_{t}(\sigma(x))=(\tilde{\mathcal{H}}^{\eps,\zeta}\rho)_{t}(x),
\]
with the last identity by the inductive hypothesis, \cites[Prop.~5.23]{hairer:2014:theory},
and an easier-to-prove analogue of \cites[Prop.~5.23]{hairer:2014:theory}
with $\mathcal{K}^{\eps,\rho}$ replaced by $\tilde{\mathcal{K}}^{\eps,\rho}$.
Since $\sgn(\tau)=1$ in this case by \zcref{eq:sgnofTexp}, this
proves \zcref{eq:Hsymmetry-trans,eq:Hsymmetry-refl} for $\tau$.

Finally, if $\tau=\rho_{1}\rho_{2}$ for $\rho_{1},\rho_{2}\in\check{\mathsf{T}}_{\ES}\cap\check{\mathsf{T}}_{\mathcal{I}'}$
and \zcref{eq:Hsymmetry-trans,eq:Hsymmetry-refl} hold for $\rho_{1}$
and $\rho_{2}$, then we can write
\begin{multline*}
\sigma\cdot(\mathcal{H}^{\eps,\zeta}(\rho_{1}\rho_{2}))_{t}(\sigma(x))\sigma\ovset{\zcref{eq:HHtildeproduct}}=\cdot[(\mathcal{H}^{\eps,\zeta}\rho_{1})_{t}(\sigma(x))\cdot((\mathcal{H}^{\eps,\zeta}\rho_{2})_{t}(\sigma(x))]\\
\ovset{\zcref{eq:sigma-product-commute}}=\sigma\cdot(\mathcal{H}^{\eps,\zeta}\rho_{1})_{t}(\sigma(x))\sigma\cdot((\mathcal{H}^{\eps,\zeta}\rho_{2})_{t}(\sigma(x))=(\mathcal{H}^{\eps,\zeta}\rho_{1})_{t}(x)(\mathcal{H}^{\eps,\zeta}\rho_{2})_{t}(x)\ovset{\zcref{eq:HHtildeproduct}}=(\mathcal{H}^{\eps,\zeta}(\rho_{1}\rho_{2}))_{t}(x),
\end{multline*}
with the penultimate identity by the inductive hypothesis, and similarly
with $\tilde{\mathcal{H}}$ replacing $\mathcal{H}$.
\end{proof}

\begin{prop}
\label{prop:eq-boundary}For each $\tau\in\check{\mathcal{T}}_{\ES}$,
we have
\begin{equation}
(\mathcal{H}^{\eps,\zeta}\tau)_{t}(x)=(\tilde{\mathcal{H}}^{\eps,\zeta}\tau)_{t}(x)\qquad\text{for all }t\in\mathbb{R}\text{ and }x_{0}\in\{0,L\}.\label{eq:eq-bdry}
\end{equation}
\end{prop}

\begin{proof}
By linearity, it suffices to check \zcref{eq:eq-bdry} for $\tau\in\check{\mathsf{T}}_{\ES}$.
We proceed inductively. For the base case, we note that if $\tau\in\{\rsnoise,\rspotential\}$,
then \zcref{eq:eq-bdry} follows immediately from \zcref{eq:HHtildebullet}.

Now we proceed to the inductive step. If $\tau\in\check{\mathsf{T}}_{\ES}\cap\mathsf{T}_{\star}$
and \zcref{eq:eq-bdry} holds for $\tau$, then 
\begin{equation}
\sgn(\tau)=-1\label{eq:tau-odd}
\end{equation}
by \zcref{eq:sgnofTexp}, and we have by \zcref{eq:HHtildeIprime,eq:KfisIpluspoly}
that
\begin{align*}
\mathcal{H}^{\eps,\zeta}(\mathcal{I}'\tau)-\tilde{\mathcal{H}}^{\eps,\zeta}(\mathcal{I}'\tau) & =\check{\varpi}_{\mathcal{I}'}\left[\partial\tilde{\mathcal{K}}^{\eps,\zeta}\tilde{\mathcal{H}}^{\eps,\zeta}\tau+\partial\mathcal{K}^{\eps,\zeta}\left(\mathcal{H}^{\eps,\zeta}\tau-\tilde{\mathcal{H}}^{\eps,\zeta}\tau\right)\right]=\check{\varpi}_{\mathcal{I}'}\partial\mathcal{I}\left(\mathcal{H}^{\eps,\zeta}\tau-\tilde{\mathcal{H}}^{\eps,\zeta}\tau\right)+f'\rsone
\end{align*}
for some continuous function $f'$. By the inductive hypothesis, this
means that
\begin{equation}
\left(\mathcal{H}^{\eps,\zeta}\tau-\tilde{\mathcal{H}}^{\eps,\zeta}\tau\right)_{t}(x_{0})=f'_{t}(x_{0})\rsone\qquad\text{for }x_{0}\in\{0,L\}.\label{eq:multipleof1}
\end{equation}
Using \zcref{eq:Hsymmetry-refl,eq:tau-odd} and recalling the action
\zcref{eq:sigma-action} of $\sigma_{\mathrm{refl}}$ on $\mathcal{T}$,
we see that
\begin{align*}
f'_{t}(0)\rsone & =\sigma_{\mathrm{refl}}\cdot\left(\mathcal{H}^{\eps,\zeta}\tau-\tilde{\mathcal{H}}^{\eps,\zeta}\tau\right)_{t}(\sigma_{\mathrm{refl}}(0))=-\left(\mathcal{H}^{\eps,\zeta}\tau-\tilde{\mathcal{H}}^{\eps,\zeta}\tau\right)_{t}(0)=-f'_{t}(0)\rsone,
\end{align*}
which means that $f'_{t}(0)=0$. The same argument with $0$ replaced
by $L$ and $\sigma_{\mathrm{refl}}$ replaced by $\sigma_{\mathrm{trans}}\sigma_{\mathrm{refl}}$
(since $\sigma_{\mathrm{trans}}\sigma_{\mathrm{refl}}(L)=L$ by \zcref{eq:sym_group_generators})
implies that $f_{t}'(L)=0$. Using these observations in \zcref{eq:multipleof1},
we conclude that \zcref{eq:eq-bdry} holds for $\mathcal{I}'\tau$.

Finally, we see that if $\tau_{1},\tau_{2}\in\check{\mathsf{T}}_{\ES}\cap\mathsf{T}_{\mathcal{I}'}$
and \zcref{eq:eq-bdry} holds for $\tau_{1},\tau_{2}$, then for $x_{0}\in\{0,L\}$
we have
\begin{align*}
(\mathcal{H}^{\eps,\zeta}(\tau_{1}\tau_{1}))_{t}(x_{0})\overset{\zcref{eq:HHtildeproduct}}{=}(\mathcal{H}^{\eps,\zeta}\tau_{1})_{t}(x_{0})(\mathcal{H}^{\eps,\zeta}\tau_{2})_{t}(x_{0})\ovset{\zcref{eq:eq-bdry}} & =(\tilde{\mathcal{H}}^{\eps,\zeta}\tau_{1})_{t}(x_{0})(\tilde{\mathcal{H}}^{\eps,\zeta}\tau_{2})_{t}(x_{0})\\
\ovset{\zcref{eq:HHtildeproduct}} & =(\tilde{\mathcal{H}}^{\eps,\zeta}(\tau_{1}\tau_{1}))_{t}(x_{0}),
\end{align*}
and so \zcref{eq:eq-bdry} holds for $\tau_{1}\tau_{2}$ as well.
This completes the proof by induction.
\end{proof}

\section{\label{sec:Analysis-of-the}Analysis of the boundary layer: preliminaries}

Now that we have introduced the framework of regularity structures
for studying the open KPZ equation, we are ready to describe the strategy
of the proof of \zcref{prop:boundaryterm} and make some preliminary
reductions. The first reduction is purely for notational convenience:
we study the forward KPZ solution $h$ rather than the backward one
$\hat{h}$.
\begin{prop}
\label{prop:boundaryterm-1-1}Fix $\uu,\vv\in\mathbb{R}$. For any
$\mathscr{F}_{0}$-measurable random variable $U$, we have
\begin{equation}
\left((W_{t})_{t\in[0,T]},U,\mathcal{B}^{\eps}_{\uu,\vv;0,T}(\varphi^{\eps}_{\uu,\vv})\right)\xrightarrow[\eps\downarrow0]{\mathrm{law}}\left((W_{t})_{t\in[0,T]},U,\Upsilon_{\uu,\vv;0,T}\right),\label{eq:boundaryterm-1-1}
\end{equation}
where 
\[
\Upsilon_{\uu,\vv;0,T}\sim\mathcal{N}\left(-\oh(\uu^{2}+\vv^{2})V_{\psi}T-\nicefrac{1}{6}T(\uu^{3}+\vv^{3}),(\uu^{2}+\vv^{2})V_{\psi}T\right)
\]
is independent of $\mathscr{F}_{T}$, with the constant $V_{\psi}$
defined in \zcref{eq:redcherry-var-limit}. In particular, we have
\begin{equation}
\mathbb{E}\left[\exp\left\{ \Upsilon_{\uu,\vv;0,T}\right\} \right]=\exp\left\{ -\frac{1}{6}(\uu+\vv)T\right\} .\label{eq:ExpUpsilon-1-1}
\end{equation}
\end{prop}

We will prove \zcref{prop:boundaryterm-1-1} at the end of \zcref{subsec:Outline-of-the}.
First we show how it is equivalent to \zcref{prop:boundaryterm}.
\begin{proof}[Proof of \zcref{prop:boundaryterm} given \zcref{prop:boundaryterm-1-1}]
First, by \zcref{prop:symmetry-timerev,eq:varphiuvdef,eq:boundaryterm-1-1},
we have that
\[
\left((\hat{W}_{t})_{t\in[0,T]},\hat{h}_{0},\hat{\mathcal{B}}^{\eps}_{\uu,\vv;0,T}(\varphi^{\eps}_{\uu,\vv})\right)\overset{\mathrm{law}}{=}\left((W_{t})_{t\in[0,T]},h_{0},\mathcal{B}^{\eps}_{-\uu,-\vv;0,T}(\varphi^{\eps}_{\uu,\vv})\right).
\]
Now by \zcref{eq:boundaryterm-1-1} and the fact that $\mathcal{B}^{\eps}_{-\uu,-\vv;0,T}(\varphi^{\eps}_{\uu,\vv})=-\mathcal{B}^{\eps}_{-\uu,-\vv;0,T}(\varphi^{\eps}_{-\uu,-\vv})$,
we know that the third component on the right side of the above display
converges in law to $-\Upsilon_{\uu,\vv;0,T}$ and moreover in the
limit is independent of $\mathscr{F}_{T}$. This completes the proof.
\end{proof}

With the above simple reduction, the proof of \zcref{thm:mainthm}
reduces to the proof of \zcref{prop:boundaryterm-1-1}, and in particular
to the analysis of the nonlinear term $\mathcal{B}^{\eps}_{\uu,\vv;0,T}(\varphi^{\eps}_{\uu,\vv})$.
As outlined in \zcref{subsec:Our-method}, the strategy is to expand
the KPZ solution locally near the boundary to justify its approximation
by finitely many terms, using the reconstruction theorem. The proof
is divided into a few steps. We first show that the remainder in the
expansion is negligible, which reduces the problem to the analysis
of finitely many terms. We then study these terms in detail.

\subsection{\label{subsec:Reduction-to-the}Reduction to the analysis of finitely
many terms: reconstruction theorem}

We would like to use the reconstruction theorem from the theory of
regularity structures to approximate the random variable $\mathcal{B}^{\eps}_{\uu,\vv;0,T}(\varphi^{\eps}_{\uu,\vv})$.
The reconstruction theorem is concerned with approximation of a function
by its ``Taylor expansion'' given by a modeled distribution, and
so it is ideally suited for studying ``local'' quantities. However,
the quantity $\mathcal{B}^{\eps}_{\uu,\vv;0,T}(\varphi^{\eps}_{\uu,\vv})$
should be thought of as an average of the KPZ nonlinearity on $[0,T]\times\supp\varphi^{\eps}_{\uu,\vv}$,
which is localized in space but not in time. Thus, in order to use
the reconstruction theorem, we first need to approximate this quantity
by local \emph{space-time }averages of the KPZ nonlinearity.

Recalling the properties \zcref[range]{eq:psiproperties,eq:psiepsdef}
of $\psi$ and its rescaling, we define for any $s<t$ the quantity
\begin{equation}
\Psi^{\eps}_{s,t;r}=\int^{t-\eps^{2}}_{s+\eps^{2}}\psi^{\eps^{2}}(q-r)\,\dif q,\label{eq:Psidef}
\end{equation}
so in particular we have
\begin{equation}
0\le\Psi^{\eps}_{s,t;r}\le1\text{ for all }r,\qquad\Psi^{\eps}_{s,t;r}=1\text{ for }r\in[s+2\eps^{2},t-2\eps^{2}],\qquad\text{and}\qquad\Psi^{\eps}_{s,t;r}=0\text{ for }r\not\in(s,t).\label{eq:Psiproperties}
\end{equation}
For a function $g$ on $[0,T]\times\mathbb{R}$, we define
\begin{equation}
\mathcal{X}^{\eps}_{\uu,\vv;s,t}(g)\coloneqq\int_{\mathbb{R}}\Psi^{\eps}_{s,t;r}\langle\varphi^{\eps}_{\uu,\vv},g_{r}\rangle\,\dif r\overset{\zcref{eq:Psidef}}{=}\int^{t-\eps^{2}}_{s+\eps^{2}}\left(\int_{\mathbb{R}}\psi^{\eps^{2}}(r-q)\langle\varphi^{\eps}_{\uu,\vv},g_{r}\rangle\,\dif r\right)\,\dif q,\label{eq:Xepsdef}
\end{equation}
so one should view $\mathcal{X}^{\eps}_{\uu,\vv;s,t}$ as a space-time
distribution that acts linearly on $g$. The inner integral $\int_{\mathbb{R}}\psi^{\eps^{2}}(r-q)\langle\varphi^{\eps}_{\uu,\vv},g_{r}\rangle\,\dif r$
should be viewed as a local \emph{space-time} average of the function
$g$ near the points $(q,0)$ and $(q,L)$, since the inner product
$\langle\cdot,\cdot\rangle$ represents a spatial integral while the
integral over $r$ is a time integral. Thus, one should view $\mathcal{X}^{\eps}_{\uu,\vv;s,t}$
as an approximation of $\int^{t}_{s}\langle\varphi^{\eps}_{\uu,\vv},g_{r}\rangle\,\dif r$,
which we write as an integral over $q$ of local space-time averages.

We now introduce an approximation of the nonlinearity $\mathcal{B}^{\eps}_{\uu,\vv;s,t}(\varphi^{\eps}_{\uu,\vv})$.
Recall first that the function $h^{\eps,\zeta}_{\uu,\vv}$ solves
\zcref{eq:hepszeta}, with the noise also mollified in the spatial
variable and the renormalized nonlinear term given by $\oh((u^{\eps,\zeta}_{\uu,\vv})^{2}-\Sh^{\zeta}_{2L}(0)+\oh\Sh^{\zeta/2}_{L})$.
On the other hand, the nonlinearity $\mathcal{B}^{\eps}_{\uu,\vv;s,t}$
is defined through the energy solution, which differs from the Hopf--Cole
solution by a factor of $\nicefrac{1}{24}$. Taking these differences
into consideration, we define
\begin{align}
\tilde{\mathcal{B}}^{\eps,\zeta}_{\uu,\vv;s,t} & \coloneqq\mathcal{X}^{\eps}_{\uu,\vv;s,t}\left((u^{\eps,\zeta}_{\uu,\vv})^{2}-\Sh^{\zeta}_{2L}(0)+\frac{1}{2}\Sh^{\zeta/2}_{\mathrm{L}}+\frac{1}{12}\right)\nonumber \\
\ovset{\zcref{eq:Xepsdef}} & =\int_{\mathbb{R}}\Psi^{\eps}_{s,t;r}\left\langle (u^{\eps,\zeta}_{\uu,\vv;r})^{2}-\Sh^{\zeta}_{2L}(0)+\frac{1}{2}\Sh^{\zeta/2}_{\mathrm{L}}+\frac{1}{12},\varphi^{\eps}_{\uu,\vv}\right\rangle \,\dif r.\label{eq:Btildeepszetadef}
\end{align}
We recall from \zcref{eq:Psiproperties} that $\Psi^{\eps}_{s,t;r}$
is an approximation of the indicator $\mathbf{1}_{[s,t]}(r)$, so
the following lemma seems natural. The proof, however, is somewhat
technical, so we defer it to \zcref{subsec:proof-bepsnoiserel-zeta}.
\begin{lem}
\label{lem:Bepsnoiserel-zeta}For any $s<t$ and each $\eps>0$, the
limit
\begin{equation}
\tilde{\mathcal{B}}^{\eps}_{\uu,\vv;s,t}\coloneqq\lim_{\zeta\to0}\tilde{\mathcal{B}}^{\eps,\zeta}_{\uu,\vv;s,t}\qquad\text{in probability}\label{eq:Btildeepszetalimit}
\end{equation}
exists. Moreover, we have
\begin{equation}
\lim_{\eps\to0}\left|\tilde{\mathcal{B}}^{\eps}_{\uu,\vv;s,t}-\mathcal{B}^{\eps}_{\uu,\vv;s,t}(\varphi^{\eps}_{\uu,\vv})\right|=0\qquad\text{in probability.}\label{eq:Btildegoodapprox}
\end{equation}
\end{lem}

With the above lemma in hand, the analysis of the nonlinearity $\mathcal{B}^{\eps}_{\uu,\vv;0,T}(\varphi^{\eps}_{\uu,\vv})$
as $\eps\to0$ reduces to the study of $\tilde{\mathcal{B}}^{\eps}_{\uu,\vv;0,T}$
in the same limit. Using \zcref{eq:Btildeepszetadef,eq:Btildeepszetalimit},
the latter can be approximated by an integral of a local space-time
average of the function $(u^{\eps,\zeta}_{\uu,\vv})^{2}-\Sh^{\zeta}_{2L}(0)+\oh\Sh^{\zeta}_{L}+\nicefrac{1}{12}$.
This allows us to apply the reconstruction theorem.

Let $H^{\eps,\zeta}$ be the solution to the lifted KPZ equation
\zcref{eq:fixedpointeqn} on $[0,T]$. We will soon apply the reconstruction
theorem to show that the remainder in the local expansion is small.
But first, we show that the model with basepoint on the boundary can
be written completely explicitly, and in fact does not depend on the
basepoint as long as it is on the boundary. This is the key proposition
where we use the symmetry of the problem at hand.
\begin{prop}
\label{prop:expand-nonlinearity-boundary}Fix $\eps,\zeta>0$. For
any $q,r\in\mathbb{R}$, any $x_{0}\in L\mathbb{Z}$, and any $x\in\mathbb{R}$,
we have
\begin{equation}
\begin{aligned}\hat{\Pi}^{\eps,\zeta}_{(q,x_{0})}\left(((\partial H)^{2})_{q}(x_{0})\right)_{r}(x)=Y^{\eps,\zeta}_{r}(x) & \coloneqq\rsrenorm[rsK]^{\eps,\zeta}_{r}(x)+2\rscherryrb[rsK]^{\eps,\zeta}_{r}(x)+\rscherryrenormr[rsK]^{\eps,\zeta}_{r}(x)+2\rselkrbr[rsK]^{\eps,\zeta}_{r}(x)\\
 & \qquad+\rselkrenormrr[rsK]^{\eps,\zeta}_{r}(x)+\rscherrybb[rsK]^{\eps,\zeta}_{r}(x)+\rscherryrenormb[rsK]^{\eps,\zeta}_{r}(x)+\frac{1}{4}\rscherryrenormrenorm[rsK]^{\eps,\zeta}_{r}(x).
\end{aligned}
\label{eq:expand-nonlinearity-boundary}
\end{equation}
\end{prop}

\begin{proof}
From \zcref{eq:nonlinearity-expansion,eq:h'0is0} we can write, for
any $q\in(s+\eps^{2},t-\eps^{2})$, any $r\in(s,t)$, and any $x\in\mathbb{R}$,
that
\begin{align}
\hat{\Pi}^{\eps,\zeta}_{(q,x_{0})} & \left(((\partial H)^{2})_{q}(x_{0})\right)_{r}(x)=\hat{\Pi}^{\eps,\zeta}_{(q,x_{0})}\left(\rscherryrr+2\rscherryrb+\rselkrrr+2\rselkrbr+\rsmooserrrr+\rscherrybb+\rselkrrb+\frac{1}{4}\rscandelabrarrrr\right)_{r}(x)\nonumber \\
 & =\rsrenorm[rsK]^{\eps,\zeta}_{r}(x)+2\rscherryrb[rsK]^{\eps,\zeta}_{r}(x)+\rscherryrenormr[rsK]^{\eps,\zeta}_{r}(x)+2\rselkrbr[rsK]^{\eps,\zeta}_{r}(x)-2\rsipcherryrb[rsK]^{\eps,\zeta}_{q}(x_{0})\rslollipopr[rsK]^{\eps,\zeta}_{r}(x)\nonumber \\
 & \qquad+\rselkrenormrr[rsK]^{\eps,\zeta}_{r}(x)-\rsipcherryrenormr[rsK]^{\eps,\zeta}_{q}(x_{0})+\frac{1}{4}C^{(2)}_{\zeta}+\rscherrybb[rsK]^{\eps,\zeta}_{r}(x)+\rscherryrenormb[rsK]^{\eps,\zeta}_{r}(x)+\frac{1}{4}\rscherryrenormrenorm[rsK]^{\eps,\zeta}_{r}(x)-\frac{1}{4}C^{(2)}_{\zeta},\label{eq:expansion-model}
\end{align}
where for the second ``$=$'' we used \zcref{tab:model-defns}.
Now we observe that $\rsipcherryrb[rsK]^{\eps,\zeta}_{q}$ and $\rsipcherryrenormr[rsK]^{\eps,\zeta}_{q}$
are odd, $2L$-periodic functions by \zcref{lem:reconstruct-symmetry},
so in fact they are $0$ when evaluated at $x_{0}\in L\mathbb{Z}$,
as they are in \zcref{eq:expansion-model}. Thus these terms disappear
from the above sum and we obtain \zcref{eq:expand-nonlinearity-boundary}.
\end{proof}

The $Y^{\eps,\zeta}$ defined in \zcref{eq:expand-nonlinearity-boundary}
is the local expansion of the KPZ nonlinearity at the boundary. One
should think of it as the ``useful'' terms from the formal expansion
\zcref{e.formalex} discussed in \zcref{subsec:Our-method}. The main
reason we rely on the theory of regularity structures is the following
key proposition, which is a completely deterministic estimate and
shows that, if the model is bounded uniformly in $\eps,\zeta$, then
the remainder in the local expansion is small as $\eps\to0$. As a
result, the analysis of $\tilde{\mathcal{B}}^{\eps}_{\uu,\vv;s,t}$
is reduced to the study of the finitely many terms on the right side
of \zcref{eq:expand-nonlinearity-boundary}. Recall that we have
fixed parameters $\kappa$ (which is very small) and $\chi$ (which
is just less than $\oh$) in \zcref{eq:kappaprop,eq:chirange}, respectively. 
\begin{prop}
\label{prop:apply-reconstruction-boundary}For any $J<\infty$, we
have a constant $C_{J}<\infty$ such that the following holds. Let
$\eps,\zeta>0$ and let $H^{\eps,\zeta}$ be solve \zcref{eq:fixedpointeqn}
on $[0,T]$. Then if $\|(\hat{\Pi}^{\eps,\zeta},\hat{\Gamma}^{\eps,\zeta})\|_{T}\le J$,
then 
\begin{equation}
\left|\tilde{\mathcal{B}}^{\eps,\zeta}_{\uu,\vv;0,T}-\mathcal{X}^{\eps}_{\uu,\vv;0,T}(Y^{\eps,\zeta}+\nicefrac{1}{12})\right|\le C_{J}\eps^{5\kappa}.\label{eq:apply-boundary-reconstruction-goal}
\end{equation}
\end{prop}

\begin{proof}
Recalling \zcref{eq:expand-nonlinearity-boundary}, we see that
\begin{equation}
\tilde{\mathcal{B}}^{\eps,\zeta}_{\uu,\vv;0,T}-\mathcal{X}^{\eps}_{\uu,\vv;0,T}(Y^{\eps,\zeta}+\nicefrac{1}{12})=\mathcal{X}^{\eps}_{\uu,\vv;0,T}\left((u^{\eps,\zeta}_{\uu,\vv})^{2}-\Sh^{\zeta}_{2L}(0)+\frac{1}{2}\Sh^{\zeta/2}_{L}-Y^{\eps,\zeta}\right).\label{eq:cancelY}
\end{equation}
On the other hand, we have by \zcref{eq:nonlinearityisRnonlinearity,eq:Echerryrr}
that 
\begin{equation}
\left(u^{\eps,\zeta}_{\uu,\vv;q}(x)\right)^{2}-\Sh^{\zeta}_{2L}(0)+\oh\Sh^{\zeta/2}_{L}(x)=\hat{\mathcal{R}}^{\eps,\zeta}((\partial H)^{2})_{q}(x).\label{eq:beginning-part}
\end{equation}
Using \zcref{eq:beginning-part} in \zcref{eq:cancelY} and then expanding
the definition \zcref{eq:Xepsdef} of $\mathcal{X}^{\eps}_{\uu,\vv;0,T}$,
we get
\begin{multline*}
\tilde{\mathcal{B}}^{\eps,\zeta}_{\uu,\vv;0,T}-\mathcal{X}^{\eps}_{\uu,\vv;0,T}(Y^{\eps,\zeta}+\nicefrac{1}{12})=\int^{T-\eps^{2}}_{\eps^{2}}\left(\int_{\mathbb{R}}\psi^{\eps^{2}}(r-q)\left\langle \varphi^{\eps}_{\uu,\vv},\hat{\mathcal{R}}^{\eps,\zeta}((\partial H)^{2})_{r}-Y^{\eps,\zeta}_{r}\right\rangle \,\dif r\right)\,\dif q\\
\ovset{\zcref{eq:varphiuvdef}}=\frac{1}{2}\sum_{x_{0}\in\{0,L\}}\int^{T-\eps^{2}}_{\eps^{2}}\left(\iint_{\mathbb{R}^{2}}\psi^{\eps^{2}}(r-q)\psi^{\eps}(x-x_{0})\left(\hat{\mathcal{R}}^{\eps,\zeta}((\partial H)^{2})_{r}(x)-Y^{\eps,\zeta}_{r}(x)\right)\,\dif r\,\dif x\right)\,\dif q,
\end{multline*}
where the factor of $1/2$ is because $\langle\cdot,\cdot\rangle$
only integrates over half of the support of $\psi^{\eps}(\cdot-x_{0})$
for $x_{0}\in\{0,L\}$. Then we use \zcref{eq:expand-nonlinearity-boundary}
to rewrite this as
\begin{align*}
 & \tilde{\mathcal{B}}^{\eps,\zeta}_{\uu,\vv;0,T}-\mathcal{X}^{\eps}_{\uu,\vv;0,T}(Y^{\eps,\zeta}+\nicefrac{1}{12})\\
 & \ =\frac{1}{2}\sum_{x_{0}\in\{0,L\}}\int^{T-\eps^{2}}_{\eps^{2}}\left(\iint_{\mathbb{R}^{2}}\psi^{\eps^{2}}(r-q)\psi^{\eps}(x-x_{0})\left(\hat{\mathcal{R}}^{\eps,\zeta}((\partial H)^{2})_{r}(x)-\hat{\Pi}^{\eps,\zeta}_{(q,x_{0})}\left(((\partial H)^{2})_{q}(x_{0})\right)_{r}(x)\right)\,\dif r\,\dif x\right)\,\dif q.
\end{align*}
We note that when we apply \zcref{eq:expand-nonlinearity-boundary},
the choice of the basepoint $(q,x_{0})$ is arbitrary, but we choose
it to match the basepoint of the space-time mollifier $\psi^{\eps^{2}}(\cdot-q)\psi^{\eps}(\cdot-x_{0})$
so that we get a good bound when we apply the reconstruction theorem
below.

Then by \cites[Prop.~7.2]{hairer:2014:theory}, a local version of
the reconstruction theorem, we have
\begin{equation}
\begin{aligned} & \left|\iint_{\mathbb{R}^{2}}\psi^{\eps^{2}}(r-q)\psi^{\eps}(x-x_{0})\left(\hat{\mathcal{R}}^{\eps,\zeta}((\partial H)^{2})_{q}(x)-\hat{\Pi}^{\eps,\zeta}_{(q,x_{0})}\left(((\partial H)^{2})_{q}(x_{0})\right)_{r}(x)\right)\,\dif x\,\dif r\right|\\
 & \qquad\le C\|(\partial H)^{2}\|_{\mathcal{D}^{\eps,\zeta}_{\mathrm{F}};T}(1+\|(\hat{\Pi}^{\eps,\zeta},\hat{\Gamma}^{\eps,\zeta})\|^{2})\eps^{9\kappa}q^{-1-2\kappa}.
\end{aligned}
\label{eq:apply-reconstruction}
\end{equation}
Here the exponents $9\kappa$ and $-1-2\kappa$ come from the exponents
in the definition \zcref{eq:DF} of $\mathcal{D}^{\eps,\zeta}_{\mathrm{F}}$.
Integrating \zcref{eq:apply-reconstruction} in time using the triangle
inequality, we get 
\begin{align*}
 & \left|\int^{T-\eps^{2}}_{\eps^{2}}\left(\iint_{\mathbb{R}^{2}}\psi^{\eps^{2}}(r-q)\psi^{\eps}(x-x_{0})\left(\hat{\mathcal{R}}^{\eps,\zeta}((\partial H)^{2})_{q}(x)-\hat{\Pi}^{\eps,\zeta}_{(q,x_{0})}\left(((\partial H)^{2})_{q}(x_{0})\right)_{r}(x)\right)\,\dif x\,\dif r\right)\,\dif q\right|\\
 & \qquad\le C\|(\partial H)^{2}\|_{\mathcal{D}^{\eps,\zeta}_{\mathrm{F}};T}(1+\|(\hat{\Pi}^{\eps,\zeta},\hat{\Gamma}^{\eps,\zeta})\|)\eps^{5\kappa}
\end{align*}
for a new constant $C$. The claimed result follows as $\|(\partial H)^{2}\|_{\mathcal{D}^{\eps,\zeta}_{\mathrm{F}};T}$
can be bounded in terms of $\|(\hat{\Pi}^{\eps,\zeta},\hat{\Gamma}^{\eps,\zeta})\|_{T}$
by \zcref{eq:squaring-cts} and the continuity statement in \zcref{cor:Tisinfinity}.
\end{proof}

\subsection{\label{subsec:Outline-of-the}Outline of the analysis of the explicit
terms}

With \zcref{prop:expand-nonlinearity-boundary,prop:apply-reconstruction-boundary},
we have reduced the problem of studying $\mathcal{B}^{\eps}_{\uu,\vv;s,t}$
to the study of $Y^{\eps,\zeta}$. The terms in the definition \zcref{eq:expand-nonlinearity-boundary}
of $Y^{\eps,\zeta}$ are completely explicit and indeed all live in
the first four Wiener chaoses, so in principle their analysis should
be straightforward. Despite this, since we require quite precise calculations
at the boundary, the analysis involved is quite lengthy. The computations
of explicit terms are carried out in \zcref{sec:Explicit-calculations},
and estimates on certain variances via the BPHZ theory (which allows
to ``automate'' some of the more repetitive bounds, and is also
used to prove \zcref{thm:model-norm-bounded}) are done in \zcref{sec:BPHZ}.
In this section, we summarize these results and show how they fit
together to prove \zcref{prop:boundaryterm-1-1}.

There is one remaining technical issue to be handled before we can
use the estimates in \zcref{sec:Explicit-calculations,sec:BPHZ}.
We will derive the results in \zcref{sec:Explicit-calculations} with
calculations done using the true heat kernel $p$ (see \zcref{eq:ptdef}),
but we obtain the estimates in \zcref{sec:BPHZ} are performed using
the truncated heat kernel $K$ (see \zcref[range]{eq:pissumofKs,eq:KKhateven}).
The use of the true heat kernel in \zcref{sec:Explicit-calculations}
is more convenient for us, because the exact formula for the heat
kernel and its Fourier transform allow us to perform exact calculations.
On the other hand, to use the results from these two different sections
together, we need to compare the stochastic objects that result from
the two distinct kernels. The following proposition will suffice for
our needs and will be proved in \zcref{subsec:prove-explicit-comparison}
below.
\begin{prop}
\label{prop:explicit-comparison}For each $\tau\in\left\{ \rsrenorm,\rscherryrb,\rscherryrenormr,\rselkrbr,\rselkrenormrr,\rscherrybb,\rscherryrenormb,\rscherryrenormrenorm\right\} $(i.e.
for each of the trees appearing on the right side of \zcref{eq:expand-nonlinearity-boundary})
and any $p<\infty$, we have
\begin{equation}
\adjustlimits\lim_{\eps\downarrow0}\sup_{\zeta\in(0,\nicefrac{\eps}{100})}\mathbb{E}\left[\left|\mathcal{X}^{\eps}_{\uu,\vv;0,T}\left(\bboxop{\tau}^{\eps,\zeta}\right)-\mathcal{X}^{\eps}_{\uu,\vv;0,T}\left(\boxop{\tau}^{\eps,\zeta}\right)\right|^{p}\right]=0.\label{eq:boundary-moments-same}
\end{equation}
\end{prop}

Now we can state the following proposition, which is about the joint
convergence of the nonlinearity and the noise, and relies on \zcref{prop:explicit-comparison}
and the results of \zcref{sec:Explicit-calculations,sec:BPHZ}.
\begin{table}
\hfill{}%
\begin{tabular}{cll}
\toprule 
$g^{\eps,\zeta}$ & $\lim\limits_{\eps\downarrow0}\lim\limits_{\zeta\downarrow0}\Law\left(\mathcal{X}^{\eps}_{\uu,\vv;0,T}\left(g^{\eps,\zeta}\right)\right)$ & Key estimates\tabularnewline
\midrule
$\rsrenorm[rsK]^{\zeta}$ & $\mathcal{N}\left(0,T(\uu^{2}+\vv^{2})V_{\psi}\right)$ & \zcref{prop:Xredcherrylimit}\tabularnewline
$2\rscherryrb[rsK]^{\eps,\zeta}$ & deterministic $0$ & symmetry, \zcref{prop:lollibd}, and \zcref{prop:var-boundary}\tabularnewline
$\rscherryrenormr[rsK]^{\zeta}$ & deterministic $0$ & symmetry and \zcref{prop:elkrrr,prop:var-boundary}\tabularnewline
$2\rselkrbr[rsK]^{\eps,\zeta}$ & deterministic $-\frac{T}{2}(\uu^{2}+\vv^{2})V_{\psi}$ & \zcref{prop:red-blue-red-cherry,prop:rbrelk,prop:var-boundary}\tabularnewline
$\rscherrybb[rsK]^{\eps}$ & deterministic $-\frac{T}{6}(\uu^{3}+\vv^{3})$ & \zcref{prop:bluecherrycontrib}\tabularnewline
$\rscherryrenormb[rsK]^{\eps,\zeta}$ & deterministic $0$ & \zcref{eq:subtract-exp,prop:rrbelk,prop:var-boundary}\tabularnewline
$\frac{1}{4}\rscherryrenormrenorm[rsK]^{\zeta}+\rselkrenormrr[rsK]^{\zeta}$ & deterministic $\frac{T}{24}(\uu+\vv)$ & \zcref{prop:boundary-value,prop:candelabra,prop:moose,prop:var-boundary}\tabularnewline
\bottomrule
\end{tabular}\hfill{}

\caption[Summary of calculations of explicit terms]{\label{tab:the-terms-1}The limiting laws of the terms appearing
on the right side of \zcref{eq:expand-nonlinearity-boundary}. The
quantity $V_{\psi}$ is defined in \zcref{eq:redcherry-var-limit},
but we will not need to know anything about it besides that it is
finite.}
\end{table}

\begin{prop}
\label{prop:convergence-explicit}For any $T\in\mathbb{R}$ and $\mathscr{F}_{T}$-measurable
random variable $Q$, we have
\begin{equation}
\begin{aligned}\lim_{\eps\downarrow0}\lim_{\zeta\downarrow0} & \left(\mathcal{X}^{\eps}_{\uu,\vv;0,T}(Y^{\eps,\zeta}+\nicefrac{1}{12}),Q\right)=(\Upsilon_{\uu,\vv;0,T},Q)\end{aligned}
\label{eq:convergence-explicit}
\end{equation}
in distribution, where $\Upsilon_{\uu,\vv;0,T}$ is as in the statement
of \zcref{prop:boundaryterm-1-1} and is independent of $Q$.
\end{prop}

\begin{proof}
First, by \zcref{prop:var-boundary}, we derive that for each $\tau\in\left\{ \rscherryrb,\rscherryrenormr,\rselkrbr,\rselkrenormrr,\rscherrybb,\rscherryrenormb,\rscherryrenormrenorm\right\} $
(i.e. each of the trees appearing on the right side of \zcref{eq:expand-nonlinearity-boundary}
except for $\rsrenorm$), we have 
\begin{equation}
\adjustlimits\lim_{\eps\downarrow0}\sup_{\zeta\in(0,\eps)}\Var\left(\mathcal{X}^{\eps}_{\uu,\vv;0,T}\left(\boxop{\tau}^{\eps,\zeta}\right)\right)=0.\label{eq:variances0}
\end{equation}
To compute the expectations of these terms, we first note that 
\begin{equation}
\mathbb{E}\left[\rscherryrb[rsK]^{\eps,\zeta}_{t}(x)\right]=\mathbb{E}\left[\rscherryrenormr[rsK]^{\zeta}_{t}(x)\right]=0\qquad\text{for all }t,x\in\mathbb{R}\label{eq:1st}
\end{equation}
by symmetry, since each of these terms contains an odd number of copies
of the noise. We also have that
\begin{equation}
\mathbb{E}\left[\rsrenorm[rsP]^{\zeta}_{t}(x)\right]=\mathbb{E}\left[\rscherryrenormb[rsP]^{\eps,\zeta}_{t}(x)\right]=0\qquad\text{for all }t,x\in\mathbb{R}\label{eq:2nd}
\end{equation}
by \zcref{eq:subtract-exp}, which implies by \zcref{prop:explicit-comparison}
that 
\begin{equation}
\adjustlimits\lim_{\eps\downarrow0}\lim_{\zeta\downarrow0}\mathbb{E}\left[\mathcal{X}^{\eps}_{\uu,\vv;0,T}\left(\boxop{\tau}^{\eps,\zeta}\right)\right]=0\qquad\text{for }\tau\in\left\{ \rsrenorm,\rscherryrenormb\right\} .\label{eq:3rd}
\end{equation}
The expectations of the rest of the terms are computed in \zcref{sec:Explicit-calculations}.
Indeed, by combining \zcref{prop:red-blue-red-cherry,prop:explicit-comparison},
we see that
\begin{equation}
\adjustlimits\lim_{\eps\downarrow0}\lim_{\zeta\downarrow0}\mathbb{E}\left[\mathcal{X}^{\eps}_{\uu,\vv;0,T}\left(\rselkrbr[rsK]^{\eps,\zeta}\right)\right]=-\frac{T}{4}(\uu^{2}+\vv^{2})V_{\psi}.\label{eq:4th}
\end{equation}
Similarly, by combining \zcref{prop:bluecherrycontrib,prop:explicit-comparison},
we obtain 
\begin{equation}
\lim_{\eps\downarrow0}\mathbb{E}\left[\mathcal{X}^{\eps}_{\uu,\vv;0,T}\left(\rscherrybb[rsK]^{\eps}\right)\right]=-\frac{T}{6}(\uu^{3}+\vv^{3}).\label{eq:5th}
\end{equation}
Finally, using \zcref{prop:boundary-value,prop:explicit-comparison},
we see that 
\begin{equation}
\adjustlimits\lim_{\eps\downarrow0}\lim_{\zeta\downarrow0}\mathbb{E}\left[\mathcal{X}^{\eps}_{\uu,\vv;0,T}\left(\rselkrenormrr[rsK]^{\zeta}+\frac{1}{4}\rscherryrenormrenorm[rsK]^{\zeta}\right)\right]=\frac{T}{24}(\uu+\vv).\label{eq:6th}
\end{equation}
Combining \zcref[comp=true]{eq:variances0,eq:1st,eq:2nd,eq:3rd,eq:4th,eq:5th,eq:6th}
and defining
\[
\tilde{Y}^{\eps,\zeta}_{r}(x)\coloneqq2\rscherryrb[rsK]^{\eps,\zeta}_{r}(x)+\rscherryrenormr[rsK]^{\zeta}_{r}(x)+2\rselkrbr[rsK]^{\zeta}_{r}(x)+\rselkrenormrr[rsK]^{\zeta}_{r}(x)+\rscherrybb[rsK]^{\eps}_{r}(x)+\rscherryrenormb[rsK]^{\eps,\zeta}_{r}(x)+\frac{1}{4}\rscherryrenormrenorm[rsK]^{\zeta}_{r}(x)
\]
so that
\begin{equation}
Y^{\eps,\zeta}_{r}(x)=\rsrenorm[rsK]^{\zeta}_{r}(x)+\tilde{Y}^{\eps,\zeta}_{r}(x),\label{eq:Ydecomp}
\end{equation}
we have
\begin{equation}
\lim_{\eps\downarrow0}\lim_{\zeta\downarrow0}\mathcal{X}^{\eps}_{\uu,\vv;0,T}(\tilde{Y}^{\eps,\zeta}_{r})=-\frac{T}{2}(\uu^{2}+\vv^{2})V_{\psi}-\frac{T}{6}(\uu^{3}+\vv^{3})+\frac{T}{24}(\uu+\vv)\qquad\text{in probability.}\label{eq:XYtildeconverges}
\end{equation}
On the other hand, combining \zcref{prop:Xredcherrylimit,prop:explicit-comparison},
we have
\begin{equation}
\lim_{\eps\downarrow0}\lim_{\zeta\downarrow0}\Law\left(\mathcal{X}^{\eps}_{\uu,\vv;0,T}\left(\rsrenorm[rsK]^{\eps}\right),Q\right)=\mathcal{N}(0,T(\uu^{2}+\vv^{2})V_{\psi})\otimes\Law(Q)\qquad\text{weakly},\label{eq:limitindistribution-apply}
\end{equation}
where $Q$ is an arbitrary $\mathscr{F}_{T}$-measurable random variable.
Combining \zcref[comp=true]{eq:Ydecomp,eq:XYtildeconverges,eq:limitindistribution-apply},
we obtain \zcref{eq:convergence-explicit}.
\end{proof}

We have finally assembled all the ingredients necessary to complete
the proof of \zcref{prop:boundaryterm-1-1}, and hence the proof of
\zcref{thm:mainthm}.
\begin{proof}[Proof of \zcref{prop:boundaryterm-1-1}]
By the portmanteau lemma, it is sufficient to show that if $Q$ is
$\mathscr{F}_{T}$-measurable, then 
\[
(\mathcal{B}^{\eps}_{\uu,\vv;0,T}(\varphi^{\eps}_{\uu,\vv}),Q)\xrightarrow[\eps\downarrow0]{\mathrm{law}}(\Upsilon_{\uu,\vv;0,T},Q)\qquad\text{as }\eps\downarrow0,
\]
where $\Upsilon_{\uu,\vv;0,T}$ is independent of $Q$. This result
is obtained by combining \zcref{eq:Btildegoodapprox} of \zcref{lem:Bepsnoiserel-zeta}
to approximate $\mathcal{B}^{\eps}_{\uu,\vv;0,T}(\varphi^{\eps}_{\uu,\vv})$
by $\tilde{\mathcal{B}}^{\eps}_{\uu,\vv;s,t}$ in probability, \zcref{prop:apply-reconstruction-boundary,thm:model-norm-bounded}
to approximate $\tilde{\mathcal{B}}^{\eps}_{\uu,\vv;s,t}$ by $\mathcal{X}^{\eps}_{\uu,\vv;0,T}(Y^{\eps,\zeta}+\nicefrac{1}{12})$
in probability, and finally \zcref{prop:convergence-explicit} to
approximate $(\mathcal{X}^{\eps}_{\uu,\vv;0,T}(Y^{\eps,\zeta}+\nicefrac{1}{12}),Q)$
by $(\Upsilon_{\uu,\vv;0,T},Q)$ in distribution.
\end{proof}

\subsection{Endpoints of the time interval: proof of \texorpdfstring{\zcref{lem:Bepsnoiserel-zeta}}{Lemma~\ref{lem:Bepsnoiserel-zeta}}
\label{subsec:proof-bepsnoiserel-zeta}}

In this section, we prove \zcref{lem:Bepsnoiserel-zeta}. Recall that
we introduced the function $\Psi^{\eps}_{s,t;r}$ as an approximation
of the indicator $\mathbf{1}_{[s,t]}(r)$, with the only discrepancy
occurring in an $\eps^{2}$ neighborhood of the endpoints $s$ and
$t$. The object of interest, $\tilde{\mathcal{B}}^{\eps,\zeta}_{\uu,\vv;s,t}$,
is expressed as an integral involving the renormalized nonlinearity
of the KPZ equation against $\Psi^{\eps}_{s,t;r}$. Our goal is to
justify that, as $\eps,\zeta\to0$, we can replace $\Psi^{\eps}_{s,t;r}$
by $\mathbf{1}_{[s,t]}(r)$ in this expression, with a small error.

Given \zcref{eq:Psiproperties}, this seems quite reasonable, but
a direct proof runs into technical issues because the results of \cite{hairer:2014:theory}
only give control over $(u^{\eps,\zeta}_{\uu,\vv})^{2}$ when integrated
against a smooth test function, but difference between $\Psi^{\eps}_{s,t;\cdot}$
and the indicator function is smooth. To avoid this, we use the mollified
KPZ equation \zcref{eq:hepszeta-eqn} to write $(u^{\eps,\zeta}_{\uu,\vv})^{2}$
as a sum of terms whose distributional regularity can be understood
classically. The cost of doing this is that we lose some regularity
on $(u^{\eps,\zeta}_{\uu,\vv})^{2}$, since $\partial_{t}h^{\eps,\zeta}_{\uu,\vv}$
and $\Delta h^{\eps,\zeta}_{\uu,\vv}$ have parabolic regularity $-\thrh-$
compared to the regularity $-1-$ of $(u^{\eps,\zeta}_{\uu,\vv})^{2}$.
It turns out that, since we are trying to prove a very modest statement
anyway, we can tolerate this loss.
\begin{proof}[Proof of \zcref{lem:Bepsnoiserel-zeta}.]
Throughout the proof, since $s<t$ are fixed, we abbreviate $\Psi^{\eps}_{r}=\Psi^{\eps}_{s,t;r}$.
First, using \zcref{eq:hepszeta-eqn} in \zcref{eq:Btildeepszetadef}
and noting that $\Psi^{\eps}_{s}=\Psi^{\eps}_{t}=0$ to integrate
by parts without boundary terms, we write 
\begin{align*}
\tilde{\mathcal{B}}^{\eps,\zeta}_{\uu,\vv;s,t} & =2\int^{t}_{s}\Psi^{\eps}_{r}\left\langle \dif h^{\eps,\zeta}_{\uu,\vv;r}-\dif W^{\zeta}_{r}+\left(\nicefrac{1}{24}-\varphi^{\eps}_{\uu,\vv}\right)\dif r,\varphi^{\eps}_{\uu,\vv}\right\rangle -\int^{t}_{s}\left\langle h^{\eps,\zeta}_{\uu,\vv;r},\Delta\varphi^{\eps}_{\uu,\vv}\right\rangle \Psi^{\eps}_{r}\,\dif r\\
 & =-2\int^{t}_{s}\left\langle h^{\eps,\zeta}_{\uu,\vv;r}-W^{\zeta}_{r},\varphi^{\eps}_{\uu,\vv}\right\rangle \,\dif\Psi^{\eps}_{r}+\int^{t}_{s}\left(\left\langle \nicefrac{1}{12}-2\varphi^{\eps}_{\uu,\vv},\varphi^{\eps}_{\uu,\vv}\right\rangle -\left\langle h^{\eps,\zeta}_{\uu,\vv;r},\Delta\varphi^{\eps}_{\uu,\vv}\right\rangle \right)\Psi^{\eps}_{r}\,\dif r.
\end{align*}
For fixed $\eps>0$, taking $\zeta\to0$ and using \zcref{lem:convofzeta},
we see that \zcref{eq:Btildeepszetalimit} indeed holds, and indeed
the limit takes the form
\begin{multline}
\tilde{\mathcal{B}}^{\eps}_{\uu,\vv;s,t}=-2\int^{t}_{s}\left\langle h^{\eps}_{\uu,\vv;r}-W_{r},\varphi^{\eps}_{\uu,\vv}\right\rangle \,\dif\Psi^{\eps}_{r}+\int^{t}_{s}\left(\left\langle \nicefrac{1}{12}-2\varphi^{\eps}_{\uu,\vv},\varphi^{\eps}_{\uu,\vv}\right\rangle -\left\langle h^{\eps}_{\uu,\vv;r},\Delta\varphi^{\eps}_{\uu,\vv}\right\rangle \right)\Psi^{\eps}_{r}\,\dif r\\
=-2\left(\int^{s+2\eps^{2}}_{s}+\int^{t}_{t-2\eps^{2}}\right)\left\langle h^{\eps}_{\uu,\vv;r}-W_{r},\varphi^{\eps}_{\uu,\vv}\right\rangle \,\dif\Psi^{\eps}_{r}+\int^{t}_{s}\left(\left\langle \nicefrac{1}{12}-2\varphi^{\eps}_{\uu,\vv},\varphi^{\eps}_{\uu,\vv}\right\rangle -\left\langle h^{\eps}_{\uu,\vv;r},\Delta\varphi^{\eps}_{\uu,\vv}\right\rangle \right)\Psi^{\eps}_{r}\,\dif r.\label{eq:Btildedef-1}
\end{multline}

On the other hand, for the time-integrated nonlinearity, we observe
from \zcref{eq:Bepsnoiserel,lem:convofzeta} that
\begin{align}
\mathcal{B}^{\eps}_{\uu,\vv;s,t}(\varphi^{\eps}_{\uu,\vv}) & =\left\langle 2(h^{\eps}_{\uu,\vv;t}-h^{\eps}_{\uu,\vv;s})-2(W_{t}-W_{s})+(t-s)(\nicefrac{1}{12}-2\varphi^{\eps}_{\uu,\vv}),\varphi^{\eps}_{\uu,\vv}\right\rangle -\int^{t}_{s}\left\langle h^{\eps}_{\uu,\vv;r},\Delta\varphi^{\eps}_{\uu,\vv}\right\rangle \,\dif r\nonumber \\
 & =-2\int^{s+\eps^{2}}_{s}\langle h^{\eps}_{\uu,\vv;s}-2W_{s},\varphi^{\eps}_{\uu,\vv}\rangle\,\dif\Psi^{\eps}_{r}-2\int^{t}_{t-\eps^{2}}\langle h^{\eps}_{\uu,\vv;t}-2W_{t},\varphi^{\eps}_{\uu,\vv}\rangle\,\dif\Psi^{\eps}_{r}\nonumber \\
 & \qquad+\int^{t}_{s}\langle\nicefrac{1}{12}-2\varphi^{\eps}_{\uu,\vv},\varphi^{\eps}_{\uu,\vv}\rangle\,\dif r-\int^{t}_{s}\left\langle h^{\eps}_{\uu,\vv;r},\Delta\varphi^{\eps}_{\uu,\vv}\right\rangle \,\dif r,\label{eq:Bepsnoiserel-apply}
\end{align}
where we used the fact that $\int^{s+\eps^{2}}_{s}\dif\Psi^{\eps}_{r}=\Psi^{\eps}_{s+\eps^{2}}-\Psi^{\eps}_{s}=1$
and $\int^{t}_{t-\eps^{2}}\dif\Psi^{\eps}_{r}=\Psi^{\eps}_{t}-\Psi^{\eps}_{t-\eps^{2}}=-1$.
Subtracting \zcref{eq:Bepsnoiserel-apply} from \zcref{eq:Btildedef-1}
and again using the fact that $\Psi^{\eps}_{r}=1$ for $r\in(s+2\eps^{2},t-2\eps^{2})$,
we get
\begin{align}
\tilde{\mathcal{B}}^{\eps}_{\uu,\vv;s,t}-\mathcal{B}^{\eps}_{\uu,\vv;s,t}(\varphi^{\eps}_{\uu,\vv}) & =-2\int^{s+2\eps^{2}}_{s}\left\langle h^{\eps}_{\uu,\vv;r}-h^{\eps}_{\uu,\vv;s}-(W_{r}-W_{s}),\varphi^{\eps}_{\uu,\vv}\right\rangle \,\dif\Psi^{\eps}_{r}\nonumber \\
 & \qquad+2\int^{t}_{t-2\eps^{2}}\left\langle h_{\uu,\vv;t}-h^{\eps}_{\uu,\vv;r}-(W_{t}-W_{r}),\varphi^{\eps}_{\uu,\vv}\right\rangle \,\dif\Psi^{\eps}_{r}\nonumber \\
 & \qquad+\left(\int^{s+2\eps^{2}}_{s}+\int^{t}_{t-\eps^{2}}\right)\left(\left\langle h^{\eps}_{\uu,\vv;r},\Delta\varphi^{\eps}_{\uu,\vv}\right\rangle -\left\langle \nicefrac{1}{12}-2\varphi^{\eps}_{\uu,\vv},\varphi^{\eps}_{\uu,\vv}\right\rangle \right)(1-\Psi^{\eps}_{r})\,\dif r.\label{eq:BBtildedif}
\end{align}
In the following, we will estimate each term on the right side of
\zcref{eq:BBtildedif} separately.

First, we can estimate 
\begin{align*}
\left|\int^{s+2\eps^{2}}_{s}\left\langle h^{\eps}_{\uu,\vv;r}-h^{\eps}_{\uu,\vv;s},\varphi^{\eps}_{\uu,\vv}\right\rangle \,\dif\Psi^{\eps}_{r}\right| & \le\|\varphi^{\eps}_{\uu,\vv}\|_{L^{1}([0,L])}\|h^{\eps}_{\uu,\vv}\|_{\mathcal{C}^{\chi}_{\mathfrak{s}}([s,t]\times[0,L])}\int^{s+2\eps^{2}}_{s}(r-s)^{\chi}\left|\partial_{r}\Psi^{\eps}_{r}\right|\,\dif r\\
 & \le C\eps^{2\chi}\|h^{\eps}_{\uu,\vv}\|_{\mathcal{C}^{\chi}_{\mathfrak{s}}([s,t]\times[0,L])},
\end{align*}
where $\chi$ can be any constant in $(0,\oh)$ and the constant $C$
depends only on $\uu,\vv,\psi,\chi$. Here we used that $\int^{s+2\eps^{2}}_{s}\left|\partial_{r}\Psi^{\eps}_{r}\right|\,\dif r$
is bounded by a constant independent of $\eps$. So by \zcref{p.holder}
we have (with an identical argument for the other end of the time
interval)
\begin{equation}
\left|\int^{s+2\eps^{2}}_{s}\left\langle h^{\eps}_{\uu,\vv;r}-h^{\eps}_{\uu,\vv;s},\varphi^{\eps}_{\uu,\vv}\right\rangle \,\dif\Psi^{\eps}_{r}\right|,\left|\int^{t}_{t-2\eps^{2}}\left\langle h^{\eps}_{\uu,\vv;r}-h^{\eps}_{\uu,\vv;s},\varphi^{\eps}_{\uu,\vv}\right\rangle \,\dif\Psi^{\eps}_{r}\right|\to0\qquad\text{in probability as }\eps\to0.\label{eq:dtsto0}
\end{equation}

Second, for the terms involving the noise $W$, we write
\[
\int^{s+2\eps^{2}}_{s}\left\langle W_{r}-W_{s},\varphi^{\eps}_{\uu,\vv}\right\rangle \,\dif\Psi^{\eps}_{r}=\int^{s+2\eps^{2}}_{s}\left\langle \int^{s+2\eps^{2}}_{q}\varphi^{\eps}_{\uu,\vv}\partial_{r}\Psi^{\eps}_{r}\,\dif r,\dif W_{q}\right\rangle ,
\]
and hence 
\begin{align}
\mathbb{E}\left[\int^{s+2\eps^{2}}_{s}\left\langle W_{r}-W_{s},\varphi^{\eps}_{\uu,\vv}\right\rangle \,\dif\Psi^{\eps}_{r}\right]^{2} & =\int^{s+2\eps^{2}}_{s}\int^{L}_{0}\left(\int^{s+2\eps^{2}}_{q}\varphi^{\eps}_{\uu,\vv}(x)\partial_{r}\Psi^{\eps}_{r}\,\dif r\right)^{2}\,\dif x\,\dif q\nonumber \\
 & \le\eps^{2}\int^{L}_{0}\varphi^{\eps}_{\uu,\vv}(x)^{2}\,\dif x\le C\eps\label{eq:noisebd}
\end{align}
for a new constant $C$ depending only on $\uu,\vv,\eps$. Again using
a similar argument for the other end of the time interval, we obtain
\begin{equation}
\int^{s+2\eps^{2}}_{s}\left\langle W_{r}-W_{s},\varphi^{\eps}_{\uu,\vv}\right\rangle \,\dif\Psi^{\eps}_{r},\int^{t}_{t-2\eps^{2}}\left\langle W_{t}-W_{r},\varphi^{\eps}_{\uu,\vv}\right\rangle \,\dif\Psi^{\eps}_{r}\to0\qquad\text{in probability as }\eps\to0.\label{eq:Wsto0}
\end{equation}

In addition, we have the following deterministic bound:
\begin{equation}
\left|\left(\int^{s+\eps^{2}}_{s}+\int^{t}_{t-\eps^{2}}\right)\left\langle \nicefrac{1}{12}-2\varphi^{\eps}_{\uu,\vv},\varphi^{\eps}_{\uu,\vv}\right\rangle \left(1-\Psi^{\eps}_{r}\right)\,\dif r\right|\le C\eps.\label{eq:detbd}
\end{equation}

To deal with the spatial integral integral $\langle h^{\eps}_{\uu,\vv;r},\Delta\varphi^{\eps}_{\uu,\vv}\rangle$,
first recall that for $\eps\ll1$, the function $\varphi^{\eps}_{\uu,\vv}$
restricted to $[0,L]$ is supported on $\eps$-neighborhoods of $0$
and $L$, so (using also the evenness of $\varphi^{\eps}_{\uu,\vv}$
about $0$ and $L$) we have $\int^{\eps}_{0}\Delta\varphi^{\eps}_{\uu,\vv}(x)\,\dif x=\int^{L}_{L-\eps}\Delta\varphi^{\eps}_{\uu,\vv}(x)\,\dif x=0$.
Thus, one can write $\int^{\eps}_{0}h^{\eps}_{\uu,\vv;r}(x)\Delta\varphi^{\eps}_{\uu,\vv}(x)\,\dif x=\int^{\eps}_{0}\left[h^{\eps}_{\uu,\vv;r}(x)-h^{\eps}_{\uu,\vv;r}(0)\right]\Delta\varphi^{\eps}_{\uu,\vv}(x)\,\dif x$
and then make use of the spatial regularity of $h^{\eps}_{\uu,\vv;r}$.
Of course, a similar argument holds for the integral on $[L-\eps,L]$.
Thus we can write (by symmetry with the same argument for the spatial
integrals near $0$ and near $L$) that
\begin{align*}
\left(\int^{s+2\eps^{2}}_{s}+\int^{t}_{t-2\eps^{2}}\right)\left|\left\langle h^{\eps}_{\uu,\vv;r},\Delta\varphi^{\eps}_{\uu,\vv}\right\rangle \right|\,\dif r & \le C\eps^{2}\|h^{\eps}_{\uu,\vv;r}\|_{\mathcal{C}^{\chi}_{\mathfrak{s}}([s,t]\times[0,L])}\int^{\eps}_{0}|x|^{\chi}|\Delta\varphi^{\eps}_{\uu,\vv}(x)|\,\dif x\\
 & \le C\eps^{2+\chi}\|h^{\eps}_{\uu,\vv;r}\|_{\mathcal{C}^{\chi}_{\mathfrak{s}}([s,t]\times[0,L])}\|\Delta\varphi^{\eps}_{\uu,\vv}(x)\|_{L^{1}([0,L])}\\
 & \le C\eps^{\chi}\|h^{\eps}_{\uu,\vv;r}\|_{\mathcal{C}^{\chi}_{\mathfrak{s}}([s,t]\times[0,L])},
\end{align*}
where the constant $C$ again depends only on $\uu,\vv,\psi$. This
means (again using \zcref{p.holder}) that
\begin{equation}
\left(\int^{s+2\eps^{2}}_{s}+\int^{t}_{t-2\eps^{2}}\right)\left\langle h^{\eps}_{\uu,\vv;r},\Delta\varphi^{\eps}_{\uu,\vv}\right\rangle \left(1-\Psi^{\eps}_{r}\right)\,\dif r\to0\qquad\text{in probability as }\eps\to0\label{eq:laplacianto0}
\end{equation}
as well. Now using \zcref[comp=true]{eq:dtsto0,eq:Wsto0,eq:detbd,eq:laplacianto0}
in \zcref{eq:BBtildedif}, we see that 
\[
\tilde{\mathcal{B}}^{\eps}_{\uu,\vv;s,t}-\mathcal{B}^{\eps}_{\uu,\vv;s,t}(\varphi^{\eps}_{\uu,\vv})\to0\qquad\text{in probability as }\eps\to0,
\]
and the proof is complete.
\end{proof}

\subsection{Relationship between the kernels: proof of \texorpdfstring{\zcref{prop:explicit-comparison}}{Proposition~\ref{prop:explicit-comparison}}
\label{subsec:prove-explicit-comparison}}

We now prove \zcref{prop:explicit-comparison}. The key ingredients
are \zcref{prop:RH,prop:eq-boundary}.
\begin{proof}[Proof of \zcref{prop:explicit-comparison}]
A brief inspection of \zcref{eq:Texpcheckexplicit} along with \zcref{tab:RS-table,tab:Mhatdef}
shows that it suffices to show that, for all $\tau\in\check{\mathsf{T}}_{\ES}$,
we have 
\begin{equation}
\adjustlimits\lim_{\eps\downarrow0}\sup_{\zeta\in(0,\nicefrac{\eps}{100})}\mathbb{E}\left[\left|\mathcal{X}^{\eps}_{\uu,\vv;0,T}\left(\boxop{\hat{M}\tau}^{\eps,\zeta}\right)-\mathcal{X}^{\eps}_{\uu,\vv;0,T}\left(\bboxop{\hat{M}\tau}^{\eps,\zeta}\right)\right|^{p}\right]=0.\label{eq:goal-abstract}
\end{equation}
Using \zcref{prop:RH}, we can write the difference in \zcref{eq:goal-abstract}
as 
\[
\boxop{\hat{M}\tau}^{\eps,\zeta}_{t}(x)-\bboxop{\hat{M}\tau}^{\eps,\zeta}_{t}(x)=\hat{\mathcal{R}}^{\eps,\zeta}\left(\mathcal{H}^{\eps,\zeta}\tau-\tilde{\mathcal{H}}^{\eps,\zeta}\tau\right)_{t}(x).
\]
Therefore, we have, using the reconstruction theorem \cites[Thm.~3.10]{hairer:2014:theory}
in a similar manner to \zcref{eq:apply-reconstruction} but using
\zcref{prop:eq-boundary} to see that $\left(\mathcal{H}^{\eps,\zeta}\tau-\tilde{\mathcal{H}}^{\eps,\zeta}\tau\right)_{q}(x_{0})=0$,
that for all $q\in\mathbb{R}$,
\begin{align*}
 & \left|\iint_{\mathbb{R}^{2}}\psi^{\eps^{2}}(r-q)\psi^{\eps}(x-x_{0})\left(\boxop{\hat{M}\tau}^{\eps,\zeta}_{r}(x)-\bboxop{\hat{M}\tau}^{\eps,\zeta}_{r}(x)\right)\,\dif x\,\dif r\right|\\
 & \qquad=\left|\iint_{\mathbb{R}^{2}}\psi^{\eps^{2}}(r-q)\psi^{\eps}(x-x_{0})\left(\hat{\mathcal{R}}^{\eps,\zeta}\left(\mathcal{H}^{\eps,\zeta}\tau-\tilde{\mathcal{H}}^{\eps,\zeta}\tau\right)_{t}(x)-\hat{\Pi}^{\eps,\zeta}_{(q,x_{0})}\left(\left(\mathcal{H}^{\eps,\zeta}\tau-\tilde{\mathcal{H}}^{\eps,\zeta}\tau\right)_{q}(x_{0})\right)_{r}(x)\right)\,\dif x\,\dif r\right|\\
 & \qquad\le C(1+\|(\hat{\Pi}^{\eps,\zeta},\hat{\Gamma}^{\eps,\zeta})\|^{2})\|\mathcal{H}^{\eps,\zeta}\tau-\tilde{\mathcal{H}}^{\eps,\zeta}\tau\|_{\mathcal{D}^{9\kappa;\eps,\zeta}}\eps^{9\kappa}.
\end{align*}
Here we can use the standard reconstruction theorem rather than the
localized version because we are considering stationary objects so
there is no singularity at $t=0$. We can estimate by the triangle
inequality that
\[
\|\mathcal{H}^{\eps,\zeta}\tau-\tilde{\mathcal{H}}^{\eps,\zeta}\tau\|_{\mathcal{D}^{9\kappa;\eps,\zeta}}\le\|\mathcal{H}^{\eps,\zeta}\tau\|_{\mathcal{D}^{9\kappa;\eps,\zeta}}+\|\tilde{\mathcal{H}}^{\eps,\zeta}\tau\|_{\mathcal{D}^{9\kappa;\eps,\zeta}},
\]
and this is bounded by a polynomial in $\|(\hat{\Pi}^{\eps,\zeta},\hat{\Gamma}^{\eps,\zeta})\|$
by the last statement n \zcref{lem:Lranges}. Integrating in time
and taking the $p$th moment using \zcref{thm:model-norm-bounded},
we obtain \zcref{eq:goal-abstract}.
\end{proof}

\section{\label{sec:Explicit-calculations}Calculations on the explicit terms}

In this section we compute the expectation of each term appearing
on the right side of \zcref{eq:expand-nonlinearity-boundary}, as
well as the renormalization constants $C^{(1)}_{\zeta}(x)$ and $C^{(2)}$
defined in \zcref[range]{eq:C1zetadef,eq:C2renorm}, and the variance/correlation
function of $\rscherryrr[rsP]^{\zeta}$. The common thread among the
calculations in this section is that they are really calculations
of explicit, nonzero quantities, rather than upper/lower bounds. In
contrast, the subsequent \zcref{sec:BPHZ} concerns upper bounds on
correlation functions that go to zero in the limit $\eps,\zeta\to0$.

\subsection{\label{subsec:First-renormalization-constant}First renormalization
constant}

We begin by investigating the two-point correlation function of $\rslollipopr[rsP]^{\zeta}$.
We recall that $\rslollipopr[rsP]^{\zeta}$ is the stationary gradient
of the reflected and periodized solution to the Edwards--Wilkinson
equation, so the following calculation closely parallels the corresponding
calculation for the Edwards--Wilkinson equation on the line or on
the torus.
\begin{lem}
\label{lem:EW-cov}We have, for all $t,t',x,x'\in\mathbb{R}$, that
\begin{equation}
\mathbb{E}\left[\rslollipopr[rsP]^{\zeta}_{t}(x)\rslollipopr[rsP]^{\zeta}_{t'}(x')\right]=p_{|t-t'|}*\Sh^{\zeta}_{2L}(x-x')-p_{|t-t'|}*\Sh^{\zeta}_{2L}(x+x').\label{eq:EW-cov}
\end{equation}
\end{lem}

\begin{proof}
From the definitions \zcref{eq:basic-reconstruction,eq:Qinductive-hat}
we see that 
\[
\rslollipopr[rsP]^{\zeta}_{t}(x)=\int^{t}_{-\infty}\int_{\mathbb{R}}p_{t-s}(x-y)\partial_{x}\dif W^{\zeta}_{s}(y)=\int^{t}_{-\infty}\int_{\mathbb{R}}\partial_{x}p_{t-s}(x-y)\,\dif W^{\zeta}_{s}(y).
\]
We can thus compute the second moment as
\begin{align*}
\mathbb{E}\left[\rslollipopr[rsP]^{\zeta}_{t}(x)\rslollipopr[rsP]^{\eps,\zeta}_{t'}(x')\right] & =\int^{t\wedge t'}_{-\infty}\iint_{\mathbb{R}^{2}}\partial_{x}p_{t-s}(x-y)\partial_{x}p_{t-s}(x'-y')\left(\Sh^{\zeta}_{2L}(y-y')+\Sh^{\zeta}_{2L}(y+y')\right)\,\dif y\,\dif y'\,\dif s\\
 & =-\int^{t\wedge t'}_{-\infty}\left(\Delta p_{t+t'-2s}*\Sh^{\zeta}_{2L}(x-x')-\Delta p_{t+t'-2s}*\Sh^{\zeta}_{2L}(x+x')\right)\,\dif s\\
 & =\int^{t\wedge t'}_{-\infty}\frac{\dif}{\dif s}\left(p_{t+t'-2s}*\Sh^{\zeta}_{2L}(x-x')-p_{t+t'-2s}*\Sh^{\zeta}_{2L}(x+x')\right)\,\dif s\\
 & =p_{|t-t'|}*\Sh^{\zeta}_{2L}(x-x')-p_{|t-t'|}*\Sh^{\zeta}_{2L}(x+x'),
\end{align*}
which is \zcref{eq:EW-cov}.
\end{proof}

From this we can in particular prove \zcref{prop:C1zeta}.
\begin{proof}[Proof of \zcref{prop:C1zeta}]
From \zcref{eq:EW-cov} we compute
\[
\mathbb{E}\left[\rscherryrr[rsP]^{\zeta}_{t}(x)\right]=\mathbb{E}\left[\rslollipopr[rsP]^{\zeta}_{t}(x)\rslollipopr[rsP]^{\zeta}_{t}(x)\right]=\Sh^{\zeta}_{2L}(0)-\Sh^{\zeta}_{2L}(2x)\overset{\zcref{eq:Shrescale}}{=}\Sh^{\zeta}_{2L}(0)-\oh\Sh^{\zeta/2}_{L}(x).\qedhere
\]
\end{proof}

\subsection{Fourier preliminaries}

Let us first fix our normalization of Fourier series. We will write
a $2L$-periodic function $f$ in terms of its Fourier series
\begin{equation}
f(x)=\sum_{k\in\mathbb{Z}}\hat{f}(k)\e^{\pi\ii kx/L},\ x\in\mathbb{R},\qquad\text{with}\qquad\hat{f}(k)=\frac{1}{2L}\int^{L}_{-L}f(x)\e^{-\pi\ii kx/L}\,\dif x,\ k\in\mathbb{Z}.\label{eq:fourierseries}
\end{equation}
With this normalization, we have the Fourier identities
\begin{equation}
\widehat{\partial_{x}f}(k)=\frac{\pi\ii k}{L}\hat{f}(k),\quad\widehat{fg}(k)=\sum_{j\in\mathbb{Z}}\hat{f}(j)\hat{g}(k-j),\quad\text{and}\quad\widehat{p_{t}*f}(k)=\e^{-\frac{k^{2}\pi^{2}}{2L^{2}}t}\hat{f}(k)\quad\text{for all }k\in\mathbb{Z},\label{eq:fourier-identities}
\end{equation}
as well as
\begin{equation}
\frac{1}{2L}\int^{L}_{-L}\e^{\pi\ii jx/L}\,\dif x=\delta_{j}\coloneqq\mathbf{1}_{j=0}.\label{eq:integratedelta}
\end{equation}
We also recall that
\begin{equation}
f\text{ even and real}\iff\hat{f}\text{ even and real}.\label{eq:fevenreal}
\end{equation}
In particular, if $f$ and $g$ are even real functions, then
\begin{equation}
\int^{L}_{0}f(x)g(x)\,\dif x=L\cdot\frac{1}{2L}\int^{L}_{-L}f(x)g(x)\,\dif x\overset{\zcref{eq:fourierseries}}{=}L\widehat{fg}(0)\overset{\zcref{eq:fourier-identities}}{=}L\sum_{j\in\mathbb{Z}}\hat{f}(j)\hat{g}(-j)\overset{\zcref{eq:fevenreal}}{=}L\sum_{j\in\mathbb{Z}}\hat{f}(j)\hat{g}(j).\label{eq:feven-int}
\end{equation}

We will only use the Fourier transform in space, never in time, and
for the realizations of our trees we will use the aesthetically nicer
notation, for example, $\widehat{\rsipcherryrb[rsP]}^{\eps,\zeta}_{t}(j)$,
rather than the perhaps-more-correct $\widehat{\rsipcherryrb[rsP]^{\eps,\zeta}_{t}}(j)$.
It is an immediate consequence of \zcref{eq:fourier-identities} that,
if we now assume that $f$ is a time-dependent periodic function,
then
\begin{equation}
\widehat{(p\circledast f)_{t}}(k)=\int^{t}_{-\infty}\e^{-\frac{\pi^{2}k^{2}}{2L^{2}}(t-s)}\hat{f}_{s}(k)\,\dif s\quad\text{and hence}\quad\widehat{(\partial_{x}p\circledast f)_{t}}(k)=\frac{\pi\ii k}{L}\int^{t}_{-\infty}\e^{-\frac{\pi^{2}k^{2}}{2L^{2}}(t-s)}\hat{f}_{s}(k)\,\dif s.\label{eq:HK-spacetime-Fourier}
\end{equation}
We also record the fact that, for each fixed $k\in\mathbb{Z}$, we
have
\begin{equation}
\lim_{\zeta\downarrow0}\widehat{\Sh}^{\zeta}_{2L}(k)=\lim_{\zeta\downarrow0}\frac{1}{2L}\int^{L}_{-L}\Sh^{\zeta}_{2L}(x)\e^{-\pi\ii kx/L}\,\dif x=\frac{1}{2L}\label{eq:limShfourier}
\end{equation}
by the definition \zcref{eq:Shepsdef}, and moreover that
\begin{equation}
\widehat{\Sh}^{\zeta}_{2L}(k)=\widehat{\Sh}^{\zeta}_{2L}(-k)\qquad\text{for all }k\in\mathbb{Z}\label{eq:Fourier-symmetric}
\end{equation}
since $\Sh^{\zeta}_{2L}$ is real and even.

We can also write
\begin{align}
\widehat{\varphi}^{\eps}_{\uu,\vv} & (k)\ovset{\zcref{eq:varphiuvdef}}=-\frac{\uu}{2L}\int^{L}_{-L}\e^{-\pi\ii kx/L}\Sh_{2L}*\psi^{\eps}(x)\,\dif x-\frac{\vv}{2L}\int^{L}_{-L}\e^{-\pi\ii kx/L}\Sh_{2L}*\psi^{\eps}(x-L)\,\dif x\nonumber \\
 & =-\frac{\uu}{2L}\int^{\infty}_{-\infty}\e^{-\pi\ii kx/L}\psi^{\eps}(x)\,\dif x-\frac{\vv}{2L}\int^{\infty}_{-\infty}\e^{-\pi\ii kx/L}\psi^{\eps}(x-L)\,\dif x\nonumber \\
\ovset{\zcref{eq:psiepsdef}} & =-\frac{\uu\eps^{-1}}{2L}\int^{\infty}_{-\infty}\e^{-\pi\ii kx/L}\psi(\eps^{-1}x)\,\dif x-\frac{\vv\eps^{-1}}{2L}\int^{\infty}_{-\infty}\e^{-\pi\ii kx/L}\psi(\eps^{-1}(x-L))\,\dif x\nonumber \\
 & =-\frac{\uu}{2L}\int^{\infty}_{-\infty}\e^{-\pi\ii\eps kx/L}\psi(x)\,\dif x-(-1)^{k}\frac{\vv}{2L}\int^{\infty}_{-\infty}\e^{-\pi\ii\eps kx/L}\psi(x)\,\dif x=-\frac{1}{2L}\left(\uu+(-1)^{k}\vv\right)\widehat{\psi}(\eps k/L)\label{eq:phihatepsilon}
\end{align}
as long as $\eps<\frac{4}{3}L$, which allows the third identity to
hold (recalling from \zcref{eq:psiproperties} that $\supp\psi\subseteq(-\nicefrac{3}{4},\nicefrac{3}{4})$).
Here we have used the notation
\[
\hat{\psi}(k)\coloneqq\int^{\infty}_{-\infty}\e^{-\pi\ii kx}\psi(x)\,\dif x
\]
for the full-line Fourier transform of $\psi$. Since $\psi$ is not
$2L$-periodic and we will only use the full-line Fourier transform
for $\psi$, we trust that this slightly abusive notation will not
cause confusion. In particular, by \zcref{eq:psiepshalflineint} we
have $\hat{\psi}(0)=1$ and hence 
\begin{equation}
\widehat{\varphi}^{\eps}_{\uu,\vv}(0)=-\frac{\uu+\vv}{2L}.\label{eq:psihatzero}
\end{equation}

We will frequently use the following statements on the Fourier transforms
of $\rslollipopr[rsP]^{\zeta}$ and $\rslollipopb[rsP]^{\eps}_{t}$:
\begin{lem}
\label{lem:EW-cov-fourier}We have
\begin{equation}
\mathbb{E}\left[\widehat{\rslollipopr[rsP]}^{\zeta}_{t}(k)\widehat{\rslollipopr[rsP]}^{\zeta}_{s}(\ell)\right]=\left(\delta_{k+\ell}-\delta_{k-\ell}\right)\widehat{\Sh}^{\zeta}_{2L}(k)\e^{-\frac{\pi^{2}k^{2}}{2L^{2}}|t-s|}\label{eq:EW-cov-fourier}
\end{equation}
and
\begin{equation}
\widehat{\rslollipopb[rsP]}^{\eps}(k)=\frac{2\ii L(1-\delta_{k})}{\pi k}\widehat{\varphi}^{\eps}_{\uu,\vv}(k).\label{eq:lollipopft}
\end{equation}
\end{lem}

\begin{proof}
For \zcref{eq:EW-cov-fourier}, we can write 
\begin{align*}
\mathbb{E}\left[\widehat{\rslollipopr[rsP]}^{\zeta}_{t}(k)\widehat{\rslollipopr[rsP]}^{\zeta}_{s}(\ell)\right]\ovset{\zcref{eq:fourier-identities}} & =\frac{1}{(2L)^{2}}\iint_{[-L,L]^{2}}\mathbb{E}\left[\rslollipopr[rsP]^{\zeta}_{t}(x)\rslollipopr[rsP]^{\zeta}_{s}(y)\right]\e^{-\pi\ii(kx+\ell y)/L}\,\dif x\,\dif y\\
\ovset{\zcref{eq:EW-cov}} & =\frac{1}{(2L)^{2}}\iint_{[-L,L]^{2}}\left(p_{|t-s|}*\Sh^{\zeta}_{2L}(x-y)-p_{|t-s|}*\Sh^{\zeta}_{2L}(x+y)\right)\e^{-\pi\ii(kx+\ell y)/L}\,\dif x\,\dif y\\
 & =\frac{1}{(2L)^{2}}\iint_{[-L,L]^{2}}p_{|t-s|}*\Sh^{\zeta}_{2L}(x)\left(\e^{-\pi\ii(k(x+y)+\ell y)/L}-\e^{-\pi\ii(k(x-y)+\ell y)/L}\right)\,\dif x\,\dif y\\
\ovset{\zcref{eq:integratedelta}} & =\left(\delta_{k+\ell}-\delta_{k-\ell}\right)\left(p_{|t-s|}*\Sh^{\zeta}_{2L}\right)^{\wedge}(j)\ovset{\zcref{eq:fourier-identities}}=\left(\delta_{k+\ell}-\delta_{k-\ell}\right)\e^{-\frac{\pi^{2}k^{2}}{2L^{2}}|t-s|}\widehat{\Sh}^{\zeta}_{2L}(j).
\end{align*}
For \zcref{eq:lollipopft}, we simply compute 
\[
\widehat{\rslollipopb[rsP]}^{\eps}(k)\overset{\zcref{eq:HK-spacetime-Fourier}}{=}\frac{\pi\ii k}{L}\int^{0}_{-\infty}\e^{\frac{\pi^{2}k^{2}}{2L^{2}}s}\widehat{\varphi}^{\eps}_{\uu,\vv}(k)\,\dif s=\frac{2\ii L(1-\delta_{k})}{\pi k}\widehat{\varphi}^{\eps}_{\uu,\vv}(k).\qedhere
\]
\end{proof}

The following lemma will be important in some of our Fourier calculations.
\begin{lem}
\label{lem:partialfractions}For any $\omega\in\mathbb{C}\setminus(\ii\RR)$
and any $y\in((2\Re\omega)^{-1}\mathbb{Z})\setminus((\omega^{-1}\mathbb{Z})\cup(\overline{\omega}^{-1}\mathbb{Z}))$,
we have
\begin{equation}
\sum_{k\in\mathbb{Z}}\frac{|\omega|^{2}y^{2}-k^{2}}{|\omega|^{4}y^{4}-2\Re(\omega^{2})k^{2}y^{2}+k^{4}}=0.\label{eq:sumiscot}
\end{equation}
In particular, for any $n\in\mathbb{Z}\setminus\{0\}$, we have
\begin{equation}
\sum_{k\in\mathbb{Z}}\frac{n^{2}-k^{2}}{n^{4}+n^{2}k^{2}+k^{4}}=0\qquad\text{and}\qquad\sum_{k\in\mathbb{Z}}\frac{n^{2}-2k^{2}}{n^{4}+4k^{4}}=0.\label{eq:seriesiszero}
\end{equation}
\end{lem}

\begin{proof}
We have the identity of meromorphic functions
\begin{equation}
\pi\cot(\pi y)=\frac{1}{y}+2y\sum^{\infty}_{k=1}\frac{1}{y^{2}-k^{2}}=\sum_{k\in\mathbb{Z}}\frac{y}{y^{2}-k^{2}},\label{eq:Euleridentity}
\end{equation}
with a pole at each element of $\mathbb{Z}$. This is a classical
formula due to Euler; see e.g.\ \cite[Chapter 26]{aigner:ziegler:proofs:2018}
for a modern exposition. Replacing $y$ with $\omega y$ in \zcref{eq:Euleridentity},
we see that, for any $y\not\in\omega^{-1}\mathbb{Z}$,
\[
\pi\cot(\pi\omega y)=\sum_{k\in\mathbb{Z}}\frac{\omega y}{\omega^{2}y^{2}-k^{2}},
\]
and hence, if $y\not\in(\omega^{-1}\mathbb{Z})\cup(\overline{\omega}^{-1}\mathbb{Z})$,
then
\begin{multline}
\pi\left(\cot(\pi\omega y)+\cot(\pi\overline{\omega}y)\right)=\sum_{k\in\mathbb{Z}}\left(\frac{\omega y}{\omega^{2}y^{2}-k^{2}}+\frac{\overline{\omega}y}{\overline{\omega}^{2}y^{2}-k^{2}}\right)\\
=\sum_{k\in\mathbb{Z}}\frac{\omega y(\overline{\omega}^{2}y^{2}-k^{2})+\overline{\omega}y(\omega^{2}y^{2}-k^{2})}{(\omega^{2}y^{2}-k^{2})(\overline{\omega}^{2}y^{2}-k^{2})}=2y(\Re\omega)\sum_{k\in\mathbb{Z}}\frac{|\omega|^{2}y^{2}-k^{2}}{|\omega|^{4}y^{4}-2\Re(\omega^{2})k^{2}y^{2}+k^{4}}.\label{eq:expandcotsum}
\end{multline}
Rearranging using the assumption $\Re\omega\ne0$, we get
\begin{align*}
\sum_{k\in\mathbb{Z}}\frac{|\omega|^{2}y^{2}-k^{2}}{|\omega|^{4}y^{4}-2\Re(\omega^{2})k^{2}y^{2}+k^{4}} & =\frac{\pi}{2y(\Re\omega)}\left(\cot(\pi\omega y)+\cot(\pi\overline{\omega}y)\right)\\
 & =\frac{\pi\ii(\e^{4\pi\ii y\Re\omega}-1)}{y(\Re\omega)(\e^{2\pi\ii y\omega}-1)(\e^{2\pi\ii y\overline{\omega}}-1)},
\end{align*}
and the last quantity is evidently zero if $y\in(2\Re\omega)^{-1}\mathbb{Z}\setminus((\omega^{-1}\mathbb{Z})\cup(\overline{\omega}^{-1}\mathbb{Z}))$,
so we obtain \zcref{eq:sumiscot}. The claims in \zcref{eq:seriesiszero}
then follow by using \zcref{eq:sumiscot} with $\omega=\frac{1}{2}\left(1+\ii\sqrt{3}\right)$
and $\omega=\frac{1}{2}(1+\ii)$, respectively.
\end{proof}

\subsection{The random boundary term}

Recall the definition \zcref{eq:Xepsdef} of $\mathcal{X}^{\eps}_{\uu,\vv;0,t}$.
In this section we calculate the limiting behavior of $\mathcal{X}^{\eps}_{\uu,\vv;0,t}\left(\rsrenorm[rsP]^{\zeta}\right)$.
Unlike the contributions from the other trees appearing on the right
side of \zcref{eq:expand-nonlinearity-boundary}, this one does not
converge to a constant when we take $\zeta\downarrow0$ and then $\eps\downarrow0$.
Instead, it converges in distribution to a Gaussian random variable
that is independent of the driving noise. Since $\mathcal{X}^{\eps}_{\uu,\vv;0,t}\left(\rsrenorm[rsP]^{\zeta}\right)$
lives in the homogeneous second Wiener chaos of the driving noise
(as it is mean $0$ by \zcref{eq:C1zetadef,eq:Qtilderenorm}), we
can study its limiting Gaussian behavior using the fourth moment theorem.
In fact, an enhanced version of the fourth moment theorem proved in
\cite{peccati:tudor:2005:gaussian} will give us the independence
from the driving noise with no additional work. The main result of
this section is the following.
\begin{prop}
\label{prop:Xredcherrylimit}We have
\begin{equation}
\lim_{\eps\downarrow0}\lim_{\zeta\downarrow0}\Law\left(\mathcal{X}^{\eps}_{\uu,\vv;0,t}\left(\rsrenorm[rsP]^{\zeta}\right),(\dif W_{t})\right)=\mathcal{N}(0,t(\uu^{2}+\vv^{2})V_{\psi})\otimes\Law\left((\dif W_{t})\right)\qquad\text{weakly},\label{eq:Xredcherryconvindistn}
\end{equation}
where 
\begin{equation}
V_{\psi}\coloneqq\frac{1}{\pi^{2}}\int_{\mathbb{R}^{2}}\frac{\hat{\psi}(x)\left(\widehat{\psi}(x)-\widehat{\psi}(x-2y)\right)}{x^{2}+(x-y)^{2}}\,\dif x\,\dif y.\label{eq:Vpsidef}
\end{equation}
\end{prop}

As mentioned, the proof of \zcref{prop:Xredcherrylimit} (which we
complete at the end of this section) relies on a version of the fourth
moment theorem, so we begin by computing the second and fourth moments.
\begin{lem}
\label{lem:redcherry-variance}We have
\begin{equation}
\adjustlimits\lim_{\eps\downarrow0}\lim_{\zeta\downarrow0}\Var\left(\mathcal{X}^{\eps}_{\uu,\vv;0,t}\left(\rsrenorm[rsP]^{\zeta}\right)\right)=t(\uu^{2}+\vv^{2})V_{\psi},\label{eq:redcherry-var-limit}
\end{equation}
where $V_{\psi}$ is defined as in \zcref{eq:Vpsidef}.
\end{lem}

\begin{proof}
We have
\begin{align*}
\mathcal{X}^{\eps}_{\uu,\vv;0,t}\left(\rsrenorm[rsP]^{\zeta}\right)\overset{\zcref{eq:Xepsdef}}{=}\int^{t}_{0}\Psi^{\eps}_{s} & \int^{L}_{0}\varphi^{\eps}_{\uu,\vv}(x)\rsrenorm[rsP]^{\zeta}_{s}(x)\,\dif x\,\dif s\overset{\zcref{eq:feven-int}}{=}L\int^{t}_{0}\Psi^{\eps}_{s}\sum_{k\in\mathbb{Z}}\widehat{\varphi}^{\eps}_{\uu,\vv}(k)\widehat{\rsrenorm[rsP]}^{\eps,\zeta}_{s}(k)\,\dif s,
\end{align*}
so
\begin{align}
\Var & \left(\mathcal{X}^{\eps}_{\uu,\vv;0,t}\left(\rsrenorm[rsP]^{\zeta}\right)\right)\nonumber \\
 & =L^{2}\sum_{j,k\in\mathbb{Z}}\widehat{\varphi}^{\eps}_{\uu,\vv}(j)\widehat{\varphi}^{\eps}_{\uu,\vv}(k)\int^{t}_{0}\int^{t}_{0}\Psi^{\eps}_{r}\Psi^{\eps}_{s}\Cov\left(\widehat{\rsrenorm[rsP]}^{\zeta}_{r}(j),\widehat{\rsrenorm[rsP]}^{\zeta}_{s}(k)\right)\,\dif s\,\dif r\nonumber \\
\ovset{\zcref{eq:subtract-exp}} & =L^{2}\sum_{j,k\in\mathbb{Z}}\widehat{\varphi}^{\eps}_{\uu,\vv}(j)\widehat{\varphi}^{\eps}_{\uu,\vv}(k)\int^{t}_{0}\int^{t}_{0}\Psi^{\eps}_{r}\Psi^{\eps}_{s}\Cov\left(\widehat{\rscherryrr[rsP]}^{\zeta}_{r}(j),\widehat{\rscherryrr[rsP]}^{\zeta}_{s}(k)\right)\,\dif s\,\dif r\nonumber \\
 & =2L^{2}\sum_{j,k,\ell,m\in\mathbb{Z}}\widehat{\varphi}^{\eps}_{\uu,\vv}(j)\widehat{\varphi}^{\eps}_{\uu,\vv}(k)\int^{t}_{0}\int^{t}_{0}\Psi^{\eps}_{r}\Psi^{\eps}_{s}\mathbb{E}\left[\widehat{\rslollipopr[rsP]}^{\zeta}_{r}(\ell)\widehat{\rslollipopr[rsP]}^{\zeta}_{s}(m)\right]\mathbb{E}\left[\widehat{\rslollipopr[rsP]}^{\zeta}_{r}(j-\ell)\widehat{\rslollipopr[rsP]}^{\zeta}_{s}(k-m)\right]\,\dif s\,\dif r\nonumber \\
\ovset{\zcref{eq:EW-cov-fourier}} & =2L^{2}\sum_{j,k,\ell,m\in\mathbb{Z}}\widehat{\varphi}^{\eps}_{\uu,\vv}(j)\widehat{\varphi}^{\eps}_{\uu,\vv}(k)(\delta_{\ell+m}-\delta_{\ell-m})(\delta_{j-\ell+k-m}-\delta_{j-\ell-(k-m)})\widehat{\Sh}^{\zeta}_{2L}(\ell)\widehat{\Sh}^{\zeta}_{2L}(j-\ell)\nonumber \\
 & \qquad\qquad\times\int^{t}_{0}\int^{t}_{0}\Psi^{\eps}_{r}\Psi^{\eps}_{s}\e^{-\frac{\pi^{2}(\ell^{2}+(j-\ell)^{2})}{2L^{2}}|r-s|}\,\dif s\,\dif r.\label{eq:varcomp}
\end{align}
We simplify the product of deltas in \zcref{eq:varcomp} by computing
\begin{align}
\sum_{m\in\mathbb{Z}}(\delta_{\ell+m}-\delta_{\ell-m})(\delta_{j-\ell+k-m}-\delta_{j-\ell-(k-m)}) & =\left(\delta_{j-\ell+k+\ell}-\delta_{j-\ell-(k+\ell)}\right)-\left(\delta_{j-\ell+k-\ell}-\delta_{j-\ell-(k-\ell)}\right)\nonumber \\
 & =\delta_{j+k}-\delta_{j-k-2\ell}-\delta_{j+k-2\ell}+\delta_{j-k}.\label{eq:prodofdeltas}
\end{align}

To estimate the integral in \zcref{eq:varcomp}, we start by observing
that, for $j,\ell$ not both $0$,
\begin{align*}
I_{j,\ell} & \coloneqq\int^{t}_{0}\int^{t}_{0}\Psi^{\eps}_{r}\Psi^{\eps}_{s}\e^{-\frac{\pi^{2}(\ell^{2}+(j-\ell)^{2})}{2L^{2}}|r-s|}\,\dif s\,\dif r=2\int^{t}_{0}\int^{t}_{r}\Psi^{\eps}_{r}\Psi^{\eps}_{s}\e^{-\frac{\pi^{2}(\ell^{2}+(j-\ell)^{2})}{2L^{2}}(s-r)}\,\dif s\,\dif r\\
 & =2\int^{t}_{0}\int^{t-r}_{0}\Psi^{\eps}_{r}\Psi^{\eps}_{s+r}\e^{-\frac{\pi^{2}(\ell^{2}+(j-\ell)^{2})}{2L^{2}}s}\,\dif s\,\dif r=2\int^{t}_{0}\left(\int^{t-s}_{0}\Psi^{\eps}_{r}\Psi^{\eps}_{s+r}\,\dif r\right)\e^{-\frac{\pi^{2}(\ell^{2}+(j-\ell)^{2})}{2L^{2}}s}\,\dif s\\
 & =2\int^{\infty}_{0}\left(\mathbf{1}_{s<t}\int^{t-s}_{0}\Psi^{\eps}_{r}\Psi^{\eps}_{s+r}\,\dif r\right)\e^{-\frac{\pi^{2}(\ell^{2}+(j-\ell)^{2})}{2L^{2}}s}\,\dif s.
\end{align*}
We note that
\[
\left|\int^{t-s}_{0}\Psi^{\eps}_{r}\Psi^{\eps}_{s+r}\,\dif r-t\right|\le|s|+\int^{t-s}_{0}\left|\Psi^{\eps}_{r}\Psi^{\eps}_{s+r}-1\right|\,\dif r\overset{\zcref{eq:Psiproperties}}{\le}|s|+4\eps^{2},
\]
so we can write 
\begin{equation}
I_{j,\ell}=\frac{4L^{2}}{\pi^{2}(\ell^{2}+(j-\ell)^{2})}(t+E^{(1)}_{\eps}+E^{(2)}_{j,\ell}),\label{eq:Ierr}
\end{equation}
where
\begin{equation}
\left|E^{(1)}_{\eps}\right|\le16\eps^{2}\label{eq:E1bd}
\end{equation}
and
\begin{align}
|E^{(2)}_{j,\ell}| & \le\frac{2\pi^{2}(\ell^{2}+(j-\ell)^{2})}{L^{2}}\int^{\infty}_{0}s^{2}\e^{-\frac{\pi^{2}(\ell^{2}+(j-\ell)^{2})}{2L^{2}}s}\,\dif s+\frac{\pi^{2}(\ell^{2}+(j-\ell)^{2})t}{L^{2}}\int^{\infty}_{t}\e^{-\frac{\pi^{2}(\ell^{2}+(j-\ell)^{2})}{2L^{2}}s}\,\dif s\nonumber \\
 & =\frac{8L^{4}}{\pi^{4}(\ell^{2}+(j-\ell)^{2})^{2}}+\e^{-\frac{\pi^{2}(\ell^{2}+(j-\ell)^{2})}{2L^{2}}t}\le\frac{C_{t,L}}{(\ell^{2}+(j-\ell)^{2})^{2}}\label{eq:E2bd}
\end{align}
for a constant $C_{t,L}$ depending only on $t$ and $L$.

Using \zcref{eq:prodofdeltas,eq:Ierr} in \zcref{eq:varcomp}, we
obtain 
\begin{align}
\Var & \left(\mathcal{X}^{\eps}_{\uu,\vv;0,t}\left(\rsrenorm[rsP]^{\zeta}\right)\right)\nonumber \\
 & =\frac{2L^{4}}{\pi^{2}}\sum_{j,\ell\in\mathbb{Z}}\widehat{\varphi}^{\eps}_{\uu,\vv}(j)\left(\widehat{\varphi}^{\eps}_{\uu,\vv}(-j)-\widehat{\varphi}^{\eps}_{\uu,\vv}(j-2\ell)-\widehat{\varphi}^{\eps}_{\uu,\vv}(2\ell-j)+\widehat{\varphi}^{\eps}_{\uu,\vv}(j)\right)\widehat{\Sh}^{\zeta}_{2L}(\ell)\widehat{\Sh}^{\zeta}_{2L}(j-\ell)\frac{t+E^{(1)}_{\eps}+E^{(2)}_{j,\ell}}{\ell^{2}+(j-\ell)^{2}}\nonumber \\
 & =\frac{4L^{4}}{\pi^{2}}\sum_{j,\ell\in\mathbb{Z}}\frac{\widehat{\varphi}^{\eps}_{\uu,\vv}(j)\left(\widehat{\varphi}^{\eps}_{\uu,\vv}(j)-\widehat{\varphi}^{\eps}_{\uu,\vv}(j-2\ell)\right)}{\ell^{2}+(j-\ell)^{2}}\widehat{\Sh}^{\zeta}_{2L}(\ell)\widehat{\Sh}^{\zeta}_{2L}(j-\ell)(t+E^{(1)}_{\eps}+E^{(2)}_{j,\ell})\nonumber \\
 & \xrightarrow[\zeta\to0]{}\frac{4L^{2}}{\pi^{2}}\sum_{j,\ell\in\mathbb{Z}}\frac{\widehat{\varphi}^{\eps}_{\uu,\vv}(j)\left(\widehat{\varphi}^{\eps}_{\uu,\vv}(j)-\widehat{\varphi}^{\eps}_{\uu,\vv}(j-2\ell)\right)}{\ell^{2}+(j-\ell)^{2}}(t+E^{(1)}_{\eps}+E^{(2)}_{j,\ell})\nonumber \\
\ovset{\zcref{eq:phihatepsilon}} & =\frac{1}{\pi^{2}}\sum_{j,\ell\in\mathbb{Z}}\frac{\left(\uu+(-1)^{j}\vv\right)\hat{\psi}\left(\frac{\eps j}{L}\right)\left(\left(\uu+(-1)^{j}\vv\right)\widehat{\psi}\left(\frac{\eps j}{L}\right)-\left(\uu+(-1)^{j-2\ell}\vv\right)\widehat{\psi}\left(\frac{\eps(j-2\ell)}{L}\right)\right)}{\ell^{2}+(j-\ell)^{2}}(t+E^{(1)}_{\eps}+E^{(2)}_{j,\ell})\nonumber \\
 & =\frac{1}{\pi^{2}}\sum_{j,\ell\in\mathbb{Z}}\frac{\left(\uu+(-1)^{j}\vv\right)^{2}\hat{\psi}\left(\frac{\eps j}{L}\right)\left(\widehat{\psi}\left(\frac{\eps j}{L}\right)-\widehat{\psi}\left(\frac{\eps(j-2\ell)}{L}\right)\right)}{\ell^{2}+(j-\ell)^{2}}(t+E^{(1)}_{\eps}+E^{(2)}_{j,\ell}).\label{eq:var-comp}
\end{align}
For fixed $j$ and $\ell$, we have $\widehat{\psi}(\eps j/L)-\widehat{\psi}(\eps(j-2\ell)/L)\to0$
as $\eps\to0$, so by \zcref{eq:E2bd} and dominated convergence we
have
\begin{equation}
\lim_{\eps\downarrow0}\sum_{j,\ell\in\mathbb{Z}}\frac{\left(\uu+(-1)^{j}\vv\right)^{2}\hat{\psi}\left(\frac{\eps j}{L}\right)\left(\widehat{\psi}\left(\frac{\eps j}{L}\right)-\widehat{\psi}\left(\frac{\eps(j-2\ell)}{L}\right)\right)}{\ell^{2}+(j-\ell)^{2}}E^{(2)}_{j,\ell}=0.\label{eq:time-error}
\end{equation}
On the other hand, we can compute
\begin{align}
\sum_{j,\ell\in\mathbb{Z}} & \frac{\left(\uu+(-1)^{j}\vv\right)^{2}\hat{\psi}\left(\frac{\eps j}{L}\right)\left(\widehat{\psi}\left(\frac{\eps j}{L}\right)-\widehat{\psi}\left(\frac{\eps(j-2\ell)}{L}\right)\right)}{\ell^{2}+(j-\ell)^{2}}\nonumber \\
 & =\frac{\eps^{2}}{L^{2}}\sum_{x,y\in\frac{\eps}{L}\mathbb{Z}}\frac{\left(\uu+(-1)^{\frac{L}{\eps}x}\vv\right)^{2}\hat{\psi}(x)\left(\widehat{\psi}(x)-\widehat{\psi}(x-2y)\right)}{y^{2}+(x-y)^{2}}\nonumber \\
 & =(\uu+\vv)^{2}\sum_{\substack{x\in\frac{2\eps}{L}\mathbb{Z}\\
y\in\frac{\eps}{L}\mathbb{Z}
}
}\frac{\hat{\psi}(x)\left(\widehat{\psi}(x)-\widehat{\psi}(x-2y)\right)}{y^{2}+(x-y)^{2}}+(\uu-\vv)^{2}\sum_{\substack{x\in\frac{\eps}{L}(2\mathbb{Z}+1)\\
y\in\frac{\eps}{L}\mathbb{Z}
}
}\frac{\hat{\psi}(x)\left(\widehat{\psi}(x)-\widehat{\psi}(x-2y)\right)}{y^{2}+(x-y)^{2}}\nonumber \\
 & \xrightarrow[\eps\to0]{}\frac{1}{2}[(\uu+\vv)^{2}+(\uu-\vv)^{2}]\int_{\mathbb{R}^{2}}\frac{\hat{\psi}(x)\left(\widehat{\psi}(x)-\widehat{\psi}(x-2y)\right)}{y^{2}+(x-y)^{2}}\,\dif x\,\dif y\nonumber \\
 & =(\uu^{2}+\vv^{2})\int_{\mathbb{R}^{2}}\frac{\hat{\psi}(x)\left(\widehat{\psi}(x)-\widehat{\psi}(x-2y)\right)}{y^{2}+(x-y)^{2}}\,\dif x\,\dif y,\label{eq:t-term}
\end{align}
where the limit is by the Riemann sum approximation of the integral.
Using \zcref{eq:E1bd,eq:time-error,eq:t-term} in \zcref{eq:var-comp},
we obtain \zcref{eq:redcherry-var-limit}.
\end{proof}

\begin{lem}
\label{lem:fourthmoment}We have
\begin{equation}
\adjustlimits\lim_{\eps\downarrow0}\lim_{\zeta\downarrow0}\mathbb{E}\left[\mathcal{X}^{\eps}_{\uu,\vv;0,t}\left(\rsrenorm[rsP]^{\zeta}\right)\right]^{4}=3\left(\mathbb{E}\left[\mathcal{X}^{\eps}_{\uu,\vv;0,t}\left(\rsrenorm[rsP]^{\zeta}\right)\right]^{2}\right)^{2}.\label{eq:fourmomenthypothesis}
\end{equation}
\end{lem}

\begin{proof}
We have
\begin{align}
\mathbb{E}\left[\mathcal{X}^{\eps}_{\uu,\vv;0,t}\left(\rsrenorm[rsP]^{\zeta}\right)\right]^{4} & =\iiiint_{[0,t]^{4}}\left(\prod^{4}_{i=1}\Phi^{\eps}_{s_{i}}\right)\iiiint_{[0,L]^{4}}\left(\prod^{4}_{i=1}\varphi^{\eps}_{\uu,\vv}(x_{i})\right)\mathbb{E}\left[\prod^{4}_{i=1}\rsrenorm[rsP]^{\zeta}_{s_{i}}(x_{i})\right]\,\dif\mathbf{x}\,\dif\mathbf{s}\nonumber \\
\ovset{\zcref{eq:subtract-exp}} & =\iiiint_{[0,t]^{4}}\left(\prod^{4}_{i=1}\Phi^{\eps}_{s_{i}}\right)\iiiint_{[0,L]^{4}}\left(\prod^{4}_{i=1}\varphi^{\eps}_{\uu,\vv}(x_{i})\right)\mathbb{E}\left[\prod^{4}_{i=1}\left(Y_{i}Z_{i}-\mathbb{E}[Y_{i}Z_{i}]\right)\right]\,\dif\mathbf{x}\,\dif\mathbf{s},\label{eq:YiZiformula}
\end{align}
where $\mathbf{x}=(x_{1},x_{2},x_{3},x_{4})$ and $\mathbf{s}=(s_{1},s_{2},s_{3},s_{4})$
and we have defined
\begin{equation}
Y_{i}\coloneqq Z_{i}\coloneqq\rslollipopr[rsP]^{\zeta}_{s_{i}}(x_{i})^{2}\label{eq:YiZidef}
\end{equation}
(These quantities of course depend on $\mathbf{x}$ and $\mathbf{s}$
as well.) Now when we compute expectation using the Isserlis theorem,
we must sum over all matchings of the eight symbols $\{Y_{i},Z_{i}\}^{4}_{i=1}$.
However, we can exclude any matching that pairs $Y_{i}$ to $Z_{i}$
for any $i$, since the contribution from such matchings is cancelled
by the expectation that is subtracted in \zcref{eq:YiZiformula}.
Of the remaining matchings, there are six possible partners for $Y_{1}$,
so we can assume by symmetry that $Y_{1}$ is matched to $Y_{2}$.
If $Z_{1}$ is matched to $Z_{2}$, then there are two possible resulting
matchings, but both are symmetric and so we can assume that the matching
is $\{(Y_{1},Y_{2}),(Z_{1},Z_{2}),(Y_{3},Y_{4}),(Z_{3},Z_{4})\}$.
The resulting contribution to \zcref{eq:YiZiformula} is
\begin{align}
12 & \iiiint_{[0,t]^{4}}\left(\prod^{4}_{i=1}\Phi^{\eps}_{s_{i}}\right)\iiiint_{[0,L]^{4}}\left(\prod^{4}_{i=1}\varphi^{\eps}_{\uu,\vv}(x_{i})\right)\mathbb{E}[Y_{1}Y_{2}]\mathbb{E}[Z_{1}Z_{2}]\mathbb{E}[Y_{3}Y_{4}]\mathbb{E}[Z_{3}Z_{4}]\,\dif\mathbf{x}\,\dif\mathbf{s}\nonumber \\
 & =12\left(\iint_{[0,t]^{2}}\left(\prod^{4}_{i=1}\Phi^{\eps}_{s_{i}}\right)\iint_{[0,L]^{2}}\left(\prod^{4}_{i=1}\varphi^{\eps}_{\uu,\vv}(x_{i})\right)\mathbb{E}[Y_{1}Y_{2}]\mathbb{E}[Z_{1}Z_{2}]\,\dif x_{1}\,\dif x_{2}\,\dif s_{1}\,\dif s_{2}\right)^{2}\nonumber \\
 & =3\left(\mathbb{E}\left[\mathcal{X}^{\eps}_{\uu,\vv;0,t}\left(\rsrenorm[rsP]^{\zeta}\right)\right]^{2}\right)^{2}.\label{eq:secondmomentpiece}
\end{align}

On the other hand, if $Z_{1}$ is not matched to $Z_{2}$, then the
remaining four choices for the partner of $Z_{1}$ are again symmetric,
so we can assume that $Z_{1}$ is matched to $Z_{3}$. Now $Y_{3}$
cannot be matched to $Z_{2}$ since then $Y_{4}$ would have to be
matched to $Z_{4}$ which we have already excluded, so there are two
symmetric choices for the partner of $Y_{3}$. Thus we assume that
$Y_{3}$ is matched to $Y_{4}$ and hence $Z_{2}$ is matched to $Z_{4}$.
The contribution to \zcref{eq:YiZiformula} is then 
\begin{align}
48E^{\eps,\zeta} & \coloneqq48\iiiint_{[0,t]^{4}}\left(\prod^{4}_{i=1}\Psi^{\eps}_{s_{i}}\right)F^{\eps,\zeta}(\mathbf{s})\,\dif\mathbf{s},\label{eq:Edef}
\end{align}
where
\begin{equation}
F^{\eps,\zeta}(\mathbf{s})\coloneqq\iiiint_{[0,L]^{4}}\left(\prod^{4}_{i=1}\varphi^{\eps}_{\uu,\vv}(x_{i})\right)\mathbb{E}[Y_{1}Y_{2}]\mathbb{E}[Y_{3}Y_{4}]\mathbb{E}[Z_{1}Z_{3}]\mathbb{E}[Z_{2}Z_{4}]\,\dif\mathbf{x}.\label{eq:Fdef}
\end{equation}
Given \zcref{eq:secondmomentpiece}, to prove \zcref{eq:fourmomenthypothesis}
it will suffice to show that
\begin{equation}
\adjustlimits\lim_{\eps\downarrow0}\lim_{\zeta\downarrow0}E^{\eps,\zeta}=0.\label{eq:Eto0}
\end{equation}

To prove \zcref{eq:Eto0}, we start by noting that, for $j\ne k$,
we have by \zcref{lem:EW-cov} that
\begin{equation}
\mathbb{E}[Y_{j}Y_{k}]=\mathbb{E}[Z_{j}Z_{k}]=p_{|s_{j}-s_{k}|}*\Sh^{\zeta}_{2L}(x_{j}-x_{k})-p_{|s_{j}-s_{k}|}*\Sh^{\zeta}_{2L}(x_{j}+x_{k}).\label{eq:EYjYk}
\end{equation}
From this it is clear that, if $s_{1},s_{2},s_{3},s_{4}$ are fixed
and all distinct, then 
\begin{equation}
\adjustlimits\lim_{\eps\downarrow0}\sup_{\zeta\in(0,\eps)}F^{\eps,\zeta}(\mathbf{s})=0.\label{eq:Flimit}
\end{equation}
On the other hand, from \zcref{eq:EYjYk} we see that there is a constant
$C<\infty$, depending only on $L$, such that as long as $\zeta\le L$
we have
\[
|\mathbb{E}[Y_{j}Y_{k}]|=|\mathbb{E}[Z_{j}Z_{k}]|\le C|s_{j}-s_{k}|^{-\oh}.
\]
Using this in \zcref{eq:Fdef}, we see that 
\begin{equation}
|F^{\eps,\zeta}(\mathbf{s})|\le C^{4}|s_{1}-s_{2}|^{-\oh}|s_{3}-s_{4}|^{-\oh}|s_{1}-s_{3}|^{-\oh}|s_{2}-s_{4}|^{-\oh}.\label{eq:Fbd}
\end{equation}
We claim that the right side of \zcref{eq:Fbd} is integrable on $[0,L]^{4}$.
Since it is symmetric in $s_{2},s_{3}$, it suffices to integrate
it over the domain $\{s_{2}<s_{3}\}$. We estimate using \zcref{lem:integration-lemma}
below that (with a possibly larger constant $C$)
\begin{align*}
\iint_{s_{2}<s_{3}} & \left(\iint_{[0,t]^{2}}|s_{1}-s_{2}|^{-\oh}|s_{3}-s_{4}|^{-\oh}|s_{1}-s_{3}|^{-\oh}|s_{2}-s_{4}|^{-\oh}\,\dif s_{1}\,\dif s_{4}\right)\,\dif s_{2}\,\dif s_{3}\\
 & =\iint_{s_{2}<s_{3}}\left(\int^{t}_{0}|s_{1}-s_{2}|^{-\oh}|s_{1}-s_{3}|^{-\oh}\,\dif s_{1}\right)\left(\int^{t}_{0}|s_{3}-s_{4}|^{-\oh}|s_{2}-s_{4}|^{-\oh}\,\dif s_{4}\right)\,\dif s_{2}\,\dif s_{3}\\
\ovset{\zcref{eq:integration-lemma}} & \le C\iint_{s_{2}<s_{3}}\left(1+\left|\log\frac{s_{3}-s_{2}}{s_{2}}\right|+\left|\log\frac{t-s_{3}}{s_{3}-s_{2}}\right|\right)^{2}\,\dif s_{2}\,\dif s_{4}<\infty.
\end{align*}
Therefore, the right side of \zcref{eq:Fbd} is indeed integrable,
so by \zcref{eq:Fdef}, \zcref{eq:Flimit}, and the dominated convergence
theorem, the limit \zcref{eq:Eto0} is proved.
\end{proof}

We used the following lemma.
\begin{lem}
\label{lem:integration-lemma}For any fixed $t>0$, we have a constant
$C=C(t)<\infty$ such that, whenever $0<s_{1}<s_{2}<t$, we have
\begin{equation}
\int^{t}_{0}|s-s_{1}|^{-\oh}|s-s_{2}|^{-\oh}\,\dif s\le C\left(1+\left|\log\frac{s_{2}-s_{1}}{s_{1}}\right|+\left|\log\frac{t-s_{2}}{s_{2}-s_{1}}\right|\right).\label{eq:integration-lemma}
\end{equation}
\end{lem}

\begin{proof}
We break the integral on the left side of \zcref{eq:integration-lemma}
into three parts. For $s<s_{1}$, we write
\begin{align*}
\int^{s_{1}}_{0}|s-s_{1}|^{-\oh}|s-s_{2}|^{-\oh}\,\dif s & =2\arsinh\sqrt{\frac{s_{1}}{s_{2}-s_{1}}}\lesssim1+\left|\log\frac{s_{1}}{|s_{2}-s_{1}|}\right|,
\end{align*}
and symmetrically
\[
\int^{t}_{s_{2}}|s-s_{1}|^{-\oh}|s-s_{2}|^{-\oh}\,\dif s\lesssim1+\left|\log\frac{t-s_{2}}{|s_{2}-s_{1}|}\right|.
\]
For the middle section, we simply change variables to write
\[
\int^{s_{2}}_{s_{1}}|s-s_{1}|^{-\oh}|s-s_{2}|^{-\oh}=\int^{1}_{0}|t|^{-\oh}|1-t|^{-\oh}<\infty.
\]
Combining the last three displays yields \zcref{eq:integration-lemma}.
\end{proof}

Now we can complete the proof of the main result of this subsection. 
\begin{proof}[Proof of \zcref{prop:Xredcherrylimit}]
Let $Q$ be a bounded, continuous function of finitely many $L^{2}(\mathbb{R}^{2})$
functions tested against $(\dif W_{t})_{t}$. Recalling \zcref{eq:C1zetadef,eq:Qtilderenorm},
we note that the random variable $\mathcal{X}^{\eps}_{\uu,\vv;0,t}\left(\rsrenorm[rsP]^{\zeta}\right)$
has mean $0$ and thus lies in the second homogeneous Wiener chaos.
Then we can use the fourth moment theorem of Peccati and Tudor, \cites[Thm.~1]{peccati:tudor:2005:gaussian},
to see that, given the second and fourth moment estimates proved in
\zcref{lem:redcherry-variance,lem:fourthmoment}, we have
\[
\lim_{\eps\downarrow0}\lim_{\zeta\downarrow0}\Law\left(\mathcal{X}^{\eps}_{\uu,\vv;0,t}\left(\rsrenorm[rsP]^{\zeta}\right),Q\right)=\mathcal{N}(0,t(\uu^{2}+\vv^{2})V_{\psi})\otimes\Law\left(Q\right)\qquad\text{weakly}.
\]
Since $Q$ was arbitrary, this implies \zcref{eq:Xredcherryconvindistn}.
\end{proof}

\subsection{Terms involving the boundary potential}

We start with the deterministic term.
\begin{prop}
\label{prop:bluecherrycontrib}We have
\begin{equation}
\lim_{\eps\downarrow0}\int^{L}_{0}\varphi^{\eps}_{\uu,\vv}(x)\rscherrybb[rsP]^{\eps}(x)\,\dif x=-\frac{1}{6}(\uu^{3}+\vv^{3}).\label{eq:bluecherrycontrib}
\end{equation}
\end{prop}

\begin{proof}
It follows immediately from \zcref{eq:ptdef,eq:basic-reconstruction,eq:Qpartial-hat}
that $\rslollipopb[rsP]^{\eps}(x)$ solves the PDE
\begin{equation}
0=\frac{1}{2}\Delta\rslollipopb[rsP]^{\eps}(x)+\partial_{x}\varphi^{\eps}_{\uu,\vv}(x),\qquad x\in\mathbb{R}.\label{eq:lolsatisfiesPDE}
\end{equation}
Since $\rslollipopb[rsP]^{\eps}$ is odd, $2L$-periodic, and continuous,
we also have 
\begin{equation}
\rslollipopb[rsP]^{\eps}(0)=\rslollipopb[rsP]^{\eps}(L)=0.\label{eq:lolbbc}
\end{equation}
Therefore, using \zcref{eq:lolbbc}, we see that in fact
\begin{equation}
0=\frac{1}{2}\partial_{x}\rslollipopb[rsP]^{\eps}(x)+\varphi^{\eps}_{\uu,\vv}(x)+\frac{\uu+\vv}{2L}.\label{eq:lolb1storder}
\end{equation}
Thus, we can write
\begin{align}
\int^{L}_{0}\varphi^{\eps}_{\uu,\vv}(x)\rscherrybb[rsP]^{\eps}(x)\,\dif x\ovset{\zcref{eq:Qprodderivs-hat}} & =\int^{L}_{0}\varphi^{\eps}_{\uu,\vv}(x)\rslollipopb[rsP]^{\eps}(x)^{2}\,\dif x\nonumber \\
\ovset{\zcref{eq:lolb1storder}} & =-\int^{L}_{0}\left(\partial_{x}\rslollipopb[rsP]^{\eps}(x)+\frac{\uu+\vv}{2L}\right)\rslollipopb[rsP]^{\eps}(x)^{2}\,\dif x=-\frac{\uu+\vv}{2L}\int^{L}_{0}\rslollipopb[rsP]^{\eps}(x)^{2}\,\dif x,\label{eq:computephichbb}
\end{align}
with the last identity again by \zcref{eq:lolbbc}. On the other hand,
from \zcref{eq:lolbbc,eq:lolb1storder} we can write explicitly
\begin{equation}
\rslollipopb[rsP]^{\eps}(x)=-2\int^{x}_{0}\varphi^{\eps}_{\uu,\vv}(y)\,\dif y-\frac{\uu+\vv}{L}x,\label{eq:lollipopb-explicit}
\end{equation}
which by \zcref{eq:convtodelta} means that
\[
\lim_{\eps\downarrow0}\rslollipopb[rsP]^{\eps}(x)=\uu-\frac{\uu+\vv}{L}x.
\]
Using this along with the dominated convergence theorem in \zcref{eq:computephichbb},
we obtain
\[
\lim_{\eps\downarrow0}\int^{L}_{0}\varphi^{\eps}_{\uu,\vv}(x)\rscherrybb[rsP]^{\eps}(x)\,\dif x=-\frac{\uu+\vv}{2L}\int^{L}_{0}\left(\uu-\frac{\uu+\vv}{L}x\right)^{2}\,\dif x.
\]
Evaluating the integral yields \zcref{eq:bluecherrycontrib}.
\end{proof}

\begin{prop}
\label{prop:red-blue-red-cherry}We have, for any $t\in\mathbb{R}$,
that
\begin{equation}
\adjustlimits\lim_{\eps\downarrow0}\lim_{\zeta\downarrow0}\mathbb{E}\left[\int^{L}_{0}\varphi^{\eps}_{\uu,\vv}(x)\rselkrbr[rsP]^{\eps,\zeta}_{t}(x)\,\dif x\right]=-\frac{\uu^{2}+\vv^{2}}{4}V_{\psi},\label{eq:elkrbrboundary}
\end{equation}
where $V_{\psi}$ is defined in \zcref{eq:redcherry-var-limit}.
\end{prop}

\begin{proof}
We begin by computing $\mathbb{E}\rselkrbr[rsP]^{\eps,\zeta}_{t}(x)$,
using the Fourier transform. We have
\begin{align*}
\widehat{\rsipcherryrb[rsP]}^{\eps,\zeta}_{t}(k) & \overset{\zcref{eq:HK-spacetime-Fourier}}{=}\frac{\pi\ii k}{L}\int^{t}_{-\infty}\e^{-\frac{\pi^{2}k^{2}}{2L^{2}}(t-s)}\widehat{\rscherryrb[rsP]}^{\eps,\zeta}_{s}(k)\,\dif s\overset{\zcref{eq:fourier-identities}}{=}\frac{\pi\ii k}{L}\int^{t}_{-\infty}\e^{-\frac{\pi^{2}k^{2}}{2L^{2}}(t-s)}\sum_{\ell\in\mathbb{Z}}\widehat{\rslollipopr[rsP]}^{\zeta}_{s}(\ell)\widehat{\rslollipopb[rsP]}^{\eps}(k-\ell)\,\dif s,
\end{align*}
and hence
\begin{align*}
\widehat{\rselkrbr[rsP]}^{\eps,\zeta}_{t}(j)\ovset{\zcref{eq:fourier-identities}} & =\sum_{k\in\mathbb{Z}}\widehat{\rslollipopr[rsP]}^{\zeta}_{t}(j-k)\widehat{\rsipcherryrb[rsP]}^{\eps,\zeta}_{t}(k)=\frac{\pi\ii}{L}\sum_{k\in\mathbb{Z}}k\widehat{\rslollipopr[rsP]}^{\zeta}_{t}(j-k)\int^{t}_{-\infty}\e^{-\frac{\pi^{2}k^{2}}{2L^{2}}(t-s)}\sum_{\ell\in\mathbb{Z}}\widehat{\rslollipopr[rsP]}^{\zeta}_{s}(\ell)\widehat{\rslollipopb[rsP]}^{\eps}(k-\ell)\,\dif s.
\end{align*}
Taking expectation and using \zcref{eq:EW-cov-fourier}, we get
\begin{align}
\mathbb{E}\widehat{\rselkrbr[rsP]}^{\eps,\zeta}_{t}(j) & =\frac{\pi\ii}{L}\sum_{k\in\mathbb{Z}}k\int^{t}_{-\infty}\e^{-\frac{\pi^{2}k^{2}}{2L^{2}}(t-s)}\sum_{\ell\in\mathbb{Z}}(\delta_{j-k+\ell}-\delta_{j-k-\ell})\widehat{\Sh}^{\zeta}_{2L}(j-k)\e^{-\frac{\pi^{2}(j-k)^{2}}{2L^{2}}(t-s)}\widehat{\rslollipopb[rsP]}^{\eps}(k-\ell)\,\dif s\nonumber \\
 & =\frac{2\ii L}{\pi}\sum_{k\in\mathbb{Z}}k\widehat{\Sh}^{\zeta}_{2L}(j-k)\frac{\widehat{\rslollipopb[rsP]}^{\eps}(j)-\widehat{\rslollipopb[rsP]}^{\eps}(2k-j)}{k^{2}+(j-k)^{2}}\nonumber \\
\ovset{\zcref{eq:lollipopft}} & =\frac{4L^{2}}{\pi^{2}}\sum_{k\in\mathbb{Z}}\frac{k\widehat{\Sh}^{\zeta}_{2L}(j-k)}{k^{2}+(j-k)^{2}}\left(\frac{\widehat{\varphi}^{\eps}_{\uu,\vv}(2k-j)(1-\delta_{2k-j})}{2k-j}-\frac{\widehat{\varphi}^{\eps}_{\uu,\vv}(j)(1-\delta_{j})}{j}\right).\label{eq:take-expectation}
\end{align}
For the first summand, we can take the limit as $\zeta\to0$ and obtain
using \zcref{eq:limShfourier} that
\begin{equation}
\frac{4L^{2}}{\pi^{2}}\sum_{k\in\mathbb{Z}}\frac{k(1-\delta_{2k-j})}{k^{2}+(j-k)^{2}}\widehat{\Sh}^{\zeta}_{2L}(j-k)\frac{\widehat{\varphi}^{\eps}_{\uu,\vv}(2k-j)}{2k-j}\xrightarrow[\zeta\to0]{}\frac{2L}{\pi^{2}}\sum_{k\in\mathbb{Z}}\frac{k\widehat{\varphi}^{\eps}_{\uu,\vv}(2k-j)(1-\delta_{2k-j})}{(k^{2}+(j-k)^{2})(2k-j)},\label{eq:first-term}
\end{equation}
since this series is absolutely summable. For the second summand in
\zcref{eq:take-expectation}, we first substitute $k\to k+j$ and
then symmetrize using \zcref{eq:Fourier-symmetric} to write
\begin{align}
 & -\frac{4L^{2}}{\pi^{2}}\sum_{k\in\mathbb{Z}}\frac{k}{k^{2}+(j-k)^{2}}\widehat{\Sh}^{\zeta}_{2L}(j-k)=-\frac{4L^{2}}{\pi^{2}}\sum_{k\in\mathbb{Z}}\frac{k+j}{(j+k)^{2}+k^{2}}\widehat{\Sh}^{\zeta}_{2L}(k)\nonumber \\
 & \quad=-\frac{2L^{2}}{\pi^{2}}\sum_{k\in\mathbb{Z}}\left(\frac{k+j}{(j+k)^{2}+k^{2}}+\frac{-k+j}{(j-k)^{2}+k^{2}}\right)\widehat{\Sh}^{\zeta}_{2L}(k)=-\frac{4L^{2}j^{3}}{\pi^{2}}\sum_{k\in\mathbb{Z}}\frac{\widehat{\Sh}^{\zeta}_{2L}(k)}{j^{4}+4k^{4}}\xrightarrow[\zeta\to0]{\zcref{eq:limShfourier}}-\frac{2Lj^{3}}{\pi^{2}}\sum_{k\in\mathbb{Z}}\frac{1}{j^{4}+4k^{4}}.\label{eq:second-term}
\end{align}
Using \zcref{eq:first-term,eq:second-term} in \zcref{eq:take-expectation},
we get
\begin{align}
\lim_{\zeta\downarrow0}\mathbb{E}\widehat{\rselkrbr[rsP]}^{\eps,\zeta}_{t}(j) & =\frac{2L}{\pi^{2}}\sum_{k\in\mathbb{Z}}\frac{k\widehat{\varphi}^{\eps}_{\uu,\vv}(2k-j)(1-\delta_{2k-j})}{(k^{2}+(j-k)^{2})(2k-j)}-\frac{2Lj^{2}\widehat{\varphi}^{\eps}_{\uu,\vv}(j)(1-\delta_{j})}{\pi^{2}}\sum_{k\in\mathbb{Z}}\frac{1}{j^{4}+4k^{4}}.\label{eq:eval-limit}
\end{align}
In particular, taking $j=0$, we have
\begin{equation}
\lim_{\zeta\downarrow0}\mathbb{E}\widehat{\rselkrbr[rsP]}^{\eps,\zeta}_{t}(0)=\frac{L}{2\pi^{2}}\sum_{k\in\mathbb{Z}\setminus\{0\}}\frac{\widehat{\varphi}^{\eps}_{\uu,\vv}(2k)}{k^{2}}.\label{eq:j0}
\end{equation}
Using \zcref{eq:eval-limit,eq:j0}, we can compute
\begin{align}
 & \lim_{\zeta\downarrow0}\int^{L}_{0}\varphi^{\eps}_{\uu,\vv}(x)\mathbb{E}\rselkrbr[rsP]^{\eps,\zeta}_{t}(x)\,\dif x\ovset{\zcref{eq:feven-int}}=L\sum_{j\in\mathbb{Z}}\widehat{\varphi}^{\eps}_{\uu,\vv}(j)\mathbb{E}\widehat{\rselkrbr[rsP]}^{\eps,\zeta}_{t}(j)\nonumber \\
 & =\frac{L^{2}}{2\pi^{2}}\widehat{\varphi}^{\eps}_{\uu,\vv}(0)\sum_{k\in\mathbb{Z}\setminus\{0\}}\frac{\widehat{\varphi}^{\eps}_{\uu,\vv}(2k)}{k^{2}}+\frac{2L^{2}}{\pi^{2}}\sum_{j\in\mathbb{Z}\setminus\{0\}}\widehat{\varphi}^{\eps}_{\uu,\vv}(j)\sum_{k\in\mathbb{Z}}\left(\frac{k\widehat{\varphi}^{\eps}_{\uu,\vv}(2k-j)(1-\delta_{2k-j})}{(k^{2}+(j-k)^{2})(2k-j)}-\widehat{\varphi}^{\eps}_{\uu,\vv}(j)\frac{j^{2}}{j^{4}+4k^{4}}\right).\label{eq:eval-integral}
\end{align}
We develop the first part of the last sum as
\begin{align}
\frac{2L^{2}}{\pi^{2}} & \sum_{\substack{j,k\in\mathbb{Z}\\
j\ne0
}
}\widehat{\varphi}^{\eps}_{\uu,\vv}(j)\frac{k}{2k-j}\cdot\frac{\widehat{\varphi}^{\eps}_{\uu,\vv}(2k-j)(1-\delta_{2k-j})}{k^{2}+(j-k)^{2}}=\frac{2L^{2}}{\pi^{2}}\sum_{\substack{j,k\in\mathbb{Z}\\
|j|\ne|k|
}
}\widehat{\varphi}^{\eps}_{\uu,\vv}(j+k)\widehat{\varphi}^{\eps}_{\uu,\vv}(k-j)\frac{k}{(k^{2}+j^{2})(k-j)}\nonumber \\
 & =\frac{L^{2}}{\pi^{2}}\sum_{\substack{j,k\in\mathbb{Z}\\
|j|\ne|k|
}
}\widehat{\varphi}^{\eps}_{\uu,\vv}(j+k)\widehat{\varphi}^{\eps}_{\uu,\vv}(k-j)\frac{k}{k^{2}+j^{2}}\left(\frac{1}{k-j}+\frac{1}{k+j}\right)\nonumber \\
 & =\frac{2L^{2}}{\pi^{2}}\sum_{\substack{j,k\in\mathbb{Z}\\
|j|\ne|k|
}
}\widehat{\varphi}^{\eps}_{\uu,\vv}(j+k)\widehat{\varphi}^{\eps}_{\uu,\vv}(k-j)\frac{k^{2}}{(k^{2}+j^{2})(k^{2}-j^{2})}=\frac{L^{2}}{\pi^{2}}\sum_{\substack{j,k\in\mathbb{Z}\\
|j|\ne|k|
}
}\frac{\widehat{\varphi}^{\eps}_{\uu,\vv}(j+k)\widehat{\varphi}^{\eps}_{\uu,\vv}(k-j)}{k^{2}+j^{2}}\nonumber \\
 & =\frac{L^{2}}{\pi^{2}}\sum_{k\in\mathbb{Z}}\sum_{j\in\mathbb{Z}\setminus\{0,2k\}}\frac{\widehat{\varphi}^{\eps}_{\uu,\vv}(j)\widehat{\varphi}^{\eps}_{\uu,\vv}(2k-j)}{k^{2}+(j-k)^{2}}\nonumber \\
 & =\frac{L^{2}}{\pi^{2}}\sum_{j\in\mathbb{Z}\setminus\{0\}}\sum_{k\in\mathbb{Z}}\frac{\widehat{\varphi}^{\eps}_{\uu,\vv}(j)\widehat{\varphi}^{\eps}_{\uu,\vv}(2k-j)}{k^{2}+(j-k)^{2}}-\frac{L^{2}}{\pi^{2}}\widehat{\varphi}^{\eps}_{\uu,\vv}(0)\sum_{k\in\mathbb{Z}\setminus\{0\}}\frac{\widehat{\varphi}^{\eps}_{\uu,\vv}(2k)}{2k^{2}}.\label{eq:first-part-of-sum}
\end{align}
On the other hand, we can write
\begin{equation}
\sum_{k\in\mathbb{Z}}\left(\frac{2j^{2}}{j^{4}+4k^{4}}-\frac{1}{k^{2}+(j-k)^{2}}\right)=\sum_{k\in\mathbb{Z}}\frac{j^{2}-2k^{2}-2jk}{j^{4}+4k^{4}}=0,\label{eq:second-part-of-sum}
\end{equation}
where in the last identity we used \zcref{eq:seriesiszero} and symmetry.
Using \zcref{eq:first-part-of-sum,eq:second-part-of-sum} in \zcref{eq:eval-integral},
we get
\begin{align*}
\lim_{\zeta\downarrow0}\int^{L}_{0}\varphi^{\eps}_{\uu,\vv}(x)\mathbb{E}\rselkrbr[rsP]^{\eps,\zeta}_{t}(x)\,\dif x & =-\frac{L^{2}}{\pi^{2}}\sum_{j\in\mathbb{Z}\setminus\{0\}}\sum_{k\in\mathbb{Z}}\frac{\widehat{\varphi}^{\eps}_{\uu,\vv}(j)\left(\widehat{\varphi}^{\eps}_{\uu,\vv}(j)-\widehat{\varphi}^{\eps}_{\uu,\vv}(2k-j)\right)}{k^{2}+(j-k)^{2}}\\
\ovset{\zcref{eq:phihatepsilon}} & =-\frac{1}{4\pi^{2}}\sum_{j\in\mathbb{Z}\setminus\{0\}}\left(\uu+(-1)^{j}\vv\right)^{2}\sum_{k\in\mathbb{Z}}\frac{\widehat{\psi}(\eps j/L)\left(\widehat{\psi}(\eps j/L)-\widehat{\psi}(\eps(j-2k)/L)\right)}{k^{2}+(j-k)^{2}}\\
 & =-\frac{\eps^{2}}{4\pi^{2}L^{2}}\left(\uu+\vv\right)^{2}\sum_{j\in2\mathbb{Z}\setminus\{0\}}\sum_{k\in\mathbb{Z}}\frac{\widehat{\psi}(\eps j/L)\left(\widehat{\psi}(\eps j)-\widehat{\psi}(\eps(j-2k)/L)\right)}{(\eps k/L)^{2}+(\eps(j-k)/L)^{2}}\\
 & \qquad\qquad-\frac{\eps^{2}}{4\pi^{2}L^{2}}\left(\uu-\vv\right)^{2}\sum_{j\in2\mathbb{Z}+1}\sum_{k\in\mathbb{Z}}\frac{\widehat{\psi}(\eps j/L)\left(\widehat{\psi}(\eps j/L)-\widehat{\psi}(\eps(j-2k)/L)\right)}{(\eps k)^{2}+(\eps(j-k)/L)^{2}}\\
 & \xrightarrow[\eps\downarrow0]{}-\frac{\uu^{2}+\vv^{2}}{4\pi^{2}}\iint_{\mathbb{R}^{2}}\frac{\widehat{\psi}(x)\left(\widehat{\psi}(x)-\widehat{\psi}(x-2y)\right)}{y^{2}+(x-y)^{2}}\,\dif x\,\dif y\overset{\zcref{eq:Vpsidef}}{=}-\frac{\uu^{2}+\vv^{2}}{4}V_{\psi},
\end{align*}
with the limit by the Riemann sum approximation.
\end{proof}

\subsection{Terms in the fourth chaos}

We finally turn to the trees $\rsmooserrrr$ and $\rscandelabrarrrr$,
whose realizations live in the fourth Wiener chaos. In this setting,
we expect both $\mathbb{E}\left[\rselkrenormrr[rsP]^{\eps,\zeta}_{t}(x)\right]$
and $\mathbb{E}\left[\rscherryrenormrenorm[rsP]^{\eps,\zeta}_{t}(x)\right]$
to diverge logarithmically as $\zeta\to0$, but in fact these logarithmic
divergences cancel when we consider the sum $\mathbb{E}\left[\rselkrenormrr[rsP]^{\eps,\zeta}_{t}(x)+4\rscherryrenormrenorm[rsP]^{\eps,\zeta}_{t}(x)\right]$.
See \cites[Lem.~6.4]{hairer:2013:solving} for the corresponding
situation in the periodic case. To analyze these terms, we follow
the proof strategy of \cites[Lem.~6.5]{hairer:2013:solving}, working
in the Fourier domain. A similar result is also established in the
proof of \cite[Thm.~6.5]{hairer:quastel:2018:class}. Since we want
to actually compute the limiting behavior of $\mathbb{E}\left[\rselkrenormrr[rsP]^{\eps,\zeta}_{t}(x)+4\rscherryrenormrenorm[rsP]^{\eps,\zeta}_{t}(x)\right]$
exactly, rather than just up to a finite constant.  we need to proceed
more precisely than in the proof of \cites[Thm.~6.5]{hairer:quastel:2018:class}.
The proof in our setting is also more complicated than that of \cites[Lem.~6.5]{hairer:2013:solving}
because the expectations $\mathbb{E}\left[\rselkrenormrr[rsP]^{\eps,\zeta}_{t}(x)\right]$
and $\mathbb{E}\left[\rscherryrenormrenorm[rsP]^{\eps,\zeta}_{t}(x)\right]$
depend on $x$. Thus, we need to study $\mathbb{E}\left[\rselkrenormrr[rsP]^{\eps,\zeta}_{t}(x)+4\rscherryrenormrenorm[rsP]^{\eps,\zeta}_{t}(x)\right]$really
as a function of $x$, not simply as a number as in the spatially
homogeneous setting. In the Fourier domain, this means that we have
to compute all of the Fourier coefficients, rather than just that
of the zero-frequency mode.

The main results of this section are \zcref{prop:boundary-value,prop:sumofcancellinglogs}
in \zcref{subsec:Conclusion} below. To work towards them, we first
perform the calculations of the Fourier coefficients of $\rscandelabrarrrr[rsP]^{\zeta}_{t}$
in \zcref{subsec:candelabra-fourier} and of $\rsmooserrrr[rsP]^{\zeta}_{t}$
in \zcref{subsec:moose-fourier}. 

\subsubsection{\label{subsec:candelabra-fourier}Fourier coefficients of $\protect\rscandelabrarrrr[rsP]^{\zeta}_{t}$}
\begin{lem}
We have, for each $n\in\mathbb{Z}$, that
\[
\mathbb{E}\left[\widehat{\rscherryrenormrenorm[rsP]}^{\zeta}_{t}(2n+1)\right]=0
\]
and
\begin{equation}
\mathbb{E}\left[\widehat{\rscherryrenormrenorm[rsP]}^{\zeta}_{t}(2n)\right]=-\frac{16L^{2}}{\pi^{2}}\sum_{k\in\mathbb{Z}\setminus\{0,2n\}}\frac{k(2n-k)}{k^{2}+(2n-k)^{2}}\sum_{\ell\in\mathbb{Z}}\widehat{\Sh}^{\zeta}_{2L}(\ell)\widehat{\Sh}^{\zeta}_{2L}(k-\ell)\frac{\delta_{n}-\delta_{n-\ell}-\delta_{k-\ell-n}+\delta_{k-n}}{k^{2}+\ell^{2}+(k-\ell)^{2}}.\label{eq:Ecandelabra-Fourier}
\end{equation}
In particular,
\begin{equation}
\mathbb{E}\left[\widehat{\rscherryrenormrenorm[rsP]}^{\zeta}_{t}(0)\right]=\frac{8L^{2}}{\pi^{2}}\sum_{k\in\mathbb{Z}\setminus\{0\}}\sum_{\ell\in\mathbb{Z}\setminus\{0,k\}}\frac{\widehat{\Sh}^{\zeta}_{2L}(\ell)\widehat{\Sh}^{\zeta}_{2L}(k-\ell)}{k^{2}+\ell^{2}+(k-\ell)^{2}}.\label{eq:Ecandelabra-Fourier-0}
\end{equation}
\end{lem}

\begin{proof}
Using \zcref{eq:Qpartial-hat,eq:fourier-identities,eq:HK-spacetime-Fourier},
we can write
\[
\widehat{\rsipcherryrr[rsP]}^{\zeta}_{t}(k)=\frac{\pi\ii k}{L}\int^{t}_{-\infty}\e^{-\frac{\pi^{2}k^{2}}{2L^{2}}(t-s)}\widehat{\rscherryrr[rsP]}^{\zeta}_{s}(k)\,\dif s=\frac{\pi\ii k}{L}\int^{t}_{-\infty}\e^{-\frac{\pi^{2}k^{2}}{2L^{2}}(t-s)}\sum_{\ell\in\mathbb{Z}}\widehat{\rslollipopr[rsP]}^{\zeta}_{s}(\ell)\widehat{\rslollipopr[rsP]}^{\zeta}_{s}(k-\ell)\,\dif s.
\]
Using this and recalling \zcref{eq:Qprodderivs-hat,eq:subtract-exp},
we see that
\begin{align}
\mathbb{E} & \left[\widehat{\rscherryrenormrenorm[rsP]}^{\zeta}_{t}(j)\right]=\sum_{k\in\mathbb{Z}}\Cov\left(\widehat{\rsipcherryrr[rsP]}^{\zeta}_{t}(k),\widehat{\rsipcherryrr[rsP]}^{\zeta}_{t}(j-k)\right)\nonumber \\
 & =-\frac{\pi^{2}}{L^{2}}\sum_{k\in\mathbb{Z}\setminus\{0,j\}}k(j-k)\iint_{(-\infty,t]^{2}}\e^{-\frac{\pi^{2}k^{2}}{2L^{2}}(t-s)-\frac{\pi^{2}(j-k)^{2}}{2L^{2}}(t-q)}\nonumber \\
 & \qquad\qquad\qquad\qquad\qquad\times\sum_{\ell,m\in\mathbb{Z}}\Cov\left(\widehat{\rslollipopr[rsP]}^{\zeta}_{s}(\ell)\widehat{\rslollipopr[rsP]}^{\zeta}_{s}(k-\ell),\widehat{\rslollipopr[rsP]}^{\zeta}_{q}(m)\widehat{\rslollipopr[rsP]}^{\zeta}_{q}(j-k-m)\right)\,\dif s\,\dif q\nonumber \\
 & =-\frac{2\pi^{2}}{L^{2}}\sum_{k\in\mathbb{Z}\setminus\{0,j\}}k(j-k)\iint_{(-\infty,t]^{2}}\e^{-\frac{\pi^{2}k^{2}}{2L^{2}}(t-s)-\frac{\pi^{2}(j-k)^{2}}{2L^{2}}(t-q)}\nonumber \\
 & \qquad\qquad\qquad\qquad\qquad\times\sum_{\ell,m\in\mathbb{Z}}\mathbb{E}\left[\widehat{\rslollipopr[rsP]}^{\zeta}_{s}(\ell)\widehat{\rslollipopr[rsP]}^{\zeta}_{s}(m)\right]\mathbb{E}\left[\widehat{\rslollipopr[rsP]}^{\zeta}_{q}(k-\ell)\widehat{\rslollipopr[rsP]}^{\zeta}_{q}(j-k-m)\right]\,\dif s\,\dif q,\label{eq:apply-isserlis-candelabra}
\end{align}
where in the last identity we used the Isserlis theorem and the symmetry
of the sum under the exchange $m\leftrightarrow j-k-m$. Then we can
use \zcref{lem:EW-cov-fourier} to compute
\begin{align*}
\mathbb{E} & \left[\widehat{\rslollipopr[rsP]}^{\zeta}_{s}(\ell)\widehat{\rslollipopr[rsP]}^{\zeta}_{s}(m)\right]\mathbb{E}\left[\widehat{\rslollipopr[rsP]}^{\zeta}_{q}(k-\ell)\widehat{\rslollipopr[rsP]}^{\zeta}_{q}(j-k-m)\right]\\
 & =\left(\delta_{\ell+m}-\delta_{\ell-m=0}\right)\left(\delta_{-\ell+j-m}-\delta_{2k-\ell-j+m}\right)\widehat{\Sh}^{\zeta}_{2L}(\ell)\widehat{\Sh}^{\zeta}_{2L}(k-\ell)\e^{-\frac{\pi^{2}|s-q|}{2L^{2}}(\ell^{2}+(k-\ell)^{2})}\\
 & =\left(\delta_{\ell+m}\delta_{j}-\delta_{\ell-m}\delta_{-2\ell+j}-\delta_{\ell+m}\delta_{2k-2\ell-j}+\delta_{\ell-m}\delta_{2k-j}\right)\widehat{\Sh}^{\zeta}_{2L}(\ell)\widehat{\Sh}^{\zeta}_{2L}(k-\ell)\e^{-\frac{\pi^{2}|s-q|}{2L^{2}}(\ell^{2}+(k-\ell)^{2})}.
\end{align*}
Using this in \zcref{eq:apply-isserlis-candelabra}, we get
\begin{align*}
\mathbb{E}\left[\widehat{\rscherryrenormrenorm[rsP]}^{\zeta}_{t}(j)\right] & =-\frac{2\pi^{2}}{L^{2}}\sum_{k\in\mathbb{Z}\setminus\{0,j\}}k(j-k)\sum_{\ell\in\mathbb{Z}}\widehat{\Sh}^{\zeta}_{2L}(\ell)\widehat{\Sh}^{\zeta}_{2L}(k-\ell)\left(\delta_{j}-\delta_{-2\ell+j}-\delta_{2k-2\ell-j}+\delta_{2k-j}\right)\\
 & \qquad\times\iint_{(-\infty,t]^{2}}\exp\left\{ -\frac{\pi^{2}}{2L^{2}}\left(-k^{2}(t-s)-(j-k)^{2}(t-q)-|s-q|(\ell^{2}+(k-\ell)^{2})\right)\right\} \,\dif s\,\dif q\\
 & =-\frac{16L^{2}}{\pi^{2}}\sum_{k\in\mathbb{Z}\setminus\{0,j\}}\frac{k(j-k)}{k^{2}+(j-k)^{2}}\sum_{\ell\in\mathbb{Z}}\widehat{\Sh}^{\zeta}_{2L}(\ell)\widehat{\Sh}^{\zeta}_{2L}(k-\ell)\frac{\delta_{j}-\delta_{-2\ell+j}-\delta_{2k-2\ell-j}+\delta_{2k-j}}{k^{2}+\ell^{2}+(k-\ell)^{2}},
\end{align*}
where for the last identity we simply evaluated the integral, which
is simply an exponential in the integration variables. It is clear
that this expression is nonzero only when $j$ is even, and setting
$j=2n$ we get \zcref{eq:Ecandelabra-Fourier}.

For the last identity in the statement, we simply plug $n=0$ into
\zcref{eq:Ecandelabra-Fourier} to obtain
\begin{align}
\mathbb{E}\left[\widehat{\rscherryrenormrenorm[rsP]}^{\zeta}_{t}(0)\right] & =\frac{8L^{2}}{\pi^{2}}\sum_{k\in\mathbb{Z}\setminus\{0\}}\sum_{\ell\in\mathbb{Z}}\widehat{\Sh}^{\zeta}_{2L}(\ell)\widehat{\Sh}^{\zeta}_{2L}(k-\ell)\frac{1-\delta_{\ell}-\delta_{k-\ell}}{k^{2}+\ell^{2}+(k-\ell)^{2}}.\label{eq:pluginj=00003D0}
\end{align}
We note that, since $k\ne0$, we have
\[
1-\delta_{\ell}-\delta_{k-\ell}=\begin{cases}
0, & \ell\in\{0,k\};\\
1, & \text{otherwise,}
\end{cases}
\]
so in fact \zcref{eq:pluginj=00003D0} becomes \zcref{eq:Ecandelabra-Fourier-0}.
\end{proof}

\begin{cor}
We have an absolute constant $C<\infty$ such that
\begin{equation}
\left|\mathbb{E}\left[\widehat{\rscherryrenormrenorm[rsP]}^{\zeta}_{t}(2n)\right]\right|\le\frac{C}{|n|}\qquad\text{for all }\zeta>0\text{ and }n\in\mathbb{Z}\setminus\{0\}\label{eq:candelabrabound}
\end{equation}
and moreover that, for each $n\in\mathbb{Z}\setminus\{0\}$,
\begin{equation}
\lim_{\zeta\downarrow0}\mathbb{E}\left[\widehat{\rscherryrenormrenorm[rsP]}^{\zeta}_{t}(2n)\right]=\frac{1}{\pi^{2}}\sum_{k\in\mathbb{Z}}\frac{n^{2}-3k^{2}}{n^{4}+n^{2}k^{2}+k^{4}}.\label{eq:Ecandelabralimit}
\end{equation}
\end{cor}

\begin{proof}
Making the change of variables $k\mapsto k+\ell$ in \zcref{eq:Ecandelabra-Fourier}
we obtain for $n\ne0$ that
\begin{align*}
\mathbb{E}\left[\widehat{\rscherryrenormrenorm[rsP]}^{\zeta}_{t}(2n)\right] & =-\frac{16L^{2}}{\pi^{2}}\sum_{k,\ell\in\mathbb{Z}}\widehat{\Sh}^{\zeta}_{2L}(\ell)\widehat{\Sh}^{\zeta}_{2L}(k)\frac{(k+\ell)(2n-k-\ell)}{(k+\ell)^{2}+(2n-k-\ell)^{2}}\cdot\frac{-\delta_{n-\ell}-\delta_{k-n}+\delta_{k+\ell-n}}{(k+\ell)^{2}+\ell^{2}+k^{2}}.
\end{align*}
The right side is symmetric in $k$ and $\ell$, so in fact we get
\[
\mathbb{E}\left[\widehat{\rscherryrenormrenorm[rsP]}^{\zeta}_{t}(2n)\right]=S^{\zeta}_{1}(n)+S^{\zeta}_{2}(n),
\]
where
\begin{align*}
S^{\zeta}_{1}(n) & \coloneqq\frac{32L^{2}}{\pi^{2}}\sum_{k\in\mathbb{Z}}\widehat{\Sh}^{\zeta}_{2L}(n)\widehat{\Sh}^{\zeta}_{2L}(k)\frac{(k+n)(n-k)}{\left((n+k)^{2}+(n-k)^{2}\right)\left((n+k)^{2}+n^{2}+k^{2}\right)}\\
 & =\frac{16L^{2}}{\pi^{2}}\sum_{k\in\mathbb{Z}}\widehat{\Sh}^{\zeta}_{2L}(n)\widehat{\Sh}^{\zeta}_{2L}(k)\frac{n^{2}-k^{2}}{\left(n^{2}+k^{2}\right)\left((n+k)^{2}+n^{2}+k^{2}\right)}\\
 & =\frac{8L^{2}}{\pi^{2}}\sum_{k\in\mathbb{Z}}\widehat{\Sh}^{\zeta}_{2L}(n)\widehat{\Sh}^{\zeta}_{2L}(k)\frac{n^{2}-k^{2}}{n^{2}+k^{2}}\left(\frac{1}{(n+k)^{2}+n^{2}+k^{2}}+\frac{1}{(n-k)^{2}+n^{2}+k^{2}}\right)\\
 & =\frac{8L^{2}}{\pi^{2}}\sum_{k\in\mathbb{Z}}\widehat{\Sh}^{\zeta}_{2L}(n)\widehat{\Sh}^{\zeta}_{2L}(k)\frac{n^{2}-k^{2}}{n^{4}+n^{2}k^{2}+k^{4}}
\end{align*}
and
\[
S^{\zeta}_{2}(n)\coloneqq-\frac{8L^{2}}{\pi^{2}}\sum_{k\in\mathbb{Z}}\frac{\widehat{\Sh}^{\zeta}_{2L}(n-k)\widehat{\Sh}^{\zeta}_{2L}(k)}{n^{2}+(n-k)^{2}+k^{2}}.
\]
From these expressions it is straightforward to see that \zcref{eq:candelabrabound}
holds, and moreover (recalling \zcref{eq:limShfourier}) that
\begin{align*}
\lim_{\zeta\downarrow0}\mathbb{E}\left[\widehat{\rscherryrenormrenorm[rsP]}^{\zeta}_{t}(2n)\right] & =\frac{2}{\pi^{2}}\sum_{k\in\mathbb{Z}}\frac{n^{2}-k^{2}}{n^{4}+n^{2}k^{2}+k^{4}}-\frac{2}{\pi^{2}}\sum_{k\in\mathbb{Z}}\frac{1}{n^{2}+(n-k)^{2}+k^{2}}\\
 & =\frac{2}{\pi^{2}}\sum_{k\in\mathbb{Z}}\frac{n^{2}-k^{2}}{n^{4}+n^{2}k^{2}+k^{4}}-\frac{1}{\pi^{2}}\sum_{k\in\mathbb{Z}}\frac{n^{2}+k^{2}}{n^{4}+n^{2}k^{2}+k^{4}}\\
 & =\frac{1}{\pi^{2}}\sum_{k\in\mathbb{Z}}\frac{n^{2}-3k^{2}}{n^{4}+n^{2}k^{2}+k^{4}},
\end{align*}
which is \zcref{eq:Ecandelabralimit}.
\end{proof}

\subsubsection{\label{subsec:moose-fourier}Fourier coefficients of $\protect\rsmooserrrr[rsP]^{\zeta}_{t}$}
\begin{lem}
\label{lem:moose-exp-fourier}For each $n\in\mathbb{Z}$, we have
\[
\mathbb{E}\left[\widehat{\rselkrenormrr[rsP]}^{\zeta}_{t}(2n+1)\right]=0
\]
and
\begin{equation}
\mathbb{E}\left[\widehat{\rselkrenormrr[rsP]}^{\zeta}_{t}(2n)\right]=-\frac{8L^{2}}{\pi^{2}}\sum_{k,\ell\in\mathbb{Z}\setminus\{0\}}\frac{k\ell}{k^{2}+(2n-k)^{2}}\widehat{\Sh}^{\zeta}_{2L}(2n-k)\widehat{\Sh}^{\zeta}_{2L}(k-\ell)\frac{\delta_{n}-\delta_{n-k}-\delta_{k-\ell-n}+\delta_{\ell-n}}{\ell^{2}+(2n-k)^{2}+(k-\ell)^{2}}.\label{eq:moose-exp-fourier}
\end{equation}
In particular, we have
\begin{equation}
\mathbb{E}\left[\widehat{\rselkrenormrr[rsP]}^{\zeta}_{t}(0)\right]=-\frac{4L^{2}}{\pi^{2}}\sum_{k\in\mathbb{Z}\setminus\{0\}}\sum_{\ell\in\mathbb{Z}\setminus\{0,k\}}\frac{\ell}{k}\cdot\frac{\widehat{\Sh}^{\zeta}_{2L}(k)\widehat{\Sh}^{\zeta}_{2L}(k-\ell)}{\ell^{2}+k^{2}+(k-\ell)^{2}},\label{eq:moose-exp-fourier-0}
\end{equation}
\end{lem}

\begin{proof}
We start by writing
\begin{align*}
\widehat{\rscherryrenormr[rsP]}^{\zeta}_{t}(k) & =\sum_{\ell\in\mathbb{Z}}\widehat{\rsiprenorm[rsP]}^{\zeta}_{t}(\ell)\widehat{\rslollipopr[rsP]}^{\zeta}_{t}(k-\ell)\\
 & =\frac{\pi\ii}{L}\sum_{\ell\in\mathbb{Z}}\ell\widehat{\rslollipopr[rsP]}^{\zeta}_{t}(k-\ell)\int^{t}_{-\infty}\e^{-\frac{\pi^{2}\ell^{2}}{2L^{2}}(t-r)}\sum_{m\in\mathbb{Z}}\left(\widehat{\rslollipopr[rsP]}^{\zeta}_{r}(m)\widehat{\rslollipopr[rsP]}^{\zeta}_{r}(\ell-m)-\mathbb{E}\left[\widehat{\rslollipopr[rsP]}^{\zeta}_{r}(m)\widehat{\rslollipopr[rsP]}^{\zeta}_{r}(\ell-m)\right]\right)\,\dif r.
\end{align*}
Therefore, we have
\begin{align*}
\widehat{\rselkrenormrr[rsP]}^{\zeta}_{t}(j) & =\frac{\pi\ii}{L}\sum_{k\in\mathbb{Z}}k\widehat{\rslollipopr[rsP]}^{\zeta}_{t}(j-k)\int^{t}_{-\infty}\e^{-\frac{\pi^{2}k^{2}}{2L^{2}}(t-s)}\widehat{\rscherryrenormr[rsP]}^{\zeta}_{s}(k)\,\dif s\\
 & =-\frac{\pi^{2}}{L^{2}}\sum_{k,\ell,m\in\mathbb{Z}}k\ell\int^{t}_{-\infty}\int^{s}_{-\infty}\e^{-\frac{\pi^{2}k^{2}}{2L^{2}}(t-s)-\frac{\pi^{2}\ell^{2}}{2L^{2}}(s-r)}\widehat{\rslollipopr[rsP]}^{\zeta}_{t}(j-k)\widehat{\rslollipopr[rsP]}^{\zeta}_{s}(k-\ell)\\
 & \qquad\qquad\qquad\qquad\qquad\qquad\times\left(\widehat{\rslollipopr[rsP]}^{\zeta}_{r}(m)\widehat{\rslollipopr[rsP]}^{\zeta}_{r}(\ell-m)-\mathbb{E}\left[\widehat{\rslollipopr[rsP]}^{\zeta}_{r}(m)\widehat{\rslollipopr[rsP]}^{\zeta}_{r}(\ell-m)\right]\right)\,\dif r\,\dif s.
\end{align*}
Now we can take expectations to get
\begin{align}
\mathbb{E}\left[\widehat{\rselkrenormrr[rsP]}^{\zeta}_{t}(j)\right] & =-\frac{\pi^{2}}{L^{2}}\sum_{k,\ell,m\in\mathbb{Z}}k\ell\int^{t}_{-\infty}\int^{s}_{-\infty}\e^{-\frac{\pi^{2}k^{2}}{2L^{2}}(t-s)-\frac{\pi^{2}\ell^{2}}{2L^{2}}(s-r)}\nonumber \\
 & \qquad\qquad\qquad\times\Cov\left(\widehat{\rslollipopr[rsP]}^{\zeta}_{t}(j-k)\widehat{\rslollipopr[rsP]}^{\zeta}_{s}(k-\ell),\widehat{\rslollipopr[rsP]}^{\zeta}_{r}(m)\widehat{\rslollipopr[rsP]}^{\zeta}_{r}(\ell-m)\right)\,\dif r\,\dif s\nonumber \\
 & =-\frac{2\pi^{2}}{L^{2}}\sum_{k,\ell,m\in\mathbb{Z}}k\ell\int^{t}_{-\infty}\int^{s}_{-\infty}\e^{-\frac{\pi^{2}k^{2}}{2L^{2}}(t-s)-\frac{\pi^{2}\ell^{2}}{2L^{2}}(s-r)}\nonumber \\
 & \qquad\qquad\qquad\times\mathbb{E}\left[\widehat{\rslollipopr[rsP]}^{\zeta}_{t}(j-k)\widehat{\rslollipopr[rsP]}^{\zeta}_{r}(m)\right]\mathbb{E}\left[\widehat{\rslollipopr[rsP]}^{\zeta}_{t}(k-\ell)\widehat{\rslollipopr[rsP]}^{\zeta}_{r}(\ell-m)\right]\,\dif r\,\dif s,\label{eq:moose-expectation}
\end{align}
where in the last identity we used the Isserlis theorem and the symmetry
under the swap $m\leftrightarrow\ell-m$. Now we use \zcref{lem:EW-cov-fourier}
to compute
\begin{align*}
\mathbb{E} & \left[\widehat{\rslollipopr[rsP]}^{\zeta}_{t}(j-k)\widehat{\rslollipopr[rsP]}^{\zeta}_{r}(m)\right]\mathbb{E}\left[\widehat{\rslollipopr[rsP]}^{\zeta}_{t}(k-\ell)\widehat{\rslollipopr[rsP]}^{\zeta}_{r}(\ell-m)\right]\\
 & =\widehat{\Sh}^{\zeta}_{2L}(j-k)\widehat{\Sh}^{\zeta}_{2L}(k-\ell)\left(\delta_{j-k+m}-\delta_{j-k-m}\right)\left(\delta_{k-m}-\delta_{k+m-2\ell}\right)\e^{-\frac{\pi^{2}}{2L^{2}}\left((j-k)^{2}(t-r)+(k-\ell)^{2}(s-r)\right)}\\
 & =\widehat{\Sh}^{\zeta}_{2L}(j-k)\widehat{\Sh}^{\zeta}_{2L}(k-\ell)\left(\delta_{j}\delta_{m-k}-\delta_{j-2k}\delta_{m-k}-\delta_{2(k-\ell)-j}\delta_{k+m-2\ell}+\delta_{2\ell-j}\delta_{k+m-2\ell}\right)\\
 & \qquad\times\e^{-\frac{\pi^{2}}{2L^{2}}\left((j-k)^{2}(t-r)+(k-\ell)^{2}(s-r)\right)}.
\end{align*}
Using this in \zcref{eq:moose-expectation}, we get
\begin{align*}
\mathbb{E}\left[\widehat{\rselkrenormrr[rsP]}^{\zeta}_{t}(j)\right] & =-\frac{2\pi^{2}}{L^{2}}\sum_{k,\ell\in\mathbb{Z}\setminus\{0\}}k\ell\widehat{\Sh}^{\zeta}_{2L}(j-k)\widehat{\Sh}^{\zeta}_{2L}(k-\ell)\left(\delta_{j}-\delta_{j-2k}-\delta_{2(k-\ell)-j}+\delta_{2\ell-j}\right)\\
 & \qquad\times\int^{t}_{-\infty}\int^{s}_{-\infty}\exp\left\{ -\frac{\pi^{2}}{2L^{2}}\left(k^{2}(t-s)+\ell^{2}(s-r)+(j-k)^{2}(t-r)+(k-\ell)^{2}(s-r)\right)\right\} \,\dif r\,\dif s\\
 & =-\frac{8L^{2}}{\pi^{2}}\sum_{k,\ell\in\mathbb{Z}\setminus\{0\}}\frac{k\ell}{k^{2}+(j-k)^{2}}\widehat{\Sh}^{\zeta}_{2L}(j-k)\widehat{\Sh}^{\zeta}_{2L}(k-\ell)\frac{\delta_{j}-\delta_{j-2k}-\delta_{2(k-\ell)-j}+\delta_{2\ell-j}}{\ell^{2}+(j-k)^{2}+(k-\ell)^{2}},
\end{align*}
where for the last identity we evaluated the integral. This expression
is evidently nonzero only for even $j$, and for $j=2n$ we obtain
\zcref{eq:moose-exp-fourier}.

For \zcref{eq:moose-exp-fourier-0}, we set $n=0$ in \zcref{eq:moose-exp-fourier}
to obtain
\[
\mathbb{E}\left[\widehat{\rselkrenormrr[rsP]}^{\zeta}_{t}(0)\right]=-\frac{4L^{2}}{\pi^{2}}\sum_{k,\ell\in\mathbb{Z}\setminus\{0\}}\frac{\ell}{k}\widehat{\Sh}^{\zeta}_{2L}(k)\widehat{\Sh}^{\zeta}_{2L}(k-\ell)\frac{1-\delta_{k-\ell}}{\ell^{2}+k^{2}+(k-\ell)^{2}},
\]
and since the indicator is zero if $k=\ell$ we obtain \zcref{eq:moose-exp-fourier-0}.
\end{proof}

\begin{cor}
There is an absolute constant $C<\infty$ such that, for all $n\in\mathbb{Z}\setminus\{0\}$,
we have
\begin{equation}
\left|\mathbb{E}\left[\widehat{\rselkrenormrr[rsP]}^{\zeta}_{t}(2n)\right]\right|\le\frac{C}{|n|},\label{eq:moosebd}
\end{equation}
and moreover we have
\begin{equation}
\lim_{\zeta\to0}\mathbb{E}\left[\widehat{\rselkrenormrr[rsP]}^{\zeta}_{t}(2n)\right]=\frac{n^{2}}{2\pi^{2}}\sum_{\ell\in\mathbb{Z}}\frac{1}{\ell^{4}+\ell^{2}n^{2}+n^{4}}.\label{eq:mooselimit}
\end{equation}
\end{cor}

\begin{proof}
For $n\in\mathbb{Z}\setminus\{0\}$, we can write \zcref{eq:moose-exp-fourier}
as
\begin{align}
\mathbb{E}\left[\widehat{\rselkrenormrr[rsP]}^{\zeta}_{t}(2n)\right] & =S^{\zeta}_{3}(n)+S^{\zeta}_{4}(n)+S^{\zeta}_{5}(n),\label{eq:moose-split}
\end{align}
where 
\begin{align*}
S^{\zeta}_{3}(n) & \coloneqq\frac{4L^{2}}{\pi^{2}n}\widehat{\Sh}^{\zeta}_{2L}(n)\sum_{\ell\in\mathbb{Z}}\widehat{\Sh}^{\zeta}_{2L}(n-\ell)\frac{\ell}{\ell^{2}+n^{2}+(n-\ell)^{2}};\\
S^{\zeta}_{4}(n) & \coloneqq\frac{8L^{2}}{\pi^{2}}\sum_{k\in\mathbb{Z}}\widehat{\Sh}^{\zeta}_{2L}(2n-k)\widehat{\Sh}^{\zeta}_{2L}(n)\frac{k(k-n)}{\left((n-k)^{2}+(2n-k)^{2}+n^{2}\right)\left(k^{2}+(2n-k)^{2}\right)};\\
S^{\zeta}_{5}(n) & \coloneqq-\frac{8L^{2}}{\pi^{2}}\sum_{k\in\mathbb{Z}}\widehat{\Sh}^{\zeta}_{2L}(2n-k)\widehat{\Sh}^{\zeta}_{2L}(k-n)\frac{kn}{\left(n^{2}+(2n-k)^{2}+(k-n)^{2}\right)\left(k^{2}+(2n-k)^{2}\right)}.
\end{align*}
We consider each term in turn.

The first term, $S^{\zeta}_{3}(n)$, is the most challenging, because
the series $\sum_{\ell\in\mathbb{Z}}\frac{\ell}{\ell^{2}+n^{2}+(n-\ell)^{2}}$
is not absolutely convergent. But we can change variables $\ell\mapsto n+\ell$
and then average with the change of variables $\ell\mapsto-\ell$
to obtain
\begin{align*}
S^{\zeta}_{3}(n) & =\frac{4L^{2}}{\pi^{2}n}\widehat{\Sh}^{\zeta}_{2L}(n)\sum_{\ell\in\mathbb{Z}}\widehat{\Sh}^{\zeta}_{2L}(\ell)\frac{n+\ell}{(n+\ell)^{2}+n^{2}+\ell^{2}}\\
 & =\frac{2L^{2}}{\pi^{2}n}\widehat{\Sh}^{\zeta}_{2L}(n)\sum_{\ell\in\mathbb{Z}}\widehat{\Sh}^{\zeta}_{2L}(\ell)\left(\frac{n+\ell}{(n+\ell)^{2}+n^{2}+\ell^{2}}+\frac{n-\ell}{(n-\ell)^{2}+n^{2}+\ell^{2}}\right)\\
 & =\frac{8n^{2}L^{2}}{\pi^{2}}\widehat{\Sh}^{\zeta}_{2L}(n)\sum_{\ell\in\mathbb{Z}}\frac{\widehat{\Sh}^{\zeta}_{2L}(\ell)}{\left((n+\ell)^{2}+n^{2}+\ell^{2}\right)\left((n-\ell)^{2}+n^{2}+\ell^{2}\right)}.
\end{align*}
Thus if we define
\begin{align}
S_{3}(n) & \coloneqq\frac{2n^{2}}{\pi^{2}}\sum_{\ell\in\mathbb{Z}}\frac{1}{\left((n+\ell)^{2}+n^{2}+\ell^{2}\right)\left((n-\ell)^{2}+n^{2}+\ell^{2}\right)}=\frac{n^{2}}{2\pi^{2}}\sum_{\ell\in\mathbb{Z}}\frac{1}{\ell^{4}+\ell^{2}n^{2}+n^{4}},\label{eq:S3def}
\end{align}
then we have
\begin{equation}
|S^{\zeta}_{3}(n)|\le S_{3}(n)\le\frac{2}{\pi^{2}}\sum_{\ell\in\mathbb{Z}}\frac{1}{n^{2}+\ell^{2}}\le\frac{C}{|n|}\qquad\text{for all }\zeta>0,n\in\mathbb{Z}\setminus\{0\},\label{eq:S3zetabound}
\end{equation}
for an absolute constant $C<\infty$ (independent of $\zeta$ and
$n$), and moreover
\begin{equation}
\lim_{\zeta\downarrow0}S^{\zeta}_{3}(n)=S_{3}(n).\label{eq:S3zetalimit}
\end{equation}

Next, we write
\[
S^{\zeta}_{4}(n)\coloneqq\frac{8L^{2}}{\pi^{2}}\sum_{k\in\mathbb{Z}}\widehat{\Sh}^{\zeta}_{2L}(2n-k)\widehat{\Sh}^{\zeta}_{2L}(n)\frac{k(k-n)}{\left((n-k)^{2}+(2n-k)^{2}+n^{2}\right)\left(k^{2}+(2n-k)^{2}\right)},
\]
from which we see that 
\begin{align}
|S^{\zeta}_{4}(n)| & \le\frac{2}{\pi^{2}}\sum_{k\in\mathbb{Z}}\frac{|k(k-n)|}{\left((n-k)^{2}+(2n-k)^{2}+n^{2}\right)\left(k^{2}+(2n-k)^{2}\right)}\nonumber \\
 & \le\frac{2}{\pi^{2}}\sum_{k\in\mathbb{Z}}\frac{1}{(n-k)^{2}+(2n-k)^{2}+n^{2}}\le\frac{C}{|n|}\label{eq:S4zetabd}
\end{align}
and also
\begin{align}
\lim_{\zeta\downarrow0}S^{\zeta}_{4}(n)=S_{4}(n) & \coloneqq\frac{2}{\pi^{2}}\sum_{k\in\mathbb{Z}}\frac{k(k-n)}{\left((n-k)^{2}+(2n-k)^{2}+n^{2}\right)\left(k^{2}+(2n-k)^{2}\right)}\nonumber \\
 & =\frac{1}{\pi^{2}}\sum_{k\in\mathbb{Z}}\frac{(k+n)k}{\left(k^{2}+(n-k)^{2}+n^{2}\right)\left(n^{2}+k^{2}\right)},\label{eq:S4zetalimit}
\end{align}
where in the last identity we changed variables $k\mapsto k+n$. 

Similarly, we can estimate
\begin{equation}
|S^{\zeta}_{5}(n)|\le\frac{2}{\pi^{2}}\sum_{k\in\mathbb{Z}}\frac{|kn|}{\left(n^{2}+(2n-k)^{2}+(k-n)^{2}\right)\left(k^{2}+(2n-k)^{2}\right)}\le\frac{C}{|n|}\label{eq:S5zetabd}
\end{equation}
and
\begin{align}
\lim_{\zeta\downarrow0}S^{\zeta}_{5}(n)=S_{5}(n) & \coloneqq-\frac{2}{\pi^{2}}\sum_{k\in\mathbb{Z}}\frac{kn}{\left(n^{2}+(2n-k)^{2}+(k-n)^{2}\right)\left(k^{2}+(2n-k)^{2}\right)}\nonumber \\
 & =-\frac{1}{\pi^{2}}\sum_{k\in\mathbb{Z}}\frac{(k+n)n}{\left(n^{2}+(n-k)^{2}+k^{2}\right)\left(n^{2}+k^{2}\right)},\label{eq:S5zetalimit}
\end{align}
where again in the last identity we changed variables $k\mapsto k+n$.

We can also evaluate the sum of \zcref{eq:S4zetalimit,eq:S5zetabd}
as
\begin{align}
S_{4} & (n)+S_{5}(n)=\frac{1}{\pi^{2}}\sum_{k\in\mathbb{Z}}\frac{k^{2}-n^{2}}{\left(k^{2}+(n-k)^{2}+n^{2}\right)\left(n^{2}+k^{2}\right)}\nonumber \\
 & =\frac{1}{2\pi^{2}}\sum_{k\in\mathbb{Z}}\frac{k^{2}-n^{2}}{n^{2}+k^{2}}\left(\frac{1}{k^{2}+(n-k)^{2}+n^{2}}+\frac{1}{k^{2}+(n+k)^{2}+n^{2}}\right)=\frac{1}{2\pi^{2}}\sum_{k\in\mathbb{Z}}\frac{k^{2}-n^{2}}{k^{4}+n^{2}k^{2}+n^{4}}=0,\label{eq:S4plusS5}
\end{align}
with the last identity by \zcref{eq:seriesiszero}.

Using the bounds \zcref{eq:S3zetabound,eq:S4zetabd,eq:S5zetabd} in
\zcref{eq:moose-split} immediately gives us \zcref{eq:moosebd} by
the triangle inequality. We also see by using the limits \zcref{eq:S3zetalimit,eq:S4zetalimit,eq:S5zetalimit}
in \zcref{eq:moose-split}, and then using \zcref{eq:S3def,eq:S4plusS5},
that 
\begin{align*}
\lim_{\zeta\downarrow0}\mathbb{E}\left[\widehat{\rselkrenormrr[rsP]}^{\zeta}_{t}(2n)\right] & =S_{3}(n)+S_{4}(n)+S_{5}(n)=\frac{n^{2}}{2\pi^{2}}\sum_{\ell\in\mathbb{Z}}\frac{1}{\ell^{4}+\ell^{2}n^{2}+n^{4}},
\end{align*}
which is \zcref{eq:mooselimit}.
\end{proof}

\subsubsection{Fourier coefficients of the sum}

We now combine the results of the previous two subsections to study
the Fourier coefficients of the sum $\mathbb{E}\left[\rscherryrenormrenorm[rsP]^{\zeta}_{t}+4\rselkrenormrr[rsP]^{\zeta}_{t}\right]$.
First we consider the zero-frequency mode:
\begin{lem}
\label{lem:zeromodeminus13}We have
\begin{equation}
\lim_{\zeta\to0}\mathbb{E}\left[\widehat{\rscherryrenormrenorm[rsP]}^{\zeta}_{t}(0)+4\widehat{\rselkrenormrr[rsP]}^{\zeta}_{t}(0)\right]=-\frac{1}{3}.\label{eq:zeromodeminus13}
\end{equation}
\end{lem}

\begin{proof}
Combining \zcref{eq:Ecandelabra-Fourier-0,eq:moose-exp-fourier-0}
and then making a change of variables $\ell\leftarrow k-\ell$ in
the second sum, we obtain
\begin{align}
\frac{\pi^{2}}{8L^{2}}\mathbb{E}\left[\widehat{\rscherryrenormrenorm[rsP]}^{\zeta}_{t}(0)+4\widehat{\rselkrenormrr[rsP]}^{\zeta}_{t}(0)\right] & =\sum_{\substack{k,\ell\in\mathbb{Z}\setminus\{0\}\\
k\ne\ell
}
}\frac{\widehat{\Sh}^{\zeta}_{2L}(\ell)\widehat{\Sh}^{\zeta}_{2L}(k-\ell)}{k^{2}+\ell^{2}+(k-\ell)^{2}}-2\sum_{\substack{k,\ell\in\mathbb{Z}\setminus\{0\}\\
k\ne\ell
}
}\frac{\ell}{k}\cdot\frac{\widehat{\Sh}^{\zeta}_{2L}(k)\widehat{\Sh}^{\zeta}_{2L}(k-\ell)}{\ell^{2}+k^{2}+(k-\ell)^{2}}\nonumber \\
 & =\sum_{\substack{k,\ell\in\mathbb{Z}\\
k,\ell\ne0\text{ and }k\ne\ell
}
}\left(1-\frac{2(k-\ell)}{k}\right)\frac{\widehat{\Sh}^{\zeta}_{2L}(\ell)\widehat{\Sh}^{\zeta}_{2L}(k)}{(k-\ell)^{2}+\ell^{2}+k^{2}}\nonumber \\
 & =\sum_{k,\ell\in\mathbb{Z}\setminus\{0\}}\left(1-\frac{2(k-\ell)}{k}\right)\frac{\widehat{\Sh}^{\zeta}_{2L}(\ell)\widehat{\Sh}^{\zeta}_{2L}(k)}{(k-\ell)^{2}+\ell^{2}+k^{2}}-\frac{1}{2}\sum_{k\in\mathbb{Z}\setminus\{0\}}\frac{\widehat{\Sh}^{\zeta}_{2L}(k)^{2}}{k^{2}}.\label{eq:combine-fourierzero}
\end{align}
Exactly as in \cite[(127)]{hairer:2013:solving}, we can write
\begin{align}
 & \sum_{k,\ell\in\mathbb{Z}\setminus\{0\}}\widehat{\Sh}^{\zeta}_{2L}(\ell)\widehat{\Sh}^{\zeta}_{2L}(k)\left(\frac{1-\frac{2(k-\ell)}{k}}{(k-\ell)^{2}+\ell^{2}+k^{2}}\right)=\frac{1}{2}\sum_{k,\ell\in\mathbb{Z}\setminus\{0\}}\widehat{\Sh}^{\zeta}_{2L}(\ell)\widehat{\Sh}^{\zeta}_{2L}(k)\frac{\frac{2\ell-k}{k}+\frac{2k-\ell}{\ell}}{(k-\ell)^{2}+\ell^{2}+k^{2}}\nonumber \\
 & \ =\sum_{k,\ell\in\mathbb{Z}\setminus\{0\}}\widehat{\Sh}^{\zeta}_{2L}(\ell)\widehat{\Sh}^{\zeta}_{2L}(k)\frac{\ell^{2}-k\ell+k^{2}}{k\ell[(k-\ell)^{2}+\ell^{2}+k^{2}]}=\sum_{k,\ell\in\mathbb{Z}\setminus\{0\}}\frac{\widehat{\Sh}^{\zeta}_{2L}(\ell)\widehat{\Sh}^{\zeta}_{2L}(k)}{k\ell}=\left(\sum_{k\in\mathbb{Z}\setminus\{0\}}\frac{\widehat{\Sh}^{\zeta}_{2L}(k)}{k}\right)^{2}=0,\label{eq:first-cancellation}
\end{align}
where in the first identity we used the symmetry of the sum under
exchanging $(k,\ell)\leftrightarrow(-\ell,-k)$ and in the last identity
we recalled \zcref{eq:Fourier-symmetric}. Using \zcref{eq:first-cancellation}
in \zcref{eq:combine-fourierzero}, we get
\[
\mathbb{E}\left[\widehat{\rscherryrenormrenorm[rsP]}^{\zeta}_{t}(0)+4\widehat{\rselkrenormrr[rsP]}^{\zeta}_{t}(0)\right]=-\frac{4L^{2}}{\pi^{2}}\sum_{k\in\mathbb{Z}\setminus\{0\}}\frac{\widehat{\Sh}^{\zeta}_{2L}(k)^{2}}{k^{2}},
\]
and hence
\[
\lim_{\zeta\to0}\mathbb{E}\left[\widehat{\rscherryrenormrenorm[rsP]}^{\zeta}_{t}(0)+4\widehat{\rselkrenormrr[rsP]}^{\zeta}_{t}(0)\right]=-\frac{4L^{2}}{\pi^{2}}\lim_{\zeta\to0}\sum_{k\in\mathbb{Z}\setminus\{0\}}\frac{\widehat{\Sh}^{\zeta}_{2L}(k)^{2}}{k^{2}}\overset{\zcref{eq:limShfourier}}{=}-\frac{1}{\pi^{2}}\sum_{k\in\mathbb{Z}\setminus\{0\}}\frac{1}{k^{2}}=-\frac{1}{3},
\]
which is \zcref{eq:zeromodeminus13}.
\end{proof}

Now, we consider the higher modes:
\begin{prop}
There is an absolute constant $C<\infty$ such that
\begin{equation}
\left|\mathbb{E}\left[\widehat{\rscherryrenormrenorm[rsP]}^{\zeta}_{t}(2n)+4\widehat{\rselkrenormrr[rsP]}^{\zeta}_{t}(2n)\right]\right|\le\frac{C}{|n|}\qquad\text{for all }\zeta>0\text{ and }n\in\mathbb{Z}\setminus\{0\}.\label{eq:candelabra-moose-bd}
\end{equation}
Moreover, we have
\begin{equation}
\lim_{\zeta\to0}\mathbb{E}\left[\widehat{\rscherryrenormrenorm[rsP]}^{\zeta}_{t}(2n)+4\widehat{\rselkrenormrr[rsP]}^{\zeta}_{t}(2n)\right]=0\qquad\text{for all }n\in\mathbb{Z}\setminus\{0\}.\label{eq:Ecandelabra4mooselimitwhennisnotzero}
\end{equation}
\end{prop}

\begin{proof}
The bound \zcref{eq:candelabra-moose-bd} follows immediately from
the triangle inequality applied to \zcref{eq:candelabrabound,eq:moosebd}.
For \zcref{eq:Ecandelabra4mooselimitwhennisnotzero}, we combine \zcref{eq:Ecandelabralimit,eq:mooselimit}
to obtain
\begin{align*}
\lim_{\zeta\to0}\mathbb{E}\left[\widehat{\rscherryrenormrenorm[rsP]}^{\zeta}_{t}(2n)+4\widehat{\rselkrenormrr[rsP]}^{\zeta}_{t}(2n)\right] & =\frac{3}{\pi^{2}}\sum_{k\in\mathbb{Z}}\frac{n^{2}-k^{2}}{n^{4}+n^{2}k^{2}+k^{4}}.
\end{align*}
Then we conclude \zcref{eq:Ecandelabra4mooselimitwhennisnotzero}
by \zcref{eq:seriesiszero}.
\end{proof}

\subsubsection{\label{subsec:Conclusion}Conclusion}
\begin{prop}
\label{prop:boundary-value}We have, for any fixed $\eps<\frac{4}{3}L$,
that
\begin{equation}
\lim_{\zeta\downarrow0}\mathbb{E}\left[\int^{L}_{0}\varphi^{\eps}_{\uu,\vv}(x)\left(\rscherryrenormrenorm[rsP]^{\zeta}_{t}(x)+4\rselkrenormrr[rsP]^{\zeta}_{t}(x)\right)\,\dif x\right]=\frac{1}{6}(\uu+\vv).\label{eq:moose-candelabra-boundary}
\end{equation}
\end{prop}

\begin{proof}
By \zcref{eq:feven-int}, the expectation on the left side of \zcref{eq:moose-candelabra-boundary}
is equal to 
\begin{equation}
L\sum_{j\in\mathbb{Z}}\widehat{\varphi^{\eps}_{\uu,\vv}}(j)\mathbb{E}\left[\widehat{\rscherryrenormrenorm[rsP]}^{\zeta}_{t}(k)+4\widehat{\rselkrenormrr[rsP]}^{\zeta}_{t}(k)\right].\label{eq:boundary-fourier}
\end{equation}
Note that we were allowed to interchange the limits and expectations
using \zcref{eq:candelabra-moose-bd} and the fact that $\hat{\varphi}^{\eps}_{\uu,\vv}(j)$
decays rapidly with $j$ since $\varphi^{\eps}_{\uu,\vv}$ is qualitatively
smooth. Indeed, this same summability allows us to take the limit
$\zeta\downarrow0$ in \zcref{eq:boundary-fourier} using \zcref{eq:zeromodeminus13,eq:candelabra-moose-bd}
to obtain
\[
\lim_{\zeta\downarrow0}\mathbb{E}\left[\int^{L}_{0}\varphi^{\eps}_{\uu,\vv}(x)\left(\rscherryrenormrenorm[rsP]^{\zeta}_{t}(x)+4\rselkrenormrr[rsP]^{\eps,\zeta}_{t}(x)\right)\,\dif x\right]=-\frac{1}{3}L\hat{\varphi}^{\eps}_{\uu,\vv}(0)\overset{\zcref{eq:psihatzero}}{=}\frac{1}{6}(\uu+\vv).\qedhere
\]
\end{proof}

\begin{prop}
\label{prop:sumofcancellinglogs}For any $\tilde{\kappa}>0$, we have
\begin{equation}
\limsup_{\zeta\downarrow0}\left\Vert \mathbb{E}\left[\rscherryrenormrenorm[rsP]^{\zeta}_{t}\right]-C^{(2)}_{\zeta}\right\Vert _{\mathcal{C}^{-\tilde{\kappa}}(\mathbb{R})}<\infty.\label{eq:moose-renorm-close}
\end{equation}
and
\begin{equation}
\limsup_{\zeta\downarrow0}\left\Vert 4\mathbb{E}\left[\rselkrenormrr[rsP]^{\zeta}_{t}\right]+C^{(2)}_{\zeta}\right\Vert _{\mathcal{C}^{-\tilde{\kappa}}(\mathbb{R})}<\infty.\label{eq:cherry-renorm-close}
\end{equation}
\end{prop}

\begin{proof}
Recalling the definition \zcref{eq:C2zetadef} of $C^{(2)}_{\zeta}$,
we have
\[
\left(\mathbb{E}\left[\rscherryrenormrenorm[rsP]^{\zeta}_{t}\right]-C^{(2)}_{\zeta}\right)^{\wedge}(k)=\begin{cases}
0, & k=0;\\
\mathbb{E}\left[\widehat{\rscherryrenormrenorm[rsP]}^{\zeta}_{t}(k)\right], & k\ne0.
\end{cases}
\]
Then \zcref{eq:moose-renorm-close} follows from this and \zcref{eq:candelabrabound}.
To see the bound \zcref{eq:cherry-renorm-close}, we write 
\[
4\mathbb{E}\left[\rselkrenormrr[rsP]^{\zeta}_{t}\right]+C^{(2)}_{\zeta}=\left(4\mathbb{E}\left[\rselkrenormrr[rsP]^{\zeta}_{t}\right]+\mathbb{E}\left[\rscherryrenormrenorm[rsP]^{\zeta}_{t}\right]\right)-\left(\mathbb{E}\left[\rscherryrenormrenorm[rsP]^{\zeta}_{t}\right]-C^{(2)}_{\zeta}\right)
\]
and then use \zcref{eq:zeromodeminus13,eq:candelabra-moose-bd} to
control the first term and \zcref{eq:moose-renorm-close} to control
the second term.
\end{proof}

\input{open-kpz-invariant-section8.tex}

\appendix

\section{\label{sec:Basic-properties-of}Basic properties of the stochastic
heat equation}

In this section, we present several results on the stochastic heat
equation.

\subsection{Mild vs.\ energy solutions}

In this section, we will prove a result on the relationship between
almost stationary energy solutions to the open KPZ equation and mild
solutions of the stochastic heat equation. This result is almost a
combination of \cites[Prop. 3.13]{goncalves:perkowski:simon:2020:derivation}
and \cites[Prop. 4.2]{parekh:2019:KPZ}. There are two improvements
that we need. First, we need a result at the level of coupled equations
(as the results of \cite{goncalves:perkowski:simon:2020:derivation}
are only written at the level of laws of solutions). Second, the notions
of weak solution in \cite{goncalves:perkowski:simon:2020:derivation}
and \cite{parekh:2019:KPZ} are not exactly the same: \cite{goncalves:perkowski:simon:2020:derivation}
considers test functions with Neumann boundary conditions (and then
must add an additional boundary potential by hand to obtain Robin
conditions for the solution), while \cite{parekh:2019:KPZ} considers
test functions with Robin boundary conditions. 

We define $\mathcal{S}_{\mathrm{Neu}}$ and $\nabla_{\kappa}$ as
in \zcref[range]{eq:SNeu,eq:deltakappadef}. Let $\vartheta\in\{\pm1\}$
and let $(k_{t})_{t\ge0}$ be an almost stationary energy solution
to the open KPZ equation
\begin{subequations}
\begin{align}
\dif k_{t}(x) & =\frac{1}{2}\Delta k_{t}(x)\dif t+\frac{\vartheta}{2}(\partial_{x}k_{t}(x))^{2}\dif t+\dif W_{t}(x),\qquad t>0,x\in(0,L);\label{eq:openKPZ-formal}\\
\partial_{x}k_{t}(0) & =\partial_{x}\tilde{h}_{t}(L)=0,\qquad t>0\label{eq:openKPZ-formal-bc}
\end{align}
\end{subequations}
in the sense of \cites[Thm.~3.7]{goncalves:perkowski:simon:2020:derivation}.
In particular, for any $\varphi\in\mathcal{S}_{\mathrm{Neu}}$ and
$t\ge s\ge0$, the limit
\[
\mathcal{B}_{s,t}(\varphi)\coloneqq\lim_{\kappa\downarrow0}\int^{t}_{s}\int^{L}_{0}\varphi(x)\left\{ (\nabla_{\kappa}k_{r}(x))^{2}-\kappa^{-1}\right\} \,\dif x\,\dif r
\]
exists in $L^{2}(\mathbb{P})$, and if we define
\begin{equation}
\langle W_{t},\varphi\rangle\coloneqq\langle k_{t},\varphi\rangle-\langle k_{0},\varphi\rangle-\frac{1}{2}\int^{t}_{s}\langle k_{r},\Delta\varphi\rangle\,\dif r-\frac{\vartheta}{2}\mathcal{B}_{s,t}(\varphi),\label{eq:Wtdef}
\end{equation}
then $(\dif W_{t})$ is a space-time white noise on $[0,\infty)\times[0,L]$,
independent of $\tilde{h}_{0}$, by \cites[Thm. 3.7(3)]{goncalves:perkowski:simon:2020:derivation}
and Lévy's characterization of Brownian motion.
\begin{prop}
\label{prop:cole-hopf}If we define
\[
Z_{t}=\e^{\vartheta k_{t}(x)-t/24},
\]
then $(Z_{t})_{t\ge0}$ has a version that is the (unique by \cites[Prop.~4.2]{parekh:2019:KPZ})
mild solution to the open stochastic heat equation
\begin{align*}
\dif Z_{t}(x) & =\oh\Delta Z_{t}(x)\dif t+\vartheta Z_{t}(x)\dif W_{t}(x), &  & t>0,x\in(0,L);\\
\partial_{x}Z_{t}(0) & =-\nicefrac{1}{2}Z_{t}(0)\qquad\text{and}\qquad\partial_{x}Z_{t}(L)=\nicefrac{1}{2}Z_{t}(L), &  & t\in(0,T];\\
Z_{0}(x) & =\e^{h_{0}(x)}, &  & x\in(0,L)
\end{align*}
in the sense of \cites[Defn.~4.1]{parekh:2022:ergodicity}.
\end{prop}

\begin{proof}[Proof of \zcref{prop:cole-hopf}]
It is shown in \cites[display following (6.21)]{goncalves:perkowski:simon:2020:derivation}
that, for any $f\in\mathcal{S}_{\mathrm{Neu}}$ and any $t>0$, we
have
\begin{equation}
\langle f,Z_{t}\rangle-\langle f,Z_{0}\rangle=\frac{1}{2}\int^{t}_{0}\langle\Delta f,Z_{s}\rangle\,\dif s+\frac{1}{4}\int^{t}_{0}\left(Z_{s}(0)f(0)+Z_{s}(L)f(L)\right)\,\dif s+\int^{t}_{0}\langle Z_{s}f,\dif W_{s}\rangle.\label{eq:weak-Neumann}
\end{equation}
Moreover, by \cites[Prop.~3.10]{goncalves:perkowski:simon:2020:derivation},
for each $t>0$, the map that takes $x$ to the random variable $Z_{t}(x)$
is continuous in $L^{2}(\mathbb{P})$. By \cites[Prop.~4.4]{parekh:2019:KPZ},\footnote{The statement of \cites[Prop.~4.4]{parekh:2019:KPZ} assumes that
the solution takes values in $\mathcal{C}([0,L])$, which is not part
of the definition of almost stationary energy solution in \cite{goncalves:perkowski:simon:2020:derivation},
but the proof of \cites[Prop.~4.4]{parekh:2019:KPZ} does not require
continuity.} it suffices to show that, if we define
\[
\mathcal{S}_{\mathrm{Rob};\mu,\nu}\coloneqq\left\{ \psi\in\mathcal{C}^{\infty}([0,L])\st\psi'(0)=\mu\psi(0)\text{ and }\psi'(L)=\nu\psi(L)\right\} ,
\]
then we have for all $g\in\mathcal{S}_{\mathrm{Rob};-\oh,\oh}$ that
\begin{equation}
\langle g,Z_{t}\rangle-\langle g,Z_{0}\rangle=\frac{1}{2}\int^{t}_{0}\langle\Delta g,Z_{s}\rangle\,\dif s+\int^{t}_{0}\langle Z_{s}g,\dif W_{s}\rangle.\label{eq:weak-Neumann-1}
\end{equation}
In other words, we are trying to show that these two notions of weak
solution are equivalent. So let $g\in\mathcal{S}_{\mathrm{Rob};-\oh,\oh}$.
For $\eps<\frac{4}{3}L$, define $\psi^{\eps}$ as in \zcref{eq:psiepsdef}
and let
\begin{equation}
q^{\eps}(x)\coloneqq2g'(0)\psi^{\eps}(x)-2g'(L)\psi^{\eps}(x-L).\label{eq:qdef}
\end{equation}
and
\begin{equation}
\eta^{\eps}(x)\coloneqq\int^{x}_{0}\left(-g'(0)+\int^{y}_{0}q^{\eps}(z)\,\dif z\right)\,\dif y,\label{eq:etadef}
\end{equation}
so
\[
\partial_{x}\eta^{\eps}(0)=-g'(0)\qquad\text{and}\qquad\partial_{x}\eta^{\eps}(L)=-g'(0)+\int^{L}_{0}\left[2g'(0)\psi^{\eps}(x)-2g'(L)\psi^{\eps}(x-L)\right]\,\dif x\overset{\zcref{eq:psiepshalflineint}}{=}-g'(L),
\]
and thus $g+\eta^{\eps}\in\mathcal{S}_{\mathrm{Neu}}$. Therefore,
we can take $f=g+\eta^{\eps}$ in \zcref{eq:weak-Neumann} and obtain
\begin{align}
0 & =\langle g+\eta^{\eps},Z_{t}\rangle-\langle g+\eta^{\eps},Z_{0}\rangle-\frac{1}{2}\int^{t}_{0}\langle\Delta(g+\eta^{\eps}),Z_{s}\rangle\,\dif s\nonumber \\
 & \qquad-\frac{1}{4}\int^{t}_{0}\left(Z_{s}(0)\left(g(0)+\eta^{\eps}(0)\right)+Z_{s}(1)\left(g(L)+\eta^{\eps}(L)\right)\right)\,\dif s-\int^{t}_{0}\langle Z_{s}(g+\eta^{\eps}),\dif W_{s}\rangle\label{eq:takefgpluseta}
\end{align}
We have from the definition \zcref{eq:etadef} that
\[
\Delta\eta^{\eps}(x)=q^{\eps}(x),
\]
so in fact using \zcref[comp=true]{eq:qdef,eq:psiproperties,eq:psiepsdef,eq:psiepshalflineint}
along with the continuity of $x\mapsto Z_{s}(x)$ in $L^{2}(\mathbb{P})$
noted above, we get the limit in $L^{2}(\mathbb{P})$ (and hence in
probability)
\begin{equation}
\frac{1}{2}\int^{t}_{0}\langle\Delta\eta^{\eps},Z_{s}\rangle\,\dif s\xrightarrow[\eps\to0]{}\frac{1}{2}\int^{t}_{0}\left(g'(0)Z_{s}(0)-g'(L)Z_{s}(L)\right)\,\dif s=-\frac{1}{4}\int^{t}_{0}\left(g(0)Z_{s}(0)+g(L)Z_{s}(L)\right)\,\dif s.\label{eq:boundarycorrect}
\end{equation}
Moreover, we have
\[
\sup_{x\in[0,L]}|\partial_{x}\eta^{\eps}(x)|\le\max\left\{ |\mu\varphi(0)|,|\nu\varphi(0)|\right\} 
\]
while
\[
\partial_{x}\eta^{\eps}(x)=0\qquad\text{for all }x\in[2\eps,L-2\eps].
\]
Together, the last two displays along with the fact that $\eta^{\eps}(0)=0$
imply that
\begin{equation}
\sup_{x\in[0,L]}|\eta^{\eps}(x)|\le\eps\max\left\{ |g(0)|,|g(L)|\right\} .\label{eq:etasmall-1}
\end{equation}
This means that 
\begin{equation}
\langle\eta^{\eps},Z_{t}\rangle-\langle\eta^{\eps},Z_{0}\rangle-\frac{1}{4}\int^{t}_{0}\left(Z_{s}(0)\eta^{\eps}(0)+Z_{s}(1)\eta^{\eps}(L)\right)\,\dif s-\int^{t}_{0}\langle Z_{s}\eta^{\eps},\dif W_{s}\rangle\xrightarrow[\eps\to0]{}0\qquad\text{in probability.}\label{eq:bulketavanishes}
\end{equation}
Using \zcref{eq:boundarycorrect,eq:bulketavanishes} in \zcref{eq:takefgpluseta},
we get
\[
0=\langle g,Z_{t}\rangle-\langle g,Z_{0}\rangle-\frac{1}{2}\int^{t}_{0}\langle\Delta g,Z_{s}\rangle\,\dif s-\int^{t}_{0}\langle Z_{s}g,\dif W_{s}\rangle,
\]
which is \zcref{eq:weak-Neumann-1}.
\end{proof}

\subsection{Properties of solutions to stochastic heat equations}

While most of these results are fairly standard, the presence of boundary
potentials in our setting, together with the use of the two scales
$\eps,\zeta>0$ on which we approximate the boundary potential and
the noise respectively, makes it difficult to find precise references.
For this reason, we provide relatively complete proofs for the convenience
of the reader. These results will then be used to establish \zcref{prop:epstozeroconv},
\ref{lem:l.conZeps}, and \zcref{lem:convofzeta}.

Before turning to the proof, we recall an elementary fact. Let $K$
be a compact set in $\mathbb{R}^{d}$ and let $X_{1},X_{2},\ldots$
and $X$ be $\mathcal{C}(K)$-valued random variables. To establish
the convergence in probability of $X_{n}\to X$ in $\mathcal{C}(K)$,
it is sufficient to show that
\begin{enumerate}
\item for any $t\in K$, we have $X_{n}(t)\to X(t)$ in probability; and
\item There is an $\alpha>0$ such that for any $p\in[1,\infty)$, we have
$\sup_{n}\mathbb{E}|X_{n}(t)-X_{n}(s)|^{p}\leq C_{p}|t-s|^{\alpha p}$
for any $t,s\in K$, where $C_{p}<\infty$ is a constant depending
on $p$ but not on $t$ and $s$.
\end{enumerate}
Throughout the section, we fix $T>0$ and denote $U\coloneqq[0,T]\times[0,L]$.
Since the goal is to show the convergence in probability of $h^{\eps,\zeta}_{\uu,\vv},Z^{\eps,\zeta}_{\uu,\vv}$
in $\mathcal{C}(U)$ as $\eps,\zeta\to0$, to apply the above criterion,
it suffices to show the pointwise convergence and the moment estimates
on the modulus of continuity.

\subsubsection{Feynman--Kac formula and moment estimates}

An important tool in establishing various moment estimates on the
solution to stochastic heat equation is through the Feynman-Kac formula.
We write the solution to \zcref{Zepszeta} as 
\begin{align}
Z^{\eps,\zeta}_{\uu,\vv;t}(x) & =\mathrm{E}_{t,x}\left[\e^{A^{\zeta}(X_{0})}\exp\left\{ \int^{t}_{0}\dif W^{\zeta}(s,X_{s})-\frac{1}{2}t\Sh^{\zeta}_{2L}(0)-\frac{1}{4}\int^{t}_{0}\Sh^{\zeta/2}_{L}(X_{s})\,\dif s\right\} \right.\nonumber \\
 & \qquad\qquad\qquad\left.\times\exp\left\{ \int^{t}_{0}\left(\frac{1}{2}\Sh^{\zeta/2}_{L}(X_{s})+\varphi^{\eps}_{\uu,\vv}(X_{s})\right)\,\dif s\right\} \right]\nonumber \\
 & =\mathrm{E}_{t,x}\exp\left\{ A^{\zeta}(X_{0})+\int^{t}_{0}\dif W^{\zeta}(s,X_{s})-\frac{1}{2}t\Sh^{\zeta}_{2L}(0)+\int^{t}_{0}\left(\frac{1}{4}\Sh^{\zeta/2}_{L}(X_{s})+\varphi^{\eps}_{\uu,\vv}(X_{s})\right)\,\dif s\right\} ,\label{eq:FeynmanKac}
\end{align}
where $\mathrm{E}_{t,x}$ denotes expectation with respect to an auxiliary
backward-in-time Brownian motion $(X_{s})_{s\in[0,t]}$ with $X_{t}=x$.
This is essentially the usual Feynman--Kac formula for the stochastic
heat equation \cite{bertini:cancrini:1995:stochastic}, and can be
proved in the same way. Since our noise has a spatially inhomogeneous
covariance, the Itô correction term takes a slightly different form,
coming from the quadratic variation of the noise given in \zcref{eq:WzetaQV},
i.e. the term $-\frac{1}{2}t\Sh^{\zeta}_{2L}(0)-\frac{1}{4}\int^{t}_{0}\Sh^{\zeta/2}_{L}(X_{s})\,\dif s$
appearing in \zcref{eq:FeynmanKac}.

Using the above representation, we prove the following moment estimate.

\begin{prop}
\label{prop:moment-bd}For each $p\in(-\infty,\infty)$, we have 
\[
\sup_{\substack{t\in[0,T],x\in[0,L]\\
\eps,\zeta\in(0,1)
}
}\mathbb{E}\left[|Z^{\eps,\zeta}_{\uu,\vv;t}(x)|^{p}\right]<\infty\qquad\text{and}\qquad\sup_{\substack{\hat{t}\in[0,T],x\in[0,L]\\
\eps,\zeta\in(0,1)
}
}\mathbb{E}\left[|\hat{Z}^{\eps,\zeta}_{\uu,\vv;\hat{t}}(x)|^{p}\right]<\infty.
\]
\end{prop}

\begin{proof}
We will prove the result for $Z^{\eps,\zeta}_{\uu,\vv;t}(x)$; the
result for $\hat{Z}^{\eps,\zeta}_{\uu,\vv;\hat{t}}(x)$ follows from
\zcref{prop:symmetry-timerev}. We consider the cases $p>0$ and $p<0$
separately.

For $p>0$, by Hölder's inequality it suffices to consider integer
$p$. In this case we introduce replicas of Brownian motions and take
expectations in the Feynman--Kac formula \zcref{eq:FeynmanKac} to
write 
\begin{align*}
\mathbb{E}\left[Z^{\eps,\zeta}_{\uu,\vv;t}(x)^{p}\right] & =\mathrm{E}^{\otimes p}_{t,x}\left(\exp\left\{ \sum^{p}_{j=1}A^{\zeta}(X^{j}_{0})+\sum^{p}_{j=1}\int^{t}_{0}\left(\frac{1}{2}\Sh^{\zeta/2}_{L}(X^{j}_{s})+\varphi^{\eps}_{\uu,\vv}(X^{j}_{s})\right)\,\dif s\right\} \right.\\
 & \qquad\qquad\left.\cdot\exp\left\{ \sum_{1\le j<k\le p}\int^{t}_{0}\left(\Sh^{\zeta}_{2L}(X^{j}_{s}-X^{k}_{s})+\Sh^{\zeta}_{2L}(X^{j}_{s}+X^{k}_{s})\right)\,\dif s\right\} \right).
\end{align*}
Here, $\mathrm{E}^{\otimes p}_{t,x}$ is the expectation under which
each $X^{1},\ldots,X^{p}$ are independent Brownian motions. It is
not difficult to check that each of the term in the exponentials has
exponential moments of all orders, which are bounded uniformly in
$\eps$ and $\zeta$, and hence the desired moment bound follows from
Hölder's inequality.

For $p<0$, by Jensen's inequality it suffices to consider $p=-q\le-1$.
Since 
\[
\sup_{\zeta\in(0,1)}\mathbb{E}\exp\left\{ -q\inf_{x\in[0,L]}A^{\zeta}(x)\right\} <\infty
\]
for all $q\in[1,\infty)$, it suffices to bound $\mathbb{E}\left(\mathrm{E}_{t,x}\e^{\mathcal{Y}+\mathcal{Z}}\right)^{-q}$,
where 
\[
\mathcal{Y}\coloneqq\int^{t}_{0}\dif W^{\zeta}(s,X_{s})-\frac{1}{2}t\Sh^{\zeta}_{2L}(0)-\frac{1}{4}\int^{t}_{0}\Sh^{\zeta/2}_{L}(X_{s})\,\dif s
\]
and 
\[
\mathcal{Z}\coloneqq\int^{t}_{0}\left(\frac{1}{2}\Sh^{\zeta/2}_{L}(X_{s})+\varphi^{\eps}_{\uu,\vv}(X_{s})\right)\,\dif s.
\]
We have 
\[
\left(\mathrm{E}_{t,x}\e^{\mathcal{Y}+\mathcal{Z}}\right)^{-q}=\left(\frac{\mathrm{E}_{t,x}\e^{\mathcal{Y}+\mathcal{Z}}}{\mathrm{E}_{t,x}\e^{\mathcal{Y}}}\right)^{-q}\left(\mathrm{E}_{t,x}\e^{\mathcal{Y}}\right)^{-q}\le\left(\mathrm{E}_{t,x}\e^{\mathcal{Y}-q\mathcal{Z}}\right)\left(\mathrm{E}_{t,x}\e^{\mathcal{Y}}\right)^{-q-1},
\]
by the Jensen's inequality with respect to the measure $\frac{\e^{\mathcal{Y}}}{\mathrm{E}_{t,x}\e^{\mathcal{Y}}}\dif\mathrm{P}_{t,x}$.
Then we take $\mathbb{E}$ of both sides and apply Hölder's inequality.
All moments of $\mathrm{E}_{t,x}\e^{\mathcal{Y}-q\mathcal{Z}}$ are
bounded by the positive moment estimate. It remains to bound $\left(\mathrm{E}_{t,x}\e^{\mathcal{Y}}\right)^{-q'}$
for all $q'\in[1,\infty)$, and this can be done following the proof
of \cites[Thm. 4.6]{hu:le:2022:asymptotics}. The proof is complete. 
\end{proof}

Recall the parabolic Hölder norm is defined in \zcref{eq:parabolic-holder-norm}.
Using the moment estimates obtained in \zcref{prop:moment-bd} and
the Schauder estimate of the heat semigroup, we derive the uniform
moment bound on the Hölder norm of $Z^{\eps,\zeta}_{\uu,\vv}$ and
$h^{\eps,\zeta}_{\uu,\vv}$. 
\begin{prop}
\label{p.holder} For any $p\in[1,\infty)$ and $\alpha<\tfrac{1}{2}$,
we have 
\[
\sup_{\eps,\zeta\in(0,1)}\|Z^{\eps,\zeta}_{\uu,\vv}\|^{p}_{\mathcal{C}^{\alpha}_{\mathfrak{s}}(U)}+\sup_{\eps,\zeta\in(0,1)}\|h^{\eps,\zeta}_{\uu,\vv}\|^{p}_{\mathcal{C}^{\alpha}_{\mathfrak{s}}(U)}<\infty.
\]
\end{prop}

\begin{proof}
We first explain how to obtain the estimate on $h^{\eps,\zeta}_{\uu,\vv}$
from the estimate on $Z^{\eps,\zeta}_{\uu,\vv}$. Since $h^{\eps,\zeta}_{\uu,\vv}=\log Z^{\eps,\zeta}_{\uu,\vv}$,
by the elementary facts of $|\log x|\leq C(x+x^{-1})$ and $|\log x-\log y|\leq(x^{-1}+y^{-1})|x-y|$,
we have 
\[
\|h^{\eps,\zeta}_{\uu,\vv}\|_{\mathcal{C}^{\alpha}_{\mathfrak{s}}(U)}\leq C\left(\sup_{(t,x)\in U}Z^{\eps,\zeta}_{\uu,\vv;t}(x)+\sup_{(t,x)\in U}(Z^{\eps,\zeta}_{\uu,\vv;t}(x))^{-1}\right)\left(1+\|Z^{\eps,\zeta}_{\uu,\vv}\|_{\mathcal{C}^{\alpha}_{\mathfrak{s}}(U)}\right).
\]
Thus, with the moment estimate on $\|Z^{\eps,\zeta}_{\uu,\vv}\|_{\mathcal{C}^{\alpha}_{\mathfrak{s}}(U)}$,
by applying Hölder inequality, we only need to derive the moment estimate
on $\sup_{(t,x)\in U}(Z^{\eps,\zeta}_{\uu,\vv;t}(x))^{-1}$ to obtain
the desired bound on $\|h^{\eps,\zeta}_{\uu,\vv}\|_{\mathcal{C}^{\alpha}_{\mathfrak{s}}(U)}$.
This can be done through a standard chaining argument combined with
the modulus of continuity endowed by the bound of $\|Z^{\eps,\zeta}_{\uu,\vv}\|_{\mathcal{C}^{\alpha}_{\mathfrak{s}}(U)}$
and the negative moment estimate obtained in \zcref{prop:moment-bd};
see for example \cite[Prop.~5.8]{khoshnevisan:kim:mueller:shiu:2020:dissipation}.
Now it is enough to establish the moment bound on $\|Z^{\eps,\zeta}_{\uu,\vv}\|_{\mathcal{C}^{\alpha}_{\mathfrak{s}}(U)}$.
The following argument is similar to \cite[Thm.~2.1]{parekh:2022:ergodicity},
so we do not provide all details.

Since $Z^{\eps,\zeta}_{\uu,\vv}$ solves the equation \eqref{Zepszeta},
we write it in the mild formulation (a periodized version of \eqref{eq:Zepszetamild})
\begin{equation}
\begin{aligned}Z^{\eps,\zeta}_{\uu,\vv;t}(x) & =\int^{L}_{-L}q_{t}(x-y)\e^{A^{\zeta}(y)}\,\dif y+\int^{t}_{0}\int^{L}_{-L}q_{t-s}(x-y)Z^{\eps,\zeta}_{\uu,\vv;s}(y)\left(\varphi^{\eps}_{\uu,\vv}+\oh\Sh^{\zeta/2}_{L}\right)(y)\,\dif y\,\dif s\\
 & \qquad+\int^{t}_{0}\int^{L}_{-L}q_{t-s}(x-y)Z^{\eps,\zeta}_{\uu,\vv;s}(y)\,\dif W^{\zeta}_{s}(y).
\end{aligned}
\label{e.mildperiodic}
\end{equation}
where $q$ is the periodized heat kernel $q_{t}(x)=\sum_{n\in\mathbb{Z}}p_{t}(x+2Ln)$
with $p$ defined in \zcref{eq:ptdef}. There are three terms on the
right side of \eqref{e.mildperiodic}, and we will bound the $\|\cdot\|_{\mathcal{C}^{\alpha}_{\mathfrak{s}}(U)}$
norm of each of them. For the last two terms, using a Schauder estimate
(\cite[Thm.~2.8]{chandra:weber:2017:stochastic} for the version we
need), it is enough to bound the $\|\cdot\|_{\mathcal{C}^{\alpha-2}_{\mathfrak{s}}(U)}$
norm of the source terms, namely $Z^{\eps,\zeta}_{\uu,\vv;s}(y)(\varphi^{\eps}_{\uu,\vv}+\oh\Sh^{\zeta/2}_{L})(y)$
and $Z^{\eps,\zeta}_{\uu,\vv;s}(y)\,\dif W^{\zeta}_{s}(y)$. They
can be handled in the same way as in the proof of \cites[Thm.~2.1]{parekh:2022:ergodicity},
using the moment bounds on $Z^{\eps,\zeta}_{\uu,\vv}$ in \zcref{prop:moment-bd}.
We have 
\[
\sup_{\eps,\zeta\in(0,1)}\mathbb{E}\|Z^{\eps,\zeta}_{\uu,\vv}(\varphi^{\eps}_{\uu,\vv}+\oh\Sh^{\zeta/2}_{L})\|^{p}_{\mathcal{C}^{\alpha-2}_{\mathfrak{s}}(U)}+\sup_{\eps,\zeta\in(0,1)}\mathbb{E}\|Z^{\eps,\zeta}_{\uu,\vv}\dif W^{\zeta}\|^{p}_{\mathcal{C}^{\alpha-2}_{\mathfrak{s}}(U)}\leq C_{p,\alpha},
\]
for any $p\in[1,\infty)$ and $\alpha<\tfrac{1}{2}$. As a matter
of fact, for the first term on the left side, we can take $\alpha<1$
as $\varphi^{\eps}_{\uu,\vv}+\oh\Sh^{\zeta/2}_{L}$ approximates Dirac
functions which lives in $\mathcal{C}^{-1}$. It remains to deal with
the first term on the right side of \zcref{e.mildperiodic}. Recall
that $A^{\zeta}$ is an integral of a mollified white noise, defined
in \zcref{eq:Azetadef}, then it is elementary to check that $\sup_{\zeta\in(0,1)}\mathbb{E}\|e^{A^{\zeta}}\|^{p}_{\mathcal{C}^{\alpha}[-L,L]}<\infty$,
for any $\alpha<\tfrac{1}{2}$. By the standard Schauder estimate,
we obtain the moment bound for $\|q_{\cdot}*e^{A^{\zeta}}\|_{\mathcal{C}^{\alpha}_{\mathfrak{s}}(U)}$
for any $\alpha<\tfrac{1}{2}$. This completes the proof. 
\end{proof}

\subsubsection{\label{subsec:Convergence-as-:}Convergence as $\zeta,\protect\eps\to0$}

In this section, we prove, for each fixed $(t,x)$, the convergence
of $Z^{\eps,\zeta}_{\uu,\vv}(t,x)$ as $\zeta,\eps\to0$. The proof
of the convergence as we remove the mollification (i.e. as $\zeta\to0$)
is rather standard. We sketch it here for the convenience of the reader.
\begin{lem}
\label{lem:convofzeta-HE}For each fixed $\eps>0$ and $(t,x)\in U$,
we have 
\begin{equation}
\lim_{\zeta\to0}Z^{\eps,\zeta}_{\uu,\vv;t}(x)=Z^{\eps}_{\uu,\vv;t}(x)\qquad\text{and}\qquad\lim_{\zeta\to0}\hat{Z}^{\eps,\zeta}_{\uu,\vv;\hat{t}}(x)=\hat{Z}^{\eps}_{\uu,\vv;\hat{t}}(x)\qquad\text{in probability.}\label{eq:convofzeta-HE}
\end{equation}
In addition, for each fixed $(t,x)\in U$, we have 
\begin{equation}
\lim_{\eps\to0}Z^{\eps}_{\uu,\vv;t}(x)=Z_{\uu,\vv;t}(x)\qquad\text{and}\qquad\lim_{\eps\to0}\hat{Z}^{\eps}_{\uu,\vv;\hat{t}}(x)=\hat{Z}_{\uu,\vv;\hat{t}}(x)\qquad\text{in probability.}\label{eq:convofeps-HE}
\end{equation}
\end{lem}

\begin{proof}
Again, we only need to consider $Z^{\eps,\zeta}_{\uu,\vv;t}(x)$;
the corresponding result for $\hat{Z}^{\eps,\zeta}_{\uu,\vv;\hat{t}}(x)$
then follows from \zcref{prop:symmetry-timerev}. We follow the strategy
of \cites[Thm.~3.6]{hu:huang:nualart:tindel:2015:stochastic}. The
proof proceeds in two steps: first, we show that the sequence of random
variables converges, and then we identify the limit. Recall that $Z^{\eps,\zeta}_{\uu,\vv}$
satisfies the integral equation \zcref{eq:Zepszetamild}. 
\begin{thmstepnv}
\item \label{Step:CauchyinLp}First we show that, for each fixed $(t,x)$,
the sequence $(Z^{\eps,\zeta}_{\uu,\vv;t}(x))_{\zeta\in(0,1)}$ is
Cauchy in $L^{2}(\mathbb{P})$ as $\zeta\to0$. Using the Feynman--Kac
formula, the mixed second moment can be computed as 
\begin{equation}
\begin{aligned}\mathbb{E}[Z^{\eps,\zeta_{1}}_{\uu,\vv;t}(x)Z^{\eps,\zeta_{2}}_{\uu,\vv;t}(x)] & =\mathrm{E}^{\otimes2}_{t,x}\exp\left\{ A^{\zeta_{1}}(X^{1}_{0})+A^{\zeta_{2}}(X^{2}_{0})+\sum^{2}_{j=1}\int^{t}_{0}\left(\frac{1}{2}\Sh^{\zeta_{j}/2}_{L}(X^{j}_{s})+\varphi^{\eps}_{\uu,\vv}(X^{j}_{s})\right)\,\dif s\right\} \\
 & \qquad\qquad\times\exp\left\{ \int^{t}_{0}\left(\Sh^{\zeta_{1},\zeta_{2}}_{2L}(X^{2}_{s}-X^{1}_{s})+\Sh^{\zeta_{1},\zeta_{2}}_{2L}(X^{2}_{s}+X^{1}_{s})\right)\,\dif s\right\} .
\end{aligned}
\label{e.mix2nd}
\end{equation}
Here $\Sh^{\zeta_{1},\zeta_{2}}_{2L}(x-x')+\Sh^{\zeta_{1},\zeta_{2}}_{2L}(x+x')$
is the spatial covariance between $\dif W^{\zeta_{1}}_{t}(x)$ and
$\dif W^{\zeta_{2}}_{t'}(x')$, defined similarly to as in \zcref{eq:dWzetacov}.
As $\zeta\to0$, the integrals $\int^{t}_{0}\Sh^{\zeta/2}_{L}(X^{j}_{s})\,\dif s$,
$\int^{t}_{0}\Sh^{\zeta_{1},\zeta_{2}}_{2L}(X^{2}_{s}-X^{1}_{s})\,\dif s$,
and $\int^{t}_{0}\Sh^{\zeta_{1},\zeta_{2}}_{2L}(X^{2}_{s}+X^{1}_{s})\,\dif s$
each converge to the local time of a corresponding Brownian motion.
This implies that the two-point function $\mathbb{E}[Z^{\eps,\zeta_{1}}_{\uu,\vv;t}(x)Z^{\eps,\zeta_{2}}_{\uu,\vv;t}(x)]$
converges as $\zeta_{1},\zeta_{2}\to0$, which in particular means
that the sequence $(Z^{\eps,\zeta}_{\uu,\vv;t}(x))_{\zeta\in(0,1)}$
is Cauchy in $L^{2}(\mathbb{P})$. With some abuse of notation, we
let $Z^{\eps,0}_{\uu,\vv;t}(x)$ denote the limit.
\item Now we identify the limit. For a smooth random variable $F$ (in the
sense of Malliavin calculus; see \cite{nualart:2006:malliavin} for
background), we can write using Gaussian integration by parts that
\begin{align*}
\mathbb{E} & \left[F\int^{t}_{0}\int^{\infty}_{-\infty}p_{t-s}(x-z)Z^{\eps,\zeta}_{\uu,\vv;t}(z)\,\dif W^{\zeta}_{s}(z)\right]\\
 & =\sum_{\substack{q\in2L\mathbb{Z}\\
\iota\in\{\pm1\}
}
}\mathbb{E}\left[F\int^{t}_{0}\int^{L}_{0}\left(\int^{\infty}_{-\infty}p_{t-s}(x-y)Z^{\eps,\zeta}_{\uu,\vv;t}(y)\rho^{\zeta}(y-z+\iota q)\,\dif y\right)\,\dif W_{s}(z)\right]\\
 & =\sum_{\substack{q\in2L\mathbb{Z}\\
\iota\in\{\pm1\}
}
}\mathbb{E}\left[\int^{t}_{0}\int^{L}_{0}\Dif_{s,z}F\int^{\infty}_{-\infty}p_{t-s}(x-y)Z^{\eps,\zeta}_{\uu,\vv;t}(y)\rho^{\zeta}(y-z+\iota q)\,\dif y\,\dif z\,\dif s\right].
\end{align*}
Now we take $\zeta\to0$, noting that as we do so $Z^{\eps,\zeta}_{\uu,\vv;t}(y)\to Z^{\eps,0}_{\uu,\vv;t}(y)$
in $L^{2}(\Omega)$ and $\rho^{\zeta}$ converges to a delta distribution,
and then integrate by parts again on the limit to obtain
\begin{align*}
\mathbb{E} & \left[F\int^{t}_{0}\int^{L}_{0}p_{t-s}(x-z)Z^{\eps,\zeta}_{\uu,\vv;t}(z)\,\dif W^{\zeta}_{s}(z)\right]\\
\to & \sum_{\substack{q\in2L\mathbb{Z}\\
\iota\in\{\pm1\}
}
}\mathbb{E}\left[F\int^{t}_{0}\int^{L}_{0}p_{t-s}(x-z+\iota q)Z^{\eps,0}_{\uu,\vv;t}(z-\iota q)\,\dif W_{s}(z)\right]=\mathbb{E}\left[F\int^{t}_{0}\int^{\infty}_{-\infty}p_{t-s}(x-z)Z^{\eps,0}_{\uu,\vv;t}(z)\,\dif W_{s}(z)\right].
\end{align*}
This implies that as $\zeta\to0$, 
\[
\int^{t}_{0}\int^{L}_{0}p_{t-s}(x-z)Z^{\eps,\zeta}_{\uu,\vv;t}(z)\,\dif W^{\zeta}_{s}(z)\to\int^{t}_{0}\int^{\infty}_{-\infty}p_{t-s}(x-z)Z^{\eps,0}_{\uu,\vv;t}(z)\,\dif W_{s}(z)\qquad\text{weakly in }L^{2}(\Omega).
\]
The convergence of the other terms in \zcref{eq:Zepszetamild} as
$\zeta\to0$ is clear, and we end up with 
\[
\begin{aligned}Z^{\eps,0}_{\uu,\vv;t}(x) & =\int^{\infty}_{-\infty}p_{t}(x-y)\e^{A(y)}\,\dif y+\int^{t}_{0}\int^{\infty}_{-\infty}p_{t-s}(x-y)Z^{\eps,0}_{\uu,\vv;s}(y)\left(\varphi^{\eps}_{\uu,\vv}+\oh\Sh_{L}\right)(y)\,\dif y\\
 & \qquad+\int^{t}_{0}\int^{\infty}_{-\infty}p_{t-s}(x-y)Z^{\eps,0}_{\uu,\vv;t}(y)\,\dif W^ {}_{s}(y).
\end{aligned}
\]
But this is exactly the mild solution formula for \zcref{eq:Zeps}.
This completes the proof of \zcref{eq:convofzeta-HE}.
\item For \zcref{eq:convofeps-HE}, the proof is essentially the same, only
with simplifications. For example, in \zcref{Step:CauchyinLp} we
need the mixed second moment expression for $\mathbb{E}[Z^{\eps_{1}}_{\uu,\vv}(t,x)Z^{\eps_{2}}_{\uu,\vv}(t,x)]$,
which can be derived similarly to as in \eqref{e.mix2nd}. We do not
repeat the argument here. \qedhere
\end{thmstepnv}
\end{proof}

With \zcref{p.holder} and \zcref{lem:convofzeta-HE}, we complete
the proofs of \zcref{prop:epstozeroconv}, \zcref{lem:l.conZeps},
and \zcref{lem:convofzeta}.

\input{open-kpz-invariant-appendixB.tex}

\label{tree-notation-index}\printnomenclature[0.55in]

\printbibliography

\end{document}

%% file: makeboxes2.tex
\makesavedRS{lollipopr}
\makesavedRS{lollipopb}
\makesavedRS{lollipopn}
\makesavedRS{lollipopnc}
\makesavedRS{balloonr}
\makesavedRS{balloonb}
\makesavedRS{balloonn}
\makesavedRS{irenorm}
\makesavedRS{iprenorm}
\makesavedRS{iprenormtri}
\makesavedRS{cherryrr}
\makesavedRS{cherryrb}
\makesavedRS{cherrybb}
\makesavedRS{cherryrenormr}
\makesavedRS{cherryrenormb}
\makesavedRS{cherryrenormrenorm}
\makesavedRS{icherryrr}
\makesavedRS{icherryrb}
\makesavedRS{icherryrenormr}
\makesavedRS{ipcherryrenormr}
\makesavedRS{ipcherryrr}
\makesavedRS{ipcherryrb}
\makesavedRS{elkrenormrr}
\makesavedRS{elkrrr}
\makesavedRS{elkrrb}
\makesavedRS{elkrbr}
\makesavedRS{ielkrrr}
\makesavedRS{ipelkrrr}
\makesavedRS{candelabrarrrr}
\makesavedRS{mooserrrr}
\makesavedRS{clawrr}
\makesavedRS{iplollipopr}
\makesavedRS{ilollipopr}
\makesavedRS{one}
\makesavedRS{potential}
\makesavedRS{noise}
\makesavedRS{x}
\makesavedRS{t}
\makesavedRS{renorm}
\makesavedRS{cherryrbc}
\makesavedRS{lollirc}
\makesavedRS{elkrenormrrXA}
\makesavedRS{elkrenormrrXB}
\makesavedRS{elkrbrG}
\makesavedRS{elkrbrXA}
\makesavedRS{elkrbrXB}
\makesavedRS{elkrbrc}
\makesavedRS{elkrbrcc}
\makesavedRS{elkrbrccone}
\makesavedRS{elkrbrccc}
\makesavedRS{elkrbrcccone}
\makesavedRS{elkrbrcS}
\makesavedRS{elkrbrcSG}
\makesavedRS{elkrbrcSN}
\makesavedRS{clawrG}
\makesavedRS{clawrrX}
\makesavedRS{clawrrXA}
\makesavedRS{clawrrcA}
\makesavedRS{clawrrcAG}
\makesavedRS{clawrrcAS}
\makesavedRS{clawrrc}
\makesavedRS{clawrrcNORECENTER}
\makesavedRS{clawrrcc}
\makesavedRS{clawrrccc}
\makesavedRS{elkrrbc}
\makesavedRS{elkrrbcc}
\makesavedRS{elkrrbccpl}
\makesavedRS{elkrrbccplX}
\makesavedRS{cherryrrc}
\makesavedRS{cherryrrcc}
\makesavedRS{iplollipoprG}
\makesavedRS{ipcherryrbG}
\makesavedRS{elkrrrXA}
\makesavedRS{elkrrrXAX}
\makesavedRS{elkrrrc}
\makesavedRS{elkrrrcXA}
\makesavedRS{elkrrrcXB}
\makesavedRS{elkrrrcc}
\makesavedRS{elkrrrccXA}
\makesavedRS{elkrrrccc}
\makesavedRS{elkrrrcccARM}
\makesavedRS{elkrrrcccXA}
\makesavedRS{elkrrrcccXB}
\makesavedRS{elkrrrcccXC}
\makesavedRS{elkrrrcccc}
\makesavedRS{elkrrrccccXA}
\makesavedRS{elkrrrccccc}
\makesavedRS{elkrrrcccccc}
\makesavedRS{elkrrrccccccXA}
\makesavedRS{candelabrarrrrcA}
\makesavedRS{candelabrarrrrcB}
\makesavedRS{candelabrarrrrcC}
\makesavedRS{candelabrarrrrcD}
\makesavedRS{candelabrarrrrcE}
\makesavedRS{candelabrarrrrcF}
\makesavedRS{mooserrrrR}
\makesavedRS{mooserrrrD}
\makesavedRS{mooserrrrcA}
\makesavedRS{mooserrrrcAD}
\makesavedRS{mooserrrrcAXA}
\makesavedRS{mooserrrrcAXB}
\makesavedRS{mooserrrrcAXC}
\makesavedRS{mooserrrrcAXD}
\makesavedRS{mooserrrrcADI}
\makesavedRS{mooserrrrcAR}
\makesavedRS{mooserrrrcBD}
\makesavedRS{mooserrrrcBR}
\makesavedRS{mooserrrrcCD}
\makesavedRS{mooserrrrcCDI}
\makesavedRS{mooserrrrcCR}
\makesavedRS{mooserrrrccA}
\makesavedRS{mooserrrrccB}
\makesavedRS{mooserrrrccBI}
\makesavedRS{mooserrrrccBKA}
\makesavedRS{mooserrrrccBKB}
\makesavedRS{mooserrrrccBKC}
\makesavedRS{mooserrrrccBR}
\makesavedRS{mooserrrrcPA}
\makesavedRS{mooserrrrcPB}
\makesavedRS{mooserrrrcPBR}
\makesavedRS{mooserrrrcPBRI}
\makesavedRS{mooserrrrcPC}
\makesavedRS{mooserrrrcPD}
\makesavedRS{mooserrrrcPE}
\makesavedRS{mooserrrrcPF}
\makesavedRS{mooserrrrccAI}
\makesavedRS{bphz}
\makesavedRS{bphzXX}
\makesavedRS{triangle}
\makesavedRS{cherryrenormbtr}
\makesavedRS{mooserrrrcDD}
\makesavedRS{mooserrrrcDDI}
\makesavedRS{mooserrrrcDDIP}
\makesavedRS{mooserrrrcDR}
\makesavedRS{ipelkrrrR}
\makesavedRS{mooserrrrcPG}
\makesavedRS{mooserrrrcPGI}
\makesavedRS{mooserrrrRA}
\makesavedRS{mooserecenterA}
\makesavedRS{mooserecenterB}
\makesavedRS{mooserecenterC}
\makesavedRS{mooserrrrccAA}
\makesavedRS{mooserrrrccBKD}
\makesavedRS{mooserrrrccBKE}
\makesavedRS{mooserrrrccBKF}
\makesavedRS{mooserrrrccBKG}
\makesavedRS{cherryrenormrXA}
\makesavedRS{cherryrenormrXB}
\makesavedRS{cherryrenormrenormXA}
\makesavedRS{cherryrenormrenormXB}

%% file: rs-commands-lyx-2.tex
\newcommandx\rsxn[2][usedefault, addprefix=\global, 1=rsN]{\rsify{\tikzxn{#2}}{#1}}%
\newcommandx\rstn[2][usedefault, addprefix=\global, 1=rsN]{\rsify{\tikztn{#2}}{#1}}%
\newcommandx\rsspecial[2][usedefault, addprefix=\global, 1=rsN]{\rsify{\tikzmathgeneral{gray}{#2}}{#1}}%
\newcommandx\rsmath[2][usedefault, addprefix=\global, 1=rsN]{\rsify{\tikzmath{#2}}{#1}}%
\newcommandx\rslollipopr[1][usedefault, addprefix=\global, 1=rsN]{\myrslollipopr{#1}}%
\newcommandx\rslollipopb[1][usedefault, addprefix=\global, 1=rsN]{\myrslollipopb{#1}}%
\newcommandx\rslollipopn[1][usedefault, addprefix=\global, 1=rsN]{\myrslollipopn{#1}}%
\newcommandx\rslollipopnc[1][usedefault, addprefix=\global, 1=rsN]{\myrslollipopnc{#1}}%
\newcommandx\rsballoonr[1][usedefault, addprefix=\global, 1=rsN]{\myrsballoonr{#1}}%
\newcommandx\rsballoonb[1][usedefault, addprefix=\global, 1=rsN]{\myrsballoonb{#1}}%
\newcommandx\rsballoonn[1][usedefault, addprefix=\global, 1=rsN]{\myrsballoonn{#1}}%
\newcommandx\rsirenorm[1][usedefault, addprefix=\global, 1=rsN]{\myrsirenorm{#1}}%
\newcommandx\rsiprenorm[1][usedefault, addprefix=\global, 1=rsN]{\myrsiprenorm{#1}}%
\newcommandx\rsiprenormtri[1][usedefault, addprefix=\global, 1=rsN]{\myrsiprenormtri{#1}}%
\newcommandx\rscherryrr[1][usedefault, addprefix=\global, 1=rsN]{\myrscherryrr{#1}}%
\newcommandx\rscherryrb[1][usedefault, addprefix=\global, 1=rsN]{\myrscherryrb{#1}}%
\newcommandx\rscherrybb[1][usedefault, addprefix=\global, 1=rsN]{\myrscherrybb{#1}}%
\newcommandx\rscherryrenormr[1][usedefault, addprefix=\global, 1=rsN]{\myrscherryrenormr{#1}}%
\newcommandx\rscherryrenormb[1][usedefault, addprefix=\global, 1=rsN]{\myrscherryrenormb{#1}}%
\newcommandx\rscherryrenormrenorm[1][usedefault, addprefix=\global, 1=rsN]{\myrscherryrenormrenorm{#1}}%
\newcommandx\rsicherryrr[1][usedefault, addprefix=\global, 1=rsN]{\myrsicherryrr{#1}}%
\newcommandx\rsicherryrb[1][usedefault, addprefix=\global, 1=rsN]{\myrsicherryrb{#1}}%
\newcommandx\rsicherryrenormr[1][usedefault, addprefix=\global, 1=rsN]{\myrsicherryrenormr{#1}}%
\newcommandx\rsipcherryrenormr[1][usedefault, addprefix=\global, 1=rsN]{\myrsipcherryrenormr{#1}}%
\newcommandx\rsipcherryrr[1][usedefault, addprefix=\global, 1=rsN]{\myrsipcherryrr{#1}}%
\newcommandx\rsipcherryrb[1][usedefault, addprefix=\global, 1=rsN]{\myrsipcherryrb{#1}}%
\newcommandx\rselkrenormrr[1][usedefault, addprefix=\global, 1=rsN]{\myrselkrenormrr{#1}}%
\newcommandx\rselkrrr[1][usedefault, addprefix=\global, 1=rsN]{\myrselkrrr{#1}}%
\newcommandx\rselkrrb[1][usedefault, addprefix=\global, 1=rsN]{\myrselkrrb{#1}}%
\newcommandx\rselkrbr[1][usedefault, addprefix=\global, 1=rsN]{\myrselkrbr{#1}}%
\newcommandx\rsielkrrr[1][usedefault, addprefix=\global, 1=rsN]{\myrsielkrrr{#1}}%
\newcommandx\rsipelkrrr[1][usedefault, addprefix=\global, 1=rsN]{\myrsipelkrrr{#1}}%
\newcommandx\rscandelabrarrrr[1][usedefault, addprefix=\global, 1=rsN]{\myrscandelabrarrrr{#1}}%
\newcommandx\rsmooserrrr[1][usedefault, addprefix=\global, 1=rsN]{\myrsmooserrrr{#1}}%
\newcommandx\rsclawrr[1][usedefault, addprefix=\global, 1=rsN]{\myrsclawrr{#1}}%
\newcommandx\rsiplollipopr[1][usedefault, addprefix=\global, 1=rsN]{\myrsiplollipopr{#1}}%
\newcommandx\rsilollipopr[1][usedefault, addprefix=\global, 1=rsN]{\myrsilollipopr{#1}}%
\newcommandx\rsone[1][usedefault, addprefix=\global, 1=rsN]{\myrsone{#1}}%
\newcommandx\rspotential[1][usedefault, addprefix=\global, 1=rsN]{\myrspotential{#1}}%
\newcommandx\rsnoise[1][usedefault, addprefix=\global, 1=rsN]{\myrsnoise{#1}}%
\newcommandx\rsx[1][usedefault, addprefix=\global, 1=rsN]{\myrsx{#1}}%
\newcommandx\rst[1][usedefault, addprefix=\global, 1=rsN]{\myrst{#1}}%
\newcommandx\rsrenorm[1][usedefault, addprefix=\global, 1=rsN]{\myrsrenorm{#1}}%
\newcommandx\rscherryrbc[1][usedefault, addprefix=\global, 1=rsN]{\myrscherryrbc{#1}}%
\newcommandx\rslollirc[1][usedefault, addprefix=\global, 1=rsN]{\myrslollirc{#1}}%
\newcommandx\rselkrenormrrXA[1][usedefault, addprefix=\global, 1=rsN]{\myrselkrenormrrXA{#1}}%
\newcommandx\rselkrenormrrXB[1][usedefault, addprefix=\global, 1=rsN]{\myrselkrenormrrXB{#1}}%
\newcommandx\rselkrbrG[1][usedefault, addprefix=\global, 1=rsN]{\myrselkrbrG{#1}}%
\newcommandx\rselkrbrXA[1][usedefault, addprefix=\global, 1=rsN]{\myrselkrbrXA{#1}}%
\newcommandx\rselkrbrXB[1][usedefault, addprefix=\global, 1=rsN]{\myrselkrbrXB{#1}}%
\newcommandx\rselkrbrc[1][usedefault, addprefix=\global, 1=rsN]{\myrselkrbrc{#1}}%
\newcommandx\rselkrbrcc[1][usedefault, addprefix=\global, 1=rsN]{\myrselkrbrcc{#1}}%
\newcommandx\rselkrbrccone[1][usedefault, addprefix=\global, 1=rsN]{\myrselkrbrccone{#1}}%
\newcommandx\rselkrbrccc[1][usedefault, addprefix=\global, 1=rsN]{\myrselkrbrccc{#1}}%
\newcommandx\rselkrbrcccone[1][usedefault, addprefix=\global, 1=rsN]{\myrselkrbrcccone{#1}}%
\newcommandx\rselkrbrcS[1][usedefault, addprefix=\global, 1=rsN]{\myrselkrbrcS{#1}}%
\newcommandx\rselkrbrcSG[1][usedefault, addprefix=\global, 1=rsN]{\myrselkrbrcSG{#1}}%
\newcommandx\rselkrbrcSN[1][usedefault, addprefix=\global, 1=rsN]{\myrselkrbrcSN{#1}}%
\newcommandx\rsclawrG[1][usedefault, addprefix=\global, 1=rsN]{\myrsclawrG{#1}}%
\newcommandx\rsclawrrX[1][usedefault, addprefix=\global, 1=rsN]{\myrsclawrrX{#1}}%
\newcommandx\rsclawrrXA[1][usedefault, addprefix=\global, 1=rsN]{\myrsclawrrXA{#1}}%
\newcommandx\rsclawrrcA[1][usedefault, addprefix=\global, 1=rsN]{\myrsclawrrcA{#1}}%
\newcommandx\rsclawrrcAG[1][usedefault, addprefix=\global, 1=rsN]{\myrsclawrrcAG{#1}}%
\newcommandx\rsclawrrcAS[1][usedefault, addprefix=\global, 1=rsN]{\myrsclawrrcAS{#1}}%
\newcommandx\rsclawrrc[1][usedefault, addprefix=\global, 1=rsN]{\myrsclawrrc{#1}}%
\newcommandx\rsclawrrcNORECENTER[1][usedefault, addprefix=\global, 1=rsN]{\myrsclawrrcNORECENTER{#1}}%
\newcommandx\rsclawrrcc[1][usedefault, addprefix=\global, 1=rsN]{\myrsclawrrcc{#1}}%
\newcommandx\rsclawrrccc[1][usedefault, addprefix=\global, 1=rsN]{\myrsclawrrccc{#1}}%
\newcommandx\rselkrrbc[1][usedefault, addprefix=\global, 1=rsN]{\myrselkrrbc{#1}}%
\newcommandx\rselkrrbcc[1][usedefault, addprefix=\global, 1=rsN]{\myrselkrrbcc{#1}}%
\newcommandx\rselkrrbccpl[1][usedefault, addprefix=\global, 1=rsN]{\myrselkrrbccpl{#1}}%
\newcommandx\rselkrrbccplX[1][usedefault, addprefix=\global, 1=rsN]{\myrselkrrbccplX{#1}}%
\newcommandx\rscherryrrc[1][usedefault, addprefix=\global, 1=rsN]{\myrscherryrrc{#1}}%
\newcommandx\rscherryrrcc[1][usedefault, addprefix=\global, 1=rsN]{\myrscherryrrcc{#1}}%
\newcommandx\rsiplollipoprG[1][usedefault, addprefix=\global, 1=rsN]{\myrsiplollipoprG{#1}}%
\newcommandx\rsipcherryrbG[1][usedefault, addprefix=\global, 1=rsN]{\myrsipcherryrbG{#1}}%
\newcommandx\rselkrrrXA[1][usedefault, addprefix=\global, 1=rsN]{\myrselkrrrXA{#1}}%
\newcommandx\rselkrrrXAX[1][usedefault, addprefix=\global, 1=rsN]{\myrselkrrrXAX{#1}}%
\newcommandx\rselkrrrc[1][usedefault, addprefix=\global, 1=rsN]{\myrselkrrrc{#1}}%
\newcommandx\rselkrrrcXA[1][usedefault, addprefix=\global, 1=rsN]{\myrselkrrrcXA{#1}}%
\newcommandx\rselkrrrcXB[1][usedefault, addprefix=\global, 1=rsN]{\myrselkrrrcXB{#1}}%
\newcommandx\rselkrrrcc[1][usedefault, addprefix=\global, 1=rsN]{\myrselkrrrcc{#1}}%
\newcommandx\rselkrrrccXA[1][usedefault, addprefix=\global, 1=rsN]{\myrselkrrrccXA{#1}}%
\newcommandx\rselkrrrccc[1][usedefault, addprefix=\global, 1=rsN]{\myrselkrrrccc{#1}}%
\newcommandx\rselkrrrcccARM[1][usedefault, addprefix=\global, 1=rsN]{\myrselkrrrcccARM{#1}}%
\newcommandx\rselkrrrcccXA[1][usedefault, addprefix=\global, 1=rsN]{\myrselkrrrcccXA{#1}}%
\newcommandx\rselkrrrcccXB[1][usedefault, addprefix=\global, 1=rsN]{\myrselkrrrcccXB{#1}}%
\newcommandx\rselkrrrcccXC[1][usedefault, addprefix=\global, 1=rsN]{\myrselkrrrcccXC{#1}}%
\newcommandx\rselkrrrcccc[1][usedefault, addprefix=\global, 1=rsN]{\myrselkrrrcccc{#1}}%
\newcommandx\rselkrrrccccXA[1][usedefault, addprefix=\global, 1=rsN]{\myrselkrrrccccXA{#1}}%
\newcommandx\rselkrrrccccc[1][usedefault, addprefix=\global, 1=rsN]{\myrselkrrrccccc{#1}}%
\newcommandx\rselkrrrcccccc[1][usedefault, addprefix=\global, 1=rsN]{\myrselkrrrcccccc{#1}}%
\newcommandx\rselkrrrccccccXA[1][usedefault, addprefix=\global, 1=rsN]{\myrselkrrrccccccXA{#1}}%
\newcommandx\rscandelabrarrrrcA[1][usedefault, addprefix=\global, 1=rsN]{\myrscandelabrarrrrcA{#1}}%
\newcommandx\rscandelabrarrrrcB[1][usedefault, addprefix=\global, 1=rsN]{\myrscandelabrarrrrcB{#1}}%
\newcommandx\rscandelabrarrrrcC[1][usedefault, addprefix=\global, 1=rsN]{\myrscandelabrarrrrcC{#1}}%
\newcommandx\rscandelabrarrrrcD[1][usedefault, addprefix=\global, 1=rsN]{\myrscandelabrarrrrcD{#1}}%
\newcommandx\rscandelabrarrrrcE[1][usedefault, addprefix=\global, 1=rsN]{\myrscandelabrarrrrcE{#1}}%
\newcommandx\rscandelabrarrrrcF[1][usedefault, addprefix=\global, 1=rsN]{\myrscandelabrarrrrcF{#1}}%
\newcommandx\rsmooserrrrR[1][usedefault, addprefix=\global, 1=rsN]{\myrsmooserrrrR{#1}}%
\newcommandx\rsmooserrrrD[1][usedefault, addprefix=\global, 1=rsN]{\myrsmooserrrrD{#1}}%
\newcommandx\rsmooserrrrcA[1][usedefault, addprefix=\global, 1=rsN]{\myrsmooserrrrcA{#1}}%
\newcommandx\rsmooserrrrcAD[1][usedefault, addprefix=\global, 1=rsN]{\myrsmooserrrrcAD{#1}}%
\newcommandx\rsmooserrrrcAXA[1][usedefault, addprefix=\global, 1=rsN]{\myrsmooserrrrcAXA{#1}}%
\newcommandx\rsmooserrrrcAXB[1][usedefault, addprefix=\global, 1=rsN]{\myrsmooserrrrcAXB{#1}}%
\newcommandx\rsmooserrrrcAXC[1][usedefault, addprefix=\global, 1=rsN]{\myrsmooserrrrcAXC{#1}}%
\newcommandx\rsmooserrrrcAXD[1][usedefault, addprefix=\global, 1=rsN]{\myrsmooserrrrcAXD{#1}}%
\newcommandx\rsmooserrrrcADI[1][usedefault, addprefix=\global, 1=rsN]{\myrsmooserrrrcADI{#1}}%
\newcommandx\rsmooserrrrcAR[1][usedefault, addprefix=\global, 1=rsN]{\myrsmooserrrrcAR{#1}}%
\newcommandx\rsmooserrrrcBD[1][usedefault, addprefix=\global, 1=rsN]{\myrsmooserrrrcBD{#1}}%
\newcommandx\rsmooserrrrcBR[1][usedefault, addprefix=\global, 1=rsN]{\myrsmooserrrrcBR{#1}}%
\newcommandx\rsmooserrrrcCD[1][usedefault, addprefix=\global, 1=rsN]{\myrsmooserrrrcCD{#1}}%
\newcommandx\rsmooserrrrcCDI[1][usedefault, addprefix=\global, 1=rsN]{\myrsmooserrrrcCDI{#1}}%
\newcommandx\rsmooserrrrcCR[1][usedefault, addprefix=\global, 1=rsN]{\myrsmooserrrrcCR{#1}}%
\newcommandx\rsmooserrrrccA[1][usedefault, addprefix=\global, 1=rsN]{\myrsmooserrrrccA{#1}}%
\newcommandx\rsmooserrrrccB[1][usedefault, addprefix=\global, 1=rsN]{\myrsmooserrrrccB{#1}}%
\newcommandx\rsmooserrrrccBI[1][usedefault, addprefix=\global, 1=rsN]{\myrsmooserrrrccBI{#1}}%
\newcommandx\rsmooserrrrccBKA[1][usedefault, addprefix=\global, 1=rsN]{\myrsmooserrrrccBKA{#1}}%
\newcommandx\rsmooserrrrccBKB[1][usedefault, addprefix=\global, 1=rsN]{\myrsmooserrrrccBKB{#1}}%
\newcommandx\rsmooserrrrccBKC[1][usedefault, addprefix=\global, 1=rsN]{\myrsmooserrrrccBKC{#1}}%
\newcommandx\rsmooserrrrccBR[1][usedefault, addprefix=\global, 1=rsN]{\myrsmooserrrrccBR{#1}}%
\newcommandx\rsmooserrrrcPA[1][usedefault, addprefix=\global, 1=rsN]{\myrsmooserrrrcPA{#1}}%
\newcommandx\rsmooserrrrcPB[1][usedefault, addprefix=\global, 1=rsN]{\myrsmooserrrrcPB{#1}}%
\newcommandx\rsmooserrrrcPBR[1][usedefault, addprefix=\global, 1=rsN]{\myrsmooserrrrcPBR{#1}}%
\newcommandx\rsmooserrrrcPBRI[1][usedefault, addprefix=\global, 1=rsN]{\myrsmooserrrrcPBRI{#1}}%
\newcommandx\rsmooserrrrcPC[1][usedefault, addprefix=\global, 1=rsN]{\myrsmooserrrrcPC{#1}}%
\newcommandx\rsmooserrrrcPD[1][usedefault, addprefix=\global, 1=rsN]{\myrsmooserrrrcPD{#1}}%
\newcommandx\rsmooserrrrcPE[1][usedefault, addprefix=\global, 1=rsN]{\myrsmooserrrrcPE{#1}}%
\newcommandx\rsmooserrrrcPF[1][usedefault, addprefix=\global, 1=rsN]{\myrsmooserrrrcPF{#1}}%
\newcommandx\rsmooserrrrccAI[1][usedefault, addprefix=\global, 1=rsN]{\myrsmooserrrrccAI{#1}}%
\newcommandx\rsbphz[1][usedefault, addprefix=\global, 1=rsN]{\myrsbphz{#1}}%
\newcommandx\rsbphzXX[1][usedefault, addprefix=\global, 1=rsN]{\myrsbphzXX{#1}}%
\newcommandx\rstriangle[1][usedefault, addprefix=\global, 1=rsN]{\myrstriangle{#1}}%
\newcommandx\rscherryrenormbtr[1][usedefault, addprefix=\global, 1=rsN]{\myrscherryrenormbtr{#1}}%
\newcommandx\rsmooserrrrcDD[1][usedefault, addprefix=\global, 1=rsN]{\myrsmooserrrrcDD{#1}}%
\newcommandx\rsmooserrrrcDDI[1][usedefault, addprefix=\global, 1=rsN]{\myrsmooserrrrcDDI{#1}}%
\newcommandx\rsmooserrrrcDDIP[1][usedefault, addprefix=\global, 1=rsN]{\myrsmooserrrrcDDIP{#1}}%
\newcommandx\rsmooserrrrcDR[1][usedefault, addprefix=\global, 1=rsN]{\myrsmooserrrrcDR{#1}}%
\newcommandx\rsipelkrrrR[1][usedefault, addprefix=\global, 1=rsN]{\myrsipelkrrrR{#1}}%
\newcommandx\rsmooserrrrcPG[1][usedefault, addprefix=\global, 1=rsN]{\myrsmooserrrrcPG{#1}}%
\newcommandx\rsmooserrrrcPGI[1][usedefault, addprefix=\global, 1=rsN]{\myrsmooserrrrcPGI{#1}}%
\newcommandx\rsmooserrrrRA[1][usedefault, addprefix=\global, 1=rsN]{\myrsmooserrrrRA{#1}}%
\newcommandx\rsmooserecenterA[1][usedefault, addprefix=\global, 1=rsN]{\myrsmooserecenterA{#1}}%
\newcommandx\rsmooserecenterB[1][usedefault, addprefix=\global, 1=rsN]{\myrsmooserecenterB{#1}}%
\newcommandx\rsmooserecenterC[1][usedefault, addprefix=\global, 1=rsN]{\myrsmooserecenterC{#1}}%
\newcommandx\rsmooserrrrccAA[1][usedefault, addprefix=\global, 1=rsN]{\myrsmooserrrrccAA{#1}}%
\newcommandx\rsmooserrrrccBKD[1][usedefault, addprefix=\global, 1=rsN]{\myrsmooserrrrccBKD{#1}}%
\newcommandx\rsmooserrrrccBKE[1][usedefault, addprefix=\global, 1=rsN]{\myrsmooserrrrccBKE{#1}}%
\newcommandx\rsmooserrrrccBKF[1][usedefault, addprefix=\global, 1=rsN]{\myrsmooserrrrccBKF{#1}}%
\newcommandx\rsmooserrrrccBKG[1][usedefault, addprefix=\global, 1=rsN]{\myrsmooserrrrccBKG{#1}}%
\newcommandx\rscherryrenormrXA[1][usedefault, addprefix=\global, 1=rsN]{\myrscherryrenormrXA{#1}}%
\newcommandx\rscherryrenormrXB[1][usedefault, addprefix=\global, 1=rsN]{\myrscherryrenormrXB{#1}}%
\newcommandx\rscherryrenormrenormXA[1][usedefault, addprefix=\global, 1=rsN]{\myrscherryrenormrenormXA{#1}}%
\newcommandx\rscherryrenormrenormXB[1][usedefault, addprefix=\global, 1=rsN]{\myrscherryrenormrenormXB{#1}}%

%% file: open-kpz-invariant-section8.tex
\section{Stochastic estimates\label{sec:BPHZ}}

\newcommand{\x}{\mathbf{x}}
\newcommand{\y}{\mathbf{y}}
\newcommand{\z}{\mathbf{z}}
\newcommand{\w}{\mathbf{w}}
\newcommand{\sr}{\sigma_{\mathrm{refl}}}
\renewcommand{\mS}{\mathbb{S}}

This section is concerned with upper bounds on the moments of terms appearing in
the model $(\hat\Pi^{\eps,\zeta},\hat\Gamma^{\eps,\zeta})$. We need these
estimates for two reasons. First, we need to prove
\zcref{thm:model-norm-bounded}, i.e.\ to show that the model norm of
$(\hat\Pi^{\eps,\zeta},\hat\Gamma^{\eps,\zeta})$ is bounded uniformly in
$\eps,\zeta$. Second, we need to prove \zcref{prop:var-boundary} below, which
shows that the variance of a number of the explicit terms in the regularity
structure expansion of $(\partial_x h^{\eps,\zeta})^2$, when
evaluated at the spatial boundary and integrated in time, goes to $0$. Each of these
results is obtained by bounding the moments of
various iterated stochastic integrals arising in the definitions of
$\hat{\boldsymbol{\varPi}}^{\eps,\zeta}\tau$.

Estimates on the model for the KPZ equation have previously been obtained in
\cite{hairer:2013:solving, hairer:quastel:2018:class,
gubinelli:perkowski:2017:kpz}. Our estimates follow a largely similar strategy
to those, the most important difference being that we must deal with the fact
that the noise considered here has a reflection symmetry arising from the boundary condition. See
\cite{gerencer:hairer:2019:singular,chouk:zuijlen:2021:anderson} for previous work
on singular SPDE with boundary conditions.
The particular challenge created by the boundary condition in our setting is in
controlling the model at the boundary. 
The main issue is that the reflection symmetry of the noise
leads to potential blow-ups at the boundary. A similar problem appears in
\cite{gerencer:hairer:2019:singular}, where it is overcome by introducing spaces of modeled
distributions that allow for controlled blow-up at the boundary. In particular, in
\cite{gerencer:hairer:2019:singular} there is no need to reprove the
stochastic estimates on the model (which is kept to be the periodic one) -
except for the first renormalisation constant. We do not follow the same
approach because we must solve two additional issues. 
First, we do not want to
allow such blow-up of the modeled distributions, since our application of
the reconstruction theorem (the estimate \zcref{eq:apply-reconstruction} in the
proof of \zcref{prop:apply-reconstruction-boundary}) takes place exactly at the
boundary. Furthermore, in order to control the norm of the model with Neumann
boundary conditions, we can in principle allow the
stochastic terms to explode at the boundary, as long as they remain genuine
distributions on the torus (i.e.\ as long as the blow-up is integrable, which turns out to always  be the case except for the 
first renormalisation constant, which is not part of the model). However, we will eventually need to control
the stochastic terms also at the boundary, and this turns out to be rather
subtle, since \zcref{tab:the-terms-1} shows that some of the stochastic
terms do actually contribute to the expectation. (See \zcref{rem:claw-exp-isnt-smooth}
below for a similar example of the additional complexity caused by the
boundary). One of the main contributions of this section is to prove that the two-point correlation functions of all relevant stochastic terms (except for
$\rsrenorm$) vanish at the boundary. Obtaining these sharp estimates requires us to revisit
existing results on convergent Feynman diagrams
\cite{hairer:quastel:2018:class,hairer:2018:analyst}; see
\zcref{prop:graph-bound}.

\subsection{Notation, reductions, and basic estimates}

In order to control the model norm of $(\hat{\Pi}^{\eps,\zeta},\hat{\Gamma}^{\eps,\zeta})$, we must obtain stochastic estimates for all the basis elements of the regularity structure $\check{\mathsf{T}}$ defined in \zcref{sec:trunc-model-space}. However, since the elements of $\check{\mathsf{T}}_{\mathcal{I}}, \check{\mathsf{T}}_{\mathcal{I}'}$ can be obtained from other elements through integration and differentiation, it suffices to control the model norm for elements in $$\check{\mathsf{T}}_{\mathrm{stoch}} = \check{\mathsf{T}}_{\bullet} \cup \check{\mathsf{T}}_{\star} .$$
Before we proceed, let us introduce some notation. Due to the large number of integration variables appearing in this
section, we depart from the previous notation of using subscripts
for time indices. Instead, we write for example $\mathbf{x}=(t,x)$
and then write $f(\mathbf{x})$ or $f(t,x)$ where previously we had written $f_{t}(x)$.

We note that the elements of $\check{\mathsf{T}}_{\mathrm{stoch}}$ are all periodic, in the sense that for every $\x \in \RR^2$, the function $ (s, y) \mapsto \hat{\Pi}^{\eps, \zeta}_\x \tau (s, y)$ is $2L$-periodic in the $y$ variable.\footnote{This is not the case for all elements of $\check{\mathsf{T}}$, for example not for the polynomial term $\rsx\in\check{\mathsf{T}}_{\mathrm{poly}}$.}
In particular, all the functions and distributions that we consider in this section are naturally defined on the strip
\begin{equation*}
    \mathbb{S}_{2L} = \RR \times \TT_{2 L} ,
\end{equation*}
where $\mathbb{T}_{2L}$ is the $2L$-torus $\TT_{2L } = \RR/(2L \ZZ)$, which we naturally identify with the interval $[-L, L)$ with periodic boundary.
We will denote by
\begin{equation}
    \llangle f,g\rrangle_{\mS_{2L}^d} \coloneqq \int_{\mS_{2L}^{d}} f (\mathbf{x}_{1},\ldots,\mathbf{x}_{d}) g(\mathbf{x}_{1},\ldots,\mathbf{x}_{d})\,\dif\mathbf{x}_{1}\cdots\dif\mathbf{x}_{d} \label{eq:llip}
\end{equation}
the space-time pairing between two functions $f,g\colon \mS_{2L}^{d}\to\mathbb{R}$ whenever the integral is well-defined. This is not to be confused
with the spatial pairing $\langle\cdot,\cdot\rangle$ on $[0,L]$ defined in \zcref{eq:fgIP}. In the upcoming calculations we will usually drop the subscript $\mS^d_{2L}$, since the number of variables $d$ will be clear from context and the strip size $L$ will be fixed throughout the section. %

For clarity, let us now recall the model norms that we are going to estimate in this section.  We start by defining a space of periodic test functions. We say that a function $g \colon \mS_{2L} \to \RR$ is a test function if $g$ is the periodic extension of a function with compact support in $ [-1, 1] \times (-L, L) $. We then define
\begin{equation*}
\begin{aligned}
  \mathcal{C}_{\mathrm{c}}^{1, 2} = \{ g \colon \mS_{2L} \to \RR \text{ such that } g \text{ is a test function with } \partial_{t} g, \partial_{x}^{2} g
\text{ continuous } \} ,
\end{aligned}
\end{equation*}
and we endow this space with the usual uniform topology
\begin{equation*}
\begin{aligned}
  \| g \|_{\mathcal{C}^{1, 2}_{\mathrm {c}}} = \sup_{(t, x) \in \mS_{2L} } \left(
| g (t, x) | + | \partial_{t} g (t, x) | + |
\partial_{x}^{2} g (t, x) | \right) .
\end{aligned}
\end{equation*}
For any $g \in \mathcal{C}^{1,2}_{\mathrm{c}}$, we define the rescaled test function $g^\lambda_\x$, for fixed $ \lambda \in (0, 1] $ and $ \x \in \mS_{2L}$, by
\begin{equation*}
\begin{aligned}
g_\x^\lambda (\y) = g^{\lambda}_{(t, x)} (s, y) = \lambda^{- 3} g ( 
\lambda^{-2}(t-s),  \lambda^{-1} (x-y))  \qquad \text{for all } \y = (s, y) \in (t +[-1, 1]) \times (x + (-L, L))  ,
\end{aligned}
\end{equation*}
by which we mean the (spatially) $2L$-periodic extension of a function with compact support
in $(t +[-1, 1]) \times (x + (-L, L))$, using that $\lambda \leq 1$. 

Then, for $T>0$, consider $\|(\hat{\Pi}^{\eps,\zeta},\hat{\Gamma}^{\eps,\zeta})\|_{T}$
the norm of the model restricted to $[-T,T]\times\mathbb{R}$, as in
\cite[Defn.~2.17 and Rmk.~2.20]{hairer:2014:theory}. We first note that the model  norm is controlled by the regularity of the basis elements of $\check{\mathsf{T}}_{\mathrm{stoch}}$:
\begin{lem}\label{lem:red-to-stoch}
  For any $T > 0$, we have
  \begin{equation*}
    \|(\hat{\Pi}^{\eps,\zeta},\hat{\Gamma}^{\eps,\zeta})\|_{T} \lesssim_T \adjustlimits\sup_{\| g\|_{\mathcal{C}^{1,2}_{\mathrm{c}}} \leq 1}  \sup_{\tau \in \check{\mathsf{T}}_{\mathrm{stoch}}} \sup_{\x \in \mS_{2L}} \sup_{\lambda \in (0, 1)} \left( | \hat{\Pi}_{\x}^{\ve, \zeta} \tau
(g^{\lambda}_{\x}) | \lambda^{- | \tau |}   \right) .
\end{equation*}
\end{lem}
\begin{proof}
  This follows from the fact that $\check{\mathsf{T}}_{\mathcal{I}},
  \check{\mathsf{T}}_{\mathcal{I}'}$ are obtained from
  $\check{\mathsf{T}}_{\mathrm{stoch}}$ by integration and differentiation, and
  from the fact that the heat kernel regularizes by two degrees of regularity in
  parabolic scaling. Similarly, the analytic estimates on $\Gamma^{\ve, \zeta}$
  follow from the regularity of all the basis elements of $\check{\mathsf{T}}$,
  in view of \zcref{tab:model-defns}.
\end{proof}
Now, by following the same proof as that of \cites[Theorem 10.7]{hairer:2014:theory},
which is essentially an application of Gaussian hypercontractivity,
we can make the following reduction in order to prove \zcref{thm:model-norm-bounded}. (The major difference in this
setting is that we do not have spatial homogeneity, but this does
not play an important role in the proof.) 
\begin{prop}
\label{prop:hypercontractivity}For any $ \kappa \in (0, 1)$ there exists a $p_0(\kappa) \in
[2,\infty)$ such that for any $p \geq p_0(\kappa) $, we have
\begin{equation}
  \begin{aligned}
    \mathbb{E} &\left[ \adjustlimits\sup_{\| g\|_{\mathcal{C}^{1,2}_{\mathrm{c}}} \leq 1}  \sup_{\tau \in \check{\mathsf{T}}_{\mathrm{stoch}}}  \sup_{\x \in \mS_{2L}} \sup_{\lambda \in (0, 1)} | \hat{\Pi}_{\x}^{\ve, \zeta} \tau
(g^{\lambda}_{\x}) |^p \lambda^{- | \tau | p}  \right] \\
  & \lesssim_{p, L, T, \kappa} \adjustlimits\sup_{\|g \|_{\mathcal C^{1,2}_{\mathrm c}}\leq 1}\sup_{\tau\in\check{\mathsf{T}}_{\mathrm{stoch}}}\sup_{x\in \TT_{2L}}\sup_{\lambda\in(0,1)}\left(\lambda^{-2|\tau| - \kappa}\mathbb{E}|\llangle\hat{\Pi}^{\eps,\zeta}_{(0,x)}(\tau),g^{\lambda}_{(0,x)}\rrangle_{\mS_{2L}}|^{2}\right)^{p/2}.\label{eq:hypercontractivity}
  \end{aligned}
\end{equation}
\end{prop}
Note that the homogeneity $|\tau|$ of basis elements $\tau \in
  \check{\mathsf{T}}$ contains an arbitrarily small, but strictly positive
  parameter $\kappa$, which for convenience is chosen to be the same parameter that appears in this
  proposition. 
The right side of \zcref{eq:hypercontractivity} only depends on second
moments, and so the rest of this section will be on estimating
the second moments of $\hat{\Pi}^{\eps,\zeta}_{(0,x)}(\tau)$ for
the trees $\tau\in\check{\mathsf{T}}_{\mathrm{stoch}}$. In particular, we will
obtain the following result.
\begin{prop}\label{prop:stochastic-estimates}
  For each $\tau \in \check{\mathsf{T}}_{\mathrm{stoch}}$, we have
  \begin{equation*}
    \adjustlimits \sup_{\ve, \zeta \in (0, 1)} \sup_{\|g \|_{\mathcal C^{1,2}_{\mathrm c}}\leq 1}\sup_{x\in \TT_{2L}}\sup_{\lambda\in(0,1)} \lambda^{-2|\tau| - \kappa}\mathbb{E}|\llangle\hat{\Pi}^{\eps,\zeta}_{(0,x)}(\tau),g^{\lambda}_{(0,x)}\rrangle_{\mS_{2L}}|^{2} < \infty .
  \end{equation*}
\end{prop}
\begin{proof}
  This is a simple computation for $\tau\in\check{\mathsf{T}}_{\bullet}$.
  For each of the (nine; see \zcref{tab:RS-table}) $\tau\in \check{\mathsf{T}}_{\star}$, the result is a consequence of the estimates in
  \zcref{prop:cherryrr-bound,prop:claw-bound,prop:rbrelk,prop:rrbelk,prop:elkrrr,prop:candelabra,prop:moose}
  below, as shown in the following table:
\[
\begin{tabular}{c|cccccccccc}
  \toprule
  $\tau$ &$\rscherryrr$&$\rscherryrb$&$\rscherrybb$&$\rsclawrr$&$\rselkrbr$&$\rselkrrb$&$\rselkrrr$&$\rscandelabrarrrr$&$\rsmooserrrr$\\
  \midrule
  Estimate & \zcref{eq:cherry-model-bound} &\zcref{eq:redbluecherrybd}&\zcref{eq:lollipopb-bound} & \zcref[range]{eq:claw-exp-bd,eq:claw-cov-bd} & \zcref[range]{e:rbr-aim,eq:elkrbrcovbd} & \zcref{e:rrb-aim} & \zcref[range]{eq:elkrrr-exp,eq:elkrrr-cov} & \zcref[range]{eq:candelabra-cov-bd,eq:candelabra-mean-bd} & \zcref{e:moose-aim-2}  \\
  \bottomrule
\end{tabular}.\qedhere
\]
\end{proof}
Before beginning the stochastic estimates in earnest, %
we 
introduce some notation and preliminary results.
For $\mathbf{x}_{i}=(t_{i},x_{i}) \in \mS_{2L}$, $i=1,2$, we define 
\begin{equation}
\mathcal{E}^{\zeta}_{\pm}(\mathbf{x}_{1},\mathbf{x}_{2})\coloneqq\delta(t_{1}-t_{2})\Sh^{\zeta}_{2L}(x_{1}\pm x_{2})\label{eq:Ezetapmdef}
\end{equation}
and then
\begin{equation}
\mathcal{E}^{\zeta}(\mathbf{x}_{1},\mathbf{x}_{2})\coloneqq\mathcal{E}^{\zeta}_{+}(\mathbf{x}_{1},\mathbf{x}_{2})+\mathcal{E}^{\zeta}_{-}(\mathbf{x}_{1},\mathbf{x}_{2}),\label{eq:Ezetadef}
\end{equation}
which is the covariance function for the reflected noise, the quantity appearing %
on the right side of \zcref{eq:dWzetacov}. %
We will also extend the definition \zcref{eq:sym_group_generators}
by 
\[
\sr(t,x)\coloneqq(t,\sr(x))\qquad\text{for }(t,x)\in\mS_{2L}.
\]
We define the parabolically-scaled periodic norm by
\[
    |\x|_{\mathfrak{s}} =|(t,x)|_{\mathfrak{s}}\coloneqq\sqrt{t}+ \min_{z \in 2L \ZZ}|x -z| ,
\]
which induces a metric on $\mS_{2L}$ and a pseudo-metric on $\RR\times \RR$. With this metric we define balls
\begin{equation}\label{eq:Urdef}
    \mathcal{U}_{c}(\x) \coloneqq \{ \y \text{ such that } |\x - \y|_{\mf{s}} \leq c\} \subseteq \mS_{2L}.
\end{equation}
We also define the pseudo-metric
\begin{equation}
d_{\mathfrak{s},\mathscr{S}}(\mathbf{x}_{1},\mathbf{x}_{2})\coloneqq \min\left\{ |\mathbf{x}_{1}-\sr(\mathbf{x}_{2})|_{\mathfrak{s}}, |\mathbf{x}_{1}-\mathbf{x}_{2}|_{\mathfrak{s}} \right\},\label{eq:dsS}
\end{equation}
which will show up later in the calculations due to the even extension of the noise. 
We note that $d_{\mathfrak{s},\mathscr{S}}$ is a true metric when
restricted to $\mathbb{R}\times[0,L)$. Finally, we introduce the periodized kernel
\begin{equation}
  J(\x)= J(t,x) \coloneqq \sum_{z \in 2L\ZZ} K(t, x+z),\label{eq:Jdef}
\end{equation}
so that $J$ is defined on $\mS_{2L}$.
By the properties of the kernel $K$ imposed in \zcref{subsec:The-kernels},
as well as standard properties of the heat kernel, we see that 
\begin{equation}
|J(\mathbf{x})|\lesssim|\mathbf{x}|^{-1}_{\mathfrak{s}} , \qquad|J'(\mathbf{x})|\lesssim|\mathbf{x}|^{-2}_{\mathfrak{s}} , \qquad\text{and}\qquad|J''(\mathbf{x})|\lesssim|\mathbf{x}|^{-3}_{\mathfrak{s}}\qquad\text{for all }\x\in \mathbb{S}_{2L}.\label{eq:Kbds}
\end{equation}
When we deal with the recentering of trees that takes place in the model definition (see \zcref{tab:model-defns}), we will often need to deal with a ``recentered'' version of $J$. We define, for $\x,\w\in\mathbb{S}_{2L}$,
\begin{equation}\label{eq:recenter-J-def}
  J_{\x}(\w) \coloneqq J(\w)-J(\w+\x).
\end{equation}
The following estimate will be used frequently throughout the analysis.
\begin{lem}\label{lem:Kx-bd}
For any $\delta \in [0,1]$ and $ \x, \w \in \mS_{2L}$, we have
\begin{equation}\label{e:Kx-bd}
  |J'_\x(\w)|\lesssim |\x|^\delta_{\mf{s}} (|\w|_{\mf{s}}^{-2 - \delta}+|\w +\x|_{\mf{s}}^{-2 - \delta})  .
\end{equation}
\end{lem}
\begin{proof}
  Let $\x=(t,x)$.
We consider two cases. First, suppose that $|\x|_{\mf{s}} \leq
\oh|\w|_{\mf{s}}$. In this case, for some $\xi \in [0, 1]$ we have
\begin{equation*}
  |J'(\w) - J'(\w + (0, x))|  = |J''(\w + \xi (0, x)) x| \lesssim |\w|_{\mf{s}}^{-3} |x|\lesssim |\w|_{\mf s}^{-2-\delta}|\x|_{\mf s}^\delta
  \end{equation*}
 since $|\w + \xi (0, x)|_{\mf{s}}, |\w + \x|_{\mf{s}} 
\gtrsim |\w|_{\mf{s}}$ Similarly, using in addition the estimate $|\partial_s J' (s, y)|
\lesssim |(s, y)|_{\mf{s}}^{-4}$, we have
  \begin{equation*}
    |J'(\w+ \x) - J'(\w + (0, x))|  \lesssim |\w|_{\mf s}^{-4}|t|\lesssim |\w|_{\mf s}^{-4}|\x|_{\mf s}|t|^{1/2} \lesssim |\w|_{\mf{s}}^{-2-\delta} |\x|_{\mf s}^{\delta}
\end{equation*}
Then in this case \zcref{e:Kx-bd} follows
from the triangle inequality.
On the other hand, if $|\w|_{\mf{s}} <  2 |\x|_{\mf{s}}$, then we use the triangle inequality to bound
\begin{align*}
  |J'(\w) - J'(\w+\x)| & \lesssim |\w|_\mf{s}^{-2-\delta} |\w|_\mf{s}^{\delta} + |\w+\x|_\mf{s}^{-2-\delta}|\w+\x|_\mf{s}^{\delta} \\
  & \lesssim |\x|^\delta_\mf{s} (|\w|_\mf{s}^{-2-\delta}  + |\w+\x|_\mf{s}^{-2-\delta}) .\qedhere
  \end{align*}
\end{proof}

\subsection{Feynman diagrams\label{subsec:feynman-diagrams}}

We now extend the graphical notation introduced in \zcref{subsec:Tree-diagrams}
to help formulate our estimates. For each $\tau\in\check{\mathsf{T}}_{\star}$,
we associate $\tau$ with a directed graph $\mathsf{G}(\tau)=(\mathsf{V}(\tau),\mathsf{E}(\tau))$
with vertex set $\mathsf{V}(\tau)$ and edge set $\mathsf{E}(\tau)$.
The edges $e=(e_{\downarrow},e_{\uparrow})\in\mathsf{E}(\tau)$ are
directed, with $e_{\downarrow}$ and $e_{\uparrow}$ representing
the vertices closer to and farther away from the root (lower and higher
in our diagrams), respectively. The leaf nodes of the tree can be
written as $\mathsf{V}_{\rsnoise}(\tau)\sqcup\mathsf{V}_{\rspotential}(\tau)$,
where $\mathsf{V}_{\rsnoise}(\tau)$ and $\mathsf{V}_{\rspotential}(\tau)$
contain the nodes representing $\rsnoise$ and $\rspotential$ terms,
respectively. %

For trees $\tau_{1},\tau_{2}\in\check{\mathsf{T}}_{\mathcal{I}'}\cup\check{\mathsf{T}}_{\star}$,
we define the set $\mathscr{C}(\tau_{1},\tau_{2})$ of \emph{contractions}
to be the set of all perfect matchings of $\mathsf{V}_{\rsnoise}(\tau_{1})\sqcup\mathsf{V}_{\rsnoise}(\tau_{2})$.
Given a contraction $\nu\in\mathscr{C}(\tau_{1},\tau_{2})$, we define
a new ``contracted'' graph $\tau_{1}\hash_{\nu}\tau_{2}\coloneqq(\mathsf{V}(\tau_{1})\sqcup\mathsf{V}(\tau_{2}),\mathsf{E}(\tau_{1})\sqcup\mathsf{E}(\tau_{2})\sqcup\nu)$
by taking the union of the graphs $\mathsf{G}(\tau_{1})$ and $\mathsf{G}(\tau_{2})$
and then adding edges corresponding to the matching $\nu$. With some abuse of notations, we will also use $\nu$ to denote the collection of all new edges corresponding to the matching. We consider
the edges added from $\nu$ to be \emph{undirected}. Symbolically,
we draw edges corresponding to terms in $\nu$ as dashed red lines.
\nomenclature{$\rslollirc$}{Contraction}
Here are some examples of symbolic representations of contracted trees:
\[
\rslollirc,\qquad\rselkrbrcSN,\qquad \rsclawrrcNORECENTER,\qquad \rscherryrrcc,\qquad\rselkrrbcc,\qquad \rselkrrrccccc, \qquad \rsmooserrrrcPD.
\]
As the third example shows, the graph $\tau_{1}\hash_{\nu}\tau_{2}$
need not be connected. 

We write
\[
\mathsf{V}_{\square}(\tau_{1}\hash_{\nu}\tau_{2})\coloneqq\mathsf{V}_{\square}(\tau_{1})\sqcup\mathsf{V}_{\square}(\tau_{2})\qquad\text{for }\square\in\{\rsnoise,\rspotential\},
\]
\begin{equation}
\mathsf{V}_{\varrho}(\tau_{1}\hash_{\nu}\tau_{2})\coloneqq\{\varrho(\tau_{1}),\varrho(\tau_{2})\},%
\end{equation}
where $\varrho(\tau_{i})$ is the root of $\tau_i$, and
\[
\mathsf{V}_{\mathrm{int}}(\tau_{1}\hash_{\nu}\tau_{2})\coloneqq\mathsf{V}(\tau_{1}\hash_{\nu}\tau_{2})\setminus\mathsf{V}_{\varrho}(\tau_{1}\hash_{\nu}\tau_{2}),
\]
as well as
\[
\mathsf{E}_{\mathcal{I}'}(\tau_{1}\hash_{\nu}\tau_{2})\coloneqq\mathsf{E}(\tau_{1})\sqcup\mathsf{E}(\tau_{2}).
\]

We now define the (deterministic) distribution represented by the contraction $\boxop{\tau_{1}\hash_{\nu}\tau_{2}}^{\eps,\zeta}$.
This represents the term in the Isserlis theorem expansion of $\mathbb{E}\left[\boxop{\tau_{1}}^{\eps,\zeta}\boxop{\tau_{2}}^{\eps,\zeta}\right]$
corresponding to the matching $\nu$. Specifically, we use the Isserlis theorem to write
\begin{equation}
\mathbb{E}\left[\boxop{\tau_{1}}^{\eps,\zeta}(\mathbf{x}_{\rho(\tau_{1})})\boxop{\tau_{2}}^{\eps,\zeta}(\mathbf{x}_{\rho(\tau_{2})})\right]=\sum_{\nu\in\mathscr{C}(\tau_{1},\tau_{2})}\boxop{\tau_{1}\hash_{\nu}\tau_{2}}^{\eps,\zeta}(\mathbf{x}_{\varrho(\tau_{1})},\mathbf{x}_{\varrho(\tau_{2})}), \label{eq:Isserlis-contractions}
\end{equation}
where we define
\begin{align}
& \boxop{\tau_{1}\hash_{\nu}\tau_{2}}^{\eps,\zeta}(\mathbf{x}_{\varrho_{1}},\mathbf{x}_{\varrho_{2}})\notag\\
 & \qquad \coloneqq\int_{(\mathbb{R}^{2})^{\mathsf{V}_{\mathrm{int}}}}\left(\prod_{e\in\mathsf{E}_{\mathcal{I}'}(\tau_{1}\hash_{\nu}\tau_{2})}K'(\mathbf{x}_{e_{\downarrow}}-\mathbf{x}_{e_{\uparrow}})\right)\left(\prod_{e\in\nu}\mathcal{E}^{\zeta}(\mathbf{x}_{e_{1}},\mathbf{x}_{e_{2}})\right)\left(\prod_{v\in V_{\rspotential}}\varphi^{\eps}_{\uu,\vv}(x_{v})\right)\prod_{v\in\mathsf{V}_{\mathrm{int}}}\,\dif\mathbf{x}_{v} \notag\\
 & \qquad = \int_{\mS_{2L}^{\mathsf{V}_{\mathrm{int}}}}\left(\prod_{e\in\mathsf{E}_{\mathcal{I}'}(\tau_{1}\hash_{\nu}\tau_{2})}J'(\mathbf{x}_{e_{\downarrow}}-\mathbf{x}_{e_{\uparrow}})\right)\left(\prod_{e\in\nu}\mathcal{E}^{\zeta}(\mathbf{x}_{e_{1}},\mathbf{x}_{e_{2}})\right)\left(\prod_{v\in V_{\rspotential}}\varphi^{\eps}_{\uu,\vv}(x_{v})\right)\prod_{v\in\mathsf{V}_{\mathrm{int}}}\dif\mathbf{x}_{v} .
\label{eq:define-contraction}
\end{align}
Here, to lighten the notation we have used the shorthand $\mathsf{V}_{\mathrm{int}}$ for $\mathsf{V}_{\mathrm{int}}(\tau_{1}\hash_{\nu}\tau_{2})$ and similarly for all other vertex sets, and we have denoted $\mathbf{x}_{v}=(t_{v},x_{v})$ for $v\in\mathsf{V}_{\mathrm{int}}(\tau_{1}\hash_{\nu}\tau_{2})$,
$\varrho_{i}=\varrho(\tau_{i})$ for $i=1,2$. (This notation will
be particularly useful when $\tau_{1}$ and $\tau_{2}$ are the same tree $\tau$,
in which case the $\varrho_{i}$s are different elements of the disjoint
union.) Moreover, we have written $e=(e_{1},e_{2})$ for $e\in\nu$, where the order is arbitrary
since $\mathcal{E}^{\zeta}$ is symmetric (but each pair of vertices
only appears in the product \emph{once}). Finally, the last identity in \zcref{eq:define-contraction} follows from the fact that $\mE$ and $\varphi$ are $2L$-periodic and the fact that for any $2L$-periodic function $f$, one has
\begin{equation}\label{e:periodic-ind}
    \int_{\RR^2} K'(\x - \y) f(\y) \, \dif \y = \int_{\mS_{2L}} J'(\x - \y) f(\y) \, \dif \y .
\end{equation}
Indeed, note that one can start by applying this formula to the contracted leaf variables, where one has
\begin{equation}\label{e:step}
  \int_{\RR^2} K'(\x_{\iota(e_1)} - \x_{e_1}) \mE^\zeta(\x_{e_1}, \x_{e_2}) \ud \x_{e_1} = \int_{\mS_{2L}} J'(\x_{\iota(e_1)} - \x_{e_1}) \mE^\zeta(\x_{e_1}, \x_{e_2}) \ud \x_{e_1} ,
\end{equation}
where $\iota(e_1)$ is the parent of $e_1$. Then after applying this formula to all the variables in $V_{\rspotential} \sqcup V_{\rsnoise}$ one can proceed in the same manner, inductively from the leaves towards the roots. In fact, the result of \zcref{e:step} is again a periodic function, and every inner vertex has exactly degree three and is connected to two children and only one parent in its tree, meaning that in the induction one obtains integrals of the form
\begin{equation*}
  \int_{\RR^2} K'(\x_{\iota(v)} - \x_v) f_1(\x_v) f_2(\x_v) \ud \x_v ,
\end{equation*}
where $f_1, f_2$ are periodic and $\iota(v)$ is the parent of $v$ in its tree, so that \zcref{e:periodic-ind} can be applied again.

In the sequel, we will need to bound quite a number of integrals of the form \zcref{eq:define-contraction}. 
A first useful estimate will be the following bound on the convolution of kernels appearing in \zcref{eq:Kbds}.
\begin{lem}
\label{lem:one-int-bound}Let $\alpha,\beta\in(0,3)$ and let $\Theta\subset\mS_{2L}^{2}$
be a compact set. Then, locally uniformly over $\x_1,\x_2\in \mS_{2L}$,
we have
\begin{equation}
    \int_{\Theta}|\x_1-\z|^{-\alpha_1}_{\mathfrak{s}}|\x_2-\z|^{-\alpha_2}_{\mathfrak{s}}\,\dif\mathbf{z}
  \lesssim\begin{cases}
|\x_1-\x_2|^{-\alpha_1-\alpha_2+3}_{\mathfrak{s}}, & \alpha_1+\alpha_2>3;\\
\log(2+|\x_1-\x_2|^{-1}_{\mathfrak{s}}), & \alpha_1+\alpha_2=3;\\
1, & \alpha_1+\alpha_2<3.
\end{cases}\label{eq:boundsingle}
\end{equation}
  Similarly for the pseudo-metric $d_{\mf{s},\mathscr{S}}$ defined in \zcref{eq:dsS},
  we have
  \begin{equation}
    \int_{\Theta} d_{\mf{s}, \mathscr{S}} (\x_1, \z)^{-\alpha_1} d_{\mf{s}, \mathscr{S}}( \x_2, \z)^{-\alpha_2} \,\dif\mathbf{z}
    \lesssim\begin{cases}
  d_{\mf{s}, \mathscr{S}}(\x_1, \x_2)^{-\alpha_1-\alpha_2+3}, & \alpha_1+\alpha_2>3;\\
  \log(2+d_{\mf{s}, \mathscr{S}}(\x_1, \x_2)^{-1}), & \alpha_1+\alpha_2=3;\\
  1, & \alpha_1+\alpha_2<3.
  \end{cases}\label{eq:boundsingle-dsS}
  \end{equation}
\end{lem}
\begin{proof}
  We leave the bound \zcref{eq:boundsingle} without proof, as if follows for
  example along similar but much simpler arguments to the proof of \zcref{prop:hepp-prop} below. As for the second bound \zcref{eq:boundsingle-dsS}, we may assume by enlarging the domain that $\Theta = [-c,c]\times \TT_{2L}$ for some $c\in (0,\infty)$, and then we note that
  \begin{equation*}
    \begin{aligned}
      \int_{\Theta} d_{\mf{s}, \mathscr{S}} (\x_1, \z)^{-\alpha_1} d_{\mf{s}, \mathscr{S}}( \x_2, \z)^{-\alpha_2} \,\dif\mathbf{z} \lesssim &\int_{\Theta}|\x_1-\z|^{-\alpha_1}_{\mathfrak{s}}|\x_2-\z|^{-\alpha_2}_{\mathfrak{s}}\,\dif\mathbf{z} + \int_{\Theta}|\sr \x_1-\z|^{-\alpha_1}_{\mathfrak{s}}|\x_2-\z|^{-\alpha_2}_{\mathfrak{s}}\,\dif\mathbf{z} \\
      + \int_{\Theta} & |\x_1-\z|^{-\alpha_1}_{\mathfrak{s}}|\sr \x_2-\z|^{-\alpha_2}_{\mathfrak{s}}\,\dif\mathbf{z} + \int_{\Theta}|\sr \x_1-\z|^{-\alpha_1}_{\mathfrak{s}}|\sr \x_2-\z|^{-\alpha_2}_{\mathfrak{s}}\,\dif\mathbf{z} .
    \end{aligned}
  \end{equation*}
  In the case $\alpha_1 + \alpha_2 >3$ we use
  \zcref{eq:boundsingle} to obtain
  \begin{equation*}
    \begin{aligned}
      \int_{\Theta} d_{\mf{s}, \mathscr{S}} (\x_1, \z)^{-\alpha_1} d_{\mf{s}, \mathscr{S}}( \x_2, \z)^{-\alpha_2} \,\dif\mathbf{z}  \lesssim & |\x_1 - \x_2|_{\mf{s}}^{- \alpha_1 -\alpha_2 + 3} + | \sr \x_1 - \x_2|_{\mf{s}}^{- \alpha_1 -\alpha_2 + 3}
       + |\x_1 - \sr \x_2|_{\mf{s}}^{- \alpha_1 -\alpha_2 + 3} \\
      & + |\sr \x_1 - \sr \x_2|_{\mf{s}}^{- \alpha_1 -\alpha_2 + 3} \lesssim d_{\mf{s}, \mathscr{S}}(\x_1, \x_2)^{- \alpha_1 -\alpha_2 + 3} .
    \end{aligned}
  \end{equation*}
  The other cases follow similarly.
\end{proof}

However, \zcref{lem:one-int-bound} is not sufficient for the more complicated trees, and we will require some more sophisticated     estimates that are presented in \zcref{subsec:Convergent-Feynman-diagrams} below. First, we have to derive some simple bounds on some small trees, which in particular will serve as building blocks for the proofs in \zcref{subsec:Convergent-Feynman-diagrams}.

\subsection{Simple terms} %

In this section we derive bounds on $\rslollipopr[rsK]^{\zeta}$,
$\rslollipopb[rsK]^{\eps}$, $\rscherrybb[rsK]^\eps$, $\rscherryrb[rsK]^{\eps,\zeta}$, and $\rscherryrr[rsK]^\zeta$. %
Very similar bounds on $\rslollipopr[rsP]^{\zeta}$
and $\rslollipopb[rsP]^{\eps}$ have been proved in \zcref{lem:EW-cov,prop:bluecherrycontrib}.
In particular, it follows immediately from \zcref{eq:lollipopb-explicit}
and the fact that $\rslollipopb[rsK]^{\eps}-\rslollipopb[rsP]^{\eps}$
is smooth, uniformly in $\eps$, that 
\begin{equation}
  \left|\rslollipopb[rsK]^{\eps}(\mathbf{x})\right|,\left|\rscherrybb[rsK]^{\eps}(\mathbf{x})\right|\lesssim1,\qquad\text{uniformly in $\x$ and $\eps$.}\label{eq:lollipopb-bound}
\end{equation}
We now prove an estimate on the covariance of $\rslollipopr[rsK]^\zeta$ that is analogous to  \zcref{lem:EW-cov}, along with a few consequences.
We define a new symbol $\rstriangle$,  %
with the (deterministic) realization 
\begin{equation}\label{eq:black-triangle-def}
  \rstriangle[rsK]^\zeta (\y ) \coloneqq \rscherryrrc[rsK]^{\zeta}(\mathbf{y}) - C^{(1)}_\zeta (\y) \overset{\zcref{eq:exp-renorm}}= \mathbb{E}\rsrenorm[rsK]^\zeta(\y).
\end{equation}\nomenclature{$\rstriangle$}{Expectation of $\rsrenorm$}
Furthermore, let us define the distance from the boundary
\begin{equation}\label{e:dLdef}
  d_L(y) = \min_{k \in L\ZZ} |y - k|\qquad\text{for $y\in\TT_{2L}$} .
\end{equation}
  We note that this does \emph{not} coincide with the distance from zero in the torus, because we are choosing $k \in L \ZZ$ rather than $k \in 2L \ZZ$. Instead, this is the distance from the boundaries where the reflection occurs.

\begin{prop}
  \label{prop:lollibd}The difference
  \begin{equation}\label{eq:lollirc-diff-bd}
    \rslollirc[rsK]^\zeta(\x_1,\x_2)-\rslollirc[rsP]^\zeta(\x_1,\x_2)\text{ is bounded and Hölder equicontinuous,}
  \end{equation}
  uniformly in $\x_1,\x_2\in\mS_{2L}$ and $\zeta\in(0,1]$. In particular
  we have, uniformly over $\x_1, \x_2 \in \mS_{2L}$ and $\zeta\in (0,1]$, that
\begin{equation}\label{eq:redcherryx1x2bd}
\left|\rslollirc[rsK]^{\zeta}(\mathbf{x}_{1},\mathbf{x}_{2})\right|\lesssim d_{\mathfrak{s},\mathscr{S}}(\mathbf{x}_{1},\mathbf{x}_{2})^{-1}
\end{equation}
and
\begin{equation}\label{eq:redbluecherrybd}
\left|\rscherryrbc[rsK]^{\eps,\zeta}(\mathbf{x}_{1},\mathbf{x}_{2})\right|\lesssim d_{\mathfrak{s},\mathscr{S}}(\mathbf{x}_{1},\mathbf{x}_{2})^{-1}.
\end{equation}
Similarly, we have
\begin{equation}
\left\lvert\rstriangle[rsK]^\zeta(\y)\right\rvert=\left\lvert\rscherryrrc[rsK]^{\zeta}(\mathbf{y})-C^{(1)}_{\zeta}(\mathbf{y})\right\rvert\lesssim1,\label{eq:renorm-constants-same}
\end{equation}
uniformly over $\zeta\in(0,1]$ and $\mathbf{y}\in\mS_{2L}$.
Finally, for any $s \neq t \in \RR$ and $c\in(0,\infty)$, we have
\begin{equation}\label{e:newrbc}
  \adjustlimits\lim_{\ve \downarrow 0} \sup_{\zeta \in (0, \ve)} \sup_{d_L(x),d_L(y) \leq c \ve} \; \left\vert \rscherryrbc[rsK]^{\eps,\zeta}( (t,x), (s, y)) \right\vert = 0 ,
\end{equation}
where $d_L$ is defined in \zcref{e:dLdef}.
\end{prop}

\begin{proof}
We have
\[
\rslollipopr[rsK]^{\zeta}(\mathbf{x})=((p'-\tilde{K}')\circledast\dif W^{\zeta})(\mathbf{x})=\rslollipopr[rsP]^{\zeta}(\mathbf{x})-(\tilde{K}'\circledast\dif W^{\zeta})(\mathbf{x}).
\]
This means that
\begin{equation}\label{eq:lollipop-cov-decomp}
  \begin{aligned}
\rslollirc[rsK]^{\zeta}(\mathbf{x}_{1},\mathbf{x}_{2}) & =\mathbb{E}\left[\rslollipopr[rsP]^{\zeta}(\mathbf{x}_{1})\rslollipopr[rsP]^{\zeta}(\mathbf{x}_{2})\right]-\mathbb{E}\left[\rslollipopr[rsP]^{\zeta}(\mathbf{x}_{1})(\tilde{K}'\circledast\dif W^{\zeta})(\mathbf{x}_{2})\right]\\
 & \qquad-\mathbb{E}\left[(\tilde{K}'\circledast\dif W^{\zeta})(\mathbf{x}_{1})\rslollipopr[rsP]^{\zeta}(\mathbf{x}_{2})\right]+\mathbb{E}\left[(\tilde{K}'\circledast\dif W^{\zeta})(\mathbf{x}_{1})(\tilde{K}'\circledast\dif W^{\zeta})(\mathbf{x}_{2})\right].
\end{aligned}
\end{equation}
We recall from \zcref{lem:EW-cov} that \[\mathbb{E}\left[\rslollipopr[rsP]^{\zeta}(\mathbf{x}_{1})\rslollipopr[rsP]^{\zeta}(\mathbf{x}_{2})\right] = p_{|t_1-t_2|}*\Sh^\zeta_{2L}(x_1-x_2)-p_{|t_1-t_2|}*\Sh^\zeta_{2L}(x_1+x_2).\] On the other hand, the last term of \zcref{eq:lollipop-cov-decomp} is bounded and Hölder
continuous, uniformly in $\mathbf{x}_{1},\mathbf{x}_{2},\zeta$, by
the smoothness and rapid decay of $\tilde{K}'$. For the cross-terms, we can write
(for the first one, the other one being symmetrical)
\begin{align*}
\mathbb{E} & \left[\rslollipopr[rsP]^{\zeta}(\mathbf{x}_{1})(\tilde{K}'\circledast\dif W^\zeta)(\mathbf{x}_{2})\right]\\
 & =\int^{t_{1}\wedge t_{2}}_{-\infty}\iint p_{t_1-s}'(x_1-y_{1})\tilde{K}'_{t_2-s}(x_{2}-y_{2})(\Sh^{\zeta}(y_{1}-y_{2})+\Sh^{\zeta}(y_{1}+y_{2}))\,\dif y_{1}\,\dif y_{2}\,\dif s\\
 & =\int^{t_{1}\wedge t_{2}}_{-\infty}\left(p_{t_1-s}'*\tilde{K}_{t_2-s}'*\Sh^{\zeta}(x_{1}-x_{2})-p_{t_1-s}'*\tilde{K}_{t_2-s}'*\Sh^{\zeta}(x_{1}+x_{2})\right)\,\dif s,
\end{align*}
and this term is also Hölder continuous, uniformly in $\mathbf{x}_{1},\mathbf{x}_{2},\zeta$,
by the smoothness and rapid decay of $\tilde{K}'$. These estimates together
imply \zcref{eq:lollirc-diff-bd,eq:redcherryx1x2bd,eq:renorm-constants-same}. The estimate 
\zcref{eq:redbluecherrybd} then follows from \zcref{eq:lollipopb-bound}. The
estimate \zcref{e:newrbc} then follows by using the Hölder continuity estimates
(and the Hölder continuity of $p_t(x)$ in $x$ for fixed $t>0$) along with the
fact that each of the terms on the right side of \zcref{eq:lollipop-cov-decomp}
is zero when either $\x_1 = (t_1, x_1) \in \{(t_1, 0), (t_1, L)\}$ or $\x_2 =
(t_2, x_2) \in \{(t_2, 0), (t_2, L)\}$.
\end{proof}

Now we can address the model applied to $\rscherryrr$.
\begin{prop}
\label{prop:cherryrr-bound}We have 
\begin{equation}
\left|\mathbb{E}\left[\hat{\Pi}^{\eps,\zeta}_{\mathbf{x}}(\rscherryrr)(\mathbf{y}_{1})\cdot\hat{\Pi}^{\eps,\zeta}_{\mathbf{x}}(\rscherryrr)(\mathbf{y}_{2})\right]\right|\lesssim d_{\mathfrak{s},\mathscr{S}}(\mathbf{y}_{1},\mathbf{y}_{2})^{-2}.\label{eq:cherry-model-bound}
\end{equation}
\end{prop}

\begin{proof}
  We can compute, referring to \zcref{tab:model-defns} and recalling the definition \zcref{eq:black-triangle-def}, that 
\begin{equation}
  \mathbb{E}\left[\hat{\Pi}^{\eps,\zeta}_{\mathbf{x}}(\rscherryrr)(\mathbf{y}_{1})\cdot\hat{\Pi}^{\eps,\zeta}_{\mathbf{x}}(\rscherryrr)(\mathbf{y}_{2})\right]=\left(\left(\rstriangle[rsK]^\zeta\right)^{\otimes 2}+\textcolor{blue}{2}\rscherryrrcc[rsK]^{\zeta}\right)(\mathbf{y}_{1},\mathbf{y}_{2}).\label{eq:split-cherry-contractions}
\end{equation}
By \zcref{eq:redcherryx1x2bd}, we have 
\begin{equation}
\left|\rscherryrrcc[rsK]^{\zeta}(\mathbf{y}_{1},\mathbf{y}_{2})\right|\lesssim d_{\mathfrak{s},\mathscr{S}}(\mathbf{y}_{1},\mathbf{y}_{2})^{-2}.\label{eq:bound-cherry-contraction}
\end{equation}
Using \zcref{eq:renorm-constants-same} we complete the proof.
\end{proof}

\subsection{\label{subsec:Convergent-Feynman-diagrams}Convergent Feynman diagrams}

To estimate all of the integrals of interest, we will
require a generalization of \zcref{lem:one-int-bound} to arbitrary \emph{convergent} Feynman diagrams which appear in \zcref{eq:Isserlis-contractions}. Divergent diagrams, which must be treated with renormalization, also arise in the computations, and these are handled on a case-by-case basis as they appear. Our estimate, \zcref{prop:graph-bound}, is an analogue of known results, but it provides a quantitative bound in terms of the distance between the two root vertices, in contrast to existing estimates in the literature that usually integrate over all vertices. %
See for example \cite{weinberg:1960:high}, or \cite[Theorem A.3]{hairer:quastel:2018:class} and \cite[Proposition 2.3]{hairer:2018:analyst} for more modern accounts. In our setting concerning correlation functions, the resulting estimates contain logarithmic factors that do not appear when integrating over all variables.

Our estimates on a given two-point function $\boxop{\tau\hash_{\nu}\tau}^{\ve, \zeta}$ will be rephrased in terms of combinatorial properties of the contracted graph $\tau\hash_{\nu}\tau$,
and in particular in terms of a notion of \emph{degree}. It turns out to be simpler to define the degree on a somewhat simplified \emph{multi}graph $\mathsf{G}_{\mathrm{c}}$ derived from $\mathsf{G} = \tau\hash_\nu\tau$. Each edge $e$ of $\mathsf{G}$ will be associated with a weight $\mathsf{w}(e)$, representing the degree of blow-up of the kernel associated to that edge at the origin. (Up to this point, the only kernel we have considered is $J'$, which has weight $2$ according to \zcref{eq:Kbds}, but we will later replace this kernel %
by slightly modified kernels that could exhibit  possibly stronger singularities.)
We obtain $\mathsf{G}_{\mathrm{c}}$ from $\mathsf{G}$ by removing each copy of $\rslollipopb$ (as they are uniformly bounded according to \zcref{eq:lollipopb-bound}) and replacing each copy of $\rslollirc$ by an edge of weight $1$ (in accordance with \zcref{eq:redcherryx1x2bd}), which we will draw in purple in the diagrams. For example,
if $\mathsf{G} = \rselkrrbcc$, then $\mathsf{G}_{\mathrm{c}}$ can be represented as
$
  \begin{tikzpicture}[rstree,elk,hascontractions]
    \tikzset{
      gcnode/.style={circle, fill=black, draw=black, inner sep=0.08em},
    }
    \node[root] (root1) at (0,0) {}
    child [dummy] { }
      child [I'] {
        node[gcnode] (x2) {}
      }
    ;
    \node[root] (root2) at (0.3,0) {}
    child [I'] {
      node[gcnode] (x4) {}
    }
    child [dummy] { }
  ;
  \draw[purple,thick] (x2) to[out=90,in=90, looseness=.6] (x4);
  \draw[purple,thick] (x2) to[out=90,in=90, looseness=1.5] (x4);
  \end{tikzpicture}%
  $, where the two $\bullet$ nodes correspond the roots of the two $\rscherryrr$. We emphasize that since $\mathsf{G}_{\mathrm{c}}$ is a multigraph, pairs of vertices may be connected by multiple edges. For $u,v\in\mathsf{V}(\mathsf{G}_{\mathrm{c}})$, we define $Q(u,v)$ to be the sum of the weights of all edges connecting $u$ and $v$, which describes the singularity of the kernel connecting the two variables corresponding to $u,v$. Then, for any $\overline{\mathsf V}\subseteq\mathsf{V}(\mathsf{G}_{\mathrm{c}})$, we define the degree
  \begin{equation}
    \deg(\overline{\mathsf{V}}) = 3(|\overline{\mathsf{V}}|-1) - \sum_{\{u,v\}\in \binom{\overline{\mathsf{V}}}{2}} Q(u,v).\label{eq:degree-Vbar}
  \end{equation}

The condition that guarantees that a Feynman diagram is convergent
is that there is no sub-diagram with negative degree. In this 
case, we can prove the following quantitative bound on the integral
associated to a given graph.
\begin{prop}
\label{prop:graph-bound}Let $\tau\in\check{\mathsf{T}}_{\star}$
and let $\nu\in\mathscr{C}(\tau,\tau)$ be a contraction. Let $(\mathsf{V},\mathsf{E})=\mathsf{G}=\tau\hash_{\nu}\tau$ and define the multigraph $\mathsf{G}_{\mathrm{c}} = (\mathsf{V}_{\mathrm{c}},\mathsf{E}_{\mathrm{c}})$. %
Assume that 
\begin{equation}
  \deg(\overline{\mathsf{V}})>0\qquad\text{ for all   }\overline{\mathsf{V}}\subseteq\mathsf{V}_{\mathrm{c}}\text{ such that }\textcolor{blue}{|\overline{\mathsf{V}}|}\ge2,\label{eq:degcondition}
\end{equation}
and define
\begin{equation}
  \gamma\coloneqq\max\left\{ 3-\deg(\overline{\mathsf{V}})\st\overline{\mathsf{V}}\subseteq\mathsf{V}_{\mathrm{c}}\text{ and }\mathsf{V}_{\varrho}(\mathsf{G})\subseteq\overline{\mathsf{V}}\right\} .\label{eq:gammadef}
\end{equation}
Then we have
\begin{equation}
  \left|\boxop{\tau\hash_{\nu}\tau}^{\eps,\zeta}(\mathbf{x}_{1},\mathbf{x}_{2})\right|\lesssim d_{\mathfrak{s},\mathscr{S}}(\mathbf{x}_{1},\mathbf{x}_{2})^{-\gamma}\left(\log\left(2+d_{\mathfrak{s},\mathscr{S}}(\mathbf{x}_{1},\mathbf{x}_{2})^{-1}\right)\right)^{k_{0}-2} ,\label{eq:graph-bound-goal}
\end{equation}
where $k_{0} \coloneqq \lvert\mathsf{V}_{\mathrm{int}} (\mathsf{G})
\setminus (\mathsf{V}_{\rspotential}(\mathsf{G}) \cup
\mathsf{V}_{\rsnoise}(\mathsf{G}))\rvert$.
\end{prop}

\begin{proof}
Recalling \zcref{eq:define-contraction}, since the kernel $J$
is compactly supported we can find a compact set $\Theta\subseteq\mS_{2L}$,
depending on $|\mathsf{V}(\mathsf{G})|$, such that
\begin{equation}
\boxop{\tau\hash_{\nu}\tau}^{\eps,\zeta}(\mathbf{x}_{\varrho_{1}},\mathbf{x}_{\varrho_{2}})=\int_{(\x_{\varrho_1}+\Theta)^{\mathsf{V}_{\mathrm{int}}(\mathsf{G})}}\left(\prod_{e\in\mathsf{E}_{\mathcal{I}'}(\mathsf{G})}J'(\mathbf{x}_{e_{\uparrow}}-\mathbf{x}_{e_{\downarrow}})\right)\left(\prod_{e\in\nu}\mathcal{E}^{\zeta}(\mathbf{x}_{e_{1}},\mathbf{x}_{e_{2}})\right)\left(\prod_{v\in V_{\rspotential}(\mathsf{G})}\varphi^{\eps}_{\uu,\vv}(x_{v})\right)\prod_{v\in\mathsf{V}_{\mathrm{int}}(\mathsf{G})}\dif\mathbf{x}_{v}.\label{eq:contraction-compact}
\end{equation}
Note that the set $\Theta$ and the Lebesgue measure $|\x_{\varrho_1} +
\Theta|$ of its shift are independent of $\x_{\varrho_1}, \x_{\varrho_2}$. In addition,
it does not matter if we center the integration domain around $\x_{\varrho_1}$ or $\x_{\varrho_2}$. We will use these observations
throughout the proof.
In the following, we integrate out the edges directly connecting to a $\rsnoise$ or $\rspotential$ node to reduce the problem to an integral related to the graph $\mathsf{G}_{\mathrm{c}}$. %
First, we note that each $v\in\mathsf{V}_{\rspotential}(\mathsf{G})$ is
connected to the rest of the graph only through a single edge $(v,\iota(v))$,
where $\iota(v)$ is the parent vertex of $v$ in (the appropriate
copy of) $\tau$. Also, each edge $e=\{e_1,e_2\}$ of $\nu$ connects two distinct vertices $e_1$ and $e_2$ which represent $\rsnoise$ nodes, and each of these nodes has one other incident edge $(e_i,\iota(e_i))$ in $\mathsf{G}$. The set of vertices that do not correspond to $\rsnoise$ or $\rspotential$ nodes is exactly $\mathsf{V}(\mathsf{G}_{\mathrm c})$. Therefore, we can write
\begin{align*}
 & \boxop{\tau\hash_{\nu}\tau}^{\eps,\zeta}(\mathbf{x}_{\varrho_{1}},\mathbf{x}_{\varrho_{2}})\\
 & \quad
 =\int_{(\x_{\varrho_1}+\Theta)^{\mathsf{V}_{\mathrm{int}}(\mathsf{G}_{\mathrm c})}}\left(\prod_{e\in\mathsf{E}_{2}(\mathsf{G}_{\mathrm c})}J'(\mathbf{x}_{e_{\uparrow}}-\mathbf{x}_{e_{\downarrow}})\right)\left(\prod_{e\in\mathsf{E}_1(\mathsf{G}_{\mathrm c})}\rslollirc[rsK]^\zeta(\x_{e_1},\x_{e_2})\right)\left(\prod_{v\in V_{\rspotential}}\rslollipopb[rsK]^{\eps}(\mathbf{x}_{\iota(v)})\right)\prod_{v\in\mathsf{V}_{\mathrm{int}}(\mathsf{G}_{\mathrm c})}\dif\mathbf{x}_{v},
\end{align*}
where we have defined $\mathsf{E}_1(\mathsf{G}_{\mathrm c})\coloneqq \{(\iota(e_1),\iota(e_2)\st \{e_1,e_2\}\in\nu\}$ and $\mathsf{E}_2(\mathsf{G}_{\mathrm c})\coloneqq\mathsf{E}(\mathsf{G}_{\mathrm c})\setminus \mathsf{E}_1(\mathsf{G}_{\mathrm c})$ and $\mathsf{V}_{\mathrm{int}}(\mathsf{G}_{\mathrm c})\coloneqq \mathsf{V}(\mathsf{G}_{\mathrm{c}})\setminus\{\varrho_1,\varrho_2\}$.
Now using the bounds \zcref{eq:lollipopb-bound,eq:redcherryx1x2bd} in this expression, and then recalling the definition of $Q$ above, we obtain
\begin{align}
\left|\boxop{\tau\hash_{\nu}\tau}^{\eps,\zeta}(\mathbf{x}_{\varrho_{1}},\mathbf{x}_{\varrho_{2}})\right| & \lesssim\int_{(\x_{\varrho_1}+\Theta)^{\mathsf{V}_{\mathrm{int}}(\mathsf{G}_{\mathrm{c}})}}\left(\prod^{2}_{i=1}\prod_{e\in\mathsf{E}_{i}(\mathsf{G}_{\mathrm{c}})}d_{\mathfrak{s},\mathscr{S}}(\mathbf{x}_{e_{1}},\mathbf{x}_{e_{2}})^{-i}\right)\prod_{v\in\mathsf{V}_{\mathrm{int}}(\mathsf{G}_{\mathrm{c}})}\dif\mathbf{x}_{v}\nonumber \\
 & =\int_{(\x_{\varrho_1}+\Theta)^{\mathsf{V}_{\mathrm{int}}(\mathsf{G}_{\mathrm{c}})}}\left(\prod_{\{u,v\}\in\binom{\mathsf{V}(\mathsf{G}_{\mathrm{c}})}{2}}d_{\mathfrak{s},\mathscr{S}}(\mathbf{x}_{u},\mathbf{x}_{v})^{-Q(u,v)}\right)\prod_{v\in\mathsf{V}_{\mathrm{int}}(\mathsf{G}_{\mathrm{c}})}\dif\mathbf{x}_{v}.\nonumber%
\end{align}
The proof is then completed by \zcref{prop:hepp-prop} below.
\end{proof}

We close this section with another result that is useful in checking the conditions of \zcref{prop:graph-bound}. The issue we address is that sometimes it is hard to check the condition $\deg ( \overline{\mathsf{G}}) > 0$ for all subgraphs $\overline{\mathsf{G}}$ because there are a large number of possible subgraphs. The following lemma reduces the number of subgraphs that need to be checked to a manageable number.

\begin{lem}
  Let $\mathsf{G}=(\mathsf{V},\mathsf{E})=\tau\hash_\nu\tau$ for some $\tau\in\check{\mathsf{T}}_{\mathrm{stoch}}$, let $\mathsf{G}_{\mathrm{c}} = (\mathsf{V}_{\mathrm{c}},\mathsf{E}_{\mathrm{c}})$ be the simplified multigraph defined above and suppose that the following three conditions hold:
  \begin{enumerate}
    \item We have $Q(u,v)\le 2$ for all $u,v\in\mathsf{V}_{\mathrm{c}}$.
    \item We have $Q(u,v)+Q(v,w)+Q(u,w)\le 5$ for all distinct $u,v,w\in\mathsf{V}_{\mathrm{c}}$.
    \item If distinct elements $u,v,w,z\in\mathsf{V}_{\mathrm{c}}$ are such that either $Q(u,v)+Q(v,w)+Q(w,z) = 6$ or $Q(u,v)+Q(u,w)+Q(u,z)=6$, then $Q(u,v)+Q(u,w)+Q(u,z)+Q(v,w)+Q(v,z)+Q(w,z)\le 8$.
    \end{enumerate}
  Then $\deg(\overline{\mathsf{V}})>0$ for all $\overline{\mathsf{V}}\subseteq\mathsf{V}_{\mathrm{c}}$ with $|\overline{\mathsf{V}}|\ge 2$.
  \label{lem:check-degcond-simpler}
\end{lem}
\begin{proof}
  If $|\mathsf{V}|$ is $2$ or $3$, then the conclusion follows immediately from the definition \zcref{eq:degree-Vbar} and the first or second condition, respectively. If $|\mathsf{V}|=4$, then by the third assumption we can assume that, out of the six elements $\{u,v\}\in\binom{\overline{\mathsf V}}{2}$, there are at most two such that $Q(u,v)\ge 2$. The remaining four elements $\{u,v\}$ must have $Q(u,v)\le 1$, which means that $\sum_{\{u,v\}\in\binom{\overline {\mathsf{V}}}{2}} Q(u,v)\le 8$ and the conclusion again follows from the definition \zcref{eq:degree-Vbar}.

  Now a quick look at \zcref{tab:RS-table} shows that $\sum_{\{u,v\}\in\binom{\mathsf{V}_{\mathrm{c}}}{2}} Q(u,v) \le 12$ (as there are at most four $\rslollirc$s of weight $1$ each and four remaining $\mathcal{I}'$ edges of weight $2$ each), and the inequality is strict if the sum is restricted to $\binom{\overline{\mathsf{V}}}{2}$ for $\overline{\mathsf{V}}$ a proper subset of $\mathsf{V}_{\mathrm{c}}$. Using \zcref{eq:degree-Vbar} again, this establishes the conclusion for $|\mathsf{V}|\ge 5$.
\end{proof}

\begin{rem}\label{rem:how-to-check}
  \zcref{lem:check-degcond-simpler} means that, to check the condition \zcref{eq:degcondition}, it is sufficient to consider the subgraph with edge set $\{\{u,v\}\in \binom{\mathsf{V}_{\mathrm{c}}}{2} \st Q(u,v)\ge 2\}$, check that this graph has no edges with $Q\ge 3$ and no triangles, and finally check that the four vertices of each path of length $3$ and each neighborhood of a degree-$3$ vertex in this graph satisfy the degree condition \zcref{eq:degcondition}.
\end{rem}

\subsection{Second chaos: terms involving recentering}

In this section we handle the trees $\rsclawrr$, $\rselkrbr$, and
$\rselkrrb$. The first two of these terms involve recentering when the model is applied to them, and the last term benefits from the use of recentering in its analysis. Thus, we begin by introducing some notations for studying quantities involving recentering.

\subsubsection{$\protect\rsclawrr$}

Referring to \zcref{tab:model-defns}, we see that
\[
\hat{\Pi}^{\eps,\zeta}_{\mathbf{x}}(\rsclawrr)(\mathbf{y})=\rsclawrr[rsK]^{\zeta}(\mathbf{y})-\rsiplollipopr[rsK]^{\zeta}(\mathbf{x})\rslollipopr[rsK]^{\zeta}(\mathbf{y})=\hat{\Pi}^{\eps,\zeta}_{\mathbf{x}}(\rsiplollipopr)(\mathbf{y})\hat{\Pi}^{\eps,\zeta}_{\mathbf{x}}(\rslollipopr)(\mathbf{y}).
\]
At this stage it is convenient to introduce a new notation for recentered trees.\label{recentered-notation}\nomenclature{$\rslollipopnc$}{Recentered integration}
When we draw a cross on an edge of the tree in a realization, it means
that we subtract the value of the subtree stemming from that edge
at a point $\mathbf{x}$, which we add to the superscript. For example,
we have
\begin{equation}\label{e:recenter-notation}
\rsiplollipoprG[rsK]^{\zeta,\mathbf{x}}(\mathbf{y})\coloneqq\rsiplollipopr[rsK]^{\zeta}(\mathbf{y})-\rsiplollipopr[rsK]^{\zeta}(\mathbf{x})=\hat{\Pi}^{\eps,\zeta}_{\mathbf{x}}(\rsiplollipopr)(\mathbf{y})
\end{equation}
and
\begin{equation}\label{eq:recenter-notation-claw}
  \rsclawrG[rsK]^{\zeta,\mathbf{x}}(\mathbf{y})=\rsiplollipoprG[rsK]^{\zeta,\x}(\y)\rslollipopr[rsK]^{\zeta}(\y)\overset{\zcref{e:recenter-notation}}=\left(\rsiplollipopr[rsK]^{\zeta}(\mathbf{y})-\rsiplollipopr[rsK]^{\zeta}(\mathbf{x})\right)\rslollipopr[rsK]^{\zeta}(\mathbf{y})=\hat{\Pi}^{\eps,\zeta}_{\mathbf{x}}(\rsclawrr)(\mathbf{y}).
\end{equation}
We will use crossed edges in trees with contractions (introduced in \zcref{subsec:feynman-diagrams})
with an analogous meaning.

The recentered tree $\rsiplollipoprG[rsK]^{\zeta,\mathbf{x}}$
lives in the first homogeneous Wiener--Itô chaos, just like $\rsiplollipopr[rsK]^{\zeta}$.
Indeed, the only difference is that the edge with a cross corresponds
to the kernel 
\[
J'_{\x - \y} (\y - \z)=
J'(\mathbf{y}-\mathbf{z})-J'(\mathbf{x}-\mathbf{z})
\]defined in \zcref{eq:recenter-J-def}. %
Therefore, we can still
compute the first and second moments in terms of (linear combinations of) Feynman
diagrams. 
In particular, we have 
\begin{equation}
  \EE\left[\hat{\Pi}^{\eps,\zeta}_{\mathbf{x}}(\rsclawrr)(\mathbf{y})\right]  \ovset{\zcref{eq:recenter-notation-claw}}=\mathbb{E}\left[\rsclawrG[rsK]^{\zeta,\x}(\y)\right] = \rsclawrrcA[rsK]^{\zeta}(\mathbf{y})-\rsclawrrcAS[rsK]^{\zeta}(\mathbf{x},\mathbf{y})\label{eq:claw-E-expand}
\end{equation}
and similarly
\begin{equation}
  \Cov\left(\hat{\Pi}^{\eps,\zeta}_{\mathbf{x}}(\rsclawrr)(\mathbf{y}_{1}),\hat{\Pi}^{\eps,\zeta}_{\mathbf{x}}(\rsclawrr)(\mathbf{y}_{2})\right) %
  =\rsclawrrcc[rsK]^{\zeta,\mathbf{x}}(\mathbf{y}_{1},\mathbf{y}_{2})+\rsclawrrccc[rsK]^{\zeta,\mathbf{x}}(\mathbf{y}_{1},\mathbf{y}_{2}).\label{eq:claw-cov-expand}
\end{equation}
The term $\rsclawrrcA[rsK]^{\zeta}(\mathbf{y})$ is most challenging to estimate, since 
$\deg(\rsclawrrcA)=3\cdot2-2\cdot3=0$
is not positive, so we cannot use \zcref{prop:graph-bound}. In fact, power counting would suggest a logarithmic blow-up. Such blow-up, however, in fact does not occur because of the antisymmetry of the integrand. This analysis is left to \zcref{lem:claw-symm-bd} below.
Overall, we obtain the following estimate.
\begin{prop}\label{prop:claw-bound}
We have 
\begin{equation}
\left|\mathbb{E}\left[\hat{\Pi}^{\eps,\zeta}_{\mathbf{x}}(\rsclawrr)(\mathbf{y})\right]\right|\lesssim\log(2+d_{\mathfrak{s},\mathscr{S}}(\mathbf{x},\mathbf{y})^{-1})\label{eq:claw-exp-bd}
\end{equation}
and, for any $\tilde{\kappa}>0$,
\begin{equation}
  \begin{aligned}\left\lvert\Cov\left(\hat{\Pi}^{\eps,\zeta}_{\mathbf{x}}(\rsclawrr)(\mathbf{y}_{1}),\hat{\Pi}^{\eps,\zeta}_{\mathbf{x}}(\rsclawrr)(\mathbf{y}_{2})\right)\right\rvert & \lesssim_{{\tilde{\kappa}}}(\log(2+d_{\mf{s}, \mathscr{S}}(\mathbf{y}_{1}, \mathbf{y}_{2})^{-1}+d_{\mf{s}, \mathscr{S}}(\mathbf{x},\mathbf{y}_{1})^{-1}+d_{\mf{s}, \mathscr{S}}(\mathbf{x}, \mathbf{y}_{2})^{-1}))^{2}\\
 & \qquad+d_{\mathfrak{s},\mathscr{S}}(\mathbf{y}_{1},\mathbf{y}_{2})^{-1}|\mathbf{x}-\mathbf{y}_{1}|^{\oh-{\tilde{\kappa}}}_{\mathfrak{s}}|\mathbf{x}-\mathbf{y}_{2}|^{\oh-{\tilde{\kappa}}}_{\mathfrak{s}}.
\end{aligned}
\label{eq:claw-cov-bd}
\end{equation}
\end{prop}

\begin{proof}
Using \zcref{eq:Kbds, lem:one-int-bound,prop:lollibd},
we can compute that
\begin{equation}\label{eq:elbowlollipop}
\left|\rsclawrrcAS[rsK]^{\zeta}(\mathbf{x},\mathbf{y})\right|\lesssim\log(2+d_{\mathfrak{s},\mathscr{S}}(\mathbf{x},\mathbf{y})^{-1}).
\end{equation}
Using this and \zcref{lem:claw-symm-bd} in \zcref{eq:claw-E-expand},
we get \zcref{eq:claw-exp-bd}. %

Now we turn our attention to \zcref{eq:claw-cov-bd}.
For any fixed $\tilde{\kappa}>0$, we use \zcref{e:Kx-bd} for $\delta = 1/2 - \tilde{\kappa}$, together with
\zcref{eq:redcherryx1x2bd,lem:one-int-bound},
to estimate that
\begin{align}
  \left\lvert\rsclawrrcc[rsK]^{\zeta,\x}(\y_1,\y_2)\right\rvert &\lesssim_{\tilde\kappa} \left\lvert \rslollirc[rsK]^\zeta(\y_1,\y_2)\right\rvert \cdot \left\lvert \int_{\mathbb{S}^2_{2L}} J'_{\x-\y_1}(\y_1-\z_1) \rslollirc[rsK]^\zeta(\z_1,\z_2)J'_{\x-\y_2}(\y_2-\z_2)\,\dif\z_1\,\dif\z_2\right\rvert\notag\\
 & \lesssim_{{\tilde{\kappa}}} d_{\mathfrak{s},\mathscr{S}}(\mathbf{y}_{1},\mathbf{y}_{2})^{-1}|\mathbf{x}-\mathbf{y}_{1}|^{\oh-{\tilde{\kappa}}}_{\mathfrak{s}}|\mathbf{x}-\mathbf{y}_{2}|^{\oh-{\tilde{\kappa}}}_{\mathfrak{s}}.\label{eq:bound-claw-second-term}
\end{align}
For the third term, we do not need to take advantage of the recentering, and we simply use the triangle inequality and \zcref{eq:elbowlollipop}
to estimate 
\begin{align}
  &\left|\rsclawrrccc[rsK]^{\zeta,\mathbf{x}}(\mathbf{y}_{1},\mathbf{y}_{2})\right|
  = \left\lvert\left(\rsclawrrcAS[rsK]^\zeta(\y_1,\y_2) - \rsclawrrcAS[rsK]^\zeta(\x,\y_2)\right) \left(\rsclawrrcAS[rsK]^\zeta(\y_2,\y_1) - \rsclawrrcAS[rsK]^\zeta(\x,\y_1)\right)\right\rvert \notag\\
&\qquad\lesssim(\log(2+d_{\mf{s}, \mathscr{S}}(\mathbf{y}_{1}, \mathbf{y}_{2})^{-1}))^2+(\log(2+d_{\mf{s}, \mathscr{S}} (\mathbf{x}, \mathbf{y}_{1})^{-1}))^2+(\log(2+d_{\mf{s}, \mathscr{S}}(\mathbf{x}, \mathbf{y}_{2})^{-1} ))^2.\label{eq:bound-claw-third-term}
\end{align}
Using \zcref{eq:bound-claw-second-term,eq:bound-claw-third-term}
in \zcref{eq:claw-cov-expand}, we get \zcref{eq:claw-cov-bd},
and the proof is complete.
\end{proof}

\subsubsection{$\protect\rselkrbr$}

Now we turn our attention to bounding the mean and covariance functions of $\hat \Pi_{\x}^{\eps,\zeta}(\rselkrbr[rsN])$. Similarly
to the previous term, this tree requires a recentering. With the same
notation as introduced in \zcref{recentered-notation}, and referring
to \zcref{tab:model-defns}, we see that
\[
\hat{\Pi}^{\eps,\zeta}_{\mathbf{x}}(\rselkrbr)(\mathbf{y})=\rselkrbr[rsK]^{\eps,\zeta}(\mathbf{y})-\rsipcherryrb[rsK]^{\eps,\zeta}(\mathbf{x})\rslollipopr[rsK]^{\zeta}(\mathbf{y})=\rsipcherryrbG[rsK]^{\eps,\zeta,\mathbf{x}}(\mathbf{y})\rslollipopr[rsK]^{ \zeta} (\y) \eqqcolon \rselkrbrG[rsK]^{\eps,\zeta,\x}(\y),
\]
where
\begin{equation*}
  \begin{aligned}
    \rsipcherryrbG[rsK]^{\ve, \zeta, \x}(\y) & = \rsipcherryrb[rsK]^{\ve, \zeta}(\y) - \rsipcherryrb[rsK]^{\ve, \zeta}(\x) = \hat{\Pi}_{\x}^{\ve, \zeta} (\rsipcherryrb[rsN]) (\y) .
  \end{aligned}
\end{equation*}
For the present tree we find the following estimate. Recall the distance $d_L$
defined in \zcref{e:dLdef}.
\begin{prop}\label{prop:rbrelk}
    We have for any $\tilde{\kappa} > 0$ that
    \begin{equation}\label{e:rbr-aim}
        \left\vert \EE \left[ \hat{\Pi}_{\x}^{\ve, \zeta} (\rselkrbr) (\y) \right] \right\vert \lesssim  \log(2 +d_{\mf{s},\mathscr{S}}(\x, \y)^{-1}) + \log(2 +1/d_L(y))^2 %
      \end{equation}
      and\begin{equation}
        \begin{aligned}
        \Cov \left( \hat{\Pi}_{\x}^{\ve, \zeta} (\rselkrbr) (\y_1) ,
        \hat{\Pi}_{\x}^{\ve, \zeta}  (\rselkrbr)  (\y_2) \right)
        \lesssim_{\tilde{\kappa}} & (\log( 2 +d_{\mf{s}, \mathscr{S}}(\y_1, 
        \y_2)^{-1} + d_{\mf{s}, \mathscr{S}} (\x,  \y_1)^{-1} + d_{\mf{s},
        \mathscr{S}}(\x, \y_2)^{-1}))^2 \\
        & + d_{\mf{s}, \mathscr{S}}(\y_1, \y_2)^{-1} |\x - \y_1|_{\mf{s}}^{1/2 - {\tilde{\kappa}}} |\x - \y_2|_{\mf{s}}^{1/2- {\tilde{\kappa}}}  . 
      \end{aligned}\label{eq:elkrbrcovbd}
      \end{equation}
    Moreover, for any $s \neq t \in \RR$ and $c\in(0,\infty)$, we have
\begin{equation}\label{e:newrbc2}
  \adjustlimits\lim_{\ve \downarrow 0} \sup_{\zeta \in (0, \ve)} \sup_{d_L(x),d_L(y) \leq c \ve} \; \left\vert  \Cov \left( \rselkrbr[rsK]^{\ve, \zeta} (t, x), \rselkrbr[rsK]^{\ve, \zeta} (s, y)  \right) \right\vert = 0.
\end{equation}
\end{prop}
\begin{proof}
  By definition, we have
  \begin{equation*}
    \EE \left[ \hat{\Pi}_{\x}^{\ve, \zeta} (\rselkrbr) (\y) \right] =\rselkrbrcS[rsK]^{\ve, \zeta, \x} (\y)= \rselkrbrcSG[rsK]^{\ve, \zeta} (\y) -  \rselkrbrcSN[rsK]^{\ve, \zeta} (\x, \y) ,
  \end{equation*}
  where on the right hand-side we have the non-recentered trees. Now, the second
  term can be treated via \zcref{prop:lollibd, lem:one-int-bound, eq:lollipopb-bound} and we  obtain
  \begin{equation}\label{e:claw-opened-b}
    \left\vert \rselkrbrcSN[rsK]^{\ve, \zeta} (\x, \y) \right\vert \lesssim  \log(2 +d_{\mf{s},\mathscr{S}}(\x, \y)^{-1}) ,
  \end{equation}
  which is an estimate of the desired order. This leaves us with the tree $\rselkrbrcSG $. By %
  the compact support property of $J'$, we have
  \begin{equation*}
    \rselkrbrcSG[rsK]^{\ve, \zeta} (\y)  = \int_{\y + \Theta} \rsbphz[rsK]^{\zeta, J'}(\y , \z) \varphi^{\ve}_{\uu,\vv} (z) \ud \z , 
  \end{equation*}
  for some compact set $\Theta \subseteq \mS_{2L}$ that is independent of $\y$. Here $\rsbphz[rsK]^{\zeta,
  J'}$ is defined in \zcref{eq:bphzdef} below and we write the integration
  variable as $\z=(r,z)$. Now by \zcref{lem:bphzbd} applied with $J_1 = J'$ and
  $\alpha =2$ (note that the assumption \zcref{e:assu-J1} on the kernel $J'$ is
  satisfied by \zcref{e:Kx-bd}), we
  have that 
  \begin{equation*}
    \Big\vert \rsbphz[rsK]^{\zeta, J'}(\y, \z)  \Big\vert \lesssim d_{\mf{s}, \mathscr{S}}(\y , \z)^{-2}  \log(2 +d_{\mf{s}, \mathscr{S}}(\y , \z)^{-1}) .
  \end{equation*}
  Then if $d_{L}(y)
  \geq  2\ve$, we find for some $c > 0$ satisfying $\Theta
  \subseteq [-c, c] \times \TT_{2L}$ that
  \begin{equation*}
    \begin{aligned}
    \left\vert \int_{\y + \Theta} |\y - \z|_{\mf{s}}^{-2}  \log(2 +d_{\mf{s}, \mathscr{S}} (\y , \z)^{-1})
    \varphi^{\ve}_{\uu , \vv} (z) \ud \z \right\vert & \lesssim \log(2 +1/d_L(y) )\int_{[-c,
    c] } \frac{1}{|r |+ d_L (y)^{2}} \ud r\\& \lesssim \log{ (2+ 1/d_L(y))}^2 ,
    \end{aligned}
    \end{equation*}
    where we have used that $\varphi^\ve_{\uu,
    \vv}$ is supported in regions of radius $\eps$ about $L\ZZ$. 
    On the other hand, for
    $|d_L(y)|< 2 \ve$ we have the
    following upper bound:
    \begin{equation*}
      \begin{aligned}
        \left\vert \int_{\y+ \Theta} |\y - \z|_{\mf{s}}^{-2}  \log(2 +d_{\mf{s}, \mathscr{S}} (\y , \z)^{-1})
    \varphi^{\ve}_{\uu , \vv} (z) \ud \z \right\vert & \leq 2 \int_{\y+ \Theta} |\y - \z|_{\mf{s}}^{-2}  \log(2 + |\y - \z|_{\mf{s}}^{-1})
    \varphi^{\ve}_{\uu , \vv} (z) \ud \z \\
    & \lesssim \ve^{-1} \int_{[-c,
      c] \times [-  \ve, \ve] } \frac{\log(2 +(\sqrt{r} + |z|)^{-1})}{|r |+ | z |^{2}} \ud r \ud z \\
      & \lesssim \ve^{-1} \int_0^{ \ve} \log{ (2+ 1/z) }^2  \ud z \\
      & \lesssim \log(2+ 1/\ve)^2 \lesssim \log(2+ 1/ d_L(y))^2,
      \end{aligned}
    \end{equation*}
  where in the first line we used the reflection symmetry of $\varphi^\ve$
  about $0$ and $L$.
  Therefore, we obtain the estimate
  \begin{equation}\label{eq:thirdtermbound}
      \Big\vert \rselkrbrcS[rsK]^{\ve, \zeta, \x} (\y) \Big\vert  \lesssim  \log(2 +d_{\mf{s},\mathscr{S}}(\x, \y)^{-1}) + \log(2 +1/d_L(y))^2.
  \end{equation}
  This completes the proof of \zcref{e:rbr-aim}.

  For the covariance, we have %
  \begin{equation} \label{e:to-bd-selkrbr}
      \Cov \left( \hat{\Pi}_{\x}^{\ve, \zeta} (\rselkrbr[rsN]) ( \y_1) ,\hat{\Pi}_{\x}^{\ve, \zeta} (\rselkrbr[rsN]) (\y_2) \right)  =  \rselkrbrcc[rsK]^{\ve, \zeta, \x} \left( \y_1, \y_2 \right) + \rselkrbrccc[rsK]^{\ve, \zeta, \x} \left( \y_1, \y_2 \right).
  \end{equation}
  The two terms on the right side are analogous to the terms bounded in \zcref{eq:bound-claw-second-term,eq:bound-claw-third-term}, the only difference being the attachment of additional $\rslollipopb$s. But the contribution of the $\rslollipopb$s is bounded by \zcref{eq:lollipopb-bound}, and thus, in the same way as  the bounds \zcref[range]{eq:bound-claw-second-term,eq:bound-claw-third-term}, we obtain
  \begin{equation}\label{eq:firstterm-elkrbr}
    \begin{aligned}
      \left\vert \rselkrbrcc[rsK]^{\ve, \zeta, \x} \left( \y_1, \y_2 \right) \right\vert  \lesssim_{\tilde\kappa} d_{\mf{s},\mathscr{S}}(\y_1, \y_2)^{-1} |\x - \y_1|_{\mf{s}}^{1/2-{\tilde{\kappa}}} |\x - \y_2|_{\mf{s}}^{1/2- {\tilde{\kappa}}},
  \end{aligned}
  \end{equation}
  and 
  \begin{equation}\label{eq:second-term-elkrbr}
    \left\vert \rselkrbrccc[rsK]^{\ve, \zeta, \x} \left( \y_1, \y_2 \right) \right\vert \lesssim (\log( 2 +d_{\mf{s}, \mathscr{S}}(\y_1 , \y_2)^{-1} + d_{\mf{s}, \mathscr{S}}(\x ,\y_1)^{-1} + d_{\mf{s}, \mathscr{S}}(\x, \y_2)^{-1}))^2 .
  \end{equation}
  Using \zcref{eq:firstterm-elkrbr,eq:second-term-elkrbr} in \zcref{e:to-bd-selkrbr}, we obtain \zcref{eq:elkrbrcovbd}.

  As for \zcref{e:newrbc2}, let us consider only the case $|x| \leq c
  \ve$ (the case $|x-L| \leq c \ve $ is treated analogously). Note also
  that it suffices to fix one of the two variables (in this case $x$) close to
  $\{0, L\}$, and leave $y$ free. Since $
  \rsipcherryrb[rsK]^{\ve, \zeta}(t, 0) = 0 $ we find that
  \begin{align}
     \Cov \left( \rselkrbr[rsK]^{\ve, \zeta} (t, x), \rselkrbr[rsK]^{\ve, \zeta} (s, y)  \right)    &=   \Cov \left( \rselkrbr[rsK]^{\ve, \zeta} (t, x) - \rsipcherryrb[rsK]^{\ve, \zeta}(t, 0) \rslollipopr[rsK]^\zeta(t, x), \rselkrbr[rsK]^{\ve, \zeta} (s, y)  \right)   \notag \\
     &=   \Cov \left( \rselkrbrXB[rsK]^{\ve, \zeta, (t, 0)} (t, x) , \rselkrbr[rsK]^{\ve, \zeta} (s, y)  \right)   . \label{eq:introduce-dash}
  \end{align}
  Therefore, we obtain terms similar to
  \zcref{eq:firstterm-elkrbr,eq:second-term-elkrbr}, but with the recentering
  only in one of the trees. We estimate via
  \zcref{prop:lollibd,lem:one-int-bound,e:Kx-bd} for any $\delta \in (0, 1)$ and
  with $\x = (t,x)$ that 
  \begin{equation*}
    \begin{aligned}
      \Big\vert & \rselkrbrccone[rsK]^{\ve, \zeta, (t, 0)} \left( (t,x), (s,y) \right) \Big\vert \\
      & \lesssim d_{\mf{s}, \mathscr{S}}( (t, x), (s, y))^{-1} |x|^\delta \int_{(\x+\Theta)^2} ( |(t, x) - \z_1 |_{\mf{s}}^{-2-\delta}+ |(t, 0) - \z_1 |_{\mf{s}}^{-2-\delta} ) d_{\mf{s}, \mathscr{S}}( \z_1, \z_2)^{-1} |(s, y) - \z_2|_{\mf{s}}^{-2} \ud \z_1 \ud \z_2 \\
      & \lesssim d_{\mf{s}, \mathscr{S}}( (t, x), (s, y))^{-1} |x|^\delta \lesssim_{s,t} |x|^\delta .
    \end{aligned}
  \end{equation*}
  For the second term, we similarly estimate
  \begin{equation*}
    \left\vert \rselkrbrcccone[rsK]^{\ve, \zeta, (t,0)} \left( (t,x), (s, y)\right) \right\vert \lesssim_{s, t} |x|^\delta \int_{\x+\Theta} (|(t, x) - \z|_{\mf{s}}^{-2-\delta}+|(t, 0) - \z|_{\mf{s}}^{-2-\delta}) d_{\mf{s}, \mathscr{S}}(\z, (s, y))^{-1} \ud \z \lesssim_{s,t} |x|^\delta .
  \end{equation*}
  Overall, we have obtained that
  \begin{equation*}
   \Big| \Cov \left( \rselkrbrXB[rsK]^{\ve, \zeta, (t, 0)} (t, x) , \rselkrbr[rsK]^{\ve, \zeta} (s, y)  \right) \Big| \lesssim_{s,t, \delta} |x|^\delta .
  \end{equation*}
This concludes the proof of \zcref{e:newrbc2}.
\end{proof}

\subsubsection{$\protect\rselkrrb$}
For the tree $\rselkrrb$, there is no recentering involved. However, the tree is renormalized, meaning that
\begin{equation*}
  \hat{\Pi}_\x^{\ve, \zeta} (\rselkrrb[rsN]) (\y) = \rscherryrenormb[rsK]^{\ve, \zeta} (\y)  .
\end{equation*}
In particular, the expected value of the model on this tree is given by
\begin{equation*}
  \EE \left[ \hat{\Pi}_\x^{\ve, \zeta} \rselkrrb[rsN] (\y) \right] = \EE  \rscherryrenormb[rsK]^{\ve, \zeta} (\y) = \rscherryrenormbtr[rsK]^{\ve, \zeta} (\y) ,
\end{equation*}
where the filled black triangle is as in \zcref{eq:black-triangle-def}.

For the present tree we find the following estimate.
\begin{prop}\label{prop:rrbelk}
    We have
    \begin{equation} \label{e:rrb-aim}
      \begin{aligned}
        \left\vert \EE \left[ \hat{\Pi}_{\x}^{\ve, \zeta} (\rselkrrb) (\y)  \right] \right\vert \lesssim 1 \qquad\text{and} \qquad
        \left\lvert\Cov \left( \hat{\Pi}_{\x}^{\ve, \zeta} (\rselkrrb) (\y_1) , \hat{\Pi}_{\x}^{\ve, \zeta}  (\rselkrrb)  (\y_2) \right)\right\rvert \lesssim \log ( 2+ d_{\mf{s},\mathscr{S}} (\y_1, \y_2)^{-1} )  . 
      \end{aligned}
    \end{equation}
    Moreover, for any $s \neq t \in \RR$, we have
\begin{equation}\label{e:newrbc3}
  \adjustlimits \lim_{\ve \downarrow 0} \sup_{\zeta \in (0, \ve)} \sup_{d_L(x),d_L(y) \leq c \ve} \; \left\vert  \Cov \left( \rselkrrb[rsK]^{\ve, \zeta} (t, x), \rselkrrb[rsK]^{\ve, \zeta} (s, y)  \right) \right\vert = 0 .
\end{equation}
\end{prop}

\begin{proof}
  For the mean, we use \zcref{eq:renorm-constants-same,eq:lollipopb-bound} to obtain
  \begin{equation*}
    \left\vert \EE \left[ \hat{\Pi}_{\x}^{\ve, \zeta} (\rselkrrb) (\y)  \right] \right\vert = \big\vert \rselkrrbc[rsK]^{\ve, \zeta} (\y) \big\vert \lesssim 1.
  \end{equation*}
  For the covariance we find that
  \begin{equation*}
    \begin{aligned}
      \Cov \left( \hat{\Pi}_{\x}^{\ve, \zeta} (\rselkrrb) (\y_1), \hat{\Pi}_{\x}^{\ve, \zeta} (\rselkrrb) ( \y_2 )  \right) & =  2 \, \rselkrrbcc[rsK]^{\ve, \zeta} \left( \y_1, \y_2 \right).
    \end{aligned}
  \end{equation*}
  By \zcref{lem:one-int-bound,prop:lollibd} and \zcref{eq:lollipopb-bound}, we have
  \begin{equation*}
    \begin{aligned}
      \left\vert \rselkrrbcc[rsK]^{\ve, \zeta} (\y_1, \y_2) \right\vert & \lesssim \log ( 2+ d_{\mf{s},\mathscr{S}} (\y_1, \y_2)^{-1} )   ,
    \end{aligned}
    \end{equation*}
  which completes the proof of \zcref{e:rrb-aim}.

  As for \zcref{e:newrbc3}, we have by \zcref{eq:lollipopb-bound} that
  \begin{equation*}
    \left\vert \Cov \left( \rselkrrb[rsK]^{\ve, \zeta} (t, x), \rselkrrb[rsK]^{\ve, \zeta} (s, y)  \right) \right\vert \lesssim \left\vert \rselkrrbccpl[rsK]^{\ve, \zeta} ((t,x), (s,y)) \right\vert .
  \end{equation*}
  Now by \zcref{tab:RS-table}, we see that for fixed $\ve, \zeta \in (0, 1)$,
  the random function $x \mapsto \rsipcherryrr[rsK]^{\ve, \zeta} (t, x)$ is odd and
  continuous. Therefore, we have
  \begin{equation*}
    \rselkrrbccpl[rsK]^{\ve, \zeta} ((t,0 ), (s,y)) = 0 ,
  \end{equation*}
  and we can rewrite
  \begin{equation*}
    \begin{aligned}
     \left\vert \rselkrrbccpl[rsK]^{\ve, \zeta} ((t,x), (s,y)) \right\vert & = \left\vert \rselkrrbccpl[rsK]^{\ve, \zeta} ((t,x), (s,y)) - \rselkrrbccpl[rsK]^{\ve, \zeta} ((t,0 ), (s,y)) \right\vert \\
     & = \left\vert \rselkrrbccplX[rsK]^{\ve, \zeta, (t, 0)} ((t,x), (s,y)) \right\vert.
    \end{aligned}
  \end{equation*}
  Therefore, 
  combining \zcref{e:Kx-bd} with \zcref{lem:one-int-bound,prop:lollibd}, we obtain for any
  $\delta \in (0, 1)$ and a compact set $\Theta \subseteq \mS_{2L}$ independent
  of $\x= (t, x), \y=(s, y)$ that 
  \begin{equation*}
    \begin{aligned}
      \left\vert \rselkrrbccplX[rsK]^{\ve, \zeta, (t,0)} (\x, \y) \right\vert &  \lesssim |x|^\delta \int_{( \x+\Theta)^2} (|(t,0) - \z_1|_{\mf{s}}^{-2-\delta}+ |\x - \z_1|_{\mf{s}}^{-2-\delta}) d_{\mf{s}, \mathscr{S}}(\z_1, \z_2)^{-2} |\y - \z_2|_{\mf{s}}^{-2}\ud \z_1 \ud \z_2 \\
      & \lesssim_{s, t} |x|^\delta ,
    \end{aligned}
  \end{equation*}
  from which \zcref{e:newrbc3} follows.%
\end{proof}

\subsection{Third and fourth chaoses}
Finally, we are left with estimating the covariance functions of the remaining three trees, $\rselkrrr, \rscandelabrarrrr$, and $\rsmooserrrr$.  The main technical difficulty arises from those pairings that generate subtrees of the form $\rsbphz$. Consequently, we must handle singular integrals involving kernels with logarithmic divergence. In this section we will invoke several technical estimates on such singular integrals whose proofs are deferred to \zcref{s.renormalizationblabla}.

\subsubsection{$\protect\rselkrrr$}

The tree $\rselkrrr$ does not require any recentering, just the renormalization of the cherry $\rscherryrr$. Our estimates are obtained similarly to the ones obtained in the previous section.

\begin{prop}\label{prop:elkrrr}
We have 
\begin{equation}
\mathbb{E}\left[\hat{\Pi}^{\eps,\zeta}_{\mathbf{x}}(\rselkrrr)(\mathbf{y})\right]=0\label{eq:elkrrr-exp}
\end{equation}
and, for any ${\tilde{\kappa}}>0$,
\begin{equation}
\Cov\left(\hat{\Pi}^{\ve, \zeta}_{\mathbf{x}}(\rselkrrr)(\mathbf{y}_{1}),\hat{\Pi}^{\ve, \zeta}_{\mathbf{x}}(\rselkrrr)(\mathbf{y}_{2})\right)\lesssim d_{\mathfrak{s},\mathscr{S}}(\mathbf{y}_{1},\mathbf{y}_{2})^{-1-{\tilde{\kappa}}} .\label{eq:elkrrr-cov}
\end{equation}
Moreover, for any $s \neq t \in \RR$ and $c\in(0,\infty)$, we have
\begin{equation}\label{e:newrbc4}
  \adjustlimits\lim_{\ve \downarrow 0} \sup_{\zeta \in (0, \ve)} \sup_{d_L(x),d_L(y) \leq c \ve} \; \left\vert  \Cov \left( \rscherryrenormr[rsK]^{ \zeta} (t, x), \rscherryrenormr[rsK]^{\zeta} (s, y)  \right) \right\vert = 0 .
\end{equation}
\end{prop}

\begin{proof}
The expectation \zcref{eq:elkrrr-exp} is immediate by symmetry.
We thus begin by proving \zcref{eq:elkrrr-cov}. Referring to \zcref{tab:model-defns},
we see that there is no recentering, so we can expand
\begin{equation}
\begin{aligned}
\Cov\left(\hat{\Pi}^{\ve, \zeta}_{\mathbf{x}}(\rselkrrr)(\mathbf{y}_{1}),\hat{\Pi}^{\ve, \zeta}_{\mathbf{x}}(\rselkrrr)(\mathbf{y}_{2})\right) & \lesssim\left|\rselkrrrc[rsK]^{\zeta}(\mathbf{y}_{1},\mathbf{y}_{2})\right|+\left|\rselkrrrccc[rsK]^{\zeta}(\mathbf{y}_{1},\mathbf{y}_{2})\right| + \left|\rselkrrrcccccc[rsK]^{\zeta}(\mathbf{y}_{1},\mathbf{y}_{2})\right| \\
 & \qquad + \left|\rselkrrrcc[rsK]^{\zeta}(\mathbf{y}_{1},\mathbf{y}_{2})\right|+\left|\rselkrrrcccc[rsK]^{\zeta}(\mathbf{y}_{1},\mathbf{y}_{2})\right|.
\end{aligned}\label{eq:elkrrr-cov-exp}
\end{equation}
For the first term on the right side, we use \zcref{eq:renorm-constants-same,eq:redcherryx1x2bd} to obtain
\begin{equation*}
  \left\vert \rselkrrrc[rsK]^{\zeta} \right\vert \left( \y_1, \y_2 \right)  \lesssim d_{\mathfrak{s},\mathscr{S}} (\y_1, \y_2)^{-1} .
\end{equation*}
For the last two terms of \zcref{eq:elkrrr-cov-exp} we use \zcref{prop:graph-bound} to obtain
\begin{equation*}
  \left\vert \rselkrrrcc[rsK]^{\zeta} \right\vert \left( \y_1, \y_2 \right) + \left\vert \rselkrrrcccc[rsK]^{\zeta} \right\vert \left( \y_1, \y_2 \right)  \lesssim d_{\mathfrak{s},\mathscr{S}} (\y_1, \y_2)^{-1}.
\end{equation*}
Thus, to complete the proof of \zcref{eq:elkrrr-cov}, it remains to bound the second and third term on the right side of \zcref{eq:elkrrr-cov-exp}.

The second term, involving  $\rselkrrrccc$, requires some care. We would like to
bound it via \zcref{lem:bphzbd}. This is somewhat complicated by the fact that
there are \emph{two} copies of the tree $\rsbphz$ joined by a single
contraction. The contraction represents a \emph{spatial} mollification of a
delta function, arising from our spatial mollification of the noise rather than
a spatiotemporal mollification, and this distinction leads to certain
integrability issues if we try to apply \zcref{lem:bphzbd} twice with $J_1=J'$.
To avoid this problem, we instead apply \zcref{lem:bphzbd} with $J_1$ taken to
be $\rselkrrrcccARM[rsK]^\zeta$,
which is to say the kernel represented by the entire remainder of the tree. Of course, then to check
the hypothesis of \zcref{lem:bphzbd} with this choice of $J_1$ we need to apply
\zcref{lem:bphzbd} once again. In fact, this will additionally require the
further refined continuity result given by \zcref{lem:bphz-increment}.

We start by using \zcref{eq:lollirc-diff-bd,eq:EW-cov} to write %
\begin{equation*}
  \rslollirc[rsK]^\zeta( \x, \y) = q_{|t-s|} * \Sh^\zeta_{2L} (x -y )- q_{|t-s|} * \Sh^\zeta_{2L} (x + y ) + R^\zeta(\x, \y) ,
\end{equation*}
where $q$ is the periodic heat kernel and $R^\zeta$ is a remainder that is bounded and H\"older continuous (uniformly in $\zeta$), and we use the notation
$\x= (t,x)$ and $\y=(s, y)$. We then define
\begin{equation} \label{e:Jidef}
  J_1 (\x, \y) = q_{|t-s|} * \Sh^\zeta_{2L} (x -y ) , \qquad J_2(\x, \y) = - q_{|t-s|} * \Sh^\zeta_{2L} (x + y ) , \qquad\text{and}\qquad J_3 (\x, \y) = R^\zeta(\x, \y) .
\end{equation}
Each $J_i$ is symmetric in the sense that $J_i(\x, \y) = J_i(\y, \x)$. In addition, $J_1$
satisfies \zcref{e:assu-J1} with $\alpha = 1$ (and therefore also for any
$\alpha > 1$) and any $\delta \in (0,
1)$. Moreover, $J_3$ satisfies \zcref{e:assu-J1} for any $\alpha > 0$ and $\delta
\in (0, 1)$. Finally, $J_2$ satisfies
\begin{equation}\label{e:symm12}
  J_2(\x, \y) = - J_1(\x, \sr \y) .
\end{equation}
It follows that we can apply \zcref{lem:bphzbd,lem:bphz-increment} to both $J_1$
and $J_3$. We do not apply them directly to $J_2$ (because it does not satisfy
\zcref{e:assu-J1}), but instead for $J_2$ we will reduce ourselves to $J_1$ via
\zcref{e:symm12}. 

Now for $i =1,2,3$ we define
\begin{equation*}
  M_i (\x, \y) = \rsbphz[rsK]^{\zeta, J_i} (\y, \x) ,
\end{equation*}
using the notation defined in \zcref{eq:bphzdef} below.
Note that the arguments on the right side are exchanged ($(\y, \x)$ rather than
$(\x, \y)$), which allows us to rewrite 
\begin{equation*}
  \rselkrrrccc[rsK]^{\zeta} (\y_1, \y_2 )= \sum_{i=1}^3 \rsbphz[rsK]^{\zeta, M_i} (\y_1, \y_2) .
\end{equation*}
We observe that by \zcref{e:symm12} we have
\begin{equation} \label{e:symm123}
  \rsbphz[rsK]^{\zeta, M_2} (\y_1, \y_2) = \rsbphz[rsK]^{\zeta, M_1} (\sr \y_1, \y_2) .
\end{equation}
Next, we apply \zcref{lem:bphzbd} first to $M_i$, for $i =1,3$ and then to $
\rsbphz[rsK]^{\zeta, M_i}$ for $i=1,3$. In the latter case we are able to reapply \zcref{lem:bphzbd}
as a consequence of \zcref{lem:bphz-increment}, which verifies the hypothesis \zcref{e:assu-J1-2}. In both cases we apply the lemmas with arbitrary $\alpha\in (1,2)$ and $\delta\in(0,2-\alpha)$. We obtain, for any $\tilde\kappa>0$,
\begin{equation*}
  \Big\vert \rsbphz[rsK]^{\zeta, M_i} (\y_1, \y_2 ) \Big\vert \lesssim_{\tilde{\kappa}} |\y_1 -  \y_2|_{\mf{s}}^{-1 - \tilde{\kappa}},
\end{equation*}
for $i=1,3$. Finally, by \zcref{e:symm123} we
conclude
\begin{equation*}
  \Big\vert \rselkrrrccc[rsK]^{\zeta} (\y_1, \y_2 ) \Big\vert \lesssim_{\tilde{\kappa}} d_{\mf{s}, \mathscr{S}}(\y_1 , \y_2)^{-1 - \tilde{\kappa}} ,
\end{equation*}
for any $\tilde{\kappa} > 0$, which is of the desired order for \zcref{eq:elkrrr-cov}.

Finally, for the third term in \zcref{eq:elkrrr-cov-exp}, we use \zcref{eq:renorm-constants-same} and
\zcref{lem:bphzbd} with $J_1, J_2, J_3$ as in the previous case to estimate
\begin{equation*}
  \left|\rselkrrrcccccc[rsK]^{\zeta}(\mathbf{y}_{1},\mathbf{y}_{2})\right| \lesssim_{\tilde{\kappa}} d_{\mathfrak{s},\mathscr{S}}(\y_1, \y_2)^{-1 - {\tilde{\kappa}}} ,
\end{equation*}
which is of the desired order for \zcref{eq:elkrrr-cov}. This completes the proof of \zcref{eq:elkrrr-cov}.

It remains to prove \zcref{e:newrbc4}. As in previous cases, we treat only the boundary $x=0$,
since the boundary $x =L$ is treated similarly. We first use %
the fact that $\rsiprenorm[rsK]^\zeta
(t, 0) = 0$ to find, for fixed
$\ve, \zeta \in (0,1)$ and $t \in \RR$, that
\begin{align} 
     \Cov \left( \rscherryrenormr[rsK]^{ \zeta} (t, x), \rscherryrenormr[rsK]^{\zeta} (s, y)  \right)  &=   \Cov \left( \rscherryrenormr[rsK]^{ \zeta} (t, x) - \rsiprenorm[rsK]^\zeta
    (t, 0) \rslollipopr[rsK]^\zeta (t, x)  , \rscherryrenormr[rsK]^{\zeta} (s, y)  \right) \notag \\
    & =  \Cov \left(  \rscherryrenormrXA[rsK]^{ \zeta, (t, 0)} (t, x) , \rscherryrenormr[rsK]^{\zeta} (s, y)  \right) .\label{e:cov-interm-elkrrr}
  \end{align}
Here we have again used the recentering notation introduced in 
\zcref{e:recenter-notation} above. 
When computing the covariances in \zcref{e:cov-interm-elkrrr}, we find the analogues of the same
diagrams that appear in \zcref{eq:elkrrr-cov-exp}, each with the left edge incoming into the
left root replaced  by a crossed edge. 

For the analogue of the first term in \zcref{eq:elkrrr-cov-exp}, we combine \zcref{eq:renorm-constants-same} with
\zcref{e:Kx-bd,prop:lollibd} to obtain for any $\delta \in (0, 1)$ and $s\neq t$ that
\begin{equation*}
  \left\vert \rselkrrrcXB[rsK]^{\zeta, (t,0)} \right\vert \left((t, x), (s,y) \right)  \lesssim_{s, t, \delta} |x|^\delta .
\end{equation*} 
Similarly, following the steps above, together with \zcref{e:Kx-bd,lem:bphzbd,prop:lollibd}, we bound
\begin{equation*}
  \left( \left|\rselkrrrccccccXA[rsK]^{\zeta, (t,0)} \right| + \left\vert \rselkrrrccXA[rsK]^{\zeta, (t,0)} \right\vert + \left\vert \rselkrrrccccXA[rsK]^{\zeta, (t,0)} \right\vert \right) \left((t, x), (s,y) \right) \lesssim_{s, t,\delta} |x|^\delta .
\end{equation*}

Therefore, we are left with estimating 
\begin{equation}
  \left\vert \rselkrrrcccXA[rsK]^{\zeta, (t,0)} \right\vert \left((t, x), (s,y) \right).\label{eq:lasttwoterms}
\end{equation}
Again, this term is delicate because it contains $\rsclawrrcA$ as a subtree. %
We estimate
\begin{equation}\label{e:another-elk}
  \left\vert \rselkrrrcccXA[rsK]^{\zeta, (t,0)} \right\vert \left((t, x), (s,y) \right) \leq \Big\vert \rselkrrrcccXC[rsK]^{\zeta, (t,0),\x} ((t, x), (s,y)) \Big\vert + \Big\vert \rsclawrrX[rsK]^{\zeta, (t,0)} (t, x)\cdot \rselkrrrXA[rsK]^{\zeta} ((t, x),(s,y)) \Big\vert .
\end{equation}
Here we define, for any kernel $J_1$, %
\begin{equation}\label{e:defzigzag}
  \rsbphzXX[rsK]^{\zeta, \w, \y_1, J_1} ( \y_1, \y_2) \coloneqq \int_{\mS_{2L}}   ( J_1(\z, \y_2) - J_1(\y_1, \y_2)) (J'(\y_1 - \z)-J'(\w-\z))\rslollirc[rsK]^\zeta (\z, \y_1) \ud \z  
\end{equation}
in accordance with \zcref{e:new-zigzagdef} below,
and then
\begin{equation*}
  \begin{aligned}
    \rselkrrrcccXC[rsK]^{\zeta, (t,0),\x} (\x, \y) & =  \sum_{i=1}^3 \rsbphzXX[rsK]^{\zeta, (t,0), \x, M_i} ( \x, \y) \\
    & = \sum_{i=1}^3  \int_{\mS_{2L}} J'_{(0, -x)} (\x- \z)  ( M_i (\z,  \y) - M_i(\x ,  \y)) \rslollirc[rsK]^\zeta (\z, \x) \ud \z 
  \end{aligned}
\end{equation*}
with $\x = (t, x), \y=(s, y)$, and $(t,0)=(0,-x)+\x$.
Let us start by estimating the second term in \zcref{e:another-elk}. For the first factor, via \zcref{lem:claw-symm-bd} and the same arguments that lead to
\zcref{e:claw-opened-b}, we have
\begin{equation}
  \Big\vert \rsclawrrX[rsK]^{\zeta, (t,0)} (t, x) \Big\vert \lesssim \log(2 +1/|x|) .\label{eq:clawrrX}
\end{equation}
This logarithmic upper bound is not optimal, and actually we expect this term to be bounded uniformly in $x$. However, this does not matter, because for the other factor we are about to obtain an algebraic bound (of arbitrarily small power) anyway. %
Indeed, for the second factor in the second term, since
$\rslollipopr[rsK]^\zeta(t, 0) = 0$, we have for any $\delta \in (0, 1)$ that
\begin{equation}
  \Big\vert \rselkrrrXA[rsK]^{\zeta} ((t, x),(s,y)) \Big\vert = \Big\vert \rselkrrrXA[rsK]^{\zeta} ((t, x),(s,y)) -\rselkrrrXA[rsK]^{\zeta} ((t, 0),(s,y)) \Big\vert %
  \lesssim_{s, t, \delta} |x|^\delta ,\label{eq:penelktimate}
\end{equation}
where the last estimate is a consequence of \zcref{lem:bphz-increment} applied
to the kernels $J_i$ defined in \zcref{e:Jidef}. %

Finally, we turn to the first term on the right side of \zcref{e:another-elk}. 
Here we apply \zcref{lem:zigzag} with the kernel $ J_1=M_i$ and any $\alpha \in
(0, 1), \delta \in (0, 2- \alpha)$. In this way, we obtain
\begin{equation}\label{e:zzbd}
  \Big\vert \rselkrrrcccXC[rsK]^{\zeta, (t,0), \x} (\x, \y) \Big\vert \lesssim_{s,t} |x|^\delta ,
\end{equation}
as required. 
This completes the proof of \zcref{e:newrbc4} and therefore of the
entire proposition.
\end{proof}

\subsubsection{$\protect \rscandelabrarrrr$}

For the tree $\rscandelabrarrrr$ there is no recentering, but there is renormalization; see \zcref{tab:model-defns}. We have the following result.
\begin{prop} \label{prop:candelabra}
We have
\begin{equation}
  \left\vert \Cov \left( \hat{\Pi}_{\x}^{\ve, \zeta} (\rscandelabrarrrr) (\y_1) , \hat{\Pi}_{\x}^{\ve, \zeta}  (\rscandelabrarrrr)  (\y_2) \right) \right\vert \lesssim \log (2 +d_{\mathfrak{s},\mathscr{S}}(\y_1, \y_2)^{-1})^4.\label{eq:candelabra-cov-bd}
  \end{equation}
  Also, for any ${\tilde{\kappa}}>0$, uniformly over $\lambda \in (0, 1)$ and  $g \in \mathcal{C}_{\mathrm{c}}^{1,2}$ with $\|g \|_{\mathcal C^{1,2}_{\mathrm c}}\leq 1$, we have
  \begin{equation}\label{eq:candelabra-mean-bd}
    \left\vert \EE \left \llangle\hat{\Pi}_{\x}^{\ve, \zeta} (\rscandelabrarrrr) ,g^\lambda_{\x}\right\rrangle \right\vert \lesssim_{{\tilde{\kappa}}}  \lambda^{-{\tilde{\kappa}}} .
  \end{equation}
  Moreover, for any $s \neq t \in \RR$ and $c\in(0,\infty)$, we have
\begin{equation}\label{e:newrbc5}
  \adjustlimits\lim_{\ve \downarrow 0} \sup_{\zeta \in (0, \ve)} \sup_{d_L(x),d_L(y) \leq c \ve} \; \left\vert  \Cov \left( \rscherryrenormrenorm[rsK]^{ \zeta} (t, x), \rscherryrenormrenorm[rsK]^{\zeta} (s, y)  \right) \right\vert = 0 .
\end{equation}
\end{prop}

\begin{proof}
  Using the symmetries of the diagram, we see that the covariance is bounded by the following contractions:
\begin{equation*}
  \begin{aligned}
    \left\vert \Cov \left( \hat{\Pi}_{\x}^{\ve, \zeta} (\rscandelabrarrrr) (\y_1) , \hat{\Pi}_{\x}^{\ve, \zeta}  (\rscandelabrarrrr)  (\y_2) \right) \right\vert & \lesssim \left( \left\vert \rscandelabrarrrrcA[rsK]^{\zeta} \right\vert + \left\vert \rscandelabrarrrrcB[rsK]^{\zeta} \right\vert +  \left\vert \rscandelabrarrrrcC[rsK]^{\zeta} \right\vert  +  \left\vert \rscandelabrarrrrcD[rsK]^{\zeta} \right\vert \right) \left(\y_1, \y_2 \right) .
  \end{aligned}
\end{equation*}
Here we have used the black triangle notation introduced in \zcref{eq:black-triangle-def}.
We use \zcref{lem:one-int-bound,eq:redcherryx1x2bd,eq:renorm-constants-same} to estimate
\begin{equation*}
  \left( \left\vert \rscandelabrarrrrcA[rsK]^{\zeta} \right\vert
  +  \left\vert \rscandelabrarrrrcC[rsK]^{\zeta}\right\vert \right) \left(\y_1, \y_2\right) \lesssim \log(2 +d_{\mathfrak{s},\mathscr{S}}(\y_1, \y_2)^{-1})^2.
\end{equation*}
The proof of \zcref{eq:candelabra-cov-bd} will thus be complete once we show that %
\begin{equation}
  \begin{aligned}
    \left( \left\vert  \rscandelabrarrrrcB[rsK]^{\zeta} \right\vert
    + \left\vert \rscandelabrarrrrcD[rsK]^{\zeta} \right\vert \right) \left(\y_1, \y_2\right) \lesssim \log(2 +d_{\mathfrak{s},\mathscr{S}}(\y_1, \y_2)^{-1})^{4}.
  \end{aligned}\label{eq:two-contractions}
\end{equation}
We obtain this last estimate by applying \zcref{prop:graph-bound}. The power of $4$ is
because each of the contracted diagrams has $4$ inner vertices (disregarding both roots and noise vertices). In order to apply \zcref{prop:graph-bound} we must make sure that the $\gamma$ from \zcref{eq:gammadef} is zero, and furthermore, that for any subdiagram $\overline{\mathsf{G}}$, we have $\deg (\overline{\mathsf{G}}) >0$. The latter condition is a consequence of \zcref{lem:check-degcond-simpler,rem:how-to-check}. Since there is no standard reference for computations of two-point correlations, we check the condition on $\gamma$ precisely. (Incidentally, the same approach proves also that there is no sub-diagram with negative degree.)

The key is to graphically represent the different subdiagrams. As a first step, rather than the graph $\mathsf{G}= \tau\hash_{\nu}\tau$, we can consider the reduced diagram $\mathsf{G}_c$ constructed in \zcref{subsec:Convergent-Feynman-diagrams}, in which we have replaced every pair of edges connecting to the same contracted vertex with a single edge of weight $-1$. (We draw this new edge in purple below.) %
Then, instead of considering subdiagrams, we consider only subsets of vertices of $\mathsf{G}_{\mathrm{c}}$, to which we associate the full (induced) subdiagram that they span. This is equivalent by \zcref{eq:gammadegtilde}. 
To start, here is a representation of $\mathsf{G}_{\mathrm{c}}$ in the case of $\rscandelabrarrrrcB$:
\begin{equation*}
  \begin{tikzpicture}[scale=0.5]
    \tikzset{
      bullet/.style={circle, fill=black, draw=black, inner sep=1.6pt},
      hollow/.style={circle, draw=black, fill=white, inner sep=1.6pt}
    }
    \node[bullet] (L) at (-0.7,1) {};
    \node[bullet] (R) at ( 0.7,1) {};
    \node[bullet] (C) at (0,0) {};
    \node[bullet] (L1) at (2.3,1) {};
    \node[bullet] (R1) at ( 3.7,1) {};
    \node[bullet] (D) at (3,0) {};
    \draw (C) -- (L);
    \draw (C) -- (R);
    \draw (D) -- (L1);
    \draw (D) -- (R1);
    \draw[purple, thick] (R) to[out=90,in=90, looseness=1.0] (L1);
    \draw[purple, thick] (L) to[out=90,in=90, looseness=1.0] (L1);
    \draw[purple, thick] (R) to[out=90,in=90, looseness=1.0] (R1);
    \draw[purple, thick] (L) to[out=90,in=90, looseness=1.0] (R1);
  \end{tikzpicture}.
\end{equation*}
Now, from \zcref{eq:gammadef} we must only consider subdiagrams that contain both roots. We represent with white bullets the vertices that are not in the subdiagram, and with black bullets the vertices that are in the subdiagram. Thus, for example, one of the subgraphs of $\mathsf{G}_{\mathrm{c}}$ containing three vertices is represented as follows: %
Here are two examples of the correspondence between subgraphs of $\mathsf{G}_{\mathrm{c}}$ and subdiagrams:
\begin{equation*}
  \begin{tikzpicture}[scale=0.5,baseline=0.5]
    \tikzset{
      bullet/.style={circle, fill=black, draw=black, inner sep=1.6pt},
      hollow/.style={circle, draw=black, fill=white, inner sep=1.6pt}
    }
    \node[bullet] (L) at (-0.7,1) {};
    \node[hollow] (R) at ( 0.7,1) {};
    \node[bullet] (C) at (0,0) {};
    \node[hollow] (L1) at (2.3,1) {};
    \node[hollow] (R1) at ( 3.7,1) {};
    \node[bullet] (D) at (3,0) {};
  \end{tikzpicture}
  \quad\rightsquigarrow\quad
  \begin{tikzpicture}[scale=0.5,baseline=0.5]
    \tikzset{
      bullet/.style={circle, fill=black, draw=black, inner sep=1.6pt},
      hollow/.style={circle, draw=black, fill=white, inner sep=1.6pt}
    }
    \node[bullet] (L) at (-0.7,1) {};
    \node[hollow] (R) at ( 0.7,1) {};
    \node[bullet] (C) at (0,0) {};
    \node[hollow] (L1) at (2.3,1) {};
    \node[hollow] (R1) at ( 3.7,1) {};
    \node[bullet] (D) at (3,0) {};
    \draw (C) -- (L);
  \end{tikzpicture}
  \qquad\qquad\text{and}\qquad\qquad
  \begin{tikzpicture}[scale=0.5,baseline=0.5]
    \tikzset{
      bullet/.style={circle, fill=black, draw=black, inner sep=1.6pt},
      hollow/.style={circle, draw=black, fill=white, inner sep=1.6pt}
    }
    \node[bullet] (L) at (-0.7,1) {};
    \node[bullet] (R) at ( 0.7,1) {};
    \node[bullet] (C) at (0,0) {};
    \node[bullet] (L1) at (2.3,1) {};
    \node[hollow] (R1) at ( 3.7,1) {};
    \node[bullet] (D) at (3,0) {};
  \end{tikzpicture}
\quad\rightsquigarrow\quad
  \begin{tikzpicture}[scale=0.5,baseline=0.5]
    \tikzset{
      bullet/.style={circle, fill=black, draw=black, inner sep=1.6pt},
      hollow/.style={circle, draw=black, fill=white, inner sep=1.6pt}
    }
    \node[bullet] (L) at (-0.7,1) {};
    \node[bullet] (R) at ( 0.7,1) {};
    \node[bullet] (C) at (0,0) {};
    \node[bullet] (L1) at (2.3,1) {};
    \node[hollow] (R1) at ( 3.7,1) {};
    \node[bullet] (D) at (3,0) {};
    \draw (C) -- (L);
    \draw (C) -- (R);
    \draw (D) -- (L1);
    \draw[purple, thick] (R) to[out=90,in=90, looseness=1.0] (L1);
    \draw[purple, thick] (L) to[out=90,in=90, looseness=1.0] (L1);
  \end{tikzpicture}.
\end{equation*}
We observe that the map from subsets of vertices to subdiagrams is one-to-one. Moreover, since we always include the two roots, all possibilities are reduced to the choice of which subset $\overline{\mathsf{V}}$ of the top four vertices are included in the subdiagram. To each such $\mathsf{V}$  we associate the number $\gamma(\overline{\mathsf{V}}) \coloneqq \deg(\overline{\mathsf{V}}\cup\{\varrho_1,\varrho_2\}) -3$ as defined in \zcref{eq:degree-Vbar}. %
We can list all possible subdiagrams (up to symmetries) %
and their associated $\gamma$s as follows:
\newcommand{\diagramfour}[4]{\begin{tikzpicture}[scale=1]
    \tikzset{bullet/.style={circle, fill=black, draw=black, inner sep=1.6pt},
             hollow/.style={circle, draw=black, fill=white, inner sep=1.6pt}}
    \node[#1] (A) at (0,0) {};
    \node[#2] (B) at (0.3,0) {};
    \node[#3] (C) at (0.6,0) {};
    \node[#4] (D) at (0.9,0) {};
\end{tikzpicture}}
\begin{equation}
\begin{tabular}{c|cccccc}
  \toprule
  $\overline{\mathsf{V}}$ &
\diagramfour{hollow}{hollow}{hollow}{hollow}&
\diagramfour{bullet}{hollow}{hollow}{hollow}&
\diagramfour{bullet}{bullet}{hollow}{hollow}&
\diagramfour{bullet}{hollow}{bullet}{hollow}&
\diagramfour{bullet}{bullet}{bullet}{hollow}&
\diagramfour{bullet}{bullet}{bullet}{bullet}\\
  \midrule
  $\gamma(\overline{\mathsf{V}})$ &
  $0$&$1$ & $2$ & $1$ & $1$ & $0$ \\
  \bottomrule
\end{tabular}.
\end{equation}
We see that $\min_{\overline{\mathsf{V}}} \gamma(\overline{\mathsf{V}}) = 0$, %
which implies the desired result.

  The second contraction on the left side of \zcref{eq:two-contractions} can be treated in exactly the same manner. This time we start from the diagram
  \begin{equation*}
    \begin{tikzpicture}[scale=0.5]
      \tikzset{
        bullet/.style={circle, fill=black, draw=black, inner sep=1.6pt},
        hollow/.style={circle, draw=black, fill=white, inner sep=1.6pt}
      }
      \node[bullet] (L) at (-0.7,1) {};
      \node[bullet] (R) at ( 0.7,1) {};
      \node[bullet] (C) at (0,0) {};
      \node[bullet] (L1) at (2.3,1) {};
      \node[bullet] (R1) at ( 3.7,1) {};
      \node[bullet] (D) at (3,0) {};
      \draw (C) -- (L);
      \draw (C) -- (R);
      \draw (D) -- (L1);
      \draw (D) -- (R1);
      \draw[purple, thick] (R) to[out=90,in=90, looseness=1.0] (L1);
      \draw[purple, thick] (L) to[out=90,in=90, looseness=1.0] (R);
      \draw[purple, thick] (R1) to[out=90,in=90, looseness=1.0] (L1);
      \draw[purple, thick] (L) to[out=90,in=90, looseness=1.0] (R1);
    \end{tikzpicture}.
  \end{equation*}
  We again list all possible subdiagrams (up to symmetries) and associated $\gamma$s:
  \begin{equation}
    \begin{tabular}{c|ccccccc}
  \toprule
  $\overline{\mathsf{V}}$ &
\diagramfour{hollow}{hollow}{hollow}{hollow}&
\diagramfour{bullet}{hollow}{hollow}{hollow}&
\diagramfour{bullet}{bullet}{hollow}{hollow}&
\diagramfour{bullet}{hollow}{bullet}{hollow}&
\diagramfour{bullet}{hollow}{hollow}{bullet}&
\diagramfour{bullet}{bullet}{bullet}{hollow}&
\diagramfour{bullet}{bullet}{bullet}{bullet}
   \\
  \midrule
  $\gamma(\overline{\mathsf{V}})$ &
  $0$&$1$ & $1$ & $2$ & $1$ & $1$ & $0$ \\
  \bottomrule
\end{tabular}.
\end{equation}
 We find again that $\min_{\overline{\mathsf{V}}} \gamma(\overline{\mathsf{V}}) = 0$, which allows us to apply \zcref{prop:graph-bound} to obtain the desired estimate. This concludes the analysis of the covariance (i.e.\ the proof of \zcref{eq:candelabra-cov-bd}).

  Now we turn our attention to the analysis of the mean, i.e.\ the proof of \zcref{eq:candelabra-mean-bd}.
We find, recalling \zcref{tab:model-defns} and the black triangle notation from \zcref{eq:black-triangle-def}, that
\begin{equation}
  \EE \left[ \hat{\Pi}_{\x}^{\zeta} (\rscandelabrarrrr) (\y)\right] = \EE\rscherryrenormrenorm[rsK]^\zeta(\y)-C^{(2)}_\zeta = \rscandelabrarrrrcE[rsK]^{\zeta} (\y) + \rscandelabrarrrrcF[rsK]^{\zeta}(\y) - C_\zeta^{(2)}.\label{eq:exp-candelabra-decompose}
\end{equation}
By \zcref{eq:renorm-constants-same} we have
\begin{equation*}%
  \left\lvert \rscandelabrarrrrcE[rsK]^{\zeta}(\y) \right\rvert %
  \lesssim 1 .
\end{equation*}
Using this and
\zcref{prop:sumofcancellinglogs} in \zcref{eq:exp-candelabra-decompose}, %
we can
estimate for any fixed $\tilde{\kappa}>0$, uniformly over $\lambda \in (0,1)$ and bounded $g \in \mathcal{C}_{\mathrm{c}}^{1,2}$, that
\begin{equation*}
  \left\lvert \left\llangle \EE \left[ \hat{\Pi}_{\x}^{\zeta} (\rscandelabrarrrr) (\y)\right] %
    ,g^\lambda_\x\right\rrangle %
    \right\vert \lesssim_{{\tilde{\kappa}}} \lambda^{-{\tilde{\kappa}}} + \sup_{\y 
  \in \mS_{2L}} \left\vert \rscandelabrarrrrcF[rsK]^{\zeta} (\y) - \rscandelabrarrrrcF[rsP]^{\zeta} (\y) \right\rvert .
\end{equation*}
To estimate the right side of this inequality, we can represent the term inside the supremum as the integral of two different kernel sets (one with $q'$ and one with $J'$) over the same diagram $\mathsf{G}=\rscandelabrarrrrcF[rsN]$, in the following sense:
\begin{equation*}
  \begin{aligned}
    \rscandelabrarrrrcF[rsK]^{\zeta} (\y) & - \rscandelabrarrrrcF[rsP]^{\zeta} (\y) \\
                                                     & = \int_{\mS_{2L}^{\mathsf{V}_{\mathrm{int}}}} \left[ \left( \prod_{e \in \mathsf{E}_{\mI'}} J^{\prime}(\x_{e_{\downarrow}} - \x_{e_{\uparrow}}) \right) - \left( \prod_{e \in \mathsf{E}_{\mI'}} q^{\prime}(\x_{e_{\downarrow}} - \x_{e_{\uparrow}}) \right) \right]
    \left( \prod_{e \in \mathsf{E}_\nu} \mE^{\zeta}(\x_{e_1}, \x_{e_2})  \right) \prod_{v\in\mathsf{V}_{\mathrm{int}}}\dif\mathbf{x}_{v} \\
    & = \sum_{\emptyset \neq A \subseteq \mathsf{E}_{\mI'}} (-1)^{|A|} \int_{\mS_{2L}^{\mathsf{V}_{\mathrm{int}}}} \left( \prod_{e \in A } \tilde{J}'(\x_{e_{\downarrow}} - \x_{e_{\uparrow}}) \right) \left(\prod_{e \in \mathsf{E}_{\mI'}\setminus A } q' (\x_{e_{\downarrow}} - \x_{e_{\uparrow}})\right)
    \left( \prod_{e \in \mathsf{E}_\nu} \mE^{\zeta}(\x_{e_1}, \x_{e_2})  \right) \prod_{v\in\mathsf{V}_{\mathrm{int}}}\dif\mathbf{x}_{v},
  \end{aligned}
\end{equation*}
where we used the same notation as in \zcref{eq:define-contraction}, together with the multinomial theorem and the identity $q = J + \tilde{J}$. For $\emptyset \ne A \subseteq \mathsf{E}_{\mI'}$, let us pick arbitrarily some $e_A\in A$. Now since in every term of the last sum at least one of the kernels is smooth, we can estimate the integral as follows:
Now for $\varnothing\ne A\subseteq\mathsf{E}_{\mathcal{I}'}$, since at least one of the kernels is smooth, we can estimate
\begin{equation*}
    \bigg\vert \int_{\mS_{2L}^{\mathsf{V}_{\mathrm{int}}}} \bigg(  \prod_{e \in A } \tilde{J}'  (\x_{e_{\downarrow}} - \x_{e_{\uparrow}}) \bigg) \left(\prod_{e \in \mathsf{E}_{\mI'}\setminus A } q' (\x_{e_{\downarrow}} - \x_{e_{\uparrow}})\right)
    \left( \prod_{e \in \mathsf{E}_\nu} \mE^{\zeta}(\x_{e_1}, \x_{e_2})  \right) \prod_{v\in\mathsf{V}_{\mathrm{int}}}\dif\mathbf{x}_{v} \bigg\vert \\
    \lesssim 1 
\end{equation*}
by \zcref{lem:one-int-bound,eq:redcherryx1x2bd}.
The only difference between this situation and the setting of \zcref{lem:one-int-bound} is that here we are not
integrating over a compact set, but this can be handled easily by the decay
$|\tilde{J}'(\y)| \lesssim e^{-c|\y|^2_{\mf{s}}}$ of the kernel. This concludes
the proof of \zcref{eq:candelabra-mean-bd}.

As for \zcref{e:newrbc5}, we proceed similarly to previous cases. Since
$\rsiprenorm[rsK]^\zeta (t, 0)= 0$ we find that 
\begin{equation*}
  \left\vert  \Cov \left( \rscherryrenormrenorm[rsK]^{ \zeta} (t, x), \rscherryrenormrenorm[rsK]^{\zeta} (s, y)  \right) \right\vert \leq \left\vert  \Cov \left( \rscherryrenormrenormXA[rsK]^{ \zeta, (0, x)} (t, 0), \rscherryrenormrenorm[rsK]^{\zeta} (s, y)  \right) \right\vert + \left\vert  \Cov \left( \rscherryrenormrenormXB[rsK]^{ \zeta, (0, -x)} (t, x), \rscherryrenormrenorm[rsK]^{\zeta} (s, y)  \right) \right\vert ,
\end{equation*}
where we have followed the notation introduced in the proofs of
\zcref{prop:rbrelk,prop:rrbelk, prop:elkrrr}. In particular, a combination of
the arguments that we have already used in this proof, together with the
estimate \zcref{e:Kx-bd}, leads to the following estimate for any $\delta \in (0,1)$:
\begin{equation*}
  \left\vert  \Cov \left( \rscherryrenormrenormXA[rsK]^{ \zeta} (t, 0; x), \rscherryrenormrenorm[rsK]^{\zeta} (s, y)  \right) \right\vert + \left\vert  \Cov \left( \rscherryrenormrenormXB[rsK]^{ \zeta} (t, x; -x), \rscherryrenormrenorm[rsK]^{\zeta} (s, y)  \right) \right\vert \lesssim_{s, t, \delta} |x|^\delta.
\end{equation*}
To prove this directly, $\delta$ should be chosen sufficiently small that all the degrees of the
subdiagrams in the trees remain strictly positive. (The kernel $J'_{(0,x)}$ counts
as having explosion of rate $-2-\delta$ via \zcref{e:Kx-bd}, compared to the
rate $-2$ for $J'$.) Note that the
degrees of the original subdiagrams take integer values and we have proven that
they are all strictly positive. Therefore, a subdiagram
$\overline{\mathsf{G}}$ passes from $\deg (\overline{\mathsf{G}}) = \alpha$ to
$\deg(\overline{\mathsf{G}}) \geq \alpha - \delta$ when one replaces an
edge in the diagram with a crossed edge. %
Since in the original diagram $\alpha
\geq 1$, we will have $\alpha - \delta > 0$ as long as $\delta < 1$, which implies that we can choose any
$\delta < 1$
without affecting the convergence of the Feynman diagram. 
This implies \zcref{e:newrbc5} and completes the proof of the proposition.
\end{proof}

\subsubsection{$\protect \rsmooserrrr$}

The last stochastic estimate concerns the tree $\rsmooserrrr$. This tree is the most complicated one, requiring both recentering and renormalization. However, we will be able to use the estimates that we have already obtained for $\rselkrrr$ to simplify the analysis somewhat. We recall from \zcref{tab:model-defns} that
\begin{equation}\label{e:recentre-moose}
  \hat{\Pi}^{\ve, \zeta}_\x  \left( \rsmooserrrr \right) (\y) = \left(\rsipcherryrenormr[rsK]^{\zeta} (\y)-\rsipcherryrenormr[rsK]^{\zeta} (\x) \right)\rslollipopr[rsK]^{\zeta} (\y) +\frac{1}{4}C^{(2)}_{\zeta} ,
\end{equation}
where $C^{(2)}_{\zeta}$ is defined in \zcref{eq:C2zetadef}.
\begin{prop}\label{prop:moose}
    We have uniformly over $\y_1, \y_2 \in \mS_{2L}$ that
    \begin{equation}\label{e:moose-aim-1}
        \begin{aligned}
          \left\vert \Cov \left( \rselkrenormrr[rsK]^{ \zeta} (\y_1) ,\rselkrenormrr[rsK]^{ \zeta} (\y_2)\right) \right\vert \lesssim d_{\mathfrak{s},\mathscr{S}}(\y_1, \y_2)^{-1} .
        \end{aligned}
      \end{equation}
      In addition, for any ${\tilde{\kappa}} > 0$ and uniformly over $\x \in
      \mS_{2L}, \lambda \in (0,1)$, and $g \in \mathcal{C}^{1,2}_{\mathrm{c}}$ with $\|g \|_{\mathcal C^{1,2}_{\mathrm c}}\leq 1$, we have
      \begin{equation} \label{e:moose-aim-2}
        \EE \left[  \left\llangle\hat{\Pi}^{\ve, \zeta}_{\x} (\rsmooserrrr), g^\lambda_\x\right\rrangle^2  \right] \lesssim_{{\tilde{\kappa}}} \lambda^{-{\tilde{\kappa}}}.
      \end{equation}
      Moreover, for any $s \neq t \in \RR$ and $c\in(0,\infty)$, we have
      \begin{equation}\label{e:newrbc6}
        \adjustlimits\lim_{\ve \downarrow 0} \sup_{\zeta \in (0, \ve)} \sup_{d_L(x),d_L(y) \leq c \ve} \; \left\vert  \Cov \left( \rselkrenormrr[rsK]^{ \zeta} (t, x), \rselkrenormrr[rsK]^{\zeta} (s, y)  \right) \right\vert = 0 .
      \end{equation}
\end{prop}

\begin{proof}
  The proof is more complicated than that for $\rscandelabrarrrr$ because of the need for recentering.
  However the overall strategy is the same as in the previous estimates.
  
  In the case of this tree, we will not draw all the contractions because the tree has few symmetries and we would end up with too many terms. Instead, for a tree $\tau$ in a fixed $n$-th homogeneous chaos, we write $\mathscr{C}^{(n)}(\tau)$ for all the possible pairings of two identical copies of $\tau$ (a pairing $\nu \in \mathscr{C}^{(n)}(\tau)$ does not contain any internal contractions). We find from \zcref{e:recentre-moose}, writing $\mathscr{C}^{(n)}$ instead of $\mathscr{C}^{(n)}(\tau)$ since the reference tree is clear from context, that
\begin{equation*}
  \begin{aligned}
    \left\vert \EE \left[ \hat{\Pi}_{\x}^{\ve, \zeta} (\rsmooserrrr) (\y_1) \hat{\Pi}_{\x}^{\ve, \zeta} (\rsmooserrrr) (\y_2) \right] \right\vert  \lesssim  \bigg[ & \sum_{\nu \in \mathscr{C}^{(4)}} \boxop{\rsmooserrrrD  \hash_\nu \rsmooserrrrR }^{\zeta, \x}  + \sum_{\nu \in \mathscr{C}^{(2)}} \boxop{ \rsmooserrrrcAD \hash_\nu \rsmooserrrrcAR }^{
  \zeta, \x} + \sum_{\nu \in \mathscr{C}^{(2)}} \boxop{\rsmooserrrrcDD \hash_\nu \rsmooserrrrcDR }^{
    \zeta, \x} \\
  & + \sum_{\nu \in \mathscr{C}^{(2)}} \boxop{ \rsmooserrrrcBD  \hash_\nu \rsmooserrrrcBR }^{
  \zeta, \x} + \sum_{\nu \in \mathscr{C}^{(2)}} \boxop{ \rsmooserrrrcCD  \hash_\nu \rsmooserrrrcCR }^{\zeta, \x} \\
  & + \Big\vert \rsmooserrrrccAA[rsK]^{\zeta, \x} + 2 \, \rsmooserrrrccB[rsK]^{\zeta, \x} - \frac{1}{4} C^{(2)}_{\zeta} \Big\vert^{\otimes 2} \bigg]\left( \y_1 , \y_2 \right).
  \end{aligned}
\end{equation*}
Moreover, we have
\begin{align}
  \left\vert \Cov \left( \hat{\Pi}_{\x}^{\ve, \zeta} (\rsmooserrrr) (\y_1) ,  \hat{\Pi}_{\x}^{\ve, \zeta} (\rsmooserrrr) (\y_2) \right) \right\vert  \lesssim  \bigg[ & \sum_{\nu \in \mathscr{C}^{(4)}} \boxop{ \rsmooserrrrD \hash_\nu \rsmooserrrrR }^{\zeta, \x} + \sum_{\nu \in \mathscr{C}^{(2)}} \boxop{ \rsmooserrrrcAD \hash_\nu \rsmooserrrrcAR }^{
  \zeta, \x} + \sum_{\nu \in \mathscr{C}^{(2)}} \boxop{ \rsmooserrrrcDD \hash_\nu \rsmooserrrrcDR }^{
    \zeta, \x}\notag \\
  & + \sum_{\nu \in \mathscr{C}^{(2)}} \boxop{\rsmooserrrrcBD \hash_\nu \rsmooserrrrcBR }^{
  \zeta, \x} + \sum_{\nu \in \mathscr{C}^{(2)}} \boxop{ \rsmooserrrrcCD \hash_\nu \rsmooserrrrcCR }^{\zeta, \x} \bigg]\left( \y_1 , \y_2 \right).\label{e:moose-cov-to-estimate}
  \end{align}
In both cases we used the crossed edge notation introduced in
\zcref{e:recenter-notation}\ to indicate recentering. 

Moreover, we note that if $\x$ is of the form $\x=(t,0)$, then the recentering is not needed, since $\rsipcherryrenormr[rsK]^\zeta(\x) =0$ and hence \zcref{e:recentre-moose} reduces to %
\begin{equation}\label{e:reconst-moose-boundary}
  \hat{\Pi}_{(t,0)}^{\ve, \zeta} (\rsmooserrrr) (\y) %
  = \rselkrenormrr[rsK]^{\zeta} (\y) + \tfrac14C^{(2)}_\zeta.
\end{equation}
This observation will be useful in some cases (in particular, when contractions ``factorize,'' which is considered in \zcref{step:factorize} below) to obtain the bounds \zcref{e:moose-aim-1,e:moose-aim-2} simultaneously.

The estimates follow mostly along the lines of the arguments that we have already used. We proceed in several steps.

\begin{thmstepnv}
\item\emph{Contractions that factorize.}\label{step:factorize} We start by
defining $\mathscr{C}_{\mathrm{l}}\subseteq \mathscr{C} (\rsmooserrrrD ,
\rsmooserrrrD)$
to be
the set of contractions in which the bottom noise nodes of each tree are matched
with each other. %
  Such contractions appear both in the fourth and in the second chaos, and we single them out in order to bound them separately. Two examples of  contractions in $\mathscr{C}_{\mathrm{l}}$ are
\begin{equation*}
  \rsmooserrrrcPA \qquad\text{and}  \qquad \rsmooserrrrcPE .
\end{equation*}
We note that by definition %
if $\nu\in\mathscr{C}_{\mathrm{l}}$, then there exists a $\nu_1 \in \mathscr{C}(\rsipelkrrr , \rsipelkrrr)$ such that
\begin{equation}
  \boxop{\rsmooserrrrD  \hash_\nu \rsmooserrrrR}^{\zeta, \x} (\y_1, \y_2) = \boxop{ \rsipelkrrrR \hash_{\nu_1} \rsipelkrrrR }^{\zeta, \x} (\y_1, \y_2)  \cdot  \rslollirc[rsK]^{\zeta} (\y_1, \y_2).\label{eq:Clfactors}
\end{equation}
We can treat the first factor on the right side of \zcref{eq:Clfactors} %
via \zcref{prop:elkrrr}  (or more directly via the bounds on the contractions on the right side of \zcref{eq:elkrrr-cov-exp} in its proof).
Indeed, that result along with \zcref{e:Kx-bd} implies that, for any fixed $\tilde\kappa>0$, 
\begin{align*}
    \Big\vert & \boxop{ \rsipelkrrrR \hash_{\nu_1} \rsipelkrrrR }^{\zeta, \x} (\y_1, \y_2) \Big\vert\notag\\&\le \int_{\mS_{2L}^2} |J'(\y_1 - \z_1) - J'(\x - \z_1)| | J'(\y_2 - \z_2) - J'(\x - \z_2) \left\lvert\boxop{ \rselkrrr \hash_{\nu_1} \rselkrrr }^{\zeta} (\z_1, \z_2)\right\rvert %
    \ud \z_{1} \ud \z_2 \\
              & \lesssim  |\x - \y_1|^{\delta} |\x - \y_2|^{\delta} \int_{\mS_{2L}^2} (|\y_1 - \z_1|^{-2 -\delta} + |\x- \z_1|^{-2 -\delta}) (|\y_2 - \z_2|^{-2 -\delta} + |\x- \z_2|^{-2 -\delta})\notag\\&\hspace{4in}\cdot d_{\mathfrak{s},\mathscr{S}}(\z_1, \z_2)^{-1-\tilde\kappa} \ud \z_{1} \ud \z_2 \\
  & \lesssim |\x - \y_1|^{\delta} |\x - \y_2|^{\delta}.
  \end{align*}
Here the last estimate is a consequence of \zcref{lem:one-int-bound}, as long as we choose $\delta$ and $\tilde\kappa$ small enough that $2\delta+\tilde\kappa<1$. Using this estimate together with together with \zcref{prop:lollibd} in \zcref{eq:Clfactors}, we obtain for any $\nu \in \mathscr{C}_{\mathrm{l}}$ that for any ${\tilde{\kappa}} > 0$,
\begin{equation*}
  \left\lvert\boxop{\rsmooserrrrD  \hash_\nu \rsmooserrrrR}^{\zeta, \x} (\y_1, \y_2)\right\rvert \lesssim_{\tilde{\kappa}} |\x - \y_1|^{1/2 - {\tilde{\kappa}}} |\x - \y_2|^{1/2 - {\tilde{\kappa}}} d_{\mathfrak{s},\mathscr{S}}(\y_1, \y_2)^{-1} .
\end{equation*}
This is an estimate of the right order for the covariance bound \zcref{e:moose-aim-1} (setting $\x=(t,0)$ and recalling \zcref{e:reconst-moose-boundary}), and by scaling it also implies that
\begin{equation*}
  \left\vert \left\llangle\boxop{\rsmooserrrrD  \hash_\nu \rsmooserrrrR}^{\zeta, \x}, ( g^\lambda_\x)^{\otimes 2}\right\rrangle \right\vert  \lesssim \lambda^{- 2 {\tilde{\kappa}}} ,
\end{equation*}
which is of the desired order for \zcref{e:moose-aim-2}.

\item \label{step:cov-bd} \emph{The covariance bound.} For the remaining contractions, the recentering causes some additional complications, and so the proofs of \zcref{e:moose-aim-1} and \zcref{e:moose-aim-2} require somewhat different arguments. Thus in this step we complete the proof of \zcref{e:moose-aim-1}, leaving the proof of \zcref{e:moose-aim-2} to \zcref{step:recentering,step:renormalization} below. In \zcref{step:factorize} we already obtained the requisite bound for contractions in $\mathscr{C}_{\mathrm{l}}$, so we now turn our attention to the remaining contractions, working chaos by chaos. %

We start with contractions in $\mathscr{C}^{(4)} \setminus
\mathscr{C}_{\mathrm{l}}$. In this case we apply \zcref{prop:graph-bound} with
$\gamma =0$. To apply the proposition we must control the minimal degree over
all subdiagrams. Here we follow the same approach taken in the proof of
\zcref{prop:candelabra}. Namely, any contraction $\nu \in \mathscr{C}^{(4)}
\setminus \mathscr{C}_{\mathrm{l}}$ leads to a reduced diagram $\mathsf{G}_\nu$
over the inner vertices only, and where purple lines represent edges associated
to kernels with blow-up controlled by $d_{\mathfrak{s},\mathscr{S}}( \cdot
)^{-1}$.
For example,%
\begin{equation}
  \text{the contraction}\qquad
  \rsmooserrrrcPGI
  \qquad
  \text{leads to the reduced diagram}\qquad
  \begin{tikzpicture}[scale=0.5,baseline=(current bounding box.center)]
    \tikzset{
      bullet/.style={circle, fill=black, draw=black, inner sep=1.6pt},
      hollow/.style={circle, draw=black, fill=white, inner sep=1.6pt}
    }
    \node[bullet] (L1) at (-0.7,0) {};
    \node[bullet] (L2) at (-0.7,1) {};
    \node[bullet] (L3) at (-0.7,2) {};
    \node[bullet] (R1) at (0.7,0) {};
    \node[bullet] (R2) at (0.7,1) {};
    \node[bullet] (R3) at (0.7,2) {};
    \draw[thick] (L1) to (L2);
    \draw[thick] (L2) to (L3);
    \draw[thick] (R1) to (R2);
    \draw[thick] (R2) to (R3);
    \draw[purple, thick] (L1) to[out=0,in=180, looseness=1.0] (R2);
    \draw[purple, thick] (L2) to[out=0,in=180, looseness=1.0] (R1);
    \draw[purple, thick] (L3) to[out=110,in=70, looseness=2.0] (R3);
    \draw[purple, thick] (L3) to[out=70,in=110, looseness=1.0] (R3);
  \end{tikzpicture}.\label{eq:example-contraction}
\end{equation}
Now for any contraction $\nu\in\mathscr{C}^{(4)}$, we let $\mathsf{G}_{\mathrm{c},\nu}$ denote the simplified multigraph corresponding to $\rsmooserrrr\hash_\nu\rsmooserrrr$ as constructed in \zcref{subsec:Convergent-Feynman-diagrams}. For any subset $\overline{\mathsf{V}}\subseteq\mathsf{V}(\mathsf{G}_{\mathrm{c},\nu})$ (which in fact does not depend on the contraction $\nu$), we define $\deg_{\mathsf{G}_{\mathrm{c},\nu}}(\overline{\mathsf{V}})$ to be the quantity defined in \zcref{eq:degree-Vbar} with respect to the graph $\mathsf{G}_{\mathrm{c},\nu}$. Then we define $\gamma(\nu,\overline{\mathsf{V}}) \coloneqq
\deg_{\mathsf{G}_{\mathrm{c},\nu}} (\overline{\mathsf{V}}) -3$. %
For example, in the case of $\nu$ depicted in \zcref{eq:example-contraction}, we have the following association of degree to subsets of vertices:
\newcommand{\diagramsix}[6]{%
  \begin{tikzpicture}[scale=0.5, baseline ={(0,-0.85)}]
    \tikzset{
      bullet/.style={circle, fill=black, draw=black, inner sep=1.6pt},
      hollow/.style={circle, draw=black, fill=white, inner sep=1.6pt}
    }
    \node[#1] (L1) at (-0.3,-2.0) {};
    \node[#2] (L2) at (-0.3,-1.5) {};
    \node[#3] (L3) at (-0.3,-1.0) {};
    \node[#4] (R1) at (0.3,-2.0) {};
    \node[#5] (R2) at (0.3,-1.5) {};
    \node[#6] (R3) at (0.3,-1.0) {};
\end{tikzpicture}}
\begin{equation*}
    \diagramsix{bullet}{hollow}{bullet}{bullet}{hollow}{bullet}
  \mapsto  \gamma(\nu,\overline{\mathsf{V}}) = 4 , \qquad 
    \diagramsix{bullet}{bullet}{bullet}{bullet}{hollow}{bullet}
  \mapsto  \gamma (\nu,\overline{\mathsf{V}}) = 3 .
\end{equation*}
However, we can go one step further, and for a given choice of inner vertices it is easy to actually compute the minimum of $\gamma (\nu,\overline{\mathsf{V}})$ over all $\nu \in \mathscr{C}^{(4)} \setminus \mathscr{C}_{\mathrm{l}}$. This is because the only choice involved lies in where the purples edges are attached, under the constraint of not joining the two bottom vertices, and it is easy to see for any given $\overline{\mathsf{V}}$ what the maximum number of purple edges in the induced subgraph can be. %
Indeed, if for any subset $\overline{\mathsf{V}}$ of inner vertices, we write \[\gamma(\overline{\mathsf{V}}) \coloneqq \min_{\nu \in \mathscr{C}^{(4)} \setminus \mathscr{C}_{\mathrm{l}}} \gamma(\nu,\overline{\mathsf{V}}).\]
then we can summarize the values of $\gamma(\overline{\mathsf{V}})$ as follows:
\newcommand{\diagramsixtight}[6]{\begin{tikzpicture}[scale=0.5, baseline={(0,-0.85)}]
    \tikzset{bullet/.style={circle, fill=black, draw=black, inner sep=1.6pt},
             hollow/.style={circle, draw=black, fill=white, inner sep=1.6pt}}
    \node[#1] (L1) at (-0.3,-2.0) {};
    \node[#2] (L2) at (-0.3,-1.6) {};
    \node[#3] (L3) at (-0.3,-1.2) {};
    \node[#4] (R1) at (0.3,-2.0) {};
    \node[#5] (R2) at (0.3,-1.6) {};
    \node[#6] (R3) at (0.3,-1.2) {};
\end{tikzpicture}}
\begin{equation}
  \begin{tabular}{c|cccccccccc}
  \toprule
  $\overline{\mathsf{V}}$ &
  $\diagramsixtight{bullet}{hollow}{hollow}{bullet}{hollow}{hollow}$ &
  $\diagramsixtight{bullet}{bullet}{hollow}{bullet}{hollow}{hollow}$ &
  $\diagramsixtight{bullet}{hollow}{bullet}{bullet}{hollow}{hollow}$ &
  $\diagramsixtight{bullet}{bullet}{hollow}{bullet}{bullet}{hollow}$ &
  $\diagramsixtight{bullet}{hollow}{bullet}{bullet}{bullet}{hollow}$ &
  $\diagramsixtight{bullet}{hollow}{bullet}{bullet}{hollow}{bullet}$ &
  $\diagramsixtight{bullet}{bullet}{bullet}{bullet}{hollow}{hollow}$ &
  $\diagramsixtight{bullet}{bullet}{bullet}{bullet}{bullet}{hollow}$ &
  $\diagramsixtight{bullet}{bullet}{bullet}{bullet}{hollow}{bullet}$ &
  $\diagramsixtight{bullet}{bullet}{bullet}{bullet}{bullet}{bullet}$ \\[3pt]
  \midrule
  $\gamma(\overline{\mathsf{V}})$ & $0$ & $0$ & $2$ & $0$ & $2$ & $3$ & $1$ & $1$ & $2$ & $0$ \\
  \bottomrule
\end{tabular}.
\end{equation}
Since $\min_{\overline{\mathsf{V}}} \gamma(\overline{\mathsf{V}}) =0$, we can apply \zcref{prop:graph-bound} with $\gamma =0$ (the positive degree condition verified via \zcref{lem:check-degcond-simpler,rem:how-to-check}), and we obtain for any $\nu \in \mathscr{C}^{(4)} \setminus \mathscr{C}_{\mathrm{l}}$ the estimate
\begin{equation*}
  \max_{\nu \in \mathscr{C}^{(4)}\setminus \mathscr{C}_{\mathrm{l}}}\big\vert \boxop{ \rsmooserrrr \hash_\nu \rsmooserrrrRA }^{\zeta}  \big\vert (\y_1, \y_2) \lesssim \log( 2+ d_{\mathfrak{s},\mathscr{S}}(\y_1, \y_2)^{-1} )^4 , 
\end{equation*}
which is a correct bound on the covariance. The same approach delivers the bound
\begin{equation*}
  \begin{aligned}
    \left\lvert \rsmooserrrrcPC[rsK]^{\zeta} + \rsmooserrrrcPD[rsK]^{\zeta} \right\rvert (\y_1, \y_2) \lesssim \log( 2+ d_{\mathfrak{s},\mathscr{S}}(\y_1, \y_2)^{-1} )^4 .
  \end{aligned}
\end{equation*}
where we now define $\gamma(\overline{\mathsf{V}})$ as the analogous minimum over these two contractions and summarize its values as follows:
\begin{equation}
\begin{tabular}{c|cccccccccc}
  \toprule
  $\overline{\mathsf{V}}$ &
  $\diagramsixtight{bullet}{hollow}{hollow}{bullet}{hollow}{hollow}$ &
  $\diagramsixtight{bullet}{bullet}{hollow}{bullet}{hollow}{hollow}$ &
  $\diagramsixtight{bullet}{hollow}{bullet}{bullet}{hollow}{hollow}$ &
  $\diagramsixtight{bullet}{bullet}{hollow}{bullet}{bullet}{hollow}$ &
  $\diagramsixtight{bullet}{hollow}{bullet}{bullet}{bullet}{hollow}$ &
  $\diagramsixtight{bullet}{hollow}{bullet}{bullet}{hollow}{bullet}$ &
  $\diagramsixtight{bullet}{bullet}{bullet}{bullet}{hollow}{hollow}$ &
  $\diagramsixtight{bullet}{bullet}{bullet}{bullet}{bullet}{hollow}$ &
  $\diagramsixtight{bullet}{bullet}{bullet}{bullet}{hollow}{bullet}$ &
  $\diagramsixtight{bullet}{bullet}{bullet}{bullet}{bullet}{bullet}$ \\[3pt]
  \midrule
  $\gamma(\overline{\mathsf{V}})$ & $0$ & $1$ & $2$ & $1$ & $2$ & $3$ & $1$ & $1$ & $2$ & $0$ \\

  \bottomrule
\end{tabular}.
\end{equation}

Let us proceed with contractions that contain the element $\rsbphz$. Via \zcref{lem:bphzbd,lem:one-int-bound,prop:lollibd}, for any ${\tilde{\kappa}} > 0$ (i.e.\ we absorb the logarithms into an arbitrarily small negative power), we estimate for some compact domain $\Theta \subseteq \mS_{2L}$ that
\begin{equation*}
  \left\vert \rsmooserrrrcPB[rsK]^{\zeta} \right\vert (\y_1, \y_2) \lesssim_{\tilde{\kappa}}  \int_{\Theta^2} |\y_1 - \z|_{\mf{s}}^{-2 -{\tilde{\kappa}} /2} |\y_2   - \z'|_{\mf{s}}^{-2-{\tilde{\kappa}} /2} d_{\mathfrak{s},\mathscr{S}}(\z, \z')^{-2} \ud \z \ud \z' \lesssim d_{\mathfrak{s},\mathscr{S}}(\y_1, \y_2)^{- {\tilde{\kappa}}} .
\end{equation*}
Similarly, we find a compact domain $\Theta \subseteq \mS_{2L}$ such that for any ${\tilde{\kappa}} > 0$, we have
\begin{equation*}
  \left\vert \rsmooserrrrcPF[rsK]^{\zeta}(\y_1, \y_2) \right\vert \lesssim_{\tilde{\kappa}} \int_{\Theta^2} |\y_1 - \z|_{\mf{s}}^{-2-{\tilde{\kappa}}/2} |\y_2 - \z'|_{\mf{s}}^{-2-{\tilde{\kappa}}/2} d_{\mathfrak{s},\mathscr{S}}(\y_1, \z')^{-1} d_{\mathfrak{s},\mathscr{S}}(\y_2, \z)^{-1} \ud \z \ud \z' \lesssim d_{\mathfrak{s},\mathscr{S}}(\y_1, \y_2)^{-{\tilde{\kappa}}} ,
\end{equation*}
where we have again used \zcref{lem:bphzbd,lem:one-int-bound}. Note that the only
other homogeneous (meaning that it does not pair vertices in the same tree that are not already paired in $\rsmooserrrrcCDI$) contraction 
of $\rsmooserrrrcCDI$ lies in $\mathscr{C}_{\mathrm{l}}$, so it has already been
dealt with.

Finally, the last remaining contraction can be bounded as %
\begin{equation*}
  \Big\vert \rsmooserrrrcDDIP[rsK]^\zeta \Big\vert (\y_1 ,\y_2) \lesssim \log( 2 +d_{\mf{s}, \mathscr{S}}(\y_1, \y_2)^{-1})^2
\end{equation*}
by \zcref{eq:renorm-constants-same,lem:one-int-bound,prop:lollibd}. This is a bound of the desired order and completes the proof of the bound on the covariance in \zcref{e:moose-aim-1}. %

\item\textit{Recentering.}\label{step:recentering} In this case we deal with the
recentering, in order to obtain \zcref{e:moose-aim-2}. Note that we must
only consider contractions that do not belong to
$\mathscr{C}_{\mathrm{l}}$ (as the latter contractions have already been handled in \zcref{step:factorize}). In other words, we must control all the contractions
that have been controlled in \zcref{step:cov-bd}, only this time including the
recentering. Our approach follows the one in \cites{hairer:quastel:2018:class},
namely we view the test function $g_\x^{\lambda}$ as an additional kernel that
is being integrated over, which therefore corresponds to a new edge in the
Feynman diagram associated to the contraction. For example, consider the contraction
  $\rsmooserrrrcPG$.
When this diagram is tested against $ (g_{\x}^\lambda)^{\otimes 2}$, we can represent $g^\lambda(\x -
\y)$ by a dashed line, so the overall integral
\[
\left\llangle \rsmooserrrrcPG[rsK]^{\zeta,\x},(g^\lambda_\x)^{\otimes 2}\right\rrangle = \iint \rsmooserrrrcPG[rsK]^{\zeta, \x}(\y_1,\y_2)g^\lambda(\x -\y_1)g^\lambda(\x-\y_2)\,\dif\y_1 \,\dif \y_2
\]
can be decomposed
into a linear combination of the three Feynman diagrams
\begin{equation} \label{e:recenter-decomp}
  \rsmooserecenterA, \qquad \qquad  \rsmooserecenterB, \qquad \qquad\text{and}\qquad\qquad  \rsmooserecenterC,
\end{equation}
each with the root evaluated at $\x$.
The three integrals appear in analogy to the decomposition of the product $(f(\y_1)
- f(\x)) h(\y_1) g^\lambda(\x - \y_1) (f(\y_2)
- f(\x)) h(\y_2) g^\lambda(\x - \y_2)$ into the following three terms (up to $\y_1
\mapsto \y_2$ symmetry):
\begin{equation*}
  \begin{aligned}
  &f(\y_1)f(\y_2) h(\y_1) h(\y_2) g^\lambda(\x - \y_1)g^\lambda(\x - \y_2) , \\
    &f(\x)f(\x) h(\y_1)h(\y_2) g^\lambda(\x - \y_1) g^\lambda(\x - \y_2), \\
  &f(\y_1)f(\x) h(\y_1)h(\y_2) g^\lambda(\x - \y_1)g^\lambda(\x - \y_2) .
  \end{aligned}
\end{equation*}
where $f$ corresponds to the $\rsipelkrrr$ diagram and $h$
corresponds to $\rslollipopr$. %
Now we estimate the test function $g^\lambda_\x$ by
\begin{equation*}
  |g_{\x}^{\lambda}(\y)| \lesssim \lambda^{-{\tilde{\kappa}}/2} |\y - \x|_{\mf{s}}^{{\tilde{\kappa}}/2 - 3} ,
\end{equation*}
for any ${\tilde{\kappa}}>0$. With the above estimate, the Feynman diagrams in
\zcref{e:recenter-decomp} become convergent (in the sense that they do not
contain any diverging subdiagram) by a similar argument to that of
\zcref{lem:check-degcond-simpler,rem:how-to-check}). 
Thus we can apply \cites[Proposition
2.4]{hairer:2018:analyst}, which is just \zcref{prop:graph-bound} with only one
vertex variable fixed rather than two, to obtain the bound
\begin{equation*}
\left\llangle \rsmooserrrrcPG[rsK]^{\zeta,\x},(g^\lambda_\x)^{\otimes 2}\right\rrangle 
  \lesssim \lambda^{ - {\tilde{\kappa}}} .
\end{equation*}
The same approach works for all the other remaining contractions. Indeed this approach
coincides with the one taken in \cites[Section 6.2.5]{hairer:quastel:2018:class}. To complete the proof of \zcref{e:moose-aim-2}, we must now   estimate the mean terms, which is done in \zcref{step:renormalization} below.

\item\emph{Renormalization.}\label{step:renormalization} Finally, we need to estimate the mean
terms, and in particular also the compensation due to the renormalization
constant. We start with $\rsmooserrrrccAA$, for which we have
\begin{equation*}
  \Big\vert \rsmooserrrrccAA[rsK]^{\ve, \zeta, \x} \Big\vert (\y) \lesssim 1 ,
\end{equation*}
by \zcref{eq:renorm-constants-same,lem:bphzbd}. Next we consider
$\rsmooserrrrccBI$, which requires compensation through the renormalization
constant $C^{(2)}_\zeta$. In this case, it follows from \zcref{prop:sumofcancellinglogs}
that for any ${\tilde{\kappa}} >0$, uniformly over $\x \in \mS_{2L}$, $\lambda \in (0, 1)$
and bounded $g \in \mathcal{C}^{1,2}_{\mathrm{c}}$, that
\begin{equation*}
  \left\lvert \left\llangle\rsmooserrrrccBI[rsK]^{\zeta} ,g_{\x}^\lambda\right\rrangle  + \frac{1}{8} C^{(2)}_{\zeta} \right\rvert  \lesssim \lambda^{-{\tilde{\kappa}}}+  \sup_{\y \in \mS_{2L}} \left\lvert \rsmooserrrrccBI[rsK]^{\zeta} - \rsmooserrrrccBI[rsP]^{\zeta} \right\rvert (\y) .
\end{equation*}
For the last term on the right side,  we use the same approach as in the proof of \zcref{prop:candelabra}. Using the multinomial theorem, we rewrite the difference, for any $\x \in \mS_{2L}$, as
\begin{equation*}
  \begin{aligned}
    \rsmooserrrrccBI[rsK]^{\zeta} (\x) & - \rsmooserrrrccBI[rsP]^{\zeta} (\x) \\
    & = \sum_{\emptyset \neq \mathsf{A} \subseteq \mathsf{E}_{\mI'}} (-1)^{|\mathsf{A}|} \int_{\mS_{2L}^{\mathsf{V}_{\mathrm{int}}}} \left( \prod_{e \in \mathsf{A} } \tilde{J}'(\x_{e_{\downarrow}} - \x_{e_{\uparrow}}) \right) \left(\prod_{e \in \mathsf{E}_{\mI'} \setminus \mathsf{A} } q' (\x_{e_{\downarrow}} - \x_{e_{\uparrow}})\right)
    \left( \prod_{e \in \nu} \mE^{\zeta}(\x_{e_1}, \x_{e_2})  \right) \prod_{v \in \mathsf{V}_{\mathrm{int}}} \ud x_{v} ,
  \end{aligned}
\end{equation*}
where $\mathsf{E}_{\mI'}$ is the set of edges that are not part of the
contraction in $\rsmooserrrrccBI$. Unlike in the proof of \zcref{prop:candelabra}, it is not the case that as soon
as $\mathsf{A}$ is nonempty, then there are no more divergences.
Indeed, if we denote
with a dashed internal line the kernel $\tilde{J}'$ (recall that $\tilde{J}=q-J$), then we must take
care of the following diagrams, which are still divergent because of the
innermost contraction: 
\begin{equation}\label{e:mooseC2A}
  \rsmooserrrrccBKA[rsP]^{\zeta} \;, \qquad \rsmooserrrrccBKB[rsP]^{\zeta} \;, \qquad \rsmooserrrrccBKC[rsP]^{\zeta} \;,
\end{equation}
as well as the following ones, in which we have replaced two kernels:
\begin{equation}\label{e:mooseC2B}
  \rsmooserrrrccBKD[rsP]^{\zeta} \;, \qquad \rsmooserrrrccBKE[rsP]^{\zeta} \;, \qquad \rsmooserrrrccBKF[rsP]^{\zeta} \;,
\end{equation}
and finally the last one, in which we have replaced three kernels:
\begin{equation}\label{e:mooseC2C}
  \rsmooserrrrccBKG[rsP]^{\zeta} \;.
\end{equation}
All the terms in \zcref{e:mooseC2A} can be treated in the same way, so we only
consider the first one. Here we apply \zcref{lem:bphzbd} with $J_1 = q$ to estimate for any
${\tilde{\kappa}} >0$ and some $c>0$ that
\begin{equation*}
  \begin{aligned}
    \bigg\vert \rsmooserrrrccBKA[rsP]^{\zeta} (\y) \bigg\vert & \lesssim_{{\tilde{\kappa}}} \int_{\mS_{2L}^2} |\tilde{J}' (\y - \z)| (|\z - \z'|^{-2- {\tilde{\kappa}}}_{\mf{s}} +1) |q'(\z' - \y)| \ud \z \ud \z'  \\
    & \lesssim \int_{\mS_{2L}^2} e^{-c |\y - \z|_{\mf{s}}^{2}} (|\z - \z'|^{-2- {\tilde{\kappa}}}_{\mf{s}} +1) (e^{-|\z'- \y|_{\mf{s}}^2} + |\z' - \y|_{\mf{s}}^{-2} 1_{\mathcal{U}_c(\y)}(\z')) \ud \z \ud \z' \\
    & \lesssim \int_{\mS_{2L}} e^{-c |\y - \z|_{\mf{s}}^{2}} \ud \z \lesssim 1 ,
  \end{aligned}
\end{equation*}
where we have used \zcref{lem:one-int-bound}, together with the estimate
$|\tilde{J}'(\x)| \lesssim e^{-c|\x|_{\mf{s}}^2}$ due to the spectral gap of $q$
on the torus.

The terms in \zcref{e:mooseC2B,e:mooseC2C} are even simpler. For example, the first term in
\zcref{e:mooseC2B} can be bounded by applying \zcref{lem:bphzbd} with $J_1 = \tilde{J}'$:
\begin{equation*}
  \rsmooserrrrccBKD[rsP]^{\zeta} (\y) \lesssim_{{\tilde{\kappa}}}\int_{\mS_{2L}^2} |\tilde{J}' (\y - \z)| (|\z - \z'|^{-2- {\tilde{\kappa}}}_{\mf{s}} +1) |q'(\z' - \y)| \ud \z \ud \z' \lesssim 1 ,
\end{equation*}
All of the other terms in \zcref{e:mooseC2B} can be bounded with the same
argument. 
Finally, the single term in \zcref{e:mooseC2C} is estimated via \zcref{lem:bphzbd}
with $J_1 = \tilde{J}$ as follows:
\begin{equation*}
  \begin{aligned}
  \bigg\lvert \rsmooserrrrccBKG[rsP]^{\zeta} (\y) \bigg\rvert & \lesssim_{{\tilde{\kappa}}} \int_{\mS_{2L}^2} |\tilde{J}' (\y - \z)| |\tilde{J}'(\z' - \y)| (|\z - \z'|^{-2- {\tilde{\kappa}}}_{\mf{s}} +1)  \ud \z \ud \z' \\
  & \lesssim \int_{\mS_{2L}^2} e^{-c |\y - \z|_{\mf{s}}^2} e^{-c |\z' - \y|_{\mf{s}}^2} (|\z - \z'|^{-2- {\tilde{\kappa}}}_{\mf{s}} +1)  \ud \z \ud \z' \lesssim 1.
  \end{aligned}
\end{equation*}  
This is once more an estimate of the desired order, and it completes the proof
of \zcref{e:moose-aim-2}.

\item\emph{Proof of \zcref{e:newrbc6}.}
  We note that $\rselkrenormrr[rsK]^{
\zeta}(t, 0) = 0$ because $\rsipcherryrenormr[rsK] (t, 0) = 0 $.
Therefore, we obtain
\begin{equation*}
  \begin{aligned}
    \left\vert  \Cov \left( \rselkrenormrr[rsK]^{ \zeta} (t, x), \rselkrenormrr[rsK]^{\zeta} (s, y)  \right) \right\vert & =  \left\vert  \Cov \left( \rselkrenormrr[rsK]^{ \zeta} (t, x) - \rsipcherryrenormr[rsK] (t, 0) \rslollipopr[rsK]^\zeta (t, x), \rselkrenormrr[rsK]^{\zeta} (s, y)  \right) \right\vert \\
    & = \left\vert  \Cov \left( \rselkrenormrrXB[rsK]^{ \zeta, (t, 0)} (t, x), \rselkrenormrr[rsK]^{\zeta} (s, y)  \right) \right\vert ,
  \end{aligned}
\end{equation*}
where with the crossed line we follow the same notation introduced in
\zcref{e:recenter-notation}. Now we
must control the same contractions that appear in
\zcref{e:moose-cov-to-estimate}. However, most of those contractions correspond
to Feynman diagrams that are convergent via \zcref{prop:graph-bound}, so that in
that case our result follows via \zcref{e:Kx-bd} together with the same
explanation as in the proof of \zcref{prop:candelabra}. The same applies to
contractions of the form $\rsmooserrrrcPF$ in which the central two noise
vertices are paired, by additionally using \zcref{lem:bphzbd}. Therefore, the
only issue arises when estimating the term
\begin{equation*}
  \Big\vert \rsmooserrrrcAXB[rsK]^{\zeta, (t, 0)} \Big\vert ( (t, x), (s, y)) .
\end{equation*}
To bound this term, we follow the same proof as in \zcref{prop:elkrrr}. Using notation analogous to that appearing in \zcref{e:another-elk}, we decompose
\begin{equation}
  \Big\vert \rsmooserrrrcAXB[rsK]^{\zeta, (t, 0)} \Big\vert ( (t, x), (s, y)) \leq \Big\vert \rsmooserrrrcAXC[rsK]^{\zeta, (t, 0),(t,x)} \Big\vert ( (t, x), (s, y)) +  \Big\vert \rsclawrrX[rsK]^{\zeta, (t, 0)} \Big\vert (t,x) \Big\vert \rsmooserrrrcAXD[rsK]^{\zeta} \Big\vert ( (t, x), (s, y)),\label{eq:second-term-decompose}
\end{equation}
We bound the first term
similarly to \zcref{e:zzbd}: for any $\delta \in (0,1)$, we estimate
\begin{equation} \label{e:moose-aim-new}
  \Big\vert \rsmooserrrrcAXC[rsK]^{\zeta, (t,0),(t,x)} \Big\vert ( (t, x), (s, y)) \lesssim_{s, t, \delta} |x|^\delta .
\end{equation}
More precisely, from \zcref{lem:zigzag} and using the notation of \zcref{e:defzigzag}, we obtain
\begin{equation*}
  \begin{aligned}
    \big\vert \rsbphzXX[rsK]^{\zeta, (t, 0), (t,x), J'} ( (t, x), \y_2) \big\vert \lesssim |x|^\delta  & (d_{\mf{s}, \mathscr{S}} ((t, x) , \y_2 )^{-2 - \delta }+ d_{\mf{s}, \mathscr{S}} ((t, 0) , \y_2 )^{-2 - \delta }) ,
  \end{aligned}
\end{equation*}
and combining this estimate with
\zcref{lem:bphzbd,lem:one-int-bound,prop:lollibd} leads to \zcref{e:moose-aim-new}. For the
second term in \zcref{eq:second-term-decompose}, we follow the proofs of \zcref{eq:clawrrX,eq:penelktimate} in \zcref{prop:elkrrr} to
obtain
\begin{equation*}
  \Big\vert \rsclawrrX[rsK]^{\zeta, (t, 0)} \Big\vert (t, x) \Big\vert \rsmooserrrrcAXD[rsK]^{\zeta} \Big\vert ( (t, x), (s, y)) \lesssim_{s, t, \delta} |x|^\delta \log(2 + 1/|x|) ,
\end{equation*}
This completes the proof of \zcref{e:newrbc6} and
thus of the entire proposition. \qedhere
\end{thmstepnv}
\end{proof}

\subsection{Renormalization estimates}
\label{s.renormalizationblabla}

In this section we obtain a number of estimates that concern the tree $\rsclawrr$, whose contraction $\rsclawrrcA$ is logarithmically divergent by power counting, but   in our setting it turns out to be uniformly bounded.

\begin{lem}
    \label{lem:claw-symm-bd}We have
    \[
    \left|\rsclawrrcA[rsK]^{\zeta}(\mathbf{x})\right|\lesssim1
    \]
    uniformly over all $\mathbf{x}\in\mS_{2L}$ and $\zeta\in(0,1)$.
    \end{lem}

    \begin{rem}\label{rem:claw-exp-isnt-smooth}
     In contrast to the periodic situation considered in \cite{hairer:2013:solving,hairer:quastel:2018:class}, in which case this expectation is just zero by antisymmetry, in \zcref{lem:claw-symm-bd} the expectation is not even equicontinuous as $\zeta\to 0$.
    \end{rem}
    
    \begin{proof}
      We start by %
      explicitly computing $\mathbb{E}\widehat{\rsclawrr[rsP]}^{\zeta}_{t}(j)$
      using arguments similar to those in
      \zcref{sec:Explicit-calculations}. Following
      the same notation as in that section,
      we have
      \[
      \widehat{\rsiplollipopr[rsP]}^{\zeta}_{t}(k)=\frac{\pi\ii k}{L}\int^{t}_{-\infty}\e^{-\frac{\pi^{2}k^{2}}{2L^{2}}(t-s)}\widehat{\rslollipopr[rsP]}^{\zeta}_{s}(k)\,\dif s.
      \]
      Therefore, we have
      \begin{align*}
      \mathbb{E}\widehat{\rsclawrr[rsP]}^{\zeta}_{t}(j) & =\frac{\pi\ii}{L}\sum_{k\in\mathbb{Z}}k\int^{t}_{-\infty}\e^{-\frac{\pi^{2}k^{2}}{2L^{2}}(t-s)}\mathbb{E}\left[\widehat{\rslollipopr[rsP]}^{\zeta}_{t}(j-k)\widehat{\rslollipopr[rsP]}^{\zeta}_{s}(k)\right]\,\dif s\\
      \ovset{\zcref{eq:EW-cov-fourier}} & =\frac{\pi\ii}{L}\sum_{k\in\mathbb{Z}\setminus\{0\}}k\int^{t}_{-\infty}\e^{-\frac{\pi^{2}k^{2}}{2L^{2}}(t-s)}\left(\delta_{j}-\delta_{j-2k}\right)\widehat{\Sh}^{\zeta}_{2L}(k)\e^{-\frac{\pi^{2}k^{2}}{2L^{2}}(t-s)}\,\dif s\\
       & =\frac{\pi\ii}{L}\sum_{k\in\mathbb{Z}\setminus\{0\}}\left(\delta_{j}-\delta_{j-2k}\right)\widehat{\Sh}^{\zeta}_{2L}(k)k\int^{t}_{-\infty}\e^{-\frac{\pi^{2}k^{2}}{L^{2}}(t-s)}\,\dif s\\
       & =\frac{L\ii}{\pi}\sum_{k\in\mathbb{Z}\setminus\{0\}}\frac{\delta_{j}-\delta_{j-2k}}{k}\cdot\widehat{\Sh}^{\zeta}_{2L}(k).
      \end{align*}
      This can be rewritten as
      \[
      \mathbb{E}\widehat{\rsclawrr[rsP]}^{\zeta}_{t}(j) = \begin{cases}
        0,&0=j\text{ or }j\in 2\mathbb{Z}+1;\\
        -\frac{L\ii\widehat{\Sh}^{\zeta}_{2L}(n)}{\pi n},&0\neq j=2n\in 2\mathbb{Z} .      \end{cases}
      \]
      This is the Fourier transform of a $\zeta$-smoothed version of a
      ``sawtooth'' function. Therefore we have
      \begin{equation*}
        \Big|\rsclawrrcA[rsP]^{\zeta}( \x)\Big|\lesssim 1 \qquad\text{for all } \x \in \mS_{2L} .
      \end{equation*} 
      Thus we are left with estimating the difference 
      \begin{equation*}
        \Big|\rsclawrrcA[rsP]^{\zeta}( \x) - \rsclawrrcA[rsK]^{\zeta}( \x) \Big| \leq \int_{\mS_{2L}} |\tilde{J}' (\x - \z)| | \rslollirc[rsP]^\zeta (\z, \x)| \ud \z   + \int_{\mS_{2L}} |J' (\x - \z)| | \rslollirc[rsP]^\zeta (\z, \x) - \rslollirc[rsK]^\zeta (\z, \x) | \ud \z ,
      \end{equation*}
      where we defined $\tilde{J}' = q' - J'$, with $q$ the
      periodic heat kernel. In particular, $\tilde{J}'$ is smooth, with decay
      $|\tilde{J}' (\z)| \lesssim e^{-c |\z|_{\mf{s}}^2}$ for some $c>0$ because
      of the spectral gap of $q$ on the torus, and similarly also $(\z, \x)
      \mapsto \rslollirc[rsP]^\zeta (\z, \x) - \rslollirc[rsK]^\zeta (\z, \x)$
      is smooth by the proof of \zcref{prop:lollibd}. Hence, by \zcref{eq:redcherryx1x2bd} of \zcref{prop:lollibd}, for some compact set
      $\Theta \subseteq \mS_{2L}$ we have
      \begin{equation*}
        \Big|\rsclawrrcA[rsP]^{\zeta}( \x) - \rsclawrrcA[rsK]^{\zeta}( \x) \Big| \lesssim \int_{\mS_{2L}} e^{-c |\x - \z|_{\mf{s}}^2}d_{\mf{s}, \mathscr{S}} (\x, \z)^{-1} \ud \z   + \int_{\Theta} |\x - \z|_{\mf{s}}^{-2} \ud \z \lesssim 1.
      \end{equation*}
      This completes the proof.
    \end{proof}

In the next result, we treat a specific tree, which requires a renormalization
of BPHZ type and which was necessary for the study of several
  more complicated trees in previous sections. We define
\begin{equation}
\rsbphz[rsK]^{\zeta,J_{1}}(\x,\y)\coloneqq\int_{\mS_{2L}}J_{1}(\z, \y)J'(\x-\z) \rslollirc[rsK]^\zeta (\x, \z) \,\dif\mathbf{z},\label{eq:bphzdef}
\end{equation}
and similarly
\begin{equation*}
  \rsbphz[rsP]^{\zeta,J_{1}}(\x,\y)\coloneqq\int_{\mS_{2L}}J_{1}(\z, \y)q'(\x-\z) \rslollirc[rsP]^\zeta (\x, \z) \,\dif\mathbf{z} ,
\end{equation*}
where in the latter case the internal edges are replaced by the periodic heat kernel $q$. The curlicue edge in the diagram denotes the kernel $J_1$, which is added to the superscript.
We note that in the diagrams above there are two leaf nodes, one at the top left and
the other at the root of the tree, which are associated to the variables
$\mathbf{y}$ and $\mathbf{x}$, respectively. In \zcref{eq:bphzdef}, there are
two kernels $J_{1}$ and $J'$ involved on the right side, with the top left edge represented by $J_{1}$. We are mostly
interested in the case $ J_{1}  \in \{J, q, \rslollirc[rsK]^\zeta,
\rslollirc[rsP]^\zeta\}$, where we recall the periodic heat
kernel $q = J + \tilde{J}$. However, it will be useful to consider $J_{1}$ as a
general kernel, namely we assume $ (\z, \w) \mapsto J_{1} (\z, \w) $ is a smooth
map on $\mS_{2L}^2 \setminus D$, where $D = \{ (\z, \w) \in \mS_{2L}^2 \st \z
= \w\}$ is the diagonal, satisfying for some $\alpha \in (0, 3)$, $\delta \in (0, 1]$, and (small) $a>0$ that
\begin{subequations} \label{e:assu-J1}%
  \begin{align}
  |J_{1}(\z, \w)| & \lesssim |\z - \w|^{-\alpha}_{\mathfrak{s}} +1 && \text{for all }  \z, \w \in \mS_{2L} ,\text{ and}\label{e:assu-J1-1}\\
  |J_{1}(\z', \w) - J_{1}(\z, \w)|_{\mathfrak{s}} &  \lesssim |\z - \z'|^{\delta}_{\mathfrak{s}}|\z - \w|^{-\alpha - \delta}_{\mathfrak{s}} && \text{for all }  \z, \z', \w\in \mS_{2L} \text{ s.t. } |\z' - \z|_{\mf{s}} \leq a |\z - \w|_{\mf{s}} .\label{e:assu-J1-2}
  \end{align}%
\end{subequations}
We have the following estimates on the functions defined above.
\begin{lem}
\label{lem:bphzbd}Suppose that $ (\z, \w) \mapsto J_{1} (\z, \w) $ is a smooth
map on $\mS_{2L}^2 \setminus D$,  %
and that $J_1$
satisfies \zcref{e:assu-J1} for some $\alpha \in (0, 3), \delta \in (0, 1]$, and
$a >0$. Then, uniformly over $\mathbf{x},\mathbf{y}\in\mS_{2L}$ and $\zeta \in (0,1)$,
we have
\begin{equation}
\Big|\rsbphz[rsK]^{\zeta,J_{1}}(\mathbf{x},\mathbf{y})\Big| + \Big|\rsbphz[rsP]^{\zeta,J_{1}}(\mathbf{x},\mathbf{y})\Big|\lesssim 1+ \begin{cases}
  |\x - \y|_{\mf{s}}^{- \alpha} \log(2 + 1/|\x - \y|_{\mf{s}} ) \qquad & \text{if } \alpha \in (0,2) ; \\
  d_{\mf{s}, \mathscr{S}}(\x , \y)^{- \alpha} \log(2 + 1/d_{\mf{s}, \mathscr{S}}(\x , \y) )  \qquad & \text{if } \alpha \in [2, 3) .
\end{cases} 
\label{eq:bphzbd}
\end{equation}
\end{lem}

\begin{proof}
We start by analysing $\rsbphz[rsK]^{\zeta,J_{1}}$.
For notational convenience, since $J_{1}$ is fixed, we abbreviate
$\rsbphz[rsK]^{\zeta}=\rsbphz[rsK]^{\zeta,J_{1}}$ in the proof. Then define
\begin{equation}\label{e:Fdef}
F(\mathbf{x},\mathbf{y})\coloneqq\rsbphz[rsK]^{\zeta}(\mathbf{x},\mathbf{y})-J_{1}(\mathbf{x},\mathbf{y})\rsclawrrcA[rsK]^{\zeta}(\mathbf{x}),
\end{equation}
so by \zcref{lem:claw-symm-bd} and \zcref{e:assu-J1-1} we have
\[
\left|\rsbphz[rsK]^{\zeta}(\mathbf{x},\mathbf{y})\right|\lesssim|F(\mathbf{x},\mathbf{y})|+|\mathbf{x}-\mathbf{y}|^{-\alpha}_{\mf{s}}.
\]
Hence to prove \zcref{eq:bphzbd} it suffices to prove the same bound
on $F$. To this end, we bound $|F|\leq F^{+}+F^{-}$, with $F^{+}, F^{-}$
defined through the following calculation:
  \begin{align}%
    |F (\mathbf{x},\mathbf{y}) | & = \left\vert \int_{\mS_{2L}}(J_{1}(\mathbf{z}, \mathbf{y})-J_{1}(\mathbf{x}, \mathbf{y}))J'(\mathbf{x}-\mathbf{z}) \rslollirc[rsK]^\zeta (\x, \z) \ud \z \right\vert \notag\\
    & \lesssim \int_{\x + \Theta}|J_{1}(\mathbf{z}, \mathbf{y})-J_{1}(\mathbf{x}, \mathbf{y})|  | \x - \z|_{\mf{s}}^{-3} \ud \z + \int_{\x+\Theta}|J_{1}(\mathbf{z}, \mathbf{y})-J_{1}(\mathbf{x}, \mathbf{y})| | \x - \z|_{\mf{s}}^{-2}| \sr \x - \z|_{\mf{s}}^{-1} \ud \z \notag \\
    & = F^- (\x, \y) + F^+ (\x, \y),\label{eq:Fpmdef}
  \end{align}
  where we used \zcref{prop:lollibd} and the usual estimate \zcref{eq:Kbds} on $J'$. Here, $\Theta$ is a compact set independent of $\x$ and $\y$, which we can find thanks to the compact support property of $J$.

We start by estimating $F^-$, and the idea is to use the regularity of $J_1$ in the first variable. We split the integration domain into two. With $a>0$ as in \zcref{e:assu-J1-2}, we define 
  \begin{equation*}
    A (\x, \y) = \{ \z \in \mS_{2L} \st |\x - \z|_{\mf{s}} \leq  a |\x - \y|_{\mf{s}} \}\cap  (\x+\Theta)\qquad\text{and}\qquad B (\x, \y) = (\x+\Theta) \setminus A(\x, \y) .
  \end{equation*}
  For $\z\in A(\x,\y)$, we have $\left| J_{1}(\z,  \y) - J_{1}(\x,  \y) \right|\lesssim 
      |\x - \z|_{\mf{s}}^\delta |\x - \y|_{\mf{s}}^{-\alpha-\delta} 
      $, which leads to
  \begin{equation*}
    \begin{aligned}
    \int_{A(\x, \y)}\left| J_{1}(\z , \y) - J_{1}(\x , \y) \right| |\x - \z|_{\mf{s}}^{-3} \ud\z & \lesssim |\x - \y|_{\mf{s}}^{-\alpha - \delta}\int_{A(\x, \y)} |\x - \z|_{\mf{s}}^{-3 + \delta} \ud\z  \lesssim |\x - \y|_{\mf{s}}^{-\alpha} .
    \end{aligned}
  \end{equation*}
  On the other hand, we have
  \begin{equation*}
    \left| J_{1}(\z,  \y) - J_{1}(\x,  \y) \right|\lesssim 1+
      |\z - \y|_{\mf{s}}^{-\alpha} + |\x - \y|_{\mf{s}}^{-\alpha}
  \end{equation*}
  for $\z\in B(\x,\y)$, and so
  \begin{equation*}
    \begin{aligned}
    \int_{B(\x, \y)}\left| J_{1}(\z , \y) - J_{1}(\x , \y) \right| |\x - \z|_{\mf{s}}^{-3} \ud\z & \lesssim \int_{B(\x, \y)} (1+|\z - \y|_{\mf{s}}^{-\alpha} + |\x - \y|_{\mf{s}}^{-\alpha}) |\x - \z|_{\mf{s}}^{-3} \ud\z \\
    & \lesssim (1+|\x - \y|_{\mf{s}}^{-\alpha}) (1 + \log(2 +|\x - \y|_{\mf{s}}^{-1})),
    \end{aligned}
  \end{equation*}
  where the last estimate comes from a direct integration. 
  Together, these bounds yield an estimate of the desired order for $F^{-}$. 

  Let us move on to estimating $F^+$. Here we must take some additional
  precautions, since a na\"ive estimate on the second term in the definition of $F^+$ (following the
  same steps as above) would lead to
  \begin{equation}\label{e:naive}
    \begin{aligned}
      \int_{\x+\Theta} & |J_{1}(\mathbf{x}, \mathbf{y})| | \x - \z|_{\mf{s}}^{-2}| \sr \x - \z|_{\mf{s}}^{-1} \ud \z  & \lesssim (1+|\x - \y|_{\mf{s}}^{-
      \alpha } )\int_{\x+\Theta} |\x - \z|_{\mf{s}}^{-2} |\sr \x - \z|_{\mf{s}}^{-1}  \ud \z \\
      & \lesssim (1+|\x - \y|_{\mf{s}}^{-\alpha})  \log( 2 +1/|x|) . 
    \end{aligned}
  \end{equation}
  Here we recall that $\x=(t,x)$. 
  To obtain \zcref{eq:bphzbd}, we cannot tolerate any blow-up as $x\to 0$, so we must improve this estimate. Of course, this is the same issue faced in the proof of
  \zcref{lem:claw-symm-bd}, only now in a slightly different setting because of the reflection. For simplicity and since the problem is time-homogeneous, we assume that $t=0$, i.e.\ $\x=(0,x)$, and that $|x|\le L/2$ so $d_L(\x)=|x|$. We consider two different cases depending on the relative magnitudes of $|x|=|\x|_{\mathfrak{s}}$ and $|\x-\y|_{\mathfrak{s}}$. %

  \emph{The case $|\x - \y|_{\mf{s}} < \oh |x|$.} In this case we use \zcref{e:naive} to obtain
  \begin{equation*}
    \begin{aligned}
      \int_{\x+\Theta} |J_{1}(\mathbf{x}, \mathbf{y})| | \x - \z|_{\mf{s}}^{-2}| \sr \x - \z|_{\mf{s}}^{-1} \ud \z\lesssim (1+|\x - \y|_{\mf{s}}^{-\alpha})  \log( 2 +|\x - \y|^{-1}_{\mf{s}}) ,
    \end{aligned}
  \end{equation*}
so that we must only estimate  
\begin{align} 
  \int_{\x+\Theta} & |J_{1}(\mathbf{z}, \mathbf{y})| | \x - \z|_{\mf{s}}^{-2}| \sr \x - \z|_{\mf{s}}^{-1} \ud \z \lesssim \int_{\x+\Theta} (1+|\y - \z|_{\mf{s}}^{-\alpha} ) |\x - \z|_{\mf{s}}^{-2}   |\sr \x - \z|_{\mf{s}}^{-1}   \ud\z . \label{e:to-bd-bphz}
\end{align}
The integral involving the constant $1$ is easy to deal with, so we focus on the other one. To estimate this integral we slice $\Theta$ into three regions $A, B, C$. Define 
\begin{equation*}
    \begin{aligned}
    A=A(\x, \y) & \coloneqq \{ \z \in \x+\Theta \st |\x - \z|_{\mf{s}}  < |x| \} , \\
    B=B(\x, \y) & \coloneqq \{ \z \in \x+\Theta \st |x| \leq |\x - \z|_{\mf{s}}  < 3|x| \} , \\
    C=C(\x, \y) & \coloneqq \{ \z \in \x+\Theta \st 3|x| \leq |\x - \z|_{\mf{s}}  \} .
    \end{aligned}
\end{equation*}
By definition we have $A \cup B \cup C =\x+  \Theta$. %
Furthermore, we have $|\sr \x - \z|_{\mf{s}}\ge |\sr\x-\x|_{\mf s}-|\x-\z|_{\mf s} =2|x|-|\x-\z|_{\mf s} \geq |x|$ for all $\z \in A$. Therefore, on $A$ we can estimate the integral  by 
\begin{equation*}
  \begin{aligned}
    \int_A |\y - \z|_{\mf{s}}^{-\alpha}  |\x - \z|_{\mf{s}}^{-2}   |\sr \x - \z|_{\mf{s}}^{-1}  \, \dif \z & \lesssim |x|^{-1} \int_A |\y - \z|_{\mf{s}}^{-\alpha}  |\x - \z|_{\mf{s}}^{-2}  \dif \z \\
    & \lesssim |x|^{-1} \begin{cases}
    |x|^{1 - \alpha} \qquad & \text{ if } \alpha \in (0, 1); \\
    \log(2 + 1/|\x-\y|_{\mf{s}}) \qquad & \text{ if } \alpha = 1; \\
    |\x-\y|_{\mf{s}}^{-\alpha + 1} \qquad & \text{ if } \alpha \in (1, 3)
    \end{cases} \\
    &  \lesssim %
      |\x-\y|_{\mf{s}}^{-\alpha}\log(2 + 1/|\x-\y|_{\mf{s}}), %
  \end{aligned}
\end{equation*}
as desired. On the set $B$ we have that $|\y - \z|_{\mf{s}} \geq |\x - \z|_{\mf{s}}
- |\y - \x|_{\mf{s}} \ge \oh |x|$, and so
\begin{equation*}
    \int_B |\y - \z|_{\mf{s}}^{-\alpha}  |\x - \z|_{\mf{s}}^{-2}   |\sr \x - \z|_{\mf{s}}^{-1}  \, \dif \z \lesssim |x|^{-2-\alpha}\int_B |\sr \x - \z|_{\mf{s}}^{-1}  \, \dif \z \lesssim %
|x|^{-2 - \alpha} |x|^2
    \lesssim |\x - \y|_{\mf{s}}^{-\alpha}.
\end{equation*}
as desired. Finally, on the set $C$, we have
\begin{equation*}
    \int_C |\y - \z|_{\mf{s}}^{-\alpha}  |\x - \z|_{\mf{s}}^{-2}   |\sr \x - \z|_{\mf{s}}^{-1} \,  \dif \z \lesssim  \int_C |\x - \z|_{\mf{s}}^{-3 -\alpha}  \, \dif \z \lesssim |x|^{-\alpha } \lesssim |\x - \y|_{\mf{s}}^{-\alpha } ,
\end{equation*}
which concludes the proof in the region  $|\x - \y|_{\mf{s}} < \oh |x|$.

  \emph{The case $|\x - \y|_{\mf{s}} \geq \oh|x|$.}
  In this case we again split the integral over $\z$ into three parts. Recalling that the parameter $a$ was fixed so that \zcref{e:assu-J1-2} is satisfied,  we define %
  \begin{subequations}
  \begin{align}
    A&\coloneqq\{\z\in \x+ \Theta\st |\z|_{\mf s} \le 4 |\x-\y|_{\mf s}\text{ and } |\z-\x|_{\mf s}\le a |\x-\y|_{\mf s}\};\label{eq:A3}\\
    B&\coloneqq\{\z\in \x+\Theta\st |\z|_{\mf s} \le 4 |\x-\y|_{\mf s}\text{ and } |\z-\x|_{\mf s}> a |\x-\y|_{\mf s}\};\label{eq:B3}\\
    C&\coloneqq\{\z\in \x+\Theta\st |\z|_{\mf s} > 4 |\x-\y|_{\mf s}%
    \}.\label{eq:C3}
  \end{align}%
    \label{eq:ABC3}
\end{subequations}
For the  integral over $A$ we estimate using \zcref{e:assu-J1-2} that
\begin{align*}
  \int_A &|J_1(\z,\y)-J_1(\x,\y)||\x-\z|_{\mf s}^{-2}|\sigma_{\mathrm{refl}}\x-\z|^{-1}_{\mf s}\,\dif \z\lesssim |\x-\y|^{-\alpha-\delta}_{\mf{s}}\int_{|\z|_{\mf s}\le 4|\x-\y|_{\mf s}} |\x-\z|^{-2+\delta}_{\mathfrak s} |\sigma_{\mathrm{refl}}\x-\z|^{-1}_{\mathfrak s}\,\dif \z
  \\&\lesssim |\x-\y|^{-\alpha-\delta}_{\mf{s}}\int_{|\z|_{\mf s}\le 4|\x-\y|_{\mf s}} \left(|\x-\z|^{-3+\delta}_{\mathfrak s} +|\sigma_{\mathrm{refl}}\x-\z|^{-3+\delta}_{\mathfrak s}\right)\,\dif \z\lesssim |\x-\y|^{-\alpha}_{\mf s},
\end{align*}
where in the last inequality we used the assumption that $|\x-\y|_{\mf s}\ge \oh|x|= \oh d_L(\x)$.

For the integral over $B$, we use \zcref{e:assu-J1,eq:B3} to write
\begin{align}
  \int_B &|J_1(\z,\y)-J_1(\x,\y)||\x-\z|_{\mf s}^{-2}|\sigma_{\mathrm{refl}}\x-\z|^{-1}_{\mf s}\,\dif \z\notag
  \\&\le |\x-\y|_{\mf s}^{-2}\int_B (1+|\z-\y|^{-\alpha}_{\mf s}+|\x-\y|_{\mf s}^{-\alpha})|\sigma_{\mathrm{refl}}\x-\z|_{\mf s}^{-1}\,\dif\z. \label{eq:Bbound}
\end{align}
For the last term, we integrate over $\{\z\st |\z|_{\mf s}\le 4|\x-\y|_{\mf s}\}$ %
to see that
$\int_B |\sigma_{\mathrm{refl}}\x-\z|_{\mf s}^{-1}\,\dif\z \lesssim |\x-\y|_{\mf s}^{2}$, which is a sufficient bound in that case.
For the second term in \zcref{eq:Bbound}, we note that if $\alpha<2$, then by Young's inequality we can write
$|\z-\y|^{-\alpha}_{\mf s} |\sigma_{\mathrm{refl}}\x-\z|_{\mf s}^{-1}\lesssim |\z-\y|^{-\alpha-1}_{\mf s} + |\sigma_{\mathrm{refl}}\x-\z|^{-\alpha-1}_{\mf s}$, so we can proceed to obtain $\int_B |\z-\y|^{-\alpha}_{\mf s}|\sigma_{\mathrm{refl}}\x-\z|^{-1}_{\mf s}\,\dif\z\lesssim |\x-\y|^{2-\alpha}$. On the other hand, if $\alpha\ge 2$, then by \zcref{lem:one-int-bound} we have
\[\int_{\x+\Theta} |\z-\y|^{-\alpha}_{\mf s}|\sigma_{\mathrm{refl}}\x-\z|^{-1}_{\mf s}\,\dif\z
         \lesssim\begin{cases}
           \log(2+1/d_{\mf s,\mathscr{S}}(\x,\y)),&\alpha=2;\\
           d_{\mf s,\mathscr{S}}(\x,\y)^{-\alpha+2},&\alpha\in (2,3).
           \end{cases}
\]
Altogether, we obtain
\begin{equation}
  \int_B |J_1(\z,\y)-J_1(\x,\y)||\x-\z|_{\mf s}^{-2}|\sigma_{\mathrm{refl}}\x-\z|^{-1}_{\mf s}\,\dif \z
         \lesssim\begin{cases}
           |\x-\y|_{\mf s}^{-\alpha},&\alpha \in (0,2);\\
           |\x-\y|_{\mf s}^{-\alpha}\log(2+1/d_{\mf s,\mathscr{S}}(\x,\y)),&\alpha=2;\\
           d_{\mf s,\mathscr{S}}(\x,\y)^{-\alpha},&\alpha\in (2,3).
         \end{cases}\label{eq:Bbound-final}
         \end{equation}

         Finally, for the integral over $C$,
we use the fact that, for $\z\in C$, we have $|\z|_{\mf s}>4|\x-\y|_{\mf s}\ge 2 |\x|_{\mf s}$ and thus
$|\z-\y|_{\mf s}\ge |\z|_{\mf s}-|\x-\y|_{\mf s} -|\x|_{\mf s} \ge \nicefrac16|\z|_{\mf s}$ 
as well as $|\z-\x|_{\mf s}\ge |\z|_{\mf s}-|\x|_{\mf s} \ge \nicefrac12 |\z|_{\mf s}$ to estimate from \zcref{e:assu-J1-1} that 
\begin{align*}
  \int_C &|J_1(\z,\y)-J_1(\x,\y)||\x-\z|_{\mf s}^{-2}|\sigma_{\mathrm{refl}}\x-\z|^{-1}_{\mf s}\,\dif \z\notag\\
         &\lesssim \int_C (1+|\z-\y|^{-\alpha}_{\mf s}+|\x-\y|_{\mf s}^{-\alpha})|\x-\z|^{-2}|\sigma_{\mathrm{refl}}\x-\z|_{\mf s}^{-1}\,\dif\z. %
       \\&\lesssim \int_{\z\in\x+\Theta:|\z|_{\mf s}\ge 4|\x-\y|_{\mf s}} (1+|\z|_{\mf s}^{-3-\alpha}+|\x-\y|^{-\alpha}_{\mf s}|\z|_{\mf s}^{-3})\,\dif\z\lesssim 1+|\x-\y|^{-\alpha}_{\mf s}\log(2+1/|\x-\y|_{\mf s})
\end{align*}
  This completes the proof of the result for the case of $\rsbphz[rsK]^{\zeta,J_1}$.

  \emph{The case $\rsbphz[rsP]^{\zeta, J_1}$.} The only difference when
  considering $\rsbphz[rsP]$ rather than $\rsbphz[rsK]$ is that the kernel $q$ does
  not have compact support. However, the difference $\tilde{J} = q - J$ is
  smooth and, due to the spectral gap of $q'$, decays exponentially fast, i.e.\ $| \tilde{J}' (\x)| \lesssim e^{-c
  |\x|_{\mf{s}}^2} $ for some $c >0$.
  Thus, if we write the difference as
  \begin{equation}
    \begin{aligned}
    \rsbphz[rsP]^{\zeta,J_{1}} - \rsbphz[rsK]^{\zeta,J_{1}} &=  \int_{\mS_{2L}} J_{1}(\mathbf{z}, \mathbf{y}) \tilde{J}'(\mathbf{x}-\mathbf{z}) \rslollirc[rsP]^\zeta (\x, \z)\,\dif\mathbf{z} \\&\qquad+ \int_{\mS_{2L}} J_{1}(\mathbf{z}, \mathbf{y}) J'(\mathbf{x}-\mathbf{z}) \left(\rslollirc[rsP]^\zeta (\x, \z) - \rslollirc[rsK]^\zeta (\x, \z)  \right)\,\dif\mathbf{z}  ,\end{aligned}\label{eq:bphzdiffbd}
  \end{equation}
  then the integrals on the right side are convergent, similarly to the situation %
  in the proofs of \zcref{prop:candelabra, prop:moose}. %
  Indeed, for the first integral we estimate
    \begin{align}
    \bigg\vert \int_{\mS_{2L}}  J_{1}(\mathbf{z}, \mathbf{y}) \tilde{J}'(\mathbf{x}-\mathbf{z}) \rslollirc[rsP]^\zeta (\x, \z)\,\dif\mathbf{z}  \bigg\vert & \lesssim \int_{\mS_{2L}} e^{- c |\x - \z|_{\mf{s}}^2} (|\z - \y|_{\mf{s}}^{-\alpha} +1) d_{\mf{s}, \mathscr{S}} (\x , \z)_{\mf{s}}^{-1} \,\dif\mathbf{z}\notag \\\label{e:replace-1}
    & \lesssim \begin{cases}
      1 \qquad & \text{if } \alpha \in (0, 2) ; \\
      \log(2 + d_{\mf{s}, \mathscr{S}} (\x, \y)) \qquad & \text{if } \alpha =2 ; \\
      1 + d_{\mf{s}, \mathscr{S}} (\x, \y)^{- \alpha + 2} \qquad & \text{if } \alpha \in (2,3) ,
    \end{cases}
    \end{align} 
  where we have employed \zcref{lem:one-int-bound,prop:lollibd}.

  On the other hand, the second integral in \zcref{eq:bphzdiffbd} reduces again to an integral over a compact domain
  $\x +\Theta$ by the compact support property of $J'$. We find, since the kernel $\rslollirc[rsP]^\zeta (\x, \z) - \rslollirc[rsK]^\zeta (\x, \z) $
  is uniformly bounded by \zcref{eq:lollirc-diff-bd}, that
  \begin{align}%
      \left\vert \int_{\mS_{2L}} J_{1}(\mathbf{z}, \mathbf{y}) J'(\mathbf{x}-\mathbf{z}) \left(\rslollirc[rsP]^\zeta (\x, \z) - \rslollirc[rsK]^\zeta (\x, \z)  \right)\,\dif\mathbf{z} \right\vert & \lesssim  \int_{\x +\Theta} (1+|\z - \y|_{\mf{s}}^{-\alpha}) |\x - \z|_{\mf{s}}^{-2} \,\dif\mathbf{z}\notag \\
      & \lesssim 1+\begin{cases}
        1 \qquad & \text{if } \alpha \in (0, 1) ; \\
        \log(2 + |\x - \y|_{\mf{s}}) \qquad & \text{if } \alpha = 1 ; \\
        |\x - \y|_{\mf{s}}^{- \alpha + 1} \qquad & \text{if } \alpha \in (1,3) .
      \end{cases} \label{e:replace-2}
  \end{align}%
  The estimates \zcref{e:replace-1,e:replace-2} are strictly better than the ones in \zcref{eq:bphzbd}, and
  therefore imply \zcref{eq:bphzbd} also for $ \rsbphz[rsP]^{\zeta,J_{1}}$.
  This completes the proof of the result.
\end{proof}

The next result controls the regularity of $\rsbphz[rsK]^{\zeta,J_{1}}$ in its
second variable.

\begin{lem} \label{lem:bphz-increment}
  In the same setting as \zcref{lem:bphzbd}, fix any $\alpha \in (1, 2), \delta \in
  (0, 2 - \alpha)$, and assume that the kernel $J_1$ and its reflection
  $J_1^{r} (\x, \y) = J_{1}(\y, \x)$ both satisfy \zcref{e:assu-J1} with these
  $\alpha, \delta$. Then we have
  \begin{equation}
    \Big|\rsbphz[rsK]^{\zeta,J_{1}}(\mathbf{x},\mathbf{y})- \rsbphz[rsK]^{\zeta,J_{1}}(\mathbf{x},\mathbf{y'})\Big| \lesssim |\x - \y|^{-\alpha - \delta}_{\mf{s}} |\y - \y'|_{\mf{s}}^\delta \log(2 + 1/|\x - \y|_{\mf{s}} +  1/|\y - \y'|_{\mf{s}}) .\label{eq:bphzdiffbd-2}
    \end{equation}
\end{lem}

\begin{proof}
  The proof follows along the same lines of the proof of \zcref{lem:bphzbd}. In
  particular, for any fixed $d <\infty$, if $|\x - \y|_{\mf{s}} \leq d |\y
  - \y'|_{\mf{s}}$, then the upper bound follows directly from
  \zcref{lem:bphzbd} since
  \begin{equation*}
    |\x - \y|_{\mf{s}}^{-\alpha} = |\x - \y|_{\mf{s}}^{-\alpha- \delta }|\x - \y|_{\mf{s}}^\delta \lesssim_{d} |\x - \y|_{\mf{s}}^{-\alpha- \delta } |\y - \y'|_{\mf{s}}^\delta .
  \end{equation*}
  Therefore, let us consider only the case  $|\x - \y|_{\mf{s}} > d |\y
  - \y'|_{\mf{s}}$. The value of $d > 0$ will be allowed to vary throughout the
  proof, and is eventually required to be sufficiently large for all the
  estimates to hold. We assume that $d \geq 2/a$, implying that
  \begin{equation}
  |\y-\y'|_{\mf s} \le \nicefrac{a}2|\x-\y|_{\mf s},\label{eq:key-a-relation}
  \end{equation}
  and hence since $J_1^r$
  satisfies \zcref{e:assu-J1} that \begin{equation}\label{eq:J1diffbd}
    |J_{1}(\mathbf{x},\mathbf{y}) - J_{1}(\mathbf{x},\mathbf{y}')| \lesssim |\x - \y|_{\mf{s}}^{-\alpha - \delta} |\y - \y'|_{\mf{s}}^\delta.
  \end{equation}
  Together with \zcref{lem:claw-symm-bd}, this implies that if we define
\begin{equation*}
G(\mathbf{x},\mathbf{y})\coloneqq\rsbphz[rsK]^{\zeta, J_1}(\mathbf{x},\mathbf{y})-\rsbphz[rsK]^{\zeta, J_1}(\mathbf{x},\mathbf{y}')-J_{1}(\mathbf{x},\mathbf{y})\rsclawrrcA[rsK]^{\zeta}(\mathbf{x}) + J_{1}(\mathbf{x},\mathbf{y}')\rsclawrrcA[rsK]^{\zeta}(\mathbf{x}) ,
\end{equation*}
then we have
\[
\left|\rsbphz[rsK]^{\zeta, J_1}(\mathbf{x},\mathbf{y})-\rsbphz[rsK]^{\zeta, J_1}(\mathbf{x},\mathbf{y}')\right|\lesssim|G(\mathbf{x},\mathbf{y})|+|\mathbf{x}-\mathbf{y}|^{-\alpha - \delta }_{\mf{s}} |\y - \y'|_{\mf{s}}^\delta.
\]
Hence to prove \zcref{eq:bphzdiffbd-2} it suffices to prove the same bound
on $G$. We fix $\y,\y'$   in the proof, and to simplify the notation we define the recentered kernel
\begin{equation*}
  M_{1} (\w, \w') = J_{1}(\w, \w') - J_{1}(\w, \w'+ \y'- \y).
\end{equation*}
Note that  $M_1(\w,\y) = J_1(\w,\y)-J_1(\w,\y')$ for any $\w\in\mS_{2L}$. %
Let us proceed as in the proof of \zcref{lem:bphzbd}, following the same notational conventions. %
We can estimate using \zcref{eq:Kbds,eq:redcherryx1x2bd}  that
\begin{equation*}
  \begin{aligned}
    |G (\mathbf{x},\mathbf{y}) | & = \left\vert \int_{\mS_{2L}}(M_{1}(\mathbf{z}, \mathbf{y})-M_{1}(\mathbf{x}, \mathbf{y}))J'(\mathbf{x}-\mathbf{z}) \rslollirc[rsK]^\zeta (\x, \z) \ud \z \right\vert \\
    & \lesssim \int_{\x + \Theta}|M_{1}(\mathbf{z}, \mathbf{y})-M_{1}(\mathbf{x}, \mathbf{y})|  | \x - \z|_{\mf{s}}^{-3} \ud \z + \int_{\x+ \Theta}|M_{1}(\mathbf{z}, \mathbf{y})-M_{1}(\mathbf{x}, \mathbf{y})| | \x - \z|_{\mf{s}}^{-2}| \sr \x - \z|_{\mf{s}}^{-1} \ud \z  \\
    & \eqqcolon G^- (\x, \y) + G^+ (\x, \y).
  \end{aligned}
\end{equation*}

We start by estimating $G^-$. We define 
  \begin{equation*}
    A (\x, \y, \y') \coloneqq \{ \z \in \mS_{2L} \st |\x - \z|_{\mf{s}} \leq |\y - \y'|_{\mf{s}} \} \qquad\text{and}\qquad B (\x, \y, \y') \coloneqq (\x+\Theta) \setminus A(\x, \y, \y') .
  \end{equation*}
For $\z \in A (\x, \y, \y')$ we have $|\x-\z|_{\mf s} \le |\y-\y'|_{\mf s}\le \nicefrac a 2|\x-\y|_{\mf s}$ by \zcref{eq:key-a-relation}. Using this and then also assuming that $d\ge 2$ to ensure that $|\x-\y'|_{\mf s} \ge |\x-\y|_{\mf s}-|\y-\y'|_{\mf s} \ge \oh |\x-\y|_{\mf s}$, we get $|\x-\z|_{\mf s} \le a|\x-\y'|_{\mf s}$ as well. Therefore, we can use \zcref{e:assu-J1-2} to estimate that
\begin{equation*}
  \left| M_{1}(\z , \y) - M_{1}(\x , \y) \right|  \leq |J_1 (\x, \y) - J_1(\z, \y)| + |J_1(\x, \y') - J_1 (\z, \y')| \lesssim |\x - \y|^{-\alpha -\delta}_{\mf{s}} |\x - \z|_{\mf{s}}^\delta .
\end{equation*}
Thus for the integral over $A(\x, \y, \y')$ we find that
  \begin{equation*}
    \begin{aligned}
    \int_{A(\x, \y, \y')}\left| M_{1}(\z , \y) - M_{1}(\x , \y) \right| |\x - \z|_{\mf{s}}^{-3} \ud\z & \lesssim |\x - \y|_{\mf{s}}^{-\alpha - \delta}\int_{A(\x, \y, \y')} |\x - \z|_{\mf{s}}^{-3 + \delta} \ud\z  \lesssim |\x - \y|_{\mf{s}}^{-\alpha - \delta} |\y - \y'|_{\mf{s}}^{\delta} .
    \end{aligned}
  \end{equation*}
Next, for the integral on $B(\x,\y,\y')$, by \zcref{e:assu-J1-2} applied to $J_1^r(\z, \y) = J_1(\y, \z)$, we have
\begin{equation*}
  \begin{aligned}
    |M_1(\x, \y)| & = |J_1(\x, \y) - J_1(\x, \y')| \lesssim |\x - \y|^{-\alpha- \delta} |\y - \y'|^\delta,
  \end{aligned}
\end{equation*}
so we can estimate
\begin{equation*}
  \begin{aligned}
    \int_{B(\x, \y, \y')}\left| M_{1}(\x , \y) \right| |\x - \z|_{\mf{s}}^{-3} \ud\z & \lesssim |\x - \y|^{-\alpha- \delta}_{\mf{s}} |\y - \y'|^\delta_{\mf{s}} \int_{B(\x, \y, \y')}|\x - \z|_{\mf{s}}^{-3} \ud \z \\
    & \lesssim |\x - \y|^{-\alpha- \delta}_{\mf{s}} |\y - \y'|^\delta_{\mf{s}} \log(2 + 1/|\y - \y'|_{\mf{s}}) .
  \end{aligned}
\end{equation*}
To estimate the term involving $|M_1(\z,\y)|$ on $B(\x,\y,\y')$, we further divide $B(\x,\y,\y')$ into two sets. Define $c\coloneqq2\vee\nicefrac1a$. If $\z\in B(\x,\y,\y')\cap \mathcal{U}_{c|\y-\y'|_{\mf s}}(\y)$, then we have $|\x-\z|_{\mf s}\ge |\x-\y|_{\mf s}-|\y-\z|_{\mf s} \ge \oh|\x-\y|_{\mf s}$. We can then use \zcref{e:assu-J1-1}  along with this inequality to estimate
\begin{align*}
\int_{B(\x,\y,\y')\cap\mathcal{U}_{c|\y-\y'|_{\mf s}}(\y)} |M_1(\z,\y)||\x-\z|^{-3}_{\mf s} \,\dif \z %
&\lesssim |\x-\y|^{-3}_{\mf s}\int_{B(\x,\y,\y')\cap\mathcal{U}_{c|\y-\y'|_{\mf s}}(\y)}(|\z-\y|_{\mf s}^{-\alpha} + |\z-\y'|^{-\alpha}_{\mf s})\,\dif \z\\
&\lesssim |\x-\y|^{-3}_{\mf s}|\y-\y'|^{3-\alpha}_{\mf s}\lesssim |\x-\y|^{-\alpha-\delta}_{\mf s}|\y-\y'|^{\delta}_{\mf s},
\end{align*}
where in the last estimate we used that $|x-y|_{\mf s}\gtrsim |\y-\y'|_{\mf s}$.
On the other hand, we can estimate, using the fact that for $z\not \in \mathcal{U}_{c|\y-\y'|_{\mf s}}$ we have $|\y-\y'|\le \nicefrac 1 c |\z-\z|_{\mf s} \le a|\z-\y|_{\mf s}$, that
\begin{equation*}
  \begin{aligned}
    \int_{B(\x, \y, \y')}\left| M_{1}(\z , \y) \right| |\x - \z|_{\mf{s}}^{-3} \ud\z &\lesssim %
   |\y - \y'|^\delta_{\mf{s}} \int_{B(\x, \y, \y')  \setminus \mathcal{U}_{c|\y - \y'|_{\mf{s}}} (\y)} |\z - \y|_{\mf{s}}^{-\alpha -\delta}  |\x - \z|_{\mf{s}}^{-3} \ud \z\\&\lesssim |\x-\y|_{\mf s}^{-\alpha-\delta}|\y-\y'|_{\mf s}^{\delta}\log(2+1/|\y-\y'|_{\mf s}).
  \end{aligned}
\end{equation*}
The last estimate is obtained by further breaking up the integration domain into three sets: the first two when $\z$ is within a ball of radius proportional to $|\x-\y|_{\mf s}$ of $\x$ and of $\y$, respectively, and then the last being the remaining set.

  Let us move on to estimating $G^+$. We fix constants $c_1,c_2,c_3,c_4$ and define the sets
  \begin{equation*}
    A \coloneqq\mathcal{U}_{c_1 |\x - \y|_{\mf{s}}} (\x) , \qquad B \coloneqq \mathcal{U}_{c_2 |\x - \y|_{\mf{s}}} (\y) , \qquad C \coloneqq \mathcal{U}_{c_3 |\y - \y'|_{\mf{s}}} (\y) , \qquad D \coloneqq \mathcal{U}_{c_4 |\y - \y'|_{\mf{s}}} (\x) .
  \end{equation*}
  We assume that $c_1,c_2\le \nicefrac 13$, so $A\cap B=\varnothing$. We also assume that $c_4 \le dc_1$ and $c_3\le dc_2$, so that $C\subseteq B$ and $D\subseteq A$. Finally, we assume that $c_4 \le da$ and that $c_3 \ge 1/a$. 
Furthermore, set $E = (\x + \Theta) \setminus (A
  \cup B)$.
We estimate the integral on each set. 

  For $\z\in D$, using \zcref{e:assu-J1-2} and the assumption that  $c_4<da$, we can estimate
  \begin{equation*}
    |M_1(\z, \y) - M_1(\x, \y)| = |J_1(\z,\y)-J_1(\x,\y)| + |J_1(\x,\y')-J_1(\x,\y')| \lesssim |\x - \z|^{\delta}_{\mf{s}} | \y - \x |^{-\alpha - \delta}_{\mf{s}} 
  \end{equation*}
  to obtain
  \begin{equation*}
    \begin{aligned}
      \int_D |M_1(\z, \y) - M_1(\x, \y)| | \x - \z|_{\mf{s}}^{-2}| \sr \x - \z|_{\mf{s}}^{-1} \ud \z & \lesssim  | \y - \x |^{-\alpha - \delta}_{\mf{s}} \int_D |\x - \z|^{\delta}_{\mf{s}} | \x - \z|_{\mf{s}}^{-2}| \sr \x - \z|_{\mf{s}}^{-1}  \ud \z \\
      & \lesssim | \y - \x |^{-\alpha - \delta}_{\mf{s}}  |\y - \y'|_{\mf{s}}^{\delta},
    \end{aligned}
  \end{equation*}
  where for the last inequality we simply used Young's inequality to write $|\x-\z|^{-2+\delta}_{\mf s}|\sr\x-\z|^{-1}_{\mf s} \lesssim |\x-\z|^{-3+\delta}_{\mf s}|+|\sr\x-\z|^{-3+\delta}$ and then noted that integral of the second term is smaller than the integral of the first since $D$ is centered around $\x$.
  
  In $A\setminus D$, we use that $|\z - \y|_{\mf{s}} \geq |\y-\x|_{\mf{s}}-|\x-\z|_{\mf{s}}\geq (1-c_1) |\x - \y|_{\mf{s}}$
  to estimate
  \begin{equation*}
  \begin{aligned}
    |M_1(\z, \y) - M_1(\x, \y)|&\le |J_1(\z,\y)-J_1(\z,\y')| + |J_1(\x,\y)-J_1(\x,\y')| \\&\lesssim |\y - \y'|^{\delta}_{\mf{s}} | \y - \x |^{-\alpha - \delta}_{\mf{s}} + |\y - \y'|^{\delta}_{\mf{s}} | \y - \z |^{-\alpha - \delta}_{\mf{s}} \lesssim |\y - \y'|^{\delta}_{\mf{s}} | \y - \x |^{-\alpha - \delta}_{\mf{s}},
    \end{aligned}
  \end{equation*}
  and so we obtain
  \begin{equation*}
    \begin{aligned}
      \int_{A\setminus D} |M_1(\z, \y) - M_1(\x, \y)| | \x - \z|_{\mf{s}}^{-2}| \sr \x - \z|_{\mf{s}}^{-1} \ud \z & \lesssim |\y - \y'|^{\delta}_{\mf{s}} | \y - \x |^{-\alpha - \delta}_{\mf{s}} \int_{A\setminus D} | \x - \z|_{\mf{s}}^{-2}| \sr \x - \z|_{\mf{s}}^{-1} \ud \z \\
      & \lesssim |\y - \y'|^{\delta}_{\mf{s}} | \y - \x |^{-\alpha - \delta}_{\mf{s}} \log(2 + |\y - \y'|_{\mf{s}}^{-1}).
    \end{aligned}
  \end{equation*}
  To obtain the last estimate we have distinguished the cases when $|x|\lesssim |\y-\y'|_{\mf s}$, in which case $\sr \x$ is uniformly in the interior of $D$ and so we can obtain the bound from the estimate $|\x-\z|^{-2}_{\mf s}|\sr \x-\z|^{-1}_{\mf s} \lesssim |\x-\z|^{-3}_{\mf s}+|\sr \x-\z|^{-3}_{\mf s}$, and the complementary case, for which we use \zcref{lem:one-int-bound}.

  Next, on the set $C$ we use the estimate 
  \begin{equation*}
    |M_1(\z, \y) - M_1(\x, \y)| \le |J_1(\z,\y)|+|J_1(\z,\y')+|J_1(\x,\y)-J_1(\x,\y')| \lesssim |\z - \y|_{\mf{s}}^{-\alpha} + |\z - \y'|_{\mf{s}}^{-\alpha} + |\x - \y|^{-\alpha- \delta} |\y - \y'|^\delta .
  \end{equation*} 
  We then use that $|\z - \x|_{\mf{s}} \gtrsim |\y - \x|_{\mf{s}}$ for $\z \in
  C$ to  obtain
  \begin{equation*}
    \begin{aligned}
      \int_{C} |M_1(\z, \y) - M_1(\x, \y)| | \x - \z|_{\mf{s}}^{-2}| \sr \x - \z|_{\mf{s}}^{-1} \ud \z & \lesssim \int_{C} (|\z - \y|_{\mf{s}}^{-\alpha} + |\z - \y'|_{\mf{s}}^{-\alpha} ) | \x - \z|_{\mf{s}}^{-2}| \sr \x - \z|_{\mf{s}}^{-1} \ud \z \\
      & \quad + |\x - \y|^{-\alpha- \delta} |\y - \y'|^\delta \int_{C}| \x - \z|_{\mf{s}}^{-2}| \sr \x - \z|_{\mf{s}}^{-1} \ud \z \\
      & \lesssim | \x - \y|_{\mf{s}}^{-2} \int_{C} (|\z - \y|_{\mf{s}}^{-\alpha} + |\z - \y'|_{\mf{s}}^{-\alpha} ) | \sr \x - \z|_{\mf{s}}^{-1} \ud \z \\
      & \quad + |\x - \y|^{-\alpha- \delta - 2} |\y - \y'|^{\delta+2} \\
      & \lesssim | \x - \y|_{\mf{s}}^{-2} |\y - \y'|_{\mf{s}}^{2 -\alpha } + |\x - \y|^{-\alpha- \delta } |\y - \y'|^{\delta}  \\
      & \lesssim | \x - \y|_{\mf{s}}^{-\alpha - \delta} |\y - \y'|_{\mf{s}}^{\delta},
    \end{aligned}
  \end{equation*}
  where we have used that $\alpha < 2$ and $\alpha+\delta<2$.   
  For $ \z \in B \setminus C$ we use \zcref{e:assu-J1} and the assumption that $c_3 \ge 1/a$ to estimate
  \begin{equation*}
    |M_1(\z, \y) - M_1(\x, \y)| \lesssim |\y - \y'|_{\mf{s}}^{\delta} (|\z - \y|_{\mf{s}}^{-\alpha - \delta}+ |\x - \y|_{\mf{s}}^{-\alpha - \delta} )\lesssim |\y - \y'|_{\mf{s}}^{\delta} |\z - \y|_{\mf{s}}^{-\alpha - \delta} .
  \end{equation*}
 Then we bound
  \begin{equation*}
    \begin{aligned}
      \int_{B \setminus C} |M_1(\z, \y) - M_1(\x, \y)| | \x - \z|_{\mf{s}}^{-2}| \sr \x - \z|_{\mf{s}}^{-1} \ud \z & \lesssim |\x - \y|_{\mf{s}}^{-2} |\y - \y'|_{\mf{s}}^\delta \int_{B \setminus C} |\z - \y|_{\mf{s}}^{-\alpha - \delta} | \sr \x - \z|_{\mf{s}}^{-1} \ud \z \\
      & \lesssim|\x - \y|_{\mf{s}}^{-2} |\y - \y'|_{\mf{s}}^\delta |\x - \y|_{\mf{s}}^{2 - \alpha - \delta } \\
      & \lesssim |\y - \y'|_{\mf{s}}^\delta |\x - \y|_{\mf{s}}^{- \alpha - \delta }  ,
    \end{aligned}
  \end{equation*}
  where we have used that $\delta + \alpha < 2$.

  Finally, on the set $E$, we use the estimate  
  \begin{equation*}
    |M_1(\z, \y) - M_1(\x, \y)| \le |M_1(\z,\y)|+M_1(\x,\y)| \lesssim |\y - \y'|_{\mf{s}}^{\delta} (|\z - \y|_{\mf{s}}^{-\alpha - \delta} +|\x-\y|_{\mf s}^{-\alpha-\delta})\lesssim |\y - \y'|_{\mf{s}}^{\delta} |\x-\y|_{\mf s}^{-\alpha-\delta}
  \end{equation*}
  to obtain
  \begin{equation*}
    \begin{aligned}
      \int_{E} |M_1(\z, \y) - M_1(\x, \y)| | \x - \z|_{\mf{s}}^{-2}| \sr \x - \z|_{\mf{s}}^{-1} \ud \z & \lesssim |\y - \y'|_{\mf{s}}^\delta |\x-\y|^{-\alpha-\delta}_{\mf s}\int_{E} |\z - \x|_{\mf{s}}^{- 2}|\sr \x-\z|^{-1} \ud \z %
    \end{aligned}
  \end{equation*}
  We now consider two cases: if $|x|\le \nicefrac {c_1}2|\x-\y|_{\mf s}$, then we use Young's inequality to bound the last integral by $\log (2+|\x-\y|_{\mf s}^{-1})$, while if $|x|\ge \nicefrac{c_1}2|\x-\y|_{\mf s}$, then we use \zcref{lem:one-int-bound} to bound the last integral by $\log (2+|\x-\sr\x|_{\mf s}^{-1})\le \log (2+|\x-\y|_{\mf s}^{-1})$. Therefore, in either case we bound the integral over $E$ by $|\y-\y'|^\delta_{\mf s}|\x-\y|^{-\alpha-\delta}_{\mf s}\log(2+|\x-\y|^{-1}_{\mf s})$.
  This concludes the proof of the desired result.
\end{proof}
Finally, we consider the renormalization of $\rsbphz[rsK]^{\zeta,J_{1}}$ with recenterings. %
We define
\begin{equation}\label{e:new-zigzagdef}
  \rsbphzXX[rsK]^{\zeta, \w,\mathbf{u}, J_1} ( \x, \y) \coloneqq \int_{\mS_{2L}} (J_1 (\z , \y) - J_1(\mathbf{u} , \y)) J'_{\w-\x} (\x - \z)   \rslollirc[rsK]^\zeta (\x, \z) \ud \z ,
\end{equation}
where we recall that $J'_{\w-\x} (\x - \z)=J'(\x-\z)-J'(\w-\z)$.
Then we obtain the following estimate, the proof of which follows along the same
lines as the proofs of \zcref{lem:bphzbd,lem:bphz-increment}.
\begin{lem}\label{lem:zigzag}
  Consider the same setting of \zcref{lem:bphzbd}, and in particular let $J_1$
  satisfy \zcref{e:assu-J1} for some $\alpha \in (1,3)$, $\delta \in (0,1]$, and
  $a> 0$. We have for any $0 < \delta_1 < \delta$ and locally uniformly over $\x, \y \in
  \mS_{2L}$ and $\w \in \{ (t,0), (t, L)\}$, where $\x = (t, x)$, that
  \begin{equation*}
    \Big\vert \rsbphzXX[rsK]^{\zeta, \w, \x,J_1} ( \x, \y) \Big\vert \lesssim |\w - \x|_{\mf{s}}^{\delta_1} ( d_{\mf{s}, \mathscr{S}}(\x , \y)^{-\alpha - \delta_1} + d_{\mf{s}, \mathscr{S}}(\w , \y)^{-\alpha - \delta_1}).
  \end{equation*}
\end{lem}
The restriction of choosing the recentering points $\w \in \{(t,0), (t, L)\}$ and $\mathbf{u} = \x$, which are tied to  the
base point $\x = (t,x)$ in
\zcref{e:new-zigzagdef}, is purely out of convenience, since it allows us to
shorten the proof and is the only case in which we need this estimate.

\begin{proof}
  We follow similar steps to the proofs of
  \zcref{lem:bphzbd,lem:bphz-increment}. Let us fix $\w = (t, 0)$,
  where $\x = (t,x)$. (The case $\w = (t, L)$ follows from identical arguments.) As in the proof of \zcref{lem:bphz-increment}, it suffices to consider
  the case $|\w - \x|_{\mf{s}} = |x| \leq \nicefrac1d |\x - \y|_{\mf{s}}$ for a constant
  $d > 0$ which will be assumed large enough  throughout the proof.
  Otherwise, the estimate is a consequence of
  \zcref{lem:bphzbd}, since in that case we estimate
  $d_{\mf{s}, \mathscr{S}}(\x, \y) \lesssim |x|$ and similarly $d_{\mf{s},
  \mathscr{S}}(\w, \y) \leqslant d_{\mf{s}, \mathscr{S}}(\x, \y) + d_{\mf{s}, \mathscr{S}}(\x, \w) \lesssim |x|$. 
  We start by estimating
\begin{equation*}
  \begin{aligned}
    \Big\vert \int_{\mS_{2L}}&J'_{\w - \x}(\mathbf{x}-\mathbf{z})(J_{1}  (\mathbf{z}, \mathbf{y})-J_{1} (\mathbf{x}, \mathbf{y})) \rslollirc[rsK]^\zeta (\x, \z) \ud \z \Big\vert \\&
    \lesssim   |\w -\x|^{\delta_1}_{\mf{s}} \int_{\x + \Theta}|J_{1}(\mathbf{z}, \mathbf{y})-J_{1}(\mathbf{x}, \mathbf{y})|  | \x - \z|_{\mf{s}}^{-1} (|\w - \z|^{-2- \delta_1}_{\mf{s}} + |\x - \z|^{-2 - \delta_1}_{\mf{s}}) \ud \z \\
    & \qquad +|\w - \x|^{\delta_1}_{\mf{s}} \int_{\x+\Theta}|J_{1}(\mathbf{z}, \mathbf{y})-J_{1}(\mathbf{x}, \mathbf{y})| (|\w - \z|^{-2- \delta_1}_{\mf{s}} + |\x - \z|^{-2 -\delta_1}_{\mf{s}})| \sr \x - \z|_{\mf{s}}^{-1} \ud \z  \\
    &=  |\w - \x|_{\mf{s}}^{\delta_1} (F^- (\x, \y, \w) + F^+ (\x, \y, \w) ),
  \end{aligned}
\end{equation*}
where we used \zcref{prop:lollibd} and the estimate \zcref{e:Kx-bd} on $J'_{\w- \x}$, together with
the compact support property of $J$, in order to restrict the integration to a
compact $\x+\Theta \subseteq \mS_{2L}$, where $\Theta$ is independent of $\w,\x,
\y$.

We start by estimating $F^-$. We split the integration domain into two. With
$a>0$ as in \zcref{e:assu-J1}, we define 
  \begin{equation*}
    A (\x, \y) = \{ \z \in \mS_{2L} \st |\x - \z|_{\mf{s}} \leq  a |\x - \y|_{\mf{s}} \} \qquad\text{and}\qquad B (\x, \y) = (\x+\Theta) \setminus A(\x, \y) .
  \end{equation*}
  We estimate the difference of kernels appearing in $F^-$ as follows:
  \begin{equation*}
    \left| J_{1}(\z,  \y) - J_{1}(\x,  \y) \right|\lesssim \begin{cases}
      |\x - \z|_{\mf{s}}^\delta |\x - \y|_{\mf{s}}^{-\alpha-\delta} & \text{ if } \z \in A(\x, \y) ; \\
      |\z - \y|_{\mf{s}}^{-\alpha} + |\x - \y|_{\mf{s}}^{-\alpha}& \text{ if } \z \in B(\x, \y) .
    \end{cases}
  \end{equation*}
  Then on $A(\x, \y)$ we find
  \begin{equation*}
    \begin{aligned}
    \int_{A(\x, \y)}\left| J_{1}(\z , \y) - J_{1}(\x , \y) \right| & (|\w - \z|^{-2- \delta_1}_{\mf{s}} + |\x - \z|^{-2 -\delta_1}_{\mf{s}}) |\x - \z|_{\mf{s}}^{-1} \ud\z \\
    & \lesssim |\x - \y|_{\mf{s}}^{-\alpha - \delta}\int_{A(\x, \y)} (|\w - \z|^{-2- \delta_1}_{\mf{s}} + |\x - \z|^{-2 -\delta_1}_{\mf{s}})|\x - \z|_{\mf{s}}^{-1 + \delta} \ud\z  \lesssim |\x - \y|_{\mf{s}}^{-\alpha - \delta_1} ,
    \end{aligned}
  \end{equation*}
  where we have used the assumption $\delta_1 < \delta$.
  On the other hand, on $B(\x, \y)$ we estimate
  \begin{equation*}
    \begin{aligned}
    \int_{B(\x, \y)} & \left| J_{1}(\z , \y) - J_{1}(\x , \y) \right| (|\w - \z|^{-2- \delta_1}_{\mf{s}} + |\x - \z|^{-2 -\delta_1}_{\mf{s}}) |\x - \z|_{\mf{s}}^{-1} \ud\z \\
    &  \lesssim \int_{B(\x, \y)} (|\w - \z|^{-3- \delta_1}_{\mf{s}} + |\x - \z|^{-3 -\delta_1}_{\mf{s}}) (|\z - \y|_{\mf{s}}^{-\alpha} + |\x - \y|_{\mf{s}}^{-\alpha}) \ud\z \\
    & \lesssim  |\w - \y|_{\mf{s}}^{-\alpha - \delta_1} + |\x - \y|_{\mf{s}}^{-\alpha - \delta_1} \lesssim  |\x - \y|_{\mf{s}}^{-\alpha - \delta_1} ,
    \end{aligned}
  \end{equation*}
  where we have used the fact that in the domain of integration  $|\w-\z|_{\mf{s}}\lesssim |\x-\z|_{\mf{s}}$, $\delta_1 < \delta < 1$ and that $  |\w - \x|_{\mf{s}}
  \leq (1/2) |\x - \y|_{\mf{s}}$ provided we fix $d \geq 2$, which means that $|\w-\y|_{\mf s}\geq |\x-\y|_{\mf s}-|\w-\x|_{\mf s}\gtrsim |\x-\y|_{\mf s}$.

  This is overall an estimate of the desired order for $F^{-}$, so we move to
  $F^{+}$.  
  In this case we again split the integral over $\z$ into three parts. Define
  \begin{equation*}
    \begin{aligned}
      A&\coloneqq\{\z\in \x+ \Theta\st |\z - \x|_{\mf s} \le 3 |x| \text{ and } |\z-\x|_{\mf s}\le a |\x-\y|_{\mf s}\};\\
      B&\coloneqq\{\z\in \x+\Theta\st |\z - \x|_{\mf s} \le 3 |x| \text{ and } |\z-\x|_{\mf s}> a |\x-\y|_{\mf s}\};\\
      C&\coloneqq\{\z\in \x+\Theta\st |\z - \x|_{\mf s} > 3 |x| \}.
    \end{aligned}
  \end{equation*}
where the radius $3 |x|$ is chosen so that $\x, \sr \x$ and $\w$ all lie in the
interior of $A \cup B$. Moreover, by assuming that $3 /d \leq a $ we can ensure $B = \varnothing$,
so we must only estimate the integral on $A \cup C$. On $A$ we employ \zcref{e:assu-J1-2} and estimate
\begin{align*}
  \int_A &|J_1(\z,\y)-J_1(\x,\y)| (|\w - \z|^{-2- \delta_1}_{\mf{s}} + |\x - \z|^{-2 -\delta_1}_{\mf{s}}) |\sigma_{\mathrm{refl}}\x-\z|^{-1}_{\mf s}\,\dif \z \\
  & \lesssim  |\x-\y|^{-\alpha-\delta}_{\mf{s}} \int_{A} |\x-\z|^{-2+\delta -\delta_1}_{\mathfrak s} |\sigma_{\mathrm{refl}}\x-\z|^{-1}_{\mathfrak s}\,\dif \z + |\x-\y|^{-\alpha - \delta }_{\mf{s}} |x|^\delta \int_{A} |\w -\z|_{\mf{s}}^{-2-\delta_1} |\sigma_{\mathrm{refl}}\x-\z|^{-1}_{\mathfrak s}\,\dif \z  
  \\
  &\lesssim |\x-\y|^{-\alpha-\delta}_{\mf{s}}\int_{A} \left(|\x-\z|^{-3+\delta - \delta_1}_{\mathfrak s} +|\sigma_{\mathrm{refl}}\x-\z|^{-3+\delta - \delta_1 }_{\mathfrak s}\right)\,\dif \z  + |\x-\y|^{\textcolor{blue}{-\alpha-\delta}}_{\mf s} |x|^{ \delta - \delta_1}\\
  & \lesssim |\x-\y|^{-\alpha - \delta_1}_{\mf s}  + |\x-\y|^{-\alpha- \delta}_{\mf s} |x|^{ \delta - \delta_1} \lesssim |\x-\y|^{-\alpha - \delta_1}_{\mf s}.
\end{align*}
Here we used that $|\x-\y|_{\mf s} \gtrsim |x|$. Also, to
estimate the integral with $\w$, we used that $|\w - \sr
\x|_{\mf{s}} = |x|$ by assumption, along with \zcref{lem:one-int-bound}.

For the integral over $C$, we further split up the integral in domains similar
to those used in the proof of \zcref{lem:bphzbd}. We define
\begin{equation*}
  \begin{aligned}
    C_1&\coloneqq\{\z\in \x+ \Theta\st 3|x| < |\z - \x|_{\mf s} \le 3 |\x-\y|_{\mf s}\text{ and } |\z-\x|_{\mf s}\le a |\x-\y|_{\mf s}\};\\
    C_2&\coloneqq\{\z\in \x+\Theta\st 3|x| < |\z - \x|_{\mf s} \le 3 |\x-\y|_{\mf s}\text{ and } |\z-\x|_{\mf s}> a |\x-\y|_{\mf s}\};\\
    C_3&\coloneqq\{\z\in \x+\Theta\st |\z - \x|_{\mf s} > 3 |\x-\y|_{\mf s} \}.
  \end{aligned}
\end{equation*}
On $C_1$ we use \zcref{e:assu-J1-2} to estimate
\begin{align*}
  \int_{C_1} &|J_1(\z,\y)-J_1(\x,\y)|(|\w - \z|^{-2- \delta_1}_{\mf{s}} + |\x - \z|^{-2 -\delta_1}_{\mf{s}})|\sigma_{\mathrm{refl}}\x-\z|^{-1}_{\mf s}\,\dif \z\notag
  \\&\le |\x-\y|_{\mf s}^{-\alpha - \delta} \int_{C_1} |\z - \x|_{\mf{s}}^\delta (|\w - \z|^{-2- \delta_1}_{\mf{s}} + |\x - \z|^{-2 -\delta_1}_{\mf{s}})|\sigma_{\mathrm{refl}}\x-\z|_{\mf s}^{-1}\,\dif\z \\
  & \lesssim |\x-\y|_{\mf s}^{-\alpha - \delta} \int_{C_1}  (|\w - \z|^{-2- \delta_1+\delta}_{\mf{s}} + |\x - \z|^{-2 -\delta_1 +\delta}_{\mf{s}})|\sigma_{\mathrm{refl}}\x-\z|_{\mf s}^{-1}\,\dif\z \lesssim |\x- \y|_{\mf{s}}^{- \alpha - \delta_1}. 
\end{align*}
where we have used that from the definition of $C_1$ we have $|\w - \z|_{\mf{s}}
\simeq |\x - \z|_{\mf{s}}\simeq |\sr \x-\z|_{\mf s} \simeq |\z|_{\mf{s}}$, together with the assumption
$\delta_1 < \delta$.

On $C_2$ we estimate with \zcref{e:assu-J1-1} that
\begin{align*}
  \int_{C_2} &|J_1(\z,\y)-J_1(\x,\y)|(|\w - \z|^{-2- \delta_1}_{\mf{s}} + |\x - \z|^{-2 -\delta_1}_{\mf{s}})|\sigma_{\mathrm{refl}}\x-\z|^{-1}_{\mf s}\,\dif \z\notag
  \\
  &\lesssim \int_{C_2} (|\z - \y|_{\mf{s}}^{-\alpha} +|\x - \y|_{\mf{s}}^{-\alpha}) (|\w - \z|^{-2- \delta_1}_{\mf{s}} + |\x - \z|^{-2 -\delta_1}_{\mf{s}})|\sigma_{\mathrm{refl}}\x-\z|_{\mf s}^{-1}\,\dif\z .
\end{align*}
Now for the term involving $|\z - \y|_{\mf{s}}^{-\alpha}$ we use that $|\w - \z|_{\mf{s}}
\simeq |\x - \z|_{\mf{s}}\simeq|\z-\sr\x|_{\mf s} \simeq |\z|_{\mf{s}}$ to find
\begin{equation*}
  \begin{aligned}
  \int_{C_2} |\z - \y|_{\mf{s}}^{-\alpha} (| \w - \z|^{-2- \delta_1}_{\mf{s}} + | \x - \z|^{-2 -\delta_1}_{\mf{s}})|\z - \sigma_{\mathrm{refl}}\x|_{\mf s}^{-1}\,\dif\z
  &\lesssim \int_{C_2} |\z - \y|_{\mf{s}}^{-\alpha} | \z - \x|^{-3- \delta_1}_{\mf{s}}\,\dif\z
\\&\lesssim|\x-\y|^{-3-\delta_1}_{\mf s}\int_{C_2} |\z-\y|^{-\alpha}_{\mf s}\,\dif \z\lesssim |\x - \y|_{\mf{s}}^{- \alpha - \delta_1}.
\end{aligned}
\end{equation*}
Similarly, for the term involving $|\x - \y|_{\mf{s}}^{-\alpha}$, we get
\begin{equation*}
  |\x - \y|_{\mf{s}}^{-\alpha} \int_{C_2}  (|\w - \z|^{-2- \delta_1}_{\mf{s}} + |\x - \z|^{-2 -\delta_1}_{\mf{s}})|\sigma_{\mathrm{refl}}\x-\z|_{\mf s}^{-1}\,\dif\z \lesssim |\x - \y|_{\mf{s}}^{- \alpha - \delta_1} .
\end{equation*}
Finally, on $C_3$ we estimate
\begin{align*}
  \int_{C_3} &|J_1(\z,\y)-J_1(\x,\y)|(|\w - \z|^{-2- \delta_1}_{\mf{s}} + |\x - \z|^{-2 -\delta_1}_{\mf{s}})|\sigma_{\mathrm{refl}}\x-\z|^{-1}_{\mf s}\,\dif \z\notag
  \\
  &\lesssim \int_{C_3} (|\z - \y|_{\mf{s}}^{-\alpha} +|\x - \y|_{\mf{s}}^{-\alpha}) (|\w - \z|^{-2- \delta_1}_{\mf{s}} + |\x - \z|^{-2 -\delta_1}_{\mf{s}})|\sigma_{\mathrm{refl}}\x-\z|_{\mf s}^{-1}\,\dif\z \\
  & \lesssim \int_{C_3} \left(|\z - \y|_{\mf{s}}^{-\alpha - 3 - \delta_1} + |\x - \y|_{\mf{s}}^{-\alpha} |\z - \y|_{\mf{s}}^{- 3 - \delta_1}\right) \,\dif\z \lesssim |\x - \y|_{\mf{s}}^{-\alpha - \delta_1}.
\end{align*}
This completes the proof of the result.
\end{proof}

\subsection{Variance at the boundary}

We recall from \zcref{subsec:Reduction-to-the} the definition %
\begin{equation*}
  \mathcal{X}^{\eps}_{\uu,\vv;s,t}(g)=\int_{\mathbb{R}}\Psi^{\eps}_{s, t ; r} \langle\varphi^{\eps}_{\uu,\vv},g_{r}\rangle \ud r,
\end{equation*}
In this section we prove the following proposition.

\begin{prop}\label{prop:var-boundary}
  For $\tau\in\left\{ \rscherryrb,\rscherryrenormr,\rselkrbr,\rscherrybb,\rscherryrenormb,\rscherryrenormrenorm,\rselkrenormrr\right\} $ we have
  \[
  \adjustlimits\lim_{\eps\downarrow0}\sup_{\zeta\in(0,\eps)}\Var\left(\mathcal{X}^{\eps}_{\uu,\vv;0,T}\left(\boxop{\tau}^{\eps,\zeta}\right)\right)=0.
  \]
  \end{prop}

The strategy of the proof is the same for all the trees. We will use the
dominated convergence theorem, together with the fact that the integrands are
uniformly bounded in $\eps$ and $\zeta$, and that for fixed $\ve, \zeta$ they
vanish at the boundary. In particular, we note that
\begin{equation}\label{e:var-identity}
    \Var\left(\mathcal{X}^{\eps}_{\uu,\vv;0,T}\left(\boxop{\tau}^{\eps,\zeta}\right)\right)  =  \int_{\mathbb{R}^2}
    \Psi^{\eps}_{s, t; r_1} \Psi^{\eps}_{s, t; r_2} \left\langle (\varphi^{\eps}_{\uu,\vv})^{\otimes 2},\Cov \left( \boxop{\tau}^{\eps,\zeta}_{r_1}, \boxop{\tau}^{\eps,\zeta}_{r_2} \right)  \right\rangle  \ud r_{12}  ,
\end{equation}
where we have defined
\begin{equation*}
  \Cov \left( \boxop{\tau}^{\eps,\zeta}_{r_1}, \boxop{\tau}^{\eps,\zeta}_{r_2} \right) (x, y) = \Cov \left( \boxop{\tau}^{\eps,\zeta} (r_1, x), \boxop{\tau}^{\eps,\zeta} (r_2, y) \right) \qquad \text{for all } x, y \in \TT_{2L}
\end{equation*}
and used the notation $\langle f,g\rangle = \iint_{[0,L]^2} f(x,y)g(x,y)\,\dif x\,\dif y$ (which extends the inner product $\langle\cdot,\cdot\rangle$ defined above to functions of two variables).
Now, for $\x=(t,x), \y=(s,y)$ such that $t \neq s$, we have the following bound.
\begin{lem}\label{lem:boundary-cov-vanish} 
  For $\tau\in\left\{
  \rscherryrb,\rscherryrenormr,\rselkrbr,\rscherrybb,\rscherryrenormb,\rscherryrenormrenorm,\rselkrenormrr\right\}
  $ we have for any $t \neq s$ and any $c > 0$ that
  \begin{equation*}
    \adjustlimits\lim_{\ve \downarrow 0} \sup_{\zeta \in (0, \ve)} \sup_{d_L(x),d_L(y) \leq c \ve} \; \left\vert \Cov \left( \boxop{\tau}^{\eps,\zeta}_{t}, \boxop{\tau}^{\eps,\zeta}_{s} \right) (x, y) \right\vert = 0 .
  \end{equation*}
\end{lem}

\begin{proof}
  This result is a consequence of the estimates
  \zcref{e:newrbc,e:newrbc2,e:newrbc3,e:newrbc4,e:newrbc5,e:newrbc6},
  respectively for $\tau \in
  \{\rscherryrb,\rselkrbr,\rscherryrenormb,\rscherryrenormr,\rscherryrenormrenorm,\rselkrenormrr
  \}$, while the case $\tau = \rscherrybb$ is trivial because the term is
  deterministic and the covariance is zero.
\end{proof}
With this bound in hand, we are ready to prove \zcref{prop:var-boundary}.

\begin{proof}[Proof of \zcref{prop:var-boundary}]
  We use the identity \zcref{e:var-identity}, and observe that by
  \zcref{lem:boundary-cov-vanish,eq:Psiproperties}, we have for any fixed $c>0$ and $r_1 \neq r_2$ that
  \begin{equation*}
    \begin{aligned}
    \lim_{\ve \to 0} \Big\vert \Psi^{\eps}_{s, t; r_1} \Psi^{\eps}_{s, t; r_2} & \left\langle (\varphi^{\eps}_{\uu,\vv})^{\otimes 2},\Cov \left( \boxop{\tau}^{\eps,\zeta}_{r_1}, \boxop{\tau}^{\eps,\zeta}_{r_2} \right)  \right\rrangle_{\TT_{2L}^2} \Big\vert \\
    & \lesssim\lim_{\ve \to 0 }\sup_{\zeta \in (0, \ve)} \sup_{ d_L(x),d_L(y)\leq c \ve}  \Big\vert \Cov \left( \boxop{\tau}^{\eps,\zeta}_{r_1}, \boxop{\tau}^{\eps,\zeta}_{r_2} \right) (x, y)\Big\vert \left\langle (\varphi^{\eps}_{\uu,\vv})^{\otimes 2}, 1 \right\rangle  = 0 .
    \end{aligned}
  \end{equation*}
  Therefore, the result follows by dominated convergence once we prove that
  the integrand 
  \begin{equation*}
    (r_1, r_2) \mapsto \Psi^{\eps}_{s, t; r_1} \Psi^{\eps}_{s, t; r_2} \left\langle (\varphi^{\eps}_{\uu,\vv})^{\otimes 2},\Cov \left( \boxop{\tau}^{\eps,\zeta}_{r_1}, \boxop{\tau}^{\eps,\zeta}_{r_2} \right)  \right\rangle
  \end{equation*}
  is bounded, uniformly in $\ve \in (0, 1)$ and $\zeta \in (0, \ve)$, by an integrable function in $r_1, r_2$. This is a consequence of
  \zcref{eq:redbluecherrybd,eq:elkrbrcovbd,e:rrb-aim,eq:elkrrr-cov,eq:candelabra-cov-bd,e:moose-aim-1},
  which imply that
for each $\tau$ as in the statement of this proposition (the only exception is
$\rscherrybb$, which is deterministic so that in that case the result is trivial
  because the covariance vanishes), we can estimate for any $r_1 \neq r_2$ and
  any $x, y \in \TT_{2L}$ that
\begin{equation*}
  \begin{aligned}
    \left\vert \Cov \left( \boxop{\tau}^{\eps,\zeta}, \boxop{\tau}^{\eps,\zeta} \right) ( (r_1, x), (r_2, y)) \right\vert & \lesssim d_{\mf{s},\mathscr{S}}(\x, \y)^{-3/2} \lesssim |r_1 - r_2|^{-3/4} .
  \end{aligned}
\end{equation*} 
Since 
$  \int_{[s, t]^2} |r_1 - r_2|^{-3/4} \ud r_1 \ud r_2 < \infty$,
we have obtained the required dominating function and completed the proof.
 \end{proof}

%% file: open-kpz-invariant-appendixB.tex
\section{An estimate on multiple integrals}

Let $\mathsf{S}$ be a finite set with two distinguished elements
$\varrho_{1},\varrho_{2}\in\mathsf{S}$, let $\mathsf{S}_{0}\coloneqq\mathsf{S}\setminus\{\varrho_{1},\varrho_{2}\}$,
and let $Q\colon\binom{\mathsf{S}}{2}\to\{0,1,2\ldots\}$. For a compact set $\Theta\subseteq \mS_{2L}$, define
\begin{equation}\label{eq:Idef}
\mathscr{I}[\mathsf{S},Q,\Theta](\mathbf{x}_{\varrho_{1}},\mathbf{x}_{\varrho_{2}})\coloneqq\int_{(\x_{\varrho_1}+\Theta)^{\mathsf{S}_{0}}}\left(\prod_{\{u,v\}\in\binom{\mathsf{S}}{2}}d(\mathbf{x}_{u},\mathbf{x}_{v})^{-Q(u,v)}\right)\prod_{v\in\mathsf{S}_{0}}\dif\mathbf{x}_{v}.
\end{equation}
For $\varnothing\ne\overline{\mathsf{S}}\subseteq\mathsf{S}$, define
\begin{equation}
\deg_{Q}(\overline{\mathsf{S}})\coloneqq3(|\overline{\mathsf{S}}|-1)-\sum_{\{u,v\}\in\binom{\overline{\mathsf{S}}}{2}}Q(u,v).\label{eq:degQ}
\end{equation}
We note that if $\overline{\mathsf{S}}$ is a singleton, then the
sum on the right side of \zcref{eq:degQ} is zero, and so $\deg_{Q}(\overline{\mathsf{S}})=0$.
\begin{prop}
\label{prop:hepp-prop}Let $\Theta\subseteq \mathbb{R}\times\mathbb{T}_{2L}$
be a compact set. Suppose that
\begin{equation}\label{eq:posdeg-1}
\deg_{Q}(\overline{\mathsf{S}})>0\qquad\text{for all }\overline{\mathsf{S}}\subseteq\mathsf{S}\text{ with }|\overline{\mathsf{S}}|\ge2,
\end{equation}
and define
\begin{equation}
  \gamma\coloneqq\max\left\{ 3-\deg_{Q}(\overline{\mathsf{S}})\st\{\varrho_{1},\varrho_{2}\}\subseteq\overline{\mathsf{S}}\subseteq\mathsf{S}\right\} .\label{eq:gammadegtilde}
\end{equation}
Then we have
\[
\left|\mathscr{I}[\mathsf{S},Q,\Theta](\mathbf{x}_{\varrho_{1}},\mathbf{x}_{\varrho_{2}})\right|\lesssim d_{\mathfrak{s},\mathscr{S}}(\mathbf{x}_{\varrho_{1}},\mathbf{x}_{\varrho_{2}})^{-\gamma}\left(\log\left(2+d_{\mathfrak{s},\mathscr{S}}(\mathbf{x}_{\varrho_{1}},\mathbf{x}_{\varrho_{2}})^{-1}\right)\right)^{|\mathsf{S}|-2}.
\]
\end{prop}

The proof of this proposition builds on typical Hepp-sector decompositions of the integral domain. Since the proposition does not seem to have appeared in the literature in this form, we include the details.

We begin by introducing the notion of a Hepp sector. 
Hepp
sectors provide a way to partition the integration domain in terms of the relative
distances between the integration variables. Our definition will be
adapted to the reflected periodic distance $d_{\mathfrak{s},\mathscr{S}}$
introduced in \zcref{eq:dsS}.
In particular, we define the set $\mathscr{T}(\mathsf{S})$ of \emph{Hepp sectors} to be
the set of binary, rooted trees with $|\mathsf{S}|$ leaves indexed
by $\mathsf{S}$. For a binary tree $\mathfrak{t}\in\mathscr{T}(\mathsf{S})$,
we write $\mathfrak{t}^{\circ}$ for the set of \emph{non}-leaf vertices
of $\mathfrak{t}$. We say that a map $\mathbf{n}\colon\mathfrak{t}^{\circ}\to\mathbb{N}$
is an \emph{admissible scaling} if
\begin{equation}
u\preccurlyeq_{\mathfrak{t}}v\implies\mathbf{n}(u)\le\mathbf{n}(v),\label{eq:admissible-scaling}
\end{equation}
and we write $\mathcal{A}(\mathfrak{t})$ for the set of all admissible
scalings of a given tree $\mathfrak{t}$. Here the partial order $\preccurlyeq_{\mathfrak{t}}$
on $\mathfrak{t}$ is induced by the tree structure, with the root
of the tree being the minimal element. We write $u\curlywedge_{\mathfrak{t}}v$
for the least common ancestor between $u$ and $v$, and write $u\prec_{\mathfrak{t}}v$
if $u\preccurlyeq_{\mathfrak{t}}v$ and $u\ne v$.\label{treedefs}
We will often drop the subscript $\mathfrak{t}$ if the tree $\mathfrak{t}$
is clear from context.

For any admissible scaling $\mathbf{n}$, we define
\newcommand{\sfS}{\mathsf{S}}
\begin{equation}
\mathcal{U}(\mathfrak{t},\mathbf{n},\mathbf{x}_{\varrho_{1}},\mathbf{x}_{\varrho_{2}})\coloneqq
\left\{
(\mathbf{x}_{v})_{v\in\mathsf{S}_{0}}\in(\x_{\varrho_1}+\Theta)^{|\mathsf{S}_{0}|}
\st
C^{-1}2^{-\mathbf{n}(v\curlywedge_{\mathfrak{t}}w)}\le d_{\mathfrak{s},\mathscr{S}}(\mathbf{x}_{v},\mathbf{x}_{w})
\le C2^{-\mathbf{n}(v\curlywedge_{\mathfrak{t}}w)}
\text{ for all }
v,w\in\sfS\right\},
  \label{eq:Udef}
\end{equation}
where $C$ is a constant chosen sufficiently large (and allowed to
depend on $|\mathsf{S}|$ and $\Theta$, but not on $\mathfrak{t}$, $\mathbf{n}$, $\mathbf{x}_{\varrho_{1}}$, or $\mathbf{x}_{\varrho_{2}}$)
that for $\mathbf{x}_{\varrho_{1}},\mathbf{x}_{\varrho_{2}}\in(\x_{\varrho_1}+\Theta)^{2}$,
we have
\begin{equation}
\bigcup\left\{ \mathcal{U}(\mathfrak{t},\mathbf{n},\mathbf{x}_{\varrho_{1}},\mathbf{x}_{\varrho_{2}})\st\mathfrak{t}\in\mathscr{T}(\mathsf{S}),\mathbf{n}\in\mathcal{A}(\mathfrak{t}),\mathbf{n}(\mathbf{x}_{\varrho_{1}}\curlywedge_{\mathfrak{t}}\mathbf{x}_{\varrho_{2}})=\lfloor\log_{2}d_{\mathfrak{s},\mathscr{S}}(\mathbf{x}_{\varrho_{1}},\mathbf{x}_{\varrho_{2}})\rfloor\right\} =(\x_{\varrho_1}+\Theta)^{n-2}.\label{eq:coveralloftheta}
\end{equation}
We make a couple of brief comments at this stage:
\begin{itemize}
\item The condition that $\mathbf{n}(\mathbf{x}_{\varrho_{1}}\curlywedge_{\mathfrak{t}}\mathbf{x}_{\varrho_{2}})=\lfloor\log_{2}d_{\mathfrak{s},\mathscr{S}}(\mathbf{x}_{\varrho_{1}},\mathbf{x}_{\varrho_{2}})\rfloor$
is in some sense unnecessary, since one can take $v=\varrho_{1}$
and $w=\varrho_{2}$ in the condition in \zcref{eq:Udef} to see that
this condition must hold up to an additive constant in order for $\mathcal{U}(\mathfrak{t},\mathbf{n},\mathbf{x}_{\varrho_{1}},\mathbf{x}_{\varrho_{2}})$
to be nonempty. We fix the precise value for concreteness.
\item The sets $\mathcal{U}(\mathfrak{t},\mathbf{n},\mathbf{x}_{\varrho_{1}},\mathbf{x}_{\varrho_{2}})$
depend on the choice of the distinguished elements $\varrho_{1},\varrho_{2}\in\mathsf{S}$,
but we suppress this in the notation to keep the notation light.
\end{itemize}
We will need the following lemma.
\begin{lem}
\label{lem:Usizebound}Let $\mathfrak{t}\in\mathscr{T}(\mathsf{S})$
be fixed. With $v_{\star}\coloneqq\varrho_{1}\curlywedge_{\mathfrak{t}}\varrho_{2}$,
we have, using $|\cdot|$ to denote the Lebesgue measure of a set,
\[
|\mathcal{U}(\mathfrak{t},\mathbf{n},\mathbf{x}_{\varrho_{1}},\mathbf{x}_{\varrho_{2}})|\lesssim\prod_{v\in\mathfrak{t}^{\circ}\setminus\{v_{\star}\}}2^{-3\mathbf{n}(v)}.
\]
\end{lem}

\begin{proof}
For notational convenience, let us assume in this proof that $\mathsf{S}=\{1,\ldots,n\}$
and $\varrho_{i}=i$ for $i=1,2$, so $\mathsf{S}_{0}=\{3,\ldots,n\}$.
We begin by defining a labeling map $\mathrm{I}\colon\mathfrak{t}\to\mathsf{S}$
by first setting $\mathrm{I}(j)\coloneqq j$ for all $j\in\mathsf{S}$
and then, whenever $u\in\mathfrak{t}^{\circ}$ has children $v$ and
$w$, setting $\mathrm{I}(u)\coloneqq\mathrm{I}(v)\wedge\mathrm{I}(w)$.
Now, for each $j\in\mathsf{S}_{0}$, we let $u_{j}\in\mathfrak{t}^{\circ}$
be the parent of the minimal element $u$ of $\mathfrak{t}$ with
$\mathrm{I}(u)=j$. (This minimal element always does have a parent
because the root of $\mathfrak{t}$ has label $1$ and we have $j\ge3$.)
Let $v_{j}$ be the child of $u_{j}$ such that $\mathrm{I}(v_{j})\ne j$,
and let $k_{j}\coloneqq\mathrm{I}(v_{j})$, so the construction implies
that $k_{j}<j$ (as we know that $j\ne\mathsf{I}(u_{j})=j\wedge\mathsf{I}(v_{j})=j\wedge k_{j}$)
and that 
\begin{equation}
j\curlywedge_{\mathfrak{t}}k_{j}=u_{j}.\label{eq:jwedgekjisuj}
\end{equation}
It is also straightforward to check that
\begin{equation}
\{u_{j}\st j\in\mathsf{S}_{0}\}=\mathfrak{t}^{\circ}\setminus\{v_{\star}\}.\label{eq:vertices-exhaust}
\end{equation}
Now we think of choosing an element of the set $\mathcal{U}(\mathfrak{t},\mathbf{n},\mathbf{x}_{1},\mathbf{x}_{2})$
by choosing the values of the $\mathbf{x}_{j}$s inductively in order
of $j$. The designated elements $\mathbf{x}_{\varrho_{1}}=\mathbf{x}_{1}$
and $\mathbf{x}_{\varrho_{2}}=\mathbf{x}_{2}$ are already fixed.
For each $j\ge3$, in order to have $(\mathbf{x}_{j})_{j\in\mathsf{S}_{0}}\in\mathcal{U}(\mathfrak{t},\mathbf{n},\mathbf{x}_{\varrho_{1}},\mathbf{x}_{\varrho_{2}})$,
we must have 
\begin{equation}
d_{\mathfrak{s},\mathscr{S}}(\mathbf{x}_{j},\mathbf{x}_{k_{j}})\le C2^{-\mathbf{n}(j\curlywedge_{\mathfrak{t}}k_{j})}\overset{\zcref{eq:jwedgekjisuj}}{=}C2^{-\mathbf{n}(u_{j})},\label{eq:exhaustintermediate}
\end{equation}
and hence $\mathbf{x}_{j}$ must lie in a set of volume at most $\tilde{C}2^{-3\mathbf{n}(u_{j})}$
for a larger constant $\tilde{C}$ depending on $\Theta$. (The factor
of $3$ is due to the parabolic scaling.) Using this observation inductively,
we see that 
\[
|\mathcal{U}(\mathfrak{t},\mathbf{n},\mathbf{x}_{\varrho_{1}},\mathbf{x}_{\varrho_{2}})|\lesssim\prod^{n}_{j=3}2^{-3\mathbf{n}(u_{j})}\overset{\zcref{eq:exhaustintermediate}}{=}\prod_{v\in\mathfrak{t}^{\circ}\setminus\{v_{\star}\}}2^{-3\mathbf{n}(v)},
\]
as claimed. 
\end{proof}

Now we can prove \zcref{prop:hepp-prop}.
\begin{proof}[Proof of \zcref{prop:hepp-prop}]

  \begin{thmstepnv}
\item 
  We decompose the domain $(\x_{\varrho_1}+\Theta)^{\mathsf{S}_0}$
into Hepp sectors. 
We apply \zcref{eq:coveralloftheta} to \zcref{eq:Idef}
to obtain the bound, for all $\mathbf{x}_{\varrho_{1}},\mathbf{x}_{\varrho_{2}}\in[0,L]$,
\begin{equation}
\begin{aligned} & \left|\mathscr{I}[\mathsf{S},Q,\Theta](\mathbf{x}_{\varrho_{1}},\mathbf{x}_{\varrho_{2}})\right| %
 \lesssim\sum_{\mathfrak{t}\in\mathscr{T}_{\mathsf{S}}}\sum_{\substack{\mathbf{n}\in\mathcal{A}(\mathfrak{t})\\
\mathbf{n}(u_{\star})=n_{\star}
}
}\int_{\mathcal{U}(\mathfrak{t},\mathbf{n},\mathbf{x}_{\varrho_{1}},\mathbf{x}_{\varrho_{2}})}\left(\prod_{\{u,v\}\in\binom{\mathsf{S}}{2}}d_{\mathfrak{s},\mathscr{S}}(\mathbf{x}_{u},\mathbf{x}_{v})^{-Q(u,v)}\right)\prod_{v\in\mathsf{S}_0}\dif\mathbf{x}_{v},
\end{aligned}
\label{eq:hepp-separate}
\end{equation}
where we have defined
\begin{equation}
u_{\star}=\varrho_{1}\curlywedge\varrho_{2}\qquad\text{and}\qquad n_{*}\coloneqq-\lfloor\log_{2}d_{\mathfrak{s},\mathscr{S}}(\mathbf{x}_{\varrho_{1}},\mathbf{x}_{\varrho_{2}})\rfloor.\label{eq:nstarustar}
\end{equation}

\item Now we bound each term of the outer sum on the right side of   \zcref{eq:hepp-separate} 
individually. Thus, for the remainder
of the proof, we fix a $\mathfrak{t}\in\mathscr{T}_{\mathsf{S}}$,
and we will use the notation $\preccurlyeq$, $\curlywedge$, etc.\ 
with respect to $\mathfrak{t}$. For an admissible
scaling $\mathbf{n}\in\mathcal{A}(\mathfrak{t})$, we recall from the definition \zcref{eq:Udef} that whenever
$(\mathbf{x}_{v})_{v\in\mathsf{S}_0}\in\mathcal{U}(\mathfrak{t},\mathbf{n},\mathbf{x}_{\varrho_{1}},\mathbf{x}_{\varrho_{2}})$,
we have
\begin{equation}
d_{\mathfrak{s},\mathscr{S}}(\mathbf{x}_{u},\mathbf{x}_{v})\simeq2^{-\mathbf{n}(u\curlywedge v)}.\label{eq:dbd}
\end{equation}
Using \zcref{lem:Usizebound,eq:dbd} in \zcref{eq:hepp-separate},
we see that the summand on the right side of \zcref{eq:hepp-separate}
is bounded by
\begin{equation}
  \mathscr{I}_{\mathfrak{t}}\coloneqq\sum_{\substack{\mathbf{n}\in\mathcal{A}(\mathfrak{t})\\
\mathbf{n}(u_{\star})=n_{\star}
}
}\left(\prod_{\{u,v\}\in\binom{\mathsf{S}}{2}}2^{Q(u,v)\mathbf{n}(u\curlywedge_{\mathfrak{t}}v)}\right)\left(\prod_{v\in\mathfrak{t}^{\circ}\setminus\{u_{\star}\}}2^{-3\mathbf{n}(v)}\right).\label{eq:hepp-separate-1}
\end{equation}
We would like to estimate $\mathscr{I}_{\mathfrak{t}}$ by integrating ``from
the leaves to the root.'' For this purpose, we rewrite the right
side of \zcref{eq:hepp-separate-1} as a node-by-node sum. For $w\in\mathfrak{t}^{\circ}$,
we define
\begin{equation}
\eta(w)\coloneqq3\cdot\mathbf{1}_{w\ne u_{\star}}-\sum_{\substack{\{u,v\}\in\binom{\mathsf{S}}{2}\\
u\curlywedge v=w
}
}Q(u,v),\label{eq:etadef-1}
\end{equation}
and then we can rewrite \zcref{eq:hepp-separate-1} as 
\begin{equation}
\mathscr{I}_{\mathfrak{t}}=\sum_{\substack{\mathbf{n}\in\mathcal{A}(\mathfrak{t})\\
\mathbf{n}(u_{\star})=n_{\star}
}
}\prod_{w\in\mathfrak{t}^{\circ}}2^{-\eta(w)\mathbf{n}(w)}.\label{eq:hepp-separate-1-1}
\end{equation}
\item In order to bound \zcref{eq:hepp-separate-1-1}, we need to develop some identities involving the $\eta(w)$s. For $w\in\mathfrak{t}^{\circ}$, we define 
\[
\mathsf{S}_{w}=\{u\in\mathsf{S}\st w\prec u\}
\]
to be the set of leaf nodes that are descendants of $w$. Using the
definitions \zcref{eq:posdeg-1,eq:etadef-1} we see that for any $u\in\mathfrak{t}^{\circ}$
with children $w_{1},w_{2}$, we have
\begin{align}
  \eta(u)&+\deg_Q(\mathsf{S}_{w_{1}})+\deg_Q(\mathsf{S}_{w_{2}})\notag\\ & =3\cdot\mathbf{1}_{u\ne u_{\star}}-\sum_{\substack{\{z_{1},z_{2}\}\in\binom{\mathsf{S}}{2}\\
z_{1}\curlywedge z_{2}=u
}
}Q(z_{1},z_{2})+\sum^{2}_{i=1}\left\{ 3(|\mathsf{S}_{w_{i}}|-1)-\sum_{\{z_{1},z_{2}\}\in\binom{\mathsf{S}_{w_i}}{2}}Q(z_{1},z_{2})\right\} \nonumber \\
 & =3(|\mathsf{S}_{u}|-1-\mathbf{1}_{u=u_{\star}})-\sum_{\substack{\{z_{1},z_{2}\}\in\binom{\mathsf{S}_{u}}{2}\\
z_{1}\curlywedge z_{2}=u
}
}Q(z_{1},z_{2})=\deg_Q(\mathsf{S}_{u})-3\cdot\mathbf{1}_{u=u_{\star}}.\label{eq:degtilderecursion}
\end{align}
Now we claim that
\begin{equation}
\text{if }u\not\prec u_{\star}\text{, then }\sum_{\substack{v\in\mathfrak{t}^{\circ}\\
v\succcurlyeq u
}
}\eta(v)=\deg_Q(\mathsf{S}_{w})-3\cdot\mathbf{1}_{u=u_{\star}}.\label{eq:sumetasisdef}
\end{equation}
We prove \zcref{eq:sumetasisdef} by induction. If $u$ has just two
descendants, i.e. $\mathsf{S}_{u}=\{z_{1},z_{2}\}$, then we have
\[
\sum_{\substack{v\in\mathfrak{t}^{\circ}\\
v\succcurlyeq u
}
}\eta(v)=\eta(u)\overset{\zcref{eq:etadef-1}}{=}3\cdot\mathbf{1}_{u\ne u_{\star}}-Q(z_{1},z_{2})\overset{\zcref{eq:degQ}}{=}3\cdot\mathbf{1}_{u\ne u_{\star}}+\deg_Q(\overline{\mathsf{S}}_{u})-3=\deg_Q(\overline{\mathsf{S}}_{u})-3\cdot\mathbf{1}_{u=u_{\star}},
\]
as claimed. On the other hand, if $u$ has two children $w_{1}$ and
$w_{2}$ for which \zcref{eq:sumetasisdef} holds, then we can write,
using the induction hypothesis, that
\[
\sum_{\substack{v\in\mathfrak{t}^{\circ}\\
v\succcurlyeq u
}
}\eta(v)=\eta(u)+\deg_Q(\mathsf{S}_{w_{1}})+\deg_Q(\mathsf{S}_{w_{2}})\overset{\zcref{eq:degtilderecursion}}{=}\deg_Q(\mathsf{S}_{u})-3\cdot\mathbf{1}_{u=u_{\star}},
\]
as well.

\item Now we use the identity \zcref{eq:sumetasisdef} to sum from the leaves
  towards the root in \zcref{eq:hepp-separate-1-1}. For $u\in\mathfrak{t}^{\circ}$, we define $\mathfrak{t}_{\succcurlyeq u}$
to be the subtree of $\mathfrak{t}$ rooted at $u$. Then any $\mathbf{n}\in\mathcal{A}(\mathfrak{t},n)$
can be split into a map $\mathbf{n}|_{\mathfrak{t}_{\succcurlyeq u}}\in\mathcal{A}(\mathfrak{t}_{\succcurlyeq u})$
and $\mathbf{n}|_{\mathfrak{t}\setminus\mathfrak{t}_{\succcurlyeq u}}\in\mathcal{A}(\mathfrak{t}\setminus\mathfrak{t}_{\succcurlyeq u},n)$,
and in fact this splitting gives a bijection
\begin{equation}
\mathcal{A}(\mathfrak{t},n)\cong\left\{ (\mathbf{n}^{(1)},\mathbf{n}^{(2)})\in\mathcal{A}(\mathfrak{t}_{\succcurlyeq u})\times\mathcal{A}(\mathfrak{t}\setminus\mathfrak{t}_{\succcurlyeq u},n)\st\begin{gathered}\mathbf{n}^{(1)}(u)\ge\mathbf{n}^{(2)}(v)\text{ if \ensuremath{v} is the parent of \ensuremath{u}}\\
\text{or }\mathbf{n}^{(1)}(u)\ge n\text{ if }u\text{ is the root}
\end{gathered}
\right\} .\label{eq:bijection}
\end{equation}
Therefore, for fixed $n\in\mathbb{N}$, let us define the set 
\begin{equation}
\mathcal{A}(\mathfrak{t}_{\succcurlyeq u},n)=\{\mathbf{n}\in\mathcal{A}(\mathfrak{t}_{\succcurlyeq u})\st\mathbf{n}(u)\ge n\}.\label{eq:Atndef}
\end{equation}
We now claim that, whenever $u\not\preccurlyeq u_{\star}$, we have
\begin{equation}
\sum_{\mathbf{n}\in\mathcal{A}(\mathfrak{t}_{\succcurlyeq u},n)}\prod_{v\in\mathfrak{t}^{\circ}_{\succcurlyeq u}}2^{-\eta(v)\mathbf{n}(v)}\lesssim2^{-n\deg_Q(\mathsf{S}_{u})}.\label{eq:sumprodindeq}
\end{equation}
To prove \zcref{eq:sumprodindeq}, we again use induction. If $u$
has just two descendants (which are thus leaves of $\mathfrak{t}$),
then the left side of \zcref{eq:sumprodindeq} reduces to 
\[
\sum^{\infty}_{n_{u}=n}2^{-\eta(u)\mathbf{n}(u)}\overset{\zcref{eq:sumetasisdef}}{=}\sum^{\infty}_{n_{u}=n}2^{-\eta(u)\deg_Q(\mathsf{S}_{u})}\simeq2^{-n\deg_Q(\mathsf{S}_{u})}
\]
since $\deg_Q(\mathsf{S}_{u})>0$ by \zcref{eq:posdeg-1}. Otherwise,
if $u\not\preccurlyeq u_{\star}$ has children $z_{1},z_{2}$ and
\zcref{eq:sumprodindeq} holds for $z_{1}$ and $z_{2}$, then we
have
\begin{align*}
\sum_{\mathbf{n}\in\mathcal{A}(\mathfrak{t}_{\succcurlyeq u},n)}\prod_{v\in\mathfrak{t}^{\circ}_{\succcurlyeq u}}2^{-\eta(v)\mathbf{n}(v)} & =\sum^{\infty}_{n_{u}=n}2^{-\eta(u)n_{u}}\prod^{2}_{i=1}\left(\sum_{\mathbf{n}_{i}\in\mathcal{A}(\mathfrak{t}_{\succcurlyeq z_{i}},n_{u})}\prod_{v\in\mathfrak{t}^{\circ}_{\succcurlyeq z_{i}}}2^{-\eta(v)\mathbf{n}(v)}\right)\\
\ovset{\zcref{eq:sumprodindeq}} & \lesssim\sum^{\infty}_{n_{u}=n}2^{-\left(\eta(u)+\deg_Q(\mathsf{S}_{z_{1}})+\deg_Q(\mathsf{S}_{z_{2}})\right)n_{u}}\overset{\zcref{eq:degtilderecursion}}{=}\sum^{\infty}_{n_{u}=n}2^{-n_{u}\deg_Q(\mathsf{S}_{u})}\lesssim2^{-n\deg_Q(\mathsf{S}_{u})},
\end{align*}
again using the fact that $\deg_Q(\mathsf{S}_{u})>0$. %

Similarly to \zcref{eq:sumprodindeq}, we have, if $u_{\star}$ has
children $\{z_{1},z_{2}\}$, that
\begin{equation}
\begin{aligned}\sum_{\substack{\mathbf{n}\in\mathcal{A}(\mathfrak{t}_{\succcurlyeq u_{\star}})\\
\mathbf{n}(u_{\star})=n_{\star}
}
}\prod_{v\in\mathfrak{t}^{\circ}_{\succcurlyeq u_{\star}}}2^{-\eta(v)\mathbf{n}(v)} & =2^{-\eta(u_{\star})n_{\star}}\prod^{2}_{i=1}\sum_{\mathbf{n}\in\mathcal{A}(\mathfrak{t}_{\succcurlyeq z_{i}},n_{\star})}\prod_{u\in\mathfrak{t}^{\circ}_{\succcurlyeq z_{i}}}2^{-\eta(u)\mathbf{n}(u)}\\
 & \ovset{\zcref{eq:sumprodindeq}}\lesssim2^{-\left(\deg_Q(\mathsf{S}_{z_{1}})+\deg_Q(\mathsf{S}_{z_{2}})+\eta(u_{\star})\right)n_{\star}}\ovset{\zcref{eq:degtilderecursion}}\simeq2^{\left(\deg_Q(\mathsf{S}_{u_{\star}})-3\right)n_{\star}}.
\end{aligned}
\label{eq:sumproduct-ustar}
\end{equation}
In particular, if $u_{\star}$ is in fact the root of $\mathfrak{t}$,
then combining \zcref{eq:hepp-separate-1-1,eq:sumproduct-ustar} and
then recalling \zcref{eq:nstarustar,eq:gammadegtilde} gives us
\begin{equation}
\mathscr{I}_{\mathfrak{t}}\lesssim2^{\left(\deg_Q(\mathsf{S}_{u_{\star}})-3\right)n_{\star}}\simeq d_{\mathfrak{s},\mathscr{S}}(\mathbf{x}_{\varrho_{1}},\mathbf{x}_{\varrho_{2}})^{\gamma},\label{eq:ustarrootestimate}
\end{equation}
and the right hand side is bounded by the right side of \zcref{eq:graph-bound-goal}.

\item To conclude, we must adapt the estimate \zcref{eq:ustarrootestimate}
to the case when $u_{\star}$ is not the root of $\mathfrak{t}$.
We start by enumerating all of the inner vertices that connect $u_{\star}$
to the root as $\{u\st u\preccurlyeq u_{\star}\}=\{u_{i}\}^{k}_{i=0}$
with $u_{\star}=u_{0}\succ u_{1}\succ\cdots\succ u_{k}$, with $u_{k}$
the root of $\mathfrak{t}$. Each $u_{i}$, $i=1,\ldots,k$, has two
children, one being $u_{i-1}$ and the other which we denote by $v_{i}$.
Now we use \zcref{eq:sumproduct-ustar} as well as $k$ applications
of \zcref{eq:sumprodindeq}, with $u$ taken to be each $v_{i}$,
to obtain the upper bound
\begin{equation}
\mathscr{I}_{\mathfrak{t}}\lesssim2^{-\left(\deg_Q(\mathsf{S}_{u_{\star}})-3\right)n_{\star}}\sum_{\mathbf{n}\in\mathcal{A}^{k}_{1}}\prod^{k}_{i=1}2^{-\left(\deg_Q(\mathsf{S}_{v_{i}})+\eta(u_{i})\right)\mathbf{n}(u_{i})}.\label{eq:startoffline}
\end{equation}
Here we have defined, for $\ell\in\{1,\ldots,k\}$,
\begin{equation}
\mathcal{A}^{k}_{\ell}\coloneqq\left\{ \mathbf{n}\colon\{u_{i}\}^{k}_{i=\ell}\to\mathbb{N}\st n_{\star}\ge\mathbf{n}(u_{\ell})\ge\mathbf{n}(u_{\ell+1})\ge\cdots\ge\mathbf{n}(u_{k})\right\} .\label{eq:Akelldef}
\end{equation}
We note that, by \zcref{eq:degtilderecursion}, we have, for each
$i=1,\ldots,k$, that
\[
\deg_Q(\mathsf{S}_{u_{i}})=\eta(u_{i})+\deg_Q(\mathsf{S}_{u_{i-1}})+\deg_Q(\mathsf{S}_{v_{i}}),
\]
so if we abbreviate 
\begin{equation}
d_{i}\coloneqq\deg_Q(\mathsf{S}_{u_{i}})\label{eq:didef}
\end{equation}
 then we can rewrite \zcref{eq:startoffline} as 
\begin{equation}
\mathscr{I}_{\mathfrak{t}}\lesssim2^{-\left(d_{0}-3\right)n_{\star}}\sum_{\mathbf{n}\in\mathcal{A}^{k}_{1}}\prod^{k}_{i=1}2^{-(d_{i}-d_{i-1})\mathbf{n}(u_{i})}=S_{\mathfrak{t};1,0}.\label{eq:startofflinebd}
\end{equation}
where we define
\begin{align}
S_{\mathfrak{t};\ell,m} & \coloneqq(n_{\star}+1)^{m}2^{-(d_{m}-3)n_{\star}}\sum_{\mathbf{n}\in\mathcal{A}^{k}_{\ell}}2^{-(d_{\ell}-d_{m})\mathbf{n}(u_{\ell})}\prod^{k}_{i=\ell+1}2^{-(d_{i}-d_{i-1})\mathbf{n}(u_{i})}\label{eq:Stdef}\\
\ovset{\zcref{eq:Akelldef}} & =(n_{\star}+1)^{m}2^{-(d_{m}-3)n_{\star}}\sum_{\mathbf{n}\in\mathcal{A}^{k}_{\ell+1}}\left(\prod^{k}_{i=\ell+1}2^{-(d_{i}-d_{i-1})\mathbf{n}(u_{i})}\right)\sum^{n_{\star}}_{n_{\ell}=\mathbf{n}(u_{\ell+1})}2^{-(d_{\ell}-d_{m})\mathbf{n}(u_{\ell})}\label{eq:Stbd}
\end{align}
if $\ell\in\{1,\ldots,k\}$. If $\ell=k+1$ then we use the convention
\begin{equation}
S_{\mathfrak{t};k+1,m}=(n_{\star}+1)^{m}2^{-(d_{m}-3)n_{\star}},\label{eq:ellisk+1}
\end{equation}
which should be interpreted as the case of the empty product on the
right side of \zcref{eq:Stdef} and the set $\mathcal{A}^{k}_{k+1}$
comprising the (single) empty sequence. Now there are two cases for
the last sum in \zcref{eq:Stbd}, depending on the sign of $d_{\ell}-d_{m}$.
If $d_{\ell}-d_{m}>0$, then we have
\[
\sum^{n_{\star}}_{n_{\ell}=\mathbf{n}(u_{\ell+1})}2^{-(d_{\ell}-d_{m})\mathbf{n}(u_{\ell})}\lesssim2^{-(d_{\ell}-d_{m})\mathbf{n}(u_{\ell+1})},
\]
and so we obtain in this case that
\begin{align}
S_{\mathfrak{t};\ell,m} & \lesssim(n_{\star}+1)^{m}2^{-(d_{m}-3)n_{\star}}\sum_{\mathbf{n}\in\mathcal{A}^{k}_{\ell+1}}2^{-(d_{\ell}-d_{m})\mathbf{n}(u_{\ell+1})}\prod^{k}_{i=\ell+1}2^{-(d_{i}-d_{i-1})\mathbf{n}(u_{i})}\nonumber \\
 & =(n_{\star}+1)^{m}2^{-(d_{m}-3)n_{\star}}\sum_{\mathbf{n}\in\mathcal{A}^{k}_{\ell+1}}2^{-(d_{\ell+1}-d_{m})\mathbf{n}(u_{\ell+1})}\prod^{k}_{i=\ell+2}2^{-(d_{i}-d_{i-1})\mathbf{n}(u_{i})}\overset{\zcref{eq:Stdef}}{=}S_{\mathfrak{t};\ell+1,m}.\label{eq:poscase}
\end{align}
On the other hand, if $d_{\ell}-d_{m}\le0$, then we have 
\[
\sum^{n_{\star}}_{n_{\ell}=\mathbf{n}(u_{\ell+1})}2^{-(d_{\ell}-d_{m})\mathbf{n}(u_{\ell})}\lesssim(n_{\star}-n_{\ell}+1)2^{-(d_{\ell}-d_{m})n_{\star}}\le(n_{\star}+1)2^{-(d_{\ell}-d_{m})n_{\star}},
\]
and using this bound in \zcref{eq:Stbd}, we get
\[
S_{\mathfrak{t};\ell,m}\lesssim(n_{\star}+1)^{m+1}2^{-(d_{\ell}-3)n_{\star}}\sum_{\mathbf{n}\in\mathcal{A}^{k}_{\ell+1}}\prod^{k}_{i=\ell+1}2^{-(d_{i}-d_{i-1})\mathbf{n}(u_{i})}.
\]
Therefore, in this case we have, as long as $m\le\ell$, that $S_{\mathfrak{t};\ell,m}\lesssim S_{\mathfrak{t};\ell+1,\ell+1}$.
Combining this observation with \zcref{eq:poscase}, we get, as long
as $m\le\ell$, that
\[
S_{\mathfrak{t};\ell,m}\lesssim\max\{S_{\mathfrak{t};\ell+1,m},S_{\mathfrak{t};\ell+1,\ell+1}\}.
\]
Using this bound inductively, we obtain
\begin{equation}\label{eq:stfinalbound}
S_{\mathfrak{t};1,0}\lesssim\max^{k+1}_{m=1}S_{\mathfrak{t};k+1,m}\overset{\zcref{eq:ellisk+1}}{=}\max^{k}_{m=1}(n_{\star}+1)^{m}2^{-(d_{m}-3)n_{\star}}\lesssim d_{\mathfrak{s},\mathscr{S}}(\mathbf{x}_{\varrho_{1}},\mathbf{x}_{\varrho_{2}})^{\gamma}\left(\log\left(2+d_{\mathfrak{s},\mathscr{S}}(\mathbf{x}_{\varrho_{1}},\mathbf{x}_{\varrho_{2}})^{-1}\right)\right)^{ k} \;,
\end{equation}
where by construction $k\le k_0$ as defined in the statement of the proposition,
and in the last inequality of \zcref{eq:stfinalbound} we recalled the definitions \zcref{eq:nstarustar,eq:gammadegtilde,eq:didef}.
This matches the right side of \zcref{eq:graph-bound-goal}, and so
the proof is complete in light of \zcref{eq:hepp-separate,eq:hepp-separate-1,eq:startofflinebd}.\qedhere
\end{thmstepnv}
\end{proof}